\documentclass[10pt]{book}
\usepackage{amssymb}
\usepackage{amsthm,amsmath}
\usepackage{cite}
\usepackage{makeidx}

\usepackage{ccfonts,eulervm}
\usepackage[T1]{fontenc}

\title{A Generalization of Bohr-Mollerup's Theorem for Higher Order Convex Functions}

\author{Jean-Luc Marichal\thanks{University of Luxembourg, Department of Mathematics, Maison du Nombre, 6, avenue de la Fonte, L-4364 Esch-sur-Alzette, Luxembourg. Email: jean-luc.marichal@uni.lu}
\and
Na\"im Zena\"idi\thanks{University of Li\`ege, Department of Mathematics, All\'ee de la D\'ecouverte, 12 - B37, B-4000 Li\`ege, Belgium. Email: nzenaidi@uliege.be}}

\date{Revised, \today}

\makeindex
\begin{document}

\theoremstyle{plain}
\newtheorem{theorem}{Theorem}[chapter]
\newtheorem{lemma}[theorem]{Lemma}
\newtheorem{proposition}[theorem]{Proposition}
\newtheorem{corollary}[theorem]{Corollary}
\newtheorem{fact}[theorem]{Fact}

\theoremstyle{definition}
\newtheorem{definition}[theorem]{Definition}
\newtheorem{ex}[theorem]{Example} 
    \newenvironment{example}    
    {\renewcommand{\qedsymbol}{$\lozenge$}%
    \pushQED{\qed}\begin{ex}}
    {\popQED\end{ex}}

\theoremstyle{remark}
\newtheorem{rem}[theorem]{Remark} 
    \newenvironment{remark}    
    {\renewcommand{\qedsymbol}{$\lozenge$}%
    \pushQED{\qed}\begin{rem}}
    {\popQED\end{rem}}
\newtheorem{pro}[theorem]{Project} 
    \newenvironment{project}    
    {\renewcommand{\qedsymbol}{$\lozenge$}%
    \pushQED{\qed}\begin{pro}}
    {\popQED\end{pro}}
\newtheorem{conjecture}[theorem]{Conjecture}
\newtheorem*{claim}{Claim}

\newcommand{\R}{\mathbb{R}}
\newcommand{\N}{\mathbb{N}}
\newcommand{\Z}{\mathbb{Z}}
\renewcommand{\S}{\mathrm{S}}
\newcommand{\cC}{\mathcal{C}}
\newcommand{\cD}{\mathcal{D}}
\newcommand{\cK}{\mathcal{K}}
\newcommand{\cR}{\mathcal{R}}
\newcommand{\emptybox}{\mathstrut^{\mathstrut}_{\mathstrut}}

\def\tchoose#1#2{{\textstyle{{{#1}\choose{#2}}}}}
\def\parag#1{\medskip\noindent{\textbf{#1.}}}

\frontmatter
\maketitle

\newpage
\null

\newpage
\null

\vspace{20ex}
\begin{large}
\begin{flushright}
\emph{To Pascale, Olivia, Jean-Philippe, and Claudia}\\
\emph{Jean-Luc Marichal}
\end{flushright}
\end{large}

\vspace{10ex}
\begin{large}
\begin{flushright}
\emph{To Elise, Nils, and Eva}\\
\emph{Na\"im Zena\"idi}
\end{flushright}
\end{large}

\newpage
\null

\clearpage
\addcontentsline{toc}{chapter}{Preface}
\chaptermark{Preface}
\chapter*{Preface}

In this work, we provide a general and unified setting for a systematic and in-depth investigation of a broad variety of functions, including several special functions like the Euler gamma function, the polygamma functions, the Barnes $G$-function, the Hurwitz zeta function, and the generalized Stieltjes constants.

We know for instance that the gamma function
$$
\Gamma(x) ~=~ \int_0^{\infty}t^{x-1}{\,}e^{-t}{\,}dt
$$
satisfies several fundamental properties and identities such as Bohr-Mollerup's characterization, Euler's infinite product, Gauss' multiplication formula, Stirling's formula, and Weierstrass' infinite product. In this book, we show through a series of new and elementary results that a large range of functions of mathematical analysis satisfy analogues of several properties of the gamma function, including those mentioned above.

The starting point of our theory is the remarkable characterization of the gamma function on the open half-line $\R_+=(0,\infty)$ by Harald Bohr and Johannes Mollerup \cite{BohMol22}. It simply states that the log-gamma function $f(x)=\ln\Gamma(x)$ is the unique convex solution vanishing at $x=1$ to the equation
$$
f(x+1) - f(x) ~=~ \ln x,\qquad x>0.
$$
This result can actually be slightly generalized as follows, where $\Delta$ denotes the classical forward difference operator.
\begin{quote}
\emph{All eventually convex solutions to the equation $\Delta f(x)=\ln x$ on $\R_+$ are of the form $f(x)=c+\ln\Gamma(x)$, where $c\in\R$.}
\end{quote}
(Here and throughout, a function is said to be eventually convex if it is convex in a neighborhood of infinity.)

This characterization was later generalized to a wide class of functions by Wolfgang Krull \cite{Kru48} and then independently by Roger Webster \cite{Web97b}. They essentially showed that for any eventually concave function $g\colon\R_+\to\R$ having the asymptotic property that the sequence $n\mapsto\Delta g(n)$ converges to zero, there exists exactly one (up to an additive constant) eventually convex solution $f\colon\R_+\to\R$ to the equation $\Delta f=g$. When $g(x)=\ln x$, this latter result clearly reduces to the above Bohr-Mollerup characterization of the gamma function.

Krull-Webster's result constitutes an important contribution to the resolution of the difference equation $\Delta f=g$ on the real half-line $\R_+$. Indeed, it provides analogues of Bohr-Mollerup's characterization for many functions, including the gamma function, the digamma function, and the $q$-gamma functions. Nevertheless, we can see that the asymptotic condition imposed on the function $g$ remains rather restrictive. For instance, it is not satisfied by the functions $g(x)=x\ln x$ and $g(x)=\ln\Gamma(x)$. In fact, it is not even satisfied by the identity function $g(x)=x$.

In this monograph, we generalize Krull-Webster's result by relaxing considerably the asymptotic condition into requiring that the sequence $n\mapsto\Delta^p g(n)$ converges to zero for some nonnegative integer $p$. Each of the functions $g(x)=x\ln x$, $g(x)=\ln\Gamma(x)$, and $g(x)=x$ clearly satisfies this new assumption for $p=2$. Moreover, in our generalization the convexity and concavity properties used by Krull and Webster are naturally replaced with their $p$-order versions. On this matter, it is noteworthy that many of the familiar functions of real analysis are eventually convex or concave of any order.

The solutions arising from Krull-Webster's characterization are called \emph{$\log\Gamma$-type functions}. Those arising from our generalized version are called \emph{multiple $\log\Gamma$-type functions}. As we demonstrate through this work, this latter class of functions is very rich and includes a wide variety of special functions.

In the diagram opposite, we describe how our result generalizes to any nonnegative integer $p$ the special case when $p=1$ obtained by Krull and Webster, who both generalized Bohr-Mollerup's theorem.

\begin{figure}[p]
\bigskip
\begin{center}
\fbox{
\begin{minipage}{65ex}
\begin{center}
\textbf{Higher order version of Krull-Webster's theory}\par\medskip $\Delta f(x) ~=~ g(x)$\par $g$ is eventually $p$-concave and $\Delta^p g(n)\to 0$\par $f$ is eventually $p$-convex \par \bigskip\textbf{Solutions: Multiple $\log\Gamma$-type functions}
\end{center}
\end{minipage}
}

\vspace{4ex} $\mbox{\Large $\boldsymbol{\uparrow}$}$ \vspace{4ex}

\fbox{
\begin{minipage}{50ex}
\begin{center}
\textbf{Krull-Webster's theory}\par\medskip $\Delta f(x) ~=~ g(x)$ \par $g$ is eventually concave and $\Delta g(n)\to 0$ \par $f$ is eventually convex \par\bigskip\textbf{Solutions: $\log\Gamma$-type functions}
\end{center}
\end{minipage}
}

\vspace{4ex} $\mbox{\Large $\boldsymbol{\uparrow}$}$ \vspace{4ex}

\fbox{
\begin{minipage}{50ex}
\begin{center}
\textbf{Bohr-Mollerup's characterization}\par\medskip $\Delta f(x) ~=~ \ln x$\par $f$ is eventually convex\par\bigskip \textbf{Solutions: $f(x) = c+\ln\Gamma(x)$}
\end{center}
\end{minipage}
}

\end{center}
\end{figure}

We also follow and generalize Webster's approach and provide for multiple $\log\Gamma$-type functions analogues of \emph{Euler's constant}, \emph{Euler's infinite product}, \emph{Gauss' limit}, \emph{Gauss' multiplication formula}, \emph{Gautschi's inequality}, \emph{Legendre's duplication formula}, \emph{Raabe's formula}, \emph{Stirling's constant}, \emph{Stirling's formula}, \emph{Wallis's product formula}, \emph{Weierstrass' infinite product}, and \emph{Wendel's inequality} for the gamma function. We also introduce and discuss analogues of \emph{Binet's function}, \emph{Burnside's formula}, \emph{Euler's reflection formula}, \emph{Fontana-Mascheroni's series}, \emph{Gauss' digamma theorem}, and \emph{Webster's functional equation}. Some additional properties of multiple $\log\Gamma$-type functions are also provided and discussed, including asymptotic equivalences, asymptotic expansion formulas, Taylor series expansion formulas, and Gregory formula-based series representations.

Lastly, we apply our results thoroughly to several usual special functions, including the gamma and digamma functions, the polygamma functions, the $q$-gamma function, the Barnes $G$-function, the Hurwitz zeta function and its higher order derivatives, and the generalized Stieltjes constants. We also briefly discuss some further special functions such as the Gauss error function, the exponential integral, the regularized incomplete gamma function, the multiple gamma functions, and the Bernoulli polynomials. All these examples illustrate how powerful is our theory to produce formulas and identities almost mechanically.

For example, applying our results to the gamma function $\Gamma\colon\R_+\to\R_+$ itself, we easily retrieve the following Gauss limit
$$
\Gamma(x) ~=~ \lim_{n\to\infty}\frac{n!{\,}n^x}{x(x+1){\,}\cdots{\,}(x+n)}{\,},\qquad x>0,
$$
and the Weierstrass infinite product
$$
\Gamma(x) ~=~ \frac{e^{-\gamma x}}{x}\,\prod_{k=1}^{\infty}\frac{e^{\frac{x}{k}}}{1+\frac{x}{k}}{\,},\qquad x>0,
$$
where $\gamma$ is the Euler constant. We also easily establish the double inequality
$$
\left(1+\frac{1}{x}\right)^{-\frac{1}{2}} \leq ~ \frac{\Gamma(x)}{\sqrt{2\pi}{\,}x^{x-\frac{1}{2}}{\,}e^{-x}} ~\leq \left(1+\frac{1}{x}\right)^{\frac{1}{2}}{\,},\qquad x>0,
$$
from which we immediately derive the Stirling formula
$$
\Gamma(x) ~\sim ~ \sqrt{2\pi}{\,}x^{x-\frac{1}{2}}e^{-x}\qquad\text{as $x\to\infty$}.
$$

To give another example, let us consider the restriction to $\R_+$ of the Barnes $G$-function (see Barnes \cite{Bar99,Bar01,Bar04}). That is, the function $G\colon\R_+\to\R_+$ whose logarithm $f(x)=\ln G(x)$ is the unique $2$-convex solution vanishing at $x=1$ to the equation
$$
f(x+1) - f(x) ~=~ \ln\Gamma(x),\qquad x>0.
$$
Thus defined, the function $\ln G(x)$ is a multiple $\log\Gamma$-type function, and we can therefore state the following analogue of Bohr-Mollerup's characterization.
\begin{quote}
\emph{All eventually $2$-convex solutions to the equation $\Delta f(x)=\ln\Gamma(x)$ on $\R_+$ are of the form $f(x)=c+\ln G(x)$, where $c\in\R$.}
\end{quote}
Using our results, we can also easily show that the Barnes $G$-function satisfies the following analogue of Gauss' limit for the gamma function
$$
G(x) ~=~ \lim_{n\to\infty}\frac{\Gamma(1)\Gamma(2){\,}\cdots{\,}\Gamma(n)}{\Gamma(x)\Gamma(x+1){\,}\cdots{\,}\Gamma(x+n)}{\,}
n!^x{\,}n^{{x\choose 2}},\qquad x>0.
$$
Moreover, it satisfies the following analogue of Weierstrass' infinite product
$$
G(x) ~=~ \frac{e^{(-\gamma -1){x\choose 2}}}{\Gamma(x)}{\,}
\prod_{k=1}^{\infty}\frac{\Gamma(k)}{\Gamma(x+k)}{\,}k^xe^{\psi_1(k){\,}{x\choose 2}},\qquad x>0,
$$
where $\psi_1$ is the trigamma function defined by the equation
$$
\psi_1(x) ~=~ D^2\ln\Gamma(x)\qquad\text{for $x>0$}.
$$
We also establish the double inequality
$$
\left(1+\frac{1}{x}\right)^{-\frac{5}{12}} \leq ~ \frac{G(x)\,\Gamma(x)^{\frac{1}{2}}{\,}A^2{\,}(2\pi)^{\frac{1}{4}}}{x^{\frac{1}{12}}
{\,}e^{\psi_{-2}(x)+\frac{1}{12}}} ~\leq \left(1+\frac{1}{x}\right)^{\frac{5}{12}},\qquad x>0,
$$
from which we immediately derive the following analogue of Stirling's formula
$$
G(x) ~\sim ~ A^{-2}{\,}(2\pi)^{-\frac{1}{4}}{\,}x^{\frac{1}{12}}
\,\Gamma(x)^{-\frac{1}{2}}{\,}e^{\psi_{-2}(x)+\frac{1}{12}}\qquad\text{as $x\to\infty$},
$$
where $\psi_{-2}$ is the polygamma function defined by the equation
$$
\psi_{-2}(x) ~=~ \int_0^x\ln\Gamma(t){\,}dt\qquad\text{for $x>0$}
$$
and $A$ is the Glaisher-Kinkelin constant defined by the equation
$$
\zeta'(-1) ~=~ \frac{1}{12}-\ln A.
$$
In this work, we also derive many other properties of the Barnes $G$-function simply as analogues of properties of the gamma function.

To sum up, in this monograph we develop a far-reaching generalization of the Bohr-Mollerup theorem, along lines initiated by Krull, Webster, and some others but going considerably further than past work. In particular, we show using elementary techniques that many classical properties of the gamma function have counterparts for a very wide variety of functions.

In this regard we observe that, in his outstanding exposition of the gamma function, Emil Artin \cite[p.~vi]{Art15} wrote:
\begin{quote}
``\emph{I feel that this monograph will help to show that the gamma function can be thought of as one of the elementary functions, and that all of its basic properties can be established using elementary methods of the calculus.}''
\end{quote}
In writing this book, our hope is to convince the reader that Artin's statement applies also to all the multiple $\log\Gamma$-type functions.

Lastly, since Bohr-Mollerup's theorem dates back to 1922, this work is also an opportunity to mark the 100th anniversary of this remarkable result and to spark the interest and enthusiasm of a large number of researchers in this theory.

\bigskip

\begin{flushright}
Jean-Luc Marichal\\
Na\"im Zena\"idi
\end{flushright}


\vspace{8ex}

\begin{footnotesize}
\noindent {\em 2020 Mathematics Subject Classification.}

\noindent Primary: 39B22, 39A06, 26A51. Secondary: 39A60, 33B15, 33B20.
\end{footnotesize}

%
%
%

\vspace{3ex}

\begin{footnotesize}
\noindent{\em Key words and phrases.}

\noindent Difference equation, higher order convexity, Bohr-Mollerup-Artin's theorem, Krull-Webster's theory, generalized Stirling's formula, generalized Stirling's constant, generalized Euler's constant, Euler's reflection formula, Euler's infinite product, Weierstrass' infinite product, Gauss' multiplication theorem, Gauss' digamma theorem, Raabe's formula, Wallis's product formula, Fontana-Mascheroni's series, Barnes $G$-function, Hurwitz zeta function, gamma-related function, multiple gamma-type function, generalized Stieltjes constant.
\end{footnotesize}

\clearpage
\addcontentsline{toc}{chapter}{List of main symbols}
\chaptermark{List of main symbols}
\chapter*{List of main symbols}

\begin{tabbing}
$\sim$ \hspace{12ex}\= asymptotic equivalence, p.~\pageref{p:AEq7} \\
$\to_{\S}$ \> convergence over $\S\in\{\N,\R\}$, p.~\pageref{p:xtoS} \\
$\lfloor x\rfloor$ \> floor of $x$, p.~\pageref{p:floor} \\
$\lceil x\rceil$ \> ceiling of $x$, p.~\pageref{p:ceiling} \\
$\{x\}$ \> fractional part of $x$, p.~\pageref{p:fracpart} \\
$A$ \> Glaisher-Kinkelin's constant, p.~\pageref{p:Glaisher} \\
$B_n$ \> $n$th Bernoulli number, p.~\pageref{p:BernN} \\
 \> $n$th Bernoulli polynomial, p.~\pageref{p:BernP} \\
$\cC^k$ \> set of $k$ times continuously differentiable functions on $\R_+$, p.~\pageref{p:Ck} \\
$\cC^k(I)$ \> set of $k$ times continuously differentiable functions on $I$, p.~\pageref{p:CkI} \\
$\cC^{\infty}$ \> $\bigcap_{k\geq 0}\cC^k$, p.~\pageref{p:Cinf} \\
$\cC^{\infty}(I)$ \> $\bigcap_{k\geq 0}\cC^k(I)$, p.~\pageref{p:Cinf} \\
$\deg f$ \> asymptotic degree of $f$, see Definition~\ref{de:de4g3f5}, p.~\pageref{p:degf} \\
$\mathrm{dom}(\Sigma)$ \> domain of the map $\Sigma$, p.~\pageref{p:domS} \\
$D$ \> ordinary derivative operator, p.~\pageref{p:Diff} \\
$\cD^p_{\S}$ ($p\geq 0$) \> $\{g\colon\R_+\to\R:\text{$\Delta^pg(t)\to 0$ as $t\to_{\S}\infty$}\}$, p.~\pageref{p:DpS} \\
$\cD^{\infty}_{\S}$ \> $\bigcup_{p\geq 0}\cD^p_{\S}{\,}$, p.~\pageref{p:Dinfty} \\
$\cD^{-1}_{\N}$ \> $\{g\colon\R_+\to\R:\text{the sequence $n\mapsto g(n)$ is summable}\}$, p.~\pageref{p:D-1} \\
$\cD^{-1}_{\S}$ \> subset of $\cD^{-1}_{\N}$ introduced in Definition~\ref{de:D-15R23}, p.~\pageref{p:D-1R} \\
$f[x_0,\ldots,x_p]$ \> divided difference of $f$ at the points $x_0,\ldots,x_p$, p.~\pageref{p:DivDiff} \\
$f^p_n[g]$ \> function defined in \eqref{eq:defFpn}, p.~\pageref{p:fnp} \\
$G_n$ \> $n$th Gregory coefficient, $G_n=\int_0^1{t\choose n}{\,}dt$, p.~\pageref{p:Gn} \\
$\overline{G}_n$ \> $1-\sum_{j=1}^n|G_j|$, p.~\pageref{p:bGn} \\
$H_x$ \> harmonic number function, p.~\pageref{p:harmonic} \\
$I$ \> arbitrary real interval whose interior is nonempty, p.~\pageref{p:IntI} \\
$J^q[g]$ \> generalized Binet's function defined in \eqref{eq:Binet643780}, p.~\pageref{p:Jqg}\\
$\cK^p$ \> $\cK^p_+\cup\cK^p_-$, p.~\pageref{p:KpI} \\
$\cK^p_+$ \> set of functions $f\colon\R_+\to\R$ that are eventually $p$-convex, p.~\pageref{p:Kppm} \\
$\cK^p_-$ \> set of functions $f\colon\R_+\to\R$ that are eventually $p$-concave, p.~\pageref{p:Kppm} \\
$\cK^p(I)$ \> $\cK^p_+(I)\cup\cK^p_-(I)$, p.~\pageref{p:KpI} \\
$\cK^p_+(I)$ \> set of functions $f\colon I\to\R$ that are $p$-convex, p.~\pageref{p:KppmI} \\
$\cK^p_-(I)$ \> set of functions $f\colon I\to\R$ that are $p$-concave, p.~\pageref{p:KppmI} \\
$\cK^{\infty}$ \> $\bigcap_{p\geq 0}\cK^p$, p.~\pageref{p:Kinfty} \\
$\mathrm{Log}\Gamma_p$ \> set of $\log\Gamma_p$-type functions, p.~\pageref{p:LGps} \\
$\N$ \> $\{0,1,2,\ldots\}$, p.~\pageref{p:NN} \\
$\N^*$ \> $\{1,2,\ldots\}$, p.~\pageref{p:NNs} \\
$P_p[f]$ \> interpolating polynomial of degree $\leq p$ of $f$, p.~\pageref{p:IntPol} \\
$\overline{P}_p[f]$ \> piecewise polynomial function defined in \eqref{eq:PieceWPol48}, p.~\pageref{p:PieceWPol23} \\
$\R_+$ \> open half-line $(0,\infty)$, p.~\pageref{p:Rplus} \\
$\cR^p_{\S}$ \> $\{g\colon\R_+\to\R:\text{for each $x>0$, $\rho^p_t[g](x)\to 0$ as $t\to_{\S}\infty$}\}$, p.~\pageref{p:RpS} \\
$\cR^{\infty}_{\S}$ \> $\bigcup_{p\geq 0}\cR^p_{\S}{\,}$, p.~\pageref{p:Rinfty} \\
$R^q_{m,n}[g]$ \> remainder in Gregory's summation formula \eqref{eq:GregoryMN}, p.~\pageref{p:Rqmng} \\
$\mathrm{ran}(\Sigma)$ \> range of the map $\Sigma$, p.~\pageref{p:ranS} \\
$\S$ \> $\S=\N$ or $\R$, p.~\pageref{p:sS} \\
$x^{\underline{k}}$ \> $x(x-1){\,}\cdots{\,}(x-k+1)$, p.~\pageref{p:xFFP} \\
$x\to_{\S}\infty$ \> $x$ tends to infinity over $\S\in\{\N,\R\}$, p.~\pageref{p:xtoS} \\
$x_+$ \> $\max\{0,x\}$, p.~\pageref{p:xPl} \\
$\gamma$ \> Euler's constant, p.~\pageref{p:EuC} \\
$\gamma[g]$ \> generalized Euler's constant associated with $g$, p.~\pageref{p:genEC} \\
$\gamma_q$ \> Stieltjes's constants, p.~\pageref{p:Sc630} \\
 \> generalized Stieltjes's constants, p.~\pageref{p:gSc63} \\
$\Gamma$ \> gamma function, p.~\pageref{p:gamma3} \\
$\Gamma_p$ \> multiple gamma function $\Gamma_p$, p.~\pageref{p:Gp1s} \\
 \> set of $\Gamma_p$-type functions, p.~\pageref{p:Gps} \\
$\Gamma_q$ \> $q$-gamma function, p.~\pageref{p:qG53} \\
$\varepsilon_k(x)$ \> sign of $x^{\underline{k}}$, p.~\pageref{p:xSigFFP} \\
$\zeta(s)$ \> Riemann zeta function, p.~\pageref{p:Rzeta} \\
$\zeta(s,x)$ \> Hurwitz zeta function, p.~\pageref{p:Hzeta} \\
$\Delta$ \> forward difference operator, p.~\pageref{p:Delta} \\
$\Delta_{[h]}$ \> forward difference operator with step $h$, p.~\pageref{p:DeltaSh} \\
$\rho^p_a[f]$ \> function defined in \eqref{eq:deflambdapt}, p.~\pageref{p:rho} \\
$\sigma[g]$ \> asymptotic constant associated with $g$; see \eqref{eq:sigmagg86}, p.~\pageref{p:sigmag} \\
$\overline{\sigma}[g]$ \> logarithm of the generalized Stirling constant associated with $g$, p.~\pageref{p:bsigmag} \\
$\Sigma$ \> map defined in \eqref{eq:DefGStar}, p.~\pageref{p:Sigma} \\
$\psi$ \> digamma function, p.~\pageref{p:digamma} \\
$\psi_{\nu}$ \> polygamma functions, p.~\pageref{p:polygamma}
\end{tabbing}

\tableofcontents

\mainmatter

\chapter{Introduction}
\label{chapter:1}

Let $\R_+$ denote the open half-line $(0,\infty)$\label{p:Rplus} and let $\Delta$ denote the forward difference operator on the space of functions from $\R_+$ to $\R$. In this book, we are interested in the classical difference equation $\Delta f=g$ on $\R_+$, which can be written explicitly as
$$
f(x+1)-f(x) ~=~ g(x),\qquad x>0,
$$
where $g\colon\R_+\to\R$ is a given function. This equation appears naturally in the theory of the Euler gamma function, with $f(x) = \ln\Gamma(x)$\label{p:gamma3} and $g(x) = \ln x$, but also in the study of many other special functions such as the Barnes $G$-function and the Hurwitz zeta function (see Examples~\ref{ex:Bar261} and \ref{ex:HuZe82} below).

It is easily seen that, for any function $g\colon\R_+\to\R$, the equation above has infinitely many solutions, and each of them can be uniquely determined by prescribing its values in the interval $(0,1]$. Moreover, any two solutions always differ by a $1$-periodic function, i.e., a periodic function of period $1$.

For certain functions $g$, however, special solutions can be determined by their local properties or their asymptotic behaviors. On this issue, a seminal result is the very nice characterization of the gamma function by Bohr and Mollerup~\cite{BohMol22}. We recall this important result in the following theorem.

\begin{theorem}[Bohr-Mollerup's theorem]\label{thm:BM538Thm9}\index{Bohr-Mollerup theorem|textbf}
All log-convex solutions $f\colon\R_+\to\R_+$ to the equation
\begin{equation}\label{eq:FunctEqMainMult}
f(x+1) ~=~ x{\,}f(x),\qquad x>0,
\end{equation}
are of the form $f(x)=c{\,}\Gamma(x)$, where $c>0$.
\end{theorem}

The additive, but equivalent, version of this result, obtained by taking the logarithm of both sides of \eqref{eq:FunctEqMainMult}, can be stated as follows.
\begin{quote}
\emph{For $g(x)=\ln x$, all convex solutions $f\colon\R_+\to\R$ to the difference equation $\Delta f=g$ are of the form $f(x)=c+\ln\Gamma(x)$, where $c\in\R$.}
\end{quote}

As we can see, this characterization enables one to single out the gamma function as a kind of \emph{principal solution} to its equation (N\"orlund~\cite[Chapter~5]{Nor24} calls it the ``Hauptl\"osung'').

It is noteworthy that the proof of Bohr-Mollerup's characterization was simplified later by Artin \cite{Art31} (see also Artin~\cite{Art15}) and, as observed by Webster~\cite{Web97b}, this result has then become known also ``as the Bohr-Mollerup-Artin Theorem, and was adopted by Bourbaki \cite{Bou51} as the starting point for his exposition of the gamma function.''

\begin{remark}
In their original result, Bohr and Mollerup actually considered the additional assumption that $f(1)=1$, thus leading to the gamma function as the unique solution (see Artin~\cite[p.~14]{Art15}). However, it is easy to see that Theorem~\ref{thm:BM538Thm9} immediately follows from this original result (just replace $f(x)$ with $f(x)/f(1)$).
\end{remark}

A remarkable generalization of Bohr-Mollerup's theorem was provided by Krull \cite{Kru48,Kru49} and then independently by Webster \cite{Web97a,Web97b}. Recall that a function $g\colon\R_+\to\R$ is said to be eventually convex (resp.\ eventually concave) if it is convex (resp.\ concave) in a neighborhood of infinity. Krull \cite{Kru48} essentially showed that for any eventually concave function $g\colon\R_+\to\R$ having the asymptotic property that, for each $h>0$,
\begin{equation}\label{eq:KWac3g}
g(x+h)-g(x) ~\to ~ 0\qquad\text{as $x\to\infty$},
\end{equation}
there exists exactly one (up to an additive constant) eventually convex solution $f\colon\R_+\to\R$ to the equation $\Delta f=g$ (and dually, if $g$ is eventually convex, then $f$ is eventually concave). He also provided an explicit expression for this solution as a pointwise limit of functions, namely
$$
f(x) ~=~ f(1)+\lim_{n\to\infty}f^1_n[g](x),\qquad x>0,
$$
where
\begin{equation}\label{eq:defF1n04}
f^1_n[g](x) ~=~ -g(x)+\sum_{k=1}^{n-1}(g(k)-g(x+k))+x{\,}g(n).
\end{equation}
Much later, and independently, Webster \cite{Web97a,Web97b} established the multiplicative version of Krull's result.

We can actually show that this result still holds if we replace the asymptotic condition \eqref{eq:KWac3g} imposed on the function $g$ with the slightly more general condition that the sequence $n\mapsto\Delta g(n)$ converges to zero. However, although this result constitutes a very nice generalization of Bohr-Mollerup's theorem, we note that the latter asymptotic condition remains a rather restrictive assumption. For instance, it is not satisfied by the functions $g(x)=x\ln x$ and $g(x)=\ln\Gamma(x)$.

In this work, we generalize Krull-Webster's result above by relaxing the asymptotic condition on $g$ into the much weaker requirement that the sequence $n\mapsto\Delta^pg(n)$ converges to zero for some nonnegative integer $p$. More precisely, we show that Krull-Webster's result still holds if we assume this weaker condition, provided that we replace the convexity and concavity properties with the $p$-convexity\index{$p$-convexity} and $p$-concavity\index{$p$-concavity} properties (see Definition~\ref{de:pconcconv}) and the function $f_n^1[g]$ defined in \eqref{eq:defF1n04} with an appropriate version of it, which we now introduce.

Throughout this book, we let $\N$\label{p:NN} denote the set of nonnegative integers and we let $\N^*$\label{p:NNs} denote the set of strictly positive integers.

\begin{definition}
For any $p\in\N$, any $n\in\N^*$, and any $g\colon\R_+\to\R$, we define the function $f^p_n[g]\colon\R_+\to\R$ by the equation\label{p:fnp}
\begin{equation}\label{eq:defFpn}
f^p_n[g](x) ~=\null -g(x)+\sum_{k=1}^{n-1}(g(k)-g(x+k))+\sum_{j=1}^p\tchoose{x}{j}\,\Delta^{j-1}g(n).
\end{equation}
\end{definition}

We now state our result in the following existence theorem. It actually constitutes the $p$-order version of Krull-Webster's result.

\begin{theorem}[Existence]\label{thm:int2}\index{existence theorem}
Let $p\in\N$ and suppose that the function $g\colon\R_+\to\R$ is eventually $p$-convex or eventually $p$-concave and has the asymptotic property that the sequence $n\mapsto\Delta^pg(n)$ converges to zero. Then there exists a unique (up to an additive constant) eventually $p$-convex or eventually $p$-concave solution $f\colon\R_+\to\R$ to the difference equation $\Delta f=g$. Moreover,
\begin{equation}\label{eq:int2}
f(x) ~=~ f(1)+\lim_{n\to\infty}f^p_n[g](x),\qquad x>0,
\end{equation}
and $f$ is $p$-convex (resp.\ $p$-concave) on any unbounded subinterval of\/ $\R_+$ on which $g$ is $p$-concave (resp.\ $p$-convex).
\end{theorem}

Webster \cite[Theorem 3.1]{Web97b} also established (in the multiplicative notation) a uniqueness theorem, which does not require the function $g$ to be eventually convex or eventually concave. In the next theorem, we provide the $p$-order version of this result.

\begin{theorem}[Uniqueness]\label{thm:int1}\index{uniqueness theorem}
Let $p\in\N$ and let the function $g\colon\R_+\to\R$ have the property that the sequence $n\mapsto\Delta^pg(n)$ converges to zero. Suppose that $f\colon\R_+\to\R$ is an eventually $p$-convex or eventually $p$-concave function satisfying the difference equation $\Delta f=g$. Then $f$ is uniquely determined (up to an additive constant) by $g$ through the equation
$$
f(x) ~=~ f(1)+\lim_{n\to\infty}f^p_n[g](x),\qquad x>0.
$$
\end{theorem}

We observe that Theorem~\ref{thm:int2} was first proved in the case when $p=0$ by John~\cite{Joh39}. As mentioned above, it was also established in the case when $p=1$ by Krull \cite{Kru48} and then by Webster \cite{Web97b}. More recently, the case when $p=2$ was investigated by Rassias and Trif \cite{RasTri07}, but the asymptotic condition they imposed on the function $g$ is much stronger than ours and hence it defines a very specific subclass of functions. (We discuss Rassias and Trif's result in Appendix~\ref{chapter:A-KW561}.) We also observe that attempts to establish Theorem~\ref{thm:int2} for any value of $p$ were made by Kuczma~\cite[Theorem~1]{Kuc64} (see also Kuczma~\cite[pp.\ 118--121]{Kuc68}) and then by Ardjomande \cite{Ard68}. However, the representation formulas they provide for the solutions are rather intricate. Thus, to the best of our knowledge, both Theorems~\ref{thm:int2} and \ref{thm:int1}, as stated above in their full generality and simplicity, were previously unknown.

For any solution $f$ arising from Theorem~\ref{thm:int2} when $p=1$, Webster~\cite{Web97b} calls the function $\exp\circ f$ a \emph{$\Gamma$-type function}.\index{$\Gamma$-type function} In fact, $\exp\circ f$ reduces to the gamma function (i.e., $f(x)=\ln\Gamma(x)$) when $\exp\circ g$ is the identity function (i.e., $g(x)=\ln x$), which simply means that the gamma function restricted to $\R_+$ is itself a $\Gamma$-type function. In this particular case, the limit given in \eqref{eq:int2} reduces to the following Gauss well-known limit\index{Gauss' limit}\index{Gauss' limit!analogue} for the gamma function (see Artin~\cite[p.~15]{Art15})
\begin{equation}\label{eq:GaussLimit42}
\Gamma(x) ~=~ \lim_{n\to\infty}\frac{n!{\,}n^x}{x(x+1){\,}\cdots{\,}(x+n)}{\,},\qquad x>0.
\end{equation}

Similarly, for any fixed $p\in\N$ and any solution $f$ arising from Theorem~\ref{thm:int2}, we call the function $\exp\circ f$ a \emph{$\Gamma_p$-type function},\index{$\Gamma_p$-type function} and we naturally call the function $f$ a \emph{$\log\Gamma_p$-type function}.\index{$\log\Gamma_p$-type function} When the value of $p$ is not specified, we call these functions \emph{multiple $\Gamma$-type function}\index{multiple $\Gamma$-type function} and \emph{multiple $\log\Gamma$-type function},\index{multiple $\log\Gamma$-type function} respectively. This terminology will be introduced more formally and justified in Section~\ref{subsec:MLG-t}.

Interestingly, Webster established for $\Gamma$-type functions analogues of \emph{Euler's constant}, \emph{Gauss' multiplication formula}, \emph{Legendre's duplication formula}, \emph{Stirling's formula}, and \emph{Weierstrass' infinite product} for the gamma function. In this work, we also establish for multiple $\Gamma$-type functions and multiple $\log\Gamma$-type functions analogues of all the formulas above as well as analogues of \emph{Euler's infinite product}, \emph{Gautschi's inequality}, \emph{Raabe's formula}, \emph{Stirling's constant}, \emph{Wallis's product formula}, and \emph{Wendel's inequality}. We also introduce and discuss analogues of \emph{Binet's function}, \emph{Burnside's formula}, \emph{Euler's reflection formula}, \emph{Fontana-Mascheroni's series}, and \emph{Gauss' digamma theorem}. Thus, (to paraphrase Webster~\cite[p.~607]{Web97b}) for each multiple $\Gamma$-type function, it is no longer surprising for instance that ``some analogue of \emph{Legendre's duplication formula} must hold, almost rendering a formal proof unnecessary!''

All these results, together with the uniqueness and existence theorems above, show that the theory we develop in this book provides a very general and unified framework to study the properties of a large variety of functions. Thus, for each of these functions we can retrieve known formulas and sometimes establish new ones.

At the risk of repeating a large part of our preface, we now present two representative examples to illustrate the way our results can be applied to derive formulas methodically.

\begin{example}[The Barnes $G$-function, see Section~\ref{sec:Barnes558}]\label{ex:Bar261}\index{Barnes's $G$-function|textbf}
The restriction to $\R_+$ of the Barnes $G$-function can be defined as the function $G\colon\R_+\to\R_+$ whose logarithm $f(x)=\ln G(x)$ is the unique eventually 2-convex solution that vanishes at $x = 1$ to the equation
$$
f(x+1)-f(x) ~=~ \ln\Gamma(x),\qquad x>0.
$$
Thus, our Theorems~\ref{thm:int2} and \ref{thm:int1} apply with $g(x)=\ln\Gamma(x)$ and $p=2$, which shows that the function $\ln G(x)$ is a $\log\Gamma_2$-type function and hence that the function $G(x)$ is a $\Gamma_2$-type function. In particular, formula \eqref{eq:int2} provides the following analogue of \emph{Gauss' limit} for the gamma function
$$
G(x) ~=~ \lim_{n\to\infty}\frac{\Gamma(1)\Gamma(2){\,}\cdots{\,}\Gamma(n)}{\Gamma(x)\Gamma(x+1){\,}\cdots{\,}\Gamma(x+n)}{\,}
n!^x{\,}n^{{x\choose 2}}{\,}.
$$
Using some of our new results, we are also able to derive various unusual formulas and properties. For instance, we have the following analogue of \emph{Euler's infinite product}
$$
G(x) ~=~ \frac{1}{\Gamma(x)}{\,}
\prod_{k=1}^{\infty}\frac{\Gamma(k)}{\Gamma(x+k)}{\,}k^x(1+1/k)^{{x\choose 2}}
$$
and the following analogue of \emph{Weierstrass' infinite product}
$$
G(x) ~=~ \frac{e^{(-\gamma -1){x\choose 2}}}{\Gamma(x)}{\,}
\prod_{k=1}^{\infty}\frac{\Gamma(k)}{\Gamma(x+k)}{\,}k^xe^{\psi'(k){\,}{x\choose 2}},
$$
where $\gamma$ is the Euler constant\index{Euler's constant} and $\psi$ is the digamma function.\index{digamma function} We also have the following analogue of \emph{Stirling's formula}
$$
G(x) ~\sim ~ A^{-2}{\,}(2\pi)^{-\frac{1}{4}}{\,}x^{\frac{1}{12}}\,\Gamma(x)^{-\frac{1}{2}}{\,}e^{\psi_{-2}(x)+\frac{1}{12}}\qquad\text{as $x\to\infty$}{\,},
$$
where $\psi_{-2}$ is the polygamma function defined by the equation
$$
\psi_{-2}(x) ~=~ \int_0^x\ln\Gamma(t){\,}dt\qquad\text{for $x>0$}{\,},
$$
and $A$ is Glaisher-Kinkelin's constant\index{Glaisher-Kinkelin's constant|textbf} defined by the equation\label{p:Glaisher}
$$
\zeta'(-1) ~=~ \frac{1}{12}-\ln A{\,}.
$$
(Here the map $s\mapsto\zeta'(s)$ denotes the derivative of the Riemann zeta function.) We can also easily derive the following analogue of \emph{Wendel's double inequality}
$$
\left(1+\frac{a}{x}\right)^{-\left|{a-1\choose 2}\right|} \leq ~ \frac{G(x+a)}{G(x)\,\Gamma(x)^a{\,}x^{a\choose 2}} ~\leq \left(1+\frac{a}{x}\right)^{\left|{a-1\choose 2}\right|},
$$
which holds for any $x>0$ and any $a\geq 0$. As a corollary, this inequality immediately provides the following asymptotic equivalence
$$
\frac{G(x+a)}{G(x)} ~\sim ~ \Gamma(x)^a{\,}x^{{a\choose 2}} \qquad\text{as $x\to\infty$}{\,},
$$
which reveals the asymptotic behavior of $G(x+a)/G(x)$ for large values of $x$.
\end{example}

\begin{example}[The Hurwitz zeta function, see Section~\ref{sec:Hurw49}]\label{ex:HuZe82}\index{Hurwitz zeta function|textbf}
Consider the Hurwitz zeta function\label{p:Hzeta} $s\mapsto\zeta(s,a)$, defined when $\Re(a)>0$ as an analytic continuation to $\mathbb{C}\setminus\{1\}$ of the series
$$
\sum_{k=0}^{\infty}(a+k)^{-s}{\,\,},\qquad\Re(s)>1{\,}.
$$
This function is known to satisfy the difference equation
$$
\zeta(s,a+1)-\zeta(s,a) ~=~ -a^{-s}.
$$
Thus, it is not difficult to see that, for any $s\in\R\setminus\{1\}$, the restriction of the map $x\mapsto\zeta(s,x)$ to $\R_+$ is a $\log\Gamma_{p(s)}$-type function, where
$$
p(s) ~=~ \max\{0,\lfloor 1-s\rfloor\}.
$$
Theorem~\ref{thm:int1} then tells us that all eventually $p(s)$-convex or eventually $p(s)$-concave solutions $f_s\colon\R_+\to\R$ to the difference equation
$$
f_s(x+1)-f_s(x) ~=~ -x^{-s}
$$
are of the form
$$
f_s(x) ~=~ c_s+\zeta(s,x),
$$
where $c_s\in\R$. Moreover, equation \eqref{eq:int2} provides the following analogue of \emph{Gauss' limit} for the gamma function
$$
\zeta(s,x) ~=~ \zeta(s) + x^{-s}
+\lim_{n\to\infty}
\left(\sum_{k=1}^{n-1}\left((x+k)^{-s}-k^{-s}\right)-\sum_{j=1}^{p(s)}\tchoose{x}{j}\,\Delta_n^{j-1}n^{-s}\right),
$$
where $s\mapsto\zeta(s)=\zeta(s,1)$\label{p:Rzeta} is the Riemann zeta function\index{Riemann zeta function}. Some of our results also enable us to derive the following analogues of \emph{Stirling's formula}
\begin{eqnarray*}
\zeta(s,x) + \frac{x^{1-s}}{1-s}-\sum_{j=1}^{p(s)}G_j\,\Delta_x^{j-1}x^{-s} & \to & 0\qquad\text{as $x\to\infty$},\\
\zeta(s,x) + \frac{1}{1-s}\sum_{j=0}^{p(s)}\tchoose{1-s}{j}\,\frac{B_j}{x^{s+j-1}} & \to & 0\qquad\text{as $x\to\infty$},
\end{eqnarray*}
where $G_n$ is the $n$th Gregory coefficient\index{Gregory coefficients} and $B_n$ is the $n$th Bernoulli number.\index{Bernoulli numbers} For instance, setting $s=-\frac{3}{2}$ in these asymptotic formulas, we obtain
\begin{eqnarray*}
\textstyle{\zeta\left(-\frac{3}{2},x\right)+\frac{2}{5}{\,}x^{5/2}-\frac{7}{12}{\,}x^{3/2}
+\frac{1}{12}{\,}(x+1)^{3/2}} &\to & 0\qquad\text{as $x\to\infty$}{\,},\\
\textstyle{\zeta\left(-\frac{3}{2},x\right)+\frac{2}{5}{\,}x^{5/2}-\frac{1}{2}{\,}x^{3/2}+\frac{1}{8}{\,}x^{1/2}} &\to & 0\qquad\text{as $x\to\infty$}{\,}.
\end{eqnarray*}
Many more formulas and properties involving the Hurwitz zeta function will be provided and discussed in Section~\ref{sec:Hurw49}.
\end{example}

The two examples above illustrate the scope of our theory and the diversity of our results. These examples and many others will be explored and discussed in the last chapters of this book. However, in the first chapters we will almost always use the basic function $g(x)=\ln x$ as the guiding example to illustrate our results.

\parag{Outline of the book} Let us now see how this book is organized. On the whole, Chapters~\ref{chapter:2} to \ref{chapter:8} are devoted to the conceptual part: we develop our theory and establish our results. Chapters~\ref{chapter:10} to \ref{chapter:12} focus on applications to a large number of functions, including several classical special functions. In between, Chapter~\ref{chapter:9} presents an overview and a summary of our results. After reading this introduction, the reader interested by such an overview can go immediately to Chapter~\ref{chapter:9}.

In Chapter~\ref{chapter:2}, we present some definitions and preliminary results on Newton interpolation theory as well as on higher order convexity properties.

In Chapter~\ref{chapter:3}, we establish Theorems~\ref{thm:int2} and \ref{thm:int1} and provide conditions for the sequence $n\mapsto f^p_n[g](x)$ to converge uniformly on any bounded subset of $\R_+$. We also examine the particular case when the sequence $n\mapsto g(n)$ is summable, and we provide historical remarks on some improvements of Krull-Webster's theory.

In Chapter~\ref{chapter:4}, we investigate the functions that satisfy the asymptotic condition stated in Theorems~\ref{thm:int2} and \ref{thm:int1}. We also investigate those functions that are eventually $p$-convex or eventually $p$-concave.

In Chapter~\ref{chapter:5}, we introduce, investigate, and characterize the multiple $\log\Gamma$-type functions.

Chapter~\ref{chapter:6} is devoted to an asymptotic analysis of multiple $\log\Gamma$-type functions. More specifically, in that chapter we show how Euler's constant, Stirling's constant, Stirling's formula, and Wendel's inequality for the gamma function can be generalized to the multiple $\Gamma$-type functions and multiple $\log\Gamma$-type functions and we introduce and discuss analogues of Binet's function and Burnside's formula. We also show how the so-called Gregory summation formula, with an integral form of the remainder, can be very easily derived in this setting.

In Chapter~\ref{chapter:7}, we discuss conditions for the multiple $\log\Gamma$-type functions to be differentiable and establish several important properties of the higher order derivatives of these functions.

In Chapter~\ref{chapter:8}, we explore further properties of the multiple $\log\Gamma$-type functions. Specifically, we provide asymptotic expansions of these functions as well as analogues of Euler's infinite product, Fontana-Mascheroni's series, Gauss' multiplication formula, Gautschi's inequality, Raabe's formula, Wallis's product formula, and Weierstrass' infinite product for the gamma function. We also discuss analogues of Euler's reflection formula and Gauss' digamma theorem, and we define and solve a generalized version of a functional equation proposed by Webster.

Chapter~\ref{chapter:9} is the transition from the theory to the applications. It provides a catalogue of our most relevant results, which can be used as a checklist to investigate the multiple $\log\Gamma$-type functions. Chapter~\ref{chapter:9} is self-contained and can be read right after this introduction.

In Chapters~\ref{chapter:10} to \ref{chapter:12}, we apply our results to a number of multiple $\Gamma$-type functions and multiple $\log\Gamma$-type functions, some of which are well-known special functions related to the gamma function.

In Chapter~\ref{chapter:13}, we make some concluding remarks and propose a list of interesting open questions.

\parag{Notation and basic definitions} Throughout this book, we use the following notation and definitions. Further definitions will be given in the subsequent chapters.

Unless indicated otherwise, the symbol $I$\label{p:IntI} always denotes an arbitrary interval of the real line whose interior is nonempty.

The symbol $\S$\label{p:sS} represents either $\N$ or $\R$. For any $\S\in\{\N,\R\}$, the notation $x\to_{\S}\infty$\label{p:xtoS} means that $x$ tends to infinity, assuming only values in $\S$. We sometimes omit the subscript $\S$ when no confusion may arise.

Two functions $f\colon\R_+\to\R$ and $g\colon\R_+\to\R$ such that $f(x)/g(x)\to 1$ as $x\to_{\S}\infty$ are said to be asymptotically equivalent (over $\S$). In this case, we write
$$
f(x) ~\sim ~ g(x)\qquad\text{as $x\to_{\S}\infty$}.\label{p:AEq7}
$$

For any $x\in\R$, we set
$$
x_+ ~=~ \max\{0,x\}.\label{p:xPl}
$$
As usual, we also let $\lfloor x\rfloor$ denote the floor of $x$,\label{p:floor} i.e., the greatest integer less than or equal to $x$. Similarly, we let $\lceil x\rceil$ denote the ceiling of $x$,\label{p:ceiling} i.e., the smallest integer greater than or equal to $x$. When no confusion may arise, we let $\{x\}$ denote the fractional part of $x$,\label{p:fracpart} i.e., $\{x\}=x-\lfloor x\rfloor$.

For any $x\in\R$ and any $k\in\N$, we set
$$
x^{\underline{k}} ~=~ x(x-1){\,}\cdots{\,}(x-k+1) ~=~ \frac{\Gamma(x+1)}{\Gamma(x-k+1)}\label{p:xFFP}
$$
and we let
$$
\varepsilon_k(x) ~\in ~\{-1,0,1\}\label{p:xSigFFP}
$$
denote the sign of $x^{\underline{k}}$.

For any $k\in\N$ and any nonempty open real interval $I$, we let $\cC^k(I)$\label{p:CkI} denote the set of $k$ times continuously differentiable functions on $I$, and we set $\cC^k=\cC^k(\R_+)$.\label{p:Ck} We also introduce the intersection sets
$$
\cC^{\infty}(I) ~=~ \bigcap_{k\geq 0}\cC^k(I)\qquad\text{and}\qquad\cC^{\infty} ~=~ \bigcap_{k\geq 0}\cC^k.\label{p:Cinf}
$$

We let $\Delta$\label{p:Delta} and $D$\label{p:Diff} denote the usual difference and derivative operators, respectively. We sometimes add a subscript to specify the variable on which the operator acts, e.g., writing $\Delta_n$ and $D_x$.

Recall that the digamma function\index{digamma function|textbf} $\psi$\label{p:digamma} is defined on $\R_+$ by the equation
$$
\psi(x) ~=~ D\ln\Gamma(x)\qquad\text{for $x>0$}.
$$
The polygamma functions\index{polygamma functions|textbf} $\psi_{\nu}$\label{p:polygamma} ($\nu\in\Z$) are defined on $\R_+$ as follows (see, e.g., Srivastava and Choi \cite{SriCho12}). If $\nu\in\N$, then
$$
\psi_{\nu}(x) ~=~ D^{\nu}\psi(x) ~=~ \psi^{(\nu)}(x).
$$
In particular, $\psi_0=\psi$ is the digamma function. If $\nu\in\Z\setminus\N$, then we introduce the functions
$$
\psi_{-1}(x) ~=~ \ln\Gamma(x)
$$
and
$$
\psi_{\nu-1}(x) ~=~ \int_0^x\psi_{\nu}(t){\,}dt ~=~ \int_0^x\frac{(x-t)^{-\nu-1}}{(-\nu-1)!}{\,}\ln\Gamma(t){\,}dt.
$$

Recall also that the harmonic number function\index{harmonic number function|textbf} $x\mapsto H_x$\label{p:harmonic} is defined on $(-1,\infty)$ by the equation
$$
H_x ~=~ \sum_{k=1}^{\infty}\left(\frac{1}{k}-\frac{1}{x+k}\right)\qquad\text{for $x>-1$}.
$$
Clearly, this function has the property that
$$
\Delta_x H_x ~=~ \frac{1}{x+1}{\,},\qquad x > -1.
$$
Moreover, both functions $H_x$ and $\psi(x)$ are strongly related: we have
$$
H_{x-1} ~=~ \psi(x)+\gamma{\,},\qquad x>0{\,},
$$
where $\gamma$\label{p:EuC} is Euler's constant\index{Euler's constant|textbf} (also called Euler-Mascheroni constant).

We end this first chapter by introducing some new concepts that will be very useful in this book.

\begin{definition}
For any $a>0$, any $p\in\N$, and any $g\colon\R_+\to\R$, we define the function $\rho^p_a[g]\colon [0,\infty)\to\R$ by the equation\label{p:rho}
\begin{equation}\label{eq:deflambdapt}
\rho^p_a[g](x) ~=~ g(x+a)-\sum_{j=0}^{p-1}\tchoose{x}{j}\,\Delta^jg(a)\qquad\text{for $x\geq 0$}.
\end{equation}
\end{definition}

Identity \eqref{eq:deflambdapt} clearly shows that the function $\rho^p_a[g]$ is actually defined on the open interval $(-a,\infty)$. However, in this work we will almost always consider it as a function defined on the interval $[0,\infty)$. We also note that $\rho^p_a[g](0)=0$.

\begin{definition}\label{de:R0Sp42}
For any $p\in\N$ and any $\S\in\{\N,\R\}$, we let $\cR^p_{\S}$\label{p:RpS} denote the set of functions $g\colon\R_+\to\R$ having the asymptotic property that, for each $x>0$,
$$
\rho^p_a[g](x) ~\to ~0\qquad\text{as $a\to_{\S}\infty$}.
$$
We also let $\cD^p_{\S}$\label{p:DpS} denote the set of functions $g\colon\R_+\to\R$ having the asymptotic property that
$$
\Delta^pg(x) ~\to ~0\qquad\text{as $x\to_{\S}\infty$}.
$$
\end{definition}

We immediately observe that the inclusion $\cD^p_{\S}\subset\cD^{p+1}_{\S}$ holds for every $p\in\N$. We will see in Sections~\ref{sec:31M8r0} and \ref{sec:Ac1} that the inclusion $\cR^p_{\S}\subset\cR^{p+1}_{\S}$ also holds for every $p\in\N$.

\chapter{Preliminaries}
\label{chapter:2}

This chapter is devoted to some basic definitions and results that are needed in this book. We essentially focus on the Newton interpolation theory and the higher order convexity and concavity properties.

Recall that, unless indicated otherwise, the symbol $I$ always denotes an arbitrary real interval whose interior is nonempty.

\section{Newton interpolation theory}
\index{divided difference|(}

In this first section, we recall some basic facts about Newton interpolation theory and divided differences. We also establish a result on the derivatives of interpolating polynomials. For background see, e.g., de Boor \cite[Chapter~1]{deBoo01}, Gel'fond \cite[Chapter 1]{Gel71}, Quarteroni {\em et al.}\cite[Section 8.2.2]{QuaSacSal10}, and Stoer and Bulirsch \cite[Section 2.1.3]{StoBul10}.

Let $n\in\N$ and let $x_0,x_1,\ldots,x_n$ be any (not necessarily distinct) points of $I$. Let also $f\colon I\to\R$ be so that $D^{m_i-1}f(x_i)$ exists for $i=0,\ldots,n$, where $m_i$ is the multiplicity of $x_i$ among the points $x_0,x_1,\ldots,x_n$.

We let
$$
f[x_0,x_1,\ldots,x_n]\label{p:DivDiff}
$$
denote the divided difference of $f$ at the points $x_0,x_1,\ldots,x_n$, and we let the map
$$
x ~\mapsto ~ P_n[f](x_0,x_1,\ldots,x_n;x)\label{p:IntPol}
$$
denote the interpolating polynomial\index{interpolating polynomial|textbf} of $f$ with nodes at $x_0,x_1,\ldots,x_n$, i.e., the unique polynomial $P$ satisfying the equations
$$
D^kP(x_i) ~=~ D^kf(x_i),\qquad 0\leq k \leq m_i-1,\qquad i=0,\ldots,n.
$$
This polynomial has degree at most $n$.

Recall that $f[x_0,x_1,\ldots,x_n]$ is precisely the coefficient of $x^n$ in the interpolating polynomial $P_n[f](x_0,x_1,\ldots,x_n;x)$. More precisely, the Newton interpolation formula\index{Newton interpolation formula|textbf} states that
\begin{equation}\label{NewIntFor582}
P_n[f](x_0,x_1,\ldots,x_n;x) ~=~ \sum_{k=0}^nf[x_0,x_1,\ldots,x_k]\,\prod_{i=0}^{k-1}(x-x_i).
\end{equation}
Moreover, the corresponding interpolation error\index{interpolation error|textbf} at any $x\in I$ can take the following form
\begin{equation}\label{IntError582}
f(x) - P_n[f](x_0,x_1,\ldots,x_n;x) ~=~ f[x_0,x_1,\ldots,x_n,x]\,\prod_{i=0}^n(x-x_i).
\end{equation}

Recall also that the map
$$
(z_0,z_1,\ldots,z_n) ~\mapsto ~ f[z_0,z_1,\ldots,z_n]
$$
is symmetric, i.e., invariant under any permutation of its arguments. Moreover, the divided differences of $f$ can be computed via the following recurrence relation. For any $k\in\{0,1,\ldots,n\}$, we have $f[x_k]=f(x_k)$ and
\begin{equation}\label{eq:DivDiffRec7}
f[x_0,\ldots,x_k] ~=~
\begin{cases}
\displaystyle{\frac{f[x_1,\ldots,x_k]-f[x_0,\ldots,x_{k-1}]}{x_k-x_0}}{\,}, & \text{if $x_k\neq x_0$},\\
\displaystyle{\frac{1}{k!}{\,}D^kf(x_0)}{\,}, & \text{if $x_0 =x_1 =\cdots =x_k$}.
\end{cases}
\end{equation}
When the points $x_0,x_1,\ldots,x_n$ are pairwise distinct, we also have the following explicit expression
\begin{equation}\label{eq:DivDiffPai4Dis82}
f[x_0,x_1,\ldots,x_n] ~=~ \sum_{k=0}^n\frac{f(x_k)}{\prod_{j\neq k}(x_k-x_j)}{\,}.
\end{equation}

We now establish a proposition that shows how the derivative of an interpolating polynomial of a differentiable function $f$ is related to the derivative of $f$.

\begin{proposition}\label{prop:2A4DiffInt4Pol}
Suppose that $I$ is an arbitrary nonempty open real interval. For any $n\in\N^*$, any system $x_0 < x_1 < \cdots < x_n$ of $n+1$ points in $I$, and any differentiable function $f\colon I\to\R$, there exist $n$ points $\xi_0,\ldots,\xi_{n-1}$ in $I$ such that, for $i=0,\ldots,n-1$, we have $x_i<\xi_i<x_{i+1}$ and
\begin{equation}\label{eq:AppA3sq1}
D_xP_n[f](x_0,\ldots,x_n;x)\big|_{x=\xi_i} ~=~ f'(\xi_i).
\end{equation}
Moreover, we have
\begin{equation}\label{eq:AppA3sq2}
D_xP_n[f](x_0,\ldots,x_n;x) ~=~ P_{n-1}[f'](\xi_0,\ldots,\xi_{n-1};x)
\end{equation}
and
\begin{equation}\label{eq:AppA3sq3}
n{\,}f[x_0,\ldots,x_n] ~=~ f'[\xi_0,\ldots,\xi_{n-1}].
\end{equation}
\end{proposition}

\begin{proof}
The function $g\colon I\to\R$ defined by the equation
$$
g(x) ~=~ P_n[f](x_0,\ldots,x_n;x)-f(x)\qquad\text{for $x\in I$}
$$
vanishes at the $n+1$ points $x_0,x_1,\ldots,x_n$. The first part of the proposition then follows from applying Rolle's theorem in each interval $(x_i,x_{i+1})$. Now, identity \eqref{eq:AppA3sq2} immediately follows from \eqref{eq:AppA3sq1} and the very definition of the interpolating polynomial. Identity \eqref{eq:AppA3sq3} then follows by equating the coefficients of $x^{n-1}$ in \eqref{eq:AppA3sq2}.
\end{proof}

\index{divided difference|)}

\section{Higher order convexity and concavity}
\index{higher order convexity and concavity|(}

Let us recall the definitions of $p$-convex and $p$-concave functions and present some related results. For background see, e.g., Kuczma \cite{Kuc64}, Kuczma \cite[Chapter~15]{Kuc85}, Popoviciu \cite{Pop33}, and Roberts and Varberg \cite[pp.\ 237--240]{RobVar73}.

\begin{definition}[$p$-convexity and $p$-concavity]\label{de:pconcconv}\index{$p$-convexity|textbf}\index{$p$-concavity|textbf}
A function $f\colon I\to\R$ is said to be \emph{convex of order $p$} (resp.\ \emph{concave of order $p$}) or simply \emph{$p$-convex} (resp.\ \emph{$p$-concave}) for some integer $p\geq -1$ if for any system $x_0<x_1<\cdots < x_{p+1}$ of $p+2$ points in $I$ it holds that
$$
f[x_0,x_1,\ldots,x_{p+1}]~\geq ~0\qquad (\text{resp.}~f[x_0,x_1,\ldots,x_{p+1}]~\leq ~0).
$$
\end{definition}

Thus defined, a function $f\colon I\to\R$ is $1$-convex if it is an ordinary convex function; it is $0$-convex if it is increasing (in the wide sense); it is $(-1)$-convex if it is nonnegative.

Let us now introduce a practical notation to denote the set of $p$-convex functions and the set of $p$-concave functions.

\begin{definition}
Let $p\geq -1$ be an integer.
\begin{itemize}
\item We let $\cK^p_+(I)$ (resp.\ $\cK^p_-(I)$)\label{p:KppmI} denote the set of functions $f\colon I\to\R$ that are $p$-convex (resp.\ $p$-concave).
\item We let $\cK^p_+$ (resp.\ $\cK^p_-$)\label{p:Kppm} denote the set of functions $f\colon\R_+\to\R$ that are eventually $p$-convex (resp.\ eventually $p$-concave), i.e., $p$-convex (resp.\ $p$-concave) in a neighborhood of infinity.
\end{itemize}
We also set
$$
\cK^p(I) ~=~ \cK^p_+(I)\cup\cK^p_-(I)\qquad\text{and}\qquad\cK^p ~=~ \cK^p_+\cup\cK^p_-.\label{p:KpI}
$$
\end{definition}

The following proposition shows that both sets $\cK^p_+(I)$ and $\cK^p_-(I)$ are convex cones whose intersection is precisely the real linear space of all polynomials of degree less than or equal to $p$. A similar description of the sets $\cK^p_+$ and $\cK^p_-$ will be given in Corollary~\ref{cor:convCones48}.

\begin{proposition}\label{prop:convCones48}
For any $p\in\N$, the sets $\cK^p_+(I)$ and $\cK^p_-(I)$ are convex cones. These cones are opposite in the sense that $f$ lies in $\cK^p_+(I)$ if and only if $-f$ lies in $\cK^p_-(I)$. Moreover, the intersection $\cK^p_+(I)\cap\cK^p_-(I)$ is the real linear space of all polynomials of degree less than or equal to $p$.
\end{proposition}

\begin{proof}
That the sets $\cK^p_+(I)$ and $\cK^p_-(I)$ are convex cones is trivial; indeed, if $f_1$ and $f_2$ lie in $\cK^p_+(I)$ for instance, then so does $c_1f_1+c_2f_2$ for any $c_1,c_2\geq 0$. By definition of $\cK^p_+(I)$ and $\cK^p_-(I)$, these cones are clearly opposite. Now, let $f$ lie in $\cK^p_+(I)\cap\cK^p_-(I)$ and let $x_0<\cdots <x_p$ be $p+1$ points in $I$. By \eqref{IntError582}, for any $x\in I$ we must have
$$
f(x) - P_p[f](x_0,x_1,\ldots,x_p;x) ~=~ 0,
$$
which shows that $f$ is a polynomial of degree at most $p$. Conversely, using \eqref{IntError582} again, we can readily see that any such polynomial lies in $\cK^p_+(I)\cap\cK^p_-(I)$.
\end{proof}

We now present an important lemma. It is interesting in its own right and will be very useful in the subsequent chapters. A variant of this result can be found in Kuczma~\cite[Lemma 15.7.2]{Kuc85}.

Recall first that for any $f\colon I\to\R$, any $p\in\N$, and any $x\in I$ such that $x+p\in I$, we have
\begin{equation}\label{eq:DelDD62}
\Delta^pf(x) ~=~ p!{\,}f[x,x+1,\ldots,x+p].
\end{equation}

\begin{lemma}\label{lemma:pCInc5}
Let $p\in\N$ and let $\mathcal{I}_{p+1}$ denote the set of tuples of $I^{p+1}$ whose components are pairwise distinct. A function $f\colon I\to\R$ lies in $\cK^p_+(I)$ (resp.\ $\cK^p_-(I)$) if and only if the restriction of the map
$$
(z_0,\ldots,z_p) ~\mapsto ~ f[z_0,\ldots,z_p]
$$
to $\mathcal{I}_{p+1}$ is increasing (resp.\ decreasing) in each place. In particular, if $I$ is not upper bounded, then for any function $f$ lying in $\cK^p_+(I)$ (resp.\ $\cK^p_-(I)$), the function $\Delta^pf$ is increasing (resp.\ decreasing) on $I$.
\end{lemma}

\begin{proof}
Using the definition of $p$-convexity and the standard recurrence relation \eqref{eq:DivDiffRec7} for divided differences,\index{divided difference} we can see that $f$ lies in $\cK^p_+(I)$ if and only if, for any pairwise distinct $x_0,\ldots,x_{p+1}\in I$, we have
$$
\frac{f[x_1,x_2\ldots,x_{p+1}]-f[x_0,x_2\ldots,x_{p+1}]}{x_1-x_0} ~\geq ~ 0.
$$
Equivalently, for any pairwise distinct $x_0,\ldots,x_{p+1}\in I$, we have
$$
x_1>x_0\quad\Rightarrow\quad f[x_1,x_2\ldots,x_{p+1}]-f[x_0,x_2\ldots,x_{p+1}] ~\geq ~ 0.
$$
The latter condition exactly means that the map defined in the statement is increasing in the first place. Since this map is symmetric, it must be increasing in each place. The second part of the lemma follows from \eqref{eq:DelDD62}.
\end{proof}

We end this section with a second lemma, which provides some important connections between higher order convexity and higher order differentiability. In fact, these connections can be derived (sometimes tediously) from various results given in the references mentioned in the beginning of this section, especially the book by Kuczma~\cite[Chapter~15]{Kuc85}. However, for the sake of self-containment we provide a detailed proof in Appendix~\ref{chapter:pconv184}.

\begin{lemma}\label{lemma:PrelKp}
Let $I$ be an nonempty open real interval and let $p\in\N$. Then the following assertions hold.
\begin{enumerate}
\item[(a)] We have $\cK^{p+1}(I)\subset\cC^p(I)$.
\item[(b)] Assume that $I$ is not upper bounded. If $f\in\cK^p_+(I)$, then $\Delta^jf\in\cK^{p-j}_+(I)$ for every $j\in\{0,\ldots,p+1\}$.
\item[(c)] If $f\in\cC^{j}(I)\cap\cK^p_+(I)$ for some $j\in\{0,\ldots,p+1\}$, then $f^{(j)}\in\cK^{p-j}_+(I)$.
\item[(d)] If $f\in\cC^1(I)$ and $f'\in\cK^{p-1}_+(I)$, then $f\in\cK^p_+(I)$.
\end{enumerate}
\end{lemma}

\begin{proof}
See Appendix~\ref{chapter:pconv184}.
\end{proof}

\index{higher order convexity and concavity|)}

\section{A key lemma}
\label{sec:NewtInt62}

Let $p\in\N$, $a>0$, and $f\colon\R_+\to\R$. Combining Newton's interpolation formula \eqref{NewIntFor582}\index{Newton interpolation formula}  with identity \eqref{eq:DelDD62}, we can readily see that the unique interpolating polynomial\index{interpolating polynomial} of $f$ with nodes at the $p$ points $a,a+1,\ldots,a+p-1$ takes the form
\begin{equation}\label{eq:NewtonInt}
P_{p-1}[f](a,a+1,\ldots,a+p-1;x) ~=~ \sum_{j=0}^{p-1}\tchoose{x-a}{j}\,\Delta^jf(a).
\end{equation}
If $p=0$, then this polynomial is naturally the zero polynomial, which is assumed to have degree $-1$. Moreover, using \eqref{IntError582} we can immediately see that the corresponding interpolation error\index{interpolation error} at any $x>0$ is
\begin{equation}\label{eq:GenErrIn}
f(x)-\sum_{j=0}^{p-1}\tchoose{x-a}{j}\,\Delta^jf(a) ~=~ (x-a)^{\underline{p}}{\,}f[a,a+1,\ldots,a+p-1,x].
\end{equation}
Now, the right side of \eqref{eq:GenErrIn} is actually the remainder of the $(p-1)$th degree Newton expansion of $f(x)$ about $x=a$ (see, e.g., Graham {\em et al.} \cite[Section 5.3]{GraKnuPat94}). Note also that formula \eqref{eq:GenErrIn} is a pure identity in the sense that it is valid without any restriction on the form of $f(x)$.


Using \eqref{eq:NewtonInt} and \eqref{eq:GenErrIn} we see that, for any $a>0$, any $x\geq 0$, any $p\in\N$, and any $g\colon\R_+\to\R$, the quantity $\rho^p_a[g](x)$ defined in \eqref{eq:deflambdapt} is precisely the interpolation error at $a+x$ when considering the interpolating polynomial\index{interpolating polynomial} of $g$ with nodes at $a,a+1,\ldots,a+p-1$. We then immediately derive the following identities:
\begin{eqnarray}
\rho^p_a[g](x) &=& g(a+x)-P_{p-1}[g](a,a+1,\ldots,a+p-1;a+x){\,},\label{eq:LErrIn}\\
\rho^p_a[g](x) &=& x^{\underline{p}}{\,}g[a,a+1,\ldots,a+p-1,a+x]{\,}.\label{eq:LDDiv}
\end{eqnarray}
We note that identity \eqref{eq:LDDiv} also extends to the case when $x\in\{0,1,\ldots,p-1\}$, even if $g$ is not differentiable. Indeed, in this case we must have $\rho^p_a[g](x)=0$ by \eqref{eq:LErrIn}.

We now end this chapter with a key lemma that will be used repeatedly in this book. Although this lemma is rather technical, it is at the root of various fundamental convergence results of our theory. Recall first that, for any $k\in\N$, the symbol $\varepsilon_k(x)$ stands for the sign of $x^{\underline{k}}$.

\begin{lemma}\label{lemma:VarEpsIneq}
Let $p\in\N$, $f\in\cK^p$, and $a>0$ be so that $f$ is $p$-convex or $p$-concave on $[a,\infty)$. Then, for any $x\geq 0$, we have
\begin{eqnarray*}
0 ~\leq ~ \pm{\,}\varepsilon_{p+1}(x){\,}\rho_a^{p+1}[f](x)
&\leq &  \pm{\,}\left|\tchoose{x-1}{p}\right|\left(\Delta^pf(a+x)-\Delta^pf(a)\right)\\
&\leq & \pm\left|\tchoose{x-1}{p}\right|\sum_{j=0}^{\lceil x\rceil -1}\Delta^{p+1}f(a+j),
\end{eqnarray*}
where $\pm$ stands for $1$ or $-1$ according to whether $f$ lies in $\cK^p_+$ or $\cK^p_-${\,}. Moreover, if $x\in\{0,1,\ldots,p\}$ (i.e., $\varepsilon_{p+1}(x)=0$), then $\rho_a^{p+1}[f](x)=0$.
\end{lemma}

\begin{proof}
If $x\in\{0,1,\ldots,p\}$, then we have that $\rho_a^{p+1}[f](x)=0$ by \eqref{eq:LErrIn}, and then the inequalities hold trivially. Let us now assume that $x\notin\{0,1,\ldots,p\}$, which means that $\varepsilon_{p+1}(x)\neq 0$. Negating $f$ if necessary, we may assume that it lies in $\cK^p_+$. By \eqref{eq:LDDiv} we then have
$$
\varepsilon_{p+1}(x)\,\rho^{p+1}_a[f](x) ~=~ \varepsilon_{p+1}(x){\,}x^{\underline{p+1}}{\,}f[a,a+1,\ldots,a+p,a+x] ~\geq 0.
$$
Hence, using identities \eqref{eq:DivDiffRec7} and \eqref{eq:DelDD62} and Lemma~\ref{lemma:pCInc5}, we obtain
\begin{eqnarray*}
0 &\leq & \varepsilon_{p+1}(x)\,\rho^{p+1}_a[f](x)\\
&=& \varepsilon_{p+1}(x){\,}x^{\underline{p+1}}{\,}f[a,a+1,\ldots,a+p,a+x]\\
&=& \varepsilon_{p+1}(x){\,}(x-1)^{\underline{p}}{\,}(f[a+x,a+1,\ldots,a+p]-f[a,a+1,\ldots,a+p])\\
&\leq & \varepsilon_{p+1}(x)\,\tchoose{x-1}{p}{\,}(\Delta^pf(a+x)-\Delta^pf(a))\\
&\leq & \varepsilon_{p+1}(x)\,\tchoose{x-1}{p}{\,}(\Delta^pf(a+\lceil x\rceil)-\Delta^pf(a)),
\end{eqnarray*}
which proves the first two inequalities. The third one can be immediately proved using a telescoping sum.
\end{proof}


\chapter{Uniqueness and existence results}
\label{chapter:3}

In this chapter, we establish Theorems~\ref{thm:int2} and \ref{thm:int1} and show that, under the assumptions of these theorems, the sequence $n\mapsto f^p_n[g]$ converges uniformly on any bounded subset of $\R_+$. We also discuss the particular case where the sequence $n\mapsto g(n)$ is summable. Lastly, we provide historical notes on Krull-Webster's theory and some of its improvements.

Although their proofs are short and elementary, the main results given in this chapter are of utmost importance. They constitute the fundamental cornerstone of the whole theory developed in this book.

\section{Main results}
\label{sec:31M8r0}

We start this chapter by establishing a slightly improved version of our uniqueness Theorem~\ref{thm:int1}. We state this new version in Theorem~\ref{thm:unic} below and provide a very short proof. Let us first note that any solution $f\colon\R_+\to\R$ to the equation $\Delta f=g$ satisfies trivially the equations
\begin{eqnarray}
f(n) &=& f(1) +\sum_{k=1}^{n-1}g(k),\qquad n\in\N^*;\label{eq:fnEf1Sum}\\
f(x+n) &=& f(x) + \sum_{k=0}^{n-1}g(x+k),\qquad n\in\N.\label{eq:fxnEfxSum}
\end{eqnarray}
Moreover, using \eqref{eq:defFpn}, \eqref{eq:deflambdapt}, \eqref{eq:fnEf1Sum}, and \eqref{eq:fxnEfxSum}, we can easily derive the identity
\begin{equation}\label{eq:ff01flam}
f(x) ~=~ f(1) + f^p_n[g](x) + \rho^{p+1}_n[f](x),\qquad n\in\N^*.
\end{equation}
We also observe that the identity obtained by setting $p=0$ in \eqref{eq:ff01flam} can also be derived by subtracting \eqref{eq:fxnEfxSum} from \eqref{eq:fnEf1Sum}.

\begin{theorem}[Uniqueness]\label{thm:unic}\index{uniqueness theorem|textbf}
Let $p\in\N$ and $g\in\cD^p_{\S}$. Suppose that $f\colon\R_+\to\R$ is a solution to the equation $\Delta f=g$ that lies in $\cK^p$. Then, the following assertions hold.
\begin{enumerate}
\item[(a)] We have that $f\in\cR^{p+1}_{\S}$.
\item[(b)] For each $x>0$, the sequence $n\mapsto f^p_n[g](x)$ converges and we have
$$
f(x) ~=~ f(1) + \lim_{n\to\infty}f^p_n[g](x){\,}.
$$
\item[(c)] The sequence $n\mapsto f^p_n[g]$ converges uniformly on any bounded subset of\/ $\R_+$ to $f-f(1)$.
\end{enumerate}
\end{theorem}

\begin{proof}
We clearly have that $f\in\cD^{p+1}_{\S}$. Assertion (a) then follows from Lemma~\ref{lemma:VarEpsIneq} and the squeeze theorem. Assertion (b) follows from assertion (a) and identity \eqref{eq:ff01flam}. Now, let $E$ be any bounded subset of\/ $\R_+$. Using again identity \eqref{eq:ff01flam} and Lemma~\ref{lemma:VarEpsIneq}, for large integer $n$ we obtain
\begin{eqnarray*}
\sup_{x\in E}\left|f^p_n[g](x)-f(x)+f(1)\right| &=& \sup_{x\in E}\left|\rho^{p+1}_n[f](x)\right|\\
&\leq & \sup_{x\in E}\left|\tchoose{x-1}{p}\right|~\sum_{j=0}^{\lceil \sup E\rceil -1}\left|\Delta^{p+1}f(n+j)\right|.
\end{eqnarray*}
This establishes assertion (c).
\end{proof}

\begin{example}\label{ex:unicGa4}
Using Theorem~\ref{thm:unic} with $g(x)=\ln x$ and $p=1$, we obtain that all solutions $f\colon\R_+\to\R$ lying in $\cK^1$ to the equation $\Delta f(x)=\ln x$ are of the form $f(x)=c+\ln\Gamma(x)$, where $c\in\R$. We thus simply retrieve both Bohr-Mollerup's Theorem~\ref{thm:BM538Thm9}\index{Bohr-Mollerup theorem} and Gauss' limit \eqref{eq:GaussLimit42}, as expected. We also observe that the set $\cK^1$ cannot be replaced with $\cK^0$ in this characterization. For example, the function
$$
f(x) ~=~ \textstyle{\ln\Gamma (x)+\ln(1+\frac{1}{2}\sin(2\pi x))}
$$
is also a solution lying in $\cK^0$ to the equation $\Delta f(x)=\ln x$.
\end{example}

\begin{remark}
We note that the assumption that $\ln f$ is convex in Bohr-Mollerup's Theorem~\ref{thm:BM538Thm9} can be easily replaced with the fact that $\ln f$ lies in $\cK^1_+$ (without using the uniqueness Theorem~\ref{thm:unic}). Indeed, if $\ln f$ is convex on $[n,\infty)$ for some $n\in\N$, then using \eqref{eq:fxnEfxSum} we have that
$$
\ln f(x) ~=~ \ln f(x+n)-\sum_{k=0}^{n-1}\ln(x+k),\qquad x>0,
$$
and hence $\ln f$ must be convex on $\R_+$ (as a finite sum of convex functions on $\R_+$). We can also replace $\cK^1_+$ with $\cK^1$; indeed, assuming that $\ln f$ lies in $\cK^1_-$, we would obtain that $\Delta\ln f(x)=\ln x$ lies in $\cK^0_-$ by Lemma~\ref{lemma:PrelKp}(b), a contradiction.
\end{remark}

\begin{remark}[A proof of Bohr-Mollerup's theorem]\index{Bohr-Mollerup theorem}
We have seen in Example~\ref{ex:unicGa4} how both Bohr-Mollerup's theorem and Gauss' limit can be retrieved using our results. Let us now examine our proof in a self-contained way, using the needed arguments only. Let $f\colon\R_+\to\R$ be an eventually convex solution to the equation $\Delta f(x)=\ln x$. The nature of this equation shows that it is actually enough to assume that $x>1$ to find the form of $f(x)$. For any $n\in\N^*$ and any $x>1$, we then have
$$
f(n) ~=~ f(1)+\sum_{k=1}^{n-1}\ln k\qquad\text{and}\qquad f(x+n) ~=~ f(x)+\sum_{k=0}^{n-1}\ln(x+k)
$$
and hence also the identity
$$
f(x) ~=~ f(1)+\left(\sum_{k=1}^{n-1}\ln k-\sum_{k=0}^{n-1}\ln(x+k)+x\ln n\right)+\rho_n(x),
$$
where
$$
\rho_n(x) ~=~ f(x+n)-f(n)-x\ln n.
$$
To conclude the proof, we only need to show that, for each $x>1$, the sequence $n\mapsto \rho_n(x)$ converges to zero. Let $n\in\N^*$ be so that $f$ is convex on $[n,\infty)$. Using the convexity of $f$ we then obtain the following two inequalities
\begin{eqnarray*}
f(n+1) & \leq & \textstyle{(1-\frac{1}{x}){\,}f(n)+\frac{1}{x}{\,}f(x+n)}{\,},\\
f(n+x) & \leq & \textstyle{\frac{1}{x}{\,}f(n+1)+(1-\frac{1}{x}){\,}f(x+n+1)}{\,}.
\end{eqnarray*}
Using these inequalities and the identity $f(n+1)-f(n)=\ln n$, we obtain
\begin{eqnarray*}
0 &\leq & \rho_n(x) ~=~ f(x+n)-f(n+1)-(x-1)\ln n\\
&\leq & (x-1){\,}(f(x+n+1)-f(x+n)-\ln n) ~=~ (x-1){\,}\textstyle{\ln(1+\frac{x}{n})}.
\end{eqnarray*}
The proof is now complete since the latter expression converges to zero as $n\to\infty$. This shows to which extent the proofs of Bohr-Mollerup's theorem and Gauss' limit can be short and elementary. Note that a variant of this proof can be derived from the proof of Webster's uniqueness theorem \cite[Theorem 3.1]{Web97b}.
\end{remark}

Now that we have established the uniqueness Theorem~\ref{thm:unic}, let us prepare the ground for the existence theorem. Using the definition of $\rho^p_a[g](x)$ given in \eqref{eq:deflambdapt}, we can easily derive the following two identities
\begin{eqnarray}
\rho_a^p[g](p) &=& \Delta^pg(a){\,};\label{eq:LamEqDel2}\\
\rho^p_a[g](x)-\rho^{p+1}_a[g](x) &=& \tchoose{x}{p}\,\rho^p_a[g](p){\,}.\label{eq:fng2}
\end{eqnarray}
These identities clearly show that the inclusions $\cR^p_{\S}\subset\cD^p_{\S}$ and $\cR^p_{\S}\subset\cR^{p+1}_{\S}$ hold for any $p\in\N$. We will see in Proposition~\ref{prop:4042RsDs5} that these inclusions are actually strict.

Now, the following straightforward identities will also be useful as we continue
\begin{eqnarray}
f^p_{n+1}[g](x)-f^p_n[g](x) &=& -\rho^{p+1}_n[g](x){\,};\label{eq:fng}\\
f^p_n[g](x+1)-f^p_n[g](x) &=& g(x) - \rho^p_n[g](x){\,}.\label{eq:fng0}
\end{eqnarray}
For any integers $1\leq m\leq n$, from \eqref{eq:fng} we obtain
\begin{equation}\label{eq:fngsum}
f^p_n[g](x) ~=~ f^p_m[g](x) -\sum_{k=m}^{n-1}\rho_k^{p+1}[g](x){\,},
\end{equation}
which shows that, for any $x>0$, the convergence of the sequence $n\mapsto f^p_n[g](x)$ is equivalent to the summability of the sequence $n\mapsto \rho^{p+1}_n[g](x)$.

We now establish a slightly improved version of our existence Theorem~\ref{thm:int2}. We first present a technical lemma, which follows straightforwardly from Lemma \ref{lemma:VarEpsIneq}.

\begin{lemma}\label{lemma:VarEpsIneqSum}
Let $p\in\N$, $g\in\cK^p$, and $m\in\N^*$ be so that $g$ is $p$-convex or $p$-concave on $[m,\infty)$. Then, for any $x\geq 0$ and any integer $n\geq m$, we have
$$
\left|\sum_{k=m}^{n-1}\rho^{p+1}_k[g](x)\right| ~\leq ~ \left|\tchoose{x-1}{p}\right|\sum_{j=0}^{\lceil x\rceil -1}|\Delta^pg(n+j)-\Delta^pg(m+j)|.
$$
\end{lemma}

\begin{proof}
For any fixed $x\geq 0$, the sequence $k\mapsto\rho^{p+1}_k[g](x)$ for $k\geq m$ does not change in sign by Lemma~\ref{lemma:VarEpsIneq} and hence we have
$$
\left|\sum_{k=m}^{n-1}\rho^{p+1}_k[g](x)\right| ~=~ \sum_{k=m}^{n-1}\left|\rho^{p+1}_k[g](x)\right| ~\leq ~ \left|\tchoose{x-1}{p}\right|\sum_{j=0}^{\lceil x\rceil -1}\left|\sum_{k=m}^{n-1}\Delta^{p+1}g(k+j)\right|,
$$
where the inner sum clearly telescopes to $\Delta^pg(n+j)-\Delta^pg(m+j)$.
\end{proof}

\begin{theorem}[Existence]\label{thm:exist}\index{existence theorem|textbf}
Let $p\in\N$ and $g\in\cD^p_{\S}\cap\cK^p$. The following assertions hold.
\begin{enumerate}
\item[(a)] We have that $g\in\cR^p_{\S}${\,}.
\item[(b)] The sequence $n\mapsto f^p_n[g](x)$ converges for every $x>0$, and the function $f\colon\R_+\to\R$ defined by
$$
f(x) ~=~ \lim_{n\to\infty} f^p_n[g](x){\,},\qquad x>0,
$$
is a solution to the equation $\Delta f=g$ that is $p$-concave (resp.\ $p$-convex) on any unbounded subinterval $I$ of\/ $\R_+$ on which $g$ is $p$-convex (resp.\ $p$-concave). Moreover, we have $f(1)=0$ and
$$
|f^p_n[g](x)-f(x)| ~\leq ~ \lceil x\rceil\left|\tchoose{x-1}{p}\right|\left|\Delta^pg(n)\right|,\qquad x>0,~n\in I\cap\N^*.
$$
If $p\geq 1$, we also have the following tighter inequality
$$
|f^p_n[g](x)-f(x)| ~\leq ~ \left|\tchoose{x-1}{p}\right|\left|\Delta^{p-1}g(n+x)-\Delta^{p-1}g(n)\right|,\quad x>0,~n\in I\cap\N^*.
$$
\item[(c)] The sequence $n\mapsto f^p_n[g]$ converges uniformly on any bounded subset of\/ $\R_+$ to $f$.
\end{enumerate}
\end{theorem}

\begin{proof}
We have that $g\in\cD^p_{\S}\subset\cD^{p+1}_{\S}$. By Lemma~\ref{lemma:VarEpsIneq}, it follows immediately that $g$ lies in $\cR^{p+1}_{\S}$, and hence also in $\cR^p_{\S}$ by \eqref{eq:LamEqDel2} and \eqref{eq:fng2}. This establishes assertion (a). Now, suppose for instance that $g$ lies in $\cK^p_+$. Let $I$ be any unbounded subinterval of $\R_+$ on which $g$ is $p$-convex and let $m\in I\cap\N^*$. For any $x>0$, the sequence $k\mapsto\rho^{p+1}_k[g](x)$ for $k\geq m$ does not change in sign by Lemma~\ref{lemma:VarEpsIneq}. Thus, since $g$ lies in $\cD^p_{\N}${\,}, for any $x>0$ the series
$$
\sum_{k=m}^{\infty}\rho^{p+1}_k[g](x)
$$
converges by Lemma~\ref{lemma:VarEpsIneqSum}. By \eqref{eq:fngsum} it follows that the sequence $n\mapsto f^p_n[g](x)$ converges. Denoting the limiting function by $f$, we necessarily have $f(1)=0$. Moreover, by \eqref{eq:fng0} and assertion (a) we must have $\Delta f=g$.

Since $g$ is $p$-convex on $I$, for every $n\in\N^*$ the function $f^p_n[g]$ is clearly $p$-concave on $I$. (Note that the second sum in \eqref{eq:defFpn} is a polynomial of degree less than or equal to $p$ in $x$, hence by Proposition~\ref{prop:convCones48} it is both $p$-convex and $p$-concave.) Since $f$ is a pointwise limit of functions $p$-concave on $I$, it too is $p$-concave on $I$.

The claimed inequalities then follow from identity~\eqref{eq:ff01flam}, Lemma~\ref{lemma:VarEpsIneq}, and the observation that the restriction of the sequence $n\mapsto\Delta^pg(n)$ to $I\cap\N^*$ increases to zero by Lemma~\ref{lemma:pCInc5}. Indeed, for any $x>0$ and any $n\in I\cap\N^*$, we then have
\begin{eqnarray*}
|f^p_n[g](x)-f(x)| &=& |\rho^{p+1}_n[f](x)| ~\leq ~ \left|\tchoose{x-1}{p}\right|\left|\Delta^pf(n+x)-\Delta^pf(n)\right|\\
& \leq & \left|\tchoose{x-1}{p}\right|\,\sum_{j=0}^{\lceil x\rceil -1}|\Delta^p g(j+n)| ~\leq ~ \lceil x\rceil\left|\tchoose{x-1}{p}\right|\left|\Delta^pg(n)\right|.
\end{eqnarray*}
This proves assertion (b). Assertion (c) immediately follows from the first inequality of assertion (b).
\end{proof}

\begin{remark}\label{rem:ConvCl82}
We have shown in Theorems~\ref{thm:unic} and \ref{thm:exist} that the sequence $n\mapsto f_n^p[g]$ converges uniformly on any bounded subset of $\R_+$. In fact, we have proved the slightly more general property that the sequence $n\mapsto\rho_n^{p+1}[f]$ converges uniformly on any bounded subset of $[0,\infty)$ to $0$.
\end{remark}

Theorems~\ref{thm:unic} and \ref{thm:exist} show that the assumption that $g\in\cD^p_{\S}\cap\cK^p$ constitutes a sufficient condition to ensure both the uniqueness (up to an additive constant) and existence of solutions to the equation $\Delta f=g$ that lie in $\cK^p$. Nevertheless, we can show that this condition is actually not quite necessary. We discuss and elaborate on this natural question in Appendix~\ref{chapter:B-KW562}.

We now present an important property of the sequence $n\mapsto f^p_n[g]$. Considering the straightforward identity
$$
f^{p+1}_n[g](x)-f^p_n[g](x) ~=~ \tchoose{x}{p+1}\,\Delta^pg(n),
$$
we immediately see that if the sequence
$$
n ~\mapsto ~f^{p+1}_n[g](x)-f^p_n[g](x)
$$
approaches zero for some $x\in\R_+\setminus\{0,1,\ldots,p\}$, then $g$ must lie in $\cD^p_{\N}$. More importantly, the identity above also shows that if $g$ lies in $\cD^p_{\N}$ and if the sequence $n\mapsto f^p_n[g](x)$ converges, then so does the sequence $n\mapsto f^{p+1}_n[g](x)$ and both sequences converge to the same limit. Since the inclusion $\cD^p_{\N}\subset\cD^{p+1}_{\N}$ holds for any $p\in\N$, we immediately obtain the following important proposition.

\begin{proposition}\label{prop:ThetaFpFq}
Let $p\in\N$. If $g\in\cD^p_{\N}$ and if the sequence $n\mapsto f^p_n[g](x)$ converges, then for any integer $q\geq p$ the sequence
$$
n ~\mapsto ~|f^p_n[g](x)-f^q_n[g](x)|
$$
converges to zero. Moreover, the convergence is uniform on any bounded subset of\/ $\R_+$.
\end{proposition}

Let us end this section with the following observation about our uniqueness and existence results. In Theorem~\ref{thm:unic}, we have proved the uniqueness of the solution $f$ that lies in $\cK^p$ by first proving that this solution necessarily lies in $\cR^{p+1}_{\S}$. Although this latter asymptotic condition may seem a bit less natural than the assumption that $f$ lies in $\cK^p$, we could as well consider it as a sufficient condition to guarantee uniqueness. A similar observation can be made for the existence Theorem~\ref{thm:exist}. We can therefore establish the following two alternative results.

\begin{proposition}[Uniqueness]\label{prop:90unic41}\index{uniqueness theorem!alternative forms}
Let $p\in\N$ and let $g\colon\R_+\to\R$. Suppose that $f\colon\R_+\to\R$ is a solution to the equation $\Delta f=g$ that lies in $\cR^{p+1}_{\N}$. Then assertion (b) of Theorem~\ref{thm:unic} holds, and hence $f$ is unique (up to an additive constant).
\end{proposition}

\begin{proof}
This follows immediately from identity \eqref{eq:ff01flam}.
\end{proof}

\begin{proposition}[Existence]\index{existence theorem!alternative form}
Let $p\in\N$ and suppose that the function $g\colon\R_+\to\R$ lies in $\cD^p_{\N}$ and has the property that, for each $x>0$, the sequence $n\mapsto\rho^{p+1}_n[g](x)$ is summable. Then $g$ lies in $\cR_{\N}^p$ and there exists a unique (up to an additive constant) solution $f\colon\R_+\to\R$ to the equation $\Delta f=g$ that lies in $\cR^{p+1}_{\N}$.
\end{proposition}

\begin{proof}
Since the sequence $n\mapsto\rho^{p+1}_n[g](x)$ is summable, by \eqref{eq:fngsum} the sequence $n\mapsto f^p_n[g](x)$ converges. Denoting the limiting function by $f$, we necessarily have $f(1)=0$. By \eqref{eq:fng}, the function $g$ necessarily lies in $\cR_{\N}^{p+1}$, and hence also in $\cR_{\N}^p$ by \eqref{eq:LamEqDel2} and \eqref{eq:fng2}. Thus, we must have $\Delta f=g$ by \eqref{eq:fng0} and $f$ lies in $\cR^{p+1}_{\N}$ by \eqref{eq:ff01flam}.
\end{proof}

\begin{example}\label{ex:90unic41g}
Let us apply Proposition~\ref{prop:90unic41} to $g(x)=\ln x$ and $p=1$. We then obtain the following alternative characterization of the gamma function (in the multiplicative notation).
\begin{quote}
\emph{If $f\colon\R_+\to\R_+$ is a solution to the equation $f(x+1)=x{\,}f(x)$ having the asymptotic property that, for each $x>0$,
$$
f(x+n) ~\sim ~ n^x f(n)\qquad\text{as $n\to_{\N}\infty$},
$$
then $f(x)=c{\,}\Gamma(x)$ for some $c>0$.}
\end{quote}
It is easy to see that this characterization also holds on the whole complex domain of the gamma function, namely $\mathbb{C}\setminus (-\N)$.
\end{example}

\section{The case when the sequence $g(n)$ is summable}

Let $\cD^{-1}_{\N}$\label{p:D-1} be the set of functions $g\colon\R_+\to\R$ having the asymptotic property that the series $\sum_{k=1}^{\infty}g(k)$ converges. We immediately observe that $\cD^{-1}_{\N}\subset\cD^0_{\N}$. In this context, our uniqueness and existence results can be complemented by the following two theorems.

\begin{theorem}[Uniqueness]\label{thm:uniqzz0}\index{uniqueness theorem!when $g(n)$ is summable}
Let $g\in\cD^{-1}_{\N}$ and suppose that $f\colon\R_+\to\R$ is a solution to the equation $\Delta f=g$ that lies in $\cK^0$. Then, the following assertions hold.
\begin{enumerate}
\item[(a)] $f(x)$ has a finite limit as $x\to\infty$, denote it by $f(\infty)$.
\item[(b)] For each $x>0$, the series $\sum_{k=0}^{\infty}g(x+k)$ converges and we have
$$
f(x) ~=~ f(\infty) - \sum_{k=0}^{\infty}g(x+k){\,}.
$$
\item[(c)] The series $x\mapsto\sum_{k=0}^{\infty}g(x+k)$ converges uniformly on $\R_+$ to $f(\infty)-f$.
\end{enumerate}
\end{theorem}

\begin{proof}
The sequence $n\mapsto f(n)$ converges by \eqref{eq:fnEf1Sum}. Assuming for instance that $f$ lies in $\cK^0_+$, for any $x>0$ we obtain
$$
f(\lfloor x\rfloor +n) ~\leq ~ f(x+n) ~\leq ~ f(\lceil x\rceil +n)\quad\text{for large integer $n$}.
$$
Letting $n\to_{\N}\infty$ in these inequalities and using the squeeze theorem, we get assertion (a). Assertion (b) follows from assertion (a) and identity \eqref{eq:fxnEfxSum}. Now, for large integer $n$, by assertion (b) and identity \eqref{eq:fxnEfxSum} we have
$$
\sup_{x\in\R_+}\left|\sum_{k=n}^{\infty}g(x+k)\right| ~=~ \sup_{x\in\R_+}\left|f(x+n)-f(\infty)\right| ~\leq ~ |f(n)-f(\infty)|.
$$
This proves assertion (c).
\end{proof}

\begin{theorem}[Existence]\label{thm:existzz0}\index{existence theorem!when $g(n)$ is summable}
Let $g\in\cD^{-1}_{\N}\cap\cK^0$. The following assertions hold.
\begin{enumerate}
\item[(a)] We have that $g\in\cR^0_{\N}$.
\item[(b)] The series $\sum_{k=0}^{\infty}g(x+k)$ converges for every $x>0$, and the function $f\colon\R_+\to\R$ defined by
\begin{equation}\label{eq:3F3Sezzr72}
f(x) ~=~ -\sum_{k=0}^{\infty}g(x+k){\,},\qquad x>0,
\end{equation}
is a solution to the equation $\Delta f=g$ that is decreasing (resp.\ increasing) on any unbounded subinterval $I$ of\/ $\R_+$ on which $g$ is increasing (resp.\ decreasing). Moreover, we have $f(x)\to 0$ as $x\to\infty$ and, for every $n\in I\cap\N^*$,
$$
\left|\sum_{k=n}^{\infty}g(x+k)\right| ~=~ |f(x+n)| ~\leq ~ |f(n)|,\qquad x>0.
$$
\item[(c)] The series $x\mapsto\sum_{k=0}^{\infty}g(x+k)$ converges uniformly on $\R_+$ to $-f$.
\end{enumerate}
\end{theorem}

\begin{proof}
By Theorem~\ref{thm:exist}, assertion (a) clearly holds (since $g$ also lies in $\cD^0_{\N}$) and, for each $x>0$, the series \eqref{eq:3F3Sezzr72} converges and is a solution to the equation $\Delta f=g$ that satisfies the claimed monotonicity properties. Theorem~\ref{thm:uniqzz0} then shows that the function $f$ vanishes at infinity. The rest of assertion (b) follows from \eqref{eq:fxnEfxSum}. Assertion (c) is then immediate.
\end{proof}

Theorems~\ref{thm:uniqzz0} and \ref{thm:existzz0} motivate the following definition.

\begin{definition}\label{de:D-15R23}
For any $\S\in\{\N,\R\}$, we let $\cD^{-1}_{\S}$\label{p:D-1R} denote the set of functions $g\colon\R_+\to\R$ having the asymptotic property that, for each $x\in\S$, the series
$$
\sum_{k=0}^{\infty}g(x+k)
$$
converges and tends to zero as $x\to_{\S}\infty$.
\end{definition}

Clearly, this definition is consistent with our prior definition of $\cD^{-1}_{\N}$ and we can immediately see that the inclusion $\cD^{-1}_{\R}\subset\cD^{-1}_{\N}$ holds. Moreover, by Theorem~\ref{thm:existzz0} we have that
\begin{equation}\label{eq:d93n621ff}
\cD^{-1}_{\R}\cap\cK^0 ~=~ \cD^{-1}_{\N}\cap\cK^0.
\end{equation}

\begin{example}[The trigamma function]\index{trigamma function}
The trigamma function $\psi_1$ is defined on $\R_+$ as the derivative $\psi'$ of the digamma function.\index{digamma function} Hence, it has the property that
$$
\Delta\psi_1(x) ~=~ D\Delta\psi(x) ~=~ -1/x^2\qquad\text{for all $x>0$.}
$$
Since the function $\psi$ lies in $\cD^1_{\N}\cap\cK^1_-$, one can show (see Proposition~\ref{prop:LMpGpLMp1} in the next chapter) that $\psi_1$ lies in $\cD^0_{\N}\cap\cK^0_-$. Now, the function $g(x)=-1/x^2$ clearly lies in $\cD^{-1}_{\N}\cap\cK^0_+$ and hence also in $\cD^0_{\N}\cap\cK^0_+$. It also lies in $\cD^{-1}_{\R}\cap\cK^0_+$ by \eqref{eq:d93n621ff}. Thus, by Theorems~\ref{thm:exist}, \ref{thm:uniqzz0}, and \ref{thm:existzz0}, we see that the trigamma function $\psi_1$ is the unique decreasing solution $f$ to the equation $\Delta f=g$ that vanishes at infinity. Moreover, we have that
$$
\psi_1(x) ~=~ \sum_{k=0}^{\infty}\frac{1}{(x+k)^2}\qquad\text{and}\qquad \psi_1(1) ~=~ \sum_{k=1}^{\infty}\frac{1}{k^2} ~=~ \frac{\pi^2}{6}{\,}.
$$
Furthermore, the sequence of functions
$$
n ~\mapsto ~ \sum_{k=0}^{n-1}\frac{1}{(x+k)^2} ~=~ \psi_1(x)-\psi_1(x+n)
$$
converges uniformly on $\R_+$ to the function $\psi_1(x)$, and Theorem~\ref{thm:existzz0} provides the following inequalities
$$
0 ~ \leq ~ \psi_1(x+n) ~=~ \sum_{k=n}^{\infty}\frac{1}{(x+k)^2} ~\leq ~ \psi_1(n){\,},\qquad x>0,~n\in\N^*.
$$
Finally, Theorem~\ref{thm:exist} provides the following additional inequalities
$$
0 ~ \leq ~ \psi_1(n)-\psi_1(x+n) ~\leq ~ \frac{\lceil x\rceil}{n^2}{\,},\qquad x>0,~n\in\N^*.
$$
We will further investigate the trigamma function $\psi_1$ as a special polygamma function\index{polygamma functions} in Section~\ref{sec:Polyg&82}.
\end{example}

\section{Historical notes}

As mentioned in Chapter~\ref{chapter:1}, the uniqueness and existence result in the case when $p=1$ was established in the pioneering work of Krull~\cite{Kru48,Kru49} and then independently by Webster~\cite{Web97a,Web97b} as a generalization of Bohr-Mollerup's theorem\index{Bohr-Mollerup theorem}. We observe that it was also partially rediscovered by Dinghas~\cite{Din59}. In addition, we note that Krull's result was presented and somewhat revisited by Kuczma \cite{Kuc58} (see also Kuczma~\cite{Kuc65} and Kuczma~\cite[pp.\ 114--118]{Kuc68}) as well as by Anastassiadis \cite[pp.\ 69--73]{Ana64}. To our knowledge, the only attempts to establish uniqueness and existence results for any value of $p$ were made by Kuczma~\cite[pp.\ 118--121]{Kuc68} and Ardjomande~\cite{Ard68}. Independently of these latter results, an investigation of the special case when $p=2$, illustrated by the Barnes $G$-function,\index{Barnes's $G$-function} was made by Rassias and Trif~\cite{RasTri07} (see our Appendix~\ref{chapter:A-KW561}).

We also observe that Gronau and Matkowski~\cite{GroMat93,GroMat94} improved the multiplicative version of Krull's result by replacing the log-convexity property with the much weaker condition of geometrical convexity (see also Guan~\cite{Gua15} for a recent application of this result), thus providing another characterization of the gamma function, which was later improved by Alzer and Matkowski~\cite{AlzMat13} and Matkowski \cite{Mat15,Mat18}. (For further characterizations of the gamma function and generalizations, see also Anastassiadis~\cite{Ana64} and Muldoon~\cite{Mul78}.)

Many other variants and improvements of Krull's result can actually be found in the literature. For instance, Anastassiadis~\cite{Ana61} (see also Anastassiadis \cite[p.~71]{Ana64}) generalized it by modifying the asymptotic condition. Rohde~\cite{Roh65} also generalized it by modifying the convexity property. Gronau~\cite{Gro04} proposed a variant of Krull's result and applied it to characterize the Euler beta and gamma functions and study certain spirals (see also Gronau~\cite{Gro04b}). Merkle and Ribeiro Merkle \cite{MerRib11} proposed to combine Krull's approach with differentiation techniques to characterize the Barnes $G$-function. Himmel and Matkowski~\cite{HimMat18} also proposed improvements of Krull's result to characterize the beta and gamma functions.

\chapter{Interpretations of the asymptotic conditions}
\label{chapter:4}

In this chapter, we study some important properties of the sets $\cR^p_{\S}$ and $\cD^p_{\S}$ and provide interpretations of the asymptotic condition that defines the set $\cR^p_{\S}$.

We also investigate the sets $\cR^p_{\S}\cap\cK^p$ and $\cD^p_{\S}\cap\cK^p$ and show that they actually coincide and are independent of $\S$ (and hence we can remove this subscript). We also provide an interpretation of this common set $\cD^p\cap\cK^p$ and present some of its properties that will be very useful in the next chapters. In particular, we show that the intersection set $\cC^p\cap\cD^p\cap\cK^p$ is precisely the set of functions $g\in\cC^p$ for which $g^{(p)}$ eventually increases or decreases to zero (see Theorem~\ref{thm:intCpTpKp}).

\section{Some properties of the sets $\cR^p_{\S}$ and $\cD^p_{\S}$}
\label{sec:Ac1}


Although the definition of the set $\cR^p_{\S}$ seems rather technical (see Definition~\ref{de:R0Sp42}), the following proposition shows that this set can be nicely characterized in terms of interpolating polynomials.\index{interpolating polynomial} We omit the proof for it follows immediately from \eqref{eq:LErrIn} and \eqref{eq:LDDiv}. 

\begin{proposition}\label{prop:LamIntPol}
Let $p\in\N$. A function $g\colon\R_+\to\R$ lies in $\cR^p_{\S}$ if and only if for each $x>0$ such that $x^{\underline{p}}\neq 0$, we have that
$$
g[a,a+1,\ldots,a+p-1,a+x] ~\to ~0\qquad\text{as $a\to_{\S}\infty$}.
$$
When $\S=\R$ (resp.\ $\S=\N$), this latter condition means that $g$ asymptotically coincides with its interpolating polynomial whose nodes are any $p$ points equally spaced by $1$ (resp.\ any $p$ consecutive integers).
\end{proposition}

Interestingly, from \eqref{eq:fxnEfxSum} and \eqref{eq:ff01flam} we can also immediately derive the following alternative characterization of the set $\cR^p_{\N}$. For any function $f\colon\R_+\to\R$, we have
\begin{eqnarray*}
f\in\cR^0_{\N} & \Leftrightarrow & f(x) ~=~ -\sum_{k=0}^{\infty}\Delta f(x+k){\,},\qquad x>0{\,};\\
f\in\cR^{p+1}_{\N} & \Leftrightarrow & f(x) ~=~ f(1)+\lim_{n\to\infty}f^p_n[\Delta f](x){\,},\qquad x>0{\,}.
\end{eqnarray*}
(Note that we have already used these equivalences in the proofs of the uniqueness Theorems~\ref{thm:unic} and \ref{thm:uniqzz0} and Proposition~\ref{prop:90unic41}.)

We now present a proposition that reveals some interesting inclusions among the sets $\cR^p_{\S}$ and $\cD^p_{\S}$. In particular, it shows that just as the sets $\cD_{\S}^0, \cD_{\S}^1, \cD_{\S}^2, \ldots$ are increasingly nested, so are the sets $\cR_{\S}^0, \cR_{\S}^1, \cR_{\S}^2, \ldots$, and hence each of these families defines a filtration.

\begin{proposition}\label{prop:4042RsDs5}
For any $p\in\N$ and any $\S\in\{\N,\R\}$, the sets $\cR^p_{\S}$ and $\cD^p_{\S}$ are real linear spaces that satisfy the identity
\begin{equation}\label{eq:LpLp1Tp}
\cR^p_{\S} ~=~ \cR^{p+1}_{\S}\cap\cD^p_{\S}
\end{equation}
and the strict inclusions
$$
\cR^p_{\S} ~\varsubsetneq ~\cR^{p+1}_{\S}\qquad\text{and}\qquad\cD^p_{\S} ~\varsubsetneq ~\cD^{p+1}_{\S}.
$$
When $p\geq 1$ we also have
$$
\cR^p_{\S} ~\varsubsetneq ~\cD^p_{\S}.
$$
Finally, when $p=0$ we have
$$
\cD^0_{\R} ~=~ \cR^0_{\R} ~\varsubsetneq ~\cR^0_{\N} ~\varsubsetneq ~\cD^0_{\N}{\,}.
$$
\end{proposition}

\begin{proof}
It is clear that the sets $\cR^p_{\S}$ and $\cD^p_{\S}$ are closed under linear combinations; hence they are real linear spaces. Identity \eqref{eq:LpLp1Tp} then follows immediately from \eqref{eq:LamEqDel2} and \eqref{eq:fng2}. This identity also shows that $\cR^p_{\S}\subset\cR^{p+1}_{\S}$. As already observed, we also have $\cD^p_{\S}\subset\cD^{p+1}_{\S}$ trivially. Now, identity \eqref{eq:LErrIn} shows that the polynomial function $x\mapsto x^p$ lies in $\cR^{p+1}_{\S}\setminus\cR^p_{\S}$ and we can easily see that it lies also in $\cD^{p+1}_{\S}\setminus\cD^p_{\S}$. The inclusion $\cR^p_{\S}\subset\cD^p_{\S}$ follows from \eqref{eq:LpLp1Tp} and we can easily see that the $1$-periodic function $x\mapsto\sin(2\pi x)$ lies in  $\cD^p_{\S}\setminus\cR^p_{\S}$ for any $p\in\N^*$ as well as in $\cD^0_{\N}\setminus\cR^0_{\N}$. Finally, let us now show that $\cR^0_{\R} \varsubsetneq \cR^0_{\N}$. Using bump functions for instance, we can easily construct a smooth function $f\colon\R_+\to\R$ such that for any $n\in\N^*$, we have $f=0$ on the interval $[n-1,n-\frac{1}{n}]$ and $f(n-\frac{1}{2n})=1$. Such a function clearly lies in $\cR^0_{\N}$. However, it does not vanish at infinity, i.e., it does not lie in $\cR^0_{\R}$.
\end{proof}


We now present an important result that will be used repeatedly as we continue. It actually follows from the second of the following straightforward identities
\begin{eqnarray}
\rho_{a+1}^p[f](x) - \rho_a^p[f](x) &=& \rho_a^p[\Delta f](x){\,},\label{eq:DeltaXLam00}\\
\rho_{a}^{p+1}[f](x+1) - \rho_a^{p+1}[f](x) &=& \rho_a^p[\Delta f](x){\,}.\label{eq:DeltaXLam}
\end{eqnarray}

\begin{proposition}\label{prop:fpDpj}
Let $j,p\in\N$ be such that $j\leq p$. The following assertions hold.
\begin{enumerate}
\item[(a)] If $f\in\cR^p_{\S}$, then $\Delta^jf\in\cR^{p-j}_{\S}$.
\item[(b)] $f\in\cD^p_{\S}$ if and only if $\Delta^jf\in\cD^{p-j}_{\S}$.
\end{enumerate}
\end{proposition}

\begin{proof}
If $f$ lies in $\cR^{p+1}_{\S}$, then $\Delta f$ lies in $\cR^p_{\S}$ by \eqref{eq:DeltaXLam}. On the other hand, it is clear that $f$ lies in $\cD^{p+1}_{\S}$ if and only if $\Delta f$ lies in $\cD^p_{\S}$.
\end{proof}

It is easy to see that a function $f\colon\R_+\to\R$ whose difference $\Delta f$ lies in $\cR^p_{\S}$ for some $p\in\N$ need not lie in $\cR^{p+1}_{\S}$. For instance, the function $f\colon\R_+\to\R$ defined by the equation $f(x) = \sin(2\pi x)$ for $x>0$ does not lie in $\cR^1_{\S}$ but its difference $\Delta f=0$ lies in $\cR^0_{\S}$. However, we will see in Corollary~\ref{cor:fpDpj1} that, if $f\in\cK^{p-1}$, then the implication in assertion (a) of Proposition~\ref{prop:fpDpj} becomes an equivalence.

\begin{remark}
In view of Proposition~\ref{prop:fpDpj}(b), it is natural to wonder whether there exists a set $\cD$ of functions from $\R_+$ to $\R$ having the property that $f\in\cD^0_{\S}$ if and only if $\Delta f\in\cD$. However, such a set does not exist. Indeed, identities \eqref{eq:fnEf1Sum} and \eqref{eq:fxnEfxSum} show that if $f$ lies in $\cD^0_{\S}$, then necessarily $\Delta f$ lies in $\cD^{-1}_{\S}$. Conversely, for any $g\in\cD^{-1}_{\S}$, there are infinitely many functions $f\colon\R_+\to\R$ that satisfy $\Delta f=g$ but that do not lie in $\cD^0_{\S}$.
\end{remark}

It is clear that, for any $p\in\N$, if both functions $h$ and $g-h$ lie in the space $\cR^p_{\S}$, then so does the function $g$. For instance, if $g\colon\R_+\to\R$ has the asymptotic property that
$$
g(x)-P(x) ~\to ~0\qquad\text{as $x\to_{\S}\infty$}
$$
for some polynomial $P$ of degree less than or equal to $p-1$, then $g$ must lie in $\cR^p_{\S}$. Indeed, $P$ clearly lies in $\cR^p_{\S}$ and we also have that $g-P$ lies in $\cR^0_{\S}$ (which is included in $\cR^p_{\S}$ by Proposition~\ref{prop:4042RsDs5}). Thus, the space $\cR^p_{\S}$ contains not only every polynomial of degree less than or equal to $p-1$ but also every function that behaves asymptotically like a polynomial of degree less than or equal to $p-1$. To give another illustration of the property above, we observe for instance that both functions $\ln x$ and $H_x-\ln x$ (the latter tends to Euler's constant\index{Euler's constant} $\gamma$ as $x\to_{\S}\infty$) lie in $\cR^1_{\R}$ and hence so does the function $H_x$, which means that, for each $a\geq 0$,
$$
H_{x+a}-H_x ~\to ~0\qquad\text{as $x\to\infty$}
$$
(which, a priori, is a not completely trivial result).

These examples illustrate the fact that the spaces
$$
\cR^{\infty}_{\S} ~=~ \bigcup_{p\geq 0}\cR^p_{\S}\label{p:Rinfty}\qquad\text{and}\qquad\cD^{\infty}_{\S}=\bigcup_{p\geq 0}\cD^p_{\S}\label{p:Dinfty}
$$
are very rich and contain a huge variety of functions, including not only all the functions that have polynomial behaviors at infinity as discussed above, and in particular all the rational functions, but also many other functions. We observe, however, that they do not contain any strictly increasing exponential function. For instance, if $g(x)=2^x$, then $\Delta^p g(x)=2^x$ for any $p\in\N$, and this function does not vanish at infinity. Actually, such exponential functions grow asymptotically much faster than polynomial functions and may remain eventually $p$-convex even after adding nonconstant $1$-periodic functions. For instance, both functions $2^x$ and $2^x+\sin(2\pi x)$ are eventually $p$-convex for any $p\in\N$.

\begin{remark}
Using \eqref{eq:deflambdapt} and \eqref{eq:LpLp1Tp}, we also obtain $\cR^p_{\S} = \cR^{\infty}_{\S}\cap\cD^p_{\S}$ for any $p\in\N$.
\end{remark}

\section{The intersection sets $\cR^p_{\S}\cap\cK^p$ and $\cD^p_{\S}\cap\cK^p$}

Let us now consider the set $\cK^p$ and its subsets $\cR^p_{\S}\cap\cK^p$ and $\cD^p_{\S}\cap\cK^p$. As these sets will be used repeatedly throughout this book, it is important to study their basic properties. In this section, we present a number of results about these sets that will be very useful in the subsequent chapters.

Let us first observe that the set $\cK^p$ is not a linear space. For instance, using Lemma~\ref{lemma:PrelKp} we can see that both functions
$$
f(x) ~=~ x^{p+1}+\sin x\qquad\text{and}\qquad g(x) ~=~ x^{p+1}
$$
lie in $\cK^p$ but $f-g$ does not. We also have that $\Delta f$ does not lie in $\cK^p$ (because $D^p\Delta f=\Delta D^pf$ does not lie in $\cK^0$), which shows that $\cK^p$ is not closed under the operator $\Delta$.

The following corollary shows that $\cK^p$ is actually the union of two convex cones. This result is an immediate consequence of Proposition~\ref{prop:convCones48}.

\begin{corollary}\label{cor:convCones48}
For any $p\in\N$, the sets $\cK^p_+$ and $\cK^p_-$ are convex cones. These cones are opposite in the sense that $f$ lies in $\cK^p_+$ if and only if $-f$ lies in $\cK^p_-$. Moreover, the intersection $\cK^p_+\cap\cK^p_-$ is the real linear space of all the real functions on $\R_+$ that are eventually polynomials of degree less than or equal to $p$.
\end{corollary}

It is now clear that $\cD^p_{\S}\cap\cK^p$ is also the union of two opposite convex cones that is not a linear space. For instance, both functions
$$
f(x) ~=~ 2\ln x+\frac{\sin x}{x^2}\qquad\text{and}\qquad g(x) ~=~ 2\ln x
$$
lie in $\cD^1_{\S}\cap\cK^1$ (use, e.g., Theorem~\ref{thm:intCpTpKp}(b) below) but $f-g$ does not.

Now, the following proposition shows that, just as the sets $\cC^0, \cC^1, \cC^2, \ldots$ are decreasingly nested, so are the sets $\cK^{-1}{\!}, \cK^0, \cK^1, \ldots$. Thus, this latter family defines a descending filtration and we can therefore introduce the intersection set
$$
\cK^{\infty} ~=~ \bigcap_{p\geq 0}\cK^p.\label{p:Kinfty}
$$

\begin{proposition}\label{prop:MpDescFiltr}
For any integer $p\geq -1$, we have $\cK^{p+1}\varsubsetneq\cK^p$.
\end{proposition}

\begin{proof}
Let $f$ lie in $\cK^{p+1}$ for some integer $p\geq -1$. Suppose for instance that $f$ lies in $\cK^{p+1}_+$ and let $I$ be an unbounded subinterval of $\R_+$ on which $f$ is $(p+1)$-convex. Let $\mathcal{I}_{p+2}$ denote the set of tuples of $I^{p+2}$ whose components are pairwise distinct. By Lemma~\ref{lemma:pCInc5}, it follows that the restriction of the map
$$
(z_0,\ldots,z_{p+1}) ~\mapsto ~f[z_0,\ldots,z_{p+1}]
$$
to $\mathcal{I}_{p+2}$ is increasing in each place. If $f$ does not lie in $\cK^p_-$, then there are $p+2$ points $x_0 < \cdots < x_{p+1}$ in $I$ such that $f[x_0,\ldots,x_{p+1}]>0$. But then, $f$ is $p$-convex on the interval $(x_{p+1},\infty)$, and hence it lies in $\cK^p_+$, which establishes the inclusion. To see that the inclusion is strict, using Lemma~\ref{lemma:PrelKp} we just observe that the function $f\colon\R_+\to\R$ defined by the equation
$$
f(x) ~=~ x^{p+1}+\sin x\qquad\text{for $x>0$}
$$
lies in $\cK^p\setminus\cK^{p+1}$.
\end{proof}

Interestingly, Proposition~\ref{prop:MpDescFiltr} shows that the assumption that $g$ lies in $\cK^p$, which occurs in many statements (e.g., in Theorem~\ref{thm:exist}), can be given equivalently by the condition that $g$ lies in $\cup_{q\geq p}\cK^q$.

We now present two useful propositions. The first one is very important: it shows that the sets $\cR^p_{\S}\cap\cK^p$ and $\cD^p_{\S}\cap\cK^p$ coincide and are actually independent of $\S$.

\begin{proposition}\label{prop:ClClIntg}
For any $p\in\N$, we have
$$
\cR^p_{\R}\cap\cK^p ~=~ \cD^p_{\R}\cap\cK^p ~=~ \cR^p_{\N}\cap\cK^p ~=~ \cD^p_{\N}\cap\cK^p.
$$
\end{proposition}

\begin{proof}
We already know that $\cR^p_{\S}\subset\cD^p_{\S}$ (cf.\ Proposition~\ref{prop:4042RsDs5}) and $\cD^p_{\R}\subset\cD^p_{\N}$. Moreover, we have that $\cD^p_{\S}\cap\cK^p\subset\cR^p_{\S}$ by Theorem~\ref{thm:exist}. It remains to show that $\cD^p_{\N}\cap\cK^p\subset\cD^p_{\R}$. Let $g$ lie in $\cD^p_{\N}\cap\cK^p$. Suppose for instance that $g$ lies in $\cK^p_+$ and let $a>0$ be so that $g$ is $p$-convex on $[a,\infty)$. By Lemma~\ref{lemma:pCInc5}, $\Delta^pg$ is increasing on $[a,\infty)$. Thus, for any $x\geq a+1$, we have
$$
\Delta^pg(\lfloor x\rfloor) ~\leq ~ \Delta^pg(x) ~\leq ~ \Delta^pg(\lceil x\rceil).
$$
Letting $x\to\infty$ and using the squeeze theorem, we obtain that $g$ lies in $\cD^p_{\R}$.
\end{proof}

\begin{proposition}\label{prop:4eqgv}
If $f\in\cK^p$ for some $p\in\N$, then the following assertions are equivalent:
$$
\begin{array}{rlcrlcrlcrl}
\text{\emph{(i)}} & f\in\cR^{p+1}_{\S}, & & \text{\emph{(ii)}} & f\in\cD^{p+1}_{\S}, & & \text{\emph{(iii)}} & \Delta f\in\cR^p_{\S}{\,}, & & \text{\emph{(iv)}} & \Delta f\in\cD^p_{\S}{\,}.
\end{array}
$$
\end{proposition}

\begin{proof}
By Proposition~\ref{prop:4042RsDs5}, we clearly have that (i) implies (ii) and that (iii) implies (iv). By Proposition~\ref{prop:fpDpj}, we also have that (i) implies (iii) and that (ii) implies (iv). Finally, by Theorem~\ref{thm:unic}, we have that (iv) implies (i).
\end{proof}

Combining Proposition~\ref{prop:fpDpj} with Propositions~\ref{prop:MpDescFiltr} and \ref{prop:4eqgv}, we immediately obtain the following corollary, which naturally complements Proposition~\ref{prop:fpDpj}.

\begin{corollary}\label{cor:fpDpj1}
Let $j,p\in\N$ be such that $j\leq p$. If $f\in\cK^{p-1}$, then we have $f\in\cR^p_{\S}$ if and only if $\Delta^jf\in\cR^{p-j}_{\S}$.
\end{corollary}

Due to Proposition~\ref{prop:ClClIntg}, we will henceforth write $\cD^p\cap\cK^p$ instead of $\cD^p_{\S}\cap\cK^p$. In view of \eqref{eq:d93n621ff}, we will also write $\cD^{-1}\cap\cK^0$ instead of $\cD^{-1}_{\S}\cap\cK^0$.

Since the set $\cD^p\cap\cK^p$ is clearly a central object of our theory (cf.\ our existence Theorem~\ref{thm:exist}), it is important to investigate its properties. In this respect, we have the following two propositions.

\begin{proposition}\label{prop:LMpGpLMp1D7}
Let $j,p\in\N$ be such that $j\leq p$. The following assertions hold.
\begin{enumerate}
\item[(a)] If $g\in\cK^p_+$, then $\Delta^j g\in\cK^{p-j}_+$. More precisely, for any unbounded open interval $I$ of\/ $\R_+$, if $g$ is $p$-convex on $I$, then $\Delta^j g$ is $(p-j)$-convex on $I$.
\item[(b)] If $g\in\cD^p\cap\cK^p_+$, then $\Delta^j g\in\cD^{p-j}\cap\cK^{p-j}_+$.
\end{enumerate}
\end{proposition}

\begin{proof}
This result immediately follows from Lemma~\ref{lemma:PrelKp}(b) and Proposition~\ref{prop:fpDpj}.
\end{proof}

\begin{proposition}\label{prop:LMpGpLMp1}
Let $j,p\in\N$ be such that $j\leq p$ and let $g\in\cC^j$. The following assertions hold.
\begin{enumerate}
\item[(a)] $g\in\cK^p_+$ if and only if $g^{(j)}\in\cK^{p-j}_+$. More precisely, for any unbounded open interval $I$ of\/ $\R_+$, we have that $g$ is $p$-convex on $I$ if and only if $g^{(j)}$ is $(p-j)$-convex on $I$.
\item[(b)] $g\in\cD^p\cap\cK^p_+$ if and only if $g^{(j)}\in\cD^{p-j}\cap\cK^{p-j}_+$.
\end{enumerate}
\end{proposition}

\begin{proof}
Assertion (a) follows from assertions (c) and (d) of Lemma~\ref{lemma:PrelKp}. To see that assertion (b) holds, it is enough to show that, for any $p\geq 1$, we have $g\in\cD^p\cap\cK^p_+$ if and only if $g'\in\cD^{p-1}\cap\cK^{p-1}_+$.

Suppose first that $g$ lies in $\cD^p\cap\cK^p_+$. Then $g'$ lies in $\cK^{p-1}_+$ by assertion (a). Let $x>1$ be so that $g$ is $p$-convex on $[x-1,\infty)$. Then $\Delta^{p-1}g'$ is increasing on $[x-1,\infty)$ by assertion (a) and Proposition~\ref{prop:LMpGpLMp1D7}(a). By the mean value theorem, there exist $\xi^1_x,\xi^2_x\in (0,1)$ such that
\begin{eqnarray*}
\Delta^pg(x-1) ~=~ \Delta^{p-1}g'(x-1+\xi^1_x) & \leq & \Delta^{p-1}g'(x)\\
& \leq & \Delta^{p-1}g'(x+\xi^2_x) ~=~ \Delta^pg(x).
\end{eqnarray*}
Letting $x\to\infty$, we see that $g'$ lies in $\cD^{p-1}_{\R}$ by the squeeze theorem.

Conversely, suppose that $g'$ lies in $\cD^{p-1}\cap\cK^{p-1}_+$. Then $g$ lies in $\cK^p_+$ by assertion (a). Let $x>0$ be so that $g'$ is $(p-1)$-convex on $[x,\infty)$ and let $t\in (x,x+1)$. Then $\Delta^{p-1}g'$ is increasing on $[x,\infty)$ by Proposition~\ref{prop:LMpGpLMp1D7}(a), and hence we have
$$
\Delta^{p-1}g'(x) ~\leq ~ \Delta^{p-1}g'(t) ~\leq ~\Delta^{p-1}g'(x+1).
$$
Integrating on $t\in(x,x+1)$, we obtain
$$
\Delta^{p-1}g'(x) ~\leq ~ \Delta^pg(x) ~\leq ~\Delta^{p-1}g'(x+1).
$$
Letting $x\to\infty$, we see that $g$ lies in $\cD^p_{\R}$.
\end{proof}

\begin{remark}\label{rem:SiNNKp14}
If a function $f\colon\R_+\to\R$ is such that $\Delta f$ lies in $\cK^p$ for some $p\in\N$, then $f$ need not lie in $\cK^{p+1}$, even if $\Delta f$ lies in $\cD^p\cap\cK^p$. For instance, the function $f\colon\R_+\to\R$ defined by the equation
$$
f(x) ~=~ \frac{1}{2^x}\left(1+\frac{1}{3}\sin x\right)\qquad\text{for $x>0$}
$$
lies in $\cK^0_-\setminus\cK^1$. Indeed, $2^xf'(x)$ is $2\pi$-periodic and negative while $2^xf''(x)$ is $2\pi$-periodic and change in sign from $x=\frac{\pi}{6}$ to $x=\pi$. However, the function $\Delta f$ lies in $\cD^0\cap\cK^0_+$ for $2^x\Delta f'(x)$ is $2\pi$-periodic and positive. This example shows that the implications in Proposition~\ref{prop:LMpGpLMp1D7} cannot be equivalences.
\end{remark}

If a function $g\colon\R_+\to\R$ lies in $\cD^p\cap\cK^p$ for some $p\in\N$, then by Proposition~\ref{prop:LMpGpLMp1D7} the function $\Delta^pg$ lies in $\cD^0\cap\cK^0$, i.e., $\Delta^pg$ eventually increases or decreases to zero. However, a function $g\colon\R_+\to\R$ that satisfies this latter property need not lie in $\cD^p\cap\cK^p$, unless $g$ lies in $\cK^p$ or $p=0$. The example introduced in Remark~\ref{rem:SiNNKp14} illustrates this phenomenon when $p=1$. On the other hand, when $g$ lies in $\cC^p$, by Proposition~\ref{prop:LMpGpLMp1} we have that $g$ lies in $\cD^p\cap\cK^p$ if and only if $g^{(p)}$ lies in $\cD^0\cap\cK^0$.

We gather these important observations in the following theorem.

\begin{theorem}\label{thm:intCpTpKp}
Let $p\in\N$. The following assertions hold.
\begin{enumerate}
  \item [(a)] Let $g\in\cK^p_+$ (resp.\ $\cK^p_-$). Then $g$ lies in $\cD^p\cap\cK^p_+$ (resp.\ $\cD^p\cap\cK^p_-$) if and only if $\Delta^pg$ eventually increases (resp.\ decreases) to zero.
  \item [(b)] Let $g\in\cC^p$. Then $g$ lies in $\cD^p\cap\cK^p_+$ (resp.\ $\cD^p\cap\cK^p_-$) if and only if $g^{(p)}$ eventually increases (resp.\ decreases) to zero.
\end{enumerate}
\end{theorem}

\begin{proof}
Assertion (a) immediately follows from Propositions~\ref{prop:fpDpj} and \ref{prop:LMpGpLMp1D7}. Assertion (b) immediately follows from Proposition~\ref{prop:LMpGpLMp1}.
\end{proof}

\begin{remark}\label{rem:I5mpPConv63}
It is not difficult to see that the function $g(x)=\frac{1}{x}\sin x^2$ vanishes at infinity while its derivative does not. Theorem~\ref{thm:intCpTpKp}(b) shows that if $g$ lies in $\cC^q\cap\cD^p\cap\cK^q$ for some $p,q\in\N$ such that $p\leq q$, then all the functions $g^{(p)},g^{(p+1)},\ldots,g^{(q)}$ vanish at infinity.
\end{remark}

Propositions~\ref{prop:LMpGpLMp1D7} and \ref{prop:LMpGpLMp1} do not provide any information on the functions $\Delta g$ and $g'$ when $g$ lies in $\cD^0\cap\cK^0$ and $\cC^1\cap\cD^0\cap\cK^0$, respectively. The following proposition fills this gap under the additional assumptions that $\Delta g$  and $g'$ lie in $\cK^0$, respectively.

\begin{proposition}\label{prop:sa6f575sf}
The following assertions hold.
\begin{enumerate}
\item[(a)] If $g$ lies in $\cD^0\cap\cK^0_-$ and is such that $\Delta g$ lies in $\cK^0$, then $\Delta g$ lies in $\cD^{-1}\cap\cK^0_+$.
\item[(b)] If $g$ lies in $\cC^1\cap\cD^0\cap\cK^0_-$ and is such that $g'$ lies in $\cK^0$ (or equivalently, $g$ lies in $\cK^1$), then $g'$ lies in $\cC^0\cap\cD^{-1}\cap\cK^0_+$.
\end{enumerate}
\end{proposition}

\begin{proof}
Let us first prove assertion (a). Since $g$ is eventually decreasing, $\Delta g$ must be eventually negative. But since $\Delta g$ also lies in $\cD^0\cap\cK^0$, it must be eventually increasing to zero. On the other hand, since $g$ lies in $\cD^0$, $\Delta g$ must lie in $\cD^{-1}_{\N}$. This proves assertion (a).

Let us now prove assertion (b). Since $g$ is eventually decreasing, $g'$ must be eventually negative. Since $g'$ lies in $\cK^0$ (and hence $g$ lies in $\cK^1$ by Lemma~\ref{lemma:PrelKp}), we have that $g$ lies in $\cD^1\cap\cK^1$ (since $\cD^0_{\S}\subset\cD^1_{\S}$). Proposition~\ref{prop:LMpGpLMp1} then tells us that $g'$ lies in $\cD^0\cap\cK^0$, and hence it must be eventually increasing to zero.

It remains to show that $g'$ lies in $\cD^{-1}_{\N}$. Let $x>1$ be so that $g$ is decreasing and $g'$ is increasing on $I_x=[x-1,\infty)$. By the mean value theorem, for any integer $k\geq x$ there exist $\xi_k\in (0,1)$ such that
$$
\Delta g(k-1) ~=~ g'(k-1+\xi_k) ~\leq ~ g'(k).
$$
For any integers $m,n$ such that $x\leq m\leq n$, we then have
$$
g(n-1)-g(m-1) ~=~ \sum_{k=m}^{n-1}\Delta g(k-1) ~\leq ~\sum_{k=m}^{n-1}g'(k) ~\leq ~ 0.
$$
Letting $n\to_{\N}\infty$, we can see that $g'$ lies in $\cD^{-1}_{\N}$.
\end{proof}

\begin{remark}
The assumption that $\Delta g$ lies in $\cK^0$ cannot be ignored in Proposition~\ref{prop:sa6f575sf}(a). Indeed, take for instance the function $g=\Delta f$, where $f$ is the function defined in Remark~\ref{rem:SiNNKp14}. We have seen that this function lies in $\cD^0\cap\cK^0$. However, it is not difficult to see that $\Delta g$ does not lie in $\cK^0$. Similarly, the assumption that $g'$ lies in $\cK^0$ cannot be ignored in Proposition~\ref{prop:sa6f575sf}(b). Indeed, one can show that the same function $g$ has the property that $g'$ does not lie in $\cK^0$. To give another example, one can show that the function
$$
g(x) ~=~ \frac{1}{x^3}(x+\sin x)
$$
lies in $\cD^0\cap\cK^0$ whereas its derivative $g'$ does not lie in $\cK^0$.
\end{remark}

We also have the following two corollaries, in which the symbols $\cR$ and $\cD$ can be used interchangeably.

\begin{corollary}\label{cor:dfmm7saa5}
Let $g$ lie in $\cK^p_+$ (resp.\ $\cK^p_-$) for some $p\in\N$. Then $g$ lies in $\cD^p_{\S}$ if and only if there exists a solution $f\colon\R_+\to\R$ to the equation $\Delta f=g$ that lies in $\cD^{p+1}_{\S}\cap\cK^p_-$ (resp.\ $\cD^{p+1}_{\S}\cap\cK^p_+$).
\end{corollary}

\begin{proof}
The $\cD$-version immediately follows from Theorem~\ref{thm:exist} and Proposition~\ref{prop:fpDpj}(b). The $\cR$-version then follows from Proposition~\ref{prop:4eqgv} and Proposition~\ref{prop:fpDpj}(a).
\end{proof}

\begin{corollary}\label{cor:dfmm7s}
For any $p\in\N$, we have that
$$
\cD^p\cap\cK^p_+ ~\subset ~\cK^{p-1}_-\qquad\text{and}\qquad\cD^p\cap\cK^p_- ~\subset ~\cK^{p-1}_+.
$$
More precisely, if $g$ lies in $\cD^p\cap\cK^p$ and is $p$-convex (resp.\ $p$-concave) on an unbounded interval of\/ $\R_+$, then on this interval it is also $(p-1)$-concave (resp.\ $(p-1)$-convex).
\end{corollary}

\begin{proof}
Let $g$ lie in $\cD^p\cap\cK^p_+$. Then the function $f\colon\R_+\to\R$ defined in the existence Theorem~\ref{thm:exist} is $p$-concave on any unbounded subinterval of $\R_+$ on which $g$ is $p$-convex. By Lemma~\ref{lemma:PrelKp}(b), the function $g=\Delta f$ is also $(p-1)$-concave on this interval.
\end{proof}

We end this chapter by providing a characterization of the set $\cR^p\cap\cK^p=\cD^p\cap\cK^p$ in terms of interpolating polynomials.\index{interpolating polynomial} We also give a corollary that will be very useful in the subsequent chapters.

\begin{proposition}\label{prop:CLamBehAsyPol}
Let $g$ lie in $\cK^p$ for some $p\in\N$. Then we have that $g$ lies in $\cD^p_{\S}$ if and only if for any pairwise distinct $x_0,\ldots,x_p>0$, we have that
$$
g[a+x_0,\ldots,a+x_p]\to 0\qquad\text{as $a\to_{\S}\infty$}.
$$
This latter condition means that $g$ asymptotically coincides with its interpolating polynomial with any $p$ nodes.
\end{proposition}

\begin{proof}
(Necessity) Suppose for instance that $g$ lies in $\cD^p\cap\cK^p_+$. By Corollary~\ref{cor:dfmm7s}, it also lies in $\cK^{p-1}_-$. Let $x_0,\ldots,x_p>0$ be any pairwise distinct points and let $a>0$ be so that $g$ is $p$-convex and $(p-1)$-concave on $[a,\infty)$. Then the map
$$
x ~ \mapsto ~ g[x+x_0,\ldots,x+x_p]
$$
is nonpositive on $[a,\infty)$ and, by Lemma~\ref{lemma:pCInc5}, it is also increasing on $[a,\infty)$. By \eqref{eq:DelDD62}, we then have
$$
\frac{1}{p!}\,\Delta^pg(a) ~=~ g[a,a+1,\ldots,a+p] ~\leq ~ g[a+p+x_0,\ldots,a+p+x_p] ~\leq ~ 0,
$$
where the left side increases to zero as $a\to_{\S}\infty$.

(Sufficiency) This immediately follows from Propositions~\ref{prop:LamIntPol} and \ref{prop:ClClIntg}.
\end{proof}

%

\begin{corollary}\label{cor:Hom4}
Let $g$ lie in $\cK^p_+$ (resp.\ $\cK^p_-$) for some $p\in\N$, let $a>0$ and $b\geq 0$, and let $h\colon\R_+\to\R$ be defined by the equation $h(x)=g(ax+b)$ for $x>0$. Then
\begin{enumerate}
\item[(a)] $h$ lies in $\cK^p_+$ (resp.\ $\cK^p_-$);
\item[(b)] if $g$ lies in $\cD^p\cap\cK^p$, then $h$ lies in $\cD^p\cap\cK^p_+$ (resp.\ $\cD^p\cap\cK^p_-$).
\end{enumerate}
\end{corollary}

\begin{proof}
The result is trivial if $p=0$. So let us assume that $p\geq 1$ and for instance that $g$ is $p$-convex on $[s,\infty)$ for some $s>0$. Using \eqref{eq:DivDiffPai4Dis82}, we can easily show that for any pairwise distinct points $x_0,\ldots,x_p>0$ we have
$$
h[x_0,\ldots,x_p] ~=~ a^p{\,}g[ax_0+b,\ldots,ax_p+b].
$$
This immediately shows that $h$ is $p$-convex on $[\frac{1}{a}(s-b),\infty)$ and hence that assertion (a) holds. Now, suppose that $g$ lies in $\cD^p\cap\cK^p_+$. Then $h$ lies in $\cK^p_+$ by assertion (a). Moreover, for any pairwise distinct $x_0,\ldots,x_p>0$, by Proposition~\ref{prop:CLamBehAsyPol} we have that
$$
h[n+x_0,\ldots,n+x_p] ~=~ a^p{\,}g[an+ax_0+b,\ldots,an+ax_p+b] ~\to ~0
$$
as $n\to_{\N}\infty$. Hence $h$ also lies in $\cD^p\cap\cK^p_+$ by Proposition~\ref{prop:CLamBehAsyPol}. This establishes assertion (b).
\end{proof}

\chapter{Multiple $\log\Gamma$-type functions}
\label{chapter:5}

In this chapter, we introduce and investigate the map, denote it by $\Sigma$, that carries any function $g$ lying in
$$
\bigcup_{p\geq 0}(\cD^p\cap\cK^p)
$$
into the unique solution $f$ to the equation $\Delta f=g$ that arises from the existence Theorem~\ref{thm:exist}. We call these solutions \emph{multiple $\log\Gamma$-type functions} and we investigate certain of their properties. We also discuss the search for simple conditions on the function $g\colon\R_+\to\R$ to ensure the existence of $\Sigma g$. Further important properties of these functions, including counterparts of several classical properties of the gamma function, will be investigated in the next three chapters.

The map $\Sigma$ is actually a central concept of the theory developed here. Its definition and properties seem to show that it is as fundamental as the basic antiderivative operation. In the next chapter we show that both concepts actually share many common features.

\section{The map $\Sigma$ and its basic properties}

In this section, we introduce the map $\Sigma$ and discuss some of its basic properties. We begin with the following important definition.

\begin{definition}[Asymptotic degree]\label{de:de4g3f5}
The \emph{asymptotic degree}\index{asymptotic degree|textbf} of a function $f\colon\R_+\to\R$, denoted $\deg f$, is defined by the equation
$$
\deg f ~=~ -1+\min\{q\in\N : f\in\cD^q_{\R}\}.\label{p:degf}
$$
\end{definition}

For instance, if $f$ is a polynomial of degree $p$ for some $p\in\N$, then $\deg f=p$. If $f(x)=0$ or $f(x)=\frac{1}{x}$, or $f(x)=\ln(1+\frac{1}{x})$, then $\deg f=-1$. If $f(x)=\sin x$ or $f(x)=x+\sin x$, or $f(x)=2^x$, then $\deg f=\infty$.

It is easy to see that the identity
$$
\deg f ~=~ 1+\deg\Delta f
$$
holds whenever $\deg f$ is a nonnegative integer. However, it is no longer true when $\deg f=-1$. For instance, for the function $f(x)=0$ or the function $f(x)=\frac{1}{x}$, we have $\deg f=\deg\Delta f=-1$. This shows that in general we have
$$
(\deg f)_+ ~=~ 1+\deg\Delta f.
$$

We are now ready to introduce the map $\Sigma$. Here and throughout, the symbols $\mathrm{dom}(\Sigma)$\label{p:domS} and $\mathrm{ran}(\Sigma)$\label{p:ranS} denote the domain and range of $\Sigma$, respectively.

\begin{definition}[The map $\Sigma$]
We define the map $\Sigma\colon\mathrm{dom}(\Sigma)\to\mathrm{ran}(\Sigma)$, where
$$
\mathrm{dom}(\Sigma) ~=~ \bigcup_{p\geq 0}(\cD^p\cap\cK^p),
$$
by the following condition: if $g\in\cD^p\cap\cK^p$ for some $p\in\N$, then\label{p:Sigma}
\begin{equation}\label{eq:DefGStar}
\Sigma g ~=~ \lim_{n\to\infty}f^p_n[g].
\end{equation}
\end{definition}

It is important to note that the map is well defined; indeed, if $g$ lies in both sets $\cD^p\cap\cK^p$ and $\cD^q\cap\cK^q$ for some integers $0\leq p<q$, then by Proposition~\ref{prop:ThetaFpFq} both sequences $n\mapsto f^p_n[g]$ and $n\mapsto f^q_n[g]$ have the same limiting function. Thus, in view of Proposition~\ref{prop:MpDescFiltr}, we can see that condition \eqref{eq:DefGStar} holds for $p=1+\deg g$.

Thus defined, it is clear that the map $\Sigma$ is one-to-one; indeed, if $\Sigma g_1=\Sigma g_2$ for some functions $g_1$ and $g_2$ lying in $\mathrm{dom}(\Sigma)$, then $g_1=\Delta\Sigma g_1=\Delta\Sigma g_2=g_2$. This map is even a bijection since we have restricted its codomain to its range. We then have the following immediate result.

\begin{proposition}
The map $\Sigma$ is a bijection and its inverse is the restriction of the difference operator $\Delta$ to $\mathrm{ran}(\Sigma)$.
\end{proposition}

Just as the indefinite integral (or antiderivative) of a function $g$ is the class of functions whose derivative is $g$, the indefinite sum (or antidifference) of a function $g$ is the class of functions whose difference is $g$ (see, e.g., Graham {\em et al.} \cite[p.~48]{GraKnuPat94}). Recall also that any two indefinite integrals of a function differ by a constant while any two indefinite sums of a function differ by a $1$-periodic function. The map $\Sigma$ now enables one to refine the definition of an indefinite sum as follows.

\begin{definition}\label{de:PIS43}\index{principal indefinite sum|textbf}
We say that the \emph{principal indefinite sum}\index{principal indefinite sum|textbf} of a function $g$ lying in $\mathrm{dom}(\Sigma)$ is the class of functions $c+\Sigma g$, where $c\in\R$.
\end{definition}

\begin{example}[The log-gamma function]\label{ex:PIS43LOGG}
If $g(x)=\ln x$, then we have $\Sigma g(x)=\ln\Gamma(x)$, and we simply write
$$
\Sigma\ln x ~=~ \ln\Gamma(x),\qquad x>0.
$$
Thus, the principal indefinite sum of the function $x\mapsto \ln x$ is the class of functions $x\mapsto c+\ln\Gamma(x)$, where $c\in\R$. With some abuse of language, we can say that the principal indefinite sum of the log function is the log-gamma function.
\end{example}

Exactly as for the difference operator $\Delta$, we will sometimes add a subscript to the symbol $\Sigma$ to specify the variable on which the map $\Sigma$ acts. For instance, $\Sigma_x{\,}g(2x)$ stands for the function obtained by applying $\Sigma$ to the function $x\mapsto g(2x)$ while $\Sigma g(2x)$ stands for the value of the function $\Sigma g$ at $2x$.

The following proposition provides some straightforward properties of the map $\Sigma$ that will be very useful as we continue.

\begin{proposition}\label{prp:56GathZZPro8}
Let $g$ lie in $\cD^p\cap\cK^p$ for some $p\in\N$. The following assertions hold.
\begin{enumerate}
\item[(a)] $\Sigma g$ is the unique solution to the equation $\Delta f=g$ that lies in $\cK^p$ and that vanishes at $1$.
\item[(b)] $\Sigma g$ lies in $\cD^{p+1}\cap\cK^p=\cR^{p+1}\cap\cK^p$.
\item[(c)] $\Sigma g$ satisfies the identities
\begin{eqnarray}
\Sigma g(n) &=& \sum_{k=1}^{n-1}g(k){\,},\qquad n\in\N^*,\label{eq:RestrInt}\\
\Sigma g(x+n) &=& \Sigma g(x) + \sum_{k=0}^{n-1}g(x+k){\,},\qquad n\in\N,\label{eq:56zzSec32S6}
\end{eqnarray}
and
\begin{equation}\label{eq:33ConvSig52}
\Sigma g(x) ~=~ f^p_n[g](x) + \rho^{p+1}_n[\Sigma g](x){\,},\qquad n\in\N^*.
\end{equation}
\end{enumerate}
\end{proposition}

\begin{proof}
Assertions (a) and (b) immediately follow from Theorems~\ref{thm:unic} and \ref{thm:exist} and Proposition~\ref{prop:4eqgv}. Identities \eqref{eq:RestrInt}--\eqref{eq:33ConvSig52} follow from \eqref{eq:fnEf1Sum}--\eqref{eq:ff01flam}.
\end{proof}

Quite surprisingly, we observe that if $g$ lies in $\cD^p\cap\cK^p$ for some $p\in\N$, then $\Sigma g$ need not lie in $\cK^{p+1}$. The example given in Remark~\ref{rem:SiNNKp14} illustrates this observation.

We also have that
$$
\deg\Sigma g ~=~ 1+\deg g
$$
whenever $\deg\Sigma g$ is a nonnegative integer; but this property no longer holds if $\deg\Sigma g=-1$. For instance, considering the functions
$$
g(x) ~=~ \frac{2-x}{x(x+1)(x+2)}\qquad\text{and}\qquad\Sigma g(x) ~=~ \frac{x-1}{x(x+1)}{\,},
$$
we have $\deg g=\deg\Sigma g = -1$. Thus, in general we have
$$
(\deg\Sigma g)_+ ~=~ 1+\deg g.
$$

We now give two important propositions, which were essentially proved by Webster~\cite[Theorem 5.1]{Web97b} in the special case when $p=1$.

\begin{proposition}\label{prop:gStHa}
Let $g_1$ and $g_2$ lie in $\cD^p\cap\cK^p$ for some $p\in\N$ and let $c_1,c_2\in\R$. If $c_1g_1+c_2g_2$ lies in $\cD^p\cap\cK^p$, then
$$
\Sigma(c_1g_1+c_2g_2) ~=~ c_1\Sigma g_1+c_2\Sigma g_2.
$$
\end{proposition}

\begin{proof}
It is clear that if $g$ lies in $\cD^p\cap\cK^p$, then we have $\Sigma cg=c\Sigma g$ for any $c\in\R$. Now, suppose that $g_1$, $g_2$, and $g_1+g_2$ lie in $\cD^p\cap\cK^p$ and let us show that
$$
\Sigma(g_1+g_2) ~=~ \Sigma g_1+\Sigma g_2.
$$
It is actually enough to consider the following two cases.
\begin{itemize}
\item[1.] If both $g_1$ and $g_2$ lie in $\cD^p\cap\cK^p_+$ (resp.\ $\cD^p\cap\cK^p_-$), then so does $g_1+g_2$. It follows that the function $f=\Sigma g_1+\Sigma g_2$ is a solution to the equation $\Delta f=g_1+g_2$ that lies in $\cK^p_-$ (resp.\ $\cK^p_+$) and satisfies $f(1)=0$. By the uniqueness Theorem~\ref{thm:unic}, we must have $\Sigma (g_1+g_2)=f$.
\item[2.] If both $g_1+g_2$ and $-g_1$ lie in $\cD^p\cap\cK^p_+$ (resp.\ $\cD^p\cap\cK^p_-$), then so does $g_2$ (use the first case) and we have
$$
\Sigma g_2 ~=~ \Sigma ((g_1+g_2)+(-g_1)) ~=~ \Sigma (g_1+g_2) - \Sigma g_1.
$$
\end{itemize}
This completes the proof.
\end{proof}

\begin{proposition}\label{prop:gStHa223}
Let $g$ lie in $\cD^p\cap\cK^p_+$ (resp.\ $\cD^p\cap\cK^p_-$) for some $p\in\N$, let $a\geq 0$, and let $h\colon\R_+\to\R$ be defined by the equation $h(x)=g(x+a)$ for $x>0$. Then $h$ lies in $\cD^p\cap\cK^p_+$ (resp.\ $\cD^p\cap\cK^p_-$) and
$$
\Sigma h(x) ~=~ \Sigma_x{\,}g(x+a) ~=~ \Sigma g(x+a)-\Sigma g(a+1).
$$
\end{proposition}

\begin{proof}
Define a function $f\colon\R_+\to\R$ by the equation
$$
f(x) ~=~ \Sigma g(x+a)-\Sigma g(a+1)
$$
for $x>0$. By Corollary~\ref{cor:Hom4}, $f$ is a solution to the equation $\Delta f=h$ that lies in $\cK^p_-$ (resp.\ $\cK^p_+$) and satisfies $f(1)=0$. Hence, $\Sigma h=f$, as required.
\end{proof}

\begin{example}[{see Webster \cite{Web97b}}]
For any $a>0$, consider the function $g_a\colon\R_+\to\R$ defined by
$$
g_a(x) ~=~ \ln\frac{x}{x+a} ~=~ \ln x-\ln(x+a)\qquad\text{for $x>0$}.
$$
Then $g_a$ lies in $\cD^0\cap\cK^0_+$ (and also in $\cD^1\cap\cK^1_-$) and Propositions~\ref{prop:gStHa} and \ref{prop:gStHa223} show that
$$
\Sigma g_a(x) ~=~ \ln\frac{\Gamma(x)\Gamma(a+1)}{\Gamma(x+a)}{\,}.
$$
Also, since $g_a$ is concave on $\R_+$, we have that $\Sigma g_a$ is convex on $\R_+$. As Webster \cite[p.~615]{Web97b} observed, this is ``a not completely trivial result, but one immediate from the approach adopted here.''
\end{example}

\begin{example}[A rational function]\label{ex:5aRat7Fct8}
The function
$$
g(x) ~=~ \frac{x^4+1}{x^3+x} ~=~ x+\frac{1}{x}-\frac{2x}{x^2+1}
$$
clearly lies in $\cD^2\cap\cK^2$. Using Proposition~\ref{prop:gStHa}, we then have
$$
\Sigma g(x) ~=~ \tchoose{x}{2}+H_{x-1}-2\,\Sigma h(x),
$$
where the function
$$
h(x) ~=~ \frac{x}{x^2+1} ~=~ \Re\left(\frac{1}{x+i}\right)
$$
lies in $\cD^0\cap\cK^0$. Now, recalling that $\Sigma_x \frac{1}{x}=H_{x-1}$, it is not difficult to see that
$$
\Sigma h(x) ~=~ c+\Re H_{x+i-1}
$$
for some $c\in\R$, where the function $z\mapsto H_z$ on $\mathbb{C}\setminus(-\N^*)$ satisfies the identity
$$
H_z ~=~ \sum_{k=1}^{\infty}\left(\frac{1}{k}-\frac{1}{z+k}\right).
$$
Indeed, the function $f\colon\R_+\to\R$ defined by the equation
$$
f(x) ~=~ \Re H_{x+i-1} ~=~ \sum_{k=1}^{\infty}\left(\frac{1}{k}-\frac{x+k-1}{(x+k-1)^2+1}\right),\qquad x>0,
$$
lies in $\cK^0$ and satisfies $\Delta f=h$.
\end{example}

We also have the following surprising proposition, which says that if a function $g$ lies in $\cD^p\cap\cK^p_-\cap\cK^q$ for some integers $0\leq p\leq q$, then it actually lies in
$$
\cK^p_-\cap\cK^{p+1}_+\cap\cK^{p+2}_-\cap\cK^{p+3}_+\cap\cdots\cap\cK^q_{\pm}{\,},
$$
where the subscripts alternate in sign. The same property holds for $\Sigma g$.

\begin{proposition}\label{prop:5alternKpm}
Let $g$ lie in $\cD^p\cap\cK^p_-\cap\cK^{p+1}$ for some $p\in\N$. Then it lies in $\cK^{p+1}_+$ and $\Sigma g$ lies in $\cD^{p+1}\cap\cK^p_+\cap\cK^{p+1}_-$.
\end{proposition}

\begin{proof}
Suppose that $g$ lies in $\cK^{p+1}_-$. Since it also lies in $\cD^{p+1}\cap\cK^{p+1}_-$, by Corollary~\ref{cor:dfmm7s} it must lie in $\cK^p_+$. By Corollary~\ref{cor:convCones48}, $g$ is eventually a polynomial of degree less than or equal to $p$. But then, using Corollary~\ref{cor:convCones48} again, $g$ lies in $\cK^{p+1}_+$. The result about $\Sigma g$ is then trivial.
\end{proof}

\begin{example}
Let us apply Proposition~\ref{prop:5alternKpm} to the function $g(x)=\ln x$ with $p=1$. We then obtain that
\begin{eqnarray*}
g & \text{lies in} & \cD^1\cap\cK^1_-\cap\cK^2_+\cap\cK^3_-\cap\cK^4_+\cap\cdots\\
\text{while}\quad\Sigma g & \text{lies in} & \cD^2\cap\cK^1_+\cap\cK^2_-\cap\cK^3_+\cap\cK^4_-\cap\cdots{\,},
\end{eqnarray*}
where $\Sigma g(x)=\ln\Gamma(x)$. Moreover, it is easy to see that $g$ is $1$-concave on $\R_+$, $2$-convex on $\R_+$, and so on, and similarly for $\Sigma g$.
\end{example}

\begin{example}\label{ex:5Alt88Stiel}
Applying Proposition~\ref{prop:5alternKpm} to the function $g(x)=-\frac{1}{x}\ln x$ with $p=0$, we obtain that
\begin{eqnarray*}
g & \text{lies in} & \cD^0\cap\cK^0_+\cap\cK^1_-\cap\cK^2_+\cap\cK^3_-\cap\cdots\\
\text{while}\quad\Sigma g & \text{lies in} & \cD^1\cap\cK^0_-\cap\cK^1_+\cap\cK^2_-\cap\cK^3_+\cap\cdots{\,},
\end{eqnarray*}
where $\Sigma g(x)=\gamma_1(x)-\gamma_1$ is a generalized Stieltjes constant\index{Stieltjes constants!generalized Stieltjes constants} (see Section \ref{sec:Stie62}). Now, for every $q\in\N$, we have $g^{(q+1)}(x)=0$ if and only if $x=e^{H_{q+1}}$. Hence we can easily see that $g$ is $q$-convex or $q$-concave on the unbounded interval $(e^{H_{q+1}},\infty)$.
\end{example}

\begin{remark}\index{asymptotic degree}
Although the asymptotic degree of a function (see Definition~\ref{de:de4g3f5}) defines an important and useful concept, it is not always easy to compute. For instance, we can show after some calculus that, for any $p\in\N$, the function $h_p\colon\R_+\to\R$ defined by the equation (see Section~\ref{sec:PISGre47})
$$
h_p(x) ~=~ \frac{x^p}{\ln(x+1)}\qquad\text{for $x>0$}
$$
has the asymptotic degree $\deg h_p=p-1$. Thus, it would be useful to have a simple formula to compute easily the asymptotic degree of any function. On this matter, let us consider the limiting value (when it exists)
$$
e_f ~=~ \lim_{x\to\infty}x\,\frac{\Delta f(x)}{f(x)}{\,},
$$
which is inspired from the concept of the elasticity of a function $f$ (see, e.g., Nievergelt \cite{Nie83}). Computing this limit for the function $h_p$ above for instance, we easily obtain $e_{h_p}=p$. Interestingly, we can observe empirically that many functions $f$ lying in $\cK^0$ satisfy the double inequality
$$
\lfloor e_f\rfloor_+ ~\leq ~ 1+\deg f ~\leq ~ \lfloor 1+e_f\rfloor_+.
$$
It would then be useful to find necessary and sufficient conditions on the function $f$ for this double inequality to hold.
\end{remark}

\section{Multiple $\log\Gamma$-type functions}
\label{subsec:MLG-t}

Barnes~\cite{Bar99,Bar01,Bar04} introduced a sequence of functions $\Gamma_1,\Gamma_2,\ldots$,\label{p:Gp1s} called \emph{multiple gamma functions}\index{multiple gamma function|textbf}, that generalize the Euler gamma function. The restrictions of these functions to $\R_+$ are characterized by the equations
\begin{eqnarray*}
&& \Gamma_{p+1}(x+1) ~=~ \frac{\Gamma_{p+1}(x)}{\Gamma_p(x)}{\,},\\
&& \Gamma_1(x) ~=~ \Gamma(x),\quad \Gamma_p(1) ~=~ 1,\qquad\text{for $x>0$ and $p\in\N^*$},
\end{eqnarray*}
together with the convexity condition
$$
(-1)^{p+1}D^{p+1}\ln\Gamma_p(x) ~\geq ~0,\qquad x>0.
$$
For more recent references, see, e.g., Adamchik~\cite{Ada05,Ada14} and Srivastava and Choi \cite{SriCho12}.

Thus defined, this sequence of functions satisfies the conditions
$$
\ln\Gamma_{p+1}(x) ~=~ -\Sigma\ln\Gamma_p(x)\qquad\text{and}\qquad\deg(\ln\circ\Gamma_p) ~=~ p.
$$
Moreover, it can be naturally extended to the case when $p=0$ by setting $\Gamma_0(x)=1/x$.

Now, these observations motivate the following definition.
\begin{definition}
Let $p\in\N$.
\begin{itemize}
\item A \emph{$\Gamma_p$-type function}\index{$\Gamma_p$-type function|textbf} (resp.\ a \emph{$\log\Gamma_p$-type function})\index{$\log\Gamma_p$-type function|textbf} is a function of the form $\exp\circ\Sigma g$ (resp.\ $\Sigma g$), where $g$ lies in $\cD^p\cap\cK^p$ with $p=1+\deg g$.
\item A \emph{multiple $\Gamma$-type function}\index{multiple $\Gamma$-type function|textbf} (resp.\ \emph{multiple $\log\Gamma$-type function})\index{multiple $\log\Gamma$-type function|textbf} is a $\Gamma_p$-type function (resp.\ $\log\Gamma_p$-type function) for some $p\in\N$.
\end{itemize}
\end{definition}

When $p\geq 1$, $\exp\circ\Sigma g$ reduces to the function $\Gamma_p$ when $\exp\circ g$ is precisely the function $1/\Gamma_{p-1}$, which simply shows that the function $\Gamma_p$ restricted to $\R_+$ is itself a $\Gamma_p$-type function.

We also introduce the following notation. We let $\Gamma_p$\label{p:Gps} (resp.\ $\mathrm{Log}\Gamma_p$)\label{p:LGps} denote the set of $\Gamma_p$-type functions (resp.\ $\log\Gamma_p$-type functions). Thus, by definition the set $\mathrm{ran}(\Sigma)$ can be decomposed using the following disjoint union
$$
\mathrm{ran}(\Sigma) ~=~ \bigcup_{p\geq 0}\mathrm{ran}(\Sigma|_{\cD^p\cap\cK^p}) ~=~ \bigcup_{p\geq 0}\mathrm{Log}\Gamma_p{\,}.
$$
Thus defined, the set of $\log\Gamma_p$-type functions can be characterized as follows.

\begin{proposition}\label{prop:fffdde9}
For any function $f\colon\R_+\to\R$ and any $p\in\N$, the following assertions are equivalent.
\begin{enumerate}
  \item[(i)] $f\in\mathrm{Log}\Gamma_p${\,}.
  \item[(ii)] $f(1)=0$, $f\in\cK^p$, $\Delta f\in\cD^p\cap\cK^p$, and $\deg\Delta f=p-1$.
  \item[(iii)] $f=\Sigma\Delta f$, $\Delta f\in\cD^p\cap\cK^p$, and $\deg\Delta f=p-1$.
  \item[(iv)] $f\in\mathrm{ran}(\Sigma)$ and $\deg\Delta f=p-1$.
  \item[(v)] If $p\geq 1$, then $f\in\mathrm{ran}(\Sigma)$ and $\deg f=p${\,}. \\ If $p=0$, then $f\in\mathrm{ran}(\Sigma)$ and $\deg f\in\{-1,0\}$.
\end{enumerate}
\end{proposition}

\begin{proof}
The equivalence (i) $\Leftrightarrow$ (ii) $\Leftrightarrow$ (iii) is immediate by definition of $\Sigma$. The implications (iii) $\Rightarrow$ (iv) $\Rightarrow$ (ii) are straightforward. Finally, the equivalence (iv) $\Leftrightarrow$ (v) is trivial.
\end{proof}

From Proposition~\ref{prop:fffdde9} we immediately derive the following characterization of the set $\mathrm{ran}(\Sigma)$ of all multiple $\log\Gamma$-type functions.\index{multiple $\log\Gamma$-type function}

\begin{corollary}
A function $f\colon\R_+\to\R$ lies in $\mathrm{ran}(\Sigma)$ if and only if there exists $p\in\N$ such that $f(1)=0$, $f\in\cK^p$, and $\Delta f\in\cD^p\cap\cK^p$.
\end{corollary}

\section{Integration of multiple $\log\Gamma$-type functions}
\index{multiple $\log\Gamma$-type function!integration|(}

The uniform convergence of the sequence $n\mapsto f^p_n[g]$ (cf.\ Theorem~\ref{thm:exist}) shows that the function $\Sigma g$ is continuous whenever so is $g$. More generally, we also have the following result.

\begin{proposition}\label{prop:intMLGt}
Let $g$ lie in $\cC^0\cap\cD^p\cap\cK^p$ for some $p\in\N$. The following assertions hold.
\begin{enumerate}
\item[(a)] $\Sigma g$ lies in $\cC^0\cap\cD^{p+1}\cap\cK^p$.
\item[(b)] $\Sigma g$ is integrable at $0$ if and only if so is $g$.
\item[(c)] Let $n\in\N^*$ be so that $g$ is $p$-convex or $p$-concave on $[n,\infty)$ and let $0\leq a\leq x$. The following inequality holds
$$
\left|\int_a^x(f^p_n[g](t)-\Sigma g(t)){\,}dt\right| ~\leq ~ \int_a^x\lceil t\rceil\left|\tchoose{t-1}{p}\right| dt ~ \left|\Delta^pg(n)\right|.
$$
If $p\geq 1$, we also have the following tighter inequality
$$
\left|\int_a^x(f^p_n[g](t)-\Sigma g(t)){\,}dt\right| ~\leq ~ \int_a^x\left|\tchoose{t-1}{p}\right|\left|\Delta^{p-1}g(n+t)-\Delta^{p-1}g(n)\right|{\,}dt.
$$
Moreover, the following assertions hold.
\begin{enumerate}
\item[(c1)] The sequence
$$
n ~\mapsto ~\int_a^x\left(f^p_n[g](t)-\Sigma g(t)\right){\,}dt
$$
converges to zero.
\item[(c2)] The sequence
$$
n ~\mapsto ~\int_a^x(f^p_n[g](t)+ g(t)){\,}dt
$$
converges to
$$
\int_a^x(\Sigma g(t) + g(t)){\,}dt ~=~ \int_a^x\Sigma g(t+1){\,}dt.
$$
\item[(c3)] For any $m\in\N^*$, the sequence
$$
n ~\mapsto ~\int_a^x(f^p_n[g](t)-f^p_m[g](t)){\,}dt
$$
converges to
$$
\int_a^x(\Sigma g(t)-f^p_m[g](t)){\,}dt.
$$
\end{enumerate}
\end{enumerate}
\end{proposition}

\begin{proof}
Assertion (a) follows from Proposition~\ref{prp:56GathZZPro8} and the uniform convergence of the sequence $n\mapsto f^p_n[g]$. Assertion (b) follows from assertion (a) and the identity $\Sigma g(x+1)-\Sigma g(x)=g(x)$. Now, for any $n\in\N^*$, since $\rho_n^{p+1}[\Sigma g](0)=0$ by \eqref{eq:deflambdapt}, the function $\rho^{p+1}_n[\Sigma g]$ is clearly integrable on $(0,x)$ and hence on $(a,x)$. Using \eqref{eq:33ConvSig52}, it follows that the function $f^p_n[g]-\Sigma g$ is also integrable on $(a,x)$. The inequalities of assertion (c) then follows from Theorem~\ref{thm:exist}(b); and hence assertion (c1) also holds. Assertion (c2) follows from assertion (c1) and the identity $\Sigma g(x+1)-\Sigma g(x)=g(x)$. Finally, using \eqref{eq:fngsum} we see that the function $f^p_m[g]-f^p_n[g]$ is integrable on $(a,x)$ and hence assertion (c3) follows from assertion (c1).
\end{proof}

\begin{remark}\label{rem:5repeated0Int53}
Assertion (c) of Proposition~\ref{prop:intMLGt} has been obtained by integrating the function $\rho^{p+1}_n[\Sigma g]$ on $(a,x)$. The first inequality in assertion (c) then clearly shows that the sequences of functions defined in assertions (c1)--(c3) converge uniformly on any bounded subset of $\R_+$. Now, we also observe that the integral
$$
\int_a^x\rho^{p+1}_n[\Sigma g](t){\,}dt
$$
itself can be integrated on $(a,x)$, and we can repeat this process as often as we wish. After $n$ integrations, we obtain
$$
\frac{1}{(n-1)!}\,\int_a^x(x-t)^{n-1}\,\rho^{p+1}_n[\Sigma g](t){\,}dt,
$$
and, proceeding as in Proposition~\ref{prop:intMLGt}, it is then clear that the following inequality holds
$$
\left|\int_a^x(x-t)^{n-1}{\,}(f^p_n[g](t)-\Sigma g(t)){\,}dt\right| ~\leq ~ \int_a^x(x-t)^{n-1}\,\lceil t\rceil\left|\tchoose{t-1}{p}\right| dt ~ \left|\Delta^pg(n)\right|.
$$
In particular, this inequality shows that the left-hand integral converges uniformly on any bounded subset of $\R_+$ to zero.
\end{remark}

Let us end this section with the following important remark. In Proposition~\ref{prop:intMLGt} we have assumed the continuity of function $g$ to ensure that the integrals of both functions $g$ and $\Sigma g$ be defined. Of course, we could somewhat generalize our result by relaxing this continuity assumption into weaker properties such as local integrability of both $g$ and $\Sigma g$. However, for the sake of simplicity, in this work we will always assume the continuity of any function whenever we need to integrate it on a compact interval (see also Remark~\ref{rem:IntVSCont3}). In this case, continuity can be regarded simply as a handy assumption to keep the results simple. We then encourage the interested reader to generalize those results by searching for the weakest assumptions. This may sometimes lead to challenging but stimulating problems.

\index{multiple $\log\Gamma$-type function!integration|)}

\section{The quest for a characterization of $\mathrm{dom}(\Sigma)$}

Recall that the map $\Sigma$ is defined on the set
$$
\mathrm{dom}(\Sigma) ~=~ \bigcup_{p\geq 0}(\cD^p\cap\cK^p).
$$
In this respect, it would be useful to have a very simple test to check whether a given function $g\colon\R_+\to\R$ lies in this set. By Propositions~\ref{prop:4042RsDs5} and \ref{prop:MpDescFiltr}, the condition that $g$ lies in $\cD^{\infty}_{\N}\cap\cK^0$ is clearly necessary. In the next proposition we show that, if $g$ is not eventually identically zero, then it must also satisfy the following property
\begin{equation}\label{eq:CharSigDom2}
\limsup_{n\to_{\N}\,\infty}\,\frac{g(n+1)}{g(n)} ~\leq ~1.
\end{equation}
We first recall the following discrete version of L'Hospital's rule, also called Stolz-Ces\`aro theorem\index{Stolz-Ces\`aro theorem}. For a recent reference see, e.g., Ash {\em et al.} \cite{AshAllSte12}.

\begin{lemma}[Stolz-Ces\`aro theorem]\label{lemma:6Stolz52}
Let $n\mapsto a_n$ and $n\mapsto b_n$ be two real sequences. If the second sequence is strictly monotone and unbounded, then
$$
\liminf_{n\to\infty}\,\frac{a_{n+1}-a_n}{b_{n+1}-b_n} ~\leq ~ \liminf_{n\to\infty}\,\frac{a_n}{b_n} ~\leq ~ \limsup_{n\to\infty}\,\frac{a_n}{b_n} ~\leq ~ \limsup_{n\to\infty}\,\frac{a_{n+1}-a_n}{b_{n+1}-b_n}{\,}.
$$
In particular, if
$$
\lim_{n\to\infty}\frac{a_{n+1}-a_n}{b_{n+1}-b_n} ~=~ \ell
$$
for some $\ell\in\R$, then
$$
\lim_{n\to\infty}\frac{a_n}{b_n} ~=~ \ell.
$$
\end{lemma}

\begin{proposition}\label{prop:Conj541q}
If $g$ lies in $\mathrm{dom}(\Sigma)$ and is not eventually identically zero, then condition \eqref{eq:CharSigDom2} holds.
\end{proposition}

\begin{proof}
Assume that $g$ lies in $\cD^p\cap\cK^p$ for some $p\in\N$. Of course we can assume that $p=1+\deg g$. We can also assume that $g$ is not eventually a polynomial; for otherwise the condition \eqref{eq:CharSigDom2} clearly holds. If $p=0$, then the function $x\mapsto |g(x)|$ eventually decreases to zero and hence condition \eqref{eq:CharSigDom2} holds. Now suppose that $p\geq 1$. Then the function $\Delta^pg$ lies in $\cD^0\cap\cK^0$ and there are two exclusive cases to consider.
\begin{itemize}
\item[(a)] Suppose that the eventually monotone sequence $n\mapsto\Delta^{p-1}g(n)$ is unbounded. This sequence is actually eventually strictly monotone. Indeed, otherwise the function $\Delta^p g\in\cK^0$ would vanish in any unbounded interval of $\R_+$, and hence would eventually be identically zero. Equivalently, $g$ would eventually be a polynomial of degree less than or equal to $p-1$, a contradiction. Using the Stolz-Ces\`aro theorem (see Lemma~\ref{lemma:6Stolz52}) and the fact that condition \eqref{eq:CharSigDom2} holds for $\Delta^pg$, we then obtain
$$
\limsup_{n\to_{\N}\,\infty}\,\frac{\Delta^{p-1}g(n+1)}{\Delta^{p-1}g(n)} ~\leq ~ \limsup_{n\to_{\N}\,\infty}\,\frac{\Delta^pg(n+1)}{\Delta^pg(n)} ~\leq ~1.
$$
Iterating this process, we see that condition \eqref{eq:CharSigDom2} holds for $g$.

\item[(b)] Suppose that the sequence $n\mapsto\Delta^{p-1}g(n)$ has a finite limit (which is necessarily nonzero by minimality of $p$). If $p=1$, then condition \eqref{eq:CharSigDom2} holds trivially. If $p\geq 2$, then the eventually monotone sequence $n\mapsto\Delta^{p-2}g(n)$ is unbounded and we can show as in the previous case that it is actually eventually strictly monotone. Using the Stolz-Ces\`aro theorem, we then obtain
$$
\limsup_{n\to_{\N}\,\infty}\,\frac{\Delta^{p-2}g(n+1)}{\Delta^{p-2}g(n)} ~\leq ~ \limsup_{n\to_{\N}\,\infty}\,\frac{\Delta^{p-1}g(n+1)}{\Delta^{p-1}g(n)} ~=~1.
$$
Iterating this process, we see that condition \eqref{eq:CharSigDom2} holds.
\end{itemize}
This completes the proof.
\end{proof}

\begin{remark}
We observe that the left side of \eqref{eq:CharSigDom2} is not always a limit. For instance, the function $g\colon\R_+\to\R$ defined by the equation
$$
g(x) ~=~ \frac{1}{2^x}\left(1+\frac{1}{3}\sin x\right)\qquad\text{for $x>0$}
$$
lies in $\cD^0\cap\cK^0$ (see Remark~\ref{rem:SiNNKp14}) but the function $g(x+1)/g(x)$ is a nonconstant periodic function. The first example in Remark~\ref{rem:ToSE5P3} also illustrates this behavior.

On the other hand, a function $g\in\cK^0$ that satisfies condition \eqref{eq:CharSigDom2} need not lie in $\cD^{\infty}_{\N}$. For instance, for any $q\in\N$ the function
$$
g_q(x) ~=~ x^{q+1}+\sin x
$$
lies in $\cK^q\setminus\cK^{q+1}$, and hence also in $\cK^0$, and satisfies
$$
\lim_{n\to_{\N}\,\infty}\frac{g_q(n+1)}{g_q(n)} ~=~ 1.
$$
However, it does not lie in $\cD^{\infty}_{\N}$.
\end{remark}

We observe that condition~\eqref{eq:CharSigDom2} is very easy to check for many functions $g$ lying in $\cK^0$. Thus, this condition provides a simple and useful test. In particular, when the inequality in \eqref{eq:CharSigDom2} is strict, the sequence $n\mapsto g(n)$ is summable by the ratio test, and hence $g$ lies in $\cD^0\cap\cK^0$. On the other hand, when the inequality is an equality, it is not known whether this condition, together with the property that $g$ lies in $\cK^0$, are also sufficient for $g$ to lie in $\mathrm{dom}(\Sigma)$.

Now, it is easy to see that a function $g\colon\R_+\to\R$ lies in $\cD^{\infty}_{\N}$ if and only if there exists $p\in\N$ for which the sequence $n\mapsto\Delta^pg(n)$ converges. In particular, if we assume that $g$ lies in $\cK^{\infty}$, then $g$ does not lie in $\cD^{\infty}_{\N}$ (and hence it does not lie in $\mathrm{dom}(\Sigma)$) if and only if for every $p\in\N$ the sequence $n\mapsto\Delta^pg(n)$ tends to infinity. On the other hand, we can observe empirically that condition \eqref{eq:CharSigDom2} fails to hold for many functions $g$ lying in $\cK^{\infty}\setminus\cD^{\infty}_{\N}$. Examples of such functions include $g(x)=2^x$ and $g(x)=\Gamma(x)$. It seems then reasonable to think that this observation follows from a general rule. We then formulate the following conjecture.

\begin{conjecture}\label{conj:End661}
If a function $g\colon\R_+\to\R$ lies $\cK^{\infty}$ and is not eventually identically zero, then it also lies in $\cD^{\infty}_{\N}$ if and only if condition \eqref{eq:CharSigDom2} holds.
\end{conjecture}

\chapter{Asymptotic analysis}
\label{chapter:6}

The asymptotic behavior of the gamma function for large values of its argument can be summarized as follows: for any $a\geq 0$, we have the following asymptotic equivalences (see Titchmarsh \cite[Section 1.87]{Tit39})
\begin{eqnarray}
\Gamma(x+a) ~\sim ~ x^a\,\Gamma(x) && \text{as $x\to\infty$}{\,},\label{eq:StirAA4}\\
\Gamma(x) ~\sim ~ \sqrt{2\pi}{\,}e^{-x}x^{x-\frac{1}{2}} && \text{as $x\to\infty$}{\,},\label{eq:StirAA5}\\
\Gamma(x+1) ~\sim ~ \sqrt{2\pi x}{\,}e^{-x}x^x && \text{as $x\to\infty$}{\,},\label{eq:StirAA521}
\end{eqnarray}
where both formulas \eqref{eq:StirAA5} and \eqref{eq:StirAA521} are known by the name \emph{Stirling's formula}\index{Stirling's formula}.

In this chapter, we investigate the asymptotic behaviors of the multiple $\log\Gamma$-type functions and provide analogues of the formulas above.

More specifically, for these functions we establish analogues of \emph{Wendel's inequality}, \emph{Stirling's formula}, and \emph{Burnside's formula} for the gamma function. We also introduce the concept of the \emph{asymptotic constant}, an analogue of \emph{Stirling's constant}, and an analogue of \emph{Binet's function} related to the log-gamma function, and we show how all these generalized concepts can be used in the asymptotic analysis of multiple $\log\Gamma$-type functions. We also establish a general asymptotic equivalence for these functions.

We revisit \emph{Gregory's summation formula}, with an integral form of the remainder, and show how it can be derived very easily in this context. Using this formula, we then introduce a generalization of \emph{Euler's constant} and provide a geometric interpretation.

\section{Generalized Wendel's inequality}
\label{sec:6Wen0delI6eq2}
\index{Wendel's inequality!generalized|(}

Recall that if a function $g$ lies in $\cD^p\cap\cK^p$ for some $p\in\N$, then the function $\Sigma g$ lies in $\cR^{p+1}_{\R}$ by Proposition~\ref{prp:56GathZZPro8}. At first glance, this observation may seem rather unimportant. However, its explicit statement tells us that for any $a\geq 0$ we have
$$
\rho_x^{p+1}[\Sigma g](a)\to 0\qquad\text{as $x\to\infty$},
$$
or equivalently,
\begin{equation}\label{eq:convRes79}
\Sigma g(x+a)-\Sigma g(x)-\sum_{j=1}^p\tchoose{a}{j}\,\Delta^{j-1}g(x) ~\to ~0\qquad\text{as $x\to\infty$}{\,}.
\end{equation}
This is actually a nice convergence result that reveals the asymptotic behavior of the difference $\Sigma g(x+a)-\Sigma g(x)$ for large values of $x$. The special case when $p=1$ was established by Webster \cite[Theorem 6.1]{Web97b}.

When $g(x)=\ln x$ and $p=1$, this result reduces to
$$
\ln\Gamma(x+a)-\ln\Gamma(x)-a\ln x ~\to ~0\qquad\text{as $x\to\infty$}{\,},
$$
which is precisely the additive version of the asymptotic equivalence given in \eqref{eq:StirAA4}. We thus observe that \eqref{eq:convRes79} immediately provides an analogue of the asymptotic equivalence \eqref{eq:StirAA4} for all the multiple $\log\Gamma$-type functions.

Now, we observe that formula \eqref{eq:StirAA4} was also established by Wendel~\cite{Wen48}, who first provided a short and elegant proof of the following double inequality\index{Wendel's inequality}
\begin{equation}\label{eq:Orig9Wendel2}
\left(1+\frac{a}{x}\right)^{a-1} \leq ~ \frac{\Gamma(x+a)}{\Gamma(x){\,}x^a} ~\leq ~ 1{\,}, \qquad x>0{\,},\quad 0\leq a\leq 1{\,},
\end{equation}
or equivalently, in the additive notation,
\begin{equation}\label{eq:Orig9Wendel2add7}
(a-1)\ln\left(1+\frac{a}{x}\right) ~\leq ~ \rho_x^2[\ln\circ\Gamma](a) ~\leq ~ 0{\,}, \qquad x>0{\,},\quad 0\leq a\leq 1{\,},
\end{equation}
where
\begin{equation}\label{eq:6Rho2Ln0G24}
\rho_x^2[\ln\circ\Gamma](a) ~=~ \ln\Gamma(x+a)-\ln\Gamma(x)-a\ln x.
\end{equation}
We can readily see that this double inequality is actually a simple application of Lemma~\ref{lemma:VarEpsIneq} to the log-gamma function with $p=1$. Its generalization to all the multiple $\log\Gamma$-type functions is then straightforward and we present it in the following theorem. We call it the \emph{generalized Wendel inequality}.

\begin{theorem}[Generalized Wendel's inequality]\label{thm:AsymBehSol}\index{Wendel's inequality!generalized|textbf}
Let $g$ lie in $\cD^p\cap\cK^p$ for some $p\in\N$ and let $\pm$ stand for $1$ or $-1$ according to whether $g$ lies in $\cK^p_+$ or $\cK^p_-${\,}. Let also $x>0$ be so that $g$ is $p$-convex or $p$-concave on $[x,\infty)$ and let $a\geq 0$. Then we have
\begin{eqnarray*}
0 ~\leq ~ \pm (-1)\,\varepsilon_{p+1}(a){\,}\rho_x^{p+1}[\Sigma g](a)
&\leq &  \pm (-1)\left|\tchoose{a-1}{p}\right|\left(\Delta^p\Sigma g(x+a)-\Delta^p\Sigma g(x)\right)\\
&\leq & \pm (-1)\,\lceil a\rceil\left|\tchoose{a-1}{p}\right|\Delta^pg(x),
\end{eqnarray*}
with equalities if $a\in\{0,1,\ldots,p\}$. In particular, $\rho_x^{p+1}[\Sigma g](a)\to 0$ as $x\to\infty$. If $p\geq 1$, we also have
\begin{eqnarray*}
0 ~\leq ~ \pm (-1)\,\varepsilon_p(a){\,}\rho_x^p[g](a)
&\leq &  \pm (-1)\left|\tchoose{a-1}{p-1}\right|\left(\Delta^{p-1}g(x+a)-\Delta^{p-1}g(x)\right)\\
&\leq & \pm (-1)\,\lceil a\rceil\left|\tchoose{a-1}{p-1}\right|\Delta^pg(x),
\end{eqnarray*}
with equalities if $a\in\{0,1,\ldots,p-1\}$. In particular, $\rho_x^p[g](a)\to 0$ as $x\to\infty$.
\end{theorem}

\begin{proof}
Negating $g$ if necessary, we can assume that it is $p$-convex on $[x,\infty)$. By the existence Theorem~\ref{thm:exist}, the function $\Sigma g$ is then $p$-concave on $[x,\infty)$. By Lemma~\ref{lemma:pCInc5} and Proposition~\ref{prop:LMpGpLMp1D7}, the function $\Delta^p g$ is negative and increases to zero on $[x,\infty)$. Thus, for any $a\geq 0$ we have
$$
(-1)\sum_{j=0}^{\lceil a\rceil -1}\Delta^p g(x+j) ~\leq ~ (-1)\,\lceil a\rceil\Delta^pg(x).
$$
We then derive the first inequalities by applying Lemma~\ref{lemma:VarEpsIneq} to $f=\Sigma g$. Suppose now that $p\geq 1$. By Corollary~\ref{cor:dfmm7s}, we have that $g$ is $(p-1)$-concave on $[x,\infty)$. We then derive the remaining inequalities by applying Lemma~\ref{lemma:VarEpsIneq} to $f=g$.
\end{proof}

A symmetrized version of the generalized Wendel inequality can be easily obtained simply by taking the absolute value of each of its sides. This provides a coarsened, but simplified form of the generalized Wendel inequality. For instance, when $g(x)=\ln x$ and $p=1$ we then obtain the following inequality
\begin{equation}\label{eq:ti68a20add}
\big|\ln\Gamma(x+a)-\ln\Gamma(x)-a\ln x\big| ~\leq ~ |a-1|\,\ln\left(1+\frac{a}{x}\right), \qquad x>0{\,},~a\geq 0{\,},
\end{equation}
that is, in the multiplicative notation,
\begin{equation}\label{eq:ti68a20}
\left(1+\frac{a}{x}\right)^{-\left|a-1\right|} \leq ~ \frac{\Gamma(x+a)}{\Gamma(x){\,}x^a} ~\leq ~ \left(1+\frac{a}{x}\right)^{\left|a-1\right|}, \qquad x>0{\,},~a\geq 0{\,}.
\end{equation}
We then have the following immediate corollary, which provides a symmetrized version of the generalized Wendel inequality.

\begin{corollary}\label{cor:AsymBehSolCo}
Let $g$ lie in $\cD^p\cap\cK^p$ for some $p\in\N$. Let also $x>0$ be so that $g$ is $p$-convex or $p$-concave on $[x,\infty)$ and let $a\geq 0$. Then we have
$$
\left|\rho_x^{p+1}[\Sigma g](a)\right|
~\leq ~ \left|\tchoose{a-1}{p}\right|\left|\Delta^p\Sigma g(x+a)-\Delta^p\Sigma g(x)\right|
~\leq ~ \lceil a\rceil\left|\tchoose{a-1}{p}\right| |\Delta^pg(x)|,
$$
with equalities if $a\in\{0,1,\ldots,p\}$. In particular, $\rho_x^{p+1}[\Sigma g](a)\to 0$ as $x\to\infty$. If $p\geq 1$, we also have
$$
\left|\rho_x^p[g](a)\right|
~\leq ~  \left|\tchoose{a-1}{p-1}\right|\left|\Delta^{p-1}g(x+a)-\Delta^{p-1}g(x)\right|
~\leq ~ \lceil a\rceil\left|\tchoose{a-1}{p-1}\right| |\Delta^pg(x)|,
$$
with equalities if $a\in\{0,1,\ldots,p-1\}$. In particular, $\rho_x^p[g](a)\to 0$ as $x\to\infty$.
\end{corollary}

\begin{example}\label{ex:gamma582}
Applying Theorem~\ref{thm:AsymBehSol} and Corollary~\ref{cor:AsymBehSolCo} to the function $g(x)=\ln x$, for which we have $p=1+\deg g = 1$ and $\Sigma g(x)=\ln\Gamma(x)$, we immediately retrieve the inequalities \eqref{eq:Orig9Wendel2}--\eqref{eq:ti68a20} and hence also the asymptotic equivalence \eqref{eq:StirAA4}. Further inequalities can actually be obtained by considering higher values of $p$. For instance, since $g$ also lies in $\cD^2\cap\cK^2$, we can set $p=2$ in Corollary~\ref{cor:AsymBehSolCo} and we then obtain the inequalities
\begin{eqnarray*}
\lefteqn{\left(1+\frac{1}{x}\right)^{{a\choose 2}}\left(1+\frac{a}{x}\right)^{-\left|{a-1\choose 2}\right|}\left(1+\frac{a}{x+1}\right)^{\left|{a-1\choose 2}\right|} ~\leq ~ \frac{\Gamma(x+a)}{\Gamma(x){\,}x^a}}\\
&\leq & \left(1+\frac{1}{x}\right)^{{a\choose 2}}\left(1+\frac{a}{x}\right)^{\left|{a-1\choose 2}\right|}\left(1+\frac{a}{x+1}\right)^{-\left|{a-1\choose 2}\right|}.
\end{eqnarray*}
Thus, we can see that the central function in these inequalities can always be ``sandwiched'' by finite products of powers of rational functions. For further inequalities involving this central function, see, e.g., Srivastava and Choi \cite[pp.\ 106--107]{SriCho12}.
\end{example}

\parag{Discrete version of the generalized Wendel inequality}\index{Wendel's inequality!generalized!discrete version} The restrictions to the natural integers of the generalized Wendel inequality and its symmetrized form are obtained by setting $x=n\in\N^*$ in the inequalities of Theorem~\ref{thm:AsymBehSol} and Corollary~\ref{cor:AsymBehSolCo}. In view of identity \eqref{eq:33ConvSig52}, the symmetrized forms then reduce to those of the existence Theorem~\ref{thm:exist}.

For instance, when $g(x)=\ln x$ and $p=1$, the symmetrized version of generalized Wendel's inequality is given in \eqref{eq:ti68a20add} while its discrete version can take the form
$$
|\ln\Gamma(x)-f_n^1[\ln](x)| ~\leq ~ |x-1|\,\ln\left(1+\frac{x}{n}\right), \qquad x>0,~n\in\N^*,
$$
where
$$
f_n^1[\ln](x) ~=~ \sum_{k=1}^{n-1}\ln k-\sum_{k=0}^{n-1}\ln(x+k)+x\ln n.
$$
This latter inequality clearly generalizes Gauss' limit \eqref{eq:GaussLimit42}, which simply expresses that
$$
\ln\Gamma(x) ~=~ \lim_{n\to\infty} f_n^1[\ln](x),\qquad x>0.
$$

\index{Wendel's inequality!generalized|)}

\section{The asymptotic constant}
\index{asymptotic constant|(}

We now introduce a new important concept that will play a key role in our theory, namely the \emph{asymptotic constant}. This concept will actually be used intensively throughout the rest of this book.

\begin{definition}[Asymptotic constant]\label{de:AsyConst632}
The \emph{asymptotic constant}\index{asymptotic constant|textbf} associated with a function $g\in\cC^0\cap\mathrm{dom}(\Sigma)$ is the number\label{p:sigmag}
\begin{equation}\label{eq:sigmagg86}
\sigma[g] ~=~ \int_0^1\Sigma g(t+1){\,}dt ~=~ \int_0^1(\Sigma g(t)+g(t)){\,}dt{\,}.
\end{equation}
\end{definition}

Using Definition~\ref{de:AsyConst632}, we can readily see that the following identity holds for any function $g$ lying in $\cC^0\cap\mathrm{dom}(\Sigma)$
\begin{equation}\label{eq:ds68ffds}
\int_x^{x+1}\Sigma g(t){\,}dt ~=~ \sigma[g]+\int_1^xg(t){\,}dt,\qquad x>0.
\end{equation}
Indeed, both sides are functions of $x$ that have the same derivative and the same value at $x=1$.

\begin{example}[Raabe's formula]\label{ex:Raab286}
Taking $g(x)=\ln x$ in \eqref{eq:sigmagg86}, we obtain
$$
\sigma[g] ~=~ \int_0^1\ln\Gamma(t+1){\,}dt ~=~ -1+\frac{1}{2}\,\ln(2\pi){\,}.
$$
Combining this result with \eqref{eq:ds68ffds}, we obtain the following more general identity
$$
\int_x^{x+1}\ln\Gamma(t){\,}dt ~=~ \frac{1}{2}\,\ln(2\pi)+x\ln x-x{\,},\qquad x>0.
$$
This identity is known by the name \emph{Raabe's formula}\index{Raabe's formula} (see, e.g., Cohen and Friedman \cite{CohFri08}). We will discuss this formula and investigate its analogues in Section~\ref{sec:Raabe448}.
\end{example}

Identity \eqref{eq:ds68ffds} will also play a very important role in this work. In this respect, it is clear that the integral
\begin{equation}\label{eq:ds68ffdsf}
\int_x^{x+1}\Sigma g(t){\,}dt,\qquad x>0,
\end{equation}
cancels out the cyclic variations of any $1$-periodic additive component of $\Sigma g$ in the sense that the function
$$
x ~\mapsto ~ \int_x^{x+1}\omega(t){\,}dt
$$
is constant for any $1$-periodic function $\omega\colon\R_+\to\R$. Thus, the integral \eqref{eq:ds68ffdsf} can be interpreted as the \emph{trend} of the function $\Sigma g$, just as a moving average enables one to decompose a time series into its trend and its seasonal variation. In this light, identity \eqref{eq:ds68ffds} simply tells us that the trend of the function $\Sigma g$ is precisely the antiderivative of $g$ (up to an additive constant).

Let us end this section with the following two technical results related to the asymptotic constant.

\begin{proposition}
Let $g_1$ and $g_2$ lie in $\cD^p\cap\cK^p$ for some $p\in\N$ and let $c_1,c_2\in\R$. If $c_1g_1+c_2g_2$ lies in $\cD^p\cap\cK^p$, then
$$
\sigma[c_1g_1+c_2g_2] ~=~ c_1\sigma[g_1]+c_2\sigma[g_2].
$$
Moreover, we have $\sigma[\boldsymbol{1}]=\frac{1}{2}$, where $\boldsymbol{1}\colon\R_+\to\R$ is the constant function $\boldsymbol{1}(x)=1$.
\end{proposition}

\begin{proof}
The first part of the statement is an immediate consequence of Proposition~\ref{prop:gStHa}. Now, we clearly have $\Sigma\boldsymbol{1}=x-1$ and hence $\sigma[\boldsymbol{1}]=\frac{1}{2}$.
\end{proof}

\begin{proposition}\label{prop:SigTr62}
Let $g$ lie in $\cC^0\cap\mathrm{dom}(\Sigma)$, let $a\geq 0$, and let $h\colon\R_+\to\R$ be defined by the equation $h(x)=g(x+a)$ for $x>0$. Then
$$
\sigma[h] ~=~ \sigma[g]+\int_1^{a+1}g(t){\,}dt -\Sigma g(a+1).
$$
\end{proposition}

\begin{proof}
Using Proposition~\ref{prop:gStHa223} we obtain
$$
\sigma[h] ~=~ \int_0^1\Sigma g(t+a+1){\,}dt -\Sigma g(a+1) ~=~ \int_{a+1}^{a+2}\Sigma g(t){\,}dt -\Sigma g(a+1).
$$
We then get the result using \eqref{eq:ds68ffds}.
\end{proof}

\index{asymptotic constant|)}

\section{Generalized Binet's function}
\label{sec:6Asym4Cons6Bine}
\index{Binet's function!generalized|(}

The \emph{Binet function}\index{Binet's function|textbf} related to the log-gamma function is the function $J\colon\R_+\to\R$ defined by the equation (see, e.g., Cuyt {\em et al.} \cite[p.~224]{CuyVigVerWaaJon08})
\begin{equation}\label{eq:Binet5The}
J(x) ~=~ \ln\Gamma(x)-\frac{1}{2}\ln(2\pi)+x-\left(x-\frac{1}{2}\right)\ln x\qquad\text{for $x>0$}.
\end{equation}
Using identity \eqref{eq:6Rho2Ln0G24} and Raabe's formula (see Example~\ref{ex:Raab286}), we can easily provide the following integral form of Binet's function
$$
J(x) ~=~ -\int_0^1 \rho_x^2[\ln\circ\Gamma](t){\,}dt,\qquad x>0.
$$

This latter identity motivates the following definition, in which we introduce a generalization of Binet's function. Recall first that, for any $q\in\N$ and any $x>0$, the function $t\mapsto\rho_x^q[g](t)$ is continuous whenever so is $g$. In this case, since it also vanishes at $t=0$, it must be integrable on $(0,1)$.

\begin{definition}[Generalized Binet's function]
For any $g\in\cC^0$ and any $q\in\N$, we define the function $J^q[g]\colon\R_+\to\R$\label{p:Jqg} by the equation
\begin{equation}\label{eq:Binet643780}
J^q[g](x) ~=~ -\int_0^1\rho_x^q[g](t){\,}dt\qquad\text{for $x>0$}.
\end{equation}
We say that the function $J^q[g]$ is the \emph{generalized Binet function}\index{Binet's function!generalized|textbf}  associated with the function $g$ and the parameter $q$.
\end{definition}

Taking $g=\ln\circ\Gamma$ and $q=1+\deg g =2$ in identity \eqref{eq:Binet643780}, we thus simply retrieve the Binet function $J(x)=J^2[\ln\circ\Gamma](x)$ related to the log-gamma function, as defined in \eqref{eq:Binet5The}.

In the following two propositions, we collect a few immediate properties of the generalized Binet function. To this end, recall first that, for any $n\in\N$, the \emph{$n$th Gregory coefficient}\index{Gregory coefficients|textbf} (also called the \emph{$n$th Bernoulli number of the second kind}) is the number $G_n$ defined by the equation (see, e.g., \cite{Bla16,Bla17,BlaCop18,MerSprVer06})\label{p:Gn}
$$
G_n ~=~ \int_0^1\tchoose{t}{n}{\,}dt\qquad\text{for $n\geq 0$}.
$$
The first few values of $G_n$ are: $1, \frac{1}{2}, -\frac{1}{12}, \frac{1}{24}, -\frac{19}{720},\ldots$. These numbers are decreasing in absolute value and satisfy the equations
\begin{equation}\label{eq:Gr0Co36bis}
\sum_{n=1}^{\infty}|G_n| ~=~ 1\qquad\text{and}\qquad G_n ~=~ (-1)^{n-1}|G_n|\quad\text{for $n\geq 1$}.
\end{equation}

\begin{proposition}
Let $g\in\cC^0$ and $q\in\N$. Then, for any $x>0$, we have
\begin{equation}\label{eq:Binet64378}
J^q[g](x) ~=~ \sum_{j=0}^{q-1}G_j\Delta^jg(x)-\int_x^{x+1}g(t){\,}dt{\,}.
\end{equation}
In particular,
\begin{equation}\label{eq:Binet6141378}
\Delta J^q[g] = J^q[\Delta g]\qquad\text{and}\qquad J^{q+1}[g]-J^q[g] ~=~ G_q\,\Delta^q g.
\end{equation}
\end{proposition}

\begin{proof}
Identity \eqref{eq:Binet64378} follows immediately from \eqref{eq:deflambdapt}. The other two identities are trivial.
\end{proof}

\begin{proposition}
Let $g$ lie in $\cC^0\cap\mathrm{dom}(\Sigma)$ and let $q\in\N$.  Then, for any $x>0$ and any $n\in\N^*$, we have
\begin{eqnarray}
J^{q+1}[\Sigma g](x) &=& \Sigma g(x)-\sigma[g]-\int_1^xg(t){\,}dt + \sum_{j=1}^qG_j\Delta^{j-1} g(x){\,},\label{eq:Binet64378S}\\
J^{q+1}[\Sigma g](n) &=& \int_0^1\left(f_n^q[g](t)- \Sigma g(t)\right){\,}dt{\,}.\label{eq:Binet64378S04}
\end{eqnarray}
In particular,
$$
\Delta J^{q+1}[\Sigma g] = J^{q+1}[g]{\,},\qquad J^{q+1}[c+\Sigma g] ~=~ J^{q+1}[\Sigma g],\quad c\in\R,
$$
and
$$
\sigma[g] ~=~ -J^1[\Sigma g](1).
$$
\end{proposition}

\begin{proof}
Identity \eqref{eq:Binet64378S} follows from \eqref{eq:ds68ffds} and \eqref{eq:Binet64378}. Identity \eqref{eq:Binet64378S04} follows from \eqref{eq:33ConvSig52} and \eqref{eq:Binet643780}. The remaining identities are trivial.
\end{proof}

As we will see in the rest of this book, many subsequent definitions and results can be expressed in terms of the generalized Binet function.

\index{Binet's function!generalized|)}

\section{Generalized Stirling's formula}
\label{sec:6Gen4St2Fo0}
\index{Stirling's formula!generalized|(}

Interestingly, the Binet function\index{Binet's function} $J(x)=J^2[\Sigma\ln](x)$ defined in \eqref{eq:Binet5The} clearly satisfies the following identity (compare with Artin \cite[p.~24]{Art15})
$$
\Gamma(x) ~=~ \sqrt{2\pi}{\,}x^{x-\frac{1}{2}}{\,}e^{-x+J(x)}
$$
and hence Stirling's formula \eqref{eq:StirAA5} simply states that $J(x)\to 0$ as $x\to\infty$. This observation seems to reveal a way to find a counterpart of Stirling's formula for any continuous multiple $\log\Gamma$-type function. In fact, we only need to show that the function $J^{p+1}[\Sigma g]$ vanishes at infinity whenever $g$ lies in $\cC^0\cap\cD^p\cap\cK^p$ for some $p\in\N$. In the next theorem and its corollary, we establish this fact by simply integrating each side of the generalized Wendel inequality and its symmetrized version on $a\in (0,1)$.

Let us first define the sequence $n\mapsto\overline{G}_n$ by the equations
$$
\overline{G}_n ~=~ 1-\sum_{j=1}^n|G_j| ~=~ \sum_{j=n+1}^{\infty}|G_j|\qquad\text{for $n\in\N$}.\label{p:bGn}
$$
In view of \eqref{eq:Gr0Co36bis}, we see that the sequence $n\mapsto\overline{G}_n$ decreases to zero. Its first values are: $1, \frac{1}{2}, \frac{5}{12}, \frac{3}{8}, \frac{251}{720},\ldots$. Moreover, from the straightforward identity (see, e.g., Graham {\em et al.} \cite[p.~165]{GraKnuPat94})
$$
(-1)^n\tchoose{t-1}{n} ~=~ 1-\sum_{j=1}^n(-1)^{j-1}\tchoose{t}{j}{\,},
$$
we easily derive
\begin{equation}\label{eq:GintF4}
\int_0^1\left|\tchoose{t-1}{n}\right|{\,}dt ~=~ (-1)^n\int_0^1\tchoose{t-1}{n}{\,}dt ~=~ \left|\int_0^1\tchoose{t-1}{n}{\,}dt\right| ~=~ \overline{G}_n{\,}.
\end{equation}

We now have the following two results, which immediately follow from Theorem~\ref{thm:AsymBehSol}, Corollary~\ref{cor:AsymBehSolCo}, and identities \eqref{eq:GintF4}.

\begin{theorem}\label{thm:6GenStFo0BaIneq}
Let $g$ lie in $\cC^0\cap\cD^p\cap\cK^p$ for some $p\in\N$ and let $\pm$ stand for $1$ or $-1$ according to whether $g$ lies in $\cK^p_+$ or $\cK^p_-${\,}. Let also $x>0$ be so that $g$ is $p$-convex or $p$-concave on $[x,\infty)$. Then we have
\begin{eqnarray*}
0 ~\leq ~ \pm (-1)^p{\,}J^{p+1}[\Sigma g](x)
&\leq &  \pm {\,}(-1)^{p+1}\int_0^1\tchoose{t-1}{p}\left(\Delta^p\Sigma g(x+t)-\Delta^p\Sigma g(x)\right)dt\\
&\leq & \pm{\,}(-1)\,\overline{G}_p\,\Delta^pg(x).
\end{eqnarray*}
In particular, $J^{p+1}[\Sigma g](x)\to 0$ as $x\to\infty$. If $p\geq 1$, we also have
\begin{eqnarray*}
0 ~\leq ~ \pm (-1)^{p+1}{\,}J^p[g](x)
&\leq &  \pm {\,}(-1)^p\int_0^1\tchoose{t-1}{p-1}\left(\Delta^{p-1}g(x+t)-\Delta^{p-1}g(x)\right)dt\\
&\leq & \pm{\,}(-1)\,\overline{G}_{p-1}\,\Delta^pg(x).
\end{eqnarray*}
In particular, $J^p[g](x)\to 0$ as $x\to\infty$.
\end{theorem}

\begin{corollary}\label{cor:6GenStFo0BaIneqCo}
Let $g$ lie in $\cC^0\cap\cD^p\cap\cK^p$ for some $p\in\N$. Let also $x>0$ be so that $g$ is $p$-convex or $p$-concave on $[x,\infty)$. Then we have
$$
\left|J^{p+1}[\Sigma g](x)\right|
~\leq ~  \left|\int_0^1\tchoose{t-1}{p}\left(\Delta^p\Sigma g(x+t)-\Delta^p\Sigma g(x)\right)dt\right|
~\leq ~ \overline{G}_p{\,}|\Delta^pg(x)|.
$$
In particular, $J^{p+1}[\Sigma g](x)\to 0$ as $x\to\infty$. If $p\geq 1$, we also have
$$
\left|J^p[g](x)\right|
~\leq ~  \left|\int_0^1\tchoose{t-1}{p-1}\left(\Delta^{p-1}g(x+t)-\Delta^{p-1}g(x)\right)dt\right|
~\leq ~ \overline{G}_{p-1}{\,}|\Delta^pg(x)|.
$$
In particular, $J^p[g](x)\to 0$ as $x\to\infty$.
\end{corollary}

Both Theorem~\ref{thm:6GenStFo0BaIneq} and Corollary~\ref{cor:6GenStFo0BaIneqCo} state that $J^{p+1}[\Sigma g]$ vanishes at infinity whenever $g$ lies in $\cC^0\cap\cD^p\cap\cK^p$ for some $p\in\N$. This result is precisely the analogue of Stirling's formula for all the continuous multiple $\log\Gamma$-type functions. As it is one of the central results of our theory, we state it explicitly in the following theorem. We call it the \emph{generalized Stirling formula}. We also include the property that $J^p[g]$ vanishes at infinity.

\begin{theorem}[Generalized Stirling's formula]\label{thm:dgf7dds}\index{Stirling's formula!generalized|textbf}
Let $g$ lie in $\cC^0\cap\cD^p\cap\cK^p$ for some $p\in\N$. Then both functions $J^{p+1}[\Sigma g]$ and $J^p[g]$ vanish at infinity. More precisely, we have
\begin{equation}\label{eq:dgf7dds}
\Sigma g(x) -\int_1^x g(t){\,}dt +\sum_{j=1}^pG_j\Delta^{j-1}g(x) ~\to ~ \sigma[g]\qquad\text{as $x\to\infty$}
\end{equation}
and
\begin{equation}\label{eq:dgf7ddsp1}
\int_x^{x+1}g(t){\,}dt -\sum_{j=0}^{p-1}G_j\Delta^jg(x) ~\to ~ 0\qquad\text{as $x\to\infty$}{\,}.
\end{equation}
\end{theorem}

\begin{proof}
By Theorem~\ref{thm:6GenStFo0BaIneq}, the functions $J^{p+1}[\Sigma g]$ and $J^p[g]$ vanish at infinity when $p\geq 0$ and $p\geq 1$, respectively. The function $J^p[g]$ also vanishes at infinity  when $p=0$; indeed, in this case $|g(x)|$ eventually decreases to zero and we have
$$
|J^0[g](x)| ~=~ \left|\int_0^1g(x+t){\,}dt\right| ~\leq ~ |g(x)| ~\to ~ 0\qquad\text{as $x\to\infty$}.
$$
Formulas~\eqref{eq:dgf7dds} and \eqref{eq:dgf7ddsp1} then immediately follow from \eqref{eq:Binet64378} and \eqref{eq:Binet64378S}.
\end{proof}

The generalized Stirling formula \eqref{eq:dgf7dds} is actually the highlight of this chapter. It enables one to investigate the asymptotic behavior of the function $\Sigma g$ for large values of its argument. It also justifies the name ``asymptotic constant''\index{asymptotic constant} given to the quantity $\sigma[g]$ introduced in Definition~\ref{de:AsyConst632}. Moreover, combining \eqref{eq:convRes79} with \eqref{eq:dgf7dds}, we immediately derive the asymptotic behavior of $\Sigma g(x+a)$ for any $a\geq 0$. We also observe that alternative formulations of \eqref{eq:dgf7dds} in the case when $p=1$ were established by Krull~\cite[p.~368]{Kru48} and later by Webster \cite[Theorem 6.3]{Web97b}.

In the special case when $g$ lies in $\cD^{-1}\cap\cK^0$, the generalized Stirling formula and the asymptotic constant take very special forms.\index{asymptotic constant} We present them in the following proposition.

\begin{proposition}\label{prop:diffzz0}
If $g$ lies in $\cD^{-1}\cap\cK^0$, then we have
\begin{equation}\label{eq:diffzz092}
\Sigma g(x) ~\to ~ \sum_{k=1}^{\infty}g(k)\qquad\text{as $x\to\infty$}.
\end{equation}
If, in addition, we have $g\in\cC^0$, then $g$ is integrable at infinity and
$$
\sigma[g] ~=~ \sum_{k=1}^{\infty}g(k)-\int_1^{\infty}g(t){\,}dt.
$$
\end{proposition}

\begin{proof}
By definition of the map $\Sigma$, we have
$$
\Sigma g(x) ~=~ \sum_{k=1}^{\infty}g(k)-\sum_{k=0}^{\infty}g(x+k),\qquad x>0.
$$
where the second series tends to zero as $x\to\infty$ by Theorem~\ref{thm:existzz0}. The claimed expression for $\sigma[g]$ then immediately follows from formula \eqref{eq:dgf7dds}.
\end{proof}

\begin{example}\label{ex:StriCons}
Let us apply our results to the concave function $g(x)=\ln x$ with $p=1$. Using \eqref{eq:Binet64378} and \eqref{eq:Binet64378S}, we first obtain
\begin{eqnarray*}
J^2[\ln\circ\Gamma](x) &=& J(x) ~=~ \ln\Gamma(x)-\frac{1}{2}\ln(2\pi)+x-\left(x-\frac{1}{2}\right)\ln x{\,},\\
J^1[\ln](x) &=& 1-(x+1)\,\ln\left(1+\frac{1}{x}\right).
\end{eqnarray*}
Now, Theorem~\ref{thm:6GenStFo0BaIneq} provides the following inequalities for any $x>0$
\begin{equation}\label{eq:6QR46Wen22add}
0 ~\leq ~ J(x) ~\leq ~ \frac{1}{2}(x+1)^2\,\ln\left(1+\frac{1}{x}\right)-\frac{x}{2}-\frac{3}{4} ~\leq ~ \frac{1}{2}\,\ln\left(1+\frac{1}{x}\right){\,},
\end{equation}
$$
0 ~\leq ~ -1+(x+1)\,\ln\left(1+\frac{1}{x}\right) ~\leq ~ \ln\left(1+\frac{1}{x}\right).
$$
That is, in the multiplicative notation,
\begin{equation}\label{eq:6QR46Wen22}
1 ~\leq ~ \frac{\Gamma(x)}{\sqrt{2\pi}{\,}e^{-x}{\,}x^{x-\frac{1}{2}}} ~\leq ~ e^{-\frac{x}{2}-\frac{3}{4}}\left(1+\frac{1}{x}\right)^{\frac{1}{2}(x+1)^2} \leq ~ \left(1+\frac{1}{x}\right)^{\frac{1}{2}},
\end{equation}
$$
\left(1+\frac{1}{x}\right)^x ~\leq ~ e ~\leq ~ \left(1+\frac{1}{x}\right)^{x+1}.
$$
Thus, we retrieve Stirling's formula \eqref{eq:StirAA5} and \eqref{eq:StirAA521}, together with the well-known asymptotic equivalence (compare with Artin \cite[p.~20]{Art15})
$$
\left(1+\frac{1}{x}\right)^x ~\sim ~ e\qquad\text{as $x\to\infty$}.
$$
It is actually quite remarkable that the first two inequalities in \eqref{eq:6QR46Wen22add} and \eqref{eq:6QR46Wen22} are precisely what we get when we ``integrate'' the additive version of the Wendel inequality \eqref{eq:Orig9Wendel2} on the unit interval $(0,1)$.

Now, the coarsened inequality
$$
\left|J^{p+1}[\Sigma g](x)\right| ~\leq ~ \overline{G}_p{\,}|\Delta^pg(x)|
$$
given in Corollary~\ref{cor:6GenStFo0BaIneqCo} takes the following simple form (in the multiplicative notation)
$$
\left(1+\frac{1}{x}\right)^{-\frac{1}{2}} \leq ~ \frac{\Gamma(x)}{\sqrt{2\pi}{\,}e^{-x}{\,}x^{x-\frac{1}{2}}} ~\leq ~ \left(1+\frac{1}{x}\right)^{\frac{1}{2}}.
$$
Note that tighter inequalities can also be obtained by considering higher values of $p$ in Corollary~\ref{cor:6GenStFo0BaIneqCo}. For instance, taking $p=2$ we obtain
$$
\left(1+\frac{1}{x}\right)^{-\frac{3}{4}}\left(1+\frac{2}{x}\right)^{\frac{5}{12}} ~\leq ~ \frac{\Gamma(x)}{\sqrt{2\pi}{\,}e^{-x}{\,}x^{x-\frac{1}{2}}} ~\leq ~ \left(1+\frac{1}{x}\right)^{\frac{11}{12}}\left(1+\frac{2}{x}\right)^{-\frac{5}{12}}.
$$
Taking $p=3$ we obtain
\begin{multline*}
\left(1+\frac{1}{x}\right)^{-\frac{23}{24}}
\left(1+\frac{2}{x}\right)^{\frac{13}{12}}
\left(1+\frac{3}{x}\right)^{-\frac{3}{8}} ~\leq ~
\frac{\Gamma(x)}{\sqrt{2\pi}{\,}e^{-x}{\,}x^{x-\frac{1}{2}}}\\
\leq ~
\left(1+\frac{1}{x}\right)^{\frac{31}{24}}
\left(1+\frac{2}{x}\right)^{-\frac{7}{6}}
\left(1+\frac{3}{x}\right)^{\frac{3}{8}}.
\end{multline*}
Thus, we see that the central function in these inequalities can always be bracketed by finite products of radical functions.
\end{example}

In the last part of Example~\ref{ex:StriCons}, we have illustrated the possibility of obtaining closer bounds for the generalized Binet function\index{Binet's function!generalized} $J^{p+1}[\Sigma\ln](x)$ by considering in Corollary~\ref{cor:6GenStFo0BaIneqCo} any value of $p$ that is higher than $1+\deg g$. Actually, it is not difficult to see that this feature applies to every continuous multiple $\log\Gamma$-type function. We discuss this topic in Appendix~\ref{chapter:C-StBr49} and show that the inequalities actually get tighter and tighter as $p$ increases.

\begin{remark}\label{rem:WenSti41}
We observe that Theorem~\ref{thm:6GenStFo0BaIneq} together with the generalized Stirling formula (Theorem~\ref{thm:dgf7dds}) have been immediately obtained by ``integrating'' the generalized Wendel inequality (Theorem~\ref{thm:AsymBehSol}) on the unit interval. In turn, the generalized Wendel inequality is a straight application of Lemma~\ref{lemma:VarEpsIneq} to the function $f=\Sigma g$. These remarkable facts show the considerable importance of Lemma~\ref{lemma:VarEpsIneq} in this theory: it was first crucial to derive our uniqueness and existence results, and now it provides very nice counterparts of Wendel's inequality and Stirling's formula, with short and elegant proofs. We will use Lemma~\ref{lemma:VarEpsIneq} again in Section~\ref{sec:G8S3F5R} for an in-depth investigation of Gregory's summation formula.
\end{remark}

\parag{Improvements of Stirling's formula}\index{Stirling's formula!improvements} The following estimate of the gamma function is due to Gosper \cite{Gos78}
$$
\Gamma(x) ~\sim ~ \sqrt{2\pi}{\,}e^{-x}{\,}x^{x-\frac{1}{2}}\left(1+\frac{1}{6x}\right)^{\frac{1}{2}}\qquad\text{as $x\to\infty$},
$$
and is more accurate than Stirling's formula. On the basis of this alternative approximation, Mortici~\cite{Mor11} provided the following narrow inequalities
$$
\left(1+\frac{\alpha}{2x}\right)^{\frac{1}{2}} ~< ~ \frac{\Gamma(x)}{\sqrt{2\pi}{\,}e^{-x}{\,}x^{x-\frac{1}{2}}} ~< ~ \left(1+\frac{\beta}{2x}\right)^{\frac{1}{2}},\qquad\text{for $x\geq 2$},
$$
where $\alpha=\frac{1}{3}$ and $\beta=(391/30)^{1/3}-2\approx 0.353$. We actually observe that the quest for finer and finer bounds and approximations for the gamma function has gained an increasing interest during this last decade (see \cite{Bur19,CheLiu15,CheLin16,FenWan13,LuLiuQu17,Mor10,Mor11,Mor11b,Mor11c,XuHuTan16,YanTia19} and the references therein). Some of these investigations could be generalized to various multiple $\Gamma$-type functions. New results along this line would be welcome.

\parag{Webster's double inequality} We have seen that Theorems~\ref{thm:AsymBehSol} and \ref{thm:6GenStFo0BaIneq} provide very useful bounds for both quantities $\rho_x^{p+1}[\Sigma g](a)$ and $J^{p+1}[\Sigma g](x)$. It is actually possible to provide tighter bounds for these quantities using again the $p$-convexity\index{$p$-convexity} or $p$-concavity\index{$p$-concavity} properties of the function $g$. For instance, one can show that if $g$ lies in $\cD^1\cap\cK^1$ and if $x>0$ and $a>0$ are so that $g$ is concave on $[x+a,\infty)$, then the following double inequality hold
\begin{eqnarray}
\lefteqn{\sum_{k=0}^{\lfloor a\rfloor}g(x+k) + (\{a\}-1){\,}g(x+a)-a{\,}g(x) ~\leq ~ \rho^2_x[\Sigma g](a)}\nonumber\\
&\leq & \sum_{k=0}^{\lfloor a\rfloor} g(x+k)-g(x+a)+\{a\}{\,}g(x+\lfloor a\rfloor+1)-a{\,}g(x).\label{eq:6WebsInzz36}
\end{eqnarray}
This inequality was actually provided by Webster~\cite[Eq.~(6.4)]{Web97b}\index{Webster's inequality} to establish the limit \eqref{eq:convRes79} in the case when $p=1$.

Now, assuming that $g$ is continuous, we can integrate every expression in the inequalities above on $a\in (0,1)$, and we then obtain the following bounds for $J^2[\Sigma g](x)$
\begin{eqnarray}
0 &\leq & -J^2[g](x) ~\leq ~ J^2[\Sigma g](x)\nonumber\\
&\leq & -J^2[g](x)-\int_0^1 t{\,}g(x+t){\,}dt+\frac{1}{2}{\,}g(x+1).\label{eq:6WebsInzz36int22}
\end{eqnarray}
For instance, for $g(x)=\ln x$, we obtain (in the multiplication notation)
\begin{equation}\label{eq:6fAppF42}
1 ~\leq ~ e^{-1}\left(1+\frac{1}{x}\right)^{x+\frac{1}{2}} \leq ~ \frac{\Gamma(x)}{\sqrt{2\pi}{\,}e^{-x}{\,}x^{x-\frac{1}{2}}} ~\leq ~ e^{-\frac{x}{2}-\frac{3}{4}}\left(1+\frac{1}{x}\right)^{\frac{1}{2}(x+1)^2},
\end{equation}
which provides a better lower bound in the inequalities \eqref{eq:6QR46Wen22}.

In Appendix~\ref{chapter:E-GenWeIn5}, we discuss this interesting issue and provide a generalization to multiple $\log\Gamma$-type functions of the Webster double inequality \eqref{eq:6WebsInzz36} and its ``integrated'' version \eqref{eq:6WebsInzz36int22}.

\parag{Generalized Stirling's constant} The number $\sqrt{2\pi}$ arising in Stirling's formula \eqref{eq:StirAA5} and Example~\ref{ex:StriCons} is called \emph{Stirling's constant}\index{Stirling's constant} (see, e.g., Finch \cite{Fin03}). For certain multiple $\Gamma$-type functions, analogues of Stirling's constant can be easily defined as follows.

\begin{definition}[Generalized Stirling's constant]\label{de:GSC556}\index{Stirling's constant!generalized|textbf}
For any function $g\in\cC^0\cap\mathrm{dom}(\Sigma)$ that is integrable at $0$, we define the number
$$
\overline{\sigma}[g] ~=~ \sigma[g]-\int_0^1g(t){\,}dt ~=~ \int_0^1\Sigma g(t){\,}dt.\label{p:bsigmag}
$$
We say that the number $\exp(\overline{\sigma}[g])$ is the \emph{generalized Stirling constant} associated with $g$.
\end{definition}

When $g$ is integrable at $0$, the generalized Stirling constant exists and hence the generalized Stirling formula \eqref{eq:dgf7dds}\index{Stirling's formula!generalized} can take the following form
$$
\Sigma g(x) -\int_0^x g(t){\,}dt +\sum_{j=1}^pG_j\Delta^{j-1}g(x) ~\to ~ \overline{\sigma}[g]\qquad\text{as $x\to\infty$}{\,}.
$$

It is important to note that, contrary to the generalized Stirling constant, the asymptotic constant\index{asymptotic constant} $\sigma[g]$ exists for any function $g$ lying in $\cC^0\cap\mathrm{dom}(\Sigma)$, even if it is not integrable at $0$. For instance, for the function $g(x)=\frac{1}{x}$, we have that $\sigma[g]$ is the Euler constant\index{Euler's constant} $\gamma$ (see Example~\ref{ex:Dig5Int9}) while $\overline{\sigma}[g]$ does not exist.

This shows that the asymptotic constant\index{asymptotic constant} is the ``good'' constant to consider in this new theory. It actually enables us to derive for multiple $\log\Gamma$-type functions analogues of several properties of the gamma function. For instance, we have seen that it was very useful to derive the generalized Stirling formula. To give a second example, we will see in Section~\ref{sec:GaussMultF51} that it also enables us to derive analogues of Gauss' multiplication formula for the gamma function.

\index{Stirling's formula!generalized|)}

\section{Analogue of Burnside's formula}
\label{sec:6An4Buzzrns2}
\index{Burnside's formula!analogue|(}

Let us recall \emph{Burnside's formula}\index{Burnside's formula}, which states that
\begin{equation}\label{eq:Burn35side}
\Gamma(x) ~\sim ~ \sqrt{2\pi}\left(\frac{x-\frac{1}{2}}{e}\right)^{x-\frac{1}{2}}\qquad \text{as $x\to\infty$}.
\end{equation}
This formula actually provides a much better approximation of the gamma function than Stirling's formula. It was first established by Burnside \cite{Bur17} (see also Mortici \cite{Mor10}) and then rediscovered by Spouge \cite{Spo94}. In this section, we provide an analogue of Burnside's formula for any continuous $\Gamma_p$-type function when $p=0$ and $p=1$, and we note that such an analogue no longer exists when $p\geq 2$.

Let us first state the following corollary, which particularizes the generalized Stirling formula\index{Stirling's formula!generalized} when the function $g$ lies in $\cC^0\cap\cD^0\cap\cK^0$. This corollary actually follows immediately from \eqref{eq:ds68ffds} and \eqref{eq:dgf7dds}.

\begin{corollary}\label{cor:5Trend33}
Let $g$ lie in $\cC^0\cap\cD^0\cap\cK^0$. Then
$$
\Sigma g(x)-\int_x^{x+1}\Sigma g(t){\,}dt ~\to ~ 0 \qquad \text{as $x\to\infty$}{\,}.
$$
Equivalently,
$$
\Sigma g(x) -\int_1^xg(t){\,}dt ~\to ~ \sigma[g] \qquad \text{as $x\to\infty$}{\,}.
$$
\end{corollary}

Corollary~\ref{cor:5Trend33} tells us that, when $g$ lies in $\cC^0\cap\cD^0\cap\cK^0$, the function $\Sigma g(x)$ coincides asymptotically with its trend (i.e., the integral \eqref{eq:ds68ffdsf}) and, in a sense, behaves asymptotically like the antiderivative of function $g$.

It is natural to think that a more accurate trend of $\Sigma g$ can be obtained by considering the centered version of the integral \eqref{eq:ds68ffdsf}, namely
$$
\int_{x-\frac{1}{2}}^{x+\frac{1}{2}}\Sigma g(t){\,}dt ~=~ \sigma[g]+\int_1^{x-\frac{1}{2}}g(t){\,}dt,\qquad x>\textstyle{\frac{1}{2}}{\,}.
$$
On this matter, in the following proposition we provide a double inequality that shows that $\Sigma g(x)$ coincides asymptotically with this latter trend whenever $g$ lies in $\cC^0\cap\cD^0\cap\cK^0$ or in $\cC^0\cap\cD^1\cap\cK^1$. However, it is not difficult to see that in general this result no longer holds when $g$ lies in $\cC^0\cap\cD^2\cap\cK^2$. The logarithm of the Barnes $G$-function\index{Barnes's $G$-function} (see Section~\ref{sec:Barnes558}) could serve as an example here.

\begin{proposition}\label{prop:Burnside0}
Let $p\in\{0,1\}$, $g\in\cC^0\cap\cD^p\cap\cK^p$, and $x>0$ be so that $g$ is $p$-convex or $p$-concave on $[x,\infty)$. Then
$$
\left|\Sigma g\left(x+\frac{1}{2}\right)-\int_x^{x+1}\Sigma g(t){\,}dt\right| ~\leq ~ \left|J^{p+1}[\Sigma g](x)\right| ~\leq ~ \overline{G}_p{\,}|\Delta^p g(x)|.
$$
In particular,
$$
\Sigma g(x)-\int_{x-\frac{1}{2}}^{x+\frac{1}{2}}\Sigma g(t){\,}dt ~\to ~ 0 \qquad \text{as $x\to\infty$}{\,},
$$
or equivalently,
$$
\Sigma g(x) -\int_1^{x-\frac{1}{2}}g(t){\,}dt ~\to ~ \sigma[g] \qquad \text{as $x\to\infty$}{\,}.
$$
\end{proposition}

\begin{proof}
Using Corollary~\ref{cor:6GenStFo0BaIneqCo}, we see that it is enough to prove the first inequality. Let
$$
h(x) ~=~ \Sigma g\left(x+\frac{1}{2}\right)-\int_x^{x+1}\Sigma g(t){\,}dt.
$$
Consider first the case when $p=0$ and suppose for instance that $g$ lies in $\cK_+^0$; hence $\Sigma g$ is decreasing on $[x,\infty)$. If $h(x)\geq 0$, then we clearly have
$$
|h(x)| ~=~ h(x) ~\leq ~ \Sigma g(x)-\int_x^{x+1}\Sigma g(t){\,}dt ~=~ J^1[\Sigma g](x).
$$
If $h(x)\leq 0$, then we have
$$
|h(x)| ~=~ \int_x^{x+1}\Sigma g(t){\,}dt-\Sigma g\left(x+\frac{1}{2}\right) ~\leq ~ \int_x^{x+\frac{1}{2}}\Sigma g(t){\,}dt-\frac{1}{2}\Sigma g\left(x+\frac{1}{2}\right)
$$
and it is geometrically clear that the latter quantity is less than $J^1[\Sigma g](x)$.

Suppose now that $p=1$ and for instance that $g$ lies in $\cK_+^1$; hence $\Sigma g$ is concave on $[x,\infty)$. Applying the Hermite-Hadamard inequality to $\Sigma g$ on the interval $[x,x+1]$, we obtain that $h(x)\geq 0$. Applying the trapezoidal rule to $\Sigma g$ on the intervals $[x,x+\frac{1}{2}]$ and $[x+\frac{1}{2},x+1]$, we obtain the following inequality
$$
h(x) ~\leq ~ \int_x^{x+1}\Sigma g(t){\,}dt-\frac{1}{2}\,\Sigma g(x+1)-\frac{1}{2}\,\Sigma g(x),
$$
where the right-hand quantity is exactly $-J^2[\Sigma g](x)$. This completes the proof.
\end{proof}

Applying Proposition~\ref{prop:Burnside0} to the function $g(x)=\ln x$ with $p=1$, we retrieve Burnside's formula \eqref{eq:Burn35side}. Thus, Proposition~\ref{prop:Burnside0} gives an analogue of Burnside's formula for any continuous $\Gamma_p$-type function when $p\in\{0,1\}$. It also shows that this new formula provides a better approximation than the generalized Stirling formula whenever $g$ lies in $\cC^0\cap\cD^p\cap\cK^p$ with $p\in\{0,1\}$.

\index{Burnside's formula!analogue|)}

\section{A general asymptotic equivalence}

The following result provides a sufficient condition for a continuous multiple $\log\Gamma$-type function to be asymptotically equivalent to its (possibly shifted) trend.

\begin{proposition}\label{prop:conv6v6}
Let $g$ lie in $\cC^0\cap\mathrm{dom}(\Sigma)$ and let $a\geq 0$ and $c\in\R$. When $c+\Sigma g$ vanishes at infinity, we also assume that
\begin{equation}\label{eq:EqA62}
c+\Sigma g(n+1) ~\sim ~c+\Sigma g(n)\qquad\text{as $n\to_{\N}\infty$}.
\end{equation}
Then we have
\begin{equation}\label{eq:Ratio44t2}
c+\Sigma g(x+a) ~\sim ~ c+\int_x^{x+1}\Sigma g(t){\,}dt\qquad\text{as $x\to\infty$}.
\end{equation}
If $g$ does not lie in $\cD^{-1}_{\N}$, then we also have
$$
\Sigma g(x+a) ~\sim ~ c+\int_1^x g(t){\,}dt\qquad\text{as $x\to\infty$}.
$$
\end{proposition}

\begin{proof}
Let us first prove that \eqref{eq:EqA62} holds for any $g$ lying in $\cC^0\cap\mathrm{dom}(\Sigma)$, even if $c+\Sigma g$ does not vanish at infinity. Of course, this result clearly holds if $g$ is eventually a polynomial (since so is $\Sigma g$ in this case). Thus, we will now assume that $g$ is not eventually a polynomial.

Suppose first that $p=1+\deg g=0$. If $g$ lies in $\cD^{-1}_{\N}$, then \eqref{eq:EqA62} follows immediately from \eqref{eq:diffzz092}. If $g$ lies in $\cD^0_{\N}\setminus\cD^{-1}_{\N}$, then it is not integrable at infinity by the integral test for convergence. By the generalized Stirling formula \eqref{eq:dgf7dds}, it follows that the eventually  monotone sequence $n\mapsto\Sigma g(n)$ is unbounded. This sequence is actually eventually strictly monotone; indeed, otherwise the function $\Delta\Sigma g=g\in\cK^0$ would vanish in any unbounded interval of $\R_+$, and hence would eventually be identically zero, a contradiction. We then obtain
$$
\frac{c+\Sigma g(n+1)}{c+\Sigma g(n)} ~=~ 1+\frac{g(n)}{c+\Sigma g(n)} ~\to ~ 1\qquad\text{as $n\to_{\N}\infty$},
$$
and hence \eqref{eq:EqA62} holds whenever $p=0$.

Suppose now that $p=1+\deg g\geq 1$. In this case, we have that $\Delta^pg$ lies in $\cD^0\cap\cK^0$. By the uniqueness Theorem~\ref{thm:unic}, we also have
$$
\Delta^p\Sigma g ~=~ c_p+\Sigma\Delta^p g
$$
for some $c_p\in\R$, and it is clear (by minimality of $p$) that this latter function cannot vanish at infinity. Moreover, we can show as above that the sequence $n\mapsto\Sigma\Delta^p g(n)$ is eventually strictly monotone. In view of the first case, we then have
$$
\frac{\Delta^p\Sigma g(n+1)}{\Delta^p\Sigma g(n)} ~=~ \frac{c_p+\Sigma\Delta^p g(n+1)}{c_p+\Sigma\Delta^p g(n)} ~\to ~ 1\qquad\text{as $n\to_{\N}\infty$}.
$$
Let us now show that the sequence
$$
n ~ \mapsto ~ \frac{c+\Delta^{p-1}\Sigma g(n+1)}{c+\Delta^{p-1}\Sigma g(n)}
$$
exists for large values of $n$ and converges to $1$. By minimality of $p$, the function $\Delta^{p-1}\Sigma g$ lies in $\cD^2_{\N}\setminus\cD^1_{\N}$ and hence the sequence $n\mapsto\Delta^{p-1}\Sigma g(n)$ is unbounded. Moreover, we can show as above that this sequence is eventually strictly monotone. Hence, the sequence above eventually exists and, using the Stolz-Ces\`aro theorem (see Lemma~\ref{lemma:6Stolz52}), we have that
$$
\lim_{n\to\infty}\frac{c+\Delta^{p-1}\Sigma g(n+1)}{c+\Delta^{p-1}\Sigma g(n)} ~=~ \lim_{n\to\infty}\frac{\Delta^p\Sigma g(n+1)}{\Delta^p\Sigma g(n)} ~=~ 1.
$$
Iterating this process, we finally see that condition \eqref{eq:EqA62} holds for any $p\in\N$.

We can now easily see that
\begin{equation}\label{eq:EqA62a}
c+\Sigma g(x+a) ~\sim ~ c+\Sigma g(x)\qquad\text{as $x\to\infty$}.
\end{equation}
Indeed, this result clearly holds if both $x$ and $a$ are integers. For instance we have
$$
c+\Sigma g(n+2) ~\sim ~ c+\Sigma g(n+1) ~\sim ~ c+\Sigma g(n)\qquad\text{as $n\to_{\N}\infty$}.
$$
Otherwise, assuming for instance that $\Sigma g$ is eventually increasing and nonnegative, for sufficiently large $x$ we have
$$
\frac{c+\Sigma g(\lfloor x+a\rfloor)}{c+\Sigma g(\lceil x\rceil)} ~\leq ~ \frac{c+\Sigma g(x+a)}{c+\Sigma g(x)} ~\leq ~ \frac{c+\Sigma g(\lceil x+a\rceil)}{c+\Sigma g(\lfloor x\rfloor)}{\,},
$$
and \eqref{eq:EqA62a} then follows by the squeeze theorem.

Finally, assuming again that $\Sigma g$ is eventually increasing and nonnegative, for sufficiently large $x$ we have
$$
1 ~=~ \frac{c+\Sigma g(x)}{c+\Sigma g(x)} ~\leq ~ \frac{c+\int_x^{x+1}\Sigma g(t){\,}dt}{c+\Sigma g(x)} ~\leq ~ \frac{c+\Sigma g(x+1)}{c+\Sigma g(x)}
$$
and, using again the squeeze theorem, we immediately obtain the first claimed asymptotic equivalence.

Now, if $g$ does not lie in $\cD^{-1}_{\N}$, then $\Sigma g(x)$ tends to infinity as $x\to\infty$. Using \eqref{eq:ds68ffds}, we then have
$$
\frac{c+\int_1^x g(t){\,}dt}{\Sigma g(x+a)} ~=~ \frac{c-\sigma[g]}{\Sigma g(x+a)} + \frac{\int_x^{x+1} \Sigma g(t){\,}dt}{\Sigma g(x+a)} ~\to ~ 1\qquad\text{as $x\to\infty$}{\,},
$$
which completes the proof.
\end{proof}

\begin{remark}\label{rem:ToSE5P3}
Let us show that the assumption on the function $c+\Sigma g$ cannot be ignored in Proposition~\ref{prop:conv6v6}. Consider the functions $f\colon\R_+\to\R$ and $g\colon\R_+\to\R$ defined by the equations
$$
f(x) ~=~ \frac{x-1}{2^x}\left(1+\frac{1}{4}\,\sin x\right)\quad\text{and}\quad g(x) ~=~ \Delta f(x)\qquad\text{for $x>0$}.
$$
It is clear that $f$ lies in $\cD^0_{\N}$ and that $g$ lies in $\cD^{-1}_{\N}$. Moreover, it is not difficult to see that the inequalities
$$
-2^{x+2}f'(x) ~\geq ~ x\qquad\text{and}\qquad 2^{x+4}g'(x) ~\geq ~ x
$$
eventually hold, which shows that both $f$ and $g$ lie in $\cK^0$. By the uniqueness theorem it follows that $f=\Sigma g$. However, we can readily see that the sequence
$$
n ~\mapsto ~\frac{\Sigma g(n+1)}{\Sigma g(n)}
$$
does not converge, which shows that \eqref{eq:EqA62} does not hold when $c=0$. It is then possible to show that the equivalence \eqref{eq:Ratio44t2} does not hold either.

Now, to see that the last asymptotic equivalence in Proposition~\ref{prop:conv6v6} need not hold if $g$ lies in $\cD^{-1}_{\N}$, take for instance
$$
g(x) ~=~ \frac{2}{(x+1)(x+2)}\qquad\text{and}\qquad\Sigma g(x) ~=~ \frac{x-1}{x+1}{\,}.
$$
We then have
$$
\lim_{x\to\infty}\frac{c+\int_1^x g(t){\,}dt}{\Sigma g(x+a)} ~=~ c+\ln\frac{9}{4}{\,}.\qedhere
$$
\end{remark}

\section{The Gregory summation formula revisited}
\label{sec:G8S3F5R}

Let $g\in\cC^0$, $q\in\N$, and let $1\leq m\leq n$ be integers. Integrating both sides of identity \eqref{eq:fngsum} on $x\in (0,1)$, we immediately obtain the following identity\label{p:Rqmng}
\begin{equation}\label{eq:GregoryMN}
\int_m^ng(t){\,}dt ~=~ \sum_{k=m}^{n-1}g(k)+\sum_{j=1}^qG_j(\Delta^{j-1}g(n)-\Delta^{j-1}g(m))+R^q_{m,n}[g]{\,},
\end{equation}
where
\begin{equation}\label{eq:GregoryMN6Remai}
R^q_{m,n}[g] ~=~ \int_0^1\sum_{k=m}^{n-1}\rho_k^{q+1}[g](t){\,}dt ~=~ \int_0^1(f^q_m[g](t)-f^q_n[g](t)){\,}dt.
\end{equation}

Identity \eqref{eq:GregoryMN} is nothing other than \emph{Gregory's summation formula}\index{Gregory's summation formula} (see, e.g., \cite{BerShi65,Jor65,Mil51}) with an integral form of the remainder. Note that, just like identity \eqref{eq:GenErrIn}, equation \eqref{eq:GregoryMN} is a pure identity in the sense that it holds without any restriction on the form of $g(x)$, except that here we asked $g$ to be continuous.

Combining \eqref{eq:Binet643780} with \eqref{eq:GregoryMN6Remai} we immediately see that this identity can be simply written in terms of the generalized Binet function\index{Binet's function!generalized} as
\begin{equation}\label{eq:6GregoryMN0046}
\sum_{k=m}^{n-1}J^{q+1}[g](k) + R^q_{m,n}[g] ~=~ 0{\,}.
\end{equation}
Equivalently, if $g$ lies in $\cC^0\cap\mathrm{dom}(\Sigma g)$, using \eqref{eq:Binet64378S04} and \eqref{eq:GregoryMN6Remai} we see that this identity can also take the form
\begin{equation}\label{eq:GregoryMN004}
J^{q+1}[\Sigma g](n)-J^{q+1}[\Sigma g](m) +R^q_{m,n}[g] ~=~ 0{\,}.
\end{equation}

The next lemma, which is yet another straightforward consequence of Lemma \ref{lemma:VarEpsIneq}, provides an upper bound for $|R^q_{m,n}[g]|$ when $g$ is $q$-convex or $q$-concave on $[m,\infty)$. Under this latter assumption, we can then use Gregory's formula \eqref{eq:GregoryMN} as a quadrature method\index{Gregory's summation formula!as a quadrature formula} for the numerical computation of the integral of $g$ over the interval $[m,n)$.

\begin{lemma}\label{lemma:67cyxc}
Let $g$ lie in $\cC^0\cap\cK^q$ for some $q\in\N$ and let $m\in\N^*$ be so that $g$ is $q$-convex or $q$-concave on $[m,\infty)$. Then, for any integer $n\geq m$, we have
\begin{equation}\label{eq:saf65fs}
|R^q_{m,n}[g]| ~ \leq ~ \overline{G}_q{\,}|\Delta^q g(n)-\Delta^q g(m)|.
\end{equation}
\end{lemma}

\begin{proof}
This result is an immediate consequence of Lemma~\ref{lemma:VarEpsIneq}. Indeed, we can write
$$
|R^q_{m,n}[g]| ~=~ \left|\sum_{k=m}^{n-1}\int_0^1\rho^{q+1}_k[g](t){\,}dt\right| ~\leq ~ \overline{G}_q{\,}\left|\sum_{k=m}^{n-1}\Delta^{q+1}g(k)\right|,
$$
where the latter sum clearly telescopes to $\Delta^qg(n)-\Delta^qg(m)$.
\end{proof}

\begin{example}
Let us compute numerically the integral
$$
I ~=~ \int_{\pi}^{2\pi}\ln x{\,}dx ~=~ 4.809854526737\ldots
$$
using Gregory's summation formula \eqref{eq:GregoryMN} and the upper bound \eqref{eq:saf65fs} of its remainder. Using an appropriate linear change of variable, we obtain
$$
I ~=~ \int_1^ng(t){\,}dt,\qquad\text{where}\quad g(t) ~=~ \frac{\pi}{n-1}\,\ln\left(\frac{\pi}{n-1}{\,}(t-1)+\pi\right).
$$
Taking $n=20$ and $q=10$ for instance, we obtain
$$
I ~\approx ~ \sum_{k=1}^{19}g(k)+\sum_{j=1}^{10}G_j(\Delta^{j-1}g(20)-\Delta^{j-1}g(1)) ~=~ 4.809854526746\ldots
$$
and \eqref{eq:saf65fs} gives $|R^{10}_{1,20}[g]|\leq 5.9\times 10^{-11}$.
\end{example}

In the following result, we give sufficient conditions on the function $g$ for the sequence $q\mapsto R^q_{m,n}[g]$ to converge to zero. Gregory's formula \eqref{eq:GregoryMN} then takes a special form.

\begin{proposition}\label{prop:a7fsd6f}
Let $g\in\cC^0\cap\cK^{\infty}$, $p\in\N$, and let $1\leq m\leq n$ be integers. Suppose that, for every integer $q\geq p$, the function $g$ is $q$-convex or $q$-concave on $[m,\infty)$. Suppose also that the sequence $q\mapsto\Delta^qg(n)-\Delta^qg(m)$ is bounded. Then we have
$$
R_{m,n}^q[g] ~\to~ 0 \qquad\text{as $q\to_{\N}\infty$},
$$
or equivalently,
$$
\int_m^ng(t){\,}dt ~=~ \sum_{k=m}^{n-1}g(k)+\sum_{j=1}^{\infty}G_j(\Delta^{j-1}g(n)-\Delta^{j-1}g(m)){\,}.
$$
If $g$ lies in $\cC^0\cap\mathrm{dom}(\Sigma g)$, then the latter identity also takes the form
$$
\Sigma g(n)-\Sigma g(m) ~=~ \int_m^ng(t){\,}dt-\sum_{j=1}^{\infty}G_j(\Delta^{j-1}g(n)-\Delta^{j-1}g(m)){\,}.
$$
\end{proposition}

\begin{proof}
Under the assumptions of this proposition, the sequence $q\mapsto R^q_{m,n}[g]$ converges to zero by Lemma~\ref{lemma:67cyxc}. (Recall that the sequence $n\mapsto\overline{G}_n$ converges to zero.) The result then immediately follows from Gregory's formula \eqref{eq:GregoryMN}. The last part then follows from identity \eqref{eq:RestrInt}.
\end{proof}

\begin{example}\label{ex:a7fsd6fw}
Taking $g(x)=\ln x$ and $m=p=1$ in Proposition~\ref{prop:a7fsd6f}, we obtain the following identity
$$
\ln n! ~=~ 1-n+\left(n+\frac{1}{2}\right)\ln n+\frac{1}{12}\ln\left(\frac{n+1}{2n}\right)-\frac{1}{24}\ln\left(\frac{4n(n+2)}{3(n+1)^2}\right)+\cdots
$$
which holds for any $n\in\N^*$.
\end{example}

\parag{A geometric interpretation of Gregory's formula}\index{Gregory's summation formula!geometric interpretation} For any $g\in\cC^0$ and any $q\in\N$, we let $\overline{P}_q[g]\colon [1,\infty)\to\R$ denote the piecewise polynomial function whose restriction to any interval $[k,k+1)$, with $k\in\N^*$, is the interpolating polynomial\index{interpolating polynomial} of $g$ with nodes at $k,k+1,\ldots,k+q$. That is,\label{p:PieceWPol23}
\begin{equation}\label{eq:PieceWPol48}
\overline{P}_q[g](x) ~=~ P_q[g](k,k+1,\ldots,k+q;x),\qquad x\in [k,k+1),
\end{equation}
or equivalently, using \eqref{eq:NewtonInt},
\begin{eqnarray*}
\overline{P}_q[g](x) &=& P_q[g](\lfloor x\rfloor,\lfloor x\rfloor+1,\ldots,\lfloor x\rfloor+q;x)\\
&=& \sum_{j=0}^q\tchoose{\{x\}}{j}\,\Delta^jg(\lfloor x\rfloor){\,},\qquad x\geq 1.
\end{eqnarray*}
In the following proposition, we provide an integral expression for the remainder $R^q_{m,n}[g]$ in terms of the function $\overline{P}_q[g]$.

\begin{proposition}\label{prop:622PieceWPol48R}
For any $g\in\cC^0$, any $q\in\N$, and any integers $1\leq m\leq n$, we have
\begin{equation}\label{eq:PieceWPol48R}
R^q_{m,n}[g] ~=~ \int_m^n(g(t)-\overline{P}_q[g](t)){\,}dt.
\end{equation}
\end{proposition}

\begin{proof}
Using \eqref{eq:LErrIn} and \eqref{eq:Binet643780} we then obtain
$$
-J^{q+1}[g](k) ~=~ \int_0^1\rho_k^{q+1}[g](t){\,}dt ~=~ \int_k^{k+1}(g(t)-\overline{P}_q[g](t)){\,}dt.
$$
The result then follows from \eqref{eq:6GregoryMN0046}.
\end{proof}

Proposition~\ref{prop:622PieceWPol48R} immediately provides an interesting interpretation of Gregory's formula as a quadrature method.\index{Gregory's summation formula!as a quadrature formula} It actually shows that Gregory's formula approximates the integral of $g$ over the interval $[m,n)$ by replacing $g$ with the piecewise polynomial function $\overline{P}_q[g]$. In particular, the remainder $R^q_{m,n}[g]$ reduces to zero whenever $g$ is a polynomial of degree less than or equal to $q$.

We also observe that Gregory's formula reduces to the ``left'' rectangle method (left Riemann sum) when $q=0$, and the trapezoidal rule when $q=1$. However, it does not reduce to Simpson's rule when $q=2$. In fact, Gregory's formula does not correspond to a Newton-Cotes quadrature rule when $q\geq 2$.

Now, if $g$ is q-convex or $q$-concave on $[m,\infty)$, then for any $k\in\{m,m+1,\ldots,n-1\}$ and any $t\in [0,1)$, using Lemma~\ref{lemma:VarEpsIneq} and identity \eqref{eq:LErrIn} we obtain
$$
0 ~\leq ~ \pm (-1)^q\,\rho^{q+1}_k[g](t) ~=~  \pm (-1)^q\left(g(k+t)-\overline{P}_q[g](k+t)\right),
$$
where $\pm$ stands for $1$ or $-1$ according to whether $g$ is $q$-convex or $q$-concave on $[m,\infty)$. This observation provides the following additional geometric interpretation. It shows that, on the interval $[k,k+1)$, the graph of $g$ lies over or under that of $\overline{P}_q[g]$ according to whether $\pm (-1)^q$ is $1$ or $-1$. As an immediate consequence, the quantity $|J^{q+1}[g](k)|$ is precisely the surface area between both graphs over the interval $[k,k+1)$ while the remainder $|R^q_{m,n}[g]|$ is the surface area between both graphs over the interval $[m,n)$.

\begin{example}
With the function $g(x)=\ln x$ and the parameter $q=1$ we associate the piecewise linear function
$$
\overline{P}_1[g](x) ~=~ \ln\lfloor x\rfloor + (x-\lfloor x\rfloor)\ln\left(1+\frac{1}{\lfloor x\rfloor}\right).
$$
Since $g$ is concave, for any integer $n\geq 1$ the graph of $g$ on $[1,n)$ lies over (or on) that of $\overline{P}_1[g]$, which is the polygonal line through the points $(k,g(k))$ for $k=1,\ldots,n$. The value (see \eqref{eq:GregoryMN004})
$$
R^1_{1,n}[g] ~=~ J(1)-J(n) ~=~ -\ln\Gamma(n)+\left(n-\frac{1}{2}\right)\ln n -n+1{\,},
$$
where $J(x)$ is Binet's function\index{Binet's function} defined in \eqref{eq:Binet5The}, is then nothing other than the remainder in the trapezoidal rule on $[1,n)$ with the integer nodes $1,\ldots,n$. Geometrically, it measures the surface area between the graph of $g$ and the polygonal line.
\end{example}

\parag{Alternative integral form of the remainder} The following proposition yields an alternative integral form of the remainder $R^q_{m,n}[g]$ when $g$ lies in $\cC^{q+1}$ for some $q\in\N^*$. Consider first the (kernel) function $K^q_{m,n}\colon\R_+\to\R$ defined by the equation
$$
K^q_{m,n}(t) ~=~ \frac{1}{q!}{\,}R^q_{m,n}[(\boldsymbol{\cdot} -t)^q_+]\qquad\text{ for $t\in\R_+$}.
$$
It is not difficult to show that this function lies in $\cC^{q-1}$ and has the compact support $[m,n+q-1]$.

\begin{proposition}
Suppose that $g$ lies in $\cC^{q+1}$ for some $q\in\N^*$ and let $1\leq m\leq n$ be integers. Then we have
$$
R^q_{m,n}[g] ~=~ \int_m^{n+q-1} K^q_{m,n}(t){\,}D^{q+1}g(t){\,}dt.
$$
\end{proposition}

\begin{proof}
By Taylor's theorem, the following identity
$$
g(x) ~=~ P_q(x) + \int_m^{n+q-1}\frac{(x-t)^q_+}{q!}{\,}D^{q+1}g(t){\,}dt
$$
holds on the interval $[m,n+q-1]$ for some polynomial $P_q$ of degree less than or equal to $q$. The result then follows from the definition of the remainder $R^q_{m,n}[g]$ and the fact that $R^q_{m,n}[P_q]=0$.
\end{proof}

Interestingly, if the function $K^q_{m,n}$ does not change in sign (and we conjecture that $(-1)^q{\,}K^q_{m,n}$ is nonnegative), then by the mean value theorem for definite integrals the remainder also takes the form
$$
R^q_{m,n}[g] ~=~ D^{q+1}g(\xi)\,\int_m^{n+q-1}K^q_{m,n}(t){\,}dt
$$
for some $\xi\in [m,n+q-1]$.

\begin{remark}
We observe that Jordan \cite[p.~285]{Jor65} claimed that
$$
\text{``$~R^q_{m,n}[g] ~=~ G_{q+1}(n-m)\,\Delta^{q+1} g(\xi)~$''}
$$
for some $\xi\in (m,n)$. However, taking for instance $g(x)=x^2$ and $(q,m,n)=(0,1,2)$, we can see that this form of the remainder is not correct. Nevertheless, several examples suggest that Jordan's statement could possibly be corrected by assuming that $\xi\in (m-1,n-1)$. This question thus remains open.
\end{remark}

\parag{General Gregory's formula and Euler-Maclaurin's formula} The following proposition provides Gregory's formula in its general form using our integral expression for the remainder. 

\begin{proposition}[General form of Gregory's formula]\label{prop:GenFoGreg7}\index{Gregory's summation formula!general form}
Let $a\in\R$, $n,q\in\N$, $h>0$, and $f\in\cC^0([a,\infty))$. Then we have
\begin{eqnarray*}
\lefteqn{\frac{1}{h}\int_a^{a+nh}f(t){\,}dt ~=~ \sum_{k=0}^{n-1}f(a+kh)}\\
&& \null +\sum_{j=1}^qG_j\left((\Delta_{[h]}^{j-1}f)(a+nh)-(\Delta_{[h]}^{j-1}f)(a)\right)+R^q_{1,n+1}[f^h_a]{\,},
\end{eqnarray*}
where
$$
R^q_{1,n+1}[f^h_a] ~=~ \int_0^1\sum_{k=1}^n\rho_k^{q+1}[f^h_a](t){\,}dt\quad\text{and}\quad f^h_a(x) ~=~ f(a+(x-1)h).
$$
Moreover, if $f$ is $q$-convex or $q$-concave on $[a,\infty)$, then
$$
|R^q_{1,n+1}[f^h_a]| ~\leq ~ \overline{G}_q\left|(\Delta_{[h]}^{q}f)(a+nh)-(\Delta_{[h]}^{q}f)(a)\right|.
$$
Here, $\Delta_{[h]}$\label{p:DeltaSh} denotes the forward difference operator with step $h>0$.
\end{proposition}

\begin{proof}
This formula can be obtained immediately from \eqref{eq:GregoryMN} and \eqref{eq:GregoryMN6Remai} replacing $n$ with $n+1$ and then setting $m=1$ and $g(x)=f(a+(x-1)h)$. The last part follows from Lemma~\ref{lemma:67cyxc}.
\end{proof}

The general Gregory formula is often compared with the corresponding \emph{Euler-Maclau\-rin summation formula}. We will use the latter in Chapter~\ref{chapter:8}, so we now state it in its general form (for background see, e.g., Apostol \cite{Apo99}, Gel'fond \cite{Gel71}, Lampret \cite{Lam01}, Mariconda and Tonolo \cite{MarTon16}, and Srivastava and Choi \cite{SriCho12}).

Recall first that the \emph{Bernoulli numbers}\index{Bernoulli numbers|textbf} $B_0,B_1,B_2,\ldots$\label{p:BernN} are defined implicitly by the single equation (see, e.g., Gel'fond \cite[Chapter 4]{Gel71} and Graham {\em et al.} \cite[p.~284]{GraKnuPat94})
\begin{equation}\label{defBnN3}
\sum_{j=0}^m\tchoose{m+1}{j}B_j ~=~ 0^m,\qquad m\in\N{\,}.
\end{equation}
The first few values of $B_n$ are: $1, -\frac{1}{2}, \frac{1}{6}, 0, -\frac{1}{30}, 0,\ldots$. Recall also that, for any $n\in\N$, the $n$th degree Bernoulli polynomial\index{Bernoulli polynomials|textbf} $B_n(x)$\label{p:BernP} is defined by the equation
$$
B_n(x) ~=~ \sum_{k=0}^n\tchoose{n}{k}{\,}B_{n-k}{\,}x^k\qquad\text{for $x\in\R$}.
$$

\begin{proposition}[Euler-Maclaurin's formula]\label{prop:E3MacLau5}\index{Euler-Maclaurin summation formula}
Let $N\in\N^*$, $f\in\cC^1([a,b])$, and $h=(b-a)/N$, for some real numbers $a<b$. Then we have
\begin{eqnarray*}
h\,\sum_{k=0}^Nf(a+kh) &=& \int_a^b f(x){\,}dx+\frac{h}{2}{\,}(f(a)+f(b))\\
&&\null +h^2\int_0^N B_1(\{t\}){\,}f'(a+th){\,}dt{\,}.
\end{eqnarray*}
If, in addition, $f\in\cC^{2q}([a,b])$ for some $q\in\N^*$, then
\begin{eqnarray*}
h\,\sum_{k=0}^Nf(a+kh) &=& \int_a^b f(x){\,}dx+\frac{h}{2}{\,}(f(a)+f(b))\\
&&\null +\sum_{j=1}^qh^{2j}{\,}\frac{B_{2j}}{(2j)!}{\,}\left(f^{(2j-1)}(b)-f^{(2j-1)}(a)\right) + R{\,},
\end{eqnarray*}
where
$$
R ~=~ - h^{2q+1}\int_0^N\frac{B_{2q}(\{t\})}{(2q)!}{\,}f^{(2q)}(a+th){\,}dt
$$
and
$$
|R| ~\leq ~ h^{2q}{\,}\frac{|B_{2q}|}{(2q)!}\int_a^b|f^{(2q)}(x)|{\,}dx{\,}.
$$
Here $f\in\cC^k([a,b])$ means that $f\in\cC^k(I)$ for some open interval $I$ containing $[a,b]$.
\end{proposition}

\begin{remark}\label{rem:CompEulGr3}
We observe (to paraphrase Jordan \cite[p.~285]{Jor65}) that Euler-Maclaurin's formula is more advantageous than Gregory's formula if we deal with functions whose derivatives are less complicated than their differences. However, there are functions for which Euler-Maclaurin's formula leads to divergent series while the corresponding Gregory's formula-based series (see Proposition~\ref{prop:a7fsd6f}) are convergent. For instance, this may be due to the fact that, for any $x>0$, the sequence $n\mapsto D^n\frac{1}{x}$ is unbounded while the sequence $n\mapsto\Delta^n\frac{1}{x}$ converges to zero.
\end{remark}

\section{Generalized Euler's constant}
\label{sec:GEc53}
\index{Euler's constant!generalized|(}

In this section, we introduce and discuss an analogue of Euler's constant for any function $g$ lying in $\cC^0\cap\mathrm{dom}(\Sigma)$. We first consider a lemma.

\begin{lemma}\label{lemma:6GEC43sdf}
Let $g$ lie in $\cC^0\cap\cD^p\cap\cK^p$ for some $p\in\N$ and let $m\in\N^*$. Then the sequence $n\mapsto R^p_{m,n}[g]$ for $n\geq m$ converges. Denoting its limit by $R^p_{m,\infty}[g]$, we have
$$
R^p_{m,\infty}[g] ~=~ J^{p+1}[\Sigma g](m).
$$
\end{lemma}

\begin{proof}
The proof is an immediate consequence of \eqref{eq:GregoryMN004} and the generalized Stirling formula (Theorem~\ref{thm:dgf7dds}).
\end{proof}

Under the assumptions of Lemma~\ref{lemma:6GEC43sdf}, using \eqref{eq:GregoryMN6Remai}, \eqref{eq:6GregoryMN0046}, and \eqref{eq:PieceWPol48R} we immediately obtain the following identities
\begin{eqnarray*}
R^p_{m,\infty}[g] &=& \sum_{k=m}^{\infty}\int_0^1\rho^{p+1}_k[g](t){\,}dt ~=~ \int_0^1\sum_{k=m}^{\infty}\rho^{p+1}_k[g](t){\,}dt\\
&=& \int_0^1(f^p_{m}[g](t)-\Sigma g(t)){\,}dt
\end{eqnarray*}
and
\begin{equation}\label{eq:RpmiIneq904}
R^p_{m,\infty}[g] ~=~ -\sum_{k=m}^{\infty}J^{p+1}[g](k) ~=~ \int_{m}^{\infty}(g(t)-\overline{P}_p[g](t)){\,}dt.
\end{equation}
Moreover, if $g$ is $p$-convex or $p$-concave on $[m,\infty)$, the inequality \eqref{eq:saf65fs} reduces to
\begin{equation}\label{eq:RpmiIneq9}
|R^p_{m,\infty}[g]| ~=~  |J^{p+1}[\Sigma g](m)| ~\leq ~ \overline{G}_p{\,}|\Delta^pg(m)|{\,},
\end{equation}
which is also an immediate consequence of Corollary~\ref{cor:6GenStFo0BaIneqCo} (where a tighter inequality is also provided when $p\geq 1$).

Let us now provide a geometric interpretation of the remainder $R^p_{m,\infty}[g]$ when $g$ is $p$-convex or $p$-concave on $[m,\infty)$. Suppose for instance that $g$ is $p$-convex on $[m,\infty)$. The interpretation of Gregory's formula discussed in Section~\ref{sec:G8S3F5R} shows that, on the whole of the interval $[m,\infty)$, the graph of $g$ lies over or under that of $\overline{P}_p[g]$ according to whether $p$ is even or odd, and the remainder $|R^p_{m,\infty}[g]|$ is precisely the surface area between both graphs. Interestingly, the fact that this surface area converges to zero as $m\to_{\N}\infty$ by \eqref{eq:RpmiIneq9} provides a direct interpretation of the restriction of the generalized Stirling formula\index{Stirling's formula!generalized} to integer values.

This interpretation is particularly visual when $p=0$ or $p=1$. Consider for instance the case $p=1$ and suppose that $g$ is concave on $[m,\infty)$ (e.g., $g(x)=\ln x$). Then, the graph of $g$ on $[m,\infty)$ lies over (or on) the polygonal line through the points $(k,g(k))$ for all integers $k\geq m$. The value $|R^p_{m,\infty}[g]|$ is then the surface area between the graph of $g$ and this polygonal line. It is also the absolute value of the remainder in the trapezoidal rule on $[m,\infty)$.

We are now able to introduce an analogue of Euler's constant for any function $g$ lying in $\cC^0\cap\mathrm{dom}(\Sigma)$. We call it the \emph{generalized Euler constant}.

\begin{definition}[Generalized Euler's constant]\label{de:GEC587}\index{Euler's constant!generalized|textbf}
The \emph{generalized Euler constant} associated with a function $g\in\cC^0\cap\mathrm{dom}(\Sigma)$ is the number
$$
\gamma[g] ~=~ -R^p_{1,\infty}[g] ~=~ -J^{p+1}[\Sigma g](1){\,},\label{p:genEC}
$$
where $p=1+\deg g$.
\end{definition}

For instance, if $g$ lies in $\cC^0\cap\cD^0\cap\cK^0$, then using \eqref{eq:GregoryMN} we obtain
\begin{eqnarray}
\gamma[g] &=& \lim_{n\to\infty}\left(\sum_{k=1}^{n-1}g(k)-\int_1^ng(t){\,}dt\right)\label{eq:GEC000}\\
&=& \sum_{k=1}^{\infty}\left(g(k)-\int_k^{k+1}g(t){\,}dt\right),\nonumber
\end{eqnarray}
and this value represents the remainder in the ``left'' rectangle method on $[1,\infty)$ with the integer nodes $k=1,2,\ldots$. Similarly, if $g$ lies in $\cC^0\cap\cD^1\cap\cK^1$ and $\deg g=0$, then we get
\begin{eqnarray}
\gamma[g] &=& \lim_{n\to\infty}\left(\sum_{k=1}^{n-1}g(k)-\int_1^ng(t){\,}dt+\frac{1}{2}{\,}g(n)-\frac{1}{2}{\,}g(1)\right)\label{eq:GEC0001}\\
&=& \sum_{k=1}^{\infty}\left(g(k)-\int_k^{k+1}g(t){\,}dt+\frac{1}{2}{\,}\Delta g(k)\right),\nonumber
\end{eqnarray}
and this value represents the remainder in the trapezoidal rule on $[1,\infty)$ with the integer nodes $k=1,2,\ldots$.

Thus defined, the number $\gamma[g]$ generalizes to any function $g$ lying in $\cC^0\cap\mathrm{dom}(\Sigma)$ not only the classical Euler constant\index{Euler's constant} $\gamma$ (obtained when $g(x)=\frac{1}{x}$) but also the generalized Euler constant $\gamma[g]$ associated with a positive and strictly decreasing function $g$ as defined in \eqref{eq:GEC000} (see, e.g., Apostol \cite{Apo99} and Finch \cite[Section 1.5.3]{Fin03}). Moreover, as we will see in Section~\ref{sec:WF73}, this number plays a central role in the Weierstrassian form of $\Sigma g$ (which also justifies the choice $m=1$ in the definition of $\gamma[g]$).

The definition of $\gamma[g]$ does not require $g$ to be $p$-convex or $p$-concave on $[1,\infty)$. However, if this latter condition holds, then by \eqref{eq:RpmiIneq9} we have the inequality
\begin{equation}\label{eq:RpmiIneq91}
|\gamma[g]| ~\leq ~ \overline{G}_p{\,}|\Delta^pg(1)|
\end{equation}
and by Corollary~\ref{cor:6GenStFo0BaIneqCo} the following tighter inequality also holds when $p\geq 1$
\begin{equation}\label{eq:RpmiIneq91b}
|\gamma[g]| ~\leq ~ \int_0^1\left|\tchoose{t-1}{p}\right|\left|\Delta^{p-1}g(t+1)-\Delta^{p-1}g(1)\right|{\,}dt.
\end{equation}
We also provide and discuss finer bounds for $\gamma[g]$ in Appendix~\ref{chapter:E-GenWeIn5} (see Remark~\ref{rem:EzzBou0gam332}).

\begin{example}
If $g(x)=1/x$, then $\gamma[g]$ reduces to Euler's constant $\gamma$,\index{Euler's constant} as expected. Indeed, in this case we obtain
$$
\gamma[g] ~=~ -J^1[\psi](1) ~=~ \gamma.
$$
Using~\eqref{eq:GEC000}, we then retrieve the well-known formula
$$
\gamma ~=~ \lim_{n\to\infty}\left(\sum_{k=1}^n\frac{1}{k} -\ln n\right)
$$
and its classical geometric interpretation. If $g(x)=\ln x$, then the associated generalized Euler constant is
$$
\gamma[g] ~=~ -J^2[\ln\circ\Gamma](1) ~=~ -J(1) ~=~ -1+\frac{1}{2}\ln(2\pi) ~\approx ~ -0.081
$$
and we can see that it coincides with the associated asymptotic constant\index{asymptotic constant} $\sigma[g]$ (see Example~\ref{ex:Raab286}). Moreover, using \eqref{eq:GEC0001} we obtain the following formula
$$
\gamma[g] ~=~ \lim_{n\to\infty}\left(\ln n!+n-1-\left(n+\textstyle{\frac{1}{2}}\right)\ln n\right).
$$
The value $|\gamma[g]|=-\gamma[g]$ can then be interpreted as the surface area between the graph of $g$ on the unbounded interval $[1,\infty)$ and the polygonal line through the points $(k,g(k))$ for all integers $k\geq 1$. Moreover, Eq.~\eqref{eq:RpmiIneq91b} provides the following inequality
$$
|\gamma[g]| ~\leq ~ \ln 4-\frac{5}{4} ~\approx ~ 0.14.\qedhere
$$
\end{example}

\parag{A conversion formula between $\gamma[g]$ and $\sigma[g]$}\index{Euler's constant!generalized!in terms of the asymptotic constant} The following proposition, which immediately follows from \eqref{eq:Binet64378S} and the identity
$$
\gamma[g] ~=~ -J^{p+1}[\Sigma g](1),
$$
shows how the numbers $\gamma[g]$ and $\sigma[g]$ are related and provides an alternative way to compute the value of $\gamma[g]$.

\begin{proposition}\label{prop:linksSG46}
For any function $g$ lying in $\cC^0\cap\mathrm{dom}(\Sigma)$, we have
$$
\sigma[g] ~=~ \gamma[g]+\sum_{j=1}^pG_j\,\Delta^{j-1}g(1),
$$
where $p=1+\deg g$.
\end{proposition}

\parag{An integral form of\/ $\gamma[g]$}\index{Euler's constant!generalized!integral form} The following proposition shows that the classical integral representation of the Euler constant\index{Euler's constant}
$$
\gamma ~=~ \int_1^{\infty}\left(\frac{1}{\lfloor t\rfloor}-\frac{1}{t}\right){\,}dt
$$
can be generalized to the constant $\gamma[g]$ for any function $g$ lying in $\cC^0\cap\mathrm{dom}(\Sigma)$.

\begin{proposition}\label{prop:IntFogamma482}
For any $g\in\cC^0\cap\cD^p\cap\cK^p$, where $p=1+\deg g$, we have
$$
\gamma[g] ~=~ \int_1^{\infty}\bigg(\sum_{j=0}^pG_j\Delta^jg(\lfloor t\rfloor)-g(t)\bigg){\,}dt.
$$
In particular, when $\deg g=-1$, we have
$$
\gamma[g] ~=~ \int_1^{\infty}(g(\lfloor t\rfloor)-g(t)){\,}dt.
$$
\end{proposition}

\begin{proof}
Using \eqref{eq:Binet64378} and \eqref{eq:RpmiIneq904}, we obtain
$$
\gamma[g] ~=~ \sum_{k=1}^{\infty}J^{p+1}[g](k) ~=~ \sum_{k=1}^{\infty}\left(\sum_{j=0}^pG_j\,\Delta^jg(k)-\int_k^{k+1}g(t){\,}dt\right),
$$
which immediately provides the claimed formula.
\end{proof}

\parag{The principal indefinite sum of the generalized Binet function}\index{principal indefinite sum!of the generalized Binet function}\index{Binet's function!generalized} If $g$ lies in $\cC^0\cap\cD^p\cap\cK^p$ for some $p\in\N$, then the function $J^{p+1}[\Sigma g]$ lies in $\cD^0_{\R}$ by Theorem~\ref{thm:dgf7dds}, and hence so does
$$
\Delta J^{p+1}[\Sigma g] ~=~ J^{p+1}[g].
$$
If, in addition, $J^{p+1}[\Sigma g]$ lies in $\cK^0$, then by the uniqueness Theorem~\ref{thm:unic} we have that
$$
\Sigma J^{p+1}[g] ~=~ J^{p+1}[\Sigma g] -J^{p+1}[\Sigma g](1){\,}.
$$
Thus, if $p=1+\deg g$, then we obtain the identity
\begin{equation}\label{6FinPr5InBi2}
\Sigma J^{p+1}[g] ~=~ J^{p+1}[\Sigma g] + \gamma[g]{\,}.
\end{equation}

Now, suppose that we wish to show that a given function $f\colon\R_+\to\R$ satisfies the equation $f=J^{p+1}[\Sigma g]$ for some function $g$ lying in $\cC^0\cap\cD^p\cap\cK^p$, with $p=1+\deg g$. Using the uniqueness theorem with identity \eqref{6FinPr5InBi2}, we see that it is then enough to show that $\Delta f=J^{p+1}[g]$, $f(1)=-\gamma[g]$, and $f\in\cK^0$.

\begin{example}
Let $f\colon\R_+\to\R$ be defined by the equation $f(x)=\psi(x)-\ln x$ for $x>0$. To see that $f=J^1[\psi]$, it is enough to observe that $f$ lies in $\cK^0$, that $f(1)=-\gamma$, and that
$$
\Delta f(x) ~=~ \frac{1}{x}-\ln\left(1+\frac{1}{x}\right)
$$
is precisely the function $J^1[g](x)$ when $g(x)=1/x$.
\end{example}

\begin{example}
Binet established the following integral representation (see, e.g., Sasv\'ari \cite{Sas99})
$$
J^2[\ln\circ\Gamma](x) ~=~ J(x) ~=~ \int_0^{\infty}\left(\frac{1}{e^t-1}-\frac{1}{t}+\frac{1}{2}\right)\,\frac{e^{-x t}}{t}{\,}dt.
$$
Eq.~\eqref{6FinPr5InBi2} then provides a possible (though not immediate) proof of this identity.
\end{example}

\index{Euler's constant!generalized|)}

\chapter{Derivatives of multiple $\log\Gamma$-type functions}
\label{chapter:7}

In this chapter, we discuss the higher order differentiability properties of $\Sigma g$ when $g$ lies in $\cC^r\cap\cD^p\cap\cK^{\max\{p,r\}}$ for any $p,r\in\N$. In particular, we show the fundamental fact that $\Sigma g$ also lies in $\cC^r$ and that the sequence $n \mapsto D^rf^p_n[g]$ converges uniformly on any bounded subinterval of $\R_+$ to $D^r\Sigma g$.

We also show that the functions $(\Sigma g)^{(r)}$ and $\Sigma g^{(r)}$ differ by a constant and we investigate some properties of these functions, including asymptotic behaviors and an analogue of Euler's series representation of the constant $\gamma$. We present and discuss a procedure, that we call the ``elevator'' method, to compute $\Sigma g$ by first evaluating $\Sigma g^{(r)}$. Finally, we provide an alternative uniqueness result for higher order differentiable solutions to the equation $\Delta f=g$.

\section{Differentiability of multiple $\log\Gamma$-type functions}
\label{sec:Diff2Mu8}

In this first section we investigate the higher order differentiability of the function $\Sigma g$ when $g$ is of class $\cC^r$ for some $r\in\N$. We start with the following preliminary, but very important result.

\begin{proposition}\label{prop:7Diff5Sig216}
If $g$ lies in $\cC^r\cap\cD^p\cap\cK^{\max\{p,r\}}$ for some $r,p\in\N$, then the function $\Sigma g$ lies in $\cC^r\cap\cD^{p+1}\cap\cK^{\max\{p,r\}}$.
\end{proposition}

\begin{proof}
If $g$ lies in $\cC^r\cap\cD^p\cap\cK^{\max\{p,r\}}$ for some $r,p\in\N$, then clearly it also lies in $\cC^r\cap\cD^{\max\{p,r\}}\cap\cK^{\max\{p,r\}}$. By Proposition~\ref{prp:56GathZZPro8}, $\Sigma g$ must lie in $\cD^{p+1}\cap\cK^{\max\{p,r\}}$. Let us now show that it also lies in $\cC^r$.

We first observe that $g^{(r)}$ lies in $\cC^0\cap\cD^{(p-r)_+}\cap\cK^{(p-r)_+}$. This is clear if $r\leq p$ by Propositions~\ref{prop:LMpGpLMp1}. If $r>p$, then we first see that $g^{(p)}$ lies in $\cC^{r-p}\cap\cD^0\cap\cK^{r-p}$, and hence also in $\cK^0\cap\cK^1$. Using Proposition~\ref{prop:sa6f575sf}(b) repeatedly, we then see that $g^{(r)}$ lies in $\cC^0\cap\cD^{-1}\cap\cK^0$.

By Proposition~\ref{prop:intMLGt}, $\Sigma g^{(r)}$ must lie in $\cC^0\cap\cD^{(p-r)_++1}\cap\cK^{(p-r)_+}$. Hence, there exists $F\in\cC^r$ such that $F^{(r)}=\Sigma g^{(r)}$. By Proposition~\ref{prop:LMpGpLMp1}, $F$ must lie in $\cK^{\max\{p,r\}}$. Now, we also have
$$
D^r\Delta F ~=~ \Delta F^{(r)} ~=~ \Delta\Sigma g^{(r)} ~=~ g^{(r)},
$$
which shows that $\Delta (F+P)=g$ for some polynomial $P$ of degree at most $r$. By Corollary~\ref{cor:convCones48} we have that $F+P$ lies in $\cK^{\max\{p,r\}}$. But then, by the uniqueness Theorem~\ref{thm:unic} we must have $F+P=\Sigma g+c$ for some $c\in\R$. Hence $\Sigma g$ lies in $\cC^r$.
\end{proof}

\begin{remark}
If $g$ lies in $\cC^r\cap\cD^p\cap\cK^p$ for some integers $0\leq r < p$, then the function $\Sigma g$ lies in $\cC^r$ by Proposition~\ref{prop:7Diff5Sig216}. Interestingly, this result can also be established very easily using the following argument. Let $n\in\N$ be so that $\Sigma g$ is $p$-convex or $p$-concave on $I_n=(n,\infty)$. By Lemma~\ref{lemma:PrelKp}(a), the function $\Sigma g$ lies in $\cC^{p-1}(I_n)$ and hence also in $\cC^r(I_n)$. Using \eqref{eq:56zzSec32S6}, we immediately obtain that $\Sigma g$ lies in $\cC^r$.
\end{remark}

We now present the following important and very surprising result. It shows that Proposition~\ref{prop:7Diff5Sig216} no longer holds when $r>p$ if we ask $g$ to lie in $\cK^p$ instead of $\cK^{\max\{p,r\}}$. Since the proof is somewhat technical, we defer it to Appendix~\ref{chapter:Diff7SigLog6}.

\begin{proposition}\label{prop:71Surp4Not9Pres}
For every $p\in\N$, there exists a function $g$ lying in $\cC^{p+1}\cap\cD^p\cap\cK^p$ for which $\Sigma g$ does not lie in $\cC^{p+1}$. Thus, the operator $\Sigma$ does not always preserve differentiability when the order of differentiability exceeds that of convexity.
\end{proposition}

\begin{proof}
See Appendix~\ref{chapter:Diff7SigLog6}.
\end{proof}

The next theorem is the central result of this section. In this theorem, we recall the fundamental result given in Proposition~\ref{prop:7Diff5Sig216} and we show that, under the same assumptions, the sequence $n \mapsto D^rf^p_n[g]$ converges uniformly on any bounded subinterval of $\R_+$ to $D^r\Sigma g$. We first consider a technical lemma.

\begin{lemma}\label{lemma:7CruL044rtb}
Let $g$ lie in $\cC^r\cap\cD^p\cap\cK^p$ for some integers $0\leq r\leq p$. Then, for any $n\in\N$ the function $\rho^{p+1}_n[\Sigma g]$ lies in $\cC^r$. Moreover, the sequence $n\mapsto D^r\rho^{p+1}_n[\Sigma g]$ converges uniformly on any bounded subset of $\R_+$ to zero.
\end{lemma}

\begin{proof}
By Proposition~\ref{prop:7Diff5Sig216}, we have that $\Sigma g$ lies in $\cC^r$. Using \eqref{eq:deflambdapt} it is then clear that, for any $n\in\N$, the function $\rho^{p+1}_n[\Sigma g]$ lies in $\cC^r$.

Let us now show the second part of the lemma. Negating $g$ if necessary, we may assume that it lies in $\cK^p_-$. In this case, $D^r\Sigma g$ must lie in $\cK^{p-r}_+$ by Proposition~\ref{prop:LMpGpLMp1}. Let $n\geq p$ be an integer so that $g$ is $p$-concave on $[n,\infty)$. Using Proposition~\ref{prop:2A4DiffInt4Pol} repeatedly, we can see that there exist $p-r+1$ pairwise distinct points $\xi_0^n,\ldots,\xi^n_{p-r}\in (0,p)$ such that
$$
D^r_xP_p[\Sigma g](n,\ldots,n+p;n+x) ~=~ P_{p-r}[D^r\Sigma g](n+\xi_0^n,\ldots,n+\xi^n_{p-r};n+x).
$$
Let us now fix $x>0$. Using \eqref{eq:LErrIn} and then \eqref{IntError582} and \eqref{eq:DivDiffRec7}, we obtain
\begin{eqnarray*}
D^r\rho^{p+1}_n[\Sigma g](x) &=& D^r\Sigma g[n+\xi_0^n,\ldots,n+\xi^n_{p-r},n+x]{\,}\prod_{i=0}^{p-r}(x-\xi^n_i)\\
&=& A_n\,\prod_{i=1}^{p-r}(x-\xi^n_i),
\end{eqnarray*}
if $x\neq\xi_i^n$ for $i=0,\ldots,p-r$, and $D^r\rho^{p+1}_n[\Sigma g](x)=0$, otherwise, where
$$
A_n ~=~ D^r\Sigma g[n+\xi_1^n,\ldots,n+\xi^n_{p-r},n+x] -D^r\Sigma g[n+\xi_0^n,\ldots,n+\xi^n_{p-r}].
$$
Now, on the one hand, we clearly have
$$
\prod_{i=1}^{p-r}|x-\xi^n_i| ~\leq ~ c_x^{p-r}.
$$
where $c_x=\max\{p,\lceil x\rceil\}$. On the other hand, using Lemma~\ref{lemma:pCInc5} (with the fact that $D^r\Sigma g$ lies in $\cK^{p-r}_+$) and then \eqref{eq:DelDD62}, we obtain
\begin{eqnarray*}
|A_n| &\leq & \left|D^r\Sigma g[n+c_x,\ldots,n+c_x+p-r]-D^r\Sigma g[n-p+r,\ldots,n]\right|\\
&=& \frac{1}{(p-r)!}\left|\Delta^{p-r}D^r\Sigma g(n+c_x)-\Delta^{p-r}D^r\Sigma g(n-p+r)\right|\\
&=& \frac{1}{(p-r)!}{\,}\sum_{j=-p+r}^{c_x-1}|\Delta^{p-r}D^rg(n+j)|.
\end{eqnarray*}
Thus, for any bounded subinterval $E$ of $\R_+$, we obtain the inequality
$$
\sup_{x\in E}\left|D^r\rho^{p+1}_n[\Sigma g](x)\right| ~\leq ~ \frac{c_{\sup E}^{p-r}}{(p-r)!}{\,}\sum_{j=-p+r}^{c_{\sup E}-1}|\Delta^{p-r}D^rg(n+j)|.
$$
But the latter sum converges to zero as $n\to_{\N}\infty$ since $D^rg$ lies in $\cD^{p-r}\cap\cK^{p-r}$ by Proposition~\ref{prop:LMpGpLMp1}. This completes the proof of the lemma.
\end{proof}

\begin{theorem}[Higher order differentiability of multiple $\log\Gamma$-type functions]\label{thm:TBTDiff}
Let $g$ lie in $\cC^r\cap\cD^p\cap\cK^{\max\{p,r\}}$ for some $r,p\in\N$. The following assertions hold.
\begin{enumerate}
\item[(a)] $\Sigma g$ lies in $\cC^r\cap\cD^{p+1}\cap\cK^{\max\{p,r\}}$.
\item[(b)] The sequence $n\mapsto D^rf^p_n[g]$ converges uniformly on any bounded subset of\/ $\R_+$ to $D^r\Sigma g$.
\end{enumerate}
\end{theorem}

\begin{proof}
Assertion (a) immediately follows from Proposition~\ref{prop:7Diff5Sig216}. When $r\leq p$, assertion (b) immediately follows from Lemma~\ref{lemma:7CruL044rtb} and identity \eqref{eq:33ConvSig52}. Let us now assume that $r>p$. Using \eqref{eq:33ConvSig52} and then \eqref{eq:deflambdapt} and \eqref{eq:56zzSec32S6} we obtain
$$
D^rf^p_n[g](x) ~=~ D^r\Sigma g(x)-D^r\Sigma g(x+n) ~=~ -\sum_{k=0}^{n-1}g^{(r)}(x+k).
$$
By Proposition~\ref{prop:LMpGpLMp1}, we have that $g^{(p)}$ lies in $\cC^{r-p}\cap\cD^0\cap\cK^{r-p}$, and hence also in $\cK^0\cap\cK^1$. Using Proposition~\ref{prop:sa6f575sf}(b) repeatedly, we then see that $g^{(r)}$ lies in $\cC^0\cap\cD^{-1}\cap\cK^0$. Thus, we can apply Theorem~\ref{thm:uniqzz0} to the function $g^{(r)}$, with $f=D^r\Sigma g$. Since $f$ lies in $\cC^0\cap\cD^0\cap\cK^0$ by assertion (a) and Proposition~\ref{prop:LMpGpLMp1}, it follows from Theorem~\ref{thm:uniqzz0} that the sequence $n\mapsto D^rf^p_n[g]$ converges uniformly on $\R_+$ to $f-f(\infty)=f=D^r\Sigma g$.
\end{proof}

\begin{example}\label{ex:6afdsdf2}
The function $g(x)=\ln x$ clearly lies in $\cC^{\infty}\cap\cD^1\cap\cK^{\infty}$. Using Theorem~\ref{thm:TBTDiff}, we now see that the function $\Sigma g(x)=\ln\Gamma(x)$ lies in $\cC^{\infty}\cap\cD^2\cap\cK^{\infty}$. Moreover, for any $r\in\N^*$, we have
\begin{eqnarray*}
\psi_{r-1}(x) &=& D^r\ln\Gamma(x) ~=~ \lim_{n\to\infty}D^rf^1_n[\ln](x)\\
&=& \lim_{n\to\infty}\left(0^{r-1}\ln n +(-1)^r(r-1)!\,\sum_{k=0}^{n-1}\frac{1}{(x+k)^r}\right).
\end{eqnarray*}
If $r=1$, then we obtain
$$
\psi(x) ~=~ \lim_{n\to\infty}\left(\ln n-\sum_{k=0}^{n-1}\frac{1}{x+k}\right).
$$
If $r\geq 2$, then we get (compare with, e.g., Srivastava and Choi \cite[p.~33]{SriCho12})
$$
\psi_{r-1}(x) ~=~ (-1)^r(r-1)!\,\zeta(r,x),
$$
where $s\mapsto\zeta(s,x)$ is the Hurwitz zeta function\index{Hurwitz zeta function} (see Example~\ref{ex:HuZe82}).
\end{example}

\section{Some properties of the derivatives}

In this section, we investigate the functions $(\Sigma g)^{(r)}$ and $\Sigma g^{(r)}$ and some of their properties. We also show how the asymptotic behaviors of these functions can be analyzed from results of Chapter~\ref{chapter:6}, including the generalized Stirling formula. Finally, we provide a series representation of the asymptotic constant $\sigma[g]$ as an analogue of Euler's series representation of $\gamma$.

In the next proposition, we essentially establish the fact that the functions $(\Sigma g)^{(r)}$ and $\Sigma g^{(r)}$ are equal up to an additive constant. This result will have several important consequences in this and the next chapters.

\begin{proposition}\label{prop:fdsf6sfd}
Let $g$ lie in $\cC^r\cap\cD^p\cap\cK^{\max\{p,r\}}$ for some $p\in\N$ and $r\in\N^*$. Then $g^{(r)}$ lies in $\cC^0\cap\cD^{(p-r)_+}\cap\cK^{(p-r)_+}$. Moreover, for any $x>0$ we have
\begin{equation}\label{eq:rd0bb}
(\Sigma g)^{(r)}(x)-\Sigma g^{(r)}(x) ~=~ (\Sigma g)^{(r)}(1) ~=~ g^{(r-1)}(1) - \sigma[g^{(r)}]{\,}.
\end{equation}
If $r>p$, then
$$
\sigma[g^{(r)}] ~=~ g^{(r-1)}(1) + \sum_{k=1}^{\infty}g^{(r)}(k).
$$
\end{proposition}

\begin{proof}
As already observed in the proof of Proposition~\ref{prop:7Diff5Sig216}, the first claim follows from Propositions~\ref{prop:LMpGpLMp1} and \ref{prop:sa6f575sf}(b). Moreover, we have that $\Sigma g$ lies in $\cC^r\cap\cD^{p+1}\cap\cK^{\max\{p,r\}}$. Let us now prove \eqref{eq:rd0bb}. By Proposition~\ref{prop:LMpGpLMp1}, the function $\varphi_1=(\Sigma g)^{(r)}$ is a solution in $\cK^{(p-r)_+}$ to the equation $\Delta\varphi=g^{(r)}$. By the existence Theorem~\ref{thm:exist}, the function $\varphi_2=\Sigma g^{(r)}$ is also a solution in $\cK^{(p-r)_+}$. Thus, by the uniqueness Theorem~\ref{thm:unic}, we must have $(\Sigma g)^{(r)}-\Sigma g^{(r)}=c$ for some $c\in\R$, and hence we also have $(\Sigma g)^{(r)}(1)=c$.

Now, for any $x>0$, using \eqref{eq:ds68ffds} we then get
\begin{eqnarray*}
g^{(r-1)}(1)-\sigma[g^{(r)}] &=& g^{(r-1)}(x)-\int_x^{x+1}\Sigma g^{(r)}(t){\,}dt\\
&=& c+g^{(r-1)}(x)-\int_x^{x+1}(\Sigma g)^{(r)}(t){\,}dt.
\end{eqnarray*}
Evaluating the latter integral, we then obtain
\begin{eqnarray*}
g^{(r-1)}(1)-\sigma[g^{(r)}] &=& c+g^{(r-1)}(x)-(\Sigma g)^{(r-1)}(x+1)+(\Sigma g)^{(r-1)}(x)\\
&=& c+g^{(r-1)}(x)-\Delta (\Sigma g)^{(r-1)}(x)\\
&=& c+g^{(r-1)}(x)-(\Delta\Sigma g)^{(r-1)}(x)\\
&=& c,
\end{eqnarray*}
which proves \eqref{eq:rd0bb}. Finally, if $r>p$, then we have that $g^{(r-1)}$ lies in $\cC^1\cap\cD^0\cap\cK^1$ and that $g^{(r)}$ lies in $\cC^0\cap\cD^{-1}\cap\cK^0$ by Proposition~\ref{prop:sa6f575sf}(b). The last part of the statement then follows from applying Proposition~\ref{prop:diffzz0} to the function $g^{(r)}$.
\end{proof}

\begin{example}\label{ex:digamm59a}
The function $g(x)=\frac{1}{x}$ lies in $\cC^{\infty}\cap\cD^0\cap\cK^{\infty}$ and all its derivatives lie in $\cK^0$. By Theorem~\ref{thm:TBTDiff}, the function
$$
\Sigma g(x) ~=~ \sum_{k=0}^{\infty}\left(\frac{1}{k+1}-\frac{1}{x+k}\right) ~=~ H_{x-1} ~=~ \psi(x)+\gamma
$$
lies in $\cC^{\infty}\cap\cD^1\cap\cK^{\infty}$. Moreover, the series can be differentiated term by term infinitely many times and hence, for any $r\in\N^*$, we have
$$
(\Sigma g)^{(r)}(x) ~=~ \sum_{k=0}^{\infty}(-1)^{r+1}\frac{r!}{(x+k)^{r+1}} ~=~ \psi_r(x).
$$
By Proposition~\ref{prop:fdsf6sfd}, we also have
\begin{eqnarray*}
\sigma[g^{(r)}] &=& -(-1)^r(r-1)! + (-1)^r{\,}r!\,\sum_{k=1}^{\infty}\frac{1}{k^{r+1}}\\
&=& (-1)^r{\,}(r-1)!\left(r\,\zeta(r+1)-1\right),
\end{eqnarray*}
where $s\mapsto\zeta(s)$ is the Riemann zeta function\index{Riemann zeta function}.
\end{example}

In the next proposition we show the remarkable fact that the asymptotic equivalence \eqref{eq:Ratio44t2} still holds if we differentiate both sides.

\begin{proposition}\label{prop:conv6v6r}
Let $g$ lie in $\cC^r\cap\cD^p\cap\cK^{\max\{p,r\}}$ for some $p\in\N$ and $r\in\N^*$, and let $a\geq 0$. When $D^r\Sigma g$ vanishes at infinity, we also assume that
$$
D^r\Sigma g(n+1) ~\sim ~ D^r\Sigma g(n)\qquad\text{as $n\to_{\N}\infty$}.
$$
Then we have
$$
D^r\Sigma g(x+a) ~\sim ~ D^r_x\int_x^{x+1}\Sigma g(t){\,}dt ~=~ g^{(r-1)}(x)\qquad\text{as $x\to\infty$}.
$$
\end{proposition}

\begin{proof}
By Proposition~\ref{prop:fdsf6sfd}, we have that $g^{(r)}$ lies in $\cC^0\cap\cD^{(p-r)_+}\cap\cK^{(p-r)_+}$. Moreover, for any $x>0$ we have
$$
D^r\Sigma g(x+a) ~=~ c+\Sigma g^{(r)}(x+a)
$$
and, using \eqref{eq:ds68ffds},
\begin{eqnarray*}
D^r_x\int_x^{x+1}\Sigma g(t){\,}dt &=& g^{(r-1)}(x) ~=~ \int_x^{x+1}(\Sigma g)^{(r)}(t){\,}dt\\
&=& c+\int_x^{x+1}\Sigma g^{(r)}(t){\,}dt,
\end{eqnarray*}
where $c=g^{(r-1)}(1)-\sigma[g^{(r)}]$. The result then immediately follows from applying Proposition~\ref{prop:conv6v6} to the function $g^{(r)}$.
\end{proof}

\begin{example}
Applying Proposition~\ref{prop:conv6v6r} to the function $g(x)=\ln x$, for any $a\geq 0$ we obtain the equivalences
$$
\ln\Gamma(x+a) ~\sim ~ x\ln x{\,},\qquad\psi(x+a) ~\sim ~ \ln x\qquad\text{as $x\to\infty$},
$$
and for any $\nu\in\N$,
$$
\psi_{\nu +1}(x+a) ~\sim ~ (-1)^{\nu}\,\frac{\nu !}{x^{\nu +1}}\qquad\text{as $x\to\infty$}.\qedhere
$$
\end{example}

In the next two propositions, we mainly investigate how the convergence results in \eqref{eq:convRes79} and \eqref{eq:dgf7dds} are modified when the function $g$ is replaced with one of its higher order derivatives. The second proposition can be regarded as the ``integrated'' version of the first one, and hence it naturally  involves the generalized Binet function.\index{Binet's function!generalized}

\begin{proposition}\label{prop:gWiDr5}
Let $g$ lie in $\cC^r\cap\cD^p\cap\cK^{\max\{p,r\}}$ for some $p\in\N$ and $r\in\N^*$, and let $a\geq 0$. The following assertions hold.
\begin{enumerate}
\item[(a)] $g^{(r)}$ lies in $\cR_{\R}^{(p-r)_+}$ and both $\Sigma g^{(r)}$ and $(\Sigma g)^{(r)}$ lie in $\cR_{\R}^{(p-r)_++1}$.
\item[(b)] For any $q\in\N$, the function $x\mapsto\rho_x^{q+1}[\Sigma g](a)$ lies in $\cC^r$ and we have
    $$
    D_x^r\rho_x^{q+1}[\Sigma g](a) ~=~ \rho_x^{q+1}[\Sigma g^{(r)}](a).
    $$
\item[(c)] We have that $\rho_x^{(p-r)_++1}[\Sigma g^{(r)}](a) \to 0$ and $D_x^r\rho_x^{p+1}[\Sigma g](a) \to 0$ as $x\to\infty$.
\end{enumerate}
\end{proposition}

\begin{proof}
By Proposition~\ref{prop:fdsf6sfd}, the function $g^{(r)}$ lies in $\cC^0\cap\cD^{(p-r)_+}\cap\cK^{(p-r)_+}$. This immediately proves assertion (a). Now, using \eqref{eq:deflambdapt} and then \eqref{eq:rd0bb} we get
\begin{eqnarray*}
D_x^r\rho_x^{q+1}[\Sigma g](a) &=& \Sigma g^{(r)}(x+a)-\Sigma g^{(r)}(x)-\sum_{j=1}^q\tchoose{a}{j}\,\Delta^{j-1}g^{(r)}(x)\\
&=& \rho_x^{q+1}[\Sigma g^{(r)}](a),
\end{eqnarray*}
which proves assertion (b). Assertion (c) follows from assertions (a) and (b) and the fact that $\cR_{\R}^{(p-r)_++1}\subset\cR_{\R}^{p+1}$.
\end{proof}

\begin{proposition}\label{prop:gSfDr4}
Let $g$ lie in $\cC^r\cap\cD^p\cap\cK^{\max\{p,r\}}$ for some $p\in\N$ and $r\in\N^*$. The following assertions hold.
\begin{enumerate}
\item[(a)] For any $q\in\N$, the function $J^{q+1}[\Sigma g]$ lies in $\cC^r$ and we have
$$
D^r J^{q+1}[\Sigma g] ~=~ J^{q+1}[\Sigma g^{(r)}].
$$
In particular, we have $\sigma[g^{(r)}]=-D^rJ^1[\Sigma g](1)$.
\item[(b)] We have that $J^{(p-r)_++1}[\Sigma g^{(r)}](x) \to 0$ and $D^r J^{p+1}[\Sigma g](x) \to 0$ as $x\to\infty$. In particular, if $r>p$, then $(\Sigma g)^{(r)}\to 0$ as $x\to\infty$.
\item[(c)] We have
$$
D_x^r \int_0^1\rho_x^{p+1}[\Sigma g](t){\,}dt ~=~ \int_0^1D_x^r \rho_x^{p+1}[\Sigma g](t){\,}dt.
$$
\end{enumerate}
\end{proposition}

\begin{proof}
Using \eqref{eq:Binet64378S} and \eqref{eq:rd0bb}, we get
\begin{eqnarray*}
D^r J^{q+1}[\Sigma g](x) &=&
\Sigma g^{(r)}(x)-\sigma[g^{(r)}]-\int_1^xg^{(r)}(t){\,}dt +\sum_{j=1}^qG_j\,\Delta^{j-1}g^{(r)}(x)\\
&=& J^{q+1}[\Sigma g^{(r)}](x),
\end{eqnarray*}
which proves assertion (a). Now, setting $q=p$ in these equations we obtain
$$
D^r J^{p+1}[\Sigma g](x) ~=~ J^{(p-r)_++1}[\Sigma g^{(r)}](x)+\sum_{j=(p-r)_++1}^pG_j\,\Delta^{j-1}g^{(r)}(x).
$$
Since $g^{(r)}$ lies in $\cC^0\cap\cD^{(p-r)_+}\cap\cK^{(p-r)_+}$, this latter expression vanishes at infinity. This proves assertion (b). Finally, using Proposition~\ref{prop:gWiDr5} and assertion (a) we get
\begin{eqnarray*}
\int_0^1D_x^r \rho_x^{p+1}[\Sigma g](t){\,}dt &=& \int_0^1\rho_x^{p+1}[\Sigma g^{(r)}](t){\,}dt ~=~ -J^{p+1}[\Sigma g^{(r)}](x)\\
&=& -D^r J^{p+1}[\Sigma g](x) ~=~ D_x^r \int_0^1\rho_x^{p+1}[\Sigma g](t){\,}dt,
\end{eqnarray*}
which proves assertion (c).
\end{proof}

Assertion (c) of Proposition~\ref{prop:gWiDr5} reveals a very important fact. It shows that the convergence result in \eqref{eq:convRes79} still holds if we replace $g$ with $g^{(r)}$ and $p$ with $(p-r)_+$. But it also says that this new result can also be obtained by differentiating $r$ times both sides of \eqref{eq:convRes79} and then removing the terms that vanish at infinity.

Similarly, assertion (b) of Proposition~\ref{prop:gSfDr4} shows that this property also applies to the generalized Stirling formula \eqref{eq:dgf7dds}.

\begin{example}
The function $g(x)=\ln x$ lies in $\cC^{\infty}\cap\cD^1\cap\cK^{\infty}$ and its derivative $g'(x)=\frac{1}{x}$ lies in $\cC^{\infty}\cap\cD^0\cap\cK^{\infty}$. For any $a\geq 0$, the limit in \eqref{eq:convRes79} reduces to
$$
\ln\Gamma(x+a)-\ln\Gamma(x)-a\ln x ~\to ~ 0\qquad\text{as $x\to\infty$}.
$$
If we replace $g$ with $g'$ and set $p=0$ in \eqref{eq:convRes79}, we get
$$
\psi(x+a)-\psi(x) ~\to ~ 0\qquad\text{as $x\to\infty$}.
$$
However, this latter limit can also be obtained by differentiating both sides of the previous limit and then removing the term ($-\frac{a}{x}$) that vanishes at infinity.

Now, applying the generalized Stirling formula \eqref{eq:dgf7dds} to the function $g(x)=\ln x$, we clearly retrieve the classical Stirling formula
$$
\ln\Gamma(x)-\frac{1}{2}\ln(2\pi)+x-\left(x-\frac{1}{2}\right)\ln x ~\to ~ 0\qquad\text{as $x\to\infty$}.
$$
Proceeding similarly as above, we then obtain
$$
\psi(x)-\ln x ~\to ~ 0\qquad\text{as $x\to\infty$},
$$
which is actually the analogue of Stirling's formula for the digamma function.\index{digamma function}
\end{example}

\begin{remark}
To emphasize the similarities between Propositions~\ref{prop:gWiDr5} and \ref{prop:gSfDr4}, we could for instance extend our formalism a bit further as follows. For any $p\in\N$ and any $\S\in\{\N,\R\}$, let $\mathcal{J}^p_{\S}$ denote the set of continuous functions $g\colon\R_+\to\R$ having the asymptotic property that
$$
J^p[g](t) ~\to ~0\qquad\text{as $t\to_{\S}\infty$}.
$$
This new definition enables one to formalize some results more easily. For instance, using \eqref{eq:Binet6141378} we clearly obtain that
$$
\mathcal{J}^p_{\S}\cap\cD^p_{\S} ~=~ \mathcal{J}^{p+1}_{\S}\cap\cD^p_{\S}
$$
and this identity could be used to establish assertion (b) of Proposition~\ref{prop:gSfDr4} from assertion (a). To give another example, we can see that \eqref{eq:dgf7ddsp1} actually means that
$$
\cC^0\cap\cD^p\cap\cK^p ~\subset ~\mathcal{J}^p_{\R}.
$$
Note also that the generalized Stirling formula simply states that $\Sigma g$ lies in $\mathcal{J}^{p+1}_{\R}$ whenever $g$ lies in $\cC^0\cap\cD^p\cap\cK^p$.
\end{remark}

\parag{Taylor series expansion of $\Sigma g$} Suppose that $g$ lies in $\cC^{\infty}\cap\cD^p\cap\cK^{\infty}$ for some $p\in\N$. We know from Proposition~\ref{prop:gSfDr4} that
$$
\sigma[g^{(k)}] ~=~ -D^kJ^1[\Sigma g](1),\qquad k\in\N.
$$
Thus, the exponential generating function (see, e.g., Graham {\em et al.} \cite[Chapter~7]{GraKnuPat94}) for the sequence $n\mapsto\sigma[g^{(n)}]$ is defined by the equation
\begin{eqnarray}
\sum_{k=1}^{\infty}\sigma[g^{(k)}]\,\frac{x^k}{k!} &=& -J^1[\Sigma g](x+1)\label{eq:egf5590}\\
&=& \sigma[g]+\int_1^{x+1}g(t){\,}dt - \Sigma g(x+1).\nonumber
\end{eqnarray}
Denoting this exponential generating function by $\mathrm{egf}_{\sigma}[g](x)$, the previous equation reduces to
$$
\mathrm{egf}_{\sigma}[g](x) ~=~ -J^1[\Sigma g](x+1){\,}.
$$
If the function $J^1[\Sigma g]$ is real analytic at $1$, then the series in \eqref{eq:egf5590} converges in some neighborhood of $x=0$. Similarly, if the function $\Sigma g$ is real analytic at $1$, then the following Taylor series expansion
\begin{equation}\label{eq:Taylor52}
\Sigma g(x+1) ~=~ \sum_{k=1}^{\infty}(\Sigma g)^{(k)}(1)\,\frac{x^k}{k!}
\end{equation}
holds in some neighborhood of $x=0$, where the numbers $(\Sigma g)^{(k)}(1)$ for $k\in\N^*$ can also be computed through \eqref{eq:rd0bb}.

\begin{example}
Consider again the functions $g(x)=\ln x$ and $\Sigma g(x)=\ln\Gamma(x)$. We know from Example~\ref{ex:6afdsdf2} that
$$
D\ln\Gamma(1) ~=~ \psi(1) ~=~ \lim_{n\to\infty}\left(\ln n-\sum_{k=1}^n\frac{1}{k}\right) ~=~ -\gamma{\,},
$$
and that for any integer $k\geq 2$
$$
D^k\ln\Gamma(1) ~=~ \psi_{k-1}(1) ~=~ (-1)^k{\,}(k-1)!\,\zeta(k).
$$
We then obtain the following Taylor series expansion
$$
\ln\Gamma(x+1) ~=~ -\gamma x+\sum_{k=2}^{\infty}(-1)^k\,\frac{\zeta(k)}{k}{\,}x^k{\,},\qquad |x|<1.
$$
The values of the sequence $n\mapsto\sigma[g^{(n)}]$ can be obtained using \eqref{eq:rd0bb} or \eqref{eq:egf5590}. We get
$$
\sigma[g] ~=~ -1+\frac{1}{2}\ln(2\pi),\qquad\sigma[g'] ~=~ \gamma,
$$
and for any integer $k\geq 2$
$$
\sigma[g^{(k)}] ~=~ (-1)^k(k-2)!{\,}(1-(k-1)\zeta(k)){\,}.\qedhere
$$
\end{example}

\parag{Analogues of Euler's series representation of $\gamma$} Integrating both sides of \eqref{eq:Taylor52} on $(0,1)$ (assuming that the series can be integrated term by term), we obtain the identity\index{Euler's series representation of $\gamma$!analogue}
\begin{equation}\label{eq:EulerAnal5571}
\sigma[g] ~=~ \sum_{k=1}^{\infty}(\Sigma g)^{(k)}(1)\,\frac{1}{(k+1)!}{\,}.
\end{equation}
Similarly, integrating both sides of \eqref{eq:egf5590} on $(0,1)$ (assuming again that the series can be integrated term by term), we obtain the identity
\begin{equation}\label{eq:EulerAnal55}
\sum_{k=0}^{\infty}\sigma[g^{(k)}]\,\frac{1}{(k+1)!} ~=~ \int_1^2(2-t){\,}g(t){\,}dt.
\end{equation}
Taking for instance $g(x)=\frac{1}{x}$ in \eqref{eq:EulerAnal5571}, we immediately retrieve Euler's series representation of $\gamma$\index{Euler's series representation of $\gamma$} (see, e.g., Srivastava and Choi \cite[p.~272]{SriCho12})
$$
\gamma ~=~ \sum_{k=2}^{\infty}(-1)^k\,\frac{\zeta(k)}{k}{\,}.
$$
This formula can also be obtained taking $g(x)=\frac{1}{x}$ in \eqref{eq:EulerAnal55} and using the straightforward identity
$$
\sigma[g^{(k)}] ~=~ (-1)^kk!\left(\zeta(k+1)-\frac{1}{k}\right),\qquad k\in\N^*.
$$

Considering different functions $g(x)$ in \eqref{eq:EulerAnal5571} and \eqref{eq:EulerAnal55} enables one to derive various interesting identities. A few applications are given in the following example.

\begin{example}\label{ex:7sd55fafsd1}
Taking $g(x)=\psi(x)$ in \eqref{eq:EulerAnal55} and using the straightforward identity
$$
\sigma[g^{(k)}] ~=~ \sigma[\psi_k] ~=~ (-1)^{k-1}(k-1)(k-1)!\,\zeta(k)\qquad k\in\N,~k\geq 2,
$$
we obtain
$$
\sum_{k=2}^{\infty}(-1)^k\frac{k-1}{k(k+1)}\,\zeta(k) ~=~ 2-\ln(2\pi){\,}.
$$
Similarly, taking $g(x)=\ln x$ and then $g(x)=\ln\Gamma(x)$ in \eqref{eq:EulerAnal5571} and \eqref{eq:EulerAnal55} we obtain the identities
\begin{eqnarray*}
\sum_{k=2}^{\infty}(-1)^k\frac{1}{k(k+1)}\,\zeta(k) &=& \frac{1}{2}\,\gamma -1+\frac{1}{2}\ln(2\pi){\,},\\
\sum_{k=2}^{\infty}(-1)^k\,\frac{1}{(k+1)(k+2)}{\,}\zeta(k) &=& \frac{1}{2}+\frac{1}{6}{\,}\gamma -2\ln A{\,},\\
\sum_{k=2}^{\infty}(-1)^k\frac{k-1}{k(k+1)(k+2)}{\,}\zeta(k) &=& \frac{5}{4}-\frac{1}{4}\ln(2\pi)-3\ln A{\,},
\end{eqnarray*}
where $A$ is Glaisher-Kinkelin's constant;\index{Glaisher-Kinkelin's constant} see also Srivastava and Choi \cite[Section 3.4]{SriCho12}.
\end{example}

\section{Finding solutions from derivatives}
\label{sec:FSFD63}

Given $r\in\N^*$ and a function $g\in\cC^r$, a solution $f\in\cC^r$ to the equation $\Delta f=g$ can sometimes be found more easily by first searching for an appropriate solution $\varphi\in\cC^0$ to the equation $\Delta\varphi=g^{(r)}$ and then calculating $f$ as an $r$th antiderivative of $\varphi$.

Let us first examine a very simple example to illustrate to which extent this approach can be easily and usefully applied.

\begin{example}\label{ex:FSFD6341}
Let $g\colon\R_+\to\R$ be defined by the equation
$$
g(x) ~=~ \int_1^x\ln t{\,}dt\qquad\text{for $x>0$}.
$$
Suppose that we search for a simple expression for the indefinite sum $\Sigma g$. We can apply Proposition~\ref{prop:fdsf6sfd} and observe that $g'$ lies in $\cC^{\infty}\cap\cD^1\cap\cK^{\infty}$ and hence that $g$ lies in $\cC^{\infty}\cap\cD^2\cap\cK^{\infty}$. Moreover, we have
$$
(\Sigma g)'(x) ~=~ c+\Sigma g'(x) ~=~ c+\ln\Gamma(x)
$$
for some $c\in\R$. Thus, we obtain
$$
\Sigma g(x) ~=~ c(x-1)+\int_1^x\ln\Gamma(t){\,}dt.
$$
To find the value of $c$, we then observe that
$$
0 ~=~ g(1) ~=~\Delta\Sigma g(1) ~=~ c+\int_1^2\ln\Gamma(t){\,}dt
$$
and hence $c=1-\frac{1}{2}\ln(2\pi)$ (see Example~\ref{ex:Raab286}). Alternatively, this value can also be obtained directly from \eqref{eq:rd0bb}; we have
$$
c ~=~ g(1)-\sigma[g'] ~=~ -\sigma[g'] ~=~ 1-\frac{1}{2}\ln(2\pi){\,}.
$$
Thus, this approach amounts to first searching for a simple expression for $\Sigma g'$, and then computing $\Sigma g$ using an antiderivative of $\Sigma g'$.

Finally, we get
$$
\Sigma g(x) ~=~ -1+\left(1-\frac{1}{2}\ln(2\pi)\right)x+\psi_{-2}(x),
$$
where $\psi_{-2}$ is the polygamma function\index{polygamma functions} $\psi_{-2}(x)=\int_0^x\ln\Gamma(t){\,}dt$.
\end{example}

The approach described in Example~\ref{ex:FSFD6341} is rather simple and can sometimes be very efficient. We will refer to this technique as \emph{the elevator method}.\index{elevator method|textbf} In very basic terms, to find $\Sigma g$ one proceeds as follows.
\begin{quote}
\begin{enumerate}
\item[Step 1.] \emph{We take the elevator, go down from the ground floor to the $r$th basement level, and get the function $\Sigma g^{(r)}$ easily.}
\item[Step 2.] \emph{We go back to the ground floor by converting the latter function into the function sought $\Sigma g$ using an $r$th antiderivative.}
\end{enumerate}
$$
\begin{array}{ccc}
\Delta f ~=~ g & & f ~=~ \Sigma g \\
\downarrow & & \uparrow \\
\Delta\varphi ~=~ g^{(r)} & \quad\to\quad & \varphi ~=~ \Sigma g^{(r)}
\end{array}
$$
\end{quote}

To our knowledge, this trick was investigated thoroughly by Krull~\cite{Kru49} and then by Dufresnoy and Pisot~\cite{DufPis63}.

In the next theorem we provide a general result based on this idea. This result is actually very general: it applies to any function $g\in\cC^r$, even if $\Sigma g$ is not defined (e.g., $g(x)=2^x$).

We first observe that if $\varphi\in\cC^0$ is a solution to the equation $\Delta\varphi=g^{(r)}$, then the map
$$
x ~ \mapsto ~ \int_x^{x+1}\varphi(t){\,}dt-g^{(r-1)}(x)
$$
has a zero derivative and hence it is constant on $\R_+$. In particular, it has a finite right limit at $x=0$.

\begin{theorem}[The elevator method]\label{thm:lift}\index{elevator method}
Let $r\in\N^*$, $a>0$, $g\in\cC^r$, and let $\varphi\colon\R_+\to\R$ be a continuous solution to the equation $\Delta\varphi=g^{(r)}$. Then there exists a solution $f\in\cC^r$ to the equation $\Delta f=g$ such that $f^{(r)}=\varphi$ if and only if
\begin{equation}\label{eq:COND11sd56}
\int_a^{a+1}\varphi(t){\,}dt ~=~ g^{(r-1)}(a).
\end{equation}
If any of these equivalent conditions holds, then $f$ is uniquely determined (up to an additive constant) by
\begin{equation}\label{eq:Taylor11sd56}
f(x) ~=~ f(a) + \sum_{k=1}^{r-1}c_k{\,}\frac{(x-a)^k}{k!} + \int_a^x \frac{(x-t)^{r-1}}{(r-1)!}{\,}\varphi(t){\,}dt,
\end{equation}
where, for $k=1,\ldots,r-1$,
\begin{equation}\label{eq:Taylor11sd562}
c_k ~=~ 
\sum_{j=0}^{r-k-1}\frac{B_j}{j!}{\,}\left(g^{(j+k-1)}(a)-\int_a^{a+1} \frac{(a+1-t)^{r-j-k}}{(r-j-k)!}{\,}\varphi(t){\,}dt\right).
\end{equation}
\end{theorem}

\begin{proof}
Condition \eqref{eq:COND11sd56} is clearly necessary. Indeed, we have
$$
\int_a^{a+1}\varphi(t){\,}dt ~=~ f^{(r-1)}(a+1)-f^{(r-1)}(a) ~=~ g^{(r-1)}(a).
$$
Let us show that it is sufficient. Since $\varphi$ is continuous, there exists $f\in\cC^r$ such that $f^{(r)}=\varphi$. Taylor's theorem then provides the expansion formula \eqref{eq:Taylor11sd56} with arbitrary parameters $c_k = f^{(k)}(a)$ for $k=1,\ldots,r-1$. Now we need to determine the parameters $c_1,\ldots,c_k$ for $f$ to be a solution to the equation $\Delta f=g$. To this extent, we need the following claim.

\begin{claim}
The function $f$ satisfies the equation $\Delta f=g$ if and only if $f^{(r)}$ satisfies the equation $\Delta f^{(r)}=g^{(r)}$ and $\Delta f^{(j)}(a)=g^{(j)}(a)$ for $j=0,\ldots,r-1$.
\end{claim}

\begin{proof}[Proof of the claim]
The condition is clearly necessary. To see that it is sufficient, we simply show by decreasing induction on $j$ that $\Delta f^{(j)}=g^{(j)}$. Clearly, this is true for $j=r$. Suppose that it is true for some integer $j$ satisfying $1\leq j\leq r$. For any $x>0$ we have
\begin{multline*}
\Delta f^{(j-1)}(x)-\Delta f^{(j-1)}(a) ~=~ \int_a^x \Delta f^{(j)}(t){\,}dt ~=~ \int_a^x g^{(j)}(t){\,}dt\\
=~ g^{(j-1)}(x)-g^{(j-1)}(a) ~=~ g^{(j-1)}(x)-\Delta f^{(j-1)}(a),
\end{multline*}
which shows that the result still holds for $j-1$.
\end{proof}

\noindent By the claim, $f$ satisfies the equation $\Delta f=g$ if and only if $\Delta f^{(j)}(a)=g^{(j)}(a)$ for $j=0,\ldots,r-1$. When $j=r-1$, the latter condition is nothing other than condition \eqref{eq:COND11sd56} and hence it is satisfied. Applying Taylor's theorem to $f^{(j)}$, we obtain
$$
f^{(j)}(a+1) - f^{(j)}(a) ~=~ \sum_{k=1}^{r-j-1}\frac{1}{k!}{\,}f^{(j+k)}(a) + \int_a^{a+1} \frac{(a+1-t)^{r-j-1}}{(r-j-1)!}{\,}\varphi(t){\,}dt{\,},
$$
and hence we see that the remaining $r-1$ conditions are
$$
\sum_{k=1}^{r-j-1}\frac{1}{k!}{\,}c_{j+k} ~=~ d_j,\qquad j=0,\ldots,r-2,
$$
where
\begin{eqnarray*}
d_j &=& g^{(j)}(a)-\int_a^{a+1} \frac{(a+1-t)^{r-j-1}}{(r-j-1)!}{\,}\varphi(t){\,}dt,\qquad j=0,\ldots,r-2,\\
c_k &=& f^{(k)}(a),\qquad k=1,\ldots,r-1.
\end{eqnarray*}
It is not difficult to see that these $r-1$ conditions form a consistent triangular system of $r-1$ linear equations in the $r-1$ unknowns $c_1,\ldots,c_{r-1}$. This establishes the uniqueness of $f$ up to an additive constant.

Let us now show that formula \eqref{eq:Taylor11sd562} holds. For $k=1,\ldots,r-1$, we have
$$
\sum_{j=0}^{r-k-1}\frac{B_j}{j!}{\,}d_{j+k-1} ~=~ \sum_{j=0}^{r-k-1}\frac{B_j}{j!}{\,}\sum_{i=1}^{r-j-k}\frac{1}{i!}{\,}c_{i+j+k-1}.
$$
Replacing $i$ with $i-j-k+1$ and then permuting the resulting sums, the latter expression reduces to
$$
\sum_{j=0}^{r-k-1}\frac{B_j}{j!}{\,}\sum_{i=j+k}^{r-1}\frac{1}{(i-j-k+1)!}{\,}c_i ~=~ \sum_{i=k}^{r-1}\frac{c_i}{(i-k+1)!}{\,}\sum_{j=0}^{i-k}\tchoose{i-k+1}{j}{\,}B_j{\,},
$$
that is, using \eqref{defBnN3},
$$
\sum_{i=k}^{r-1}\frac{c_i}{(i-k+1)!}{\,}0^{i-k} ~=~ c_k{\,}.
$$
This completes the proof of the theorem.
\end{proof}

Adding an appropriate constant to $\varphi$ if necessary in Theorem~\ref{thm:lift}, we can always assume that condition \eqref{eq:COND11sd56} holds. More precisely, the function $\varphi^{\star}=\varphi +C$, where
$$
C ~=~ g^{(r-1)}(a)-\int_a^{a+1}\varphi(t){\,}dt,
$$
satisfies
$$
\int_a^{a+1}\varphi^{\star}(t){\,}dt ~=~ g^{(r-1)}(a).
$$

\begin{example}\label{ex:antiD31}
Let us see how we can apply Theorem~\ref{thm:lift} to somewhat generalize Example~\ref{ex:FSFD6341}. Let $g\in\cC^0$, let $G\in\cC^1$ be defined by the equation
$$
G(x)~=~ \int_1^xg(t){\,}dt\qquad\text{for $x>0$},
$$
and let $f\in\cC^0$ be any solution to the equation $\Delta f=g$. To find a solution $F$ to the equation $\Delta F=G$ such that $F'=f$, we just need to apply Theorem~\ref{thm:lift} to the function $G$ with $r=1$ and $a=1$. Defining the function
$$
f^{\star} ~=~ f-\int_1^2f(t){\,}dt{\,},
$$
we then obtain that the function $F\in\cC^1$ defined by the equation
$$
F(x) ~=~ \int_1^xf^{\star}(t){\,}dt ~=~ \int_1^xf(t){\,}dt-(x-1)\int_1^2f(t){\,}dt\qquad\text{for $x>0$},
$$
is the unique (up to an additive constant) solution to the equation $\Delta F=G$ such that $F'=f$. For similar results, see Krull \cite[p.\ 254]{Kru49} and Kuczma \cite[Section~2]{Kuc64}.
\end{example}

The next corollary particularizes the elevator method when the function $g$ lies in $\cC^r\cap\cD^p\cap\cK^{\max\{p,r\}}$ for some $p\in\N$ and $r\in\N^*$. We omit the proof, since it immediately follows from Theorem~\ref{thm:TBTDiff}, Proposition~\ref{prop:fdsf6sfd}, and Theorem~\ref{thm:lift}.

\begin{corollary}[The elevator method]\label{cor:saf6f}\index{elevator method}
Let $g$ lie in $\cC^r\cap\cD^p\cap\cK^{\max\{p,r\}}$ for some $p\in\N$ and $r\in\N^*$. Then $\Sigma g$ lies in $\cC^r\cap\cD^{p+1}\cap\cK^{\max\{p,r\}}$ and we have
$$
(\Sigma g)^{(r)}-\Sigma g^{(r)} ~=~ g^{(r-1)}(1)-\sigma[g^{(r)}].
$$
(This latter value reduces to $-\sum_{k=1}^{\infty}g^{(r)}(k)$ if $r>p$.) Moreover, for any $a>0$, we have
$$
\Sigma g ~=~ f_a-f_a(1),
$$
where $f_a\in\cC^r$ is defined by
$$
f_a(x) ~=~ \sum_{k=1}^{r-1}c_k(a){\,}\frac{(x-a)^k}{k!} + \int_a^x \frac{(x-t)^{r-1}}{(r-1)!}{\,}(\Sigma g)^{(r)}(t){\,}dt
$$
and, for $k=1,\ldots,r-1$,
$$
c_k(a) ~=~
\sum_{j=0}^{r-k-1}\frac{B_j}{j!}{\,}\left(g^{(j+k-1)}(a)-\int_a^{a+1} \frac{(a+1-t)^{r-j-k}}{(r-j-k)!}{\,}(\Sigma g)^{(r)}(t){\,}dt\right).
$$
\end{corollary}

Corollary~\ref{cor:saf6f} has an important practical value. It provides an explicit integral expression for $\Sigma g$ from an explicit expression for $\Sigma g^{(r)}$. Setting $a=1$ in this result, we simply obtain
$$
\Sigma g(x) ~=~ \sum_{k=1}^{r-1}c_k{\,}\frac{(x-1)^k}{k!} + \int_1^x \frac{(x-t)^{r-1}}{(r-1)!}{\,}(\Sigma g)^{(r)}(t){\,}dt,
$$
with, for $k=1,\ldots,r-1$,
$$
c_k ~=~
\sum_{j=0}^{r-k-1}\frac{B_j}{j!}{\,}\left(g^{(j+k-1)}(1)-\int_1^2 \frac{(2-t)^{r-j-k}}{(r-j-k)!}{\,}(\Sigma g)^{(r)}(t){\,}dt\right).
$$
The following three examples illustrate the use of Corollary~\ref{cor:saf6f}. In the  first one, we revisit Example~\ref{ex:FSFD6341}.

\begin{example}
The function
$$
g(x) ~=~ \int_1^x\ln t{\,}dt
$$
lies in $\cC^{\infty}\cap\cD^2\cap\cK^{\infty}$. Choosing $r=1$ and $a=1$ in Corollary~\ref{cor:saf6f}, we get
\begin{eqnarray*}
g'(x) &=& \ln x{\,},\\
\Sigma g'(x) &=& \ln\Gamma(x){\,},\\
(\Sigma g)'(x) &=& \textstyle{\ln\Gamma(x)+1-\frac{1}{2}\ln(2\pi)},
\end{eqnarray*}
and
$$
\Sigma g(x) ~=~ \left(1-\frac{1}{2}\ln(2\pi)\right)(x-1)+\int_1^x\ln\Gamma(t){\,}dt.\qedhere
$$
\end{example}

\begin{example}
The function
$$
g(x) ~=~ \int_0^x(x-t)\ln t{\,}dt
$$
lies in $\cC^{\infty}\cap\cD^3\cap\cK^{\infty}$. Choosing $r=2$ and $a=0$ (as a limiting value) in Corollary~\ref{cor:saf6f}, we get
\begin{eqnarray*}
g''(x) &=& \ln x{\,},\\
\Sigma g''(x) &=& \ln\Gamma(x){\,},\\
(\Sigma g)''(x) &=& \textstyle{\ln\Gamma(x)-\frac{1}{2}\ln(2\pi)},
\end{eqnarray*}
and
$$
\Sigma g(x) ~=~ -(\ln A){\,}x-\frac{1}{4}\ln(2\pi){\,}x^2+\int_0^x(x-t)\ln\Gamma(t){\,}dt,
$$
where $A$ is Glaisher-Kinkelin's constant\index{Glaisher-Kinkelin's constant} and the integral is the polygamma function\index{polygamma functions} $\psi_{-3}(x)$. (Here we use the identity $\psi_{-3}(1)=\ln A+\frac{1}{4}\ln(2\pi)$.)

We can also investigate the asymptotic properties of $\Sigma g$ using our results. For instance, using the generalized Stirling formula \eqref{eq:dgf7dds}, we also obtain the following asymptotic behavior of $\Sigma g$
\begin{multline*}
\Sigma g(x)+\frac{1}{72}{\,}(22x^3-27x^2+9x)-\frac{1}{48}{\,}x^2(8x-15)\ln x\\
\null -\frac{1}{12}(x+1)^2\ln(x+1)+\frac{1}{48}(x+2)^2\ln(x+2) ~\to ~ \frac{\zeta(3)}{8\pi^2}\qquad\text{as $x\to\infty$}.\qedhere
\end{multline*}
\end{example}

\begin{example}
The function $g(x)=\arctan(x)$ lies in $\cC^{\infty}\cap\cD^1\cap\cK^{\infty}$. Choosing $r=1$ and $a=0$ (as a limiting value) in Corollary~\ref{cor:saf6f}, we get (see also Example~\ref{ex:5aRat7Fct8})
\begin{eqnarray*}
g'(x) &=& (x^2+1)^{-1} ~=~ -\Im (x+i)^{-1},\\
\Sigma g'(x) &=& \Im\psi(1+i)-\Im\psi(x+i),\\
(\Sigma g)'(x) &=& c-\Im\psi(x+i),
\end{eqnarray*}
for some $c\in\R$, and hence
$$
\Sigma g(x) ~=~ c{\,}(x-1)+\Im\ln\Gamma(1+i)-\Im\ln\Gamma(x+i).
$$
Applying the operator $\Delta$ to both sides of this identity and then setting $x=1$, we obtain $c=\frac{\pi}{2}$. Thus, we have
$$
\Sigma g(x) ~=~ \frac{\pi}{2}{\,}(x-1)+\Im\ln\Gamma(1+i)-\Im\ln\Gamma(x+i).
$$
Some properties of $\Sigma g$ can be investigated. For instance, using Corollary~\ref{cor:6GenStFo0BaIneqCo} together with the identity
$$
\int_1^x\arctan(t){\,}dt ~=~ x\arctan(x)-\frac{1}{2}\ln(x^2+1)-\frac{\pi}{4}+\frac{1}{2}\ln 2{\,},
$$
we obtain the inequality
\begin{multline*}
\left|\Sigma g(x)-\left(x-\frac{1}{2}\right)\arctan(x)+\frac{1}{2}\ln(x^2+1)-1+\frac{\pi}{4}-\Im\ln\Gamma(1+i)\right|\\
\leq ~ \frac{1}{2}\arctan\frac{1}{x^2+x+1}
\end{multline*}
and hence the left side approaches zero as $x\to\infty$, which provides the asymptotic behavior of the function $\Sigma g$ for large values of its argument.
\end{example}

\section{An alternative uniqueness result}

The following theorem provides a uniqueness result for higher order differentiable solutions to the equation $\Delta f=g$. These solutions can be computed from their derivatives using Theorem~\ref{thm:lift}. We first state a surprising and useful fact.

\begin{fact}\label{fact:rlo3.4}
A periodic function $\omega\colon\R_+\to\R$ is constant if and only if it lies in $\cK^0$. In particular, if $\varphi_1,\varphi_2\colon\R_+\to\R$ are two solutions to the equation $\Delta\varphi=g$ such that $\varphi_1-\varphi_2$ lies in $\cK^0$, then $\varphi_1-\varphi_2$ is constant.
\end{fact}

\begin{theorem}[Uniqueness]\label{thm:unicDer8}\index{uniqueness theorem!alternative forms}
Let $r\in\N^*$ and $g\in\cC^r$, and assume that there exists $\varphi\in\cC^r$ such that $\Delta\varphi=g$ and $\varphi^{(r)}\in\cR^0_{\N}${\,}. Then, the following assertions hold.
\begin{enumerate}
\item[(a)] For each $x>0$, the series $\sum_{k=0}^{\infty}g^{(r)}(x+k)$ converges and we have
$$
\varphi^{(r)}(x) ~=~ -\sum_{k=0}^{\infty}g^{(r)}(x+k){\,}.
$$
\item[(b)] For any $f\in\cC^r\cap\cK^{r-1}$ such that $\Delta f=g$, we have $f=c+\varphi$ for some $c\in\R$.
\end{enumerate}
\end{theorem}

\begin{proof}
Assertion (a) follows immediately from \eqref{eq:fxnEfxSum}. Now, let $f\in\cC^r\cap\cK^{r-1}$ be such that $\Delta f=g$. By Lemma~\ref{lemma:PrelKp}(c), $f^{(r)}$ must lie in $\cK^{-1}$. Setting $\omega=f-\varphi$ and using \eqref{eq:fxnEfxSum} again, we then obtain
$$
\omega^{(r)}(x) ~=~ f^{(r)}(x)-\varphi^{(r)}(x) ~=~ \lim_{n\to\infty} f^{(r)}(x+n),
$$
which shows that $\omega^{(r)}$ also lies in $\cK^{-1}$. By Lemma~\ref{lemma:PrelKp}(d), $\omega$ lies in $\cK^{r-1}\subset\cK^0$ and, since it is $1$-periodic, it must be constant by Fact~\ref{fact:rlo3.4}. This proves assertion (b).
\end{proof}

\begin{example}
The assumptions of Theorem~\ref{thm:unicDer8} hold if $g(x)=\ln x$, $\varphi(x)=\ln\Gamma(x)$, and $r=2$. It then follows that all solutions to the equation $\Delta f=g$ that lie in $\cC^2\cap\cK^1$ are of the form $f(x)=c+\ln\Gamma(x)$, where $c\in\R$. We thus easily retrieve Bohr-Mollerup's theorem\index{Bohr-Mollerup theorem} with the additional assumption that $f$ lies in $\cC^2$. It is remarkable that this latter result can be obtained here from a very elementary theorem that relies only on Lemma~\ref{lemma:PrelKp} and Fact~\ref{fact:rlo3.4}.
\end{example}
%

\chapter{Further results}
\label{chapter:8}

As discussed in the first chapter, the main objective of our work is to generalize Krull-Webster's theory to multiple $\log\Gamma$-type functions and explore the properties of these functions that are analogues of classical properties of the gamma function.

In the previous chapters, we have presented and discussed several results related to these functions, including their differentiation and integration properties as well as important results on their asymptotic behaviors.

We are now in a position to explore further properties of multiple $\log\Gamma$-type functions. More precisely, in this chapter we provide for these functions analogues of \emph{Euler's infinite product}, \emph{Euler's reflection formula}, \emph{Gauss' multiplication formula}, \emph{Gautschi's inequality}, \emph{Raabe's formula}, \emph{Wallis's product formula}, \emph{Webster's functional equation}, and \emph{Weierstrass' infinite product} for the gamma function. We also discuss analogues of \emph{Fontana-Mascheroni's series} and \emph{Gauss' digam{\-}ma theorem} and provide a Gregory's formula-based series representation, a general asymptotic expansion formula, and a few related results.

\section{Eulerian form}
\label{sec:Eul4For631}
\index{Eulerian form|(}

Let $g$ lie in $\cD^p\cap\cK^p$ for some $p\in\N$. As we already observed in Chapter~\ref{chapter:1}, the representation of $\Sigma g$ as the pointwise limit of the sequence $n\mapsto f^p_n[g]$ is the analogue of Gauss' limit\index{Gauss' limit!analogue} for the gamma function. Using identity \eqref{eq:fngsum}, we immediately see that this form of $\Sigma g$ can be translated into a series, namely
\begin{equation}\label{eq:EulG4100}
\Sigma g(x) ~=~ f^p_1[g](x) -\sum_{k=1}^{\infty}\rho^{p+1}_k[g](x),\qquad x>0.
\end{equation}
It is a simple exercise to see that, when $g(x)=\ln x$ and $p=1$, this latter formula reduces to the following series representation of the log-gamma function
\begin{equation}\label{eq:EulG41}
\ln\Gamma(x) ~=~ -\ln x-\sum_{k=1}^{\infty}\textstyle{\left(\ln(x+k)-\ln k-x\ln\left(1+\frac{1}{k}\right)\right)}.
\end{equation}
Its multiplicative version is nothing other than the classical Eulerian form (or Euler's product form) of the gamma function (see, e.g., Srivastava and Choi \cite[p.~3]{SriCho12}). We recall this form in the following proposition.

\begin{proposition}[Eulerian form of the gamma function]\index{Eulerian form!of the gamma function}
The following identity holds
$$
\Gamma(x) ~=~ \frac{1}{x}\,\prod_{k=1}^{\infty}\frac{(1+1/k)^x}{1+x/k}{\,},\qquad x>0.
$$
\end{proposition}

We thus see that, for any multiple $\log\Gamma$-type function, the series representation \eqref{eq:EulG4100} is the analogue of the Eulerian form of the gamma function in the additive notation. Moreover, we have shown in Theorem~\ref{thm:TBTDiff} that this series can be differentiated term by term on $\R_+$. We have also shown in Proposition~\ref{prop:intMLGt} that this series can be integrated term by term on any bounded interval of $[0,\infty)$. Let us state these important facts in the following theorem.

\begin{theorem}[Eulerian form]\label{thm:SerProdReprS}\index{Eulerian form|textbf}
Let $g$ lie in $\cD^p\cap\cK^p$ for some $p\in\N$. The following assertions hold.
\begin{enumerate}
\item[(a)] For any $x>0$ we have
$$
\Sigma g(x) ~=~ -g(x)+\sum_{j=1}^p\tchoose{x}{j}{\,}\Delta^{j-1}g(1) - \sum_{k=1}^{\infty}\left(g(x+k)-\sum_{j=0}^p\tchoose{x}{j}{\,}\Delta^jg(k)\right)
$$
and the series converges uniformly on any bounded subset of $[0,\infty)$.
\item[(b)] If $g$ lies in $\cC^0$, then $\Sigma g$ lies in $\cC^0$ and the series above can be (repeatedly) integrated term by term on any bounded interval of $[0,\infty)$.
\item[(c)] If $g$ lies in $\cC^r\cap\cK^{\max\{p,r\}}$ for some $r\in\N$, then $\Sigma g$ lies in $\cC^r$ and the series above can be differentiated term by term up to $r$ times.
\end{enumerate}
\end{theorem}

\begin{proof}
Assertion (a) follows from identity \eqref{eq:fngsum} and the existence Theorem~\ref{thm:exist} (see also Remark~\ref{rem:ConvCl82}). Assertion (b) follows from Proposition~\ref{prop:intMLGt}, especially its assertion (c2), and Remark~\ref{rem:5repeated0Int53}. Assertion (c) follows from Theorem~\ref{thm:TBTDiff}.
\end{proof}

\begin{example}\label{ex:GepeA58}
Let us apply Theorem~\ref{thm:SerProdReprS} to $g(x)=\ln x$ and $p=1$. We immediately retrieve identity \eqref{eq:EulG41}. Upon differentiation, we also obtain
$$
\psi(x) ~=~ -\frac{1}{x}-\sum_{k=1}^{\infty}\left(\frac{1}{x+k}-\ln\left(1+\frac{1}{k}\right)\right)
$$
and, for any $r\in\N^*$,
$$
\psi_r(x) ~=~ (-1)^{r+1}{\,}r!\,\sum_{k=0}^{\infty}\frac{1}{(x+k)^{r+1}} ~=~ (-1)^{r+1}{\,}r!\,\zeta(r+1,x).
$$
Integrating on $(0,x)$, we obtain
$$
\psi_{-2}(x) ~=~ x-x\ln x-\sum_{k=1}^{\infty}\left((x+k)\ln\left(1+\frac{x}{k}\right)
-x-\frac{x^2}{2}\ln\left(1+\frac{1}{k}\right)\right).
$$
Integrating once more on $(0,x)$, we obtain
\begin{eqnarray*}
\lefteqn{\psi_{-3}(x) ~=~ \frac{1}{4}{\,}x^2(3-2\ln x)}\\
&& \null -\sum_{k=1}^{\infty}\left(\frac{1}{2}(x+k)^2\ln\left(1+\frac{x}{k}\right)
-\frac{k}{2}{\,}x-\frac{3}{4}{\,}x^2-\frac{1}{6}{\,}x^3\ln\left(1+\frac{1}{k}\right)\right).
\end{eqnarray*}
We can actually integrate both sides on $(0,x)$ repeatedly as we wish.
\end{example}

\index{Eulerian form|)}

\section{Weierstrassian form}
\label{sec:WF73}
\index{Weierstrassian form|(}

In the following proposition, we recall an alternative infinite product representation of the gamma function, which was proposed by Weierstrass. This representation is usually called the \emph{Weierstrass factorization} of the gamma function or the \emph{Weierstrass canonical product form} of the gamma function (see Artin \cite[pp.~15--16]{Art15} and Srivastava and Choi \cite[p.~1]{SriCho12}).

\begin{proposition}[Weierstrassian form of the gamma function]\index{Weierstrassian form!of the gamma function}
The following identity holds
\begin{equation}\label{eq:We5G129}
\Gamma(x) ~=~ \frac{e^{-\gamma x}}{x}\,\prod_{k=1}^{\infty}\frac{e^{\frac{x}{k}}}{1+\frac{x}{k}}{\,},\qquad x>0.
\end{equation}
\end{proposition}

We now show that this factorization can be generalized to any $\log\Gamma_p$-type function that is of class $\cC^p$. This new result is presented in the following two theorems, which deal with the cases $p=0$ and $p\geq 1$ separately. We observe that the special case when $p=1$ was previously established by John~\cite[Theorem B']{Joh39} and in the multiplicative notation by Webster~\cite[Theorem 7.1]{Web97b}.

It is important to note that, just as in Theorem~\ref{thm:SerProdReprS}, the partial sums that define the series of the theorems below are nothing other than the sequence $n\mapsto f^p_n[g](x)$. Thus, these series can be integrated and differentiated term by term.

\begin{theorem}[Weierstrassian form when $\deg g=-1$]\label{thm:Weierst1}\index{Weierstrassian form|textbf}
Let $g$ lie in $\cC^0\cap\cD^0\cap\cK^0$. The following assertions hold.
\begin{enumerate}
\item[(a)] We have $\gamma[g]=\sigma[g]$.
\item[(b)] For any $x>0$ we have
$$
\Sigma g(x) ~=~ \sigma[g]-g(x)-\sum_{k=1}^{\infty}\left(g(x+k)-\int_k^{k+1}g(t){\,}dt\right)
$$
and the series converges uniformly on any bounded subset of $[0,\infty)$.
\item[(c)] The function $\Sigma g$ lies in $\cC^0$ and the series above can be (repeatedly) integrated term by term on any bounded interval of $[0,\infty)$.
\item[(d)] If $g$ lies in $\cC^r\cap\cK^r$ for some $r\in\N$, then $\Sigma g$ lies in $\cC^r$ and the series above can be differentiated term by term up to $r$ times.
\end{enumerate}
\end{theorem}

\begin{proof}
Assertion (a) follows from Proposition~\ref{prop:linksSG46}. Assertion (b) follows from Theorem~\ref{thm:SerProdReprS} and identity \eqref{eq:GEC000}. Assertions (c) and (d) follow from Theorem~\ref{thm:SerProdReprS}.
\end{proof}

To establish the second theorem (the case when $\deg g\geq 0$), we need the following technical lemma.

\begin{lemma}\label{lemma:AsymBehTh}
Let $g$ lie in $\cC^1\cap\cD^p\cap\cK^p$ for some $p\in\N^*$. Then
$$
\Delta g(x) -\sum_{j=0}^{p-2}G_j\Delta^jg'(x) ~\to ~ 0\qquad\text{as $x\to\infty$}.
$$
If, in addition, $g\in\cC^{p-1}$, then
$$
\Delta^{p-1}g(x)-g^{(p-1)}(x) ~\to ~0\qquad\text{as $x\to\infty$}.
$$
\end{lemma}

\begin{proof}
By Proposition~\ref{prop:LMpGpLMp1}, we have that $g'$ lies in $\cC^0\cap\cD^{p-1}\cap\cK^{p-1}$. The first convergence result then follows immediately from the application of \eqref{eq:dgf7ddsp1} to $g'$. That is,
$$
J^{p-1}[g'](x) ~\to ~0\qquad\text{as $x\to\infty$}.
$$
Let us now assume that $g\in\cC^{p-1}$. By Propositions~\ref{prop:LMpGpLMp1D7} and \ref{prop:LMpGpLMp1}, for every $i\in\{0,\ldots,p-2\}$ the function
$$
g_i ~=~ \Delta^ig^{(p-2-i)}
$$
lies in $\cC^1\cap\cD^2\cap\cK^2$ and hence, applying the first result to $g_i$, we obtain that
$$
\Delta g_i(x)-g_i'(x) ~\to ~0\qquad\text{as $x\to\infty$}.
$$
Summing these limits for $i=0,\ldots,p-2$, we obtain the claimed limit.
\end{proof}

\begin{theorem}[Weierstrassian form when $\deg g\geq 0$]\label{thm:Weierst}\index{Weierstrassian form|textbf}
Let $g$ lie in $\cC^p\cap\cD^p\cap\cK^p$ with $\deg g=p-1$ for some $p\in\N^*$. The following assertions hold.
\begin{enumerate}
\item[(a)] We have $\gamma[g^{(p)}] = \sigma[g^{(p)}] = g^{(p-1)}(1)-(\Sigma g)^{(p)}(1)$.
\item[(b)] For any $x>0$ we have
\begin{eqnarray*}
\Sigma g(x) &=& \sum_{j=1}^{p-1}\tchoose{x}{j}{\,}\Delta^{j-1}g(1)+\tchoose{x}{p}(\Sigma g)^{(p)}(1)\\
&& -g(x)- \sum_{k=1}^{\infty}\left(g(x+k)-\sum_{j=0}^{p-1}\tchoose{x}{j}{\,}\Delta^jg(k)-\tchoose{x}{p}g^{(p)}(k)\right)
\end{eqnarray*}
and the series converges uniformly on any bounded subset of $[0,\infty)$.
\item[(c)] The function $\Sigma g$ lies in $\cC^p$ and the series above can be (repeatedly) integrated term by term on any bounded interval of $[0,\infty)$.
\item[(d)] If $g$ lies in $\cC^{\max\{p,r\}}\cap\cK^{\max\{p,r\}}$ for some $r\in\N$, then $\Sigma g$ lies in $\cC^{\max\{p,r\}}$ and the series above can be differentiated term by term up to $\max\{p,r\}$ times.
\end{enumerate}
\end{theorem}

\begin{proof}
By Proposition~\ref{prop:LMpGpLMp1}, we have that $g^{(p)}$ lies in $\cC^0\cap\cD^0\cap\cK^0$. Assertion (a) then follows from Propositions~\ref{prop:linksSG46} and \ref{prop:fdsf6sfd}. Now, using \eqref{eq:GEC000} we get
$$
\gamma[g^{(p)}] ~=~ \sum_{k=1}^{\infty}(g^{(p)}(k)-\Delta g^{(p-1)}(k)).
$$
Using Theorem~\ref{thm:SerProdReprS}, we then obtain
\begin{eqnarray*}
\Sigma g(x) &=& \sum_{j=1}^{p-1}\tchoose{x}{j}{\,}\Delta^{j-1}g(1)+\tchoose{x}{p}\left(g^{(p-1)}(1)-\gamma[g^{(p)}]\right)\\
&& -g(x)- \lim_{n\to\infty}\sum_{k=1}^{n-1}\left(g(x+k)-\sum_{j=0}^{p-1}\tchoose{x}{j}{\,}\Delta^jg(k)-\tchoose{x}{p}{\,}g^{(p)}(k)\right)\\
&& + \lim_{n\to\infty}\tchoose{x}{p}\left(\Delta^{p-1}g(n)-g^{(p-1)}(n)\right),
\end{eqnarray*}
where the latter limit is zero by Lemma~\ref{lemma:AsymBehTh}. This proves assertion (b). Assertions (c) and (d) follow from Theorem~\ref{thm:SerProdReprS}.
\end{proof}

\begin{example}\label{ex:GepeA58W}
Let us apply Theorem~\ref{thm:Weierst} to $g(x)=\ln x$ and $p=1$. We immediately get
$$
\ln\Gamma(x) ~=~ -\gamma x-\ln x-\sum_{k=1}^{\infty}\left(\ln(x+k)-\ln k-\frac{x}{k}\right),
$$
which is the additive version of the Weierstrassian form \eqref{eq:We5G129} of the gamma function. It is remarkable that we can now retrieve this formula in an effortless way. Upon differentiation, we also obtain (see, e.g., Srivastava and Choi \cite[p.~24]{SriCho12})
$$
\psi(x) ~=~ -\gamma-\frac{1}{x}-\sum_{k=1}^{\infty}\left(\frac{1}{x+k}-\frac{1}{k}\right).
$$
Integrating on $(0,x)$, we obtain
$$
\psi_{-2}(x) ~=~ -\gamma\,\frac{x^2}{2}+x-x\ln x -\sum_{k=1}^{\infty}\left((x+k)\ln\left(1+\frac{x}{k}\right)-x-\frac{x^2}{2k}\right).
$$
Integrating once more on $(0,x)$, we obtain
\begin{eqnarray*}
\lefteqn{\psi_{-3}(x) ~=~ \frac{1}{12}{\,}x^2(9-2\gamma x-6\ln x)}\\
&& \null -\sum_{k=1}^{\infty}\left(\frac{1}{2}(x+k)^2\ln\left(1+\frac{x}{k}\right)
-\frac{k}{2}{\,}x-\frac{3}{4}{\,}x^2-\frac{x^3}{6k}\right).
\end{eqnarray*}
Just as in Example~\ref{ex:GepeA58}, we can integrate both sides on $(0,x)$ repeatedly as we wish.
\end{example}

Let us end this section with an aside about some potential consequences of the technical Lemma~\ref{lemma:AsymBehTh}.

\begin{remark}
If $g$ lies in $\cC^1\cap\cD^p\cap\cK^p$ for some $p\in\N^*$, then by Propositions~\ref{prop:ClClIntg} and \ref{prop:LMpGpLMp1} we have $g'\in\cR^{p-1}_{\R}$. That is, for any $a\geq 0$
$$
g'(x+a)-\sum_{j=0}^{p-2}\tchoose{a}{j}\Delta^jg'(x) ~\to ~ 0\qquad\text{as $x\to\infty$}.
$$
Combining this result with the first part of Lemma~\ref{lemma:AsymBehTh}, we can derive surprising limits. For instance, we obtain for any $p\in\{1,2,3\}$
$$
\textstyle{\Delta g(x)-g'\left(x+\frac{1}{2}\right)} ~\to ~0\qquad\text{as $x\to\infty$}.
$$
This latter limit has the following interpretation. The mean value theorem tells us that $\Delta g(x)=g'(x+\xi_x)$ for some $\xi_x\in (0,1)$. The limit above then says that
$$
\textstyle{g'(x+\xi_x)-g'(x+\frac{1}{2})} ~\to ~ 0\quad\text{as $x\to\infty$}.
$$
In particular, if $g$ lies in $\cC^2$ and for instance eventually satisfies $g''(x)\geq c$ for some $c>0$, then
\begin{eqnarray*}
c\left|\xi_x-\frac{1}{2}\right| &\leq & \left|\int_{\frac{1}{2}}^{\xi_x}g''(x+t){\,}dt\right|\\
&=& |\textstyle{g'(x+\xi_x)-g'(x+\frac{1}{2})}| ~\to ~ 0\quad\text{as $x\to\infty$},
\end{eqnarray*}
which shows that $\xi_x\to\frac{1}{2}$ as $x\to\infty$.
\end{remark}

\index{Weierstrassian form|)}

\section{Gregory's formula-based series representation}

The following proposition provides series expressions for $\Sigma g$ and $\sigma[g]$ in terms of Gregory's coefficients (see also Proposition~\ref{prop:ddozj40} in Appendix~\ref{chapter:C-StBr49}). This proposition follows from the next lemma, which in turn immediately follows from Corollary~\ref{cor:6GenStFo0BaIneqCo}. 

\begin{lemma}\label{lemma:6699rem}
Let $g$ lie in $\cC^0\cap\cD^p\cap\cK^q$ for some $p,q\in\N$ such that $p\leq q$. Let $x>0$ be so that for $k=p,\ldots,q$ the function $g$ is $k$-convex or $k$-concave on $[x,\infty)$. Then we have
$$
|J^{k+1}[\Sigma g](x)| ~\leq ~ \overline{G}_k{\,}|\Delta^k g(x)|{\,},\qquad k=p,\ldots,q.
$$
\end{lemma}

\begin{proposition}\label{prop:6699rem}
Let $g$ lie in $\cC^0\cap\cD^p\cap\cK^{\infty}$ for some $p\in\N$. Let $x>0$ be so that for every integer $q\geq p$ the function $g$ is $q$-convex or $q$-concave on $[x,\infty)$. Suppose also that the sequence $q\mapsto\Delta^qg(x)$ is bounded. Then we have
$$
J^{q+1}[\Sigma g](x) ~\to ~ 0\qquad\text{as $q\to_{\N}\infty$},
$$
that is,
\begin{equation}\label{eq:fontana490}
\Sigma g(x) ~=~ \sigma[g]+\int_1^xg(t){\,}dt-\sum_{n=1}^{\infty}G_n\,\Delta^{n-1}g(x).
\end{equation}
In particular, if the assumptions above are satisfied for $x=1$, then we have
\begin{equation}\label{eq:fontana49}
\sigma[g] ~=~ \sum_{n=1}^{\infty}G_n\,\Delta^{n-1}g(1).
\end{equation}
\end{proposition}

\begin{proof}
This result is an immediate consequence of Lemma~\ref{lemma:6699rem} and the fact that the sequence $n\mapsto\overline{G}_n$ decreases to zero. Identity \eqref{eq:fontana490} then follows from \eqref{eq:Binet64378S}.
\end{proof}

\begin{example}\label{ex:SeriesGS7}
Applying Proposition~\ref{prop:6699rem} to the function $g(x)=\ln x$ with $p=1$, we obtain the following series representation of the log-gamma function for $x>0$
\begin{eqnarray}
\ln\Gamma(x) &=& \frac{1}{2}\ln(2\pi)-x+x\ln x-\sum_{n=0}^{\infty}G_{n+1}\,\Delta^n\ln x\label{eq:Gr62F0}\\
&=& \frac{1}{2}\ln(2\pi)-x+x\ln x-\sum_{n=0}^{\infty}|G_{n+1}|\,\sum_{k=0}^n(-1)^k\tchoose{n}{k}{\,}\ln(x+k),\nonumber
\end{eqnarray}
where we have used the classical identity (see, e.g., Graham {\em et al.} \cite[p.~188]{GraKnuPat94})
$$
\Delta^n f(x) ~=~ \sum_{k=0}^n(-1)^{n-k}\,\tchoose{n}{k}{\,}f(x+k).
$$
Equivalently, using the Binet function\index{Binet's function} $J(x)$, identity \eqref{eq:Gr62F0} can take the form
$$
J(x) ~=~ -\sum_{n=1}^{\infty}|G_{n+1}|\,\sum_{k=0}^n(-1)^k\tchoose{n}{k}{\,}\ln(x+k),\qquad x>0,
$$
where, for any $n\in\N^*$, the inner sum also reduces to the following integral (see, e.g., \cite[p.\ 192]{GraKnuPat94})
$$
(-1)^n\Delta^n\ln x ~=~ -\int_0^{\infty}\frac{e^{-xt}}{t}\left(1-e^{-t}\right)^n dt{\,},\qquad n\in\N^*.
$$
In particular,
$$
|\Delta^n\ln x| ~\leq ~ \int_0^{\infty}\frac{e^{-xt}}{t}\left(1-e^{-t}\right) dt ~=~ \Delta\ln x ~=~ \ln\left(1+\frac{1}{x}\right).
$$
In the multiplicative notation, identity \eqref{eq:Gr62F0} takes the following form
\begin{eqnarray*}
\Gamma(x) &=& \sqrt{2\pi}{\,}e^{-x}{\,}x^{x-\frac{1}{2}}\left(\frac{x+1}{x}\right)^{\frac{1}{12}}
\left(\frac{(x+2)x}{(x+1)^2}\right)^{-\frac{1}{24}}\\
&& \null\times \left(\frac{(x+3)(x+1)^3}{(x+2)^3x}\right)^{\frac{19}{720}}\ldots.
\end{eqnarray*}
Further infinite product representations and approximations of the gamma function can be found for instance in Feng and Wang~\cite{FenWan13}.
\end{example}

\section{Analogue of Fontana-Mascheroni's series}
\label{sec:8AnzzFont2Mas}
\index{Fontana-Mascheroni's series!analogue|(}

Interestingly, when $g(x)=\frac{1}{x}$ and $p=0$, identity \eqref{eq:fontana49} reduces to the well-known formula
$$
\gamma ~=~ \sum_{n=1}^{\infty}\frac{|G_n|}{n}{\,},
$$
where $\gamma$ is Euler's constant\index{Euler's constant} and the series is called \emph{Fontana-Mascheroni's series}\index{Fontana-Mascheroni's series} (see, e.g., Blagouchine \cite[p.~379]{Bla16}). Thus, the series representation of the asymptotic constant\index{asymptotic constant} $\sigma[g]$ given in \eqref{eq:fontana49} provides the analogue of Fontana-Mascheroni's series for any function $g$ satisfying the assumptions of Proposition~\ref{prop:6699rem}.

\begin{example}
The analogue of Fontana-Mascheroni's series for the function $g(x)=\ln x$ can be obtained by setting $x=1$ in \eqref{eq:Gr62F0}. We obtain
$$
\sum_{n=0}^{\infty}|G_{n+1}|\,\sum_{k=0}^n(-1)^k\tchoose{n}{k}{\,}\ln(k+1) ~=~ -1+\frac{1}{2}\ln(2\pi),
$$
or equivalently (see Example~\ref{ex:SeriesGS7}),
$$
\sum_{n=0}^{\infty}|G_{n+1}|\,\int_0^{\infty}\frac{e^{-t}}{t}\left(1-e^{-t}\right)^n dt ~=~ 1-\frac{1}{2}\ln(2\pi).\qedhere
$$
\end{example}

The following proposition provides a way to construct a function $g(x)$ that has a prescribed associated asymptotic constant\index{asymptotic constant} $\sigma[g]$ given in the form \eqref{eq:fontana49}.

\begin{proposition}\label{prop:Fontana23468}
Suppose that the series
$$
S ~=~ \sum_{n=1}^{\infty}G_n{\,}s_n
$$
converges for a given real sequence $n\mapsto s_n$ and let $g\colon\R_+\to\R$ be such that
\begin{equation}\label{eq:6DefG623}
g(n) ~=~ \sum_{k=1}^n\tchoose{n-1}{k-1}{\,} s_k{\,},\qquad n\in\N^*.
\end{equation}
If $g$ satisfies the assumptions of Proposition~\ref{prop:6699rem} with $x=1$, then the following assertions hold.
\begin{enumerate}
\item[(a)] $S=\sigma[g]$.
\item[(b)] $\Sigma g(n) = \sum_{k=1}^{n-1}{n-1\choose k}{\,} s_k$ for any $n\in\N^*$.
\item[(c)] $s_n=\Delta^{n-1}g(1)=\Delta^n\Sigma g(1)$ for any $n\in\N^*$.
\end{enumerate}
\end{proposition}

\begin{proof}
Identity \eqref{eq:6DefG623} can take the following alternative form
$$
g(n+1) ~=~ \sum_{k=0}^n\tchoose{n}{k}{\,}s_{k+1}{\,},\qquad n\in\N.
$$
Using the classical inversion formula (Graham {\em et al.} \cite[p.\ 192]{GraKnuPat94}), we then obtain
$$
s_{n+1} ~=~ \sum_{k=0}^n(-1)^{n-k}\,\tchoose{n}{k}{\,}g(k+1) ~=~ \Delta^ng(1){\,},\qquad n\in\N.
$$
This establishes assertion (c) and then assertion (a) by Proposition~\ref{prop:6699rem}. Assertion (b) is straightforward using \eqref{eq:RestrInt}.
\end{proof}

\begin{example}\label{ex:ds5a8}
Let us apply Proposition~\ref{prop:Fontana23468} to the series
$$
S ~=~ \sum_{n=1}^{\infty}\frac{|G_n|}{n^2}{\,},
$$
that is,
$$
S ~=~ \sum_{n=1}^{\infty}G_n{\,}s_n\qquad\text{with}\quad s_n ~=~ (-1)^{n-1}\frac{1}{n^2}{\,}.
$$
Let $g\colon\R_+\to\R$ be a function such that
$$
g(n) ~=~ \sum_{k=1}^n(-1)^{k-1}\tchoose{n-1}{k-1}\,\frac{1}{k^2}{\,},\qquad n\in\N^*,
$$
or equivalently (see Graham {\em et al.} \cite[p.~281]{GraKnuPat94} or Merlini {\em et al.} \cite[Lemma~4.1]{MerSprVer06}),
$$
g(n) ~=~ \frac{1}{n}{\,}H_n{\,},\qquad n\in\N^*.
$$
We naturally take $g(x)=\frac{1}{x}{\,}H_x${\,}, from which we can derive (see, e.g., Graham {\em et al.} \cite[p.~280]{GraKnuPat94})
$$
\Sigma g(x) ~=~ \frac{\pi^2}{12}-\frac{1}{2}\,\psi_1(x)+\frac{1}{2}{\,}H_{x-1}^2{\,}.
$$
Thus, we have $S=\sigma[g]$. Combining this result with the definition of $\sigma[g]$, we derive the surprising identity (compare with Blagouchine and Coppo \cite[pp.~469--470]{BlaCop18})
$$
\sum_{n=1}^{\infty}\frac{|G_n|}{n^2} ~=~ \frac{\pi^2}{12}-\frac{1}{2}+\frac{1}{2}\int_0^1 H_t^2{\,}dt{\,}.
$$
Proceeding similarly, with a bit of computation one also finds
$$
\sum_{n=1}^{\infty}\frac{|G_n|}{n^3} ~=~ \frac{1}{3}\,\zeta(3)+\frac{\pi^2}{12}\,\gamma-\frac{5}{12}+\frac{1}{6}\int_0^1 H_t^3{\,}dt{\,}.
$$
Those formulas are worth comparing with the well-known identities (see Section~\ref{sec:dig27amma1})
$$
\sum_{n=1}^{\infty}\frac{|G_n|}{n} ~=~ \gamma ~=~ \int_0^1 H_t{\,}dt{\,}.
$$
For similar formulas, see also Blagouchine and Coppo \cite{BlaCop18}.
\end{example}

\begin{example}\label{ex:ds5a82}
Let us apply Proposition~\ref{prop:Fontana23468} to the series
$$
S ~=~ \sum_{n=1}^{\infty}\frac{|G_n|}{n+a}{\,},
$$
where $a>0$. For this series, we can take
$$
g(x) ~=~ \mathrm{B}(x,a+1)\qquad\text{and}\qquad\Sigma g(x) ~=~ \frac{1}{a}-\mathrm{B}(x,a),
$$
where $(x,y)\mapsto\mathrm{B}(x,y)$ is the beta function. We then derive the identity
$$
\sum_{n=1}^{\infty}\frac{|G_n|}{n+a} ~=~ \frac{1}{a}-\int_0^1\mathrm{B}(x+1,a){\,}dx.
$$
Using the definition of the beta function as an integral, this identity also reads
$$
\sum_{n=1}^{\infty}\frac{|G_n|}{n+a} ~=~ \frac{1}{a}+\int_0^1\frac{x^a}{\ln(1-x)}{\,}dx.
$$
Setting $a=\frac{1}{2}$ for instance, we obtain
$$
\sum_{n=1}^{\infty}\frac{|G_n|}{2n+1} ~=~ 1+\frac{1}{2}\,\int_0^1\frac{\sqrt{x}}{\ln(1-x)}{\,}dx.
$$
We also observe that the decimal expansion of the latter integral is the sequence A094691 in the OEIS \cite{Slo}.
\end{example}

\index{Fontana-Mascheroni's series!analogue|)}

\section{Analogue of Raabe's formula}
\label{sec:Raabe448}
\index{Raabe's formula!analogue|(}

Recall that Raabe's formula yields, for any $x>0$, a simple explicit expression for the integral of the log-gamma function over the interval $(x,x+1)$. We state this result in the following proposition (see Example~\ref{ex:Raab286}). For recent references on Raabe's formula, see, e.g., Cohen and Friedman \cite[p.~366]{CohFri08} and Srivastava and Choi \cite[p.~29]{SriCho12}.

\begin{proposition}[Raabe's formula]\index{Raabe's formula}
The following identity holds
\begin{equation}\label{eq:RaabeLnG5}
\int_x^{x+1}\ln\Gamma(t){\,}dt ~=~ \frac{1}{2}\,\ln(2\pi)+x\ln x-x{\,},\qquad x>0.
\end{equation}
\end{proposition}

Clearly, identities \eqref{eq:sigmagg86} and \eqref{eq:ds68ffds} provide the analogue of Raabe's formula\index{Raabe's formula!analogue|textbf} for any continuous multiple $\log\Gamma$-type function $\Sigma g$. We recall this important and useful formula in the next proposition.

\begin{proposition}[Analogue of Raabe's formula]\label{prop:8Raab0036}
For any function $g$ lying in $\cC^0\cap\mathrm{dom}(\Sigma)$, we have
\begin{equation}\label{eq:ds68ffdsbis}
\int_x^{x+1}\Sigma g(t){\,}dt ~=~ \sigma[g]+\int_1^xg(t){\,}dt,\qquad x>0,
\end{equation}
where $\sigma[g]$ is the asymptotic constant\index{asymptotic constant} associated with $g$ and defined by the equation
\begin{equation}\label{eq:sigmagg86bis}
\sigma[g] ~=~ \int_0^1\Sigma g(t+1){\,}dt{\,}.
\end{equation}
\end{proposition}

The challenging part in this context is to find a nice expression for $\sigma[g]$. For instance, setting $x=1$ in Raabe's formula \eqref{eq:RaabeLnG5}, we obtain the identity
$$
\sigma[\ln] ~=~ \int_0^1\ln\Gamma(t+1){\,}dt ~=~ -1+\frac{1}{2}\,\ln(2\pi){\,}.
$$
However, in general such a closed-form expression for $\sigma[g]$ is not easy to derive.

An expression for $\sigma[g]$ as a limit can be obtained using Proposition~\ref{prop:intMLGt}(c2). Specifically, if $g$ lies in $\cC^0\cap\cD^p\cap\cK^p$ for some $p\in\N$, then we have
\begin{eqnarray}
\sigma[g] &=& \lim_{n\to\infty}\int_0^1(f^p_n[g](t)+g(t)){\,}dt\nonumber\\
&=& \lim_{n\to\infty}\left(\sum_{k=1}^{n-1}g(k)-\int_1^ng(t){\,}dt+\sum_{j=1}^pG_j\Delta^{j-1}g(n)\right),\label{eq:SgSt7FrI}
\end{eqnarray}
which is nothing other than the restriction of the generalized Stirling formula \eqref{eq:dgf7dds}\index{Stirling's formula!generalized} to the natural integers.

Series expressions for $\sigma[g]$ can also be obtained by integrating on the interval $(0,1)$ the series representations of $\Sigma g+g$ given in Theorems~\ref{thm:SerProdReprS} and \ref{thm:Weierst}. For instance, we have
\begin{equation}\label{eq:serS72}
\sigma[g] ~=~ \sum_{j=1}^pG_j{\,}\Delta^{j-1}g(1) - \sum_{k=1}^{\infty}\left(\int_{k}^{k+1}g(t){\,}dt-\sum_{j=0}^pG_j{\,}\Delta^jg(k)\right).
\end{equation}
Note also that, under certain assumptions, the latter series converges to zero as ${p\to_{\N}\infty}$. In this case, \eqref{eq:serS72} reduces to the analogue of Fontana-Mascheroni's series; see Proposition~\ref{prop:6699rem}.

\begin{example}\label{ex:Dig5Int9}
Applying \eqref{eq:SgSt7FrI} and \eqref{eq:serS72} to $g(x)=\frac{1}{x}$ and $p=0$, we obtain
$$
\sigma[g] ~=~ \lim_{n\to\infty}\left(\sum_{k=1}^n\frac{1}{k}-\ln n\right) ~=~ \sum_{k=1}^{\infty}\left(\frac{1}{k}-\ln\left(1+\frac{1}{k}\right)\right),
$$
which is Euler's constant $\gamma$.\index{Euler's constant} Identity \eqref{eq:ds68ffdsbis} then immediately provides the following analogue of Raabe's formula
$$
\int_x^{x+1}\psi(t){\,}dt ~=~ \ln x{\,},\qquad x>0.\qedhere
$$
\end{example}

The following proposition provides interesting identities that involve the antiderivative of $\Sigma g$, where $g$ is any function lying in $\cC^0\cap\mathrm{dom}(\Sigma)$. It also yields a formula for $\Sigma G$, where $G$ is the antiderivative of $g$. This result is worth comparing with Example~\ref{ex:antiD31}.

\begin{proposition}\label{prop:an4tR8}
Let $g$ lie in $\cC^0\cap\cD^p\cap\cK^p$ for some $p\in\N$ and define the function $G\colon\R_+\to\R$ by the equation
$$
G(x) ~=~ \int_1^x g(t){\,}dt\qquad\text{for $x>0$}.
$$
Then $G$ lies in $\cC^1\cap\cD^{p+1}\cap\cK^{p+1}$. Moreover, for any $x>0$ we have
$$
\Sigma G(x) ~=~ \int_1^x\Sigma g(t){\,}dt-\sigma[g]{\,}(x-1)
$$
and
$$
\Sigma_x\int_x^{x+1}\Sigma g(t){\,}dt ~=~ \int_1^x\Sigma g(t){\,}dt{\,}.
$$
\end{proposition}

\begin{proof}
We have that $G$ lies in $\cC^1\cap\cD^{p+1}\cap\cK^{p+1}$ by Proposition~\ref{prop:LMpGpLMp1}. We then obtain
$$
(\Sigma G)' ~=~ \Sigma g-\sigma[g]
$$
by Proposition~\ref{prop:fdsf6sfd}. This establishes the first formula. Combining it with \eqref{eq:ds68ffdsbis}, we obtain
$$
\Sigma_x\int_x^{x+1}\Sigma g(t){\,}dt ~=~ \sigma[g]{\,}(x-1)+\Sigma G(x) ~=~ \int_1^x\Sigma g(t){\,}dt,
$$
that is, the second formula.
\end{proof}

\begin{example}\label{ex:xlnx52}
Apply Proposition~\ref{prop:an4tR8} to the function $g(x)=\ln x$ with $p=1$, we obtain
$$
\Sigma_x\int_x^{x+1}\ln\Gamma(t){\,}dt ~=~ \int_1^x\ln\Gamma(t){\,}dt ~=~ \psi_{-2}(x)-\psi_{-2}(1).
$$
Using Raabe's formula \eqref{eq:RaabeLnG5} in the left-hand side, we finally obtain
$$
\frac{1}{2}\,\ln(2\pi)(x-1)+\Sigma_x(x\ln x) -\tchoose{x}{2} ~=~ \psi_{-2}(x)-\psi_{-2}(1),
$$
from which we immediately derive a closed-form expression for $\Sigma_x(x\ln x)$; see also Section~\ref{sec:hypFacF4}.
\end{example}

We now present a proposition, immediately followed by a corollary that provides interesting characterizations of multiple $\Gamma$-type functions based on the analogue of Raabe's formula. Example~\ref{ex:Char0Raa4} below illustrates this characterization in the special case of the log-gamma function.

\begin{proposition}\label{prop:InvRaa4}
Let $h$ lie in $\cC^1\cap\cD^{p+1}\cap\cK^{p+1}$ for some $p\in\N$ and let $f\colon\R_+\to\R$ be a function. Then $f$ lies in $\cC^0\cap\cK^p$ and satisfies the equation
\begin{equation}\label{eq:8Char55Raa023}
\int_x^{x+1}f(t){\,}dt ~=~ h(x){\,},\qquad x>0,
\end{equation}
if and only if $f=(\Sigma h)'$.
\end{proposition}

\begin{proof}
The sufficiency is trivial. Let us prove the necessity. Differentiating both sides of \eqref{eq:8Char55Raa023}, we obtain $\Delta f=h'$. Using the existence Theorem~\ref{thm:exist} and then Proposition~\ref{prop:fdsf6sfd}, we then see that $f=c+(\Sigma h)'$ for some $c\in\R$. Using \eqref{eq:8Char55Raa023} again, we then see that $c$ must be $0$.
\end{proof}

\begin{corollary}[A characterization result]\label{cor:InvRaa4c}
Let $g$ lie in $\cC^0\cap\cD^p\cap\cK^p$ and let $f\colon\R_+\to\R$ be a function. Then $f$ lies in $\cC^0\cap\cK^p$ and satisfies the equation
$$
\int_x^{x+1}f(t){\,}dt ~=~ \sigma[g]+\int_1^xg(t){\,}dt{\,},\qquad x>0,
$$
if and only if $f=\Sigma g$.
\end{corollary}

\begin{proof}
The sufficiency is trivial by \eqref{eq:ds68ffdsbis}. Let us prove the necessity. Define the function $h\colon\R_+\to\R$ by the equation
$$
h(x) ~=~ \sigma[g]+\int_1^xg(t){\,}dt\qquad\text{for $x>0$}.
$$
Then, $h$ clearly lies in $\cC^1\cap\cD^{p+1}\cap\cK^{p+1}$. Using Proposition~\ref{prop:InvRaa4} and then Proposition~\ref{prop:an4tR8}, we immediately obtain that $f=(\Sigma h)'=\Sigma g$.
\end{proof}

\begin{example}\label{ex:Char0Raa4}
Applying Corollary~\ref{cor:InvRaa4c} to the function $g(x)=\ln x$ with $p=1$, we obtain the following alternative characterization of the gamma function. A function $f\colon\R_+\to\R$ lies in $\cC^0\cap\cK^1$ and satisfies the equation
$$
\int_x^{x+1}f(t){\,}dt ~=~ \frac{1}{2}\,\ln(2\pi)+x\ln x-x{\,}, \qquad x>0,
$$
if and only if $f(x)=\ln\Gamma(x)$.
\end{example}

\index{Raabe's formula!analogue|)}

\section{Analogue of Gauss' multiplication formula}
\label{sec:GaussMultF51}
\index{Gauss' multiplication formula!analogue|(}

In the following proposition, we recall the \emph{Gauss multiplication formula} for the gamma function, also called \emph{Gauss' multiplication theorem} (see Artin \cite[p.~24]{Art15}).

\begin{proposition}[Gauss' multiplication formula]\index{Gauss' multiplication formula}
For any integer $m\geq 1$, we have the following identity
\begin{equation}\label{eq:GAU45SS}
\prod_{j=0}^{m-1}\Gamma\left(\frac{x+j}{m}\right) ~=~ \frac{\Gamma(x)}{m^{x-\frac{1}{2}}}{\,}(2\pi)^{\frac{m-1}{2}},\qquad x>0.
\end{equation}
\end{proposition}

When $m=2$, identity \eqref{eq:GAU45SS} reduces to \emph{Legendre's duplication formula}\index{Legendre's duplication formula|textbf}
$$
\Gamma\left(\frac{x}{2}\right)\Gamma\left(\frac{x+1}{2}\right) ~=~ \frac{\Gamma(x)}{2^{x-1}}{\,}\sqrt{\pi}{\,},\qquad x>0.
$$

\begin{remark}
For any fixed $m\geq 2$, the Gauss multiplication formula \eqref{eq:GAU45SS} enables one to retrieve easily the value of the asymptotic constant\index{asymptotic constant} associated with the function $g(x)=\ln x$. In particular, this value can be retrieved from Legendre's duplication formula. Indeed, taking the logarithm of both sides of \eqref{eq:GAU45SS} and then integrating on $x\in (0,1)$, we obtain
$$
\sum_{j=0}^{m-1}\int_0^1\ln\Gamma\left(\frac{x+j}{m}\right)dx ~=~ \frac{m-1}{2}\ln(2\pi)+\int_0^1\ln\Gamma(x){\,}dx.
$$
Using the change of variable $t=\frac{x+j}{m}$ in the left-hand integral, we then obtain almost immediately the following identity
$$
\int_0^1\ln\Gamma(t){\,}dt ~=~ \frac{1}{2}\ln(2\pi).
$$
Combining this result with \eqref{eq:ds68ffdsbis}, we retrieve $\sigma[\ln]=-1+\frac{1}{2}\ln(2\pi)$.
\end{remark}

Webster \cite[Theorem 5.2]{Web97b} showed how an analogue of Gauss' multiplication formula can be partially constructed for any $\Gamma$-type function. His proof is very short and essentially relies on the uniqueness and existence theorems in the special case when $p=1$. We now show how Webster's approach can be further extended to all multiple $\Gamma$-type functions. As usual, we use the additive notation.

\begin{theorem}[Analogue of Gauss' multiplication formula]\label{thm:MultThmGen}\index{Gauss' multiplication formula!analogue|textbf}
Let $g$ lie in $\mathrm{dom}(\Sigma)$ and let $m\in\N^*$. Define also the function $g_m\colon\R_+\to\R$ by the equation
$$
g_m(x) ~=~ g\left(\frac{x}{m}\right)\qquad\text{for $x>0$}.
$$
Then we have
\begin{equation}\label{eq:MultThmGeneq}
\sum_{j=0}^{m-1}\Sigma g\left(\frac{x+j}{m}\right) ~=~ \sum_{j=1}^m\Sigma g\left(\frac{j}{m}\right)+\Sigma g_m(x),\qquad x>0,
\end{equation}
and
$$
\Sigma g_m(m) ~=~ \sum_{j=1}^{m-1}g\left(\frac{j}{m}\right).
$$
\end{theorem}

\begin{proof}
Let $g$ lie in $\cD^p\cap\cK^p$ for some $p\in\N$. Then $g_m$ also lies in $\cD^p\cap\cK^p$ by Corollary~\ref{cor:Hom4}. Now, we can readily check that the function $f\colon\R_+\to\R$ defined by
$$
f(x) ~=~ \sum_{j=0}^{m-1}\Sigma g\left(\frac{x+j}{m}\right)-\sum_{j=1}^m\Sigma g\left(\frac{j}{m}\right)
$$
is a solution to the equation $\Delta f=g_m$ that lies in $\cK^p$ and such that $f(1)=0$. By the uniqueness Theorem~\ref{thm:unic}, it follows that $f=\Sigma g_m$. This establishes \eqref{eq:MultThmGeneq}. The last identity follows immediately.
\end{proof}

Theorem~\ref{thm:MultThmGen} actually provides a partial solution to the problem of finding the analogue of Gauss' multiplication formula. A more complete result would also provide a closed-form expression for the right-hand side of identity \eqref{eq:MultThmGeneq}.

Unfortunately, no general method to provide simple or compact expressions for $\Sigma g_m$ seems to be known. However, such expressions can sometimes be found.

For instance, when $g(x)=\ln x$, we obtain
$$
g_m(x) ~=~ \ln x-\ln m\qquad\text{and}\qquad\Sigma g_m(x) ~=~ \ln\Gamma(x)-(x-1)\ln m.
$$
Substituting this latter expression in identity \eqref{eq:MultThmGeneq}, we immediately obtain the formula
\begin{equation}\label{eq:GauAddGam43}
\sum_{j=0}^{m-1}\ln\Gamma\left(\frac{x+j}{m}\right) ~=~ \sum_{j=1}^m\ln\Gamma\left(\frac{j}{m}\right)+\ln\Gamma(x)-(x-1)\ln m{\,},
\end{equation}
that is, in the multiplicative notation,
$$
\prod_{j=0}^{m-1}\Gamma\left(\frac{x+j}{m}\right) ~=~ \frac{\Gamma(x)}{m^{x-1}}\,\prod_{j=1}^m\Gamma\left(\frac{j}{m}\right),\qquad x>0.
$$

It remains to find a nice expression for the latter product, and more generally for the right-hand sum of identity \eqref{eq:MultThmGeneq}. On this issue, we have the following useful result.

\begin{proposition}\label{prop:CalcSum662}
Let $g$ lie in $\cC^0\cap\mathrm{dom}(\Sigma)$ and let $m\in\N^*$. Define also the function $g_m\colon\R_+\to\R$ by the equation $g_m(x)=g(\frac{x}{m})$ for $x>0$. Then we have
\begin{eqnarray*}
\sum_{j=1}^m\Sigma g\left(\frac{j}{m}\right) &=& m{\,}\sigma[g]-\int_m^{m+1}\Sigma g_m(t){\,}dt\\
&=& m{\,}\sigma[g]-\sigma[g_m]-m\,\int_{1/m}^1g(t){\,}dt.
\end{eqnarray*}
\end{proposition}

\begin{proof}
The first identity can be proved simply by integrating both sides of \eqref{eq:MultThmGeneq} on $x\in (m,m+1)$. Indeed, using the change of variable $t=\frac{x+j}{m}$ and identity \eqref{eq:sigmagg86bis}, the left-hand side reduces to
$$
m\,\sum_{j=0}^{m-1}\int_{1+\frac{j}{m}}^{1+\frac{j+1}{m}}\Sigma g(t){\,}dt ~=~ m\int_1^2\Sigma g(t){\,}dt ~=~ m\,\sigma[g].
$$
The second identity then follows from a simple application of \eqref{eq:ds68ffdsbis}.
\end{proof}

\begin{example}\label{ex:GaMFLnx5}
Let us apply Proposition~\ref{prop:CalcSum662} to the function $g(x)=\ln x$. We obtain
$$
\sum_{j=1}^m\ln\Gamma\left(\frac{j}{m}\right) ~=~  -\frac{1}{2}\ln m+\frac{1}{2}{\,}(m-1)\ln(2\pi).
$$
Substituting this expression in \eqref{eq:GauAddGam43} and then translating the resulting formula into the multiplicative notation, we retrieve Gauss' multiplication formula \eqref{eq:GAU45SS}.
\end{example}

In the following proposition, we provide a convergence result for the function defined in the left-hand side of \eqref{eq:MultThmGeneq}, which does not require the computation of $\Sigma g_m$. This result simply reduces to the generalized Stirling formula when $m=1$.

\begin{proposition}\label{prop:8Stir44Gau7Mult}
Let $g$ lie in $\cC^0\cap\cD^p\cap\cK^p$ for some $p\in\N$ and let $m\in\N^*$. Define also the function $g_m\colon\R_+\to\R$ by the equation $g_m(x)=g(\frac{x}{m})$ for $x>0$. Then we have
$$
\sum_{j=0}^{m-1}\Sigma g\left(\frac{x+j}{m}\right) -\int_1^xg_m(t){\,}dt+\sum_{j=1}^pG_j\,\Delta^{j-1}g_m(x) ~\to ~ m\,\sigma_m[g]
$$
as $x\to\infty$, where
$$
\sigma_m[g] ~=~ \sigma[g]-\int_{1/m}^1g(t){\,}dt.
$$
\end{proposition}

\begin{proof}
Theorem~\ref{thm:MultThmGen} and Proposition~\ref{prop:CalcSum662} provide the following identity
$$
\Sigma g_m(x)-\sigma[g_m] ~=~ \sum_{j=0}^{m-1}\Sigma g\left(\frac{x+j}{m}\right)-m\,\sigma_m[g]\qquad x>0.
$$
The result is then an immediate application of the generalized Stirling formula (Theorem~\ref{thm:dgf7dds}) to the function $\Sigma g_m$ (recall that $g_m$ lies in $\cC^0\cap\cD^p\cap\cK^p$).
\end{proof}

We end this section with three corollaries. Corollaries~\ref{cor:GMT6Cc1} and \ref{cor:MultG41S} yield properties of the derivatives and antiderivatives of the function $g$ in the context of the analogue of Gauss' multiplication formula. Corollary~\ref{cor:Riem482} shows how the antiderivative of $g$ can be expressed as a limit involving the function $\Sigma g_m$.

\begin{corollary}\label{cor:GMT6Cc1}
Let $g$ lie in $\cC^r\cap\cD^p\cap\cK^{\max\{p,r\}}$ for some $p\in\N$ and $r\in\N^*$. Let also $m\in\N^*$ and define the function $g_m\colon\R_+\to\R$ by the $g_m(x)=g(\frac{x}{m})$. Then the equation obtained by replacing $g$ with $g^{(r)}$ in \eqref{eq:MultThmGeneq} can also be obtained by differentiating $r$ times both sides of \eqref{eq:MultThmGeneq}.
\end{corollary}

\begin{proof}
Differentiating $r$ times both sides of \eqref{eq:MultThmGeneq}, multiplying through by $m^r$, and then using \eqref{eq:rd0bb}, we obtain
$$
\sum_{j=0}^{m-1}\Sigma g^{(r)}\left(\frac{x+j}{m}\right) + m(\Sigma g)^{(r)}(1) ~=~ m^r\,\Sigma g_m^{(r)}(x) + m^r(\Sigma g_m)^{(r)}(1).
$$
Setting $x=1$, we then get
$$
\sum_{j=1}^m\Sigma g^{(r)}\left(\frac{j}{m}\right) + m(\Sigma g)^{(r)}(1) ~=~ m^r(\Sigma g_m)^{(r)}(1).
$$
Subtracting this latter equation from the former one, we finally get
$$
\sum_{j=0}^{m-1}\Sigma g^{(r)}\left(\frac{x+j}{m}\right) ~=~ \sum_{j=1}^m\Sigma g^{(r)}\left(\frac{j}{m}\right)+m^r\,\Sigma g_m^{(r)}(x),
$$
which is precisely the equation obtained by replacing $g$ with $g^{(r)}$ in \eqref{eq:MultThmGeneq}.
\end{proof}

\begin{corollary}\label{cor:MultG41S}
Let $p\in\N$, $m\in\N^*$, $c\in\R$, and $g\in\cC^0\cap\cD^p\cap\cK^p$. Define also the functions $G,g_m,G_m\colon\R_+\to\R$ by the equations
$$
G(x) ~=~ c+\int_1^xg(t){\,}dt,\quad g_m(x)=g\left(\frac{x}{m}\right),\quad G_m(x) ~=~ G\left(\frac{x}{m}\right)\quad\text{for $x>0$}.
$$
Then both functions $G$ and $G_m$ lie in $\cC^1\cap\cD^{p+1}\cap\cK^{p+1}$. Moreover, for any $x>0$ we have
$$
\Sigma G_m(x) ~=~ \frac{1}{m}\int_1^x\Sigma g_m(t){\,}dt+(x-1)\left(c-\frac{1}{m}\int_m^{m+1}\Sigma g_m(t){\,}dt\right).
$$
\end{corollary}

\begin{proof}
The first part follows immediately from Proposition~\ref{prop:an4tR8} and Corollary~\ref{cor:Hom4}. Now, by definition of $G_m$ we have
$$
G_m(x) ~=~ c+\frac{1}{m}\int_m^xg_m(t){\,}dt ~=~ c+\frac{1}{m}\left(\int_1^xg_m(t){\,}dt-\int_1^mg_m(t){\,}dt\right).
$$
The claimed identity can then be established easily using Proposition~\ref{prop:an4tR8} and then applying identity \eqref{eq:ds68ffdsbis}.
\end{proof}

\begin{corollary}\label{cor:Riem482}
Let $g$ lie in $\cC^0\cap\mathrm{dom}(\Sigma)$. Define also the functions $g_m\colon\R_+\to\R$ $(m\in\N^*)$ by the equation $g_m(x)=g(\frac{x}{m})$ for $x>0$. Then we have
$$
\lim_{m\to\infty}\frac{\Sigma g_m(mx)-\Sigma g_m(m)}{m} ~=~ \int_1^x g(t){\,}dt{\,},\qquad x>0.
$$
Moreover, if $g$ is integrable at $0$, then
$$
\lim_{m\to\infty}\frac{1}{m}\,\Sigma g_m(mx) ~=~ \int_0^x g(t){\,}dt{\,},\qquad x>0.
$$
\end{corollary}

\begin{proof}
Replacing $x$ with $mx$ in \eqref{eq:MultThmGeneq} and dividing through by $m$, we obtain
$$
\frac{1}{m}\,\Sigma g_m(mx) ~=~ \frac{1}{m}\,\sum_{j=0}^{m-1}\Sigma g\left(x+\frac{j}{m}\right)-\frac{1}{m}\,\sum_{j=1}^m\Sigma g\left(\frac{j}{m}\right).
$$
Letting $m\to_{\N}\infty$ in this identity and using \eqref{eq:ds68ffdsbis}, we see that the first Riemann sum on the right side converges to
$$
\int_0^1\Sigma g(x+t){\,}dt ~=~ \sigma[g]+\int_1^xg(t){\,}dt
$$
while the second one converges (if $g$ is integrable at $0$) to
$$
\int_0^1\Sigma g(t){\,}dt ~=~ \sigma[g]-\int_0^1g(t){\,}dt.
$$
This establishes the corollary.
\end{proof}

\index{Gauss' multiplication formula!analogue|)}

\section{Asymptotic expansions and related results}
\label{sec:8AsyzzExp3Rel2}
\index{asymptotic expansion|(}

In this section, we provide and investigate asymptotic expansions of (higher order differentiable) multiple $\log\Gamma$-type functions. We also establish and discuss some important consequences of these expansions, including a variant of the generalized Stirling formula and an extension of the so-called Liu formula to multiple $\log\Gamma$-type functions.

To begin with, let us first recall the asymptotic expansion of the log-gamma function (see, e.g., Gel'fond \cite[p.~342]{Gel71} and Srivastava and Choi \cite[p.~7]{SriCho12}).

\begin{proposition}
For any $q\in\N^*$, we have the following asymptotic expansion as $x\to\infty$
\begin{equation}\label{eq:AsExLog21}
\ln\Gamma(x) ~=~ \frac{1}{2}\ln(2\pi)-x+\left(x-\frac{1}{2}\right)\ln x+\sum_{k=1}^q\frac{B_{k+1}}{k(k+1){\,}x^k}+O\left(x^{-q-1}\right).
\end{equation}
\end{proposition}

For instance, setting $q=4$ in equation \eqref{eq:AsExLog21}, we obtain
$$
\ln\Gamma(x) ~=~ \frac{1}{2}\ln(2\pi)-x+\left(x-\frac{1}{2}\right)\ln x+\frac{1}{12x}-\frac{1}{360x^3}+O\left(x^{-5}\right).
$$

We now provide a generalization of this result to multiple $\log\Gamma$-type functions. Even more generally, in the next proposition we provide for any integer $m\in\N^*$ an asymptotic expansion of the function
\begin{equation}\label{eq:Riem481}
x ~\mapsto ~\frac{1}{m}\sum_{j=0}^{m-1}\Sigma g\left(x+\frac{j}{m}\right).
\end{equation}

\begin{proposition}\label{prop:Richardson9}
\begin{enumerate}
\item[(a)] Let $g$ lie in $\cC^1\cap\cD^p\cap\cK^{\max\{p,1\}}$ for some $p\in\N$. Then, for any $m\in\N^*$ and any $x>0$, we have
$$
\frac{1}{m}\sum_{j=0}^{m-1}\Sigma g\left(x+\frac{j}{m}\right) ~=~ \int_x^{x+1}\Sigma g(t){\,}dt -\frac{1}{2m}{\,}g(x) + R_m(x){\,},\label{eq:Rich4201}
$$
with
$$
R_m(x) ~=~ \frac{1}{m}\,\int_0^1B_1(\{mt\}){\,}(\Sigma g)'(x+t){\,}dt
$$
and
$$
|R_m(x)| ~\leq ~ \frac{1}{2m}\,\int_0^1|(\Sigma g)'(x+t)|{\,}dt{\,}.
$$
For large $x$ the latter integral reduces to $|g(x)|$.
\item[(b)] If $g$ lie in $\cC^{2q}\cap\cD^p\cap\cK^{\max\{p,2q\}}$ for some $p\in\N$ and some $q\in\N^*$. Then, for any $m\in\N^*$ and any $x>0$, we have
\begin{eqnarray*}
\lefteqn{\frac{1}{m}\sum_{j=0}^{m-1}\Sigma g\left(x+\frac{j}{m}\right) ~=~ \int_x^{x+1}\Sigma g(t){\,}dt -\frac{1}{2m}{\,}g(x)}\nonumber \\
&& \null\hspace{10ex} +\sum_{k=1}^q\frac{1}{m^{2k}}\,\frac{B_{2k}}{(2k)!}{\,}g^{(2k-1)}(x) + R^q_m(x){\,},\label{eq:Rich42}
\end{eqnarray*}
with
$$
R^q_m(x) ~=~ -\frac{1}{m^{2q}}\,\int_0^1\frac{B_{2q}(\{mt\})}{(2q)!}{\,}(\Sigma g)^{(2q)}(x+t){\,}dt
$$
and
$$
|R^q_m(x)| ~\leq ~ \frac{1}{m^{2q}}{\,}\frac{|B_{2q}|}{(2q)!}\,\int_0^1|(\Sigma g)^{(2q)}(x+t)|{\,}dt{\,}.
$$
For large $x$ the latter integral reduces to $|g^{(2q-1)}(x)|$.
\end{enumerate}
\end{proposition}

\begin{proof}
Let us prove assertion (b) first. The first part follows from a straightforward application of Euler-Mac{\-}laurin's formula (Proposition~\ref{prop:E3MacLau5}) to $f=\Sigma g$, with $a=x$, $b=x+1$, and $N=m$. Now, we see that the function $(\Sigma g)^{(2q)}$ lies in $\cK^{(p-2q)_+}$ by Proposition~\ref{prop:LMpGpLMp1}, and hence also in $\cK^{-1}$ by Proposition~\ref{prop:MpDescFiltr}. Thus, for sufficiently large $x$ we obtain
\begin{eqnarray*}
\int_0^1|(\Sigma g)^{(2q)}(x+t)|{\,}dt &=& \left|\int_0^1(\Sigma g)^{(2q)}(x+t){\,}dt\right|\\
&=& \left|(\Sigma g)^{(2q-1)}(x+1)-(\Sigma g)^{(2q-1)}(x)\right|.
\end{eqnarray*}
By Proposition~\ref{prop:fdsf6sfd}, the latter expression reduces to
$$
\left|\Sigma g^{(2q-1)}(x+1)-\Sigma g^{(2q-1)}(x)\right| ~=~ |g^{(2q-1)}(x)|{\,}.
$$
Assertion (a) can be proved similarly. Here we observe that $(\Sigma g)'$ lies in $\cK^{(p-1)_+}$ and hence also in $\cK^{-1}$. Thus, for sufficiently large $x$ we obtain
$$
\int_0^1|(\Sigma g)'(x+t)|{\,}dt ~=~ \left|\int_0^1(\Sigma g)'(x+t){\,}dt\right| ~=~ |g(x)|.
$$
This completes the proof.
\end{proof}

Setting $m=1$ in Proposition~\ref{prop:Richardson9}, we derive immediately an asymptotic expansion of the function $\Sigma g$ in terms of its trend and the higher order derivatives of $g$. As this special case is very important for the applications, we state it in the next proposition (in which we also use \eqref{eq:ds68ffdsbis} to evaluate the integral of $\Sigma g$ on $(x,x+1)$).

\begin{proposition}\label{prop:Richardson9c}
The following assertions hold.
\begin{enumerate}
\item[(a)] Let $g$ lie in $\cC^1\cap\cD^p\cap\cK^{\max\{p,1\}}$ for some $p\in\N$. Then, for any $x>0$ we have
$$
\Sigma g(x) ~=~ \sigma[g]+\int_1^xg(t){\,}dt -\frac{1}{2}{\,}g(x) + R_1(x){\,},\label{eq:Rich4201c}
$$
with
$$
R_1(x) ~=~ \int_0^1B_1(t){\,}(\Sigma g)'(x+t){\,}dt
$$
and
$$
|R_1(x)| ~\leq ~ \frac{1}{2}\,\int_0^1|(\Sigma g)'(x+t)|{\,}dt{\,}.
$$
For large $x$ the latter integral reduces to $|g(x)|$.
\item[(b)] If $g$ lie in $\cC^{2q}\cap\cD^p\cap\cK^{\max\{p,2q\}}$ for some $p\in\N$ and some $q\in\N^*$. Then, for any $x>0$ we have
\begin{equation}\label{eq:VvarStir78}
\Sigma g(x) ~=~ \sigma[g]+\int_1^xg(t){\,}dt -\frac{1}{2}{\,}g(x)+\sum_{k=1}^q\frac{B_{2k}}{(2k)!}{\,}g^{(2k-1)}(x) + R^q_1(x){\,},
\end{equation}
with
$$
R^q_1(x) ~=~ -\int_0^1\frac{B_{2q}(t)}{(2q)!}{\,}(\Sigma g)^{(2q)}(x+t){\,}dt
$$
and
$$
|R^q_1(x)| ~\leq ~ \frac{|B_{2q}|}{(2q)!}\,\int_0^1|(\Sigma g)^{(2q)}(x+t)|{\,}dt{\,}.
$$
For large $x$ the latter integral reduces to $|g^{(2q-1)}(x)|$.
\end{enumerate}
\end{proposition}

\begin{example}
Taking $g(x)=\ln x$ and $p=1$ in \eqref{eq:VvarStir78}, we retrieve immediately the asymptotic expansion given in \eqref{eq:AsExLog21}. The following equivalent, but more concise, formulation of this expansion is given in terms of Binet's function.\index{Binet's function} For any $q\in\N^*$, we have
$$
J(x) ~=~ \sum_{k=1}^q\frac{B_{k+1}}{k(k+1){\,}x^k}+O\left(x^{-q-1}\right)\qquad\text{as $x\to\infty$}.\qedhere
$$
\end{example}

\begin{remark}\label{rem:GreEul4cv0}
The following alternative asymptotic expansion of the Riemann sum \eqref{eq:Riem481} can be immediately obtained using the general form of Gregory's formula (Proposition~\ref{prop:GenFoGreg7}). If $g$ lies in $\cC^0\cap\cD^p\cap\cK^p$ for some $p\in\N$ and if it is $q$-convex or $q$-concave on $[x,\infty)$ for every integer $q\geq p$, then we have
\begin{eqnarray*}
\int_x^{x+1}\Sigma g(t){\,}dt &=& \frac{1}{m}\sum_{j=0}^{m-1}\Sigma g\left(x+\frac{j}{m}\right) + \frac{1}{m}\,\sum_{k=1}^q G_k\,\Delta^{k-1}g_m(mx) + R,
\end{eqnarray*}
where
$$
|R| ~\leq ~ \frac{1}{m}\,\overline{G}_q\left|\Delta^qg_m(mx)\right|\qquad\text{and}\qquad g_m(x) ~=~ g\left(\frac{x}{m}\right).
$$
(Compare with Proposition~\ref{prop:8Stir44Gau7Mult}.) If we set $m=1$ in this latter expansion, then we immediately retrieve the inequality of Lemma~\ref{lemma:6699rem} as well as the Gregory formula-based series expression for $\Sigma g$ given in \eqref{eq:fontana490}. It is then important to note that the asymptotic expansion \eqref{eq:VvarStir78} often leads to divergent series, contrary to its ``cousin'' formula \eqref{eq:fontana490}, as already observed in Remark~\ref{rem:CompEulGr3}. For instance, setting $x=1$ in \eqref{eq:AsExLog21} leads to a divergent series whereas setting $x=1$ in the ``cousin'' formula \eqref{eq:Gr62F0} leads to an analogue of Fontana-Mascheroni's series. In this regard, we observe that the Gregory coefficients\index{Gregory coefficients} have the asymptotic behavior
$$
|G_n| ~\sim ~ \frac{1}{n(\ln n)^2}\qquad\text{as $n\to\infty$}{\,},
$$
while the Bernoulli numbers\index{Bernoulli numbers} satisfy
$$
|B_{2n}| ~=~ \frac{2(2n)!}{(2\pi)^{2n}}{\,}\zeta(2n) ~\sim ~ 4\sqrt{\pi n}\left(\frac{n}{\pi e}\right)^{2n}\qquad\text{as $n\to\infty$}{\,};
$$
see, e.g., Graham {\em et al.} \cite[p.~286]{GraKnuPat94}.
\end{remark}

\parag{A variant of the generalized Stirling formula} Interestingly, from Proposition~\ref{prop:Richardson9} we can easily derive the following variant of the generalized Stirling formula.

\begin{proposition}[A variant of the generalized Stirling formula]\label{prop:VarStir6}\index{Stirling's formula!generalized!a variant}
Let $g$ lie in $\cC^{2q}\cap\cD^p\cap\cK^{2q}$ for some $q\in\N^*\cup\{\frac{1}{2}\}$ and some $p\in\N$ satisfying $p\leq 2q-1$. For any $m\in\N^*$ we have
$$
\frac{1}{m}\sum_{j=0}^{m-1}\Sigma g\left(x+\frac{j}{m}\right) -\int_1^x g(t){\,}dt -\sum_{k=1}^p\frac{B_k}{m^k k!}{\,}g^{(k-1)}(x) ~\to ~ \sigma[g]\qquad\text{as $x\to\infty$.}
$$
In particular,
\begin{equation}\label{eq:VvarStir7}
\Sigma g(x)-\int_1^x g(t){\,}dt -\sum_{k=1}^p\frac{B_k}{k!}{\,}g^{(k-1)}(x) ~\to ~ \sigma[g]\qquad\text{as $x\to\infty$.}
\end{equation}
\end{proposition}

\begin{proof}
For every $k\in\{p,\ldots,2q\}$ we clearly have that $g$ lies in $\cD^k\cap\cK^k$ and hence $g^{(k)}$ vanishes at infinity by Theorem~\ref{thm:intCpTpKp}(b). The result then follows from Proposition~\ref{prop:Richardson9}. The particular case is obtained by setting $m=1$.
\end{proof}

It is clear that the convergence result \eqref{eq:VvarStir7} coincides with the generalized Stirling formula \eqref{eq:dgf7dds} whenever $p=0$ or $p=1$. Thus, it does not bring anything new in these cases.

Now, we observe that if $g$ lies in $\cC^{\max\{2q,r\}}\cap\cD^p\cap\cK^{\max\{2q,r\}}$ for some $q\in\N^*\cup\{\frac{1}{2}\}$ and some $p\in\N$ satisfying $p\leq 2q-1$, then the convergence result in \eqref{eq:VvarStir7} still holds if we replace $g$ with $g^{(r)}$ and $p$ with $(p-r)_+$. Moreover, this modified result can also be obtained by differentiating $r$ times both sides of \eqref{eq:VvarStir7} and then removing the terms that vanish at infinity. This important fact can be easily proved similarly as for the generalized Stirling formula (see Proposition~\ref{prop:gSfDr4} and the comment that follows it).

\begin{remark}
We now see that the generalized Stirling formula \eqref{eq:dgf7dds} could also be established similarly as its variant \eqref{eq:VvarStir7}, i.e., using the Gregory formula-based asymptotic expansion of $\Sigma g$ as discussed in Remark~\ref{rem:GreEul4cv0}. However, formula \eqref{eq:dgf7dds} is a very elementary consequence of Lemma~\ref{lemma:VarEpsIneq}, as commented in Remark~\ref{rem:WenSti41}. Its proof is elementary, elegant, and leads to the whole Theorem~\ref{thm:6GenStFo0BaIneq}, which is a strong result that also provides inequalities.
\end{remark}

The restriction of the limit \eqref{eq:VvarStir7} to the natural integers provides the following alternative formula to compute the asymptotic constant\index{asymptotic constant} $\sigma[g]$. Under the assumptions of Proposition~\ref{prop:VarStir6}, we have
\begin{equation}\label{eq:sig51saz}
\sigma[g] ~=~ \lim_{n\to\infty}\left(\sum_{k=1}^{n-1} g(k)-\int_1^n g(t){\,}dt -\sum_{k=1}^p\frac{B_k}{k!}{\,}g^{(k-1)}(n)\right).
\end{equation}

\parag{Analogue of Liu's formula} Liu~\cite{Liu07} (see also Mortici \cite{Mor10}) established the following formula.\index{Liu's formula} For any $n\in\N^*$ we have
$$
n! ~=~ \Gamma(n+1) ~=~ \sqrt{2\pi n}\,\left(\frac{n}{e}\right)^n \exp\left(\int_n^{\infty}\frac{\frac{1}{2}-\{t\}}{t}{\,}dt\right).
$$
This formula provides an exact (as opposed to asymptotic) expression for the gamma function with an integer argument.

We now propose a generalization of this identity to multiple $\log\Gamma$-type functions with real arguments. We call it the \emph{generalized Liu formula}. Recall first the following Dirichlet test\index{Dirichlet test for convergence of improper integrals} for convergence of improper integrals (see, e.g., Titchmarsh \cite[p.~21]{Tit39})

\begin{lemma}[Dirichlet's test]\label{lemma:8Dirichlet4Test}
Let $a\geq 0$ and let $f\colon\R_+\to\R$ be so that the function $x\mapsto\int_a^x f(t){\,}dt$ is bounded on $[a,\infty)$. Let also $g$ lie in $\cC^1\cap\cD^0\cap\cK^0$. Then the improper integral
$$
\int_a^{\infty}f(t)g(t){\,}dt
$$
converges.
\end{lemma}

\begin{proposition}[Generalized Liu's formula]\label{prop:Liu471}\index{Liu's formula!generalized}
\begin{enumerate}
\item[(a)] If $g$ lies in $\cC^2\cap\cD^1\cap\cK^2$, then for any $x>0$ we have
$$
\Sigma g(x) ~=~ \sigma[g]+\int_1^x g(t){\,}dt -\frac{1}{2}{\,}g(x)+\int_0^{\infty}\textstyle{\left(\frac{1}{2}-\{t\}\right)g'(x+t){\,}dt}.
$$
\item[(b)] If $g$ lies in $\cC^{2q+1}\cap\cD^{2q}\cap\cK^{2q+1}$ for some $q\in\N^*$, then for any $x>0$ we have
\begin{eqnarray*}
\Sigma g(x) &=& \sigma[g]+\int_1^x g(t){\,}dt -\frac{1}{2}{\,}g(x)+\sum_{k=1}^q\frac{B_{2k}}{(2k)!}{\,}g^{(2k-1)}(x)\\
&& \null + \int_0^{\infty}\frac{B_{2q}(\{t\})}{(2q)!}{\,}g^{(2q)}(x+t){\,}dt.
\end{eqnarray*}
\end{enumerate}
\end{proposition}

\begin{proof}
Let us prove assertion (b) first. We apply assertion (b) of Proposition~\ref{prop:Richardson9c} to the function $g$ with $p=2q$. Thus, for any $x>0$ and any $n\in\N$ we have
\begin{eqnarray*}
R^q_1(x) &=& \int_{x+1}^{x+n+1}\frac{B_{2q}(\{t-x\})}{(2q)!}{\,}(\Sigma g)^{(2q)}(t){\,}dt\\
&& \null -\int_x^{x+n+1}\frac{B_{2q}(\{t-x\})}{(2q)!}{\,}(\Sigma g)^{(2q)}(t){\,}dt.
\end{eqnarray*}
By Proposition~\ref{prop:fdsf6sfd}, we have
$$
(\Sigma g)^{(2q)}(t+1)-(\Sigma g)^{(2q)}(t) ~=~ g^{(2q)}(t)
$$
and hence we obtain
$$
R^q_1(x) ~=~ S^q_n(x) + T^q_n(x),
$$
where
\begin{eqnarray*}
S^q_n(x) &=& \int_x^{x+n}\frac{B_{2q}(\{t-x\})}{(2q)!}{\,} g^{(2q)}(t){\,}dt{\,},\\
T^q_n(x) &=& -\int_{x+n}^{x+n+1}\frac{B_{2q}(\{t-x\})}{(2q)!}{\,}(\Sigma g)^{(2q)}(t){\,}dt{\,}.
\end{eqnarray*}

Now, we observe that the sequence $n\mapsto S^q_n(x)$ converges by Dirichlet's test (see Lemma~\ref{lemma:8Dirichlet4Test}). Indeed, $g^{(2q)}$ lies in $\cC^1\cap\cD^0\cap\cK^0$ by Proposition~\ref{prop:LMpGpLMp1}, and for every $u\geq x$ we have that
\begin{eqnarray*}
\left|\int_x^u\frac{B_{2q}(\{t-x\})}{(2q)!}{\,}dt\right| &=& \left|\int_0^{u-x}\frac{B_{2q}(\{t\})}{(2q)!}{\,}dt\right|\\
&=& \left|\int_{\lfloor u-x\rfloor}^{u-x}\frac{B_{2q}(\{t\})}{(2q)!}{\,}dt\right| ~\leq ~ \frac{|B_{2q}|}{(2q)!}{\,},
\end{eqnarray*}
where we have used the well-known fact that the integral on $(0,1)$ of the Bernoulli polynomial\index{Bernoulli polynomials} $B_{2q}$ is zero.

Let us now show that the sequence $n\mapsto T^q_n(x)$ approaches zero as $n\to\infty$. Using integration by parts, we obtain
\begin{eqnarray*}
T^q_n(x) &=& -\int_0^1\frac{B_{2q}(t)}{(2q)!}{\,}(\Sigma g)^{(2q)}(x+n+t){\,}dt\\
&=& \int_0^1\frac{B_{2q+1}(t)}{(2q+1)!}{\,}(\Sigma g)^{(2q+1)}(x+n+t){\,}dt.
\end{eqnarray*}
Since $(\Sigma g)^{(2q+1)}$ lies in $\cK^{-1}$, for large $n$ we obtain
\begin{eqnarray*}
|T^q_n(x)| &\leq & \frac{|B_{2q+1}|}{(2q+1)!}{\,}\left|\int_0^1(\Sigma g)^{(2q+1)}(x+n+t){\,}dt\right|\\
&=& \frac{|B_{2q+1}|}{(2q+1)!}{\,}\left|g^{(2q)}(x+n)\right|,
\end{eqnarray*}
which approaches zero as $n\to\infty$ by Theorem~\ref{thm:intCpTpKp}(b). This proves assertion (b).

Assertion (a) can be proved similarly by applying assertion (a) of Proposition~\ref{prop:Richardson9c} to function $g$ with $p=1$. For any $x>0$ and any $n\in\N$ we have
$$
R_1(x) ~=~ S_n(x) + T_n(x),
$$
where
\begin{eqnarray*}
S_n(x) &=& -\int_x^{x+n}B_1(\{t-x\}){\,} g'(t){\,}dt{\,},\\
T_n(x) &=& \int_{x+n}^{x+n+1}B_1(\{t-x\}){\,}(\Sigma g)'(t){\,}dt{\,}.
\end{eqnarray*}

We now see that the sequence $n\mapsto S_n(x)$ converges by Dirichlet's test. Moreover, the sequence $n\mapsto T_n(x)$ approaches zero as $n\to\infty$. Indeed, using integration by parts we obtain
\begin{eqnarray*}
T_n(x) &=& \int_0^1B_1(t){\,}(\Sigma g)'(x+n+t){\,}dt\\
&=& \frac{B_2}{2}{\,}g'(x+n)-\int_0^1\frac{B_2(t)}{2}{\,}(\Sigma g)''(x+n+t){\,}dt,
\end{eqnarray*}
and we conclude the proof as in assertion (b) since $g'$ lies in $\cC^1\cap\cD^0\cap\cK^0$.
\end{proof}

\begin{example}
Let us apply assertion (a) of Proposition~\ref{prop:Liu471} to $g(x)=\ln x$. We obtain
$$
\ln\Gamma(x) ~=~ \frac{1}{2}\ln(2\pi)-x+\left(x-\frac{1}{2}\right)\ln x+\int_0^{\infty}\frac{\frac{1}{2}-\{t\}}{t+x}{\,}dt,
$$
or equivalently,
$$
J(x) ~=~ J^2[\ln\circ\Gamma](x) ~=~ \int_0^{\infty}\frac{\frac{1}{2}-\{t\}}{t+x}{\,}dt,
$$
which extends the original Liu formula to a real argument.
\end{example}

\begin{example}
Applying assertion (a) of Proposition~\ref{prop:Liu471} to $g(x)=\frac{1}{x}$, we obtain the following integral expression for the digamma function\index{digamma function}
$$
\psi(x) ~=~ \ln x-\frac{1}{2x}+\int_0^{\infty}\frac{\{t\}-\frac{1}{2}}{(t+x)^2}{\,}dt.
$$
This expression seems to be previously unknown.
\end{example}

Setting $x=1$ in Proposition~\ref{prop:Liu471}, we immediately derive an integral representation of the asymptotic constant\index{asymptotic constant} $\sigma[g]$. We state this observation in the following corollary.

\begin{corollary}\label{cor:Liu471S5}
\begin{enumerate}
\item[(a)] If $g$ lies in $\cC^2\cap\cD^1\cap\cK^2$, then we have
$$
\sigma[g] ~=~ \frac{1}{2}{\,}g(1)+\int_1^{\infty}\textstyle{\left(\{t\}-\frac{1}{2}\right)g'(t){\,}dt}.
$$
\item[(b)] If $g$ lies in $\cC^{2q+1}\cap\cD^{2q}\cap\cK^{2q+1}$ for some $q\in\N^*$, then we have
$$
\sigma[g] ~=~ \frac{1}{2}{\,}g(1)-\sum_{k=1}^q\frac{B_{2k}}{(2k)!}{\,}g^{(2k-1)}(1)  - \int_1^{\infty}\frac{B_{2q}(\{t\})}{(2q)!}{\,}g^{(2q)}(t){\,}dt.
$$
\end{enumerate}
\end{corollary}

\begin{remark}
Proposition~\ref{prop:Liu471} and Corollary~\ref{cor:Liu471S5} enable one to evaluate certain improper integrals involving polynomial functions of the fractional part of the integration variable. For example, to establish the identity
$$
\int_1^{\infty}\frac{\{x\}-\frac{1}{2}}{2x+1}{\,}dx ~=~ -\frac{3}{4}+\frac{1}{4}\ln 2+\frac{1}{2}\ln 3
$$
(Srivastava and Choi \cite[p.~600, Problem 11]{SriCho12}), we simply use assertion (a) of Corollary~\ref{cor:Liu471S5} with $g(x)=\frac{1}{2}\ln(2x+1)$. In this case, we have
$$
\Sigma g(x) ~=~ \frac{1}{2}\ln 2{\,}(x-1)+\frac{1}{2}\ln\Gamma\left(x+\frac{1}{2}\right)-\frac{1}{2}\ln\Gamma\left(\frac{3}{2}\right)
$$
and the integral is simply equal to $\sigma[g]-\frac{1}{2} g(1)$.
\end{remark}

\begin{remark}\label{rem:EM5LiU3}
In Proposition~\ref{prop:Liu471}, we could substitute $\sigma[g]$ from its expression given in Corollary~\ref{cor:Liu471S5}. But then, the restriction to the natural integers of the resulting formulas will simply reduce to the application of Euler-Maclaurin's formula (Proposition~\ref{prop:E3MacLau5}) to $g$, with $a=1$, $b=n$, $h=1$, and $N=n-1$.
\end{remark}

\index{asymptotic expansion|)}

\section{Analogue of Wallis's product formula}
\label{sec:Wallis582}
\index{Wallis's product formula!analogue|(}

In the following proposition, we recall one of the different versions of Wallis's product formula (see, e.g., Finch \cite[p.~21]{Fin03}).

\begin{proposition}[Wallis's product formula]\index{Wallis's product formula}
The following limit holds
\begin{equation}\label{eq:Wallis1140}
\lim_{n\to\infty}\frac{1\cdot 3{\,}\cdots{\,}(2n-1)}{2\cdot 4{\,}\cdots{\,}(2n)}{\,}\sqrt{n} ~=~ \frac{1}{\sqrt{\pi}}{\,}.
\end{equation}
\end{proposition}

In the additive notation, identity \eqref{eq:Wallis1140} becomes
$$
\lim_{n\to\infty}\left(\frac{1}{2}\,\ln(\pi n)+\sum_{k=1}^{2n}(-1)^{k-1}\,\ln k\right) ~=~ 0.
$$

The following proposition gives an analogue of this latter formula for any function $g$ lying in $\cC^0\cap\mathrm{dom}(\Sigma)$.

\begin{proposition}\label{prop:Wallis92}\index{Wallis's product formula!analogue|textbf}
Let $g$ lie in $\cC^0\cap\cD^p\cap\cK^p$ for some $p\in\N$. Let $\tilde{g}\colon\R_+\to\R$ be the function defined by the equation $\tilde{g}(x)=2{\,}g(2x)$ for $x>0$. Let also $h\colon\N^*\to\R$ be the sequence defined by the equation
\begin{eqnarray*}
h(n) &=& \sigma[\tilde{g}]-\sigma[g]+\int_1^2(g(2n+t)-g(t)){\,}dt\\
&& \null +\sum_{j=1}^pG_j\left(\Delta^{j-1}g(2n+1)-\Delta^{j-1}\tilde{g}(n+1)\right)\qquad\text{for $n\in\N^*$}.
\end{eqnarray*}
Then we have
\begin{equation}\label{eq:Wallis93}
\lim_{n\to\infty}\left(h(n) + \sum_{k=1}^{2n}(-1)^{k-1}g(k)\right) ~=~ 0.
\end{equation}
\end{proposition}

\begin{proof}
The function $\tilde{g}$ lies in $\cC^0\cap\cD^p\cap\cK^p$ by Corollary~\ref{cor:Hom4}. By \eqref{eq:RestrInt}, for any $n\in\N^*$ we thus have
$$
\sum_{k=1}^{2n}(-1)^{k-1}g(k) ~=~ \sum_{k=1}^{2n}g(k)-\sum_{k=1}^n\tilde{g}(k) ~=~ \Sigma g(2n+1)-\Sigma\tilde{g}(n+1).
$$
Using the discrete version of the generalized Stirling formula \eqref{eq:SgSt7FrI}, we get
$$
\sigma[g] ~=~ \lim_{n\to\infty}\left(\sum_{k=1}^{2n}g(k)-\int_1^{2n+1}g(t){\,}dt+\sum_{j=1}^pG_j\,\Delta^{j-1}g(2n+1)\right)
$$
and
$$
\sigma[\tilde{g}] ~=~ \lim_{n\to\infty}\left(\sum_{k=1}^{n}\tilde{g}(k)-\int_1^{n+1}\tilde{g}(t){\,}dt
+\sum_{j=1}^pG_j\,\Delta^{j-1}\tilde{g}(n+1)\right).
$$
This establishes the claimed formula.
\end{proof}

Formula \eqref{eq:Wallis93} actually holds for infinitely many sequences $n\mapsto h(n)$. Indeed, if it holds for a sequence $h(n)$, then it also holds for instance for the sequence $h(n)+n^{-q}$ for any $q\in\N^*$. Thus, to obtain an elegant analogue of Wallis's product formula, it is advisable to choose $h$ among the simplest functions. For instance, we could consider the sequence obtained from the series expansion for $h(n)$ about infinity after removing all the summands that vanish at infinity.

\begin{example}
Let us apply Proposition~\ref{prop:Wallis92} to $g(x)=\ln x$ with $p=1$. We obtain
\begin{eqnarray*}
h(n) &=& 2n\ln(2n+2)-\left(2n+\frac{1}{2}\right)\ln(2n+1)+\ln(n+1)-1+\frac{1}{2}\ln(2\pi)\\
&=& \frac{1}{2}\ln(\pi n)+O\left(n^{-2}\right).
\end{eqnarray*}
Replacing $h(n)$ with $\frac{1}{2}\ln(\pi n)$ in \eqref{eq:Wallis93} as recommended above, we retrieve the original Wallis product formula \eqref{eq:Wallis1140}.
\end{example}

\begin{example}
Let us apply Proposition~\ref{prop:Wallis92} to the harmonic number function\index{harmonic number function} $g(x)=H_x$ with $p=1$. After a bit of calculus we get
\begin{eqnarray*}
h(n) &=& \frac{1}{2}{\,}H_{2n+1}+\frac{1}{2}\,\ln 2+\ln(n+1)-\psi(2n+3)\\
&=& \frac{1}{2}(\gamma +\ln n)+O\left(n^{-1}\right).
\end{eqnarray*}
We then obtain the following analogue of Wallis's product formula
$$
\lim_{n\to\infty}\left(-\ln n+2\,\sum_{k=1}^{2n}(-1)^kH_k\right) ~=~ \gamma{\,},
$$
which provides an alternative definition of Euler's constant\index{Euler's constant} $\gamma$.
\end{example}

\begin{example}
Let us apply Proposition~\ref{prop:Wallis92} to the harmonic number function of order $2$\index{harmonic number function!of order $2$}
$$
g(x) ~=~ H_x^{(2)} ~=~ \zeta(2)-\zeta(2,x+1)
$$
with $p=1$. After some algebra we obtain the following analogue of Wallis's product formula
$$
\lim_{n\to\infty}\sum_{k=1}^{2n}(-1)^kH_k^{(2)} ~=~ \frac{\pi^2}{24}{\,}.\qedhere
$$
\end{example}

\begin{remark}\label{rem:Wall38}
Alternative sequences for $h(n)$ may be considered in Proposition~\ref{prop:Wallis92}. For instance, if $g$ lies in $\cC^0\cap\cD^p\cap\cK^p$ for some $p\in\N$, then it is easy to see that
$$
\sum_{k=1}^{2n}(-1)^{k-1}g(k) ~=~ -\Sigma\tilde{g}(n+1),\qquad n\in\N^*,
$$
where $\tilde{g}\colon\R_+\to\R$ is the function defined by the equation $\tilde{g}(x)=\Delta g(2x-1)$ for $x>0$. Thus, assuming that $\tilde{g}$ lies in $\cK^0$, identity \eqref{eq:Wallis93} also holds for
$$
h(n) ~=~ \sigma[\tilde{g}]+\int_1^{n+1}\tilde{g}(t){\,}dt-\sum_{j=1}^{(p-1)_+}G_j\,\Delta^{j-1}\tilde{g}(n+1).
$$
Similarly, we can easily see that
$$
\sum_{k=1}^{2n}(-1)^{k-1}g(k) ~=~ g(1)-g(2n)+\Sigma\tilde{g}(n),\qquad n\in\N^*,
$$
where $\tilde{g}\colon\R_+\to\R$ is the function defined by the equation $\tilde{g}(x)=\Delta g(2x)$ for $x>0$. Thus, assuming again that $\tilde{g}$ lies in $\cK^0$, identity \eqref{eq:Wallis93} also holds for
$$
h(n) ~=~ g(2n)-g(1)-\sigma[\tilde{g}]-\int_1^n\tilde{g}(t){\,}dt+\sum_{j=1}^{(p-1)_+}G_j\,\Delta^{j-1}\tilde{g}(n).
$$
It is clear that the most appropriate function $h$ among these possibilities strongly depends on the form of the function $g$.
\end{remark}

\begin{remark}
Using summation by parts with the classical indefinite sum operator (see, e.g., Graham {\em et al.} \cite[p.~55]{GraKnuPat94}), it is not difficult to show that
\begin{equation}\label{SbP4S3}
\Sigma_x g(2x) ~=~ x{\,}g(2x)-g(2)-\Sigma_x\left((x+1)(\Delta g(2x)+\Delta g(2x+1)\right)
\end{equation}
(provided both sides exist). More generally, for any $m\in\N^*$, we can show that
$$
\Sigma_x g(mx) ~=~ x{\,}g(mx)-g(m)-\sum_{j=0}^{m-1}\Sigma_x\left((x+1)\,\Delta g(mx+j)\right).
$$
For instance, using \eqref{SbP4S3} we obtain
\begin{eqnarray*}
\Sigma_x\psi(2x) &=& x\,\psi(2x)-\psi(2)-\Sigma_x\left(1+\frac{1}{2x}+\frac{1}{4(x+\frac{1}{2})}\right)\\
&=& x\,\psi(2x)-\psi(1)-x-\frac{1}{2}(\psi(x)+\gamma)
-\frac{1}{4}\left(\psi\left(x+\frac{1}{2}\right)-\psi\left(\frac{3}{2}\right)\right)\\
&=& x\,\psi(2x)-\frac{1}{2}\,\psi(x)-x-\frac{1}{4}\,\psi\left(x+\frac{1}{2}\right)+\frac{1}{4}\left(2-2\ln 2+\gamma\right).
\end{eqnarray*}
As this example demonstrates, formula \eqref{SbP4S3} can sometimes be very useful in Proposition~\ref{prop:Wallis92} for the computation of $\sigma[\tilde{g}]$.
\end{remark}

\index{Wallis's product formula!analogue|)}

\section{Analogue of Euler's reflection formula}
\label{sec:ReflFor62}
\index{Euler's reflection formula!analogue|(}

Recall that the identity
\begin{equation}\label{eq:ReflDig500MuO}
\Gamma(z)\Gamma(1-z) ~=~ \pi\csc(\pi z)
\end{equation}
holds for any $z\in\mathbb{C}\setminus\Z$. This identity, known by the name \emph{Euler's reflection formula}\index{Euler's reflection formula} (see, e.g., Artin \cite[p.~26]{Art15} and Srivastava and Choi \cite[p.~3]{SriCho12}), can be proved for instance using the Weierstrassian form of the gamma function.

Motivated by this and similar examples, it is then natural to wonder if an analogue of Euler's reflection formula holds for any multiple $\log\Gamma$-type function, at least on $\R\setminus\Z$, or even on the interval $(0,1)$. However, this question seems rather difficult and reflection formulas as beautiful as \eqref{eq:ReflDig500MuO} are relatively exceptional.

Now, if we logarithmically differentiate both sides of \eqref{eq:ReflDig500MuO}, we obtain the following reflection formula for the digamma function\index{digamma function} (see \cite[p.~25]{SriCho12})
\begin{equation}\label{eq:ReflDig5}
\psi(x)-\psi(1-x) ~=~ -\pi\cot(\pi x){\,}.
\end{equation}
Using an appropriate integration, we also obtain the following reflection formula for the Barnes $G$-function\index{Barnes's $G$-function} (see \cite[p.~45]{SriCho12})
\begin{equation}\label{eq:ReflBGF5}
\ln G(1+x)-\ln G(1-x) ~=~ x\ln(2\pi)-\int_0^x\pi t\cot(\pi t){\,}dt{\,}.
\end{equation}

These and other examples show that the reflection formulas usually share a common pattern. Their right sides typically include $1$-periodic functions or integrals of $1$-periodic functions while their left sides are of one the following forms
$$
\Sigma g(x)\pm \Sigma g(1-x)\qquad\text{or}\qquad\Sigma g(1+x)\pm \Sigma g(1-x)
$$
for some appropriate functions $g$.

In this section, we investigate this important topic in the light of our theory. To get straight to the point, we have not found an analogue of Euler's reflection formula that is systematically applicable to any multiple $\log\Gamma$-type function. We nevertheless present a few interesting results that could hopefully be the starting point of a larger theory.

First of all, due to the presence of the arguments $x$ and $1-x$ in most of the reflection formulas, it is important to see how the domain of the functions considered in this work can be extended to a larger set. Since many functions $g$ involved in the difference equation $\Delta f=g$ have singularities at $0$ (e.g., $g(x)=\frac{1}{x}$), we suggest extending the domain of all these functions to the set $\R\setminus\{0\}$. Due to the nature of the difference operator $\Delta$, any solution $f$ is then required to be defined on $\R\setminus(-\N)$. The domains of many other associated functions and identities of this theory can be extended likewise. For instance, for any $p\in\N$ and any $n\in\N^*$, the domain of the function $f_n^p[g]$ defined in \eqref{eq:defFpn} can be extended to $\R\setminus(-\N)$. Similarly, for any $p\in\N$ and any $a\in\R\setminus\{0\}$, the domain of the function $\rho_a^p[g]$ defined in \eqref{eq:deflambdapt} can be extended to $\R\setminus\{-a\}$.

We now have the following important result.

%
%

\begin{lemma}\label{lemma:8EulRef4Fo9p}
Let $g\colon\R\setminus\{0\}\to\R$ be a function whose restriction $g|_{\R_+}$ to $\R_+$ lies in $\cD^p\cap\cK^p$ for some $p\in\N$. Then, there exists a unique function $f\colon\R\setminus (-\N)\to\R$ such that $\Delta f=g$ and $f|_{\R_+}=\Sigma (g|_{\R_+})$. Moreover,
$$
f(x) ~=~ \lim_{n\to\infty}f_n^p[g](x){\,},\qquad x\in \R\setminus (-\N).
$$
\end{lemma}

\begin{proof}
For any $m\in\N$ and any solution $f\colon\R\setminus (-\N)\to\R$ to the equation $\Delta f=g$, we must have
\begin{equation}\label{eq:8Tel7scoE}
f(x-m) ~=~f(x) - \sum_{k=1}^mg(x-k),\qquad x\in\R_+\setminus\N.
\end{equation}
This clearly establishes the first part of the lemma.

Let us now prove that for any $x\in\R_+\setminus\N$ and any integers $0\leq m\leq n$ we have
\begin{equation}\label{eq:8TescoE3zz}
f_n^p[g](x)-\sum_{k=1}^mg(x-k) ~=~ f_n^p[g](x-m)-\sum_{k=1}^m\rho_n^p[g](x-k).
\end{equation}
On the one hand, for $j=1,\ldots,p$, we have
$$
\sum_{k=1}^m\tchoose{x-k}{j-1} ~=~ \sum_{k=0}^{m-1}\tchoose{k+x-m}{j-1} ~=~ \sum_{k=0}^m\Delta_k\tchoose{k+x-m}{j} ~=~ \tchoose{x}{j}-\tchoose{x-m}{j}
$$
and hence using \eqref{eq:deflambdapt} we obtain
$$
\sum_{k=1}^m\rho_n^p[g](x-k) ~=~ \sum_{k=1}^mg(x+n-k)-\sum_{j=1}^p\left(\tchoose{x}{j}-\tchoose{x-m}{j}\right)\Delta^{j-1}g(n){\,}.
$$
On the other hand, using this latter identity and subtracting the right side of \eqref{eq:8TescoE3zz} from the left side, using \eqref{eq:defFpn} we obtain
$$
\sum_{k=0}^{n-1}(g(x-m+k)-g(x+k))-\sum_{k=1}^mg(x-k)+\sum_{k=1}^mg(x+n-k),
$$
which is identically zero. This establishes \eqref{eq:8TescoE3zz}.

Let us now show that the sequence $n\mapsto\rho^p_n[g](x-k)$ converges to zero for any $x\in\R_+\setminus\N$ and any $k\in\N$. By \eqref{eq:LDDiv} it is actually enough to show that the sequence
$$
n ~\mapsto ~ g[n,n+1,\ldots,n+p-1,n+x-k]
$$
converges to zero. However, by Lemma~\ref{lemma:pCInc5} this latter sequence can be sandwiched between the sequences
$$
n ~\mapsto ~ g[n-k,n+1-k,\ldots,n+p-1-k,n+x-k]
$$
and
$$
n ~\mapsto ~ g[n,n+1,\ldots,n+p-1,n+x],
$$
which both converge to zero by \eqref{eq:LDDiv}.

Finally, let $f\colon\R\setminus (-\N)\to\R$ be the unique function defined in the first part of this lemma. Using \eqref{eq:8Tel7scoE} and \eqref{eq:8TescoE3zz}, since $g$ lies in $\cD^p\cap\cK^p$ we obtain
\begin{eqnarray*}
f(x-m) &=& \Sigma g(x)- \sum_{k=1}^mg(x-k) ~=~ \lim_{n\to\infty}f_n^p[g](x)-\sum_{k=1}^mg(x-k)\\
&=& \lim_{n\to\infty}f_n^p[g](x-m),
\end{eqnarray*}
which establishes the second part of the lemma.
\end{proof}

Lemma~\ref{lemma:8EulRef4Fo9p} shows that the domain of the function $\Sigma g$ can be extended to $\R\setminus (-\N)$ whenever $g$ is defined on $\R\setminus\{0\}$. We then use the same symbol $\Sigma g$ for this extended function. Moreover, in this case we have
$$
\Sigma g(x) ~=~ \lim_{n\to\infty}f_n^p[g](x){\,},\qquad x\in \R\setminus (-\N)
$$
and the Eulerian form \eqref{eq:EulG4100} of $\Sigma g$ extends similarly. Actually, when $g$ is a function of a complex variable, Lemma~\ref{lemma:8EulRef4Fo9p} can be easily adapted to extend the function $\Sigma g$ to an appropriate complex domain.

Let us now establish reflection formulas on $\R\setminus\Z$ for functions $\Sigma g$ when the restriction of $g$ to $\R_+$ lies in $\cD^0\cap\cK^0$. The result is presented in the following two propositions, which deal separately with the cases when $g|_{\R\setminus\Z}$ is odd or even. The proofs of these propositions are similar and we therefore omit the second one.

\begin{proposition}\label{prop:8Eu4Ref0PO}
Let $g\colon\R\setminus\{0\}\to\R$ be such that $g|_{\R_+}$ lies in $\cD^0\cap\cK^0$ and let $\omega\colon\R\setminus\Z\to\R$ be the function defined by the equation
$$
\omega(x) ~=~ \Sigma g(x)-\Sigma g(1-x)\qquad\text{for $x\in\R\setminus\Z$}.
$$
Then the following assertions are equivalent.
\begin{enumerate}
  \item [(i)] The function $g|_{\R\setminus\Z}$ is odd.
  \item [(ii)] The function $\omega$ is $1$-periodic.
  \item [(iii)] We have that $g|_{\R\setminus\Z}$ vanishes at $-\infty$ and
  $$
  \omega(x) ~=~ \lim_{N\to\infty}\sum_{|k|\leq N}g(x+k),\qquad x\in\R\setminus\Z.
  $$
\end{enumerate}
\end{proposition}

\begin{proof}
The equivalence (i) $\Leftrightarrow$ (ii) is trivial since $\Delta\omega(x)=g(x)+g(-x)$. Let us prove the implication (iii) $\Rightarrow$ (ii). We have
\begin{eqnarray*}
\Delta\omega(x) &=& \lim_{N\to\infty} \sum_{|k|\leq N} (g(x+k+1)-g(x+k))\\
&=& \lim_{N\to\infty} (g(x+N+1)-g(x-N)) ~=~ 0.
\end{eqnarray*}
Finally, let us prove the implication (i) $\Rightarrow$ (iii). Using Lemma~\ref{lemma:8EulRef4Fo9p} we obtain
\begin{eqnarray*}
\omega(x) &=& \sum_{k=0}^{\infty}(g(x+k)+g(x-k-1))\\
&=& \lim_{N\to\infty} \left(-g(x-N-1)+\sum_{|k|\leq N} g(x+k)\right).
\end{eqnarray*}
This completes the proof.
\end{proof}

\begin{proposition}\label{prop:8Eu4Ref0PE}
Let $g\colon\R\setminus\{0\}\to\R$ be such that $g|_{\R_+}$ lies in $\cD^0\cap\cK^0$ and let $\omega\colon\R\setminus\Z\to\R$ be the function defined by the equation
$$
\omega(x) ~=~ \Sigma g(x)+\Sigma g(1-x)\qquad\text{for $x\in\R\setminus\Z$}.
$$
Then the following assertions are equivalent.
\begin{enumerate}
  \item [(i)] The function $g|_{\R\setminus\Z}$ is even.
  \item [(ii)] The function $\omega$ is $1$-periodic.
  \item [(iii)] We have that $g|_{\R\setminus\Z}$ vanishes at $-\infty$ and
  $$
  \omega(x) ~=~ -g(x)+\lim_{N\to\infty}\sum_{1\leq |k|\leq N}(g(k)-g(x+k)),\qquad x\in\R\setminus\Z.
  $$
\end{enumerate}
\end{proposition}

\begin{example}[The digamma function]\label{ex:8Ref9For3Dig}\index{digamma function}
Consider the odd function $g(x)=1/x$ on $\R\setminus\{0\}$ for which we have the identity $\Sigma g(x)=\psi(x)+\gamma$ (see Section~\ref{sec:dig27amma1}). This identity actually holds not only on $\R_+$ but also on $\R\setminus (-\N)$ since by Lemma~\ref{lemma:8EulRef4Fo9p} the digamma function $\psi$ extends to this larger domain through the following Eulerian form (see also Srivastava and Choi \cite[p.~24]{SriCho12})
$$
\psi(x) ~=~ -\gamma-\frac{1}{x}+\sum_{k=1}^{\infty}\left(\frac{1}{k}-\frac{1}{x+k}\right),\qquad x\in\R\setminus(-\N).
$$
Now, using Proposition~\ref{prop:8Eu4Ref0PO} we immediately obtain the identity
$$
\psi(x)-\psi(1-x) ~=~ \lim_{N\to\infty}\sum_{|k|\leq N}\frac{1}{x+k}{\,},\qquad x\in\R\setminus\Z,
$$
where the right-hand function is $1$-periodic. Finally, it can be proved (see, e.g., Aigner and Ziegler \cite[Chapter~26]{AigZie18}, Berndt \cite[p.~4]{Ber83}, and Graham {\em et al.} \cite[Eq.\ (6.88)]{GraKnuPat94}) that this function reduces to $-\pi\cot(\pi x)$. We then retrieve the reflection formula \eqref{eq:ReflDig5} for the digamma function.
\end{example}

\begin{example}[A variant of the digamma function]\index{digamma function}
Consider the even function $g(x)=1/{|x|}$ on $\R\setminus\{0\}$. Using Lemma~\ref{lemma:8EulRef4Fo9p}, we then obtain the following expression for $\Sigma g$ on $\R\setminus(-\N)$
$$
\Sigma g(x) ~=~ \sum_{k=0}^{\infty}\left(\frac{1}{k+1}-\frac{1}{|x+k|}\right),
$$
or equivalently,
$$
\Sigma g(x) ~=~ \sum_{k=0}^{\infty}\left(\frac{1}{k+1}-\frac{1}{x+k}\right)+\sum_{k=0}^{\infty}\left(\frac{1}{x+k}-\frac{1}{|x+k|}\right),
$$
where the first series reduces to $\psi(x)+\gamma$. If $x>0$, then the second series is zero. If $x<0$, it reduces to
\begin{eqnarray*}
\sum_{k=0}^{\infty}\min\left\{\frac{2}{x+k}{\,},0\right\} &=& \sum_{k=0}^{\lfloor -x\rfloor}\frac{2}{x+k} ~=~ 2\sum_{k=0}^{\lfloor -x\rfloor}\Delta_k\psi(x+k)\\
&=& 2{\,}(\psi(1-\{-x\})-\psi(x)).
\end{eqnarray*}
Using Proposition~\ref{prop:8Eu4Ref0PE}, we then obtain that the function
$$
\Sigma g(x)+\Sigma g(1-x) ~=~ -\frac{1}{|x|}
+\lim_{N\to\infty}\sum_{1\leq |k|\leq N}\left(\frac{1}{|k|}-\frac{1}{|x+k|}\right){\,},\qquad x\in\R\setminus\Z,
$$
is $1$-periodic. Using the reflection formula for $\psi$, we also obtain
\begin{eqnarray*}
\Sigma g(x)+\Sigma g(1-x) &=& \psi(\{x\})+\psi(1-\{x\})+2\gamma\\
&=& 2\,\psi(\{x\}) +\pi\cot(\pi x)+2\gamma{\,},\qquad x\in\R\setminus\Z,
\end{eqnarray*}
which provides a closed expression for this periodic function.
\end{example}

\begin{example}
Consider the function $g\colon\R\to\R$ defined by the equation
$$
g(x) ~=~ \frac{x+1}{x^2+1}\qquad\text{for $x\in\R$}.
$$
We observe that both functions $g(x)$ and $\tilde{g}(x)=g(-x)$ have restrictions to $\R_+$ that lie in $\cD^0\cap\cK^0$. However, the function $g$ is neither even nor odd. Denoting its even and odd parts by $g_+$ and $g_-$, respectively, we have
\begin{eqnarray*}
g_+(x) &=& \frac{g(x)+g(-x)}{2} ~=~ \frac{1}{x^2+1}{\,};\\
g_-(x) &=& \frac{g(x)-g(-x)}{2} ~=~ \frac{x}{x^2+1}{\,}.
\end{eqnarray*}
and we can derive a reflection formula for each of these functions.

Now, it is not difficult to see that (see Example~\ref{ex:5aRat7Fct8})
\begin{eqnarray*}
\Sigma g_+(x) &=& \Im (\psi(1+i)-\psi(x+i)){\,};\\
\Sigma g_-(x) &=& \Re (-\psi(1+i)+\psi(x+i)){\,}.
\end{eqnarray*}
Using Propositions~\ref{prop:8Eu4Ref0PO} and \ref{prop:8Eu4Ref0PE}, we then see that both functions
$$
\Sigma g_+(x)+\Sigma g_+(1-x)\qquad\text{and}\qquad\Sigma g_-(x)-\Sigma g_-(1-x)
$$
are $1$-periodic. Moreover, their sum $\Sigma g(x)+\Sigma\tilde{g}(1-x)$ is also $1$-periodic. Equivalently, the function
$$
\Re(\psi(x+i)-\psi(1-x+i))-\Im(\psi(x+i)+\psi(1-x+i))
$$
is $1$-periodic. However, we do not have a reflection formula for $\Sigma g$ or $\Sigma\tilde{g}$.
\end{example}

Although Propositions~\ref{prop:8Eu4Ref0PO} and \ref{prop:8Eu4Ref0PE} constitute major steps in the investigation of reflection formulas, they do not provide closed-form expressions for the $1$-periodic functions involved in these formulas. For instance, considering the reflection formula for the digamma function\index{digamma function} (see Example~\ref{ex:8Ref9For3Dig}), we see that Proposition~\ref{prop:8Eu4Ref0PO} does not yield the right-hand side of identity \eqref{eq:ReflDig5}. Moreover, it seems that such an expression, obtained for example using Herglotz's trick (see Aigner and Ziegler \cite[Chapter~26]{AigZie18}), is very specific to the case when $g(x)=1/x$. Now, finding a closed-form expression in the general case remains a very interesting open problem: such a result would provide an analogue of Euler's reflection formula for a wide class of functions. In this regard, we observe that Herglotz's trick uses an analogue of Legendre's duplication formula\index{Legendre's duplication formula} in the additive notation. Thus, a suitable adaptation of this trick could be helpful to tackle this problem.

Let us now investigate the more general case when the function $g|_{\R_+}$ lies in $\cD^p\cap\cK^p$ for some $p\in\N$. We observe that some reflection formulas can be obtained by integrating or differentiating both sides of a given reflection formula. Thus, if $g|_{\R_+}$ lies in $\cC^1\cap\cD^1\cap\cK^1$ for instance, we know from Proposition~\ref{prop:LMpGpLMp1} that $g'|_{\R_+}$ lies in $\cC^0\cap\cD^0\cap\cK^0$ and we may try to find a reflection formula for $\Sigma g'$ using Propositions~\ref{prop:8Eu4Ref0PO} and \ref{prop:8Eu4Ref0PE}. Since $\Sigma g'$ and $(\Sigma g)'$ differ by a constant by Proposition~\ref{prop:fdsf6sfd}, a reflection formula for $\Sigma g$ can then be obtained by integrating both sides of the reflection formula for $\Sigma g'$. This approach is inspired from the elevator method\index{elevator method} (as discussed in Section~\ref{sec:FSFD63}).

For instance, integrating both sides of \eqref{eq:ReflDig5} on $(\frac{1}{2},x)$, where $\frac{1}{2}<x<1$, we get the identity
$$
\ln\Gamma(x)+\ln\Gamma(1-x) ~=~ \ln(\pi\csc(\pi x)).
$$
Thus, we retrieve Euler's reflection formula on the interval $(\frac{1}{2},x)$ and this formula can be extended to the complex domain $\mathbb{C}\setminus\Z$ by analytic continuation. The identity \eqref{eq:ReflBGF5} can be obtained similarly, observing that
$$
\ln G(x+1) ~=~ \ln\Gamma(x)+\ln G(x).
$$

Now, let $g\colon\R\setminus\{0\}\to\R$ be a function such that $g|_{\R_+}$ lies in $\cD^p\cap\cK^p$ for some $p\in\N$. Let also $\omega_+[g]\colon\R\setminus\Z\to\R$ and $\omega_-[g]\colon\R\setminus\Z\to\R$ be the functions defined by the equation
$$
\omega_{\pm}[g](x) ~=~ \Sigma g(x)\pm\Sigma g(1-x)\qquad\text{for $x\in\R\setminus\Z$}.
$$
We then observe that
$$
\Delta\omega_{\pm}[g](x) ~=~ g(x) \mp g(-x),\qquad x\in\R\setminus\Z.
$$
It follows that $\omega_+$ (resp.\ $\omega_-$) is $1$-periodic if and only if $g|_{\R\setminus\Z}$ is even (resp.\ odd).

The following proposition provides an explicit expression for the function $\omega_{\pm}[g]$ whenever it is $1$-periodic. This expression is constructed from the very definition of $\Sigma g$.

\begin{proposition}\label{prop:8Per3Omeg6211}
Let $g\colon\R\setminus\{0\}\to\R$ be such that $g|_{\R_+}$ lies in $\cD^p\cap\cK^p$ for some $p\in\N$. Then the following assertions hold.
\begin{enumerate}
\item[(a)] If $g|_{\R\setminus\Z}$ is odd, then the function $\omega_-[g]$ is $1$-periodic and is equal to
$$
\lim_{n\to\infty}\bigg(-\sum_{|k|\leq n-1}g(x+k)-g(x-n)
+\sum_{j=1}^p\left(\tchoose{x}{j}-\tchoose{1-x}{j}\right)\Delta^{j-1}g(n)\bigg).
$$
\item[(b)] If $g|_{\R\setminus\Z}$ is even, then the function $\omega_+[g]$ is $1$-periodic and is equal to
\begin{multline*}
\lim_{n\to\infty}\bigg(-g(x)+\sum_{1\leq |k|\leq n-1}(g(k)-g(x+k))\\
\null -g(x-n)+\sum_{j=1}^p\left(\tchoose{x}{j}+\tchoose{1-x}{j}\right)\Delta^{j-1}g(n)\bigg).
\end{multline*}
\end{enumerate}
\end{proposition}

\begin{proof}
Let us prove assertion (a). That $\omega_-[g]$ is $1$-periodic is clear from the discussion above. Now, using Lemma~\ref{lemma:8EulRef4Fo9p} we obtain
\begin{multline*}
\omega_-[g](x) ~=~ \lim_{n\to\infty}(f_n^p[g](x)+f_n^p[g](1-x))\\
=~ \lim_{n\to\infty}\left(\sum_{k=0}^{n-1}(g(1-x+k)-g(x+k))
+\sum_{j=1}^p\left(\tchoose{x}{j}-\tchoose{1-x}{j}\right)\Delta^{j-1}g(n)\right).
\end{multline*}
This proves assertion (a). Assertion (b) can be established similarly.
\end{proof}

\begin{example}
Consider the odd function $g\colon\R\to\R$ defined by the equation
$$
g(x)=x-\frac{x}{x^2+1}\qquad\text{for $x\in\R$}.
$$
The function $g|_{\R_+}$ clearly lies in $\cD^2\cap\cK^2$ and we have (see Example~\ref{ex:5aRat7Fct8})
$$
\Sigma g(x)=\tchoose{x}{2}-\Re(\psi(x+i)).
$$
By Proposition~\ref{prop:8Per3Omeg6211}, the function
$$
\Sigma g(x)-\Sigma g(1-x) ~=~ \Re(\psi(1-x+i)-\psi(x+i))
$$
is $1$-periodic and is equal to the limit
$$
\lim_{n\to\infty}\left(-\sum_{|k|\leq n-1}h(x+k)-h(x-n)+(2x-1)h(n)\right),
$$
where $h(x)=g(x)-x$.
\end{example}

\begin{example}[Euler's reflection formula]\label{ex:8zzEu5g2s8o0}
Consider the even function $g\colon\R\setminus\{0\}$ defined by the equation $g(x) = \ln|x|$ for $x\in\R\setminus\{0\}$. The function $g|_{\R_+}$ clearly lies in $\cD^1\cap\cK^1$ and, since $\Delta_x\ln|\Gamma(x)|=\ln|x|$ on $\R\setminus (-\N)$, we must have
$$
\Sigma g(x) ~=~ \ln|\Gamma(x)|{\,},\qquad x\in\R\setminus (-\N).
$$
By Proposition~\ref{prop:8Per3Omeg6211}, the function $|\Gamma(x)\Gamma(1-x)|$ on $\R\setminus\Z$ is $1$-periodic and is equal to
$$
\lim_{n\to\infty}\left|\frac{1}{x}\prod_{1\leq |k|\leq n}\frac{k}{x+k}\right|.
$$
Euler's reflection formula then shows that this limit is also $|\pi\csc(\pi x)|$, as expected (see Artin \cite[p.~27]{Art15}).
\end{example}

\begin{remark}
We observe the following interesting link between the analogue of Euler's reflection formula and the logarithm of the generalized Stirling constant\index{Stirling's constant!generalized} (see Definition~\ref{de:GSC556}). Let $g\colon\R\setminus\{0\}\to\R$ be an even function such that $g|_{\R_+}$ lies in $\cC^0\cap\mathrm{dom}(\Sigma)$. Assume also that $g$ is integrable at $0$. Then, we have
$$
\overline{\sigma}[g|_{\R_+}] ~=~ \int_0^1\Sigma g(t){\,}dt ~=~ \frac{1}{2}\int_0^1(\Sigma g(t)+\Sigma g(1-t))dt{\,},
$$
that is,
$$
\overline{\sigma}[g|_{\R_+}] ~=~ \frac{1}{2}\int_0^1\omega_+[g](t){\,}dt.
$$
For instance, for the function $g(x) = \ln|x|$ (see Example~\ref{ex:8zzEu5g2s8o0}), we obtain
$$
\overline{\sigma}[g|_{\R_+}] ~=~ \frac{1}{2}\int_0^1\ln(\pi\csc(\pi t)){\,}dt
$$
and it is not difficult to see that this expression reduces to $\frac{1}{2}\ln(2\pi)$.
\end{remark}

\index{Euler's reflection formula!analogue|)}

\section{Analogue of Gauss' digamma theorem}
\label{sec:8GauDigTh7}
\index{digamma function!Gauss' digamma theorem!analogue|(}

The following formula, due to Gauss, enables one to compute the values of the digamma function\index{digamma function!Gauss' digamma theorem} $\psi$ for rational arguments. If $a,b\in\N^*$ with $a<b$, then we have
\begin{equation}\label{eq:GauDigThm3}
\psi\left(\frac{a}{b}\right) ~=~ -\gamma-\ln(2b)-\frac{\pi}{2}\,\cot\frac{a\pi}{b}
+2\sum_{j=1}^{\lfloor(b-1)/2\rfloor}\cos\left(2j\pi\,\frac{a}{b}\right)\ln\left(\sin\frac{j\pi}{b}\right)
\end{equation}
(see, e.g., Knuth \cite[p.~95]{Knu97} and Srivastava and Choi \cite[p.~30]{SriCho12}). This formula can be extended to all integers $a,b\in\N^*$ by means of the difference equation $\psi(x+1)-\psi(x)=1/x$.

For instance, we have
$$
\psi\left(\frac{3}{4}\right) ~=~ -\gamma +\frac{\pi}{2}-3\ln 2.
$$

It is natural to wonder if an analogue of formula \eqref{eq:GauDigThm3} holds for any multiple $\log\Gamma$-type function. Finding an analogue as beautiful as this formula seems to be hard. However, we have the following partial result.

\begin{proposition}\label{prop:GauProv5}
Let $g\in\cD^0\cap\cK^0$ and let $a,b\in\N^*$ with $a<b$. Then
$$
\Sigma g\left(\frac{a}{b}\right) ~=~ \frac{1}{b}\,\sum_{j=0}^{b-1}\left(1-\omega_b^{-aj}\right)S^b_j[g],
$$
where
$$
\omega_b ~=~ e^{\frac{2\pi i}{b}}\qquad\text{and}\qquad S^b_j[g] ~=~ \sum_{k=1}^{\infty}\omega_b^{jk}{\,}g\left(\frac{k}{b}\right).
$$
\end{proposition}

\begin{proof}
By definition of the map $\Sigma$, we have
\begin{eqnarray*}
\Sigma g\left(\frac{a}{b}\right) &=& \lim_{n\to\infty}\left(\sum_{k=1}^{n-1}g\left(\frac{bk}{b}\right)-\sum_{k=0}^{n-1}g\left(\frac{bk+a}{b}\right)\right)\\
&=& \lim_{n\to\infty}\sum_{k=1}^{bn-1}(u_b(k)-u_b(k-a)){\,}g\left(\frac{k}{b}\right),
\end{eqnarray*}
where $u_b(k)=1$, if $b$ divides $k$, and $u_b(k)=0$, otherwise; that is,
$$
u_b(k) ~=~ \frac{1}{b}\sum_{j=0}^{b-1}\omega_b^{jk}.
$$
This completes the proof.
\end{proof}

Proposition~\ref{prop:GauProv5} provides a first step in the search for an explicit expression for $\Sigma g(\frac{a}{b})$. Depending upon the function $g$, more computations may be necessary to obtain a useful expression. In this respect, the derivation of formula \eqref{eq:GauDigThm3} by means of Proposition~\ref{prop:GauProv5} can be found in Marichal \cite[p.~13]{Mar19}.

\begin{example}
Let us apply Proposition~\ref{prop:GauProv5} to the function $g_s(x)=-x^{-s}$, where $s>1$. This function lies in $\cD^0\cap\cK^0$ and we have $\Sigma g_s(x)=\zeta(s,x)-\zeta(s)$; see Example~\ref{ex:HuZe82}. Let $a,b\in\N^*$ with $a<b$. For $j=0,\ldots,b-1$, we then have
$$
S_j^b[g_s] ~=~ -b^s\,\mathrm{Li}_s(\omega_b^j),
$$
where
$$
\mathrm{Li}_s(z) ~=~ \sum_{k=1}^{\infty}\frac{z^k}{k^s}
$$
is the polylogarithm function.\index{polylogarithm function|textbf} Using Proposition~\ref{prop:GauProv5}, we then obtain
\begin{eqnarray*}
\zeta\left(s,\frac{a}{b}\right) &=& \zeta(s)-b^{s-1}\,\sum_{j=0}^{b-1}\left(1-\omega_b^{-aj}\right)\mathrm{Li}_s(\omega_b^j)\\
&=& b^{s-1}\,\sum_{j=0}^{b-1}\omega_b^{-aj}\,\mathrm{Li}_s(\omega_b^j).
\end{eqnarray*}
The inverse conversion formula is simply given by
$$
\mathrm{Li}_s(\omega_b^j) ~=~ b^{-s}\,\sum_{k=1}^b\omega_b^{jk}\,\zeta\left(s,\frac{k}{b}\right),\qquad j=1,\ldots,b-1.\qedhere
$$
\end{example}

\index{digamma function!Gauss' digamma theorem!analogue|)}

\section{Generalized Gautschi's inequality}
\label{sec:8Gautschi334}
\index{Gautschi's inequality!generalized|(}

Gautschi~\cite{Gau59} showed that the following double inequality holds for any $0\leq a\leq 1$
$$
e^{(a-1)\,\psi(x+1)} ~\leq ~ \frac{\Gamma(x+a)}{\Gamma(x+1)} ~\leq ~ x^{a-1},\qquad x>0.
$$
As a consequence, since $\psi(x)<\ln x$ for any $x>0$, he also obtained that
$$
(x+1)^{a-1} ~\leq ~ \frac{\Gamma(x+a)}{\Gamma(x+1)} ~\leq ~ x^{a-1},\qquad x>0,
$$
which is also a straightforward consequence of the Wendel inequality \eqref{eq:Orig9Wendel2}. We refer to these inequalities as the \emph{Gautschi inequality}.\index{Gautschi's inequality}

We now provide an analogue of Gautschi's inequality for certain multiple $\log\Gamma$-type functions and for any $a\geq 0$. We call it the \emph{generalized Gautschi's inequality}. As usual, we use the additive notation.

\begin{proposition}[Generalized Gautschi's inequality]\label{prop:Gautschi56}
Suppose that $g$ lie in $\cC^2\cap\cD^p\cap\cK^{\max\{p,2\}}$ for some $p\in\N$ and let $a\geq 0$ and $x>0$ be so that $\Sigma g$ is convex on $[x+\lfloor a\rfloor,\infty)$. Then we have
\begin{eqnarray*}
(a-\lceil a\rceil){\,}g(x+\lceil a\rceil) &\leq & (a-\lceil a\rceil){\,}(\Sigma g)'(x+\lceil a\rceil)\\
&\leq & \Sigma g(x+a)-\Sigma g(x+\lceil a\rceil) ~\leq ~ (a-\lceil a\rceil){\,}g(x+\lfloor a\rfloor){\,}.
\end{eqnarray*}
(The inequalities are to be reversed if $\Sigma g$ is concave on $[x+\lfloor a\rfloor,\infty)$.)
\end{proposition}

\begin{proof}
We follow the same steps as in Gautschi's proof. We can assume that $k\leq a<k+1$ for some fixed $k\in\N$. Let $x>0$ be fixed so that $\Sigma g$ is convex on $[x+k,\infty)$. Let also $f\colon [k,k+1)\to\R$ and $\varphi\colon [k,k+1)\to\R$ be the functions defined by the equations
$$
f(a) ~=~ \frac{1}{k+1-a}{\,}(\Sigma g(x+a)-\Sigma g(x+k+1))
$$
and
$$
\varphi(a) ~=~ (k+1-a)^2f'(a)
$$
for $k\leq a< k+1$. We then observe that
$$
(k+1-a){\,}f'(a) ~=~ f(a)+D_a{\,}((k+1-a){\,}f(a)) ~=~ f(a)+(\Sigma g)'(x+a).
$$
It then follows that
$$
\varphi(a) ~=~ (k+1-a){\,}(f(a)+(\Sigma g)'(x+a))
$$
and
$$
\varphi'(a) ~=~ (k+1-a){\,}(\Sigma g)''(x+a).
$$
We also have
\begin{eqnarray*}
\varphi(k) &=& \Sigma g(x+k)-\Sigma g(x+k+1)+(\Sigma g)'(x+k)\\
&=& (\Sigma g)'(x+k)-g(x+k),
\end{eqnarray*}
where
$$
g(x+k) ~=~ \int_0^1(\Sigma g)'(x+k+t){\,}dt.
$$
Since $\Sigma g$ is convex on $[x+k,\infty)$, its derivative is increasing on $[x+k,\infty)$, and hence we must have $\varphi(k)\leq 0$ and $\varphi'(a)\geq 0$. Since $\varphi(k+1)=0$, it follows that the function $\varphi$ is nonpositive and hence that the function $f$ is decreasing. Using L'Hospital's rule and the fact that $\varphi(k)\leq 0$, we then obtain the following chain of inequalities
\begin{eqnarray*}
-g(x+k+1) &\leq & -(\Sigma g)'(x+k+1)\\
&\leq & \lim_{a\to k+1}f(a) ~\leq ~ f(a) ~\leq ~ f(k) ~=~ -g(x+k).
\end{eqnarray*}
This proves the result.
\end{proof}

\begin{example}
Applying Proposition~\ref{prop:Gautschi56} to $g(x)=\ln x$ and $p=1$, we obtain for any $a\geq 0$ and any $x>0$
$$
(x+\lceil a\rceil)^{a-\lceil a\rceil} ~\leq ~ e^{(a-\lceil a\rceil)\,\psi(x+\lceil a\rceil)} ~\leq ~ \frac{\Gamma(x+a)}{\Gamma(x+\lceil a\rceil)} ~\leq ~ (x+\lfloor a\rfloor)^{a-\lceil a\rceil}{\,}.
$$
If we assume that $0\leq a\leq 1$, then we retrieve the original Gautschi inequality.
\end{example}

\begin{remark}\label{rem:Gautschi44zz}
If we wish to bracket the function $\Sigma g(x+a)-\Sigma g(x+1)$ in Proposition~\ref{prop:Gautschi56}, we can use the identity
$$
\Sigma g(x+\lceil a\rceil) ~=~ \Sigma g(x+1)+\sum_{k=1}^{\lceil a\rceil-1}g(x+k),
$$
which immediately follows from \eqref{eq:56zzSec32S6}. For instance, for $g(x)=\ln x$ we obtain the double inequality
\begin{eqnarray*}
e^{(a-\lceil a\rceil)\,\psi(x+\lceil a\rceil)}(x+\lceil a\rceil -1)^{\underline{\lceil a\rceil -1}} &\leq & \frac{\Gamma(x+a)}{\Gamma(x+1)}\\
& \leq & (x+\lfloor a\rfloor)^{a-\lceil a\rceil}(x+\lceil a\rceil -1)^{\underline{\lceil a\rceil -1}}{\,}.
\end{eqnarray*}
which holds for any $a\geq 0$ and any $x>0$.
\end{remark}

We end this section with the following corollary, which is obtained by integrating on $a\in (0,1)$ the expressions in the generalized Gautschi inequality (Proposition~\ref{prop:Gautschi56}).

\begin{corollary}
Suppose that $g$ lie in $\cC^2\cap\cD^p\cap\cK^{\max\{p,2\}}$ and let $x>0$ be so that $\Sigma g$ is convex on $[x,\infty)$. Then we have
\begin{eqnarray*}
-\frac{1}{2}{\,}g(x+1) &\leq & -\frac{1}{2}{\,}(\Sigma g)'(x+1)\\
&\leq & \int_x^{x+1}\Sigma g(t){\,}dt -\Sigma g(x+1) ~\leq ~ -\frac{1}{2}{\,}g(x){\,}.
\end{eqnarray*}
(The inequalities are to be reversed if $\Sigma g$ is concave on $[x,\infty)$.) In particular, the following assertions hold.
\begin{enumerate}
\item[(a)] If $\Sigma g$ is not eventually identically zero and if
\begin{equation}\label{eq:8GautLim55}
\lim_{x\to\infty}\frac{g(x)}{\Sigma g(x)} ~=~ 0,
\end{equation}
then
$$
\lim_{x\to\infty}\frac{(\Sigma g)'(x)}{\Sigma g(x)} ~=~ 0\qquad\text{and}\qquad \Sigma g(x) ~\sim ~ \int_x^{x+1}\Sigma g(t){\,}dt\quad\text{as $x\to\infty$}.
$$

\item[(b)] If $g$ is not eventually identically zero and if
$$
\lim_{x\to\infty}\frac{g(x+1)}{g(x)} ~=~ 1,
$$
then
$$
\lim_{x\to\infty}\frac{(\Sigma g)'(x)}{g(x)} ~=~ 1\qquad\text{and}\qquad \lim_{x\to\infty}\frac{\int_x^{x+1}\Sigma g(t){\,}dt -\Sigma g(x)}{g(x)} ~=~ \frac{1}{2}{\,}.
$$
\end{enumerate}
\end{corollary}

\begin{proof}
The inequalities are obtained by integrating on $a\in (0,1)$ the expressions in the generalized Gautschi inequality. Let us now prove assertion (a); the second one can be established similarly. If $\Sigma g$ is not eventually identically zero, then it eventually never vanishes since it lies in $\cK^0$. If condition \eqref{eq:8GautLim55} holds, then we must have
$$
\lim_{x\to\infty}\frac{\Sigma g(x+1)}{\Sigma g(x)} ~=~ \lim_{x\to\infty}\left(1+\frac{g(x)}{\Sigma g(x)}\right) ~=~ 1\qquad\text{and}\qquad\lim_{x\to\infty}\frac{g(x)}{\Sigma g(x+1)} ~=~ 0.
$$
We then complete the proof by dividing all the expressions in the inequalities by $\Sigma g(x+1)$ and letting $x\to\infty$.
\end{proof}

\index{Gautschi's inequality!generalized|)}

\section{Generalized Webster's functional equation}
\label{sec:8GenWEb3Fuc7E}
\index{Webster's functional equation!generalized|(}

In the framework of $\Gamma$-type functions, Webster \cite[Section~8]{Web97b} investigated the multiplicative version of the functional equation\index{Webster's functional equation}
$$
\textstyle{f(x)+f(x+\frac{1}{2})} ~=~ h(x),\qquad x>0,
$$
and, more generally, of the functional equation
$$
\sum_{j=0}^{m-1}f\left(x+\frac{j}{m}\right) ~=~ h(x),\qquad x>0,
$$
for any $m\in\N^*$, where $h\colon\R_+\to\R$ is a given function satisfying certain conditions.

In this section, we extend Webster's result by considering and solving the more general equation
\begin{equation}\label{eq:AFuncEq}
\sum_{j=0}^{m-1}f(x+a{\,}j) ~=~ h(x),\qquad x>0,
\end{equation}
where $a>0$ is also a given parameter. We call it the \emph{generalized Webster functional equation}. For instance, we can prove that the unique monotone solution $f\colon\R_+\to\R$ to the equation
$$
f(x)+f(x+a) ~=~ \frac{1}{x}
$$
is given by
$$
f(x) ~=~ \frac{1}{2a}\,\psi\left(\frac{x+a}{2a}\right)-\frac{1}{2a}\,\psi\left(\frac{x}{2a}\right).
$$

Our general result is stated in the following theorem, a variant of which was established by Webster \cite[Theorem 8.1]{Web97b} in the special case when $p=1$ and $a=\frac{1}{m}$.

\begin{theorem}[Generalized Webster's functional equation]\label{thm:FunctEq69}
Let $p\in\N$, $m\in\N^*$, $a>0$, and $h\in\cD^q\cap\cK^q$ for some integer $q\geq p$. Define also the function $h_a\colon\R_+\to\R$ by the equation
$$
h_a(x) ~=~ h(ax)\qquad\text{for $x>0$}.
$$
If $\Delta h_a$ lies in $\cD^p\cap\cK^p_+\cap\cK^q$ (resp.\ $\cD^p\cap\cK^p_-\cap\cK^q$), then there is a unique solution to equation \eqref{eq:AFuncEq} lying in $\cK^p$, namely
$$
f(x) ~=~ \Sigma h_{am}\left(\frac{x+a}{am}\right)-\Sigma h_{am}\left(\frac{x}{am}\right).
$$
Moreover, this solution lies in $\cK^p_-$ (resp.\ $\cK^p_+$).
\end{theorem}

\begin{proof}
Suppose for instance that $\Delta h_a$ lies in $\cD^p\cap\cK^p_+\cap\cK^q$ and let $g_a^m\colon\R_+\to\R$ be defined by the equation $g_a^m(x)=\Delta h_a(mx)$ for $x>0$. By Corollary~\ref{cor:Hom4}, the function $g_a^m$ lies in $\cD^p\cap\cK^p_+\cap\cK^q$. Suppose that $f\colon\R_+\to\R$ is a solution to equation \eqref{eq:AFuncEq}.  Then necessarily
$$
g_a^m(x) ~=~ h(amx+a)-h(amx) ~=~ \sum_{j=0}^{m-1}\Delta_jf(amx+aj) ~=~ \Delta_x f(amx).
$$
If $f$ lies in $\cK^p$, then by the uniqueness and existence theorems we have that
$$
f(amx) ~=~ f(am)+\Sigma g_a^m(x)
$$
and $f$ must lie in $\cK^p_-$. Since both $g_a^m$ and $h$ lie in $\cD^q\cap\cK^q$, by Propositions~\ref{prop:gStHa} and \ref{prop:gStHa223} we then have
\begin{eqnarray*}
f(amx) &=& f(am)+\Sigma_x h(amx+a)-\Sigma h(amx)\\
&=& f(am)+\Sigma_x h_{am}\left(x+\frac{1}{m}\right)-\Sigma h_{am}(x)\\
&=& c+\Sigma h_{am}\left(x+\frac{1}{m}\right)-\Sigma h_{am}(x),
\end{eqnarray*}
or equivalently,
\begin{equation}\label{eq:AFuncEq44}
f(x) ~=~ c+\Sigma h_{am}\left(\frac{x+a}{am}\right)-\Sigma h_{am}\left(\frac{x}{am}\right)
\end{equation}
for some $c\in\R$. But the function $f$ specified by \eqref{eq:AFuncEq44} satisfies \eqref{eq:AFuncEq} if and only if $c=0$; indeed, we then have
\begin{eqnarray*}
\sum_{j=0}^{m-1}f(x+aj) &=& mc+\sum_{j=0}^{m-1}\Delta_j\,\Sigma h_{am}\left(\frac{x+aj}{am}\right)\\
&=& mc+\Delta\Sigma h_{am}\left(\frac{x}{am}\right) ~=~ mc+h(x).
\end{eqnarray*}
This completes the proof.
\end{proof}

\begin{example}
Theorem~\ref{thm:FunctEq69} shows that the unique eventually monotone or eventually log-convex solution to the functional equation
$$
f(x)f(x+a){\,}x^p ~=~ 1,\qquad x>0,{\,}a>0,{\,}p>0,
$$
is the function
$$
f(x) ~=~ \left(\frac{\Gamma(\frac{x}{2a})}{\sqrt{2a}{\,}\Gamma(\frac{x+a}{2a})}\right)^p.
$$
This result was established by Thielman~\cite{Thi41} (see also Anastassiadis \cite{Ana60}). The special case when $p=1$ was previously shown by Mayer~\cite{May39}.
\end{example}

Combining both Theorems~\ref{thm:MultThmGen} and \ref{thm:FunctEq69}, we can derive immediately the following corollary, which in a sense provides yet another characterization of multiple $\Gamma$-type functions. For a similar result on the gamma function, see Artin \cite[p.~35]{Art15}.

\begin{corollary}
Let $p\in\N$, $m\in\N^*$, and $g\in\cD^p\cap\cK^{p+1}$. Define also the function $g_m\colon\R_+\to\R$ by the equation $g_m(x)=g(\frac{x}{m})$ for $x>0$. Then the function $f=\Sigma g$ is the unique solution lying in $\cK^p$ to the equation
$$
\sum_{j=0}^{m-1}f\left(\frac{x+j}{m}\right) ~=~ \sum_{j=1}^m\Sigma g\left(\frac{j}{m}\right)+\Sigma g_m(x),\qquad x>0.
$$
\end{corollary}

\begin{example}
For any $m\in\N^*$ the gamma function is the unique log-convex solution $f\colon\R_+\to\R_+$ to the equation
$$
\prod_{j=0}^{m-1}f\left(\frac{x+j}{m}\right) ~=~ \frac{\Gamma(x)}{m^{x-\frac{1}{2}}}{\,}(2\pi)^{\frac{m-1}{2}},\qquad x>0.
$$
Equivalently, for any $m\in\N^*$ the gamma function is the unique log-convex solution $f\colon\R_+\to\R_+$ to the equation
$$
\prod_{j=0}^{m-1}f\left(\frac{x+j}{m}\right) ~=~ \prod_{j=0}^{m-1}\Gamma\left(\frac{x+j}{m}\right),\qquad x>0.\qedhere
$$
\end{example}

\index{Webster's functional equation!generalized|)}

\chapter{Summary of the main results}
\label{chapter:9}

Now that we have collected a number of relevant results on multiple $\log\Gamma$-type functions, we naturally look forward to applying them on various examples, including not only special functions related to the gamma function but also many other useful functions of mathematical analysis. Such applications will be discussed in the next three chapters. But first and foremost, it is time to take stock of the new theory we have developed and summarize what we have found and learned thus far.

This chapter is devoted to a review of the most interesting and useful results that we have established in the previous chapters. These results are presented here as a step-by-step plan in order to perform a systematic and efficient investigation of the multiple $\log\Gamma$-type functions. We have tried to be as self-contained as possible, so that the reader can skip Chapters~\ref{chapter:2} to \ref{chapter:8} and make direct use of the summary given in this chapter.

\begin{remark}\label{rem:IntVSCont3}
At many places in this book (e.g., in Proposition~\ref{prop:intMLGt}), we have made the assumption that the function $g$ (resp.\ $g^{(r)}$ for some $r\in\N^*$) is continuous to ensure the existence of certain integrals. Although we can often relax this condition by simply requiring that $g$ (resp.\ $g^{(r)}$) is locally integrable, we have kept this continuity assumption for simplicity and consistency with similar results where higher order differentiability is assumed.
\end{remark}

\section{Basic definitions}

Let us recall a few useful concepts introduced in the previous chapters. For any $p\in\N$ and any $\S\in\{\N,\R\}$, we let $\cD^p_{\S}$ denote the set of functions $g\colon\R_+\to\R$ having the asymptotic property that
$$
\Delta^p g(x) ~\to ~0 \qquad\text{as $x\to_{\S}\infty$.}
$$
For any $p\in\N$, we also let $\cC^p$ denote the set of $p$ times continuously differentiable functions from $\R_+$ to $\R$ and we let $\cK^p$ denote the set of functions from $\R_+$ to $\R$ that are eventually $p$-convex or eventually $p$-concave, that is, $p$-convex\index{$p$-convexity} or $p$-concave\index{$p$-concavity} (see Definition~\ref{de:pconcconv}) in a neighborhood of infinity. Recall also that the sets $\cD^p_{\S}$'s are increasingly nested while the sets $\cC^p$'s and $\cK^p$'s are decreasingly nested, that is,
$$
\cD^p_{\S}\subset\cD^{p+1}_{\S},\qquad\cK^{p+1}\subset\cK^p,\qquad\text{and}\quad\cC^{p+1}\subset\cC^p\qquad\text{for any $p\in\N$.}
$$
We have also proved in Proposition~\ref{prop:ClClIntg} that
$$
\cD^p_{\N}\cap\cK^p ~=~ \cD^p_{\R}\cap\cK^p
$$
and we denote this common intersection simply by $\cD^p\cap\cK^p$.

%

In Chapter~\ref{chapter:5}, we have introduced the map $\Sigma$ that carries any function $g\colon\R_+\to\R$ lying in the set
$$
\mathrm{dom}(\Sigma) ~=~ \bigcup_{p\geq 0}(\cD^p\cap\cK^p)
$$
into the unique solution $f\colon\R_+\to\R$ that arises from Theorem~\ref{thm:int2} and satisfies $f(1)=0$. That is,
$$
\Sigma g(x) ~=~ \lim_{n\to\infty}f^p_n[g](x),\qquad x>0.
$$
The class of functions that are equal (up to an additive constant) to $\Sigma g$ is called the \emph{principal indefinite sum}\index{principal indefinite sum} of $g$ (see Definition~\ref{de:PIS43} and Example~\ref{ex:PIS43LOGG}). A function $f$ lying in the range of the map $\Sigma$ is also called a \emph{multiple $\log\Gamma$-type function}.\index{multiple $\log\Gamma$-type function}

In the previous chapters, we have established and discussed several properties of the multiple $\log\Gamma$-type functions, many of which are counterparts of classical properties of the gamma function. For instance, we have proved that every multiple $\log\Gamma$-type function satisfies an analogue of Gauss' multiplication formula for the gamma function. In the rest of this chapter, we provide a summary of these properties. The reader can use them for a systematic investigation of any multiple $\log\Gamma$-type function.

\section{ID card and main characterization}

The first step in this investigation is to choose a function $g\in\cD^p\cap\cK^p$ (for some $p\in\N$) for which we wish to study its principal indefinite sum $\Sigma g$. For instance, if we consider the function $g(x)=x\ln x$, which lies in $\cD^2\cap\cK^2$, then the function $\Sigma g$ is the logarithm of the hyperfactorial function\index{hyperfactorial function} $K(x)$ (see Section~\ref{sec:hypFacF4}), that is
$$
\Sigma g(x) ~=~ \ln K(x) ~=~ (x-1)\ln\Gamma(x)-\ln G(x),
$$
where $G$ is the Barnes $G$-function.\index{Barnes's $G$-function} Our results will then enable us to study this function through several of its properties.

Alternatively, we can start from a given function $f\in\cK^p$ (for some $p\in\N$) that we wish to investigate and whose difference $g=\Delta f$ is a function that lies in $\cD^p\cap\cK^p$. For instance, we may want to investigate the $n$th degree Bernoulli polynomial\index{Bernoulli polynomials} $f(x)=B_n(x)$ by first observing that the function
$$
g(x) ~=~ \Delta f(x) ~=~ n{\,}x^{n-1}
$$
lies in $\cD^n\cap\cK^n$. We then have
$$
\Sigma g(x) ~=~ B_n(x)-B_1(1).
$$

\begin{remark}
To investigate a function $f\colon\R_+\to\R$ through our results, it is not enough to check that the difference $g=\Delta f$ lies in $\cD^p\cap\cK^p$ for some $p\in\N$. We also need to make sure that $f$ also lies in $\cK^p$. For instance, both functions
$$
f_1(x) ~=~ x+\sin(2\pi x)\qquad\text{and}\qquad f_2(x) ~=~ x+\theta_3(\pi x,1/2),
$$
where $\theta_3(u,q)$ is the Jacobi theta function\index{Jacobi theta function} defined by the equation
$$
\theta_3(u,q) ~=~ 1+2\sum_{n=1}^{\infty}q^{n^2}\cos(2nu),
$$
have the same difference $g=\Delta f_1=\Delta f_2=1$ in $\cD^1\cap\cK^1$ (and we have $\Sigma g(x)=x-1$). However, neither $f_1$ nor $f_2$ lies in $\cK^1$.
\end{remark}

\parag{ID card} It is convenient to start our investigation of the function $\Sigma g$ by collecting some basic properties of the function $g$, thus establishing a kind of ID card for that function.

Thus, we first consider a function $g\colon\R_+\to\R$. We then determine its asymptotic degree\index{asymptotic degree}
\begin{eqnarray*}
\deg g &=& -1+\min\{q\in\N: g\in\cD^q_{\R}\}\\
&=& -1+\min\{q\in\N: \Delta^q g(x)\to 0~\text{as}~x\to\infty\}.
\end{eqnarray*}
If $\deg g=\infty$ (e.g., when $g(x)=2^x$) or if $g\notin\cK^p$ for all $p\geq 1+\deg g$ (e.g., $g(x)=x+\frac{1}{x}\sin x$), then the function $\Sigma g$ does not exist and the investigation stops here. Otherwise, the functions $g$ and $\Sigma g$ lie in $\cD^p\cap\cK^p$ and $\cD^{p+1}\cap\cK^p$, respectively, where $p=1+\deg g$.

If $\deg g=-1$, it is important to check whether $g$ also lies in the set $\cD^{-1}_{\N}$ of functions $g\colon\R_+\to\R$ for which the sequence $n\mapsto g(n)$ is summable. In this case, by Proposition~\ref{prop:diffzz0} we have that
$$
\lim_{x\to\infty} \Sigma g(x) ~=~ \sum_{k=1}^{\infty}g(k).
$$

It is also useful to determine the integer $r\in\N$, if any, for which $g$ lies in $\cC^r\cap\cK^{\max\{p,r\}}$. In this case, we know from Theorem~\ref{thm:TBTDiff} that $\Sigma g$ lies also in this set. Moreover, many functions of mathematical analysis lie in both
$$
\cC^{\infty} ~=~ \bigcap_{r\geq 0}\cC^r\qquad\text{and}\qquad\cK^{\infty} ~=~ \bigcap_{p\geq 0}\cK^p.
$$
If $g$ lies in these sets, then we can write $g\in\cC^{\infty}\cap\cD^p\cap\cK^{\infty}$.

It may be also useful to determine the domain on which $g$ is $p$-convex or $p$-concave. For instance, the function $g(x)=\frac{1}{x}\ln x$ is $0$-concave on $[e,\infty)$, $1$-convex on $[e^{3/2},\infty)$, etc. (see Example~\ref{ex:5Alt88Stiel}).

Note that, at this stage, we may not yet have any simple expression for $\Sigma g$. Limit and series representations will later emerge anyway from our investigation.

\parag{Analogue of Bohr-Mollerup's theorem} The following characterization result constitutes the analogue of Bohr-Mollerup's theorem for the function $\Sigma g$ and follows immediately from the uniqueness Theorem~\ref{thm:unic}.\index{Bohr-Mollerup theorem!analogue}
\begin{quote}
\emph{If $f\colon\R_+\to\R$ is a solution to the equation $\Delta f=g$, then it lies in $\cK^p$ if and only if $f=c+\Sigma g$ for some $c\in\R$.}
\end{quote}
This characterization sometimes enables one to establish alternative expressions for the function $\Sigma g$. For instance, if $g(x)=\frac{1}{x}$, then we have
$$
\Sigma g(x) ~=~ \psi(x)+\gamma.
$$
Using the characterization above, we can easily establish the following Gauss representation (see, e.g., Srivastava and Choi \cite[p.~26]{SriCho12})
$$
\psi(x)+\gamma ~=~ \int_0^{\infty}\frac{e^{-t}-e^{-xt}}{1-e^{-t}}{\,}dt{\,},\qquad x>0.
$$
Indeed, both sides of this identity vanish at $x=1$ and are eventually increasing solutions to the equation $\Delta f=g$. Hence, by uniqueness they must coincide on $\R_+$.

Note also that, in addition to the analogue of Bohr-Mollerup's theorem above, we also have an alternative characterization of $\Sigma g$ given in Proposition~\ref{prop:90unic41}.

\section{Extended ID card}

We now complement the ID card of the function $g$ by considering some additional related constants and mappings. From now on, we assume that $g$ is at least continuous on $\R_+$. More precisely, we assume that
$$
g\in\cC^r\cap\cD^p\cap\cK^{\max\{p,r\}}
$$
for $p=1+\deg g$ and some $r\in\N$.

Recall also that, for any $n\in\N$, the symbols $G_n$ and $B_n$ denote the $n$th Gregory coefficient and the $n$th Bernoulli number, respectively. We also let
$$
\overline{G}_n ~=~ 1-\sum_{j=1}^n|G_j|
$$
and we let $B_n(x)$ denote the $n$th degree Bernoulli polynomial (see Sections~\ref{sec:6Asym4Cons6Bine}, \ref{sec:6Gen4St2Fo0}, and \ref{sec:G8S3F5R}).

\parag{Asymptotic constant}\index{asymptotic constant} Recall that the asymptotic constant associated with $g$ (see \eqref{eq:sigmagg86}) is the number
$$
\sigma[g] ~=~ \int_0^1\Sigma g(t+1){\,}dt ~=~ \int_1^2\Sigma g(t){\,}dt.
$$
If $g$ is integrable at $0$, we also define the generalized Stirling constant\index{Stirling's constant!generalized} (see Definition~\ref{de:GSC556}) as the number $\exp(\overline{\sigma}[g])$, where
$$
\overline{\sigma}[g] ~=~ \sigma[g] -\int_0^1g(t){\,}dt ~=~ \int_0^1\Sigma g(t){\,}dt.
$$
Since this latter constant does not always exist (e.g., when $g(x)=\frac{1}{x}$), we do not use it much in our investigation.

The asymptotic constant\index{asymptotic constant} $\sigma[g]$ has the following limit, series, and integral representations (see identities \eqref{eq:SgSt7FrI}, \eqref{eq:serS72}, \eqref{eq:sig51saz}, and Corollary~\ref{cor:Liu471S5}).

\begin{enumerate}
\item[(a)] If $g$ lies in $\cC^0\cap\cD^p\cap\cK^p$, then we have
$$
\sigma[g] ~=~ \sum_{j=1}^pG_j{\,}\Delta^{j-1}g(1) - \sum_{k=1}^{\infty}\left(\int_{k}^{k+1}g(t){\,}dt-\sum_{j=0}^pG_j{\,}\Delta^jg(k)\right)
$$
and
$$
\sigma[g] ~=~ \lim_{n\to\infty}\left(\sum_{k=1}^{n-1}g(k)-\int_1^ng(t){\,}dt+\sum_{j=1}^pG_j\Delta^{j-1}g(n)\right).
$$
\item[(b)] If $g$ lies in $\cC^{2q}\cap\cD^p\cap\cK^{2q}$, where $q\in\N^*\cup\{\frac{1}{2}\}$ and $0\leq p\leq 2q-1$, then we have
$$
\sigma[g] ~=~ \lim_{n\to\infty}\left(\sum_{k=1}^{n-1} g(k)-\int_1^n g(t){\,}dt -\sum_{k=1}^p\frac{B_k}{k!}{\,}g^{(k-1)}(n)\right).
$$
\item[(c)] If $g$ lies in $\cC^2\cap\cD^1\cap\cK^2$, then we have
$$
\sigma[g] ~=~ \frac{1}{2}{\,}g(1)+\int_1^{\infty}\textstyle{\left(\{t\}-\frac{1}{2}\right)g'(t){\,}dt}.
$$
\item[(d)] If $g$ lies in $\cC^{2q+1}\cap\cD^p\cap\cK^{2q+1}$, then we have
$$
\sigma[g] ~=~ \frac{1}{2}{\,}g(1)-\sum_{k=1}^q\frac{B_{2k}}{(2k)!}{\,}g^{(2k-1)}(1)  - \int_1^{\infty}\frac{B_{2q}(\{t\})}{(2q)!}{\,}g^{(2q)}(t){\,}dt.
$$
\end{enumerate}
We also know from Proposition~\ref{prop:diffzz0} that if $g$ lies in $\cC^0\cap\cD^{-1}\cap\cK^0$ (here $\cD^{-1}$ stands for $\cD^{-1}_{\N}$), then $g$ is integrable at infinity and
$$
\sigma[g] ~=~ \sum_{k=1}^{\infty}g(k)-\int_1^{\infty}g(t){\,}dt.
$$

\parag{Analogue of Raabe's formula}\index{Raabe's formula!analogue} The analogue of Raabe's formula is simply the identity (see \eqref{eq:ds68ffdsbis})
$$
\int_x^{x+1}\Sigma g(t){\,}dt ~=~ \sigma[g]+\int_1^xg(t){\,}dt
$$
and we know by Proposition~\ref{prop:an4tR8} that any of these integrals lies in $\cC^0\cap\cD^{p+1}\cap\cK^{p+1}$.

Recall also from Corollary~\ref{cor:InvRaa4c} that a function $f\colon\R_+\to\R$ lies in $\cC^0\cap\cK^p$ and satisfies the equation
$$
\int_x^{x+1}f(t){\,}dt ~=~ \sigma[g]+\int_1^xg(t){\,}dt{\,},\qquad x>0,
$$
if and only if $f=\Sigma g$. This provides an alternative characterization of $\Sigma g$.

\parag{Generalized Binet's function}\index{Binet's function!generalized} For any $q\in\N$, the generalized Binet function associated with $g$ and $q$ is the function $J^q[g]\colon\R_+\to\R$ defined by the equation (see \eqref{eq:Binet64378})
$$
J^q[g](x) ~=~ \sum_{j=0}^{q-1}G_j\Delta^jg(x)-\int_x^{x+1}g(t){\,}dt\qquad\text{for $x>0$}.
$$
In particular, we also have (see \eqref{eq:Binet64378S})
$$
J^{q+1}[\Sigma g](x) ~=~ \Sigma g(x)-\sigma[g]-\int_1^xg(t){\,}dt + \sum_{j=1}^qG_j\Delta^{j-1} g(x){\,}.
$$
Note that several objects and formulas of our theory can be usefully expressed in terms of this latter function.

\parag{Generalized Euler's constant}\index{Euler's constant!generalized} Recall that the generalized Euler constant associated with the function $g$ is the number
$$
\gamma[g] ~=~ -J^{p+1}[\Sigma g](1),
$$
where $p=1+\deg g$ (see Definition~\ref{de:GEC587}).

Note that, contrary to the asymptotic constant $\sigma[g]$, the generalized Euler constant $\gamma[g]$ is not invariant if we replace $p$ with a higher value. Besides, by definition of $\gamma[g]$ both quantities are related through the following identity\index{Euler's constant!generalized!in terms of the asymptotic constant}
$$
\sigma[g] ~=~ \gamma[g]+\sum_{j=1}^pG_j\,\Delta^{j-1}g(1),
$$
where $p=1+\deg g$ (see Proposition~\ref{prop:linksSG46}). In particular, we have $\gamma[g]=\sigma[g]$ whenever $\deg g=-1$.

We also have the following integral representations\index{Euler's constant!generalized!integral form}
$$
\gamma[g] ~=~ \int_1^{\infty}\bigg(\sum_{j=0}^pG_j\Delta^jg(\lfloor t\rfloor)-g(t)\bigg){\,}dt
$$
and
$$
\gamma[g] ~=~ \int_1^{\infty}\left(\overline{P}_p[g](t)-g(t)\right)dt,
$$
where
$$
\overline{P}_p[g](x) ~=~ \sum_{j=0}^p\tchoose{\{x\}}{j}\,\Delta^jg(\lfloor x\rfloor),\qquad x\geq 1,
$$
is the piecewise polynomial function whose restriction to any interval $(k,k+1)$, with $k\in\N^*$, is the interpolating polynomial\index{interpolating polynomial} of $g$ with nodes at $k,k+1,\ldots,k+p$ (see Proposition~\ref{prop:IntFogamma482} and Eqs.\ \eqref{eq:PieceWPol48} and \eqref{eq:RpmiIneq904}).

If $g$ is $p$-convex or $p$-concave on $[1,\infty)$, then the graph of $g$ is always over or always under that of $\overline{P}_p[g]$ on $[1,\infty)$ and $|\gamma[g]|$ is the surface area between both graphs. In this case, we also have (see \eqref{eq:RpmiIneq91} and \eqref{eq:RpmiIneq91b})
$$
|\gamma[g]| ~\leq ~ \overline{G}_p{\,}|\Delta^pg(1)|
$$
and, if $p\geq 1$,
$$
|\gamma[g]| ~\leq ~ \int_0^1\left|\tchoose{t-1}{p}\right|\left|\Delta^{p-1}g(t+1)-\Delta^{p-1}g(1)\right|{\,}dt.
$$

\section{Inequalities}

Recall that, for any $a>0$, the function $\rho^p_a[g]\colon [0,\infty)\to\R$ is defined by the equation (see \eqref{eq:deflambdapt})
$$
\rho^p_a[g](x) ~=~ g(x+a)-\sum_{j=0}^{p-1}\tchoose{x}{j}\,\Delta^jg(a)\qquad\text{for $x>0$}.
$$
In particular, we have
$$
\rho^{p+1}_a[\Sigma g](x) ~=~ \Sigma g(x+a)-\Sigma g(a)-\sum_{j=1}^p\tchoose{x}{j}\,\Delta^{j-1}g(a){\,}.
$$

\parag{Generalized Wendel's inequality (symmetrized version)}\index{Wendel's inequality!generalized} Let $a\geq 0$ and let $x>0$ be so that $g$ is $p$-convex or $p$-concave on $[x,\infty)$. Then we have (see Corollary~\ref{cor:AsymBehSolCo})
$$
\left|\rho^{p+1}_x[\Sigma g](a)\right| ~\leq ~ \lceil a\rceil\left|\tchoose{a-1}{p}\right|\left|\Delta^pg(x)\right|{\,}.
$$
If $p\geq 1$, we also have the following tighter inequality
$$
\left|\rho^{p+1}_x[\Sigma g](a)\right| ~\leq ~ \left|\tchoose{a-1}{p}\right|\left|\Delta^{p-1}g(x+a)-\Delta^{p-1}g(x)\right|{\,}.
$$
This latter inequality is referred to as the symmetrized version of the generalized Wendel inequality (see Corollary~\ref{cor:AsymBehSolCo}). Both inequalities reduce to equalities when $a\in\{0,1,\ldots,p\}$.

Now, for any $n\in\N^*$ we have (see \eqref{eq:33ConvSig52})
$$
\rho^{p+1}_n[\Sigma g](x) ~=~ \Sigma g(x)-f_n^p[g](x),\qquad x>0.
$$
Using this identity, we immediately derive the following discrete version of the inequalities above. If $g$ is $p$-convex or $p$-concave on $[n,\infty)$, then
$$
\left|\Sigma g(x)-f_n^p[g](x)\right| ~\leq ~ \lceil x\rceil\left|\tchoose{x-1}{p}\right|\left|\Delta^pg(n)\right|{\,},\qquad x>0,
$$
and if $p\geq 1$,
$$
\left|\Sigma g(x)-f_n^p[g](x)\right| ~\leq ~ \left|\tchoose{x-1}{p}\right|\left|\Delta^{p-1}g(n+x)-\Delta^{p-1}g(n)\right|{\,},\qquad x>0.
$$
If $g$ lies in $\cD^{-1}_{\N}$, then (see Proposition~\ref{prop:diffzz0})
$$
\Sigma g(x) ~\to ~ \Sigma g(\infty) ~=~ \sum_{k=1}^{\infty}g(k)\qquad\text{as $x\to\infty$}.
$$
We then have the following additional inequality (see Theorem~\ref{thm:existzz0}). If $g$ is increasing or decreasing on $[n,\infty)$, then
$$
\left|\sum_{k=n}^{\infty}g(x+k)\right| ~=~ |\Sigma g(x+n)-\Sigma g(\infty)| ~\leq ~ \left|\Sigma g(n)-\Sigma g(\infty)\right|,\qquad x>0.
$$

\parag{Generalized Stirling's formula-based inequality (symmetrized version)} If $x>0$ is so that $g$ is $p$-convex or $p$-concave on $[x,\infty)$, then we have the inequality (see Corollary~\ref{cor:6GenStFo0BaIneqCo})
$$
\left|J^{p+1}[\Sigma g](x)\right| ~\leq ~ \overline{G}_p{\,}|\Delta^p g(x)|.
$$
If $p\geq 1$, we also have the following tighter inequality
$$
\left|J^{p+1}[\Sigma g](x)\right| ~\leq ~ \left|\int_0^1\tchoose{t-1}{p}(\Delta^{p-1}g(x+t)-\Delta^{p-1}g(x)){\,}dt\right|.
$$
Moreover, if $p=0$ or $p=1$, then (see Proposition~\ref{prop:Burnside0})
$$
\left|\Sigma g\left(x+\frac{1}{2}\right)-\sigma[g]-\int_1^x g(t){\,}dt\right| ~\leq ~ \left|J^{p+1}[\Sigma g](x)\right|.
$$

\parag{Generalized Gautschi's inequality}\index{Gautschi's inequality!generalized} Suppose that $g$ lies in $\cC^2\cap\cK^2$. Let $a\geq 0$ and let $x>0$ be so that $\Sigma g$ is convex on $[x+\lfloor a\rfloor,\infty)$. Then we have (see Proposition~\ref{prop:Gautschi56})
\begin{eqnarray*}
(a-\lceil a\rceil){\,}g(x+\lceil a\rceil) &\leq & (a-\lceil a\rceil){\,}(\Sigma g)'(x+\lceil a\rceil)\\
&\leq & \Sigma g(x+a)-\Sigma g(x+\lceil a\rceil) ~\leq ~ (a-\lceil a\rceil){\,}g(x+\lfloor a\rfloor).
\end{eqnarray*}
(The inequalities are to be reversed if $\Sigma g$ is concave on $[x+\lfloor a\rfloor,\infty)$.)

\section{Asymptotic analysis}

In this section, we gather the main results related to the asymptotic behaviors of multiple $\log\Gamma$-type functions, including the generalized Stirling formula.

\parag{Generalized Wendel's inequality-based limit} The following convergence result immediately follows from the generalized Wendel inequality (see Theorem~\ref{thm:AsymBehSol}). For any $a\geq 0$, we have
$$
\rho^{p+1}_x[\Sigma g](a) ~\to ~0\qquad\text{as $x\to\infty$}{\,},
$$
or equivalently,
$$
\Sigma g(x+a)-\Sigma g(x)-\sum_{j=1}^p\tchoose{a}{j}\,\Delta^{j-1}g(x) ~\to ~0\qquad\text{as $x\to\infty$}{\,}.
$$
This convergence result still holds if we differentiate $r$ times the left-hand side.

\parag{Generalized Stirling's formula}\index{Stirling's formula!generalized} We have (see Theorem~\ref{thm:dgf7dds})
$$
J^{p+1}[\Sigma g](x) ~\to ~ 0\qquad\text{as $x\to\infty$}{\,},
$$
or equivalently,
$$
\Sigma g(x) -\int_1^x g(t){\,}dt +\sum_{j=1}^pG_j\Delta^{j-1}g(x) ~\to ~ \sigma[g]\qquad\text{as $x\to\infty$}{\,}.
$$
If $g$ lies in $\cC^{2q}\cap\cD^p\cap\cK^{2q}$, where $q\in\N^*\cup\{\frac{1}{2}\}$ and $0\leq p\leq 2q-1$, then we also have (see Proposition~\ref{prop:VarStir6})
$$
\Sigma g(x)-\int_1^x g(t){\,}dt -\sum_{k=1}^p\frac{B_k}{k!}{\,}g^{(k-1)}(x) ~\to ~ \sigma[g]\qquad\text{as $x\to\infty$.}
$$
If $p=0$ or $p=1$, we also have the following analogue of Burnside's formula,\index{Burnside's formula!analogue} which provides a better approximation than the generalized Stirling formula (see Proposition~\ref{prop:Burnside0})
$$
\Sigma g(x) -\int_1^{x-\frac{1}{2}}g(t){\,}dt ~\to ~ \sigma[g] \qquad \text{as $x\to\infty$}{\,}.
$$
All the convergence results above still hold if we differentiate $r$ times both sides. In particular, the function $D^rJ^{p+1}[\Sigma g]$ vanishes at infinity.

\parag{Asymptotic equivalences} For any $a\geq 0$ and any $c\in\R$, we have (see Proposition~\ref{prop:conv6v6})
$$
c+\Sigma g(x+a) ~\sim ~ c+\int_x^{x+1}\Sigma g(t){\,}dt\qquad\text{as $x\to\infty$}
$$
(under the assumption that $c+\Sigma g(n+1)\sim c+\Sigma g(n)$ as $n\to_{\N}\infty$ whenever $c+\Sigma g$ vanishes at infinity). If $g$ does not lie in $\cD^{-1}_{\N}$, then we also have
$$
\Sigma g(x+a) ~\sim ~ c+\int_1^x g(t){\,}dt\qquad\text{as $x\to\infty$}.
$$
These equivalences still hold if we differentiate $r$ times both sides; that is,
$$
D^r\Sigma g(x+a) ~\sim ~ g^{(r-1)}(x)\qquad\text{as $x\to\infty$}
$$
(under the assumption that $D^r\Sigma g(n+1)\sim D^r\Sigma g(n)$ as $n\to_{\N}\infty$ whenever $D^r\Sigma g$ vanishes at infinity).

\parag{Asymptotic expansions}\index{asymptotic expansion} We have the following asymptotic expansions (see Proposition~\ref{prop:Richardson9c}).
\begin{enumerate}
\item[(a)] If $g$ lies in $\cC^1\cap\cD^p\cap\cK^{\max\{p,1\}}$, then for large $x$ we have
$$
\Sigma g(x) ~=~ \sigma[g]+\int_1^xg(t){\,}dt -\frac{1}{2}{\,}g(x) + R_1(x){\,},
$$
where
$$
|R_1(x)| ~\leq ~\frac{1}{2}|g(x)|.
$$
\item[(b)] If $g$ lies in $\cC^{2q}\cap\cD^p\cap\cK^{\max\{p,2q\}}$ for some $q\in\N^*$, then for large $x$ we have
$$
\Sigma g(x) ~=~ \sigma[g]+\int_1^xg(t){\,}dt -\frac{1}{2}{\,}g(x)+\sum_{k=1}^q\frac{B_{2k}}{(2k)!}{\,}g^{(2k-1)}(x) + R^q_1(x){\,},
$$
where
$$
|R^q_1(x)| ~\leq ~ \frac{|B_{2q}|}{(2q)!}{\,}|g^{(2q-1)}(x)|{\,}.
$$
\end{enumerate}
Asymptotic expansions of the more general function
$$
x ~\mapsto ~\frac{1}{m}\sum_{j=0}^{m-1}\Sigma g\left(x+\frac{j}{m}\right),
$$
for any $m\in\N^*$, are also provided in Proposition~\ref{prop:Richardson9}.

\parag{Generalized Liu's formula}\index{Liu's formula!generalized} The following assertions hold (see Proposition~\ref{prop:Liu471}).
\begin{enumerate}
\item[(a)] If $g$ lies in $\cC^2\cap\cD^1\cap\cK^2$, then we have
$$
\Sigma g(x) ~=~ \sigma[g]+\int_1^x g(t){\,}dt -\frac{1}{2}{\,}g(x)-\int_0^{\infty}\textstyle{\left(\{t\}-\frac{1}{2}\right)g'(x+t){\,}dt}.
$$
\item[(b)] If $g$ lies in $\cC^{2q+1}\cap\cD^{2q}\cap\cK^{2q+1}$ for some $q\in\N^*$, then we have
\begin{eqnarray*}
\Sigma g(x) &=& \sigma[g]+\int_1^x g(t){\,}dt -\frac{1}{2}{\,}g(x)+\sum_{k=1}^q\frac{B_{2k}}{(2k)!}{\,}g^{(2k-1)}(x)\\
&& \null + \int_0^{\infty}\frac{B_{2q}(\{t\})}{(2q)!}{\,}g^{(2q)}(x+t){\,}dt.
\end{eqnarray*}
\end{enumerate}

\section{Limit, series, and integral representations}

We now recall the different representations of multiple $\log\Gamma$-type functions that we established in this work as well as the way we can generate further identities by integration and differentiation.

Note that, in the special case when $g$ lies in $\cD^{-1}_{\N}$, both the Eulerian\index{Eulerian form} and Weierstrassian forms\index{Weierstrassian form} coincide with the analogue of Gauss' limit, i.e., we have
$$
\Sigma g(x) ~=~ \sum_{k=1}^{\infty}g(k)-\sum_{k=0}^{\infty}g(x+k),
$$
and the second series converges uniformly on $\R_+$ (and tends to zero as $x\to\infty$).

\parag{Analogue of Gauss' limit}\index{Gauss' limit!analogue} By definition of $\Sigma g$, we have
$$
\Sigma g(x) ~=~ \lim_{n\to\infty} f^p_n[g](x),\qquad x>0.
$$
This is precisely the analogue of Gauss' limit for the gamma function. We have also established that the sequence $n\mapsto f^p_n[g]$ converges uniformly on any bounded subset of $\R_+$ to $\Sigma g$ (see our existence Theorem~\ref{thm:exist}).

More generally, we have shown that the sequence $n\mapsto D^rf^p_n[g]$ converges uniformly on any bounded subset of $\R_+$ to $D^r\Sigma g$ (see Theorem~\ref{thm:TBTDiff}). In particular, both sides of the identity above can be differentiated $r$ times (i.e., the limit and the derivative operator commute).

Moreover, the function $f_n^p[g](x)-\Sigma g(x)$ can be (repeatedly) integrated on any bounded interval of $[0,\infty)$ and the integral converges to zero as $n\to\infty$ (see Proposition~\ref{prop:intMLGt} and Remark~\ref{rem:5repeated0Int53}).

\parag{Eulerian and Weierstrassian forms}\index{Eulerian form}\index{Weierstrassian form} We have the following Eulerian form (see Theorem~\ref{thm:SerProdReprS})
$$
\Sigma g(x) ~=~ -g(x)+\sum_{j=1}^p\tchoose{x}{j}{\,}\Delta^{j-1}g(1) - \sum_{k=1}^{\infty}\left(g(x+k)-\sum_{j=0}^p\tchoose{x}{j}{\,}\Delta^jg(k)\right).
$$
We also have the following Weierstrassian forms if $g\in\cC^p$ (see Theorems~\ref{thm:Weierst1} and \ref{thm:Weierst}).
\begin{enumerate}
\item[(a)] If $p=1+\deg g=0$, then
$$
\Sigma g(x) ~=~ \sigma[g]-g(x)-\sum_{k=1}^{\infty}\left(g(x+k)-\int_k^{k+1}g(t){\,}dt\right).
$$
\item[(b)] If $p=1+\deg g\geq 1$, then
\begin{eqnarray*}
\Sigma g(x) &=& \sum_{j=1}^{p-1}\tchoose{x}{j}{\,}\Delta^{j-1}g(1)+\tchoose{x}{p}(\Sigma g)^{(p)}(1)\\
&& -g(x)- \sum_{k=1}^{\infty}\left(g(x+k)-\sum_{j=0}^{p-1}\tchoose{x}{j}{\,}\Delta^jg(k)-\tchoose{x}{p}g^{(p)}(k)\right),
\end{eqnarray*}
where $(\Sigma g)^{(p)}(1) = g^{(p-1)}(1)-\sigma[g^{(p)}]$.
\end{enumerate}
Each of the series above converges uniformly on any bounded subset of $[0,\infty)$ and can be repeatedly integrated term by term on any bounded interval of $[0,\infty)$. It can also be differentiated term by term up to $r$ times.

\parag{Gregory's formula-based series representation} We also have the following series representation (see Proposition~\ref{prop:6699rem}). Suppose that $g$ lies in $\cK^{\infty}$ and let $x>0$ be so that for every integer $q\geq p$ the function $g$ is $q$-convex or $q$-concave on $[x,\infty)$. Suppose also that the sequence $q\mapsto\Delta^qg(x)$ is bounded. Then we have
$$
\Sigma g(x) ~=~ \sigma[g]+\int_1^xg(t){\,}dt-\sum_{n=1}^{\infty}G_n\,\Delta^{n-1}g(x).
$$
Moreover, if these latter assumptions are satisfied for $x=1$, then we also have the following analogue of Fontana-Mascheroni's series representation of $\gamma$\index{Fontana-Mascheroni's series!analogue}
$$
\sigma[g] ~=~ \sum_{n=1}^{\infty}G_n\,\Delta^{n-1}g(1).
$$

\parag{Integral representation} We have seen that an integral expression for $\Sigma g$ can sometimes be obtained by first finding an expression for $\Sigma g^{(r)}$ when $r>1$. This is the elevator method\index{elevator method} (see Corollary~\ref{cor:saf6f}).

We have
$$
(\Sigma g)^{(r)}-\Sigma g^{(r)} ~=~ g^{(r-1)}(1)-\sigma[g^{(r)}]
$$
and, if $r>p$,
$$
\sigma[g^{(r)}] ~=~ g^{(r-1)}(1) +\sum_{k=1}^{\infty}g^{(r)}(k).
$$
Moreover, for any $a>0$, we have
$$
\Sigma g ~=~ f_a-f_a(1),
$$
where $f_a\in\cC^r$ is defined by
$$
f_a(x) ~=~ \sum_{k=1}^{r-1}c_k(a){\,}\frac{(x-a)^k}{k!} + \int_a^x \frac{(x-t)^{r-1}}{(r-1)!}{\,}(\Sigma g)^{(r)}(t){\,}dt
$$
and, for $k=1,\ldots,r-1$,
$$
c_k(a) ~=~
\sum_{j=0}^{r-k-1}\frac{B_j}{j!}{\,}\left(g^{(j+k-1)}(a)-\int_a^{a+1} \frac{(a+1-t)^{r-j-k}}{(r-j-k)!}{\,}(\Sigma g)^{(r)}(t){\,}dt\right).
$$

\section{Further identities and results}
\label{sec:F3I1aR4}

In this section, we collect the remaining identities and results that may be relevant in our investigation of multiple $\log\Gamma$-type functions.

\parag{Analogue of Gauss' multiplication formula}\index{Gauss' multiplication formula!analogue} Let $m\in\N^*$ and define the function $g_m\colon\R_+\to\R$ by the equation $g_m(x)=g(\frac{x}{m})$ for $x>0$. Then we have the following analogue of Gauss' multiplication formula (see Section~\ref{sec:GaussMultF51})
$$
\sum_{j=0}^{m-1}\Sigma g\left(x+\frac{j}{m}\right) ~=~ \sum_{j=1}^m\Sigma g\left(\frac{j}{m}\right)+\Sigma g_m(mx){\,},\qquad x>0,
$$
where
$$
\sum_{j=1}^m\Sigma g\left(\frac{j}{m}\right) ~=~ m{\,}\sigma[g]-\sigma[g_m]-m\,\int_{1/m}^1g(t){\,}dt.
$$
We also have
$$
\lim_{m\to\infty}\frac{\Sigma g_m(mx)-\Sigma g_m(m)}{m} ~=~ \int_1^x g(t){\,}dt{\,},\qquad x>0,
$$
and, if $g$ is integrable at $0$,
$$
\lim_{m\to\infty}\frac{1}{m}\,\Sigma g_m(mx) ~=~ \int_0^x g(t){\,}dt{\,},\qquad x>0.
$$
A related asymptotic result is also given in Proposition~\ref{prop:8Stir44Gau7Mult}.

\parag{Analogue of Wallis's product formula}\index{Wallis's product formula!analogue} We present here in a single statement the analogue of Wallis's product formula as given in Proposition~\ref{prop:Wallis92} and Remark~\ref{rem:Wall38}.

Let $\tilde{g}_1,\tilde{g}_2,\tilde{g}_3\colon\R_+\to\R$ be the functions defined respectively by the equations
$$
\tilde{g}_1(x) ~=~ \Delta g(2x-1),\quad \tilde{g}_2(x) ~=~ \Delta g(2x), \quad \tilde{g}_3(x) ~=~ 2{\,}g(2x),\quad\text{for $x>0$}.
$$
We assume that $\tilde{g}_{\ell}$ lies in $\cK^0$ for some $\ell\in\{1,2,3\}$.

Let also $\theta_1,\theta_2,\theta_3\colon\N^*\to\R$ be the sequences defined respectively by the equations
\begin{eqnarray*}
\theta_1(n) &=& \sigma[\tilde{g}_1]+\int_1^{n+1}\tilde{g}_1(t){\,}dt - \sum_{j=1}^{(p-1)_+}G_j\,\Delta^{j-1}\tilde{g}_1(n+1){\,},\\
\theta_2(n) &=& g(2n)-g(1)-\sigma[\tilde{g}_2]-\int_1^n\tilde{g}_2(t){\,}dt + \sum_{j=1}^{(p-1)_+}G_j\,\Delta^{j-1}\tilde{g}_2(n){\,},\\
\theta_3(n) &=& \sigma[\tilde{g}_3]-\sigma[g]+\int_1^2(g(2n+t)-g(t)){\,}dt\\
&& \null +\sum_{j=1}^pG_j\left(\Delta^{j-1}g(2n+1)-\Delta^{j-1}\tilde{g}_3(n+1)\right),
\end{eqnarray*}
for $n\in\N^*$. Then we have
$$
\lim_{n\to\infty}\left(h(n) + \sum_{k=1}^{2n}(-1)^{k-1}g(k)\right) ~=~ 0,
$$
where $h(n)$ is the function obtained from the series expansion for $\theta_{\ell}(n)$ about infinity after removing all the summands that vanish at infinity.

\parag{Restriction to the natural integers} The restriction of $\Sigma g$ to $\N^*$ is the sum \eqref{eq:RestrInt}. This sum can be estimated, e.g., by means of an integral through Gregory's summation formula \eqref{eq:GregoryMN} with a bounded remainder \eqref{eq:saf65fs}. The representations of $\Sigma g$ given above can also lead to interesting identities when restricted to the natural integers.

\parag{Analogue of Euler's series representation of $\gamma$}\index{Euler's series representation of $\gamma$!analogue} When $g$ lies in $\cC^{\infty}\cap\cK^{\infty}$, the following series (see \eqref{eq:EulerAnal5571})
$$
\sigma[g] ~=~ \sum_{k=1}^{\infty}(\Sigma g)^{(k)}(1)\,\frac{1}{(k+1)!}{\,},
$$
when it converges, provides an analogue of Euler's series representation of $\gamma$. It is obtained by integrating term by term the Taylor series expansion of $\Sigma g(x+1)$ about $x=0$.

\parag{Generalized Webster's functional equation} This result can be found in Theorem~\ref{thm:FunctEq69}.

\parag{Analogues of Euler's reflection formula and Gauss' digamma theorem} These topics are discussed in Sections~\ref{sec:ReflFor62} and \ref{sec:8GauDigTh7}.

\chapter{Applications to some standard special functions}
\label{chapter:10}

We now apply our results to certain multiple $\Gamma$-type functions and multiple $\log\Gamma$-type functions that are known to be well-studied special functions, namely: the gamma function, the digamma function, the polygamma functions, the $q$-gamma function, the Barnes $G$-function, the Hurwitz zeta function and its higher order derivatives, the generalized Stieltjes constants, and the Catalan number function. For recent background on some of these functions, see, e.g., Srivastava and Choi~\cite{SriCho12}.

Each of these examples is examined and studied systematically by following the steps and results given in the previous chapter. When algebraic computations become tedious, a computer algebra system can be of great assistance in executing the details. Further examples will be discussed in the next two chapters.

In this chapter and the next, we occasionally address and solve some secondary but interesting issues. They are then presented and numbered in a \emph{Project} environment.

Most of the applications we consider in this work illustrate how powerful is our theory to produce formulas and identities methodically. Although many of these formulas and identities are already known, to our knowledge they had never been derived from such a general and unified setting.

\section{The gamma function}
\index{gamma function|(}

Since the Euler gamma function was the starting point of this theory and therefore also Webster's motivating example in his introduction of the $\Gamma$-type functions, it is natural to test our results on this function first.

The following investigation of the gamma function does not reveal quite new formulas. However, it can be regarded as a tutorial that clearly demonstrates how our results can be used to carry out this investigation in a systematic way.

In addition to the remarkable book by Artin~\cite{Art15}, the interested reader can also find a very good expository tour of the gamma function in Srinivasan's paper~\cite{Sri07}.

\parag{ID card} The following table summarizes the ID card corresponding to the log and log-gamma functions.
$$
\begin{array}{|c|c|c|c|}
\hline g(x) & \text{Membership} & \deg g & \Sigma g(x) \\
\hline \emptybox \ln x & \cC^{\infty}\cap\cD^1\cap\cK^{\infty} & 0 & \ln\Gamma(x)\\
\hline
\end{array}
$$

\parag{Bohr-Mollerup's theorem} A characterization of the gamma function is given in Bohr-Mollerup's theorem\index{Bohr-Mollerup theorem} (see Theorem~\ref{thm:BM538Thm9} and Example~\ref{ex:unicGa4}). In the additive notation, we have the following statement.
\begin{quote}
\emph{All eventually convex or concave solutions $f\colon\R_+\to\R$ to the equation
$$
f(x+1)-f(x) ~=~ \ln x
$$
are of the form $f(x)=c+\ln\Gamma(x)$, where $c\in\R$.}
\end{quote}
Using Proposition~\ref{prop:90unic41}, we can also derive the following alternative characterization of the gamma function (see Example~\ref{ex:90unic41g}).
\begin{quote}
\emph{All solutions $f\colon\R_+\to\R$ to the equation
$$
f(x+1)-f(x) ~=~ \ln x
$$
that satisfy the asymptotic condition that, for each $x>0$,
$$
f(x+n)-f(n)-x\ln n ~\to ~0\qquad\text{as $n\to_{\N}\infty$}
$$
are of the form $f(x)=c+\ln\Gamma(x)$, where $c\in\R$.}
\end{quote}

\parag{Extended ID card} The value of $\sigma[g]$ has been discussed in Example~\ref{ex:Raab286}. More precisely, we also have the following values:
$$
\begin{array}{|c|c|c|}
\hline \overline{\sigma}[g] & \sigma[g] & \gamma[g] \\
\hline \emptybox \frac{1}{2}\ln(2\pi) & -1+\frac{1}{2}\ln(2\pi) & \gamma[g]=\sigma[g] \\
\hline
\end{array}
$$

\begin{itemize}
\item \emph{Inequality}
$$
|\sigma[g]| ~\leq ~ \ln 4-\frac{5}{4} ~\approx ~ 0.14.
$$

\item \emph{Alternative representations of $\sigma[g]=\gamma[g]$}
\begin{eqnarray*}
\sigma[g] &=& \int_1^{\infty}\left(\{t\}\ln\frac{1+\lfloor t\rfloor}{t}+(1-\{t\})\ln\frac{\lfloor t\rfloor}{t}\right)dt{\,},\\
\sigma[g] &=& \lim_{n\to\infty}\left(\ln n! +n-1-\left(n+\frac{1}{2}\right)\ln n\right),\\
\sigma[g] &=& \sum_{k=1}^{\infty}\left(1-\left(k+\frac{1}{2}\right)\ln\left(1+\frac{1}{k}\right)\right),\\
\sigma[g] &=& \int_1^{\infty}\left(\frac{1}{2}\ln(\lfloor t\rfloor^2 +\lfloor t\rfloor)-\ln t\right)dt{\,},\\
\sigma[g] &=& \int_1^{\infty}\frac{\{t\}-\frac{1}{2}}{t}{\,}dt{\,},\\
\sigma[g] &=& \int_0^1\ln\Gamma(t+1){\,}dt.
\end{eqnarray*}

\item \emph{Binet's function}
$$
J^2[\ln\circ\Gamma](x) ~=~ J(x) ~=~ \ln\Gamma(x)-\frac{1}{2}\ln(2\pi)+x-\left(x-\frac{1}{2}\right)\ln x{\,},\qquad x>0.
$$

\item \emph{Raabe's formula}
$$
\int_x^{x+1}\ln\Gamma(t){\,}dt ~=~ \frac{1}{2}\,\ln(2\pi)+x\ln x-x{\,},\qquad x>0.
$$

\item \emph{Alternative characterization}. The function $f(x)=\ln\Gamma(x)$ is the unique solution lying in $\cC^0\cap\cK^1$ to the equation
$$
\int_x^{x+1}f(t){\,}dt ~=~ \frac{1}{2}\,\ln(2\pi)+x\ln x-x{\,}, \qquad x>0.
$$
\end{itemize}

\parag{Inequalities} The following inequalities hold for any $x>0$, any $a\geq 0$, and any $n\in\N^*$.

\begin{itemize}
\item \emph{Symmetrized generalized Wendel's inequality} (equality if $a\in\{0,1\}$)
$$
\big|\ln\Gamma(x+a)-\ln\Gamma(x)-a\ln x\big| ~\leq ~ |a-1|\,\ln\left(1+\frac{a}{x}\right),
$$
$$
\left(1+\frac{a}{x}\right)^{-\left|a-1\right|} \leq ~ \frac{\Gamma(x+a)}{\Gamma(x){\,}x^a} ~\leq ~ \left(1+\frac{a}{x}\right)^{\left|a-1\right|}.
$$

\item \emph{Symmetrized generalized Wendel's inequality} (discrete version)
$$
\left|\ln\Gamma(x)-\sum_{k=1}^{n-1}\ln k+\sum_{k=0}^{n-1}\ln(x+k)-x\ln n\right| ~\leq ~ |x-1|\,\ln\left(1+\frac{x}{n}\right).
$$
$$
\left(1+\frac{x}{n}\right)^{-\left|x-1\right|} \leq ~ \Gamma(x)\,\frac{x(x+1)\cdots (x+n-1)}{(n-1)!{\,}n^x} ~ \leq ~ \left(1+\frac{x}{n}\right)^{\left|x-1\right|}.
$$

\item \emph{Symmetrized Stirling's formula-based inequality}
$$
\left|J(x)\right| ~\leq ~ \frac{(x+1)^2}{2}\ln\left(1+\frac{1}{x}\right)-\frac{x}{2}-\frac{3}{4} ~\leq ~ \frac{1}{2}\,\ln\left(1+\frac{1}{x}\right),
$$
$$
\left(1+\frac{1}{x}\right)^{-\frac{1}{2}} \leq ~ \frac{\Gamma(x)}{\sqrt{2\pi}{\,}e^{-x}{\,}x^{x-\frac{1}{2}}} ~\leq ~ \left(1+\frac{1}{x}\right)^{\frac{1}{2}}.
$$

\item \emph{Burnside's formula-based inequality}
$$
\left|\ln\Gamma\left(x+\frac{1}{2}\right)-\frac{1}{2}\ln(2\pi)+x-x\ln x\right| ~\leq ~ \left|J(x)\right|.
$$

\item \emph{Generalized Gautschi's inequality}
$$
(x+\lceil a\rceil)^{a-\lceil a\rceil} ~\leq ~ e^{(a-\lceil a\rceil)\,\psi(x+\lceil a\rceil)} ~\leq ~ \frac{\Gamma(x+a)}{\Gamma(x+\lceil a\rceil)} ~\leq ~ (x+\lfloor a\rfloor)^{a-\lceil a\rceil}{\,}.
$$
\end{itemize}

\parag{Stirling's and related formulas} For any $a\geq 0$, we have the following limits and asymptotic equivalences as $x\to\infty$,
$$
\ln\Gamma(x+a)-\ln\Gamma(x)-a\ln x ~\to ~0,
$$
$$
\ln\Gamma(x)-\frac{1}{2}\ln(2\pi)+x-\left(x-\frac{1}{2}\right)\ln x ~\to ~0,
$$
$$
\ln\Gamma\left(x+\frac{1}{2}\right)-\frac{1}{2}\ln(2\pi)+x-x\ln x ~\to ~0,
$$
$$
\Gamma(x+a) ~\sim ~ x^a\,\Gamma(x),\qquad\ln\Gamma(x+a) ~\sim ~ x\ln x,
$$
$$
\Gamma(x) ~\sim ~ \sqrt{2\pi}{\,}e^{-x}x^{x-\frac{1}{2}},\qquad\Gamma(x+1) ~\sim ~ \sqrt{2\pi x}{\,}e^{-x}x^x{\,}.
$$
\emph{Burnside's approximation} (better than Stirling's approximation)
$$
\Gamma(x) ~\sim ~ \sqrt{2\pi}\left(\frac{x-\frac{1}{2}}{e}\right)^{x-\frac{1}{2}}.
$$
\emph{Further results} (obtained by differentiation)
$$
\psi(x+a)-\psi(x) ~\to ~0,\qquad \psi(x)-\ln x ~\to ~ 0,\qquad \psi(x+a) ~\sim ~ \ln x{\,},
$$
$$
\psi_k(x+a) ~\sim ~ (-1)^{k-1}\,\frac{(k-1)!}{x^k}{\,},\qquad\psi_k(x) ~\to ~0,\qquad k\in\N^*.
$$

\parag{Asymptotic expansions} For any $m,q\in\N^*$ we have the following expansion as $x\to\infty$
\begin{eqnarray*}
\frac{1}{m}\sum_{j=0}^{m-1}\ln\Gamma\left(x+\frac{j}{m}\right) &=& \frac{1}{2}\ln(2\pi)+x\ln x-x-\frac{1}{2m}\ln x\\
&& \null +\sum_{k=1}^q\frac{B_{k+1}}{k(k+1){\,}x^k{\,}m^{k+1}}+O\left(x^{-q-1}\right){\,}.
\end{eqnarray*}
Setting $m=1$ in this formula, we retrieve the known asymptotic expansion of the log-gamma function $\ln\Gamma(x)$ as $x\to\infty$ (see, e.g., \cite[p.~7]{SriCho12})
\begin{equation}\label{eq:AsExpLoG4}
\ln\Gamma(x) ~=~ \frac{1}{2}\ln(2\pi)-x+\left(x-\frac{1}{2}\right)\ln x+\sum_{k=1}^q\frac{B_{k+1}}{k(k+1){\,}x^k}+O\left(x^{-q-1}\right),
\end{equation}
or equivalently,
$$
J(x) ~=~ \sum_{k=1}^q\frac{B_{k+1}}{k(k+1){\,}x^k}+O\left(x^{-q-1}\right).
$$
For instance, setting $q=4$ in \eqref{eq:AsExpLoG4} we get
$$
\ln\Gamma(x) ~=~ \frac{1}{2}\ln(2\pi)-x+\left(x-\frac{1}{2}\right)\ln x+\frac{1}{12x}-\frac{1}{360x^3}+O\left(x^{-5}\right).
$$

\parag{Generalized Liu's formula} For any $x>0$ we have
$$
\ln\Gamma(x) ~=~ \frac{1}{2}\ln(2\pi)-x+\left(x-\frac{1}{2}\right)\ln x+\int_0^{\infty}\frac{\frac{1}{2}-\{t\}}{t+x}{\,}dt,
$$
or equivalently,
$$
J(x) ~=~ \int_0^{\infty}\frac{\frac{1}{2}-\{t\}}{t+x}{\,}dt.
$$

\parag{Limit, series, and integral representations} We now consider various representations of $\ln\Gamma(x)$, including the Eulerian and Weierstrassian forms.
\begin{itemize}
\item \emph{Eulerian form and related identities}. We have
$$
\ln\Gamma(x) ~=~ -\ln x-\sum_{k=1}^{\infty}\left(\ln(x+k)-\ln k-x\ln\left(1+\frac{1}{k}\right)\right),
$$
$$
\Gamma(x) ~=~ \frac{1}{x}\,\prod_{k=1}^{\infty}\frac{(1+\frac{1}{k})^x}{1+\frac{x}{k}}{\,}.
$$
Upon differentiation and integration, we obtain (cf.~Example~\ref{ex:GepeA58})
$$
\psi(x) ~=~ -\frac{1}{x}-\sum_{k=1}^{\infty}\left(\frac{1}{x+k}-\ln\left(1+\frac{1}{k}\right)\right),
$$
$$
\psi_k(x) ~=~ (-1)^{k-1}{\,}k!\,\zeta(k+1,x),\qquad k\in\N^*,
$$
$$
\psi_{-2}(x) ~=~ x-x\ln x-\sum_{k=1}^{\infty}\left((x+k)\ln\left(1+\frac{x}{k}\right)-x-\frac{x^2}{2}\ln\left(1+\frac{1}{k}\right)\right).
$$

\item \emph{Weierstrassian form and related identities}. We have
$$
\ln\Gamma(x) ~=~ -\gamma x-\ln x-\sum_{k=1}^{\infty}\left(\ln(x+k)-\ln k-\frac{x}{k}\right),
$$
$$
\Gamma(x) ~=~ \frac{e^{-\gamma x}}{x}\,\prod_{k=1}^{\infty}\frac{e^{\frac{x}{k}}}{1+\frac{x}{k}}{\,}.
$$
Upon differentiation and integration, we obtain (cf.~Example~\ref{ex:GepeA58W})
$$
\psi(x) ~=~ -\gamma-\frac{1}{x}-\sum_{k=1}^{\infty}\left(\frac{1}{x+k}-\frac{1}{k}\right),
$$
$$
\psi_{-2}(x) ~=~ -\gamma\,\frac{x^2}{2}+x-x\ln x -\sum_{k=1}^{\infty}\left((x+k)\ln\left(1+\frac{x}{k}\right)-x-\frac{x^2}{2k}\right).
$$

\item \emph{Gauss' limit and related identities}. The Gauss limit is
$$
\ln\Gamma(x) ~=~ \lim_{n\to\infty}\left(\ln(n-1)! -\sum_{k=0}^{n-1}\ln(x+k)+x\ln n\right).
$$
Upon differentiation and integration, we obtain
$$
\psi(x) ~=~ \lim_{n\to\infty}\left(\ln n-\sum_{k=0}^{n-1}\frac{1}{x+k}\right),
$$
$$
\psi_k(x) ~=~ (-1)^{k+1}{\,}k!\,\zeta(k+1,x),\qquad k\in\N^*,
$$
\begin{equation}\label{eq:IntPsi2m65}
\psi_{-2}(x) ~=~ \lim_{n\to\infty}\left(nx-x\ln x+(\ln n)\frac{x^2}{2}-\sum_{k=1}^{n-1}(x+k)\ln\left(1+\frac{x}{k}\right)\right).
\end{equation}

The multiplicative version of Gauss' limit reduces to the following formula (just replace $n$ with $n+1$ and note that $(n+1)^x\sim n^x$ as $n\to\infty$)
$$
\Gamma(x) ~=~ \lim_{n\to\infty}\frac{n!{\,}n^x}{x(x+1){\,}\cdots{\,}(x+n)}
$$
as stated in \eqref{eq:GaussLimit42}. We also have the following alternative form of Gauss' limit, which immediately follows from the Weierstrassian form
$$
\Gamma(x) ~=~ \frac{e^{-\gamma x}}{x}\,\lim_{n\to\infty}\prod_{k=1}^n\frac{e^{\frac{x}{k}}}{1+\frac{x}{k}} ~=~ \lim_{n\to\infty}\frac{n!{\,}e^{x\psi(n)}}{x(x+1){\,}\cdots{\,}(x+n)}{\,}.
$$
This latter limit can also be derived immediately from Gauss' limit and the well-known fact that $\psi(x)-\ln x\to 0$ as $x\to\infty$.

\item \emph{Integral representation}. Considering the antiderivative of the digamma function $\varphi =\psi$ as the solution to the equation $\Delta\varphi=g'$ (using the elevator method),\index{elevator method} we obtain
$$
\ln\Gamma(x) ~=~ \psi_{-1}(x) ~=~ \int_1^x\psi(t){\,}dt.
$$

\item \emph{Gregory's formula-based series representation}. For any $x>0$ we have the series representation (see Example~\ref{ex:SeriesGS7})
\begin{eqnarray}
\ln\Gamma(x) &=& \frac{1}{2}\ln(2\pi)-x+x\ln x-\sum_{n=0}^{\infty}G_{n+1}\Delta^n\ln(x)\label{eq:67f7afd8fg6}\\
&=& \frac{1}{2}\ln(2\pi)-x+x\ln x-\sum_{n=0}^{\infty}|G_{n+1}|\,\sum_{k=0}^n(-1)^k\tchoose{n}{k}{\,}\ln(x+k).\nonumber
\end{eqnarray}
Setting $x=1$ in this identity yields the following analogue of Fontana-Mascheroni series
$$
\sum_{n=0}^{\infty}|G_{n+1}|\,\sum_{k=0}^n(-1)^k\tchoose{n}{k}{\,}\ln(k+1) ~=~ -1+\frac{1}{2}\ln(2\pi).
$$
\end{itemize}

\parag{Gauss' multiplication formula} For any $m\in\N^*$ and any $x>0$, we have
$$
\prod_{j=0}^{m-1}\Gamma\left(\frac{x+j}{m}\right) ~=~ (2\pi)^{\frac{m-1}{2}}m^{\frac{1}{2}-x}\,\Gamma(x).
$$
Corollary~\ref{cor:Riem482} provides the following asymptotic equivalence for any $x>0$
$$
\Gamma(mx)^{\frac{1}{m}} ~\sim ~ e^{-x}x^xm^x\qquad\text{as $m\to_{\N}\infty$},
$$
which also follows from Stirling's formula.

\parag{Wallis's product formula} We have the following limits
$$
\lim_{n\to\infty}\frac{1\cdot 3{\,}\cdots{\,}(2n-1)}{2\cdot 4{\,}\cdots{\,}(2n)}{\,}\sqrt{n} ~=~ \frac{1}{\sqrt{\pi}}{\,},
$$
$$
\lim_{n\to\infty}\left(\frac{1}{2}\,\ln(\pi n)+\sum_{k=1}^{2n}(-1)^{k-1}\,\ln k\right) ~=~ 0.
$$

\parag{Restriction to the natural integers} We have the well-known identity
$$
\Gamma(n+1) ~=~ n!{\,},\qquad n\in\N.
$$
Gregory's formula states that for any $n\in\N^*$ and any $q\in\N$ we have
$$
\ln n! ~=~ 1-n+(n+1)\ln n-\sum_{j=1}^qG_j\left((\Delta^{j-1}\ln)(n)-(\Delta^{j-1}\ln)(1)\right)-R^q_n{\,},
$$
with
$$
|R^q_n| ~\leq ~ \overline{G}_q{\,}|(\Delta^q\ln)(n)-(\Delta^q\ln)(1)|.
$$
Moreover, Eq.~\eqref{eq:AsExpLoG4} yields the following asymptotic expansion as $x\to\infty$. For any $q\in\N^*$, we have
$$
\ln n! ~=~ \frac{1}{2}\ln(2\pi n)-n+n\ln n+\sum_{k=1}^q\frac{B_{k+1}}{k(k+1){\,}n^k}+O\left(n^{-q-1}\right).
$$
Similarly, Eq.~\eqref{eq:67f7afd8fg6} yields the following series representation
$$
\ln n! ~=~ \frac{1}{2}\ln(2\pi)-n+(n+1)\ln n-\sum_{k=0}^{\infty}G_{k+1}\Delta^kg(n){\,},\quad n\in\N^*.
$$
We also have Liu's formula
$$
\ln n! ~=~ \frac{1}{2}\ln(2\pi n)-n+n\ln n+\int_n^{\infty}\frac{\frac{1}{2}-\{t\}}{t}{\,}dt{\,}.
$$
Many other representations of $\ln n!$ can be derived from, e.g., the limit and series representations of the log-gamma function described above.

\parag{Generalized Webster's functional equation} For any $m\in\N^*$ and any $a>0$, there is a unique solution $f\colon\R_+\to\R_+$ to the equation
$$
\prod_{j=0}^{m-1}f(x+a j) ~=~ x
$$
such that $\ln f$ lies in $\cK^0$ (or in $\cK^1$), namely
$$
f(x) ~=~ (am)^{\frac{1}{m}}\frac{\Gamma(\frac{x+a}{am})}{\Gamma(\frac{x}{am})}{\,}.
$$

\parag{Analogue of Euler's series representation of $\gamma$} The Taylor series expansion of $\ln\Gamma(x+1)$ about $x=0$ is
$$
\ln\Gamma(x+1) ~=~ -\gamma x+\sum_{k=2}^{\infty}\frac{\zeta(k)}{k}{\,}(-x)^k{\,},\qquad |x|<1.
$$
Integrating both sides of this equation on $(0,1)$, we obtain (see Example~\ref{ex:7sd55fafsd1})
$$
\sum_{k=2}^{\infty}(-1)^k\frac{1}{k(k+1)}\,\zeta(k) ~=~ \frac{1}{2}\,\gamma -1+\frac{1}{2}\ln(2\pi){\,}.
$$

\parag{Reflection formula} For any $x\in\R\setminus\Z$, we have $\Gamma(x)\Gamma(1-x)=\pi\csc(\pi x)$.

\index{gamma function|)}

\section{The digamma and harmonic number functions}
\label{sec:dig27amma1}
\index{digamma function|(}
\index{harmonic number function|(}

Let us now see what we get if we apply our results to both the digamma function $x\mapsto\psi(x)$ and the harmonic number function $x\mapsto H_x$. Recall first that the identity
$$
H_{x-1} ~=~ \psi(x)+\gamma
$$
holds for any $x>0$.

\parag{ID card} We have the following data about the functions $1/x$ and $\psi(x)$:
$$
\begin{array}{|c|c|c|c|}
\hline g(x) & \text{Membership} & \deg g & \Sigma g(x) \\
\hline \emptybox 1/x & \cC^{\infty}\cap\cD^0\cap\cK^{\infty} & -1 & H_{x-1}=\psi(x)+\gamma\\
\hline
\end{array}
$$

\parag{Analogue of Bohr-Mollerup's theorem} The digamma function can be characterized as follows.
\begin{quote}
\emph{All eventually monotone solutions $f\colon\R_+\to\R$ to the equation
$$
f(x+1)-f(x) ~=~ \frac{1}{x}
$$
are of the form $f(x)=c+\psi(x)$, where $c\in\R$.}
\end{quote}

It is noteworthy that this characterization immediately follows from the basic version when $p=0$ of our Theorem~\ref{thm:int2}, which was established by John~\cite{Joh39}.

Interestingly, this characterization enables us to establish almost instantly the following identities for every $x>0$,
$$
H_{x-1} ~=~ \psi(x)+\gamma ~=~ \int_0^1\frac{1-t^{x-1}}{1-t}{\,}dt{\,}.
$$
Indeed, each of the three expressions above vanishes at $x=1$ and is an eventually increasing solution to the equation $f(x+1)-f(x)=1/x$. Hence, they must coincide on $\R_+$. We can actually prove many other representations similarly; for instance, the following Gauss and Dirichlet integral representations\index{digamma function!Gauss representation}\index{digamma function!Dirichlet representation} (see, e.g., \cite[p.~26]{SriCho12})
$$
\psi(x) ~=~ \int_0^{\infty}\left(\frac{e^{-t}}{t}-\frac{e^{-xt}}{1-e^{-t}}\right)dt{\,},\qquad x>0,
$$
$$
\psi(x) ~=~ \int_0^{\infty}\left(e^{-t}-\frac{1}{(t+1)^x}\right)\frac{dt}{t}{\,},\qquad x>0.
$$

Kairies~\cite{Kai70} obtained a variant of the characterization of the digamma function above by replacing the eventual monotonicity with the convexity property. This variant is also immediate from our results since $g$ also lies in $\cD^1\cap\cK^1$.

Using Proposition~\ref{prop:90unic41}, we can also derive the following alternative characterization of the digamma function.
\begin{quote}
\emph{All solutions $f\colon\R_+\to\R$ to the equation
$$
f(x+1)-f(x) ~=~ \frac{1}{x}
$$
that satisfy the asymptotic condition that, for each $x>0$,
$$
f(x+n)-f(n) ~\to ~0\qquad\text{as $n\to_{\N}\infty$}
$$
are of the form $f(x)=c+\psi(x)$, where $c\in\R$.}
\end{quote}

\parag{Extended ID card} We already know that $\sigma[g]=\gamma$ (see Example~\ref{ex:Dig5Int9}). Hence we have the following table:
$$
\begin{array}{|c|c|c|}
\hline \overline{\sigma}[g] & \sigma[g] & \gamma[g] \\
\hline \emptybox \infty & \gamma & \gamma \\
\hline
\end{array}
$$

\begin{itemize}
\item \emph{Alternative representations of $\sigma[g]=\gamma[g]=\gamma$}
\begin{eqnarray*}
\gamma &=& \lim_{n\to\infty}\left(\sum_{k=1}^n\frac{1}{k}-\ln n\right) ~=~ \sum_{k=1}^{\infty}\left(\frac{1}{k}-\ln\left(1+\frac{1}{k}\right)\right),\\
\gamma &=& \int_1^{\infty}\left(\frac{1}{\lfloor t\rfloor}-\frac{1}{t}\right){\,}dt ~=~ \frac{1}{2}-\int_1^{\infty}\frac{\{t\}-\frac{1}{2}}{t^2}{\,}dt{\,},\\
\gamma &=& \int_0^1H_t{\,}dt{\,}.
\end{eqnarray*}

\item \emph{Generalized Binet's function}. For any $q\in\N$ and any $x>0$
$$
J^{q+1}[\psi](x) ~=~ \psi(x)-\ln x+\sum_{j=1}^q|G_j|{\,}\mathrm{B}(x,j),
$$
where $(x,y)\mapsto\mathrm{B}(x,y)$ is the beta function.

\item \emph{Analogue of Raabe's formula} (see Example~\ref{ex:Dig5Int9})
$$
\int_x^{x+1}\psi(t){\,}dt ~=~ \ln x{\,}, \qquad x>0.
$$

\item \emph{Alternative characterization}. The function $f=\psi$ is the unique solution lying in $\cC^0\cap\cK^0$ to the equation
$$
\int_x^{x+1}f(t){\,}dt ~=~ \ln x{\,}, \qquad x>0.
$$
\end{itemize}

\parag{Inequalities} The following inequalities hold for any $x>0$, any $a\geq 0$, and any $n\in\N^*$.

\begin{itemize}
\item \emph{Symmetrized generalized Wendel's inequality} (equality if $a\in\{0,1\}$)
$$
|\psi(x+a)-\psi(x)| ~\leq ~ \lceil a\rceil\,\frac{1}{x}{\,}.
$$

\item \emph{Symmetrized generalized Wendel's inequality} (discrete version)
$$
\left|\psi(x)+\gamma-\sum_{k=1}^{n-1}\frac{1}{k}+\sum_{k=0}^{n-1}\frac{1}{x+k}\right| ~\leq ~ \lceil x\rceil\,\frac{1}{n}{\,}.
$$

\item \emph{Symmetrized Stirling's and Burnside's formulas-based inequalities}
$$
\left|\psi\left(x+\frac{1}{2}\right)-\ln x\right| ~\leq ~ |\psi(x)-\ln x| ~\leq ~ \frac{1}{x}{\,}.
$$
Considering for instance the value $p=1$ in Corollary~\ref{cor:6GenStFo0BaIneqCo}, we see that the latter inequality can be refined into
$$
\frac{1}{2(x+1)}-\frac{1}{x} ~\leq ~ \psi(x)-\ln x ~\leq ~ -\frac{1}{2(x+1)}{\,}.
$$

\item \emph{Generalized Gautschi's inequality}
$$
\frac{a-\lceil a\rceil}{x+\lfloor a\rfloor} ~\leq ~ \psi(x+a)-\psi(x+\lceil a\rceil) ~\leq ~ (a-\lceil a\rceil)\,\psi_1(x+\lceil a\rceil) ~\leq ~ \frac{a-\lceil a\rceil}{x+\lceil a\rceil} {\,}.
$$
\end{itemize}

\parag{Generalized Stirling's and related formulas} For any $a\geq 0$, we have the following limits and asymptotic equivalence as $x\to\infty$,
$$
\psi(x+a)-\psi(x) ~\to ~0,\qquad \psi(x)-\ln x ~\to ~ 0,\qquad \psi(x+a) ~\sim ~ \ln x.
$$
\emph{Burnside-like approximation} (better than Stirling-like approximation)
$$
\psi(x)-\ln\left(x-\frac{1}{2}\right) ~\to ~ 0.
$$
\emph{Further results} (obtained by differentiation)
$$
\psi_k(x+a) ~\sim ~ (-1)^{k-1}\,\frac{(k-1)!}{x^k}{\,},\qquad\psi_k(x) ~\to ~0,\qquad k\in\N^*.
$$

\parag{Asymptotic expansions} For any $m,q\in\N^*$ we have the following expansion as $x\to\infty$
\begin{equation}\label{eq:ExpAs0psi}
\frac{1}{m}\sum_{j=0}^{m-1}\psi\left(x+\frac{j}{m}\right) ~=~ \ln x + \sum_{k=1}^q\frac{(-1)^{k-1}{\,}B_k}{k{\,}(mx)^k}+O\left(x^{-q-1}\right){\,}.
\end{equation}
Setting $m=1$ in this formula, we retrieve the known asymptotic expansion of $\psi(x)$ as $x\to\infty$ (see, e.g., \cite[p.~36]{SriCho12})
$$
\psi(x) ~=~ \ln x + \sum_{k=1}^q\frac{(-1)^{k-1}{\,}B_k}{k{\,}x^k}+O\left(x^{-q-1}\right),
$$
or equivalently,
$$
J^1[\psi](x) ~=~ \sum_{k=1}^q\frac{(-1)^{k-1}{\,}B_k}{k{\,}x^k}+O\left(x^{-q-1}\right).
$$
For instance, setting $q=5$ we get
$$
\psi(x) ~=~ \ln x -\frac{1}{2x}-\frac{1}{12x^2}+\frac{1}{120x^4}+O\left(x^{-6}\right).
$$

\parag{Generalized Liu's formula} For any $x>0$ we have
$$
\psi(x) ~=~ \ln x-\frac{1}{2x}+\int_0^{\infty}\frac{\{t\}-\frac{1}{2}}{(t+x)^2}{\,}dt.
$$

\parag{Limit and series representations} Let us now examine the main limit and series representations of the digamma function that we obtain from our results.
\begin{itemize}
\item \emph{Eulerian and Weierstrassian forms}. We have
$$
\psi(x) ~=~ -\gamma-\frac{1}{x}+\sum_{k=1}^{\infty}\left(\frac{1}{k}-\frac{1}{x+k}\right),
$$
$$
\psi(x) ~=~ -\frac{1}{x}+\sum_{k=1}^{\infty}\left(\ln\left(1+\frac{1}{k}\right)-\frac{1}{x+k}\right).
$$
Upon differentiation, we obtain
$$
\psi_k(x) ~=~ (-1)^{k-1}{\,}k!\,\zeta(k+1,x),\qquad k\in\N^*.
$$
Moreover, integrating the Eulerian (resp.\ Weierstrassian) form of the digamma function on $(0,x)$, we retrieve the Weierstrassian (resp.\ Eulerian) form of the log-gamma function.

\item The analogue of Gauss' limit coincides with the Eulerian form.

\item \emph{Gregory's formula-based series representation}. For any $x>0$ we have the series representation
$$
\psi(x) ~=~ \ln x-\sum_{n=1}^{\infty}|G_n|{\,}\mathrm{B}(x,n) ~=~ \ln x-\sum_{n=1}^{\infty}\frac{|G_n|}{n{x+n-1\choose n}}{\,}.
$$
Setting $x=1$ in this identity, we retrieve the Fontana-Mascheroni series\index{Fontana-Mascheroni's series} (see, e.g., Blagouchine \cite[p.~379]{Bla16})
$$
\gamma ~=~ \sum_{n=1}^{\infty}\frac{|G_n|}{n}{\,}.
$$
Setting $x=2$, we get
$$
1-\ln 2 ~=~ \sum_{n=1}^{\infty}\frac{|G_n|}{n+1}{\,},
$$
which is consistent with the identities given in Example~\ref{ex:ds5a82}.
\end{itemize}

\parag{Analogue of Gauss' multiplication formula} For any $m\in\N^*$ and any $x>0$, we have (see, e.g., Berndt \cite[p.~5]{Ber83})
\begin{equation}\label{eq:GMF3PSd}
\sum_{j=0}^{m-1}\psi\left(x+\frac{j}{m}\right) ~=~ m(\psi(mx)-\ln m)
\end{equation}
and
$$
\sum_{j=0}^{m-1}H_{x+j/m} ~=~ m(H_{mx+m-1}-\ln m){\,}.
$$
Corollary~\ref{cor:Riem482} provides the following formula for any $x>0$
$$
\lim_{m\to\infty}(H_{mx-1}-H_{m-1}) ~=~ \ln x.
$$

\parag{Analogue of Wallis's product formula} The analogue of Wallis's formula reduces to the classical identity
$$
\sum_{k=1}^{\infty}(-1)^{k-1}\frac{1}{k} ~=~ \ln 2{\,}.
$$

\begin{project}\label{app:WallPsi5}
\textsl{Find the analogue of Wallis's formula for the function $g(x)=\psi(x)$.} We apply our method (see Section~\ref{sec:F3I1aR4}) to the function
$$
\tilde{g}(x) ~=~ \Delta g(2x) ~=~ \frac{1}{2x}{\,}.
$$
Thus, we get
$$
h(x) ~=~ \psi(2n)-\psi(1)-\frac{1}{2}\,\gamma-\frac{1}{2}\ln n ~=~ \frac{1}{2}(\gamma +\ln(4n))+O\left(n^{-1}\right),
$$
and the analogue of Wallis's formula for $g(x)=\psi(x)$ is
$$
\lim_{n\to\infty}\left(-\ln(4n)+2\sum_{k=1}^{2n}(-1)^k\psi(k)\right) ~=~ \gamma{\,}.
$$
This provides yet another formula to define Euler's constant $\gamma$.\index{Euler's constant}
\end{project}

\parag{Restriction to the natural integers} For any $n\in\N$ we have
$$
H_n ~=~ \sum_{k=1}^n\frac{1}{k}{\,}.
$$
Gregory's formula states that for any $n\in\N^*$ and any $q\in\N$ we have
$$
H_{n-1} ~=~ \ln n-\sum_{j=1}^q|G_j|\left(\mathrm{B}(n,j)-\frac{1}{j}\right)-R^q_n{\,},
$$
with
$$
|R^q_n| ~\leq ~\overline{G}_q \left|\mathrm{B}(n,q+1)-\frac{1}{q}\right|.
$$
Many representations of $H_n$ can be derived from, e.g., the limit and series representations of the digamma function described above. For instance, using the generalized Liu formula, we get (see also Remark~\ref{rem:EM5LiU3})
$$
H_n ~=~ \ln n+\gamma +\frac{1}{2n}+\int_n^{\infty}\frac{\{t\}-\frac{1}{2}}{t^2}{\,}dt ~=~ \ln n+\frac{1}{2}+\frac{1}{2n}-\int_1^n\frac{\{t\}-\frac{1}{2}}{t^2}{\,}dt{\,}.
$$

\parag{Generalized Webster's functional equation} For any $m\in\N^*$ and any $a>0$, there is a unique eventually monotone solution $f\colon\R_+\to\R$ to the equation
$$
\sum_{j=0}^{m-1}f(x+a j) ~=~ \frac{1}{x}{\,},
$$
namely
$$
f(x) ~=~ \frac{1}{am}\,\psi\left(\frac{x+a}{am}\right)-\frac{1}{am}\,\psi\left(\frac{x}{am}\right){\,}.
$$

\parag{Analogue of Euler's series representation of $\gamma$} We have $\psi(1)=-\gamma$ and
$$
\psi_k(1) ~=~ (-1)^{k-1} k!\,\zeta(k+1){\,},\qquad k\in\N^*.
$$
Thus, the Taylor series expansion of $\psi(x+1)$ about $x=0$ is
$$
H_x ~=~ \psi(x+1) + \gamma ~=~ \sum_{k=1}^{\infty}(-1)^{k-1}\zeta(k+1){\,}x^k{\,},\qquad |x|< 1.
$$
Integrating both sides of this equation on $(0,1)$, we retrieve Euler's series representation of $\gamma$
$$
\gamma ~=~ \sum_{k=2}^{\infty}(-1)^k\,\frac{\zeta(k)}{k}{\,}.
$$

\parag{Analogue of the reflection formula} For any $x\in\R\setminus\Z$, we have
$$
\psi(x)-\psi(1-x) ~=~ -\pi\cot(\pi x).
$$

\index{harmonic number function|)}
\index{digamma function|)}

\section{The polygamma functions}
\label{sec:Polyg&82}\index{polygamma functions|(}

We now investigate the polygamma functions $\psi_{\nu}$ for any $\nu\in\Z$. In this context, our results will prove to be particularly interesting when $\nu\leq -2$, that is, when the function $\psi_{\nu}$ has a strictly positive asymptotic degree.

For any $\nu\in\Z$, we set $g_{\nu}=\Delta\psi_{\nu}$; hence we have $g'_{\nu}=g_{\nu +1}$ and $\psi'_{\nu}=\psi_{\nu +1}$. It follows immediately that
$$
\Sigma g_{\nu}(x) ~=~ \psi_{\nu}(x)-\psi_{\nu}(1).
$$
(The cases $\nu=0$ and $\nu=-1$ correspond to the functions $\psi(x)$ and $\ln\Gamma(x)$, respectively, and have been already considered in the previous sections.) We will often deal with the cases $\nu\geq 1$ and $\nu\leq -1$ separately. In the latter case, we will often consider the value $\nu =-2$ for simplicity and brevity.

\parag{ID card when $\nu\geq 1$} Here we clearly have
$$
g_{\nu}(x) ~=~ D^{\nu}_x\frac{1}{x} ~=~ (-1)^{\nu}\,\frac{\nu!}{x^{\nu+1}}
$$
and (see Example~\ref{ex:6afdsdf2})
$$
\psi_{\nu}(1) ~=~ (-1)^{\nu+1}\nu!{\,}\zeta(\nu+1).
$$
Hence we have the following table.
$$
\begin{array}{|c|c|c|c|}
\hline g_{\nu}(x) & \text{Membership} & \deg g_{\nu} & \Sigma g_{\nu}(x) \\
\hline \emptybox (-1)^{\nu}\nu!{\,}x^{-\nu-1} & \cC^{\infty}\cap\cD^{-1}\cap\cK^{\infty} & -1 & \psi_{\nu}(x)-\psi_{\nu}(1) \\
\hline
\end{array}
$$

\parag{ID card when $\nu\leq -1$} Using \eqref{eq:ds68ffdsbis}, we obtain the following recurrence to compute the functions $g_{\nu}$. For any integer $\nu\leq -1$, we have
\begin{eqnarray*}
g_{\nu-1}(x) &=& \int_x^{x+1}\psi_{\nu}(t){\,}dt ~=~ \int_0^x g_{\nu}(t){\,}dt + \int_0^1\psi_{\nu}(t){\,}dt\\
&=& \int_0^x g_{\nu}(t){\,}dt + \psi_{\nu -1}(1).
\end{eqnarray*}
In particular,
$$
\lim_{x\to 0}g_{\nu -1}(x) ~=~ \psi_{\nu -1}(1) ~=~ \int_0^1\psi_{\nu}(t){\,}dt.
$$
Unfolding this recurrence, we obtain $g_{-1}(x)=\ln x$ and, for any integer $\nu\leq -1$,
\begin{equation}\label{eq:gvZN4}
g_{\nu -1}(x) ~=~ \int_0^x\frac{(x-t)^{-\nu -1}}{(-\nu -1)!}{\,}\ln t{\,}dt+\sum_{j=0}^{-\nu -1}
\psi_{\nu +j-1}(1)\,\frac{x^j}{j!}{\,},
\end{equation}
which is precisely the $(-\nu -1)$th order Taylor expansion of $g_{\nu -1}(x)$.

Thus, we have
\begin{eqnarray*}
g_{-1}(x) &=& \ln x{\,},\\
g_{-2}(x) &=& x\ln x-x+\frac{1}{2}\ln(2\pi){\,},\\
g_{-3}(x) &=& \frac{1}{2}{\,}x^2\ln x-\frac{3}{4}{\,}x^2+\left(\frac{1}{2}{\,}x+\frac{1}{4}\right)\ln(2\pi)+\ln A.
\end{eqnarray*}
Hence the following ID card
$$
\begin{array}{|c|c|c|c|}
\hline g_{\nu}(x) & \text{Membership} & \deg g_{\nu} & \Sigma g_{\nu}(x)\\
\hline \emptybox \text{Eq.~\eqref{eq:gvZN4}} & \cC^{\infty}\cap\cD^{-\nu}\cap\cK^{\infty} & -\nu -1 & \psi_{\nu}(x)-\psi_{\nu}(1)\\
\hline
\end{array}
$$

\parag{Analogue of Bohr-Mollerup's theorem} The function $\psi_{\nu}$ can be characterized as follows.
\begin{quote}
\emph{All solutions $f\colon\R_+\to\R$ to the equation $f(x+1)-f(x)=g_{\nu}(x)$ that lie in $\cK^{(-\nu)_+}$ are of the form $f(x)=c_{\nu}+\psi_{\nu}(x)$, where $c_{\nu}\in\R$}.
\end{quote}
When $\nu\geq 1$, this characterization enables us to prove easily the following integral representation of $\psi_{\nu}$
$$
\psi_{\nu}(x) ~=~ (-1)^{\nu -1}\int_0^{\infty}\frac{t^{\nu}{\,}e^{-xt}}{1-e^{-t}}{\,}dt{\,},\qquad x>0.
$$
Indeed, both sides of this identity coincide at $x=1$ and are eventually monotone solutions to the equation $\Delta f=g_{\nu}$. Hence they must coincide on $\R_+$.

\parag{Extended ID card} The asymptotic constant $\sigma[g_{\nu}]$ satisfies the following identity
$$
\sigma[g_{\nu}] ~=~ \int_0^1\psi_{\nu}(t+1){\,}dt -\psi_{\nu}(1) ~=~ g_{\nu-1}(1)-\psi_{\nu}(1).
$$
Moreover, if $\nu\geq 1$ we also have
\begin{eqnarray*}
\sigma[g_{\nu}] ~=~ \gamma[g_{\nu}] &=& \sum_{k=1}^{\infty}g_{\nu}(k)-\int_1^{\infty}g_{\nu}(t){\,}dt\\
&=& (-1)^{\nu}\,\Gamma(\nu){\,}(\nu\,\zeta(\nu+1)-1)
\end{eqnarray*}
and hence the following values
$$
\begin{array}{|c|c|c|}
\hline \overline{\sigma}[g_{\nu}] & \sigma[g_{\nu}] & \gamma[g_{\nu}] \\
\hline \emptybox \infty & (-1)^{\nu}\,\Gamma(\nu){\,}(\nu\,\zeta(\nu+1)-1) & \gamma[g_{\nu}]=\sigma[g_{\nu}] \\
\hline
\end{array}
$$
For $\nu\leq -1$ we have the values
$$
\begin{array}{|c|c|c|}
\hline \overline{\sigma}[g_{\nu}] & \sigma[g_{\nu}] & \gamma[g_{\nu}] \\
\hline \emptybox \psi_{\nu -1}(1)-\psi_{\nu}(1) & g_{\nu -1}(1)-\psi_{\nu}(1) & \sigma[g_{\nu}]-\sum_{j=1}^{-\nu}G_j\Delta^{j-1}g_{\nu}(1) \\
\hline
\end{array}
$$
For instance we have
$$
\overline{\sigma}[g_{-2}] ~=~ \ln A-\frac{1}{4}\ln(2\pi),\qquad \sigma[g_{-2}] ~=~ \ln A+\frac{1}{4}\ln(2\pi)-\frac{3}{4}{\,},
$$
and
$$
\gamma[g_{-2}] ~=~ \ln A+\frac{1}{6}\ln 2-\frac{1}{3}{\,}.
$$
We also have the following identities.
\begin{itemize}
\item \emph{Alternative representations of $\sigma[g]$}
\begin{eqnarray*}
\sigma[g_{\nu}] &=& \sum_{j=1}^{(-\nu)_+}G_j\,\Delta^{j-1}g_{\nu}(1) -\sum_{k=1}^{\infty}\left(\Delta g_{\nu -1}(k)-\sum_{j=0}^{(-\nu)_+}G_j\,\Delta^jg_{\nu}(k)\right),\\
\sigma[g_{\nu}] &=& \lim_{n\to\infty}\left(\sum_{k=1}^{n-1}g_{\nu}(k)+g_{\nu -1}(1)-g_{\nu -1}(n)+\sum_{j=1}^{(-\nu)_+}G_j\,\Delta^{j-1}g_{\nu}(n)\right), \\
\sigma[g_{\nu}] &=& \lim_{n\to\infty}\left(\sum_{k=1}^{n-1}g_{\nu}(k)+g_{\nu -1}(1)-g_{\nu -1}(n)-\sum_{j=1}^{(-\nu)_+}\frac{B_j}{j!}{\,}g_{\nu +j-1}(n)\right).
\end{eqnarray*}
If $\nu\geq 1$, then
$$
\sigma[g_{\nu}] ~=~ (-1)^{\nu}{\,}\nu !\left(\frac{1}{2}-(\nu +1)\int_1^{\infty}\frac{\{t\}-\frac{1}{2}}{t^{\nu +2}}{\,}dt\right).
$$
If $\nu\leq -1$, then for any integer $q\geq\lceil -\nu/2\rceil$,
$$
\sigma[g_{\nu}] ~=~ \frac{1}{2}{\,}g_{\nu}(1)-\sum_{k=1}^q\frac{B_{2k}}{(2k)!}{\,}g_{\nu +2k-1}(1)  - \int_1^{\infty}\frac{B_{2q}(\{t\})}{(2q)!}{\,}g_{\nu + 2q}(t){\,}dt.
$$
\item \emph{Representations of $\gamma[g]$}
\begin{eqnarray*}
\gamma[g_{\nu}] &=& \sigma[g_{\nu}]-\sum_{j=1}^{(-\nu)_+}G_j\Delta^{j-1}g_{\nu}(1){\,},\\
\gamma[g_{\nu}] &=& \int_1^{\infty}\left(\sum_{j=0}^{(-\nu)_+}G_j\,\Delta^jg_{\nu}(\lfloor t\rfloor)-g_{\nu}(t)\right)dt{\,},\\
\gamma[g_{\nu}] &=& \int_1^{\infty}\left(\sum_{j=0}^{(-\nu)_+}{\{t\}\choose j}\,\Delta^jg_{\nu}(\lfloor t\rfloor)-g_{\nu}(t)\right)dt{\,}.
\end{eqnarray*}
\item \emph{Generalized Binet's function}. For any $q\in\N$ and any $x>0$
$$
J^{q+1}[\psi_{\nu}](x) ~=~ \psi_{\nu}(x)-g_{\nu -1}(x)+\sum_{j=1}^qG_j\,\Delta^{j-1}g_{\nu}(x).
$$
For instance,
\begin{eqnarray*}
J^3[\psi_{-2}](x) &=& \psi_{-2}(x)-\frac{1}{12}{\,}(x+1)\ln(x+1)+\frac{1}{12}{\,}(3x-1)^2\\
&& \null -\frac{1}{12}{\,}x(6x-7)\ln x-\frac{1}{2}{\,}x\ln(2\pi)-\ln A.
\end{eqnarray*}
\item \emph{Analogue of Raabe's formula}
$$
\int_x^{x+1}\psi_{\nu}(t){\,}dt ~=~ g_{\nu -1}(x){\,},\qquad x>0.
$$
\item \emph{Alternative characterization}. The function $f=\psi_{\nu}$ is the unique solution lying in $\cC^0\cap\cK^{(-\nu)_+}$ to the equation
$$
\int_x^{x+1}f(t){\,}dt ~=~ g_{\nu -1}(x){\,}, \qquad x>0.
$$
\end{itemize}

\parag{Inequalities when $\nu\geq 1$} The following inequalities hold for any $x>0$, any $a\geq 0$, and any $n\in\N^*$.
\begin{itemize}
\item \emph{Symmetrized generalized Wendel's inequality} (equality if $a\in\{0,1\}$)
$$
\left|\psi_{\nu}(x+a)-\psi_{\nu}(x)\right| ~\leq ~ \lceil a\rceil{\,}\frac{\nu !}{x^{\nu +1}}{\,}.
$$
\item \emph{Symmetrized generalized Wendel's inequality} (discrete version)
$$
\left|\psi_{\nu}(x)-\psi_{\nu}(1)-\sum_{k=1}^{n-1}g_{\nu}(k)+\sum_{k=0}^{n-1}g_{\nu}(x+k)\right| ~\leq ~ \lceil x\rceil{\,}\frac{\nu !}{n^{\nu +1}}{\,}.
$$
\item \emph{Symmetrized Stirling's and Burnside's formulas-based inequalities}
$$
\textstyle{\left|\psi_{\nu}\left(x+\frac{1}{2}\right)-g_{\nu -1}(x)\right| ~\leq ~ \left|\psi_{\nu}(x)-g_{\nu -1}(x)\right| ~\leq ~ |g_{\nu}(x)|}{\,}.
$$
Considering for instance the value $p=1$ in Corollary~\ref{cor:6GenStFo0BaIneqCo}, we see that the latter inequality can be refined into
$$
\left|\psi_{\nu}(x)-g_{\nu -1}(x)+\frac{1}{2}{\,}g_{\nu}(x)\right| ~\leq ~ \frac{1}{2}{\,}|\Delta g_{\nu}(x)|.
$$
\item \emph{Additional inequality}
$$
|\psi_{\nu}(x+n)| ~=~ \left|\sum_{k=n}^{\infty}g_{\nu}(x+k)\right| ~\leq ~ \left|\psi_{\nu}(n)\right|.
$$

\item \emph{Generalized Gautschi's inequality}
\begin{eqnarray*}
(-1)^{\nu -1}(a-\lceil a\rceil)\,\psi_{\nu +1}(x+\lceil a\rceil) &\leq & (-1)^{\nu -1}(\psi_{\nu}(x+a)-\psi_{\nu}(x+\lceil a\rceil))\\
&\leq & (-1)^{\nu -1}(a-\lceil a\rceil){\,}g_{\nu}(x+\lfloor a\rfloor){\,}.
\end{eqnarray*}
\end{itemize}

\parag{Inequalities when $\nu\leq -1$} The following inequalities hold for any $x>0$, any $a\geq 0$, and any $n\in\N^*$.
\begin{itemize}
\item \emph{Symmetrized generalized Wendel's inequality} (equality if $a\in\{0,1,\ldots,-\nu\}$)
\begin{eqnarray*}
\lefteqn{\left|\psi_{\nu}(x+a)-\psi_{\nu}(x)-\sum_{j=1}^{-\nu}\tchoose{a}{j}\,\Delta^{j-1}g_{\nu}(x)\right|}\\
& \leq & \left|\tchoose{a-1}{-\nu}\right|\left|\Delta^{-\nu -1}g_{\nu}(x+a)-\Delta^{-\nu -1}g_{\nu}(x)\right|\\
& \leq & \lceil a\rceil\left|\tchoose{a-1}{-\nu}\right|\left|\Delta^{-\nu}g_{\nu}(x)\right|.
\end{eqnarray*}
\item \emph{Symmetrized generalized Wendel's inequality} (discrete version)
\begin{eqnarray*}
\left|\psi_{\nu}(x)-\psi_{\nu}(1)-f_n^{-\nu}[g_{\nu}](x)\right|
& \leq & \left|\tchoose{x-1}{-\nu}\right|\left|\Delta^{-\nu -1}g_{\nu}(x+n)-\Delta^{-\nu -1}g_{\nu}(n)\right|\\
& \leq & \lceil x\rceil\left|\tchoose{x-1}{-\nu}\right|\left|\Delta^{-\nu}g_{\nu}(n)\right|,
\end{eqnarray*}
where
$$
f_n^{-\nu}[g_{\nu}](x) ~=~ \sum_{k=1}^{n-1}g_{\nu}(k)-\sum_{k=0}^{n-1} g_{\nu}(x+k)+\sum_{j=1}^{-\nu}\tchoose{x}{j}\,\Delta^{j-1}g_{\nu}(n){\,}.
$$
\item \emph{Symmetrized Stirling's formula-based inequality}
\begin{eqnarray*}
\lefteqn{\left|\psi_{\nu}(x)-g_{\nu -1}(x)+\sum_{j=1}^{-\nu}G_j\Delta^{j-1}g_{\nu}(x)\right|}\\
& \leq & \int_0^1\left|\tchoose{t-1}{-\nu}\right|\left|\Delta^{-\nu -1}g_{\nu}(x+t)-\Delta^{-\nu -1}g_{\nu}(x)\right|dt\\
& \leq & \overline{G}_{-\nu}\left|\Delta^{-\nu}g_{\nu}(x)\right|
\end{eqnarray*}

\item \emph{Generalized Gautschi's inequality}

Considering the function $\psi_{-2}$, we obtain
\begin{eqnarray*}
(a-\lceil a\rceil)\,\psi_{-1}(x+\lceil a\rceil) &\leq & \psi_{-2}(x+a)-\psi_{-2}(x+\lceil a\rceil)\\
&\leq & (a-\lceil a\rceil){\,}g_{-2}(x+\lfloor a\rfloor),
\end{eqnarray*}
for any $x+\lfloor a\rfloor\geq x_0$, where $x_0=1.461\ldots$ is the unique positive zero of the digamma function.\index{digamma function}
\end{itemize}

\parag{Generalized Stirling's and related formulas when $\nu\geq 1$} For any $a\geq 0$, we have the following limit and
asymptotic equivalence as $x\to\infty$,
$$
\psi_{\nu}(x+a) ~\sim ~ g_{\nu -1}(x) ~=~ (-1)^{\nu -1}\,\frac{(\nu -1)!}{x^{\nu}}{\,},\qquad\psi_{\nu}(x)~\to ~ 0.
$$
\emph{Burnside-like approximation} (better than Stirling-like approximation)
$$
\textstyle{\psi_{\nu}(x)-g_{\nu -1}(x-\frac{1}{2})} ~\to ~ 0{\,}.
$$

\parag{Generalized Stirling's and related formulas when $\nu\leq -1$} For any $a\geq 0$, we have the following limits and
asymptotic equivalence as $x\to\infty$,
$$
\psi_{\nu}(x+a)-\psi_{\nu}(x)-\sum_{j=1}^{-\nu}\tchoose{a}{j}\,\Delta^{j-1}g_{\nu}(x) ~ \to ~ 0,
$$
$$
\psi_{\nu}(x)-g_{\nu -1}(x)+\sum_{j=1}^{-\nu}G_j\Delta^{j-1}g_{\nu}(x) ~\to ~ 0,
$$
$$
\psi_{\nu}(x)-\sum_{k=0}^{-\nu}\frac{B_k}{k!}{\,}g_{\nu +k-1}(x) ~\to ~ 0,
$$
$$
\psi_{\nu}(x+a) ~\sim ~ g_{\nu -1}(x) ~\sim ~ \frac{1}{(-\nu)!}{\,}x^{-\nu}\ln x.
$$
When $\nu=-2$ for instance, these limits reduce to
$$
\int_x^{x+a}\ln\Gamma(t){\,}dt -a \ln\left(\sqrt{2\pi}\,\frac{x^x}{e^x}\right)-\tchoose{a}{2}\ln\left(\frac{(x+1)^{x+1}}{e{\,}x^x}\right) ~\to ~0 {\,},
$$
\begin{eqnarray*}
\lefteqn{\psi_{-2}(x)-\frac{1}{12}{\,}(x+1)\ln(x+1)+\frac{1}{12}{\,}(3x-1)^2}\\
&& \null -\frac{1}{12}{\,}x(6x-7)\ln x-\frac{1}{2}{\,}x\ln(2\pi) ~\to ~ \ln A{\,},
\end{eqnarray*}
$$
\psi_{-2}(x) - \frac{1}{12}(6x^2-6x+1)\ln x +\frac{1}{4}(3x-2)x-\frac{1}{2}{\,}x\ln(2\pi) ~\to ~ \ln A{\,},
$$
$$
\psi_{-2}(x+a) ~\sim ~ \frac{1}{2}{\,}x^2\ln x{\,}.
$$

\parag{Asymptotic expansions} For any $m,q\in\N^*$ we have the following expansion as $x\to\infty$
$$
\frac{1}{m}\sum_{j=0}^{m-1}\psi_{\nu}\left(x+\frac{j}{m}\right) ~=~ \sum_{k=0}^q\frac{B_k}{m^k{\,}k!}{\,}g_{\nu +k-1}(x)+O(g_{\nu +q}(x)){\,}.
$$
Setting $m=1$ in this formula, we obtain
$$
\psi_{\nu}(x) ~=~ \sum_{k=0}^q\frac{B_k}{k!}{\,}g_{\nu +k-1}(x)+O(g_{\nu +q}(x)){\,}.
$$
For instance the asymptotic expansion of $\psi_{-2}$ is
\begin{eqnarray*}
\psi_{-2}(x) &=& \frac{1}{12}(6x^2-6x+1)\ln x -\frac{1}{4}(3x-2)x+\frac{1}{2}{\,}x\ln(2\pi)+\ln A\\
&& \null +\frac{1}{720x^2}+O\left(x^{-4}\right).
\end{eqnarray*}

\parag{Generalized Liu's formula} For any $\nu\geq 1$ and any $x>0$ we have
$$
\psi_{\nu}(x) ~=~ (-1)^{\nu -1}\,\Gamma(\nu)\left(\frac{2x+\nu}{2x^{\nu +1}}+\nu (\nu +1)\int_0^{\infty}\frac{\frac{1}{2}-\{t\}}{(t+x)^{\nu +2}}{\,}dt\right).
$$
For $\nu =-2$ and any $x>0$ we have
\begin{eqnarray*}
\psi_{-2}(x) &=& \frac{1}{12}(6x^2-6x+1)\ln x -\frac{1}{4}(3x-2)x+\frac{1}{2}{\,}x\ln(2\pi)+\ln A\\
&& \null +\int_0^{\infty}\frac{B_2(\{t\})}{2(x+t)}{\,}dt.
\end{eqnarray*}

\parag{Limit and series representations when $\nu\geq 1$} The Eulerian and Weierstrassian forms of $\psi_{\nu}$ reduce to
$$
\psi_{\nu}(x) ~=~ -\sum_{k=0}^{\infty}g_{\nu}(x+k) ~=~ (-1)^{\nu -1}{\,}\nu !\,\zeta(\nu +1,x)
$$
and this series converges uniformly on $\R_+$.

\parag{Limit and series representations when $\nu\leq -1$} The analogue of Gauss' limit is
$$
\psi_{\nu}(x) ~=~ \psi_{\nu}(1)+\lim_{n\to\infty}f^{-\nu}_n[g_{\nu}](x)
$$
and both sides can be integrated on any bounded subset of $[0,\infty)$ (the limit and the integral commute). They can also be differentiated infinitely many times (the limit and the derivative operator commute).

For instance, when $\nu=-2$ we obtain
\begin{eqnarray*}
\psi_{-2}(x) &=& \lim_{n\to\infty}\Bigg(\sum_{k=1}^{n-1}k\ln k-\sum_{k=0}^{n-1}(x+k)\ln(x+k)+x\left(n\ln n+\frac{1}{2}\ln(2\pi)\right) \\
&& \null  + {x\choose 2}\left((n+1)\ln\left(1+\frac{1}{n}\right)+\ln n-1\right)\Bigg).
\end{eqnarray*}
Comparing this formula with that of \eqref{eq:IntPsi2m65}, we see that the latter is less complicated, since it was produced from less terms in its polynomial part. Now, differentiating the formula above, we obtain a limit representation for $\ln\Gamma(x)$, but the Gauss limit is less complicated. In this context, finding the simplest limit representations seems to be an interesting problem.

The Eulerian and Weistrassian representations of $\psi_{\nu}$ take the following forms
\begin{eqnarray*}
\psi_{\nu}(x)-\psi_{\nu}(1) &=& -g_{\nu}(x)+\sum_{j=1}^{-\nu}\tchoose{x}{j}\Delta^{j-1}g_{\nu}(1)\\
&& \null + \sum_{k=1}^{\infty}\left(-g_{\nu}(x+k)+\sum_{j=0}^{-\nu}\tchoose{x}{j}\,\Delta^j g_{\nu}(k)\right)
\end{eqnarray*}
and
\begin{eqnarray*}
\psi_{\nu}(x)-\psi_{\nu}(1) &=& -g_{\nu}(x)+\sum_{j=1}^{-\nu-1}\tchoose{x}{j}\Delta^{j-1}g_{\nu}(1)-\gamma\tchoose{x}{-\nu}\\
&& \null + \sum_{k=1}^{\infty}\left(-g_{\nu}(x+k)+\sum_{j=0}^{-\nu -1}\tchoose{x}{j}\,\Delta^j g_{\nu}(k)+\tchoose{x}{-\nu}\frac{1}{k}\right),
\end{eqnarray*}
respectively. These series can be integrated term by term on any bounded subset of $[0,\infty)$. They can also be differentiated term by term infinitely many times.

For instance, when $\nu=-2$, both identities above reduce to
$$
\psi_{-2}(x) ~=~ \ln\left(\frac{(2\pi)^{\frac{1}{2}x}(\frac{4}{e})^{{x\choose 2}}}{x^x}{\,}
\prod_{k=1}^{\infty}\frac{(1+2/k)^{(k+2){x\choose 2}}}{(1+x/k)^{x+k}{\,}(1+1/k)^{(k+1)x(x-2)}}\right)
$$
and
$$
\psi_{-2}(x) ~=~ \ln\left(\frac{(2\pi)^{\frac{1}{2}x}e^{-\gamma{x\choose 2}}}{x^x}{\,}
\prod_{k=1}^{\infty}\frac{e^{\frac{1}{k}{x\choose 2}}{\,}(1+1/k)^{(k+1)x}}{(1+x/k)^{x+k}}\right).
$$
Integrating both the Eulerian and Weierstrassian forms of $\ln\Gamma(x)$, we obtain the following representations (which are simpler than the previous ones since less terms are involved; see also Examples~\ref{ex:GepeA58} and \ref{ex:GepeA58W})
$$
\psi_{-2}(x) ~=~ \ln\left(\frac{e^x}{x^x}\,\prod_{k=1}^{\infty}\frac{e^x(1+1/k)^{x^2/2}}{(1+x/k)^{x+k}}\right)
$$
and
$$
\psi_{-2}(x) ~=~ \ln\left(e^{-\gamma x^2/2}{\,}\frac{e^x}{x^x}{\,}\prod_{k=1}^{\infty}\frac{e^{x+x^2/(2k)}}{(1+x/k)^{x+k}}\right).
$$
Here again, finding the simplest Eulerian and Weierstrassian forms remains an interesting problem.

\parag{Integral representation} For any $\nu\in\Z$, we have
$$
\psi_{\nu}(x) ~=~ \psi_{\nu}(1)+\int_1^x\psi_{\nu +1}(t){\,}dt.
$$
If $\nu\geq 1$, then $\psi_{\nu}$ is not integrable at $x=0$ (since $g_{\nu}$ is not). If $\nu\leq -1$, then $\psi_{\nu}$ is integrable at $0$ by definition and we have
$$
\psi_{\nu-1}(x) ~=~ \int_0^x\psi_{\nu}(t){\,}dt ~=~ \int_0^x\frac{(x-t)^{-\nu-1}}{(-\nu-1)!}{\,}\ln\Gamma(t){\,}dt.
$$

\parag{Gregory's formula-based series representation} Proposition~\ref{prop:6699rem} gives the following series representation: for any $x>0$ we have
\begin{eqnarray*}
\psi_{\nu}(x) &=& g_{\nu -1}(x)-\sum_{n=0}^{\infty}G_{n+1}\,\Delta^ng_{\nu}(x)\\
&=& g_{\nu -1}(x)-\sum_{n=0}^{\infty}|G_{n+1}|\,\sum_{k=0}^n(-1)^k\tchoose{n}{k}g_{\nu}(x+k){\,}.
\end{eqnarray*}
Setting $x=1$ in this identity yields the analogue of Fontana-Mascheroni series. For instance, taking $\nu=1$, we derive the identity (see, e.g., Merlini {\em et al.} \cite[p.~1920]{MerSprVer06})
$$
\sum_{n=1}^{\infty}|G_n|\,\frac{H_n}{n} ~=~ \frac{\pi^2}{6}-1{\,}.
$$
Taking $\nu =2$, we obtain
$$
\sum_{n=1}^{\infty}|G_n|\,\frac{\psi_1(n+1)-H_n^2}{n} ~=~ 1-2\,\zeta(3)+\gamma\,\frac{\pi^2}{6}{\,}.
$$

\parag{Analogue of Gauss' multiplication formula} Assume first that $\nu\geq 1$. Differentiating repeatedly both sides of the multiplication formula \eqref{eq:GMF3PSd} for the digamma function $\psi$,\index{digamma function} we obtain the following formula. For any $m\in\N^*$ and any $x>0$, we have
$$
\sum_{j=0}^{m-1}\psi_{\nu}\left(\frac{x+j}{m}\right) ~=~ m^{\nu +1}\,\psi_{\nu}(x).
$$
Moreover, Corollary~\ref{cor:Riem482} provides the following limit
$$
\lim_{m\to\infty}m^{\nu}\psi_{\nu}(mx) ~=~ g_{\nu -1}(x),\qquad x>0.
$$
Assume now that $\nu\leq -1$. Applying Theorem~\ref{thm:MultThmGen} to the function $g_{\nu}$, we obtain that for any $m\in\N^*$ and any $x>0$
$$
\sum_{j=0}^{m-1}\psi_{\nu}\left(\frac{x+j}{m}\right) ~=~ \sum_{j=1}^{m-1}\psi_{\nu}\left(\frac{j}{m}\right)+\psi_{\nu}(1)+\Sigma_x{\,}g_{\nu}\left(\frac{x}{m}\right){\,}.
$$
Let us expand this formula in the special case when $\nu=-2$. First, we have
$$
g_{-2}\left(\frac{x}{m}\right) ~=~ \frac{1}{m}{\,}g_{-2}(x)-x\frac{\ln m}{m}+\frac{m-1}{m}\,\psi_{-2}(1)
$$
and hence
$$
\Sigma_x{\,}g_{-2}\left(\frac{x}{m}\right) ~=~ \frac{1}{m}{\,}\psi_{-2}(x)-{x\choose 2}\,\frac{\ln m}{m}+\left(\frac{m-1}{m}{\,}x-1\right)\,\psi_{-2}(1).
$$
Using Proposition~\ref{prop:CalcSum662}, after some algebra we also obtain
$$
\sum_{j=1}^{m-1}\psi_{-2}\left(\frac{j}{m}\right) ~=~ \left(1-\frac{1}{m}\right)\ln A-\frac{\ln m}{12{\,}m}+(m-1)\ln((2\pi)^{\frac{1}{4}}A).
$$
Now, collecting terms, we finally get the following multiplication formula for $\psi_{-2}$
\begin{eqnarray*}
\sum_{j=0}^{m-1}\psi_{-2}\left(\frac{x+j}{m}\right) &=& \frac{1}{m}\,\psi_{-2}(x)-\frac{1}{12m}{\,}(6x^2-6x+1)\ln m \\
&& \null + (m-1)\ln(2\pi)\left(\frac{x}{2m}+\frac{1}{4}\right)+\left(m-\frac{1}{m}\right)\ln A.
\end{eqnarray*}
Setting $m=2$ in the formula above, we obtain the following analogue of Legendre's duplication formula
\begin{eqnarray*}
\psi_{-2}\left(\frac{x}{2}\right)+\psi_{-2}\left(\frac{x+1}{2}\right) &=& \frac{1}{2}\,\psi_{-2}(x)-\frac{1}{24}{\,}(6x^2-6x+1)\ln 2\\
&& \null +\frac{1}{4}\ln(2\pi){\,}(x+1)+\frac{3}{2}\ln A .
\end{eqnarray*}
Taking $x=0$ in this latter identity, we obtain
$$
\psi_{-2}\left(\frac{1}{2}\right) ~=~ \frac{5}{24}\ln 2+\frac{3}{2}\ln A+\frac{1}{4}\ln\pi{\,}.
$$
Moreover, Corollary~\ref{cor:Riem482} provides the following limit
$$
\lim_{m\to\infty}\left(\frac{1}{m^2}\,\psi_{-2}(mx)-\frac{x^2}{2}\ln m\right) ~=~ \frac{1}{2}{\,}x^2\ln x-\frac{3}{4}{\,}x^2{\,},\qquad x>0.
$$

\parag{Analogue of Wallis's product formula}  If $\nu\geq 1$, then the analogue of Wallis's formula is simply
$$
\sum_{k=1}^{\infty}(-1)^{k-1}g_{\nu}(k) ~=~ (-1)^{\nu}(1-2^{-\nu}){\,}\nu!{\,}\zeta(\nu +1),
$$
or equivalently,
$$
\sum_{k=1}^{\infty}(-1)^{k-1}g_{\nu}(k) ~=~ (-1)^{\nu}{\,}\nu!{\,}\eta(\nu +1),
$$
where $\eta$ is Dirichlet's eta function.\index{Dirichlet's eta function} In the case when $\nu =-2$, after a bit of calculus we obtain the following analogue of Wallis's formula
$$
\lim_{n\to\infty}\left(h(n)+\sum_{k=1}^{2n}(-1)^{k-1}g_{-2}(k)\right) ~=~ \frac{1}{12}\ln 2-3\ln A.
$$
where
$$
h(n) ~=~ \left(n+\frac{1}{4}\right)\ln n - n(1-\ln 2).
$$

\begin{project}
\textsl{Find the analogue of Wallis's formula for the function $g(x)=\psi_{-2}(x)$.} After some algebra, we obtain
$$
\lim_{n\to\infty}\left(h(n)+\sum_{k=1}^{2n}(-1)^{k-1}\psi_{-2}(k)\right) ~=~ \ln A-\frac{1}{12}\,\ln 2{\,},
$$
where
$$
h(n) ~=~ n^2\ln(2n)-\frac{3}{2}{\,}n^2+\frac{1}{2}{\,}n\,\ln(2\pi)-\frac{1}{12}\,\ln n.
$$
This formula is a little harder to obtain than the former one; it requires the computation of both functions $\Sigma\psi_{-2}(x)$ and $2{\,}\Sigma_x\psi_{-2}(2x)$ using the elevator method\index{elevator method} (Corollary~\ref{cor:saf6f}) with $r=2$. That is,
\begin{eqnarray}
\Sigma\psi_{-2}(x) &=& -\frac{1}{12}{\,}x(x-1)(2x-1)+\frac{1}{4}{\,}x(x+1)\ln(2\pi)\nonumber\\
&& \null +2x\ln A+(x-1)\,\psi_{-2}(x)-2\,\psi_{-3}(x)\label{eq:S2psi2}
\end{eqnarray}
and
\begin{eqnarray*}
2{\,}\Sigma_x\psi_{-2}(2x) &=& -\frac{1}{6}{\,}x(2x-1)(4x-1)+(4x+3)\ln A\\
&& \null +\frac{1}{12}{\,}(-24x^2+48x+5)\ln 2-4\,\psi_{-2}(x)\\
&& \null +2x\,\psi_{-2}(2x)-2\,\psi_{-2}\left(x+\frac{1}{2}\right)-2\,\psi_{-3}(2x).
\end{eqnarray*}
These formulas can also be verified using the difference operator.
\end{project}

\parag{Restriction to the natural integers when $\nu\geq 1$} For any $n\in\N^*$, we have
$$
\psi_{\nu}(n)-\psi_{\nu}(1) ~=~ \sum_{k=1}^{n-1}g_{\nu}(k) ~=~ (-1)^{\nu}\nu!{\,}\sum_{k=1}^{n-1}\frac{1}{k^{\nu +1}}{\,}.
$$
In particular,
$$
\psi_{\nu}(1) ~=~ -\sum_{k=1}^{\infty}g_{\nu}(k).
$$
Gregory's formula states that for any $n\in\N^*$ and any $q\in\N$ we have
\begin{eqnarray*}
\sum_{k=1}^{n-1}g_{\nu}(k) &=& g_{\nu -1}(n)-g_{\nu -1}(1)\\
&& \null -\sum_{j=1}^qG_j\left(\Delta^{j-1}g_{\nu}(n)-\Delta^{j-1}g_{\nu}(1)\right)-R^q_n{\,},
\end{eqnarray*}
with
$$
|R^q_n| ~\leq ~\overline{G}_q\left|\Delta^qg_{\nu}(n)-\Delta^qg_{\nu}(1)\right|.
$$

\parag{Generalized Webster's functional equation} For any $m\in\N^*$, there is a unique solution $f\colon\R_+\to\R$ to the equation
$$
\sum_{j=0}^{m-1}f\left(x+\frac{j}{m}\right) ~=~ g_{\nu}(x)
$$
that lies in $\cK^{(-\nu)_+}$, namely
$$
f(x) ~=~ \psi_{\nu}\left(x+\frac{1}{m}\right)-\psi_{\nu}(x){\,}.
$$

\parag{Analogue of Euler's series representation of $\gamma$} Assume first that $\nu\geq 1$. In this case, for any $k\in\N$ we have
$$
\psi_{\nu}^{(k)}(1) ~=~ \psi_{\nu +k}(1) ~=~ (-1)^{\nu +k-1} (\nu +k)!\,\zeta(\nu +k+1).
$$
Thus, the Taylor series expansion of $\psi_{\nu}(x+1)$ about $x=0$ is
$$
\psi_{\nu}(x+1) ~=~ \sum_{k=0}^{\infty}(-1)^{\nu +k-1}\frac{(\nu +k)!}{k!}\,\zeta(\nu +k+1){\,}x^k,\qquad |x|<1.
$$
Integrating both sides of this equation on $(0,1)$, we obtain the identity
$$
g_{\nu -1}(1) ~=~ \sum_{k=0}^{\infty}(-1)^{\nu +k-1}\frac{(\nu +k)!}{(k+1)!}\,\zeta(\nu +k+1){\,}.
$$
We proceed similarly when $\nu\leq -1$. To keep the computations simple, let us assume that $\nu =-2$. We then have
$$
\psi_{-2}(1) ~=~ \frac{1}{2}\ln(2\pi),\quad\psi'_{-2}(1) ~=~ \psi_{-1}(1) ~=~ 0,\quad\psi''_{-2}(1) ~=~ \psi_0(1) ~=~ -\gamma,
$$
and for any integer $k\geq 3$,
$$
\psi_{-2}^{(k)}(1) ~=~ \psi_{k-2}(1) ~=~ (-1)^{k-1}(k-2)!\,\zeta(k-1).
$$
Thus, the Taylor series expansion of $\psi_{-2}(x+1)$ about $x=0$ is
$$
\psi_{-2}(x+1) ~=~ \frac{1}{2}\ln(2\pi)-\gamma\,\frac{x^2}{2}+\sum_{k=3}^{\infty}(-1)^{k-1}\frac{\zeta(k-1)}{(k-1)k}{\,}x^k,\qquad |x|<1.
$$
Integrating both sides of this equation on $(0,1)$, we obtain
$$
\sum_{k=2}^{\infty}(-1)^k\,\frac{\zeta(k)}{k(k+1)(k+2)} ~=~ \frac{1}{6}\,\gamma-\frac{3}{4}+\frac{1}{4}\ln(2\pi)+\ln A{\,}.
$$

\parag{Analogue of the reflection formula} Assume first that $\nu\geq 1$. Differentiating the reflection formula for $\psi$ repeatedly, we obtain the following formula. For any $x\in\R\setminus\Z$, we have
$$
\psi_{\nu}(x)-(-1)^{\nu}\psi_{\nu}(1-x) ~=~ -\pi{\,}D^{\nu}\cot(\pi x).
$$
When $\nu\leq -1$, a reflection formula on $(0,1)$ can be obtained by integrating both sides of the identity
$$
\ln\Gamma(x)+\ln\Gamma(1-x) ~=~ \ln\pi-\ln\sin(\pi x).
$$
For example, for any $x\in (0,1)$ we have
$$
\psi_{-2}(x)-\psi_{-2}(1-x) ~=~ x\,\ln\pi-\frac{1}{2}\ln(2\pi)-\int_0^x\ln\sin(\pi t){\,}dt.
$$
As a byproduct, we obtain
$$
\int_0^{\frac{1}{2}}\ln\sin(\pi t){\,}dt ~=~ -\frac{1}{2}\ln 2.
$$

\index{polygamma functions|)}

\section{The $q$-gamma function}
\index{$q$-gamma function|(}

For any $0<q<1$, the $q$-gamma function $\Gamma_q\colon\R_+\to\R_+$ is defined by the equation\label{p:qG53} (see, e.g., \cite[p.~490]{SriCho12})
\begin{equation}\label{eq:q1GS460}
\Gamma_q(x) ~=~ (1-q)^{1-x}\,\prod_{k=0}^{\infty}\frac{1-q^{k+1}}{1-q^{x+k}} ~=~ (1-q)^{1-x}\,\frac{(q;q)_{\infty}}{(q^x;q)_{\infty}}\qquad\text{for $x>0$}.
\end{equation}
Here we use the standard notation
$$
(a;q)_{\infty} ~=~ \prod_{k=0}^{\infty}\left(1-aq^k\right).
$$
Note that these functions should not to be confused with the multiple gamma functions discussed in Section~\ref{subsec:MLG-t} (although the same symbols are used).

The function $f_q(x)=\ln\Gamma_q(x)$ is a convex solution satisfying $f_q(1)=0$ to the equation $\Delta f_q=g_q$ on $\R_+$, where $g_q\colon\R_+\to\R$ is the function defined by the equation
$$
g_q(x) ~=~ \ln\frac{1-q^x}{1-q}\qquad\text{for $x>0$}.
$$
Since $g_q$ lies in $\cD^1\cap\cK^1$ (and $\deg g_q=0$), by the uniqueness theorem we must have
\begin{equation}\label{eq:q1GS46}
\ln\Gamma_q(x) ~=~ \Sigma g_q(x),\qquad x>0.
\end{equation}
Askey~\cite{Ask78} proved an analogue of the Bohr-Mollerup theorem for $\Gamma_q$. However, as Webster~\cite[p.~615]{Web97b} already observed, this is actually an immediate consequence of the uniqueness Theorem~\ref{thm:unic} in the special case when $p=1$.

Let us now investigate this function in the light of our results.

\begin{remark}
When $q>1$, the $q$-gamma function $\Gamma_q\colon\R_+\to\R_+$ is also defined by Eq.~\eqref{eq:q1GS46}. In this case, using L'Hospital's rule we can readily see that $\Delta g_q(x)\to\ln q$ as $x\to\infty$, and hence $\deg g_q=1$. An analogue of the Bohr-Mollerup characterization for $\Gamma_q$ was established by Moak~\cite{Moa80}. We can see now that this characterization is a trivial consequence of our uniqueness Theorem~\ref{thm:unic} in the special case when $p=2$. The complete analysis of $\Gamma_q$ through our results is similar to the case when $0<q<1$ and is left to the reader.
\end{remark}

\parag{ID card} As discussed above, the function $\Gamma_q$ is a $\Gamma$-type function and we immediately derive the following basic information.
$$
\begin{array}{|c|c|c|c|}
\hline g_q(x) & \text{Membership} & \deg g_q & \Sigma g_q(x) \\
\hline \emptybox \ln\frac{1-q^x}{1-q_{\mathstrut}} & \cC^{\infty}\cap\cD^1\cap\cK^{\infty} & 0 & \ln\Gamma_q(x) \\
\hline
\end{array}
$$

\parag{Analogue of Bohr-Mollerup's theorem} The $q$-gamma function can be characterized as follows.
\begin{quote}
\emph{All eventually convex or concave solutions $f_q\colon\R_+\to\R$ to the equation
$$
f_q(x+1)-f_q(x) ~=~ \ln\frac{1-q^x}{1-q}
$$
are of the form $f_q(x)=c_q+\ln\Gamma_q(x)$, where $c_q\in\R$.}
\end{quote}
Using Proposition~\ref{prop:90unic41}, we can also derive the following alternative characterization of the $q$-gamma function.
\begin{quote}
\emph{All solutions $f_q\colon\R_+\to\R$ to the equation
$$
f_q(x+1)-f_q(x) ~=~ \ln\frac{1-q^x}{1-q}
$$
that satisfy the asymptotic condition that, for each $x>0$,
$$
f_q(x+n)-f_q(n)-x\ln\frac{1-q^n}{1-q} ~\to ~0\qquad\text{as $n\to_{\N}\infty$}
$$
are of the form $f_q(x)=c_q+\ln\Gamma_q(x)$, where $c_q\in\R$.}
\end{quote}

\parag{Extended ID card} Interestingly, El Bachraoui~\cite{ElB17} recently established the following analogue of Raabe's formula
$$
\int_x^{x+1}\ln\Gamma_q(t){\,}dt ~=~ \left(\frac{1}{2}-x\right)\ln(1-q)-\frac{1}{\ln q}\,\mathrm{Li}_2(q^x)+\ln(q;q)_{\infty},\qquad x\geq 0,
$$
where
$$
\mathrm{Li}_s(z) ~=~ \sum_{k=1}^{\infty}\frac{z^k}{k^s}
$$
is the polylogarithm function.\index{polylogarithm function} This formula provides immediately the following values
\begin{eqnarray}
\overline{\sigma}[g_q] &=& \frac{1}{2}\ln(1-q)-\frac{\zeta(2)}{\ln q}+\ln(q;q)_{\infty}{\,},\label{eq:qGSb46}\\
\sigma[g_q] &=& -\frac{1}{2}\ln(1-q)-\frac{1}{\ln q}\,\mathrm{Li}_2(q)+\ln(q;q)_{\infty}{\,},\label{eq:qGS46}
\end{eqnarray}
and the integral
$$
\int_1^x g_q(t){\,}dt ~=~ (1-x)\ln(1-q)-\frac{1}{\ln q}\left(\mathrm{Li}_2(q^x)-\mathrm{Li}_2(q)\right).
$$
We then have the following values
$$
\begin{array}{|c|c|c|}
\hline \overline{\sigma}[g_q] & \sigma[g_q] & \gamma[g_q] \\
\hline \emptybox Eq.~\eqref{eq:qGSb46} & Eq.~\eqref{eq:qGS46} & \gamma[g_q]=\sigma[g_q] \\
\hline
\end{array}
$$

\begin{itemize}
\item \emph{Alternative representations of $\sigma[g_q]=\gamma[g_q]$}
\begin{eqnarray*}
\sigma[g_q] &=& \int_0^1\ln\Gamma_q(t+1){\,}dt{\,},\\
\sigma[g_q] &=& \log[q]\,\int_1^{\infty}\left(\frac{1}{2}-\{t\}\right)\frac{q^t}{1-q^t}{\,}dt{\,},\\
\sigma[g_q] &=& \int_1^{\infty}\ln\frac{(1-q^{\lfloor t\rfloor})^{1/2}(1-q^{\lfloor t+1\rfloor})^{1/2}}{1-q^t}{\,}dt{\,},\\
\sigma[g_q] &=& \frac{1}{2}\sum_{k=1}^{\infty}\ln\left((1-q^k)(1-q^{k+1})\right)-\frac{1}{\ln q}\,\mathrm{Li}_2(q){\,}.\\
\end{eqnarray*}

\item \emph{Generalized Binet's function}
\begin{eqnarray*}
J^2[\ln\circ\Gamma_q](x) &=& \ln\Gamma_q(x)+(x-1)\ln(1-q)+\frac{1}{\ln q}\,\mathrm{Li}_2(q^x)+\frac{1}{2}\ln(1-q^x)\\
&& \null -\ln(q;q)_{\infty}{\,}.
\end{eqnarray*}

\item \emph{Alternative characterization}. The function $f_q(x) = \ln\Gamma_q(x)$ is the unique solution lying in $\cC^0\cap\cK^1$ to the equation
$$
\int_x^{x+1}f_q(t){\,}dt ~=~ \left(\frac{1}{2}-x\right)\ln(1-q)-\frac{1}{\ln q}\,\mathrm{Li}_2(q^x)+\ln(q;q)_{\infty},\qquad x>0.
$$
\end{itemize}

\parag{Inequalities} The following inequalities hold for any $x>0$ and any $a\geq 0$.
\begin{itemize}
\item \emph{Symmetrized generalized Wendel's inequality} (equality if $a\in\{0,1\}$)
\begin{eqnarray*}
\left|\ln\Gamma_q(x+a)-\ln\Gamma_q(x)-a{\,}g_q(x)\right| &\leq & |a-1|{\,}\left|g_q(x+a)-g_q(x)\right|\\
&\leq & \lceil a\rceil{\,}|a-1|{\,}|\Delta g_q(x)|{\,},
\end{eqnarray*}
$$
\left(\frac{1-q^{x+a}}{1-q^x}\right)^{-|a-1|} \leq ~\frac{\Gamma_q(x+a)}{\Gamma_q(x)\left(\frac{1-q^x}{1-q}\right)^a} ~\leq ~ \left(\frac{1-q^{x+a}}{1-q^x}\right)^{|a-1|}.
$$

\item \emph{Symmetrized Stirling's formula-based inequality}
$$
|J^2[\ln\circ\Gamma_q](x)| ~\leq ~ \frac{1}{2}\left(g_q(x+1)-g_q(x)\right),
$$
$$
\left(\frac{1-q^{x+1}}{1-q^x}\right)^{-\frac{1}{2}} \leq ~ \frac{\Gamma_q(x){\,}(1-q)^{x-1}(1-q^x)^{\frac{1}{2}}}{(q;q)_{\infty}{\,}\exp\left(-\frac{1}{\ln q}\mathrm{Li}_2(q^x)\right)} ~\leq ~ \left(\frac{1-q^{x+1}}{1-q^x}\right)^{\frac{1}{2}}.
$$

\item \emph{Burnside's formula-based inequality}
\begin{multline*}
\left|\ln\Gamma_q\left(x+\frac{1}{2}\right)+\left(x-\frac{1}{2}\right)\ln(1-q)+\frac{1}{\ln q}\,\mathrm{Li}_2(q^x)-\ln(q;q)_{\infty}\right|\\
\leq ~ |J^2[\ln\circ\Gamma_q](x)|.
\end{multline*}

\item \emph{Generalized Gautschi's inequality}
$$
e^{(a-\lceil a\rceil)\,\psi_{q,0}(x+\lceil a\rceil)} ~\leq ~ \frac{\Gamma_q(x+a)}{\Gamma_q(x+\lceil a\rceil)} ~\leq ~ \left(\frac{1-q^{x+\lfloor a\rfloor}}{1-q}\right)^{a-\lceil a\rceil},
$$
where $\psi_{q,0}(x)=D\ln\Gamma_q(x)$.
\end{itemize}

\parag{Generalized Stirling's and related formulas} For any $a\geq 0$, we have the following limits and
asymptotic equivalences as $x\to\infty$,
$$
\ln\Gamma_q(x+a)-\ln\Gamma_q(x) ~\to ~ -a\ln(1-q),
$$
$$
\frac{\Gamma_q(x)}{\Gamma_q(x+a)} ~\sim ~ (1-q)^a,\qquad \ln\Gamma_q(x+a) ~\sim ~ -x\ln(1-q){\,},
$$
$$
\ln\Gamma_q(x)+(x-1)\ln(1-q)-\ln(q;q)_{\infty} ~\to ~ 0{\,},
$$
$$
\Gamma_q(x) ~\sim ~ (q;q)_{\infty}{\,}(1-q)^{1-x}{\,}.
$$
The generalized Stirling formula simply shows that $\ln\Gamma_q(x)$ has the oblique asymptote
$$
y ~=~ (1-x)\ln(1-q)+\ln(q;q)_{\infty}.
$$
\emph{Burnside-like approximation} (better than Stirling-like approximation)
$$
\Gamma_q(x) ~\sim ~ (q;q)_{\infty}{\,}(1-q)^{1-x}\,\exp\left(-\frac{1}{\ln q}\mathrm{Li}_2(q^{x-\frac{1}{2}})\right).
$$
\emph{Further results} (obtained by differentiation). For any $0<q<1$ and any $\nu\in\N$, let the function $\psi_{q,\nu}\colon\R_+\to\R$ denote the $q$-polygamma function defined by the equation
$$
\psi_{q,\nu}(x) ~=~ D^{\nu +1}\ln\Gamma_q(x)\qquad\text{for $x>0$}.
$$
We then have the following limits and asymptotic equivalences as $x\to\infty$,
$$
\psi_{q,0}(x+a)-\psi_{q,0}(x) ~\to ~ 0,\qquad \psi_{q,0}(x) ~\to ~ -\ln(1-q){\,},
$$
$$
\psi_{q,0}(x+a) ~\sim ~ -\ln(1-q),\qquad\psi_{q,\nu}(x) ~\to ~ 0,\qquad\nu\in\N^*.
$$

\begin{project}
\textsl{Find the generalized Stirling formula when $q>1$}. In the case when $q>1$, we have $\deg g_q=1$ and hence the generalized Stirling formula is
$$
\ln\Gamma_q(x)-\int_x^{x+1}\ln\Gamma_q(t){\,}dt+\frac{1}{2}{\,}g_q(x)-\frac{1}{12}\Delta g_q(x) ~\to ~ 0\qquad\text{as $x\to\infty$},
$$
where $\Delta g_q(x)\to\ln q$ as $x\to\infty$. However, here the integral takes the following more complicated form (see El Bachraoui \cite{ElB17} and the references therein)
\begin{eqnarray*}
\int_x^{x+1}\ln\Gamma_q(t){\,}dt &=& \ln C_q-\frac{1}{2q^x\ln q}\Bigg(\frac{1-q^x}{1-q^{-x}}{\,}(2\,\mathrm{Li}_2(q^{-x})+(\ln(1-q^{-x}))^2)\\
&& \null - 2\,\frac{1-q^x}{1-q^{-x}}\ln\frac{1-q^x}{1-q}\ln(1-q^{-x})-q^x\left(\ln\frac{1-q^x}{1-q}\right)^2\Bigg)
\end{eqnarray*}
where
$$
C_q ~=~ q^{-\frac{1}{12}}(q-1)^{\frac{1}{2}-\frac{\ln(q-1)}{2\ln q}}(q^{-1};q^{-1})_{\infty}{\,}.
$$
This is the analogue of Raabe's formula for $\ln\Gamma_q(x)$ when $q>1$.
\end{project}

\parag{Asymptotic expansions} For any $m,r\in\N^*$ we have the following expansion as $x\to\infty$
\begin{eqnarray*}
\frac{1}{m}\sum_{j=0}^{m-1}\ln\Gamma_q\left(x+\frac{j}{m}\right) &=& \left(\frac{1}{2}-x\right)\ln(1-q)-\frac{1}{\ln q}\,\mathrm{Li}_2(q^x)+\ln(q;q)_{\infty}\\
&& \null + \sum_{k=1}^r\frac{B_k}{m^k{\,}k!}{\,}g_q^{(k-1)}(x)+O\left(g_q^{(r)}(x)\right).
\end{eqnarray*}
Setting $m=1$ in this formula, we obtain the expansion of the log-$q$-gamma function
\begin{eqnarray*}
\ln\Gamma_q(x) &=& \left(\frac{1}{2}-x\right)\ln(1-q)-\frac{1}{\ln q}\,\mathrm{Li}_2(q^x)+\ln(q;q)_{\infty}\\
&& \null + \sum_{k=1}^r\frac{B_k}{k!}{\,}g_q^{(k-1)}(x)+O\left(g_q^{(r)}(x)\right).
\end{eqnarray*}

\parag{Generalized Liu's formula} For any $x>0$, we have
\begin{eqnarray*}
\ln\Gamma_q(x) &=& \left(\frac{1}{2}-x\right)\ln(1-q)-\frac{1}{\ln q}\,\mathrm{Li}_2(q^x)+\ln(q;q)_{\infty}\\
&& \null -\frac{1}{2}\ln\frac{1-q^x}{1-q}+(\ln q)\int_0^{\infty}\left(\{t\}-\frac{1}{2}\right)\,\frac{q^{x+t}}{1-q^{x+t}}{\,}dt.
\end{eqnarray*}

\parag{Limit and series representations} It is not difficult to see that both the Eulerian form of $\Sigma g_q(x)$ and the analogue of Gauss's limit reduce to the definition of the $q$-gamma function given in Eq.~\eqref{eq:q1GS460}. Let us now examine the other series representations.
\begin{itemize}
\item \emph{Weierstrassian form}. For any $x>0$, we have
$$
\ln\Gamma_q(x) ~=~ -\ln\frac{1-q^x}{1-q}+\psi_{q,0}(1){\,}x-\sum_{k=1}^{\infty}\left(\ln\frac{1-q^{x+k}}{1-q^k}+(\ln q)\,\frac{q^k}{1-q^k}{\,}x\right).
$$
Differentiating this series term by term, we obtain
$$
\psi_{q,0}(x) ~=~ (\ln q)\,\frac{q^x}{1-q^x}+\psi_{q,0}(1)+(\ln q)\sum_{k=1}^{\infty}\left(\frac{1}{1-q^{x+k}}-\frac{1}{1-q^k}\right).
$$

\item \emph{Gregory's formula-based series representation}. For any $x>0$ we have the series representation
\begin{eqnarray*}
\ln\Gamma_q(x) &=& \left(\frac{1}{2}-x\right)\ln(1-q)-\frac{1}{\ln q}\,\mathrm{Li}_2(q^x)+\ln(q;q)_{\infty}\\
&& \null -\sum_{n=0}^{\infty}|G_{n+1}|\,\sum_{k=0}^n(-1)^k\tchoose{n}{k}g_q(x+k).
\end{eqnarray*}
Setting $x=1$ in this identity yields the following analogue of Fontana-Mascheroni series
$$
\sum_{n=0}^{\infty}|G_{n+1}|\,\sum_{k=0}^n(-1)^k\tchoose{n}{k}g_q(k+1) ~=~ -\frac{1}{2}\ln(1-q)-\frac{1}{\ln q}\,\mathrm{Li}_2(q)+\ln(q;q)_{\infty}.
$$
\end{itemize}

\parag{Analogue of Gauss' multiplication formula} After first noting that
$$
g_q\left(\frac{x}{m}\right) ~=~ g_{q^{\frac{1}{m}}}(x)+g_q\left(\frac{1}{m}\right),\qquad x>0,
$$
we immediately obtain the following identity
$$
\sum_{j=0}^{m-1}\ln\Gamma_q\left(x+\frac{j}{m}\right) ~=~ \sum_{j=1}^m\ln\Gamma_q\left(\frac{j}{m}\right)+\ln\Gamma_{q^{\frac{1}{m}}}(mx)+(mx-1){\,}g_q\left(\frac{1}{m}\right).
$$
Now, using Proposition~\ref{prop:CalcSum662}, we also obtain
$$
\sum_{j=1}^m\ln\Gamma_q\left(\frac{j}{m}\right) ~=~ \frac{m-1}{2}\ln(1-q)+m\ln(q;q)_{\infty}-\ln\left(q^{\frac{1}{m}};q^{\frac{1}{m}}\right)_{\infty}.
$$
Thus, we get the following multiplication formula
$$
\prod_{j=0}^{m-1}\Gamma_q\left(x+\frac{j}{m}\right) ~=~ (1-q)^{\frac{m-1}{2}}\frac{(q;q)^m_{\infty}}{\left(q^{\frac{1}{m}};q^{\frac{1}{m}}\right)_{\infty}}
~\Gamma_{q^{\frac{1}{m}}}(mx)\left(\frac{1-q^{\frac{1}{m}}}{1-q}\right)^{mx-1}{\,},
$$
or equivalently, replacing $q$ with $q^m$,
$$
\prod_{j=0}^{m-1}\Gamma_{q^m}\left(x+\frac{j}{m}\right) ~=~ (1-q^m)^{\frac{m-1}{2}}\frac{\left(q^m;q^m\right)^m_{\infty}}{(q;q)_{\infty}}
~\Gamma_q(mx)\left(\frac{1-q}{1-q^m}\right)^{mx-1}{\,}.
$$
(See also, e.g., Srivastava and Choi \cite[p.~494]{SriCho12} and Webster \cite[p.~617]{Web97b}.) For instance, when $m=2$, we obtain the following analogue of Legendre's duplication formula
$$
\Gamma_{q^2}(x)\,\Gamma_{q^2}\left(x+\frac{1}{2}\right) ~=~ (1-q^2)^{\frac{1}{2}}\frac{\left(q^2;q^2\right)^2_{\infty}}{(q;q)_{\infty}}~\frac{\Gamma_q(2x)}{(1+q)^{2x-1}}{\,}.
$$

\parag{Analogue of Wallis's product formula} Using Proposition~\ref{prop:Wallis92} with
$$
\tilde{g}_q(x) ~=~ 2g_q(2x) ~=~ 2(g_{q^2}(x)+g_q(2)),
$$
we obtain
\begin{eqnarray*}
h(n) &=& \Sigma\tilde{g}_q(n+1)-\Sigma g_q(2n+1)\\
&=& 2\ln\Gamma_{q^2}(n+1)+2g_2(2)n-\ln\Gamma_q(2n+1).
\end{eqnarray*}
Using the generalized Stirling formula, we then have
$$
\lim_{n\to\infty}h(n) ~=~ 2\ln(q^2;q^2)_{\infty}-\ln(q;q)_{\infty}.
$$
Finally, we obtain the following analogue of Wallis's formula
$$
\lim_{n\to\infty}\sum_{k=1}^{2n}(-1)^{k-1}\ln\frac{1-q^k}{1-q} ~=~ \ln\frac{(q;q)_{\infty}}{(q^2;q^2)^2_{\infty}}{\,}.
$$

\parag{Generalized Webster's functional equation} For any $m\in\N^*$ and any $a>0$, there is a unique solution $f\colon\R_+\to\R_+$ to the equation
$$
\prod_{j=0}^{m-1}f(x+a j) ~=~ \frac{1-q^x}{1-q}
$$
such that $\ln f$ lies in $\cK^0$ (or in $\cK^1$), namely
$$
f(x) ~=~ \frac{\Gamma_{q^{am}}(\frac{x+a}{am})}{\Gamma_{q^{am}}(\frac{x}{am})}\left(\frac{1-q^{am}}{1-q}\right)^{\frac{1}{m}}.
$$

\index{$q$-gamma function|)}

\section{The Barnes $G$-function}
\label{sec:Barnes558}\index{Barnes's $G$-function|(}

The Barnes function $G\colon\R_+\to\R_+$ is the function $G=1/\Gamma_2$ as defined in Section~\ref{subsec:MLG-t}. Hence, it can be defined by the equations
$$
\ln G(x) ~=~ \Sigma\ln\Gamma(x) ~=~ \Sigma\psi_{-1}(x)\qquad\text{for $x>0$}.
$$

\parag{ID card} We have the following basic information about the Barnes $G$-function:
$$
\begin{array}{|c|c|c|c|}
\hline g(x) & \text{Membership} & \deg g & \Sigma g(x) \\
\hline \emptybox \ln\Gamma(x) & \cC^{\infty}\cap\cD^2\cap\cK^{\infty} & 1 & \ln G(x)\\
\hline
\end{array}
$$

\parag{Analogue of Bohr-Mollerup's theorem} The function $G$ can be characterized in the multiplicative notation as follows.
\begin{quote}
\emph{All solutions $f\colon\R_+\to\R_+$ to the equation $f(x+1)=\Gamma(x) f(x)$ for which $\ln f$ lies in $\cK^2$ are of the form $f(x)=c{\,}G(x)$, where $c>0$}.
\end{quote}
Interestingly, this characterization enables one to establish the following identity
\begin{equation}\label{eq:convGPSI}
\ln G(x) ~=~ -\tchoose{x}{2}+(x-1)\ln\Gamma(x)+\textstyle{\frac{1}{2}}\ln(2\pi){\,}x-\psi_{-2}(x).
\end{equation}
Indeed, both sides vanish at $x=1$ and are eventually $2$-convex solutions to the equation
$$
f(x+1)-f(x) ~=~ \ln\Gamma(x).
$$
Hence, they must coincide on $\R_+$.

Using Proposition~\ref{prop:90unic41}, we can also derive the following alternative characterization of the Barnes $G$-function.
\begin{quote}
\emph{All solutions $f\colon\R_+\to\R_+$ to the equation $f(x+1)=\Gamma(x) f(x)$ that satisfy the asymptotic condition that, for each $x>0$,
$$
f(x+n) ~\sim ~ \Gamma(n)^x{\,}n^{{x\choose 2}} f(n)\qquad\text{as $n\to_{\N}\infty$}
$$
are of the form $f(x)=c{\,}G(x)$, where $c>0$}.
\end{quote}

\parag{Extended ID card} The value of the asymptotic constant $\sigma[g]$ can be derived for instance from identity \eqref{eq:convGPSI}. One can show that (see, e.g., \cite[p.~53]{SriCho12})
$$
\sigma[g] ~=~ \int_0^1\ln G(t+1){\,}dt ~=~ \frac{1}{12}+\frac{1}{4}\ln(2\pi)-2\ln A ~\approx ~ 0.045.
$$
We then have the following values:
$$
\begin{array}{|c|c|c|}
\hline \overline{\sigma}[g] & \sigma[g] & \gamma[g] \\
\hline \emptybox \frac{1}{12}-\frac{1}{4}\ln(2\pi)-2\ln A & \frac{1}{12}+\frac{1}{4}\ln(2\pi)-2\ln A & \gamma[g]=\sigma[g] \\
\hline
\end{array}
$$

\begin{itemize}
\item \emph{Inequality}
$$
|\sigma[g]| ~\leq ~ \frac{7}{3}\ln 2-\frac{109}{72} ~\approx ~ 0.10{\,}.
$$

\item \emph{Alternative representations of $\sigma[g]=\gamma[g]$}
\begin{eqnarray*}
\sigma[g] &=& \frac{1}{2}\ln(2\pi)+\lim_{n\to\infty}\left(\sum_{k=1}^n\ln\Gamma(k)
-\psi_{-2}(n)-\frac{1}{2}\ln\Gamma(n)-\frac{1}{12}\ln n\right), \\
\sigma[g] &=& \frac{1}{2}\ln(2\pi)+\lim_{n\to\infty}\left(\sum_{k=1}^n\ln\Gamma(k)
-\psi_{-2}(n)-\frac{1}{2}\ln\Gamma(n)-\frac{1}{12}\,\psi(n)\right), \\
\sigma[g] &=& \int_1^{\infty}\left(\ln\frac{\Gamma(\lfloor t\rfloor)}{\Gamma(t)}+\{t\}\ln\lfloor t\rfloor+{\{t\}\choose 2}\ln\left(1+\frac{1}{\lfloor t\rfloor}\right)\right)dt{\,},\\
\sigma[g] &=& \int_1^{\infty}\left(\ln\frac{\Gamma(\lfloor t\rfloor)}{\Gamma(t)}+\ln\frac{\lfloor t\rfloor^{7/12}}{\lfloor t+1\rfloor^{1/12}}\right)dt{\,},\\
\sigma[g] &=& \frac{1}{12}\,\gamma -\frac{1}{2}\,\int_1^{\infty}B_2(\{t\})\,\psi_1(t){\,}dt{\,},\\
\sigma[g] &=& \ln\left(\prod_{k=1}^{\infty}\frac{\Gamma(k){\,}e^k{\,}\sqrt{k}}{\left(1+\frac{1}{k}\right)^{\frac{1}{12}}k^k{\,}\sqrt{2\pi}}\right).
\end{eqnarray*}

\item \emph{Generalized Binet's function}. For any $q\in\N$ and any $x>0$
$$
J^{q+1}[\ln\circ G](x) ~=~ \ln G(x)-\psi_{-2}(x)-\overline{\sigma}[g]+\sum_{j=1}^qG_j\,\Delta^{j-1}\ln\Gamma(x).
$$
For instance,
$$
J^3[\ln\circ G](x) ~=~ \ln G(x)-\psi_{-2}(x)-\overline{\sigma}[g]+\frac{1}{2}\ln\Gamma(x)-\frac{1}{12}\ln x.
$$

\item \emph{Analogue of Raabe's formula}
\begin{equation}\label{eq:RG723}
\int_x^{x+1}\ln G(t){\,}dt ~=~ \overline{\sigma}[g]+\psi_{-2}(x){\,},\qquad x>0.
\end{equation}

\item \emph{Alternative characterization}. The function $f(x)=\ln G(x)$ is the unique solution lying in $\cC^0\cap\cK^2$ to the equation
$$
\int_x^{x+1}f(t){\,}dt ~=~ \overline{\sigma}[g]+\psi_{-2}(x){\,}, \qquad x>0.
$$
\end{itemize}

\begin{project}\label{appl:AntiDLogG4}
\textsl{Find a closed-form expression for the integral}
$$
\int_1^x\ln G(t){\,}dt.
$$
We apply Proposition~\ref{prop:an4tR8}. Using \eqref{eq:RG723} and then \eqref{eq:S2psi2} we obtain
\begin{eqnarray*}
\int_1^x\ln G(t){\,}dt &=& \Sigma_x\int_x^{x+1}\ln G(t){\,}dt ~=~ \overline{\sigma}[g]{\,}(x-1)+\Sigma\psi_{-2}(x)\\
&=& 2\ln A+\frac{1}{4}{\,}(x^2+1)\ln(2\pi)-\frac{1}{12}{\,}(2x+1)(x-1)^2\\
&& \null +(x-1)\,\psi_{-2}(x)-2\,\psi_{-3}(x){\,}.
\end{eqnarray*}
This expression could have been obtained also by integrating both sides of \eqref{eq:convGPSI}.
\end{project}

\parag{Inequalities} The following inequalities hold for any $x>0$, any $a\geq 0$, and any $n\in\N^*$.
\begin{itemize}
\item \emph{Symmetrized generalized Wendel's inequality} (equality if $a\in\{0,1,2\}$)
$$
\left|\ln G(x+a)-\ln G(x)-a\ln\Gamma(x)-\tchoose{a}{2}\ln x\right| ~\leq ~ \left|\tchoose{a-1}{2}\right|\,\ln\left(1+\frac{a}{x}\right),
$$
$$
\left(1+\frac{a}{x}\right)^{-\left|{a-1\choose 2}\right|} ~\leq ~ \frac{G(x+a)}{G(x)\,\Gamma(x)^a{\,}x^{a\choose 2}} ~\leq ~ \left(1+\frac{a}{x}\right)^{\left|{a-1\choose 2}\right|}.
$$

\item \emph{Symmetrized generalized Wendel's inequality} (discrete version)
\begin{multline*}
\left|\ln G(x)-\sum_{k=1}^{n-1}\ln\Gamma(k)+\sum_{k=0}^{n-1}\ln\Gamma(x+k)-x\ln\Gamma(n)-\tchoose{x}{2}\ln n\right|\\
\leq ~ \left|\tchoose{x-1}{2}\right|\,\ln\left(1+\frac{x}{n}\right),
\end{multline*}
$$
\left(1+\frac{x}{n}\right)^{-\left|{x-1\choose 2}\right|} ~\leq ~ G(x){\,}\frac{\Gamma(x)\Gamma(x+1)\cdots\Gamma(x+n-1)}{\Gamma(1)\Gamma(2)\cdots\Gamma(n-1)\Gamma(n)^x n^{{x\choose 2}}}~\leq ~ \left(1+\frac{x}{n}\right)^{\left|{x-1\choose 2}\right|}.
$$

\item \emph{Symmetrized Stirling's formula-based inequality}
\begin{eqnarray*}
\left|J^3[\ln\circ G](x)\right|
&\leq & \frac{1}{12}(x+1)^2(2x+5)\ln\left(1+\frac{1}{x}\right)-\frac{1}{72}(12x^2+48x+49)\\
&\leq & \frac{5}{12}\ln\left(1+\frac{1}{x}\right),
\end{eqnarray*}
$$
\left(1+\frac{1}{x}\right)^{-5/12} \leq ~ \frac{G(x)\,\Gamma(x)^{1/2}}{x^{1/12}{\,}e^{\psi_{-2}(x)+\overline{\sigma}[g]}} ~\leq ~ \left(1+\frac{1}{x}\right)^{5/12}.
$$

\item \emph{Generalized Gautschi's inequality}
$$
\Gamma(x+\lceil a\rceil)^{a-\lceil a\rceil} ~\leq ~ e^{(a-\lceil a\rceil)D\ln G(x+\lceil a\rceil)} ~\leq ~\frac{G(x+a)}{G(x+\lceil a\rceil)} ~\leq ~ \Gamma(x+\lceil a\rceil)^{a-\lfloor a\rfloor}.
$$
(These inequalities are valid only if $x+\lfloor a\rfloor\geq x_0$, where $x_0=1.92\ldots$ is the unique positive zero of the function $D^2\ln G(x)$.)
\end{itemize}

\begin{remark}
It is not difficult to see that the first inequality in Proposition~\ref{prop:Burnside0} does not hold for large values of $x$ when $g(x)=\ln\Gamma(x)$. This shows that the analogue of Burnside's formula does not hold in general when $\deg g\geq 1$.
\end{remark}

\parag{Generalized Stirling's and related formulas} For any $a\geq 0$, we have the following limits and asymptotic equivalences as $x\to\infty$,
$$
\ln G(x+a)-\ln G(x)-a\ln\Gamma(x)-\tchoose{a}{2}\ln x ~\to ~0,
$$
$$
\ln G(x)-\psi_{-2}(x)+\frac{1}{2}\ln\Gamma(x)-\frac{1}{12}\ln x ~\to ~ \overline{\sigma}[g],
$$
$$
\ln G(x)-\psi_{-2}(x)+\frac{1}{2}\ln\Gamma(x)-\frac{1}{12}\,\psi(x) ~\to ~ \overline{\sigma}[g],
$$
$$
G(x+a) ~\sim ~ G(x)\,\Gamma(x)^a{\,}x^{{a\choose 2}},\qquad \ln G(x+a) ~\sim ~ \psi_{-2}(x),
$$
$$
G(x) ~\sim ~ \exp(\psi_{-2}(x)+\overline{\sigma}[g]){\,}\Gamma(x)^{-\frac{1}{2}}{\,}x^{\frac{1}{12}}.
$$
\emph{Further results} (obtained by differentiation)
$$
x\,\psi(x+a)-x\,\psi(x) ~\to ~a,\quad x{\,}\psi_1(x) ~\to ~1,\quad x\,\psi(x+a) ~\sim ~ \ln\Gamma(x),
$$
$$
\ln\Gamma(x)-\left(x-\frac{1}{2}\right)\psi(x)+x ~\to ~ \frac{1}{2}{\,}(1+\ln(2\pi)).
$$

\begin{remark}
Using one of the asymptotic equivalences above, we get
$$
G(x+1) ~\sim ~ \exp(\psi_{-2}(x)+\overline{\sigma}[g]){\,}\Gamma(x)^{\frac{1}{2}}{\,}x^{\frac{1}{12}}\qquad\text{as $x\to\infty$}.
$$
Combining this latter equivalence with identity \eqref{eq:convGPSI} and the Stirling formula for the gamma function, we also obtain the following simpler form
$$
G(x+1) ~\sim ~ A^{-1}{\,}x^{\frac{1}{2}x^2-\frac{1}{12}}(2\pi)^{\frac{x}{2}}{\,}e^{-\frac{3}{4}x^2+\frac{1}{12}}\qquad\text{as $x\to\infty$}.\qedhere
$$
\end{remark}

\parag{Asymptotic expansions} For any $m,q\in\N^*$ we have the following expansion as $x\to\infty$
\begin{equation}\label{eq:10Asym54GBa2}
\frac{1}{m}\sum_{j=0}^{m-1}\ln G\left(x+\frac{j}{m}\right) ~=~ \overline{\sigma}[g]+ \sum_{k=0}^q\frac{B_k}{m^k{\,}k!}{\,}\psi_{k-2}(x)+O(\psi_{q-1}(x)){\,}.
\end{equation}
Setting $m=1$ in this formula, we obtain
$$
\ln G(x) ~=~ \overline{\sigma}[g]+ \sum_{k=0}^q\frac{B_k}{k!}{\,}\psi_{k-2}(x)+O(\psi_{q-1}(x)){\,},
$$
or equivalently, if $q\geq 2$,
$$
J^3[\ln\circ G](x) ~=~ \frac{1}{12}\left(\psi(x)-\ln x\right)+ \sum_{k=3}^q\frac{B_k}{k!}{\,}\psi_{k-2}(x)+O(\psi_{q-1}(x)){\,}.
$$
Setting $q=4$ for instance, we obtain the following expansion
$$
\ln G(x) ~=~ \overline{\sigma}[g]+\psi_{-2}(x)-\frac{1}{2}\,\psi_{-1}(x)+\frac{1}{12}\,\psi(x)
-\frac{1}{720}\,\psi_2(x)+O\left(x^{-4}\right).
$$

\parag{Generalized Liu's formula} For any $x>0$ we have
$$
\ln G(x) ~=~ \overline{\sigma}[g]+\psi_{-2}(x)-\frac{1}{2}\,\psi_{-1}(x)
+\frac{1}{12}\,\psi(x)+\frac{1}{2}\,\int_0^{\infty}B_2(\{t\})\,\psi_1(x+t){\,}dt
$$
or equivalently,
$$
J^3[\ln\circ G](x) ~=~ \frac{1}{12}\left(\psi(x)-\ln x\right)+\frac{1}{2}\,\int_0^{\infty}B_2(\{t\})\,\psi_1(x+t){\,}dt.
$$

\parag{Limit, series, and integral representations} Let us now determine the main representations of the function $\ln G(x)$.
\begin{itemize}
\item \emph{Eulerian form and related identities}. We have
$$
\ln G(x) ~=~ -\ln\Gamma(x)-\sum_{k=1}^{\infty}\left(\ln\Gamma(x+k)-\ln\Gamma(k)-x\ln k-\tchoose{x}{2}\ln\left(1+\frac{1}{k}\right)\right),
$$
$$
G(x) ~=~ \frac{1}{\Gamma(x)}{\,}
\prod_{k=1}^{\infty}\frac{\Gamma(k)}{\Gamma(x+k)}{\,}k^x(1+1/k)^{{x\choose 2}}.
$$
Upon differentiation, we obtain
$$
x\,\psi(x) ~=~ x-\frac{1}{2}{\,}(1+\ln(2\pi))-\sum_{k=1}^{\infty}\left(\psi(x+k)-\ln k-\left(x-\frac{1}{2}\right)\ln\left(1+\frac{1}{k}\right)\right),
$$
$$
\psi(x)+x\,\psi_1(x) ~=~ 1-\sum_{k=1}^{\infty}\left(\psi_1(x+k)-\ln\left(1+\frac{1}{k}\right)\right),
$$
$$
(r+1)\,\psi_r(x)+x\,\psi_{r+1}(x) ~=~ -\sum_{k=1}^{\infty}\psi_{r+1}(x+k),\qquad r\in\N^*.
$$

\item \emph{Weierstrassian form and related identities}. We have
\begin{multline*}
\ln G(x) ~=~ (-1-\gamma)\tchoose{x}{2}-\ln\Gamma(x)\\
\null - \sum_{k=1}^{\infty}\left(\ln\Gamma(x+k)-\ln\Gamma(k)-x\ln k-\tchoose{x}{2}\,\psi_1(k)\right),
\end{multline*}
$$
G(x) ~=~ \frac{e^{(-\gamma -1){x\choose 2}}}{\Gamma(x)}{\,}
\prod_{k=1}^{\infty}\frac{\Gamma(k)}{\Gamma(x+k)}{\,}k^xe^{\psi_{1}(k){\,}{x\choose 2}}{\,},
$$
Upon differentiation, we obtain
$$
x\,\psi(x) +\left(x-\frac{1}{2}\right)\gamma +\frac{1}{2}\ln(2\pi) ~=~ -\sum_{k=1}^{\infty}\left(\psi(x+k)-\left(x-\frac{1}{2}\right)\psi_1(k)-\ln k\right),
$$
$$
\psi(x)+x\,\psi_1(x)+\gamma ~=~ -\sum_{k=1}^{\infty}\left(\psi_1(x+k)-\psi_1(k)\right).
$$

\item \emph{Analogue of Gauss' limit and related identities}. The analogue of Gauss' limit is
$$
\ln G(x) ~=~ \lim_{n\to\infty}\left(\sum_{k=1}^{n-1}\ln\Gamma(k)-\sum_{k=0}^{n-1}\ln\Gamma(x+k)+x\ln\Gamma(n)+\tchoose{x}{2}\ln n\right),
$$
$$
G(x) ~=~ \lim_{n\to\infty}\frac{\Gamma(1)\Gamma(2){\,}\cdots{\,}\Gamma(n)}{\Gamma(x)\Gamma(x+1){\,}\cdots{\,}\Gamma(x+n)}{\,}
n!^x{\,}n^{{x\choose 2}}{\,}.
$$
Upon differentiation, we obtain
\begin{multline*}
(x-1)\,\psi(x)-x+\frac{1}{2}(1+\ln(2\pi))\\
=~ \lim_{n\to\infty}\left(-\sum_{k=0}^{n-1}\psi(x+k)+\ln\Gamma(n)+\left(x-\frac{1}{2}\right)\ln n\right),
\end{multline*}
$$
(x-1)\,\psi_1(x)+\psi(x)-1 ~=~ \lim_{n\to\infty}\left(\ln n-\sum_{k=0}^{n-1}\psi_1(x+t)\right).
$$

\item \emph{Integral representations}. Using the elevator method\index{elevator method} on one and two levels, we obtain the following representations
$$
\ln G(x) ~=~ -\frac{1}{2}{\,}(x-1)(x-\ln(2\pi))+\int_1^x(t-1)\,\psi(t){\,}dt
$$
and
$$
\ln G(x) ~=~ -\frac{1}{2}{\,}(x-1)(x-\ln(2\pi))+\int_1^x(x-t)(\psi(t)+(t-1)\,\psi_1(t)){\,}dt.
$$
Each of these representations actually leads to identity \eqref{eq:convGPSI}.

\item \emph{Gregory's formula-based series representation}. For any $x>0$ we have the series
representation
\begin{eqnarray*}
\ln G(x) &=& \psi_{-2}(x)+\overline{\sigma}[g]-\frac{1}{2}\ln\Gamma(x)-\sum_{n=0}^{\infty}G_{n+2}\Delta^{n+1}g(x)\\
\\
&=& \psi_{-2}(x)+\overline{\sigma}[g]-\frac{1}{2}\ln\Gamma(x)
-\sum_{n=0}^{\infty}|G_{n+2}|\,\sum_{k=0}^n(-1)^k\tchoose{n}{k}{\,}\ln(x+k).
\end{eqnarray*}
Setting $x=1$ in this identity yields the analogue of Fontana-Mascheroni series
$$
\overline{\sigma}[g] ~=~ -\frac{1}{2}\ln(2\pi)+\sum_{n=0}^{\infty}|G_{n+2}|\,\sum_{k=0}^n(-1)^k\tchoose{n}{k}{\,}\ln(k+1).
$$
\end{itemize}

Note that the Eulerian and Weierstrassian forms above can also be integrated term by term on any bounded interval of $[0,\infty)$. For instance, integrating on $(1,x)$ provides series representations for the integral of $\ln G(x)$ as defined in Project~\ref{appl:AntiDLogG4}.

\parag{Analogue of Gauss' multiplication formula} For any $m\in\N^*$ and any $x>0$, we have
$$
\sum_{j=0}^{m-1}\ln G\left(\frac{x+j}{m}\right) ~=~ \sum_{j=1}^m\ln G\left(\frac{j}{m}\right)+\Sigma_x\ln\Gamma\left(\frac{x}{m}\right).
$$
For instance, setting $m=2$ in this identity, we obtain
$$
\ln G\left(\frac{x+1}{2}\right)+\ln G\left(\frac{x}{2}\right) ~=~ \ln G\left(\frac{1}{2}\right) + \Sigma_x\ln\Gamma\left(\frac{x}{2}\right).
$$
However, to make this multiplication formula interesting and usable, we need to find a simple expression for its right side. In particular, we need a closed-form expression for the function $\Sigma_x\ln\Gamma(\frac{x}{m})$. Such a result would be most welcome.

We can nevertheless investigate the asymptotic behavior of the function
$$
x ~\mapsto ~ \sum_{j=0}^{m-1}\ln G\left(\frac{x+j}{m}\right).
$$
In addition to the asymptotic expansion given in \eqref{eq:10Asym54GBa2}, Proposition~\ref{prop:8Stir44Gau7Mult} yields the following convergence result. We have
\begin{multline*}
\sum_{j=0}^{m-1}\ln G\left(\frac{x+j}{m}\right)-m\,\psi_{-2}\left(\frac{x}{m}\right)+\frac{1}{2}\,\ln\Gamma\left(\frac{x}{m}\right)\\
-\frac{1}{12}\left(\ln\Gamma\left(\frac{x+1}{m}\right)-\ln\Gamma\left(\frac{x}{m}\right)\right) ~\to ~ m\,\overline{\sigma}[g]\qquad\text{as $x\to\infty$}.
\end{multline*}

\parag{Analogue of Wallis's product formula} Using Legendre's duplication formula for the gamma function, we obtain
\begin{eqnarray*}
\Sigma_x\ln\Gamma(2x) &=& \textstyle{\ln G(x)+\ln G(x+\frac{1}{2})-\ln G(\frac{1}{2})}\\
&& \null +\textstyle{(x^2+1)\ln 2-\frac{x}{2}\ln(16\pi)}.
\end{eqnarray*}
Using this identity with Proposition~\ref{prop:Wallis92}, we can derive the surprising analogue of Wallis's formula
$$
\lim_{n\to\infty}\frac{\Gamma(1)\Gamma(3){\,}\cdots{\,}\Gamma(2n-1)}{\Gamma(2)\Gamma(4){\,}\cdots{\,}\Gamma(2n)}\left(\frac{2n}{e}\right)^n ~=~ \frac{1}{\sqrt{2}}{\,}.
$$
Note that a shorter proof of this formula can be obtained using the second sequence described in Remark~\ref{rem:Wall38}.

\begin{project}
\textsl{Find the analogue of Wallis's formula for the function $g(x)=\ln G(x)$.} After some algebra, we obtain
$$
\lim_{n\to\infty}\frac{G(1)G(3){\,}\cdots{\,}G(2n-1)}{G(2)G(4){\,}\cdots{\,}G(2n)}\,\frac{n^{n^2-\frac{1}{2}n
-\frac{1}{24}}{\,}2^{n^2-\frac{7}{24}}\,\pi^{\frac{1}{2}n}}{e^{\frac{3}{2}n^2-\frac{1}{2}n-\frac{1}{24}}} ~=~ A^{\frac{1}{2}}{\,}.
$$
This latter formula is a little harder to obtain than the former one. Using Proposition~\ref{prop:Wallis92} requires the computation of both functions $\Sigma\ln G(x)$ and $2\,\Sigma_x\ln G(2x)$ using the elevator method\index{elevator method} (Corollary~\ref{cor:saf6f}) with $r=1$. That is,
\begin{eqnarray*}
\Sigma\ln G(x) &=& -\frac{1}{8}{\,}x(x-1)(2x-5)+\frac{1}{4}{\,}x(x-3)\ln(2\pi)-x\ln A\\
&& \null + \frac{1}{2}{\,}(x-1)(x-2)\ln\Gamma(x)-\frac{1}{2}{\,}(2x-3)\,\psi_{-2}(x)+\psi_{-3}(x)
\end{eqnarray*}
and
\begin{eqnarray*}
2\,\Sigma_x\ln G(2x) &=& -\frac{1}{4}{\,}x(2x-1)(4x-7)-2x\ln A\\
&& \null +\frac{1}{2}{\,}(2x^2-3x-1)\ln 2 +x(x-2)\ln\pi\\
&& \null +\frac{1}{2}\ln\Gamma(x)+\frac{1}{2}(2x-1)(2x-3)\ln\Gamma(2x)\\
&& \null -2(x-1)\,\psi_{-2}(2x)+\psi_{-3}(2x).
\end{eqnarray*}
Here again, a shorter proof of the limit above can be obtained using the second sequence described in Remark~\ref{rem:Wall38}.
\end{project}

\parag{Restriction to the natural integers} For any $n\in\N^*$ we have
$$
G(n) ~=~ \prod_{k=0}^{n-2}k!{\,}.
$$

\parag{Generalized Webster's functional equation} For any $m\in\N^*$, there is a unique solution $f\colon\R_+\to\R_+$ to the equation
$$
\prod_{j=0}^{m-1}f\left(x+\frac{j}{m}\right) ~=~ \Gamma(x)
$$
such that $\ln f$ lies in $\cK^1$, namely
$$
f(x) ~=~ \frac{G(x+\frac{1}{m})}{G(x)}{\,}.
$$

\parag{Analogue of Euler's series representation of $\gamma$}
The Taylor series expansion of $\ln G(x+1)$ about $x=0$ is (see, e.g., \cite[p.~311]{SriCho12})
$$
\ln G(x+1) ~=~\frac{1}{2}\left(\ln(2\pi)-1\right)x-\frac{\gamma +1}{2}{\,}x^2 -\sum_{k=2}^{\infty}\,\frac{\zeta(k)}{k+1}{\,}(-x)^{k+1}{\,},\qquad |x|<1.
$$
Integrating both sides of this equation on $(0,1)$, we obtain the identity
$$
\sum_{k=2}^{\infty}(-1)^k\,\frac{\zeta(k)}{(k+1)(k+2)} ~=~ \frac{1}{2}+\frac{1}{6}{\,}\gamma -2\ln A.
$$
Also, the exponential generating function for the sequence $n\mapsto\sigma[g^{(n)}]$ is
$$
\mathrm{egf}_{\sigma}[g](x) ~=~ \ln G(x+1)-\psi_{-2}(x+1)+\frac{1}{4}\ln(2\pi)-\frac{1}{12}+2\ln A
$$
Integrating both sides of this equation on $(0,1)$ (i.e., we use \eqref{eq:EulerAnal55}), after some algebra we obtain
$$
\sum_{k=2}^{\infty}(-1)^k\,\frac{k-1}{k(k+1)(k+2)}\,\zeta(k) ~=~ \frac{5}{4}-3\ln A-\frac{1}{4}\ln(2\pi){\,}.
$$

\parag{Analogue of the reflection formula} A reflection formula for the Barnes $G$-function is given in \eqref{eq:ReflBGF5}; see, e.g., \cite[p.~45]{SriCho12}.

\index{Barnes's $G$-function|)}

\section{The Hurwitz zeta function}
\label{sec:Hurw49}\index{Hurwitz zeta function|(}

For any $x>0$, the Hurwitz zeta function $s\mapsto\zeta(s,x)$ is defined as an analytic continuation to $\mathbb{C}\setminus\{1\}$ of the series (see, e.g., \cite[p.~155]{SriCho12})
$$
\sum_{k=0}^{\infty}(x+k)^{-s} ~=~ \frac{1}{\Gamma(s)}\,\int_0^{\infty}\frac{t^{s-1}e^{-xt}}{1-e^{-t}}{\,}dt,\qquad\Re(s)>1.
$$
It is known (see, e.g., \cite[p.\ 159--160]{SriCho12}) that this function satisfies the identity
$$
D^k_x\zeta(s,x) ~=~ (-s)^{\underline{k}}\,\zeta(s+k,x){\,},\qquad k\in\N,
$$
and the difference equation
\begin{equation}\label{eq:HuZeRec5}
\zeta(s,x+1)-\zeta(s,x) ~=~ -x^{-s}.
\end{equation}
For any fixed $s\in\R\setminus\{1\}$, define the function $g_s\colon\R_+\to\R$ by the equation
$$
g_s(x) ~=~ -x^{-s}\qquad\text{for $x>0$}.
$$
We then have $g_s\in\cC^{\infty}\cap\cK^{\infty}$. If $s>0$ and $s\neq 1$, then $g_s\in\cD^0_{\N}$. If $s>1$, then $g_s\in\cD^{-1}_{\N}$. If $-p<s<1$ for some $p\in\N$, then $g_s\in\cD^p_{\N}$, and hence we can consider
$$
p ~=~ 1+\deg g_s ~=~ \lfloor 1-s\rfloor.
$$
In all cases, we have
$$
\Sigma g_s(x) ~=~ \zeta(s,x)-\zeta(s),
$$
where $s\mapsto\zeta(s)=\zeta(s,1)$ is the Riemann zeta function.\index{Riemann zeta function}

\parag{ID card} The basic information about the Hurwitz zeta function is summarized in the following table.
$$
\begin{array}{|c|c|c|c|}
\hline g_s(x) & \text{Membership} & \deg g_s & \Sigma g_s(x) \\
\hline \emptybox -x^{-s} & \begin{array}{rl}\emptybox\cC^{\infty}\cap\cD^{-1}\cap\cK^{\infty},&\text{if $s>1$},\\ \cC^{\infty}\cap\cD^{\lfloor 1-s\rfloor}\cap\cK^{\infty},&\text{if $s<1$}.\end{array} & -1+\lfloor 1-s\rfloor_+ & \zeta(s,x)-\zeta(s) \\
\hline
\end{array}
$$

\begin{project}
\textsl{Find a closed-form expression for $\Sigma g$, where}
$$
g(x) ~=~ \frac{x^2}{\sqrt{x+1}}{\,}.
$$
Expanding $x^2=(x+1-1)^2$, we obtain
$$
g(x) ~=~ (x+1)^{\frac{3}{2}}-2(x+1)^{\frac{1}{2}}+(x+1)^{-\frac{1}{2}}
$$
and hence
$$
\Sigma g(x) ~=~ \textstyle{c-\zeta(-\frac{3}{2},x+1)+2\zeta(-\frac{1}{2},x+1)-\zeta(\frac{1}{2},x+1)}
$$
for some $c\in\R$.
\end{project}

\parag{Analogue of Bohr-Mollerup's theorem} The function $\zeta(s,x)$ can be characterized as follows.
\begin{quote}
\emph{All solutions $f_s\colon\R_+\to\R$ to the equation
$$
f_s(x+1)-f_s(x) ~=~ -x^{-s}
$$
that lie in $\cK^{\lfloor 1-s\rfloor_+}$ are of the form $f_s(x)=c_s+\zeta(s,x)$, where $c_s\in\R$.}
\end{quote}

\parag{Extended ID card} The asymptotic constant $\sigma[g_s]$ satisfies the following identity
$$
\sigma[g_s] ~=~ \int_0^1\zeta(s,t+1){\,}dt-\zeta(s) ~=~ \frac{1}{s-1}-\zeta(s).
$$
Hence we have the following values
$$
\begin{array}{|c|c|c|}
\hline \overline{\sigma}[g_s] & \sigma[g_s] & \gamma[g_s] \\
\hline \emptybox \begin{array}{rl}\infty,&\text{if $s>1$},\\ -\zeta(s),&\text{if $s<1$}.\end{array} & \frac{1}{s-1}-\zeta(s) & \sigma[g_s]-\sum_{j=1}^{\lfloor 1-s\rfloor_+}G_j\,\Delta^{j-1}g_s(1) \\
\hline
\end{array}
$$
We also have the following identities.
\begin{itemize}
\item \emph{Alternative representations of $\sigma[g_s]$}
\begin{eqnarray*}
\sigma[g_s] &=& \lim_{n\to\infty}\left(\frac{1-n^{1-s}}{s-1}-\sum_{k=1}^{n-1}k^{-s}+\sum_{j=1}^{\lfloor 1-s\rfloor_+}G_j\,\Delta^{j-1}g_s(n)\right), \\
\sigma[g_s] &=& \lim_{n\to\infty}\left(\frac{1}{s-1}-\sum_{k=1}^{n-1}k^{-s}+\frac{1}{1-s}\,\sum_{j=0}^{\lfloor 1-s\rfloor_+}\tchoose{1-s}{j}{\,}\frac{B_j}{n^{s+j-1}}\right), \\
\sigma[g_s] &=& \sum_{j=1}^{\lfloor 1-s\rfloor_+}G_j\,\Delta^{j-1}g_s(1)\\
&& \null +\sum_{k=1}^{\infty}\left(\frac{k^{1-s}-(k+1)^{1-s}}{s-1}+\sum_{j=0}^{\lfloor 1-s\rfloor_+}G_j\,\Delta^jg_s(k)\right).
\end{eqnarray*}
If $s>-1$, then
$$
\sigma[g_s] ~=~ -\frac{1}{2}+s\int_1^{\infty}\frac{\{t\}-\frac{1}{2}}{t^{s+1}}{\,}dt.
$$
If $s\leq -1$, then for any integer $q\geq\lceil(1-s)/2\rceil$,
$$
\sigma[g_s] ~=~ -\frac{1}{2}+\sum_{k=1}^q\frac{B_{2k}}{(2k)!}{\,}(-s)^{\underline{2k-1}}
+\frac{(-s)^{\underline{2q}}}{(2q)!}\,\int_1^{\infty}\frac{B_{2q}(\{t\})}{t^{s+2q}}{\,}dt.
$$

\item \emph{Representations of $\gamma[g_s]$}
\begin{eqnarray*}
\gamma[g_s] &=& \sigma[g_s]-\sum_{j=1}^{\lfloor 1-s\rfloor_+}G_j\,\Delta^{j-1}g_s(1){\,},\\
\gamma[g_s] &=& \int_1^{\infty}\bigg(\sum_{j=0}^{\lfloor 1-s\rfloor_+}G_j\,\Delta^jg_s(\lfloor t\rfloor)-g_s(t)\bigg){\,}dt{\,},\\
\gamma[g_s] &=& \int_1^{\infty}\bigg(\sum_{j=0}^{\lfloor 1-s\rfloor_+}{\{t\}\choose j}\,\Delta^jg_s(\lfloor t\rfloor)-g_s(t)\bigg){\,}dt{\,}.
\end{eqnarray*}

\item \emph{Generalized Binet's function}. For any $q\in\N$ and any $x>0$
$$
J^{q+1}[\Sigma g_s](x) ~=~ \zeta(s,x)-\frac{x^{1-s}}{s-1}+\sum_{j=1}^qG_j\,\Delta^{j-1}g_s(x).
$$

\item \emph{Analogue of Raabe's formula}
$$
\int_x^{x+1}\zeta(s,t){\,}dt ~=~ \frac{x^{1-s}}{s-1}{\,},\qquad x>0.
$$

\item \emph{Alternative characterization}. The function $f_s(x)=\zeta(s,x)$ is the unique solution lying in $\cC^0\cap\cK^{\lfloor 1-s\rfloor_+}$ to the equation
$$
\int_x^{x+1}f_s(t){\,}dt ~=~ \frac{x^{1-s}}{s-1}{\,}, \qquad x>0.
$$
\end{itemize}

\parag{Inequalities} The following inequalities hold for any $x>0$, any $a>0$, and any $n\in\N^*$.
\begin{itemize}
\item \emph{Symmetrized generalized Wendel's inequality} (equality if $a\in\{0,1,\ldots,\lfloor 1-s\rfloor_+\}$)
$$
\left|\zeta(s,x+a)-\zeta(s,x)-\sum_{j=1}^{\lfloor 1-s\rfloor_+}\tchoose{a}{j}\,\Delta^{j-1}g_s(x)\right| ~\leq ~ \lceil a\rceil\left|\tchoose{a-1}{\lfloor 1-s\rfloor_+}\right|\left|\Delta^{\lfloor 1-s\rfloor_+}g_s(x)\right|.
$$
If $s\leq 0$, then
\begin{eqnarray*}
\lefteqn{\left|\zeta(s,x+a)-\zeta(s,x)-\sum_{j=1}^{\lfloor 1-s\rfloor}\tchoose{a}{j}\,\Delta^{j-1}g_s(x)\right|}\\
&\leq &  \left|\tchoose{a-1}{\lfloor 1-s\rfloor}\right|\left|\Delta^{\lfloor -s\rfloor}g_s(x+a)-\Delta^{\lfloor -s\rfloor}g_s(x)\right|.
\end{eqnarray*}

\item \emph{Symmetrized generalized Wendel's inequality} (discrete version)
$$
\left|\zeta(s,x)-\zeta(s)-f_n^{\lfloor 1-s\rfloor_+}[g_s](x)\right| ~\leq ~ \lceil x\rceil\left|\tchoose{x-1}{\lfloor 1-s\rfloor_+}\right|\left|\Delta^{\lfloor 1-s\rfloor_+}g_s(n)\right|.
$$
If $s\leq 0$, then
$$
\left|\zeta(s,x)-\zeta(s)-f_n^{\lfloor 1-s\rfloor}[g_s](x)\right| ~\leq ~ \left|\tchoose{x-1}{\lfloor 1-s\rfloor}\right|\left|\Delta^{\lfloor -s\rfloor}g_s(x+n)-\Delta^{\lfloor -s\rfloor}g_s(n)\right|.
$$
Here
$$
f^{\lfloor 1-s\rfloor_+}_n[g_s](x) ~=~ \sum_{k=0}^{n-1}(x+k)^{-s}-\sum_{k=1}^{n-1}k^{-s}
-\sum_{j=1}^{\lfloor 1-s\rfloor_+}\tchoose{x}{j}\,\Delta_n^{j-1}n^{-s}.
$$

\item \emph{Symmetrized Stirling's formula-based inequality}
$$
\left|J^{\lfloor 1-s\rfloor_++1}[\Sigma g_s](x)\right| ~\leq ~ \overline{G}_{\lfloor 1-s\rfloor_+}\left|\Delta^{\lfloor 1-s\rfloor_+}g_s(x)\right|.
$$
If $s\leq 0$, then
$$
\left|J^{\lfloor 2-s\rfloor}[\Sigma g_s](x)\right| ~\leq ~ \int_0^1\left|\tchoose{t-1}{\lfloor 1-s\rfloor}\right|\left|\Delta^{\lfloor -s\rfloor}g_s(x+t)-\Delta^{\lfloor -s\rfloor}g_s(x)\right|dt.
$$

\item \emph{Burnside's formula-based inequality if $s> -1$}
$$
\left|\zeta\left(s,x+\frac{1}{2}\right)-\frac{x^{1-s}}{s-1}\right| ~\leq ~ \left|J^{\lfloor 1-s\rfloor_++1}[\Sigma g_s](x)\right|.
$$

\item \emph{Additional inequality if $s>1$}.
$$
0 ~\leq ~ \zeta(s,x+n) ~=~ \sum_{k=n}^{\infty}(x+k)^{-s} ~\leq ~ \zeta(s,n).
$$

\item \emph{Generalized Gautschi's inequality}

If $s\geq 0$, $s\neq 1$,
\begin{eqnarray*}
(\lceil a\rceil -a)(x+\lceil a\rceil)^{-s} &\leq & s(\lceil a\rceil -a)\,\zeta(s+1,x+\lceil a\rceil)\\
&\leq & \zeta(s,x+a)-\zeta(s,x+\lceil a\rceil) ~\leq ~ (\lceil a\rceil -a)(x+\lfloor a\rfloor)^{-s}.
\end{eqnarray*}
If $s\leq 0$, then these inequalities must be reversed and they are valid only if the Hurwitz zeta function is concave on $[x+\lfloor a\rfloor,\infty)$.
\end{itemize}

\parag{Generalized Stirling's and related formulas} For any $a\geq 0$, we have the following limits and asymptotic equivalences as $x\to\infty$,
$$
\zeta(s,x+a)-\zeta(s,x)-\sum_{j=1}^{\lfloor 1-s\rfloor_+}\tchoose{a}{j}\,\Delta^{j-1}g_s(x) ~\to ~0,
$$
$$
\zeta(s,x)-\frac{x^{1-s}}{s-1}+\sum_{j=1}^{\lfloor 1-s\rfloor_+}G_j\,\Delta^{j-1}g_s(x) ~\to ~0,
$$
$$
\zeta(s,x)+\frac{1}{1-s}\sum_{j=0}^{\lfloor 1-s\rfloor_+}\tchoose{1-s}{j}\,\frac{B_j}{x^{s+j-1}} ~\to ~0,
$$
$$
\zeta(s,x+a) ~\sim ~ \frac{x^{1-s}}{s-1}{\,}.
$$
In particular, if $s>1$, then $\zeta(s,x)\to 0$ as $x\to\infty$.

For instance, setting $s=-\frac{3}{2}$ in these latter two asymptotic formulas, we obtain
\begin{eqnarray*}
\textstyle{\zeta\left(-\frac{3}{2},x\right)+\frac{2}{5}{\,}x^{5/2}-\frac{7}{12}{\,}x^{3/2}
+\frac{1}{12}{\,}(x+1)^{3/2}} &\to & 0{\,},\\
\textstyle{\zeta\left(-\frac{3}{2},x\right)+\frac{2}{5}{\,}x^{5/2}-\frac{1}{2}{\,}x^{3/2}+\frac{1}{8}{\,}x^{1/2}} &\to & 0{\,}.
\end{eqnarray*}

If $s>-1$, then we have the analogue of Burnside's formula
$$
\textstyle{\zeta(s,x)-\frac{1}{s-1}{\,}(x-\frac{1}{2})^{1-s}} ~\to ~ 0{\,},\qquad\text{as $x\to\infty$},
$$
which provides a better approximation of $\zeta(s,x)$ than the generalized Stirling formula.

\parag{Asymptotic expansions} For any $m, q\in\N^*$ we have the following expansion as $x\to\infty$
$$
\frac{1}{m}\sum_{j=0}^{m-1}\zeta\left(s,x+\frac{j}{m}\right) ~=~ \frac{1}{s-1}\,\sum_{k=0}^q\tchoose{1-s}{k}\frac{B_k}{m^k{\,}x^{s+k-1}}+O\left(x^{-q-s}\right).
$$
Setting $m=1$ in this formula, we obtain
$$
\zeta(s,x) ~=~ \frac{1}{s-1}\,\sum_{k=0}^q\tchoose{1-s}{k}\frac{B_k}{x^{s+k-1}}+O\left(x^{-q-s}\right).
$$
In particular, this clearly shows that $\zeta(s,x)$ is a $(1-s)$-degree polynomial whenever $1-s$ is a positive integer. More precisely, we have
$$
\zeta(1-n,x) ~=~ -\frac{1}{n}{\,}\sum_{k=0}^n\tchoose{n}{k}{\,}B_k{\,}x^{n-k},\qquad n\in\N^*,
$$
that is,
\begin{equation}\label{eq:HurwBern43}
\zeta(1-n,x)  ~=~ -\frac{1}{n}{\,}B_n(x),\qquad n\in\N^*.
\end{equation}

\parag{Generalized Liu's formula} We have the following formulas for $x>0$.
\begin{itemize}
\item If $s>-1$, then
$$
\zeta(s,x) ~=~ \frac{x^{1-s}}{s-1}+\frac{1}{2}{\,}x^{-s}-s{\,}\int_0^{\infty}\frac{\{t\}-\frac{1}{2}}{(x+t)^{s+1}}{\,}dt.
$$
\item If $s\leq -1$, then for any integer $q\geq\lceil(1-s)/2\rceil$,
$$
\zeta(s,x) ~=~ \frac{x^{1-s}}{s-1}+\frac{1}{2}{\,}x^{-s}-\sum_{k=1}^q\frac{B_{2k}}{(2k)!}{\,}\frac{(-s)^{\underline{2k-1}}}{x^{s+2k-1}}
-\frac{(-s)^{\underline{2q}}}{(2q)!}\,\int_0^{\infty}\frac{B_{2q}(\{t\})}{(x+t)^{s+2q}}{\,}dt.
$$
\end{itemize}

\parag{Limit and series representations when $s>1$} We simply have
$$
\zeta(s,x) ~=~ \sum_{k=0}^{\infty}(x+k)^{-s}
$$
and this series converges uniformly on $\R_+$. In particular, we retrieve the identity
$$
\psi_{\nu}(x) ~=~ (-1)^{\nu +1}\nu!\,\zeta(\nu +1,x){\,},\qquad\nu\in\N^*.
$$

\parag{Limit and series representations when $s<1$} We have the following Eulerian form
\begin{eqnarray*}
\zeta(s,x)-\zeta(s) &=& -g_s(x)+\sum_{j=0}^{\lfloor -s\rfloor}\tchoose{x}{j+1}\Delta^jg_s(1)\\
&& \null + \sum_{k=1}^{\infty}\left(-g_s(x+k)+\sum_{j=0}^{\lfloor 1-s\rfloor}\tchoose{x}{j}\,\Delta^j g_s(k)\right),
\end{eqnarray*}
and the Weierstrassian form can be obtained similarly. The associated series converge uniformly on any bounded
subset of $[0,\infty)$.

For instance, we have
\begin{eqnarray*}
\lefteqn{\textstyle{\zeta\left(-\frac{3}{2},x\right)-\zeta\left(-\frac{3}{2}\right)} ~=~ \displaystyle{x^{\frac{3}{2}}+\lim_{n\to\infty}\left(\sum_{k=1}^{n-1}\left((x+k)^{\frac{3}{2}}-k^{\frac{3}{2}}\right)
-x{\,}n^{\frac{3}{2}}-\tchoose{x}{2}\Delta_nn^{\frac{3}{2}}\right)}}\\
&=& x^{\frac{3}{2}}-x-(2\sqrt{2}-1)\tchoose{x}{2} + \sum_{k=1}^{\infty}\left((x+k)^{\frac{3}{2}}-k^{\frac{3}{2}}
-x\Delta_kk^{\frac{3}{2}}-\tchoose{x}{2}\Delta_k^2k^{\frac{3}{2}}\right)\\
&=& x^{\frac{3}{2}}-x+\textstyle{\frac{3}{4}{\,}\zeta\left(\frac{1}{2}\right)}\tchoose{x}{2} + \displaystyle{\sum_{k=1}^{\infty}\left((x+k)^{\frac{3}{2}}-k^{\frac{3}{2}}
-x\Delta_kk^{\frac{3}{2}}-\textstyle{\frac{3}{4}}\tchoose{x}{2}k^{-\frac{1}{2}}\right)}.
\end{eqnarray*}

The analogue of Gauss' limit is
$$
\zeta(s,x) ~=~ \zeta(s)+\lim_{n\to\infty}f^{\lfloor 1-s\rfloor}_n[g_s](x),\qquad x>0.
$$
where
$$
f^{\lfloor 1-s\rfloor}_n[g_s](x) ~=~ \sum_{k=0}^{n-1}(x+k)^{-s}-\sum_{k=1}^{n-1}k^{-s}
-\sum_{j=1}^{\lfloor 1-s\rfloor}\tchoose{x}{j}\,\Delta_n^{j-1}n^{-s}.
$$

\parag{Gregory's formula-based series representation} For any $x>0$ we have
\begin{eqnarray*}
\zeta(s,x) &=& \frac{x^{1-s}}{s-1}-\sum_{n=0}^{\infty}G_{n+1}\Delta^ng_s(x)\\
&=& \frac{x^{1-s}}{s-1}+\sum_{n=0}^{\infty}|G_{n+1}|\,\sum_{k=0}^n(-1)^k\tchoose{n}{k}(x+k)^{-s}{\,}.
\end{eqnarray*}
Setting $x=1$ in this identity yields a known series expression for $\zeta(s)$ that is the analogue of Fontana-Mascheroni series
$$
\zeta(s) ~=~ \frac{1}{s-1}+\sum_{n=0}^{\infty}|G_{n+1}|\,\sum_{k=0}^n(-1)^k\tchoose{n}{k}(k+1)^{-s}{\,}.
$$

\parag{Analogue of Gauss' multiplication formula} For any $m\in\N^*$ and any $x>0$, we have
$$
\sum_{j=0}^{m-1}\zeta\left(s,\frac{x+j}{m}\right) ~=~ m^s\,\zeta(s,x).
$$
Corollary~\ref{cor:Riem482} provides the following limits for any $x>0$
\begin{eqnarray*}
\lim_{m\to\infty}m^{s-1}\zeta(s,mx) &=& \frac{x^{1-s}}{s-1}{\,},\qquad s<1,\\
\lim_{m\to\infty}m^{s-1}(\zeta(s,mx)-\zeta(s,m)) &=& \frac{x^{1-s}-1}{s-1}{\,},\qquad s\neq 1.
\end{eqnarray*}

\parag{Analogue of Wallis's product formula} If $s>1$, then we have
\begin{equation}\label{eq:Dirich663let}
\sum_{k=1}^{\infty}\frac{(-1)^{k-1}}{k^s} ~=~ (1-2^{1-s})\,\zeta(s) ~=~ \eta(s),
\end{equation}
where $s\mapsto\eta(s)$ is Dirichlet's eta function.\index{Dirichlet's eta function|textbf} When $s<1$, the form of the formula strongly depends upon the value of $s$. When $s=-\frac{3}{2}$ for instance, we obtain
$$
\lim_{n\to\infty}\left(h(n)+\sum_{k=1}^{2n}(-1)^kk^{\frac{3}{2}}\right) ~=~ \textstyle{(4\sqrt{2}-1)\,\zeta(-\frac{3}{2})}.
$$
where $h(n)=-\frac{8n+3}{4}\,\sqrt{\frac{n}{2}}$.

\parag{Restriction to the natural integers} For any $n\in\N^*$ we have
$$
\zeta(s,n)-\zeta(s) ~=~ -\sum_{k=1}^{n-1}k^{-s}\qquad\text{and}\qquad\zeta(s,n) ~=~ \sum_{k=n}^{\infty}k^{-s}.
$$
Gregory's formula states that for any $n\in\N^*$ and any $q\in\N$ we have
$$
\sum_{k=1}^{n-1}k^{-s} ~=~ \frac{1-n^{1-s}}{s-1}+\sum_{j=1}^qG_j\left(\Delta^{j-1}g_s(n)-\Delta^{j-1}g_s(1)\right)+R^q_{s,n}{\,},
$$
with
$$
|R^q_{s,n}| ~\leq ~\overline{G}_q{\,}|\Delta^qg_s(n)-\Delta^qg_s(1)|.
$$
Many other representations of this sum can be derived from, e.g., the limit and series representations of the Hurwitz zeta function.

\parag{Generalized Webster's functional equation} For any $m\in\N^*$ and any $a>0$, there is a unique solution $f_s\colon\R_+\to\R$ to the equation
$$
\sum_{j=0}^{m-1}f_s\left(x+a{\,}j\right) ~=~ -x^{-s}
$$
that lies in $\cK^{\lfloor -s\rfloor_+}$, namely
$$
f_s(x) ~=~ \frac{1}{(am)^s}\,\zeta\left(s,\frac{x+a}{am}\right)-\frac{1}{(am)^s}\,\zeta\left(s,\frac{x}{am}\right){\,}.
$$

\parag{Analogue of Euler's series representation of $\gamma$} We have
$$
(\Sigma g_s)^{(k)}(1) ~=~ (-s)^{\underline{k}}\,\zeta(s+k),\qquad k\in\N^*.
$$
Thus, the Taylor series expansion of $\zeta(s,x+1)$ about $x=0$ is
$$
\zeta(s,x+1) ~=~ \sum_{k=0}^{\infty}\tchoose{-s}{k}\,\zeta(s+k){\,}x^k{\,},\qquad |x|<1.
$$
Integrating both sides of this equation on $(0,1)$, we obtain the identity
$$
\sum_{k=1}^{\infty}\tchoose{1-s}{k}\,\zeta(s+k-1) ~=~ -1{\,},\qquad s<2,~s\notin\Z{\,}.
$$
(When $s>2$, the summand in the series above does not approach zero as $k$ increases.)

\parag{Analogue of the reflection formula} A reflection formula can be derived when $s$ is an integer. Recall that we have the following special values for any $n\in\N^*$
$$
\zeta(1+n,x) ~=~ (-1)^{n-1}\frac{1}{n!}\,\psi_n(x)
$$
and
$$
\zeta(1-n,x) ~=~ -\frac{1}{n}{\,}B_n(x).
$$
It follows that for any $x\in\R\setminus\Z$, we have
$$
\zeta(s,x)+(-1)^s\,\zeta(s,1-x) ~=~
\begin{cases}
\frac{(-1)^{s-1}}{(s-1)!}\,\pi{\,}D^{s-1}\cot(\pi x), & \text{if $s-1\in\N^*$},\\
0, & \text{if $-s\in\N$}.
\end{cases}
$$
\index{Hurwitz zeta function|)}

\section{The generalized Stieltjes constants}
\label{sec:Stie62}
\index{Stieltjes constants!generalized Stieltjes constants|(}

Recall that the \emph{generalized Stieltjes constants} are the numbers $\gamma_n(x)$\label{p:gSc63} that occur in the Laurent series expansion of the Hurwitz zeta function
\begin{equation}\label{eq:7sfds}
\zeta(s,x) ~=~ \frac{1}{s-1}+\sum_{n=0}^{\infty}\frac{(-1)^n}{n!}\,\gamma_n(x)(s-1)^n.
\end{equation}
Recall also that the numbers $\gamma_n=\gamma_n(1)$, where $n\in\N$, are called the \emph{Stieltjes constants}.\index{Stieltjes constants}\label{p:Sc630} The Stieltjes constants and generalized Stieltjes constants are known to satisfy the relations
$$
\gamma_0(x) ~=~ -\psi(x)\qquad\text{and}\qquad\gamma_0 ~=~ \gamma
$$
as well as the following identities for every $q\in\N$
\begin{eqnarray*}
\gamma_q &=& \lim_{n\to\infty}\left(\sum_{k=1}^n\frac{(\ln k)^q}{k}-\frac{(\ln n)^{q+1}}{q+1}\right),\\
\gamma_q(x) &=& \lim_{n\to\infty}\left(\sum_{k=0}^n\frac{(\ln(x+k))^q}{x+k}-\frac{(\ln(x+n))^{q+1}}{q+1}\right).
\end{eqnarray*}
For recent background on these constants, see, e.g., Blagouchine~\cite{Bla15,Bla16} and Blagouchine and Coppo~\cite{BlaCop18} (see also Nan-Yue and Williams \cite{NanWil94}).

Here we naturally restrict the values of $x$ to the set $\R_+$. Interestingly, the generalized Stieltjes constants also satisfy the difference equation
$$
\gamma_q(x+1)-\gamma_q(x) ~=~ g_q(x),
$$
where $g_q\colon\R_+\to\R$ is the function defined by the equation
$$
g_q(x) ~=~ -\frac{1}{x}(\ln x)^q\qquad\text{for $x>0$}.
$$
Thus, our theory is particularly suitable for the investigation of these constants. For any $q\in\N$, the function $g_q$ lies in $\cC^{\infty}\cap\cD^0\cap\cK^{\infty}$ and is increasing on $[e^q,\infty)$. By uniqueness of $\Sigma g_q$, it follows that
$$
\Sigma g_q(x) ~=~ \gamma_q(x)-\gamma_q.
$$

\parag{ID card} The introduction above enables us to provide the following basic information about the generalized Stieltjes constants.
$$
\begin{array}{|c|c|c|c|}
\hline g_q(x) & \text{Membership} & \deg g_q & \Sigma g_q(x) \\
\hline \emptybox -\frac{1}{x}(\ln x)^q & \cC^{\infty}\cap\cD^0\cap\cK^{\infty} & -1 & \gamma_q(x)-\gamma_q\\
\hline
\end{array}
$$

\parag{Analogue of Bohr-Mollerup's theorem} The function $\gamma_q$ can be characterized as follows.
\begin{quote}
\emph{All eventually monotone solutions $f_q\colon\R_+\to\R$ to the equation
$$
f_q(x+1)-f_q(x) ~=~ -\frac{1}{x}(\ln x)^q
$$
are of the form $f_q(x)=c_q+\gamma_q(x)$, where $c_q\in\R$}.
\end{quote}
Using Proposition~\ref{prop:90unic41}, we can also derive the following alternative characterization of the function $\gamma_q$.
\begin{quote}
\emph{All solutions $f_q\colon\R_+\to\R$ to the equation
$$
f_q(x+1)-f_q(x) ~=~ -\frac{1}{x}(\ln x)^q
$$
that satisfy the asymptotic condition that, for each $x>0$,
$$
f_q(x+n)-f_q(n) ~\to ~0\qquad\text{as $n\to_{\N}\infty$}
$$
are of the form $f_q(x)=c_q+\gamma_q(x)$, where $c_q\in\R$}.
\end{quote}

\parag{Extended ID card} Using identity \eqref{eq:SgSt7FrI}, we can immediately make the remarkable observation that the asymptotic constant $\sigma[g_q]$ is exactly the opposite of the Stieltjes constant $\gamma_q$. We then have the following values
$$
\begin{array}{|c|c|c|}
\hline \overline{\sigma}[g_q] & \sigma[g_q] & \gamma[g_q] \\
\hline \emptybox \infty & -\gamma_q & -\gamma_q \\
\hline
\end{array}
$$

\begin{itemize}
\item \emph{Alternative representations of $\sigma[g_q]=\gamma[g_q]$}
\begin{eqnarray*}
\gamma_q &=& \sum_{k=1}^{\infty}\left(\frac{(\ln k)^q}{k}-\frac{(\ln(k+1))^{q+1}-(\ln(k))^{q+1}}{q+1}\right),\\
\gamma_q &=& \int_1^{\infty}\frac{\{t\}-\frac{1}{2}}{t^2}{\,}(\ln t)^{q-1}(q-\ln t){\,}dt{\,}\qquad (q\geq 1),\\
\gamma_q &=& \int_1^{\infty}\left(\frac{(\ln\lfloor t\rfloor)^q}{\lfloor t\rfloor}-\frac{(\ln t)^q}{t}\right)dt{\,}.
\end{eqnarray*}

\item \emph{Generalized Binet's function}. For any $r\in\N$ and any $x>0$
$$
J^{r+1}[\gamma_q](x) ~=~ \gamma_q(x)+\frac{(\ln x)^{q+1}}{q+1}+\sum_{j=1}^rG_j\,\Delta^{j-1}g_q(x).
$$

\item \emph{Analogue of Raabe's formula}
\begin{equation}\label{eq:RgSt4}
\int_x^{x+1}\gamma_q(t){\,}dt ~=~ -\frac{(\ln x)^{q+1}}{q+1}{\,},\qquad x>0.
\end{equation}

\item \emph{Alternative characterization}. The function $f(x)=\gamma_q(x)$ is the unique solution lying in $\cC^0\cap\cK^0$ to the equation
$$
\int_x^{x+1}f(t){\,}dt ~=~ -\frac{(\ln x)^{q+1}}{q+1}{\,}, \qquad x>0.
$$
\end{itemize}

\parag{Inequalities} The following inequalities hold for any $x>0$, any $a>0$, and any $n\in\N$.
\begin{itemize}
\item \emph{Symmetrized generalized Wendel's inequality} (equality if $a\in\{0,1\}$)

If $x\geq e^q$, we have
$$
\left|\gamma_q(x+a)-\gamma_q(x)\right| ~\leq ~ \lceil a\rceil\left|\frac{(\ln x)^q}{x}\right|.
$$

\item \emph{Symmetrized generalized Wendel's inequality} (discrete version)

If $n\geq e^q$, we have
$$
\left|\gamma_q(x) - \gamma_q-\frac{(\ln x)^q}{x}-\sum_{k=1}^{n-1}\left(\frac{(\ln(x+k))^q}{x+k}-\frac{(\ln k)^q}{k}\right)\right| ~\leq ~ \lceil x\rceil\left|\frac{(\ln n)^q}{n}\right|.
$$

\item \emph{Symmetrized Stirling's and Burnside's formulas-based inequalities}

If $x\geq e^q$, we have
$$
\left|\gamma_q\left(x+\frac{1}{2}\right)+\frac{(\ln x)^{q+1}}{q+1}\right| ~\leq ~ \left|\gamma_q(x)+\frac{(\ln x)^{q+1}}{q+1}\right| ~\leq ~ \left|\frac{(\ln x)^q}{x}\right|{\,}.
$$

\item \emph{Further inequalities}. For $0<x\leq 1$, we use the following approximations (see Nan-Yue and Williams \cite[p.\ 148]{NanWil94})
$$
\left|\gamma_0(x)-\frac{1}{x}\right| ~\leq ~\gamma
$$
and
$$
\left|\gamma_q(x)-\frac{(\ln x)^q}{x}\right| ~\leq ~ \frac{(3+(-1)^q)(2q)!}{q^{q+1}(2\pi)^q}{\,},\quad q\in\N^*.
$$
\end{itemize}

\parag{Generalized Stirling's and related formulas} For any $a\geq 0$, we have the following limits and
asymptotic equivalence as $x\to\infty$,
$$
\gamma_q(x+a)-\gamma_q(x) ~\to~ 0,\qquad\gamma_q(x)+\frac{(\ln x)^{q+1}}{q+1} ~\to ~ 0,
$$
$$
\gamma_q(x+a) ~\sim ~ -\frac{(\ln x)^{q+1}}{q+1}{\,}.
$$
\emph{Burnside-like approximation} (better than Stirling-like approximation)
$$
\gamma_q(x)+\frac{1}{q+1}\left(\ln\left(x-\frac{1}{2}\right)\right)^{q+1} ~\to ~ 0.
$$
\emph{Further results} (obtained by differentiation)
$$
\gamma'_q(x)+\frac{(\ln x)^q}{x} ~\to ~ 0,\qquad \gamma'_q(x+a) ~\sim ~ -\frac{(\ln x)^q}{x}{\,}.
$$
For any $r\in\N$,
$$
\gamma^{(r)}_q(x+a)-\gamma^{(r)}_q(x) ~\to~ 0,\qquad D^r_x\left(\gamma_q(x)+\frac{(\ln x)^{q+1}}{q+1}\right) \to ~ 0.
$$
$$
D^r_x\left(\gamma_q(x)+\frac{1}{q+1}\left(\ln\left(x-\frac{1}{2}\right)\right)^{q+1}\right) \to ~ 0.
$$

\parag{Asymptotic expansions} For any $m,r\in\N^*$ we have the following expansion as $x\to\infty$
$$
\frac{1}{m}\,\sum_{j=0}^{m-1}\gamma_q\left(x+\frac{j}{m}\right) ~=~ -\frac{(\ln x)^{q+1}}{q+1}+\sum_{k=1}^r\frac{B_k}{m^kk!}{\,}g_q^{(k-1)}(x)+O\left(g_q^{(r)}(x)\right).
$$
Setting $m=1$ in this latter formula, we obtain
$$
\gamma_q(x) ~=~ -\frac{(\ln x)^{q+1}}{q+1}+\sum_{k=1}^r\frac{B_k}{k!}{\,}g_q^{(k-1)}(x)+O\left(g_q^{(r)}(x)\right).
$$
Let us detail this expansion when $q=1$. We first observe that
$$
g_1^{(k-1)}(x) ~=~ (-1)^k(k-1)!\,\frac{\ln x -H_{k-1}}{x^k}{\,},\qquad k\in\N^*.
$$
Using \eqref{eq:ExpAs0psi}, we then obtain
\begin{eqnarray*}
\lefteqn{\frac{1}{m}\sum_{j=0}^{m-1}\gamma_1\left(x+\frac{j}{m}\right) + (\ln x){\,}\frac{1}{m}\sum_{j=0}^{m-1}\psi\left(x+\frac{j}{m}\right)}\\
&=& \frac{(\ln x)^2}{2}+\sum_{k=1}^r\frac{(-1)^{k-1}{\,}B_k{\,}H_{k-1}}{k(mx)^k}+O\left(x^{-r-1}\right).
\end{eqnarray*}
Setting $m=1$ in this latter formula, we get
$$
\gamma_1(x) ~=~ \frac{(\ln x)^2}{2}-\psi(x)\ln x+ \sum_{k=1}^r\frac{(-1)^{k-1}{\,}B_k{\,}H_{k-1}}{k{\,}x^k}+O\left(x^{-r-1}\right){\,}.
$$
Setting $r=5$ for instance, we obtain
$$
\gamma_1(x) ~=~ \frac{(\ln x)^2}{2}-\psi(x)\ln x-\frac{1}{12x^2}+\frac{11}{720x^4}+O\left(x^{-6}\right).
$$

\parag{Generalized Liu's formula} For any $q\geq 1$ and any x > 0 we have
$$
\gamma_q(x) ~=~ -\frac{(\ln x)^{q+1}}{q+1}+\frac{(\ln x)^q}{2x}+\int_0^{\infty}\frac{\{t\}-\frac{1}{2}}{(x+t)^2}{\,}(\ln(x+t))^{q-1}(q-\ln(x+t)){\,}dt.
$$

\parag{Series representations} Since the function $g_q(x)$ lies in $\cD^{-1}_{\N}$, we only have the following series representations of $\gamma_q(x)$.
\begin{itemize}
\item \emph{Eulerian and Weierstrassian forms}. We have
$$
\gamma_q(x) ~=~ \gamma_q+\frac{(\ln x)^q}{x}+\sum_{k=1}^{\infty}\left(\frac{(\ln(x+k))^q}{x+k}-\frac{(\ln k)^q}{k}\right),
$$
$$
\gamma_q(x) ~=~ \frac{(\ln x)^q}{x}+\sum_{k=1}^{\infty}\left(\frac{(\ln(x+k))^q}{x+k}-\frac{(\ln(k+1))^{q+1}-(\ln k)^{q+1}}{q+1}\right).
$$
The series can be differentiated term by term infinitely many times. For instance, we get
$$
\gamma'_q(x) ~=~ \sum_{k=0}^{\infty}\frac{(\ln(x+k))^{q-1}}{(x+k)^2}{\,}(q-\ln(x+k)).
$$

\item The analogue of Gauss' limit coincides with the Eulerian form.

\item \emph{Gregory's formula-based series representation}. For any $x>0$ satisfying the assumptions of Proposition~\ref{prop:6699rem}, we obtain
\begin{eqnarray*}
\gamma_q(x) + \frac{(\ln x)^{q+1}}{q+1} &=& \sum_{n=0}^{\infty}G_{n+1}\Delta_x^n\frac{(\ln x)^q}{x}\\
&=& \sum_{n=0}^{\infty}|G_{n+1}|\sum_{k=0}^n(-1)^k\tchoose{n}{k}\frac{(\ln(x+k))^q}{x+k}{\,}.
\end{eqnarray*}
Setting $x=1$ in this identity (provided that $x=1$ satisfies the assumptions of Proposition~\ref{prop:6699rem}), we obtain the Fontana-Mascheroni's series expression for $\gamma_q$
$$
\gamma_q ~=~ \sum_{n=0}^{\infty}|G_{n+1}|\sum_{k=0}^n(-1)^k\tchoose{n}{k}\frac{(\ln(k+1))^q}{k+1}{\,}.
$$
This latter expression can be found in Blagouchine \cite[p.\ 383]{Bla16} and the references therein.
\end{itemize}

\parag{Analogue of Gauss' multiplication formula} The following analogue of Gauss' multiplication formula was previously known (see also Blagouchine \cite[p.~542]{Bla15}) but it can be derived straightforwardly from our results.

For any $m\in\N^*$ and any $x>0$, we have
$$
\sum_{j=0}^{m-1}\gamma_q\left(\frac{x+j}{m}\right) ~=~ -\frac{m}{q+1}\left(\ln\frac{1}{m}\right)^{q+1}+m\sum_{j=0}^q\tchoose{q}{j}\left(\ln\frac{1}{m}\right)^j\gamma_{q-j}(x).
$$
In particular,
$$
\sum_{j=1}^m\gamma_q\left(\frac{j}{m}\right) ~=~ -\frac{m}{q+1}\left(\ln\frac{1}{m}\right)^{q+1}+m\sum_{j=0}^q\tchoose{q}{j}\left(\ln\frac{1}{m}\right)^j\gamma_{q-j}{\,}.
$$
Corollary~\ref{cor:Riem482} provides the following limits for $x>0$
\begin{eqnarray*}
\lim_{m\to\infty} \sum_{j=0}^q\tchoose{q}{j}\left(\ln\frac{1}{m}\right)^j\left(\gamma_{q-j}(mx)-\gamma_{q-j}(m)\right) &=& -\frac{(\ln x)^{q+1}}{q+1}{\,},\\
\lim_{m\to\infty} \left(-\frac{1}{q+1}\left(\ln\frac{1}{m}\right)^{q+1}{\!}
+\sum_{j=0}^q\tchoose{q}{j}\left(\ln\frac{1}{m}\right)^j\gamma_{q-j}(mx)\right) &=& -\frac{(\ln x)^{q+1}}{q+1}{\,}.
\end{eqnarray*}
For instance, setting $q=1$ in these formulas yields
\begin{eqnarray*}
\lim_{m\to\infty} \big(\gamma_1(mx)-\gamma_1(m)+(\ln m)(\psi(mx)-\psi(m))\big) &=& -\frac{1}{2}(\ln x)^2{\,},\\
\lim_{m\to\infty} \left(\gamma_1(mx)-\frac{1}{2}(\ln m)^2+\psi(mx)\ln m\right) &=& -\frac{1}{2}(\ln x)^2{\,}.
\end{eqnarray*}
Now, setting $m=2$ in the multiplication formula, we obtain the following analogue of Legendre's duplication formula
$$
\gamma_q\left(\frac{x}{2}\right)+\gamma_q\left(\frac{x+1}{2}\right) ~=~ -\frac{2}{q+1}\left(\ln\frac{1}{2}\right)^{q+1}+2\sum_{j=0}^q\tchoose{q}{j}\left(\ln\frac{1}{2}\right)^j\gamma_{q-j}(x).
$$
When $q=0$ and $q=1$, the multiplication formula reduces to the known formulas
\begin{eqnarray*}
\sum_{j=0}^{m-1}\psi\left(\frac{x+j}{m}\right) &=& m(\psi(x)-\ln m){\,},\\
\sum_{j=0}^{m-1}\gamma_1\left(\frac{x+j}{m}\right) &=& -\frac{m}{2}(\ln m)^2+m(\ln m)\,\psi(x)+m\,\gamma_1(x).
\end{eqnarray*}

\parag{Analogue of Wallis's product formula} The analogue of Wallis's formula for the function $g_q(x)$ is
\begin{equation}\label{eq:Wallqqgamma}
\sum_{k=1}^{\infty}(-1)^k\frac{(\ln k)^q}{k} ~=~ -\frac{(\ln 2)^{q+1}}{q+1}+\sum_{j=0}^{q-1}\tchoose{q}{j}{\,}(\ln 2)^{q-j}\gamma_j{\,}.
\end{equation}
This formula was established by Briggs and Chowla \cite[Eq.~(8)]{BriCho55}. For $q=1$, it reduces to
$$
\sum_{k=1}^{\infty}(-1)^k\,\frac{\ln k}{k} ~=~ -\frac{(\ln 2)^2}{2}+\gamma\ln 2{\,}.
$$
For $q=2$, we obtain
$$
\sum_{k=1}^{\infty}(-1)^k\,\frac{(\ln k)^2}{k} ~=~ -\frac{(\ln 2)^3}{3}+\gamma (\ln 2)^2+2\gamma_1 \ln 2{\,}.
$$
These latter two formulas were also established by Hardy~\cite{Har12}.

As an aside, let us establish conversion formulas between the sequences $q\mapsto\gamma_q$ and $q\mapsto\eta^{(q)}(1)$, where $\eta(s)$ is the Dirichlet eta function\index{Dirichlet's eta function} introduced in \eqref{eq:Dirich663let} and $\eta^{(q)}(1)$ stands for the limiting value of $\eta^{(q)}(s)$ as $s\to 1$. To ease the computations, let us instead consider the conversion formulas between the sequences $q\mapsto\gamma_q$ and $q\mapsto\lambda_q$, where
$$
\lambda_q ~=~ \frac{1}{q+1}{\,}(\ln 2)^{q+1}+(-1)^{q+1}\,\eta^{(q)}(1){\,},\qquad q\in\N.
$$
Using \eqref{eq:Wallqqgamma}, we can readily derive the following equations
\begin{equation}\label{eq:diiro4}
\lambda_q ~=~ \sum_{k=0}^{q-1}\tchoose{q}{k}{\,}(\ln 2)^{q-k}\,\gamma_k{\,},\qquad q\in\N.
\end{equation}
These equations actually consist of an infinite consistent triangular system. Solving this system provides the following conversion formula
\begin{equation}\label{eq:diiro4i}
\gamma_q ~=~ \sum_{k=0}^q\tchoose{q}{k}\,\frac{B_{q-k}}{k+1}{\,}(\ln 2)^{q-k-1}\,\lambda_{k+1}{\,},\qquad q\in\N,
\end{equation}
that is,
$$
\gamma_q ~=~ -\frac{B_{q+1}}{q+1}{\,}(\ln 2)^{q+1}+\sum_{k=0}^q(-1)^k\,\tchoose{q}{k}\,\frac{B_{q-k}}{k+1}(\ln 2)^{q-k-1}\eta^{(k+1)}(1){\,},\quad q\in\N.
$$
Indeed, plugging \eqref{eq:diiro4i} in the right side of \eqref{eq:diiro4} we obtain for any $q\in\N$
\begin{eqnarray*}
\sum_{k=0}^{q-1}\tchoose{q}{k}{\,}(\ln 2)^{q-k}\,\gamma_k
&=& \sum_{k=0}^{q-1}\tchoose{q}{k}{\,}(\ln 2)^{q-k}\,\sum_{j=0}^k\tchoose{k}{j}\,\frac{B_{k-j}}{j+1}{\,}(\ln 2)^{k-j-1}\,\lambda_{j+1}\\
&=& \sum_{j=0}^{q-1}\tchoose{q}{j}{\,}(\ln 2)^{q-j-1}\,\frac{\lambda_{j+1}}{j+1}\,\sum_{k=j}^{q-1}\tchoose{q-j}{k-j}{\,}B_{k-j}{\,},
\end{eqnarray*}
where the inner sum reduces to $0^{q-j-1}$. The latter quantity then reduces to $\lambda_q$, as expected.

\begin{remark}
The conversion formulas \eqref{eq:diiro4} and \eqref{eq:diiro4i} are not quite new. In essence, they were established by Liang and Todd \cite[Eq.\ (3.6)]{LiaTod72} and Nan-Yue and Williams \cite[Eqs.\ (1.9) and (7.1)]{NanWil94}.
\end{remark}

\parag{Generalized Webster's functional equation} For any $m\in\N^*$ and any $a>0$, there is a unique eventually monotone solution $f\colon\R_+\to\R$ to the equation
$$
\sum_{j=0}^{m-1}f\left(x+a{\,}j\right) ~=~ -\frac{1}{x}(\ln x)^q{\,},
$$
namely
$$
f(x) ~=~ S_{q,am}\left(\frac{x+a}{am}\right)-S_{q,am}\left(\frac{x}{am}\right){\,},
$$
where
$$
S_{q,am}(x) ~=~ \frac{1}{am}\,\sum_{j=0}^q{q\choose j}(\ln(am))^j\,\gamma_{q-j}(x).
$$
For instance, the unique eventually monotone solution $f\colon\R_+\to\R$ to the equation
$$
f(x)+f(x+1) ~=~ -\frac{1}{x}\ln x
$$
is
$$
f(x) ~=~ \gamma_1(x)-\gamma_1\left(\frac{x}{2}\right)+(\ln 2)\,\psi(x)+\frac{1}{2}(\ln 2)^2.
$$

\parag{Rational arguments theorem} Let us apply Proposition~\ref{prop:GauProv5} to the function $g_q(x)$. For any $a,b\in\N^*$ with $a<b$ and any $j\in\{0,\ldots,b-1\}$ we have
$$
S_j^b[g_q] ~=~ b{\,}(-1)^{q+1}\sum_{i=0}^q{q\choose i}(\ln b)^{q-i}D^i_s\,\mathrm{Li}_s(z)\big|_{(s,z)=(1,\omega_b^j)},
$$
where $\mathrm{Li}_s(z)$ is the polylogarithm function.\index{polylogarithm function} Hence, we have
$$
\gamma_q\left(\frac{a}{b}\right)-\gamma_q ~=~ (-1)^{q+1}\sum_{i=0}^q{q\choose i}(\ln b)^{q-i}\sum_{j=0}^{b-1}(1-\omega_b^{-aj})D^i_s\,\mathrm{Li}_s(z)\big|_{(s,z)=(1,\omega_b^j)}{\,}.
$$
We note that a more practical formula was derived in the special case when $q=1$ by Blagouchine~\cite{Bla15} as a generalization of Gauss' digamma theorem.

\index{Stieltjes constants!generalized Stieltjes constants|)}

\section{Higher order derivatives of the Hurwitz zeta function}
\label{sec:HDHu4Z8F3}
\index{Hurwitz zeta function!higher order derivatives|(}

Let $s\in\R\setminus\{1\}$ and $q\in\N$. Differentiating $q$ times both sides of \eqref{eq:HuZeRec5} we obtain
$$
\zeta^{(q)}(s,x+1)-\zeta^{(q)}(s,x) ~=~ (-1)^{q+1}x^{-s}(\ln x)^q{\,},\qquad x>0,
$$
where $\zeta^{(q)}(s,x)$ stands for $D^q_s\,\zeta(s,x)$. This equation shows that the investigation of the higher order derivatives of the Hurwitz zeta function can be carried out using our results. To keep our presentation simple, we will focus on some selected results only.

The interested reader can find an earlier study of these functions in Ramanujan's second notebook \cite[p.~36 \emph{et seq.}]{Ber83}.

\parag{ID card} The following basic information can be easily derived.
$$
\begin{array}{|c|c|c|c|}
\hline g_{s,q}(x) & \text{Membership} & \deg g_{s,q} & \Sigma g_{s,q}(x) \\
\hline \emptybox -x^{-s}(-\ln x)^q & \begin{array}{rl}\emptybox\cC^{\infty}\cap\cD^{-1}\cap\cK^{\infty},&\text{if $s>1$},\\ \cC^{\infty}\cap\cD^{\lfloor 1-s\rfloor}\cap\cK^{\infty},&\text{if $s<1$}.\end{array} & \begin{array}{c} -1\\ \null +\lfloor 1-s\rfloor_+\end{array} & \begin{array}{c}\zeta^{(q)}(s,x)\\ -\zeta^{(q)}(s)\end{array} \\
\hline
\end{array}
$$
We observe that this investigation can be regarded as a simultaneous generalization of the studies of the Hurwitz zeta function and the generalized Stieltjes constants. For the latter, we observe that
$$
(-1)^q\,\lim_{s\to 1}g_{s,q}(x) ~=~ -\frac{1}{x}(\ln x)^q.
$$
Setting $s=0$ in our results may also be very informative as it produces formulas involving the well-studied quantities $\zeta^{(q)}(0)$ and $\zeta^{(q)}(0,x)-\zeta^{(q)}(0)$ for any $q\in\N$.

\begin{project}
\textsl{Find a closed-form expression for the integral}
$$
\int_1^x\gamma_q(t){\,}dt.
$$
We apply Proposition~\ref{prop:an4tR8} to $g_q(x)=-\frac{1}{x}(\ln x)^q$. Using \eqref{eq:RgSt4} we obtain
\begin{eqnarray*}
\int_1^x\gamma_q(t){\,}dt &=& \Sigma_x\int_x^{x+1}\gamma_q(t){\,}dt ~=~ -\frac{1}{q+1}\,\Sigma(\ln x)^{q+1}\\
&=& \frac{(-1)^{q+1}}{q+1}\,\Sigma g_{0,q+1}(x){\,},
\end{eqnarray*}
that is,
$$
\int_1^x\gamma_q(t){\,}dt ~=~ \frac{(-1)^{q+1}}{q+1}\left(\zeta^{(q+1)}(0,x)-\zeta^{(q+1)}(0)\right).
$$
In particular,
$$
\gamma_q(x) ~=~ \frac{(-1)^{q+1}}{q+1}\, D_x\zeta^{(q+1)}(0,x){\,}.\qedhere
$$
\end{project}

\parag{Analogue of Bohr-Mollerup's theorem} The function $\zeta^{(q)}(s,x)$ can be characterized as follows.
\begin{quote}
\emph{All solutions $f_{s,q}\colon\R_+\to\R$ to the equation
$$
f_{s,q}(x+1)-f_{s,q}(x) ~=~ g_{s,q}(x)
$$
that lie in $\cK^{\lfloor 1-s\rfloor_+}$ are of the form $f_{s,q}(x)=c_{s,q}+\zeta^{(q)}(s,x)$, where $c_{s,q}\in\R$.}
\end{quote}

\parag{Extended ID card} The asymptotic constant $\sigma[g_{s,q}]$ satisfies the identity
$$
\sigma[g_{s,q}] ~=~ \int_0^1\zeta^{(q)}(s,t+1){\,}dt - \zeta^{(q)}(s) ~=~ \frac{-q!}{(1-s)^{q+1}}-\zeta^{(q)}(s).
$$
Hence we have the following values
$$
\begin{array}{|c|c|c|}
\hline \overline{\sigma}[g_{s,q}] & \sigma[g_{s,q}] & \gamma[g_{s,q}] \\
\hline \emptybox \begin{array}{rl}\infty,&\text{if $s>1$},\\ -\zeta^{(q)}(s),&\text{if $s<1$}.\end{array} & \frac{-q!}{(1-s)^{q+1}}-\zeta^{(q)}(s) & \sigma[g_{s,q}]-\sum_{j=1}^{\lfloor 1-s\rfloor_+}G_j\Delta^{j-1}g_{s,q}(1) \\
\hline
\end{array}
$$

\begin{itemize}
\item \emph{Alternative representations of $\sigma[g_{s,q}]$}
\begin{eqnarray*}
\sigma[g_{s,q}] &=& \lim_{n\to\infty}\left(\sum_{k=1}^{n-1}g_{s,q}(k)-\int_1^ng_{s,q}(t){\,}dt
+\sum_{j=1}^{\lfloor 1-s\rfloor_+}G_j\Delta^{j-1}g_{s,q}(n)\right),\\
\sigma[g_{s,q}] &=& \sum_{j=1}^{\lfloor 1-s\rfloor_+}G_j\Delta^{j-1}g_{s,q}(1)\\
&& \null
-\sum_{k=1}^{\infty}\left(\int_k^{k+1}g_{s,q}(t){\,}dt-\sum_{j=0}^{\lfloor 1-s\rfloor_+}G_j\Delta^jg_{s,q}(k)\right).
\end{eqnarray*}
Setting $s=0$ in the previous formulas, we obtain
\begin{eqnarray*}
(-1)^q(q!+\zeta^{(q)}(0)) &=& \lim_{n\to\infty}\left(\sum_{k=1}^n(\ln k)^q-\int_1^n(\ln t)^q{\,}dt-\frac{1}{2}(\ln n)^q\right)\\
&=& \sum_{k=1}^{\infty}\left(\frac{1}{2}(\ln k)^q-\int_k^{k+1}(\ln t)^q{\,}dt\right).
\end{eqnarray*}
The left-hand quantity can actually be related to the Stieltjes constants in a very simple way. Indeed, on differentiating both sides of \eqref{eq:7sfds}, we obtain the following surprising identity
$$
(-1)^q(q!+\zeta^{(q)}(0)) ~=~ \sum_{n=0}^{\infty}\frac{\gamma_{n+q}}{n!}{\,}.
$$

\item \emph{Generalized Binet's function}. For any $r\in\N$ and any $x>0$
$$
J^{r+1}[\Sigma g_{s,q}](x) ~=~ \zeta^{(q)}(s,x)-\int_x^{x+1}\zeta^{(q)}(s,t){\,}dt+\sum_{j=1}^rG_j\,\Delta^{j-1}g_{s,q}(x).
$$

\item \emph{Analogue of Raabe's formula}. We have
$$
\int_1^x g_{s,q}(t){\,}dt ~=~ \frac{q!-\Gamma(q+1,(s-1)\ln x)}{(1-s)^{q+1}}{\,},\quad x>0,
$$
and hence the analogue of Raabe's formula is
\begin{eqnarray*}
\int_x^{x+1}\zeta^{(q)}(s,t){\,}dt &=& -\frac{\Gamma(q+1,(s-1)\ln x)}{(1-s)^{q+1}}\\
&=& -q!{\,}\frac{x^{1-s}}{(1-s)^{q+1}}{\,}\sum_{j=0}^q\frac{((s-1)\ln x)^j}{j!}{\,},\quad x>0.
\end{eqnarray*}
\end{itemize}

\parag{Generalized Stirling's and related formulas} For any $a\geq 0$ we have
$$
\zeta^{(q)}(s,x+a)-\zeta^{(q)}(s,x)-\sum_{j=1}^{\lfloor 1-s\rfloor_+}\tchoose{a}{j}\,\Delta^{j-1}g_{s,q}(x) ~\to ~0 \qquad\text{as $x\to\infty$},
$$
with equality if $a\in\{0,1,\ldots,\lfloor 1-s\rfloor_+\}$. Also, we have the following analogue of Stirling's formula
$$
\zeta^{(q)}(s,x)-\int_x^{x+1}\zeta^{(q)}(s,t){\,}dt+\sum_{j=1}^{\lfloor 1-s\rfloor_+}G_j\,\Delta^{j-1}g_{s,q}(x) ~\to ~0\qquad\text{as $x\to\infty$}.
$$
Setting $s=0$ in this latter formula and then simplifying the resulting expression, we obtain
$$
\zeta^{(q)}(0,x)+\Gamma(q+1,-\ln x)+\frac{1}{2}(-1)^{q+1}(\ln x)^q ~\to ~0\qquad\text{as $x\to\infty$}.
$$
We also have
$$
\zeta^{(q)}(s,x+a) ~\sim ~ \int_x^{x+1}\zeta^{(q)}(s,t){\,}dt\qquad\text{as $x\to\infty$}.
$$
Finally, if $s>-1$, then we have the following analogue of Burnside's formula
$$
\zeta^{(q)}(s,x)-\int_{x-\frac{1}{2}}^{x+\frac{1}{2}}\zeta^{(q)}(s,t){\,}dt ~\to ~ 0{\,},\qquad\text{as $x\to\infty$},
$$
which provides a better approximation of $\zeta^{(q)}(s,x)$ than the analogue of Stirling's formula.

\parag{Eulerian and Weierstrassian forms} If $s>1$, then for any $x>0$, we simply have
$$
\zeta^{(q)}(s,x) ~=~ -\sum_{k=0}^{\infty}g_{s,q}(x+k)
$$
and this series converges uniformly on $\R_+$ and can be integrated and differentiated term by term. If $s<1$, then for any $x>0$, we obtain the following Eulerian form
\begin{eqnarray*}
\zeta^{(q)}(s,x)-\zeta^{(q)}(s) &=& -g_{s,q}(x)+\sum_{j=0}^{\lfloor -s\rfloor}\tchoose{x}{j+1}\Delta^jg_{s,q}(1)\\
&& \null + \sum_{k=1}^{\infty}\left(-g_{s,q}(x+k)+\sum_{j=0}^{\lfloor 1-s\rfloor}\tchoose{x}{j}\,\Delta^j g_{s,q}(k)\right)
\end{eqnarray*}
and the Weierstrassian form can be obtained similarly. Both associated series converge uniformly on any bounded subset of $[0,\infty)$ and can be integrated and differentiated term by term. Note that the case where $(s,q)=(0,2)$ can be found in Ramanujan's second notebook \cite[p.\ 26--27]{Ber83}.

\parag{Gregory's formula-based series representation} For any $x>0$ satisfying the assumptions of Proposition~\ref{prop:6699rem}, we have
\begin{eqnarray*}
\zeta^{(q)}(s,x) &=& \int_x^{x+1}\zeta^{(q)}(s,t){\,}dt-\sum_{n=0}^{\infty}G_{n+1}\Delta^ng_{s,q}(x)\\
&=& \int_x^{x+1}\zeta^{(q)}(s,t){\,}dt-\sum_{n=0}^{\infty}|G_{n+1}|\,\sum_{k=0}^n(-1)^k\tchoose{n}{k}{\,}g_{s,q}(x+k){\,}.
\end{eqnarray*}
Setting $x=1$ in this identity (provided that $x=1$ satisfies the assumptions of Proposition~\ref{prop:6699rem}) yields a series expression for $\zeta^{(q)}(s)$ that is the analogue of Fontana-Mascheroni series
$$
\zeta^{(q)}(s) ~=~ \frac{-q!}{(1-s)^{q+1}}-\sum_{n=0}^{\infty}|G_{n+1}|\,\sum_{k=0}^n(-1)^k\tchoose{n}{k}{\,}g_{s,q}(k+1){\,},
$$
which can also be obtained differentiating the analogue of Fontana-Mascheroni series for the Hurwitz zeta function. For instance, we have
$$
\zeta''(0) ~=~ -2+\sum_{n=0}^{\infty}|G_{n+1}|\,\sum_{k=0}^n(-1)^k\tchoose{n}{k}{\,}(\ln(k+1))^2
$$
and this latter value is also known to be (see, e.g., Berndt \cite[p.~25]{Ber83})
$$
\frac{1}{2}\,\gamma^2-\frac{\pi^2}{24}-\frac{1}{2}(\ln(2\pi))^2+\gamma_1{\,}.
$$

\parag{Analogue of Gauss' multiplication formula} Upon differentiating the analogue of Gauss' multiplication formula for the Hurwitz zeta function, we immediately obtain the following multiplication formula. For any $m\in\N^*$ and any $x>0$, we have
$$
\sum_{j=0}^{m-1}\zeta^{(q)}\left(s,\frac{x+j}{m}\right) ~=~ m^s\,\sum_{j=0}^q\tchoose{q}{j}(\ln m)^{q-j}\,\zeta^{(j)}(s,x).
$$
Moreover, Corollary~\ref{cor:Riem482} provides the following limit for any $x>0$ and any $s<1$
$$
\lim_{m\to\infty} \sum_{j=0}^q\tchoose{q}{j}(\ln m)^{q-j}\,\frac{\zeta^{(j)}(s,mx)}{m^{1-s}} ~=~ -\frac{\Gamma(q+1,(s-1)\ln x)}{(1-s)^{q+1}}{\,}.
$$
Also, for any $s\neq 1$, we have
$$
\lim_{m\to\infty} \sum_{j=0}^q\tchoose{q}{j}(\ln m)^{q-j}\,\frac{\zeta^{(j)}(s,mx)-\zeta^{(j)}(s,m)}{m^{1-s}} ~=~ \frac{q!-\Gamma(q+1,(s-1)\ln x)}{(1-s)^{q+1}}{\,}.
$$

\parag{Analogue of Wallis's product formula}  When $s<1$, the form of the analogue of Wallis's product formula strongly depends upon the value of $s$. If $s>1$, then we have
\begin{eqnarray*}
\eta^{(q)}(s) &=& \sum_{k=1}^{\infty}\frac{(-1)^{k-1}}{k^s}{\,}(-\ln k)^q\\
&=& \zeta^{(q)}(s)-2^{1-s}\sum_{j=0}^q\tchoose{q}{j}\left(\ln\frac{1}{2}\right)^{q-j}\zeta^{(j)}(s),
\end{eqnarray*}
where $s\mapsto\eta(s)$ is Dirichlet's eta function.\index{Dirichlet's eta function} Just as we did for the formulas \eqref{eq:diiro4} and \eqref{eq:diiro4i}, we can easily establish the following conversion formulas for $s>1$
\begin{eqnarray*}
\mu_q(s) &=& \sum_{k=0}^{q-1}\tchoose{q}{k}\left(\ln\frac{1}{2}\right)^{q-k}\zeta^{(k)}(s){\,},\qquad q\in\N{\,},\\
\zeta^{(q)}(s) &=& \sum_{k=0}^q\tchoose{q}{k}\,\frac{B_{q-k}}{k+1}\left(\ln\frac{1}{2}\right)^{q-k-1}\mu_{k+1}(s){\,},\qquad q\in\N{\,},
\end{eqnarray*}
where
$$
\mu_q(s) ~=~ 2^{s-1}(\zeta^{(q)}(s)-\eta^{(q)}(s))-\zeta^{(q)}(s){\,},\qquad q\in\N.
$$

\index{Hurwitz zeta function!higher order derivatives|)}

\section{The Catalan number function}
\index{Catalan number function|(}

The Catalan number function is the restriction to $\R_+$ of the map $x\mapsto C_x$ defined on $(-\frac{1}{2},\infty)$ by
$$
C_x ~=~ \frac{1}{x+1}{2x\choose x}.
$$
This function satisfies the equation
$$
C_{x+1} ~=~ \left(4-\frac{6}{x+2}\right) C_x{\,}.
$$
The additive version of this equation reads $\Delta f=g$, where the function $g$ is the logarithm of a rational function. We observe that such equations have been thoroughly investigated by Anastassiadis \cite[p.~41]{Ana64} (see also Kuczma \cite{Kuc60}).

The equation above shows that the Catalan number function can be investigated using our results. Let us briefly study this function.

\parag{ID card} The function $C_x$ is clearly a $\Gamma$-type function and we immediately derive the following basic information.
$$
\begin{array}{|c|c|c|c|}
\hline g(x) & \text{Membership} & \deg g & \Sigma g(x) \\
\hline \emptybox \ln\left(4-\frac{6}{x+2}\right) & \cC^{\infty}\cap\cD^1\cap\cK^{\infty} & 0 & \ln C_x \\
\hline
\end{array}
$$

\parag{Analogue of Bohr-Mollerup's theorem} The function $C_x$ can be characterized as follows.
\begin{quote}
\emph{All solutions $f\colon\R_+\to\R_+$ to the equation
$$
(x+2){\,}f(x+1) ~=~ (4x+2){\,}f(x)
$$
for which $\ln f$ lies in $\cK^1$ are of the form $f(x)=c{\,}C_x$, where $c>0$.}
\end{quote}

\parag{Extended ID card} We have the following values:
$$
\begin{array}{|c|c|c|}
\hline \overline{\sigma}[g] & \sigma[g] & \gamma[g] \\
\hline \emptybox \frac{1}{2}\left(3+\ln\frac{1}{8\pi}\right) & \frac{1}{2}\left(3+\ln\frac{8}{27\pi}\right) & \frac{1}{2}\left(3+\ln\frac{4}{27\pi}\right) \\
\hline
\end{array}
$$
We also have the inequality
$$
|\gamma[g]| ~\leq ~ \frac{25}{8}\ln 5+\frac{39}{8}\ln 3-16\ln 2+\frac{3}{4} ~\approx ~ 0.04
$$
and the following representations
\begin{eqnarray*}
\gamma[g] &=& \int_1^{\infty}\frac{3(\{t\}-\frac{1}{2})}{(t+2)(2t+1)}{\,}dt{\,},\\
\sigma[g] &=& \int_0^1\ln C_{t+1}{\,}dt.
\end{eqnarray*}
Moreover, the analogue of Raabe's formula is
$$
\int_x^{x+1}\ln C_t{\,}dt ~=~ \ln\left(\frac{e^{\frac{3}{2}}(4x+2)^{x+\frac{1}{2}}}{\sqrt{\pi}{\,}(x+2)^{x+2}}\right),\qquad x>0.
$$

\parag{Generalized Stirling's and related formulas} For any $a\geq 0$, we have
$$
\frac{C_{x+a}}{C_x} ~\sim ~ 4^a\qquad\text{and}\qquad C_x ~\sim ~ \frac{4^x}{x^{3/2}\,\sqrt{\pi}}\qquad\text{as $x\to\infty$}.
$$
Also, the analogue of Burnside's formula gives
$$
\ln C_x - \ln\left(\frac{e^{\frac{3}{2}}(4x)^x}{\sqrt{\pi}{\,}(x+\frac{3}{2})^{x+\frac{3}{2}}}\right) ~\to ~0 \qquad\text{as $x\to\infty$}.
$$

\parag{Restriction to the natural integers} For any $n\in\N^*$ we have
$$
C_n ~=~ \frac{1}{n+1}{\,}\tchoose{2n}{n}.
$$

\parag{Eulerian and Weierstrassian forms} For any $x>0$, we have
$$
C_x ~=~ \frac{x+2}{4x+2}{\,}2^x\,
\prod_{k=1}^{\infty}\frac{\left(2-\frac{3}{k+3}\right)^x}{\left(2-\frac{3}{k+2}\right)^{x-1}\left(2-\frac{3}{x+k+2}\right)}
$$
and
$$
C_x ~=~ \frac{x+2}{4x+2}{\,}e^{-\frac{x}{2}}\,
\prod_{k=1}^{\infty}\frac{1+\frac{x}{k+2}}{1+\frac{2x}{2k+1}}{\,}e^{\frac{3x}{(k+2)(2k+1)}}.
$$

\index{Catalan number function|)} 

\chapter{Defining new multiple $\log\Gamma$-type functions}
\label{chapter:11}

In the previous chapter, we tested our results on some multiple $\log\Gamma$-type functions that are well-known special functions. It is clear, however, that there are many other multiple $\log\Gamma$-type functions that are still to be introduced and investigated, simply as principal indefinite sums of standard functions.

In this chapter, we introduce and investigate the following functions (we use the acronym PIS for ``principal indefinite sum'')
\begin{itemize}
\item The PIS of the digamma function.
\item The PIS of the Hurwitz zeta function.
\item The PIS of the generating function for the Gregory coefficients.
\end{itemize}
The latter two examples are examined here in a broad way. A deeper investigation of these examples can be carried out simply by following all the steps and recipes given in Chapter~\ref{chapter:9}.

\section{The PIS of the digamma function}
\index{digamma function!principal indefinite sum|(}
\index{principal indefinite sum!of the digamma function|(}

Let us see what our theory tells us when $g(x)=\psi(x)$ is the digamma function. We first observe that $g$ lies in $\cC^{\infty}\cap\cD^1\cap\cK^{\infty}$.

Using summation by parts, we can easily see that
$$
\Sigma\psi(x) ~=~ (x-1)(\psi(x)-1).
$$
Moreover, from the identity $H_{x-1}=\psi(x)+\gamma$, we obtain immediately
$$
\Sigma_x H_{x-1} ~=~ (x-1)(H_{x-1} -1).
$$

This example may seem very basic at first glance, but since $H_x$ is the discrete analogue of the function $\ln x$, we expect an important analogy between $\Sigma\psi(x)$ and $\Sigma\ln x=\ln\Gamma(x)$, at least in terms of asymptotic behaviors. Actually, the analogue of Burnside's formula shows that the function
$$
\ln\Gamma\left(x-\frac{1}{2}\right)+\frac{1}{2}(1-\ln(2\pi))
$$
is a very good approximation of $\Sigma\psi(x)$.

Interestingly, using \eqref{eq:convGPSI} we can easily derive the following additional identity
$$
\Sigma\psi(x) ~=~ \frac{1}{2}\left(1-\ln(2\pi)\right)+ D\ln G(x),\qquad x>0,
$$
where $G$ is the Barnes $G$-function\index{Barnes's $G$-function} (see Section~\ref{sec:Barnes558}).

\begin{project}
\textsl{Find a closed-form expression for the function $\Sigma_x\psi^2(x)$.} Using again summation by parts, we obtain
$$
\Sigma_x\psi^2(x) ~=~ (x-1)\,\psi^2(x)-(2x-1)\,\psi(x)+2x-2-\gamma.
$$
We also note that the function $\psi^2(x)$ lies in $\cC^{\infty}\cap\cD^1\cap\cK^{\infty}$, just as does the function $\psi(x)$. The investigation of this new function in the light of our results is left to the reader.
\end{project}

\parag{ID card} The following basic information about the functions $\psi(x)$ and $\Sigma\psi(x)$ follows trivially from the discussion above.
$$
\begin{array}{|c|c|c|c|}
\hline g(x) & \text{Membership} & \deg g & \Sigma g(x) \\
\hline \emptybox \psi(x) & \emptybox\cC^{\infty}\cap\cD^1\cap\cK^{\infty} & 0 & (x-1)(\psi(x)-1) \\
\hline
\end{array}
$$

\parag{Analogue of Bohr-Mollerup's theorem} The function $\Sigma\psi(x)$ can be characterized as follows.
\begin{quote}
\emph{All eventually convex or concave solutions $f\colon\R_+\to\R$ to the equation
$$
f(x+1)-f(x) ~=~ \psi(x)
$$
are of the form $f(x)=c+\Sigma\psi(x)$, where $c\in\R$.}
\end{quote}

\parag{Extended ID card} It is not difficult to see that
$$
\sigma[g] ~=~ \int_0^1\Sigma\psi(t+1){\,}dt ~=~ \frac{1}{2}(1-\ln(2\pi)).
$$
Hence we have the values
$$
\begin{array}{|c|c|c|}
\hline \overline{\sigma}[g] & \sigma[g] & \gamma[g] \\
\hline \emptybox \infty & \frac{1}{2}(1-\ln(2\pi)) &  \frac{1}{2}(1-\ln(2\pi)+\gamma) \\
\hline
\end{array}
$$

\begin{itemize}
\item \emph{Alternative representations of $\sigma[g]$}
\begin{eqnarray*}
\sigma[g] &=& -\frac{1}{2}\,\gamma -\sum_{k=1}^{\infty}\left(\ln k-\psi(k)-\frac{1}{2k}\right),\\
\sigma[g] &=& -\frac{1}{2}\,\gamma +\int_1^{\infty}\left(\{t\}-\frac{1}{2}\right)\psi_1(t){\,}dt,\\
\sigma[g] &=& \lim_{n\to\infty}\left(\left(n-\frac{1}{2}\right)\psi(n)-\ln\Gamma(n)-n+1\right).
\end{eqnarray*}

\item \emph{Alternative representations of $\gamma[g]$}
\begin{eqnarray*}
\sigma[g] &=& \int_1^{\infty}\left(\psi(\lfloor t\rfloor)-\psi(t)+\frac{1}{2\lfloor t\rfloor}\right)dt,\\
\sigma[g] &=& \int_1^{\infty}\left(\psi(\lfloor t\rfloor)-\psi(t)+\frac{\{t\}}{\lfloor t\rfloor}\right)dt.
\end{eqnarray*}

\item \emph{Generalized Binet's function}. For any $q\in\N^*$ and any $x>0$,
\begin{eqnarray*}
J^{q+1}[\Sigma\psi](x) &=& \Sigma\psi(x)-\frac{1}{2}(1-\ln(2\pi))-\ln\Gamma(x)+\frac{1}{2}\,\psi(x)\\
&& \null +\sum_{j=0}^{q-2}G_{j+2}(-1)^j\,\mathrm{B}(j+1,x),
\end{eqnarray*}
where $(x,y)\mapsto\mathrm{B}(x,y)$ is the beta function.

\item \emph{Analogue of Raabe's formula}
$$
\int_x^{x+1}\Sigma\psi(t){\,}dt ~=~ \frac{1}{2}(1-\ln(2\pi))+\ln\Gamma(x),\qquad x>0.
$$

\item \emph{Alternative characterization}. The function $f=\Sigma\psi$ is the unique solution lying in $\cC^0\cap\cK^1$ to the equation
$$
\int_x^{x+1}f(t){\,}dt ~=~ \frac{1}{2}(1-\ln(2\pi))+\ln\Gamma(x),\qquad x>0.
$$
\end{itemize}

\parag{Inequalities} The following inequalities hold for any $x>0$, any $a\geq 0$, and any $n\in\N^*$.
\begin{itemize}
\item \emph{Symmetrized generalized Wendel's inequality} (equality if $a\in\{0,1\}$)
\begin{eqnarray*}
|\Sigma\psi(x+a)-\Sigma\psi(x)-a\psi(x)| &\leq & |a-1|{\,}|\psi(x+a)-\psi(x)|\\
& \leq & \lceil a\rceil\,\frac{|a-1|}{x}{\,}.
\end{eqnarray*}

\item \emph{Symmetrized generalized Wendel's inequality} (discrete version)
$$
|\Sigma\psi(x)-f^1_n[\psi](x)| ~\leq ~ |x-1|{\,}|\psi(n+x)-\psi(n)| ~\leq ~ \lceil x\rceil\,\frac{|x-1|}{n}{\,},
$$
where
$$
f^1_n[\psi](x) ~=~ (n+x-1)(\psi(n)-\psi(x+n))+(x-1)\,\psi(x)+1.
$$

\item \emph{Symmetrized Stirling's formula-based inequalities}
\begin{eqnarray*}
\lefteqn{\left|\Sigma\psi\left(x+\frac{1}{2}\right)-\frac{1}{2}(1-\ln(2\pi))-\ln\Gamma(x)\right|}\\
&\leq & \left|\Sigma\psi(x)-\frac{1}{2}(1-\ln(2\pi))-\ln\Gamma(x)+\frac{1}{2}\,\psi(x)\right|\\
&\leq & x\ln x-\ln\Gamma(x)-\frac{1}{2}\,\psi(x)-x+\frac{1}{2}\ln(2\pi) ~\leq ~ \frac{1}{2x}{\,}.
\end{eqnarray*}

\item \emph{Generalized Gautschi's inequality}
\begin{eqnarray*}
(a-\lceil a\rceil)\,\psi(x+\lceil a\rceil) &\leq & (a-\lceil a\rceil){\,}(\Sigma\psi)'(x+\lceil a\rceil)\\
&\leq & (\Sigma\psi)(x+a)-(\Sigma\psi)(x+\lceil a\rceil)\\
&\leq & (a-\lceil a\rceil)\,\psi(x+\lfloor a\rfloor).
\end{eqnarray*}
\end{itemize}

\parag{Generalized Stirling's and related formulas} For any $a\geq 0$, we have the following limits and asymptotic equivalence as $x\to\infty$,
$$
\Sigma\psi(x+a)-\Sigma\psi(x)-a\psi(x) ~\to ~0,\qquad\Sigma\psi(x+a) ~\sim ~ \ln\Gamma(x),
$$
$$
\Sigma\psi(x)-\ln\Gamma(x)+\frac{1}{2}\,\psi(x) ~\to ~ \frac{1}{2}(1-\ln(2\pi)),
$$
$$
\Sigma\psi(x)-\ln\Gamma\left(x-\frac{1}{2}\right) ~\to ~ \frac{1}{2}(1-\ln(2\pi)).
$$

\parag{Asymptotic expansions} For any $q\in\N^*$ we have the following expansion as $x\to\infty$
$$
\Sigma\psi(x) ~=~ \frac{1}{2}(1-\ln(2\pi))+\sum_{k=0}^q\frac{B_k}{k!}\,\psi_{k-1}(x)+O(\psi_q(x)).
$$
Setting $q=3$ for instance, we get
$$
\Sigma\psi(x) ~=~ \frac{1}{2}(1-\ln(2\pi))+\ln\Gamma(x)-\frac{1}{2}\,\psi(x)+\frac{1}{12}\,\psi_1(x)+O(x^{-3}).
$$

\parag{Generalized Liu's formula} For any $x>0$, we have
$$
\Sigma\psi(x) ~=~ \frac{1}{2}(1-\ln(2\pi))+\ln\Gamma(x)-\frac{1}{2}\,\psi(x)-\int_0^{\infty}\left(\{t\}-\frac{1}{2}\right)\psi_1(x+t){\,}dt.
$$

\parag{Limit and series representations} Let us briefly examine the main limit and series representations of $\Sigma\psi(x)$. The additional representations obtained by differentiation and integration are left to the reader.
\begin{itemize}
\item \emph{Eulerian and Weierstrassian forms}. We have
$$
\Sigma\psi(x) ~=~ -\gamma x-\psi(x)-\sum_{k=1}^{\infty}\left(\psi(x+k)-\psi(k)-\frac{x}{k}\right),
$$
$$
\Sigma\psi(x) ~=~ -(1+\gamma) x-\psi(x)-\sum_{k=1}^{\infty}\left(\psi(x+k)-\psi(k)-x\,\psi_1(k)\right).
$$
\item \emph{Analogue of Gauss' limit}. We have
$$
\Sigma\psi(x) ~=~ (x-1)\,\psi(x)+1 + \lim_{n\to\infty}(n+x-1)(\psi(n)-\psi(x+n)).
$$
\end{itemize}

\parag{Gregory's formula-based series representation} For any $x>0$ we have
$$
\Sigma\psi(x) ~=~ \frac{1}{2}(1-\ln(2\pi))+\ln\Gamma(x)-\frac{1}{2}\,\psi(x) + \sum_{n=0}^{\infty}|G_{n+2}|\,\mathrm{B}(n+1,x){\,}.
$$
Setting $x=1$ in this identity yields the following analogue of Fontana-Mascheroni’s series
$$
\sum_{n=2}^{\infty}\frac{|G_n|}{n-1} ~=~ -\frac{1}{2}+\frac{1}{2}\ln(2\pi)-\frac{1}{2}\,\gamma{\,},
$$
and the right-hand value is precisely the generalized Euler constant\index{Euler's constant!generalized} $\gamma[\psi]$ associated with the digamma function. We also observe that this latter identity was obtained by Kowalenko~\cite[p.~431]{Kow10}.

\parag{Analogue of Gauss' multiplication formula} Since we do not have any simple expression for the function $\Sigma_x\psi(\frac{x}{m})$, it seems difficult to find a usable multiplication formula here. We had the same difficulty in the investigation of the Barnes $G$-function (see Section~\ref{sec:Barnes558}). However, we can use Proposition~\ref{prop:8Stir44Gau7Mult} to derive the following convergence result. For any $m\in\N^*$ we have
$$
\sum_{j=0}^{m-1}\Sigma\psi\left(\frac{x+j}{m}\right)-m\,\ln\Gamma\left(\frac{x}{m}\right)
+\frac{1}{2}\,\psi\left(\frac{x}{m}\right) ~\to ~ \frac{m}{2}{\,}(1-\ln(2\pi))\quad\text{as $x\to\infty$}.
$$

\parag{Analogue of Wallis's product formula} The following analogue of Wallis's formula was already found in Project~\ref{app:WallPsi5}
$$
\lim_{n\to\infty}\left(-\ln(4n)+2\sum_{k=1}^{2n}(-1)^k\psi(k)\right) ~=~ \gamma{\,}.
$$

\parag{Generalized Webster's functional equation} For any $m\in\N^*$, there is a unique eventually monotone solution $f\colon\R_+\to\R$ to the equation
$$
\sum_{j=0}^{m-1}f\left(x+\frac{j}{m}\right) ~=~ \psi(x)
$$
namely
$$
f(x) ~=~ \Sigma\psi\left(x+\frac{1}{m}\right)-\Sigma\psi(x).
$$

\parag{Analogue of Euler's series representation of $\gamma$} We have $(\Sigma\psi)'(1)=-1-\gamma$ and
$$
(\Sigma\psi)^{(k)}(1) ~=~ k\,\psi_{k-1}(1) ~=~ (-1)^k k!\,\zeta(k),\qquad k\geq 2.
$$
The Taylor series expansion of $\Sigma\psi(x+1)$ about $x=0$ is
$$
\Sigma\psi(x+1) ~=~ (-1-\gamma)x+\sum_{k=2}^{\infty}\zeta(k)(-x)^k,\qquad |x|<1.
$$
Integrating both sides of this equation on $(0,1)$, we obtain
$$
\sum_{k=2}^{\infty}(-1)^k\frac{\zeta(k)}{k+1} ~=~ 1+\frac{1}{2}(\gamma -\ln(2\pi)).
$$

\parag{Analogue of the reflection formula} For any $x\in\R\setminus\Z$, we have
$$
\Sigma\psi(1+x)+\Sigma\psi(1-x) ~=~ 1-\pi x\cot(\pi x).
$$

\index{principal indefinite sum!of the digamma function|)}
\index{digamma function!principal indefinite sum|)}

\section{The PIS of the Hurwitz zeta function}
\index{Hurwitz zeta function!principal indefinite sum|(}
\index{principal indefinite sum!of the Hurwitz zeta function|(}


In this section we apply our theory to investigate the function
$$
x ~\mapsto ~\zeta_2(s,x) ~\stackrel{\text{def}}{=} ~ \Sigma_x\,\zeta(s,x)
$$
for any fixed $s\in\R\setminus\{1\}$.

Using summation by parts, we observe that if $s\neq 2$ we have
$$
\zeta_2(s,x) ~=~ (x-1)\,\zeta(s,x)-\zeta(s-1,x)+\zeta(s-1).
$$
If $s=2$, then
$$
\zeta_2(2,x) ~=~ \Sigma_x\psi_1(x) ~=~ (x-1)\,\psi_1(x)+\psi(x)+\gamma.
$$

To keep this investigation simple, here we focus on some selected results only and we restrict ourselves to the case when $s>2$, for which the sequence $n\mapsto\zeta(s,n)$ is summable. In this case, by \eqref{eq:diffzz092} we obtain immediately the following surprising identity (see also Paris \cite{Par21})
$$
\sum_{k=1}^{\infty}\zeta(s,k) ~=~ \zeta(s-1).
$$
We also have
$$
\int_1^{\infty}\zeta(s,t){\,}dt ~=~ \frac{\zeta(s-1)}{s-1}{\,}.
$$

\parag{ID card} We can easily summarize the basic information as follows:
$$
\begin{array}{|c|c|c|c|}
\hline g_s(x) & \text{Membership} & \deg g_s & \Sigma g_s(x) \\
\hline \emptybox \zeta(s,x) & \emptybox\cC^{\infty}\cap\cD^{-1}\cap\cK^{\infty} & -1 & \zeta_2(s,x) \\
\hline
\end{array}
$$

\parag{Analogue of Bohr-Mollerup's theorem} The function $\zeta_2(s,x)$ can be characterized as follows.
\begin{quote}
\emph{All eventually monotone solutions $f_s\colon\R_+\to\R$ to the equation
$$
f_s(x+1)-f_s(x) ~=~ \zeta(s,x)
$$
are of the form $f_s(x)=c_s+\zeta_2(s,x)$, where $c_s\in\R$.}
\end{quote}

\parag{Extended ID card} We immediately have
$$
\sigma[g_s] ~=~ \sum_{k=1}^{\infty}\zeta(s,k)-\int_1^{\infty}\zeta(s,t){\,}dt ~=~ \frac{s-2}{s-1}\,\zeta(s-1).
$$
Hence we have the values
$$
\begin{array}{|c|c|c|}
\hline \overline{\sigma}[g_s] & \sigma[g_s] & \gamma[g_s] \\
\hline \emptybox \infty & \frac{s-2}{s-1}\,\zeta(s-1) & \gamma[g_s]=\sigma[g_s] \\
\hline
\end{array}
$$

\begin{itemize}
\item \emph{Alternative representations of $\sigma[g_s]=\gamma[g_s]$}
\begin{eqnarray*}
\sigma[g_s] &=& \int_0^1\zeta_2(s,t+1){\,}dt ~=~ \int_1^{\infty}\left(\zeta(s,\lfloor t\rfloor)-\zeta(s,t)\right){\,}dt{\,},\\
\sigma[g_s] &=& \frac{1}{2}\,\zeta(s)+s\,\int_1^{\infty}\left(\frac{1}{2}-\{t\}\right)\zeta(s+1,t){\,}dt.
\end{eqnarray*}

\item \emph{Analogue of Raabe's formula}
$$
\int_x^{x+1}\zeta_2(s,t){\,}dt ~=~ \zeta(s-1)-\frac{\zeta(s-1,x)}{s-1}{\,},\qquad x>0.
$$
\end{itemize}

\parag{Inequalities and asymptotic analysis} For any $a\geq 0$ and any $x>0$, we have
\begin{eqnarray*}
\left|\zeta_2(s,x+a)-\zeta_2(s,x)\right| &\leq & \lceil a\rceil{\,}\zeta(s,x){\,},\\
\left|\zeta_2(s,x)-\zeta(s-1)+\frac{\zeta(s-1,x)}{s-1}\right| &\leq & \zeta(s,x).
\end{eqnarray*}
In particular, we have
$$
\zeta_2(s,x) ~\to ~\zeta(s-1)\qquad\text{as $x\to\infty$}.
$$

\parag{Generalized Liu's formula} For any $x>0$ we have
\begin{eqnarray*}
\zeta_2(s,x) &=& \zeta(s-1)-\frac{\zeta(s-1,x)}{s-1}-\frac{1}{2}\,\zeta(s,x)\\
&& \null +s\,\int_0^{\infty}\left(\{t\}-\frac{1}{2}\right)\zeta(s+1,x+t){\,}dt.
\end{eqnarray*}

\parag{Eulerian and Weierstrassian forms} For any $x>0$, we have
$$
\zeta_2(s,x) ~=~ \zeta(s-1)-\sum_{k=0}^{\infty}\zeta(s,x+k).
$$
and this series converges uniformly on $\R_+$ and can be integrated and differentiated term by
term.

\parag{Gregory's formula-based series representation} For any $x>0$ we have
\begin{eqnarray*}
\zeta_2(s,x) &=& \zeta(s-1)-\frac{\zeta(s-1,x)}{s-1}-\sum_{n=0}^{\infty}G_{n+1}\,\Delta_x^n\zeta(s,x)\\
&=& \zeta(s-1)-\frac{\zeta(s-1,x)}{s-1}-\sum_{n=0}^{\infty}|G_{n+1}|\,\sum_{k=0}^n(-1)^k\tchoose{n}{k}\zeta(s,x+k){\,}.
\end{eqnarray*}
Setting $x=1$ in this identity yields the analogue of Fontana-Mascheroni series
$$
\sum_{n=0}^{\infty}|G_{n+1}|\,\sum_{k=0}^n(-1)^k\tchoose{n}{k}\zeta(s,k+1) ~=~ \frac{s-2}{s-1}\,\zeta(s-1){\,}.
$$

\parag{Analogue of Wallis's product formula} The analogue of Wallis's formula is
\begin{eqnarray*}
\sum_{k=1}^{\infty}(-1)^{k-1}\zeta(s,k) &=& (2-2^{1-s})\zeta(s)+(1-2^{1-s})\zeta(s-1)\\
&& \null -2^{1-s}\,\sum_{k=0}^{\infty}\zeta\left(s,k+\frac{1}{2}\right).
\end{eqnarray*}
This formula is actually obtained by combining Proposition~\ref{prop:SigTr62} with the duplication formula for the Hurwitz zeta function
$$
2\,\zeta(s,2x) ~=~ 2^{1-s}\zeta(s,x)+2^{1-s}\zeta\left(s,x+\frac{1}{2}\right).
$$
On the other hand we also have (see Paris \cite{Par21})
$$
\sum_{k=1}^{\infty}(-1)^{k-1}\zeta(s,k) ~=~ (1-2^{-s})\,\zeta(s),\qquad s>1.
$$
Combining this formula with the analogue of Wallis's formula, we derive the following identity
$$
\sum_{k=0}^{\infty}\zeta\left(s,k+\frac{1}{2}\right) ~=~ (2^{s-1}-2^{-1})\,\zeta(s)+(2^{s-1}-1)\,\zeta(s-1).
$$

\parag{Taylor series expansion} We have
$$
(\Sigma g_s)^{(k)}(1) ~=~ -k!\,\tchoose{-s}{k}\,\zeta(s+k-1),\qquad k\in\N^*.
$$
The Taylor series expansion of $\zeta_2(s,x+1)$ about $x=0$ is
$$
\zeta_2(s,x+1) ~=~ -\sum_{k=1}^{\infty}\tchoose{-s}{k}\,\zeta(s+k-1){\,}x^k{\,},\qquad |x|<1.
$$

\index{principal indefinite sum!of the Hurwitz zeta function|)}
\index{Hurwitz zeta function!principal indefinite sum|)}

\section{The PIS of the generating function for the Gregory coefficients}
\label{sec:PISGre47}
\index{principal indefinite sum!of the generating function for the Gregory coefficients|(}
\index{Gregory coefficients}

Let us investigate the function $\Sigma h_p$ for any $p\in\N^*$, where $h_p\colon\R_+\to\R$ is defined by the equation
$$
h_p(x) ~=~ \frac{x^p}{\ln(x+1)} ~=~ x^p\,\mathrm{li}'(x+1)\qquad\text{for $x>0$}
$$
and $\mathrm{li}(x)$ is the logarithmic integral function\index{logarithmic integral function} defined for all positive real numbers $x\neq 1$ by the integral
$$
\mathrm{li}(x) ~=~ \int_0^x\frac{1}{\ln t}{\,}dt{\,}.
$$
Incidentally, when $p=1$, this function reduces to the ordinary generating function for the sequence $n\mapsto G_n$. That is,
$$
h_1(x) ~=~ \sum_{n=0}^{\infty}G_n{\,}x^n,\qquad |x|<1.
$$
More generally, $h_p(x)=x^{p-1}h_1(x)$ is the ordinary generating function for the right-shifted sequence $n\mapsto G_{n-p+1}$, that is the sequence
$$
0,\ldots,0,G_0,G_1,G_2,\ldots
$$
with $p-1$ leading $0$'s.

We also note that the function $h_p$ has the following integral representation
$$
h_p(x) ~=~ x^{p-1}\int_0^1(x+1)^s{\,}ds.
$$
This latter representation actually suggests introducing, for any $p\in\N^*$, the function $g_p\colon\R_+\to\R$ defined by the equation
$$
g_p(x) ~=~ \int_0^1(x+1)^{s+p-1}{\,}ds ~=~ \frac{x(x+1)^{p-1}}{\ln(x+1)}\qquad\text{for $x>0$}.
$$
The conversion formulas between the $h_p's$ and the $g_p's$ are simply given by the following equations
\begin{eqnarray*}
g_p(x) &=& \sum_{k=1}^p\tchoose{p-1}{k-1}{\,}h_k(x){\,},\\
h_p(x) &=& \sum_{k=1}^p{\,}(-1)^{p-k}\tchoose{p-1}{k-1}{\,}g_k(x){\,}.
\end{eqnarray*}
In particular, we have $g_1=h_1$.

Since the function $g_p$ has a nicer integral form than $h_p$, for the sake of simplicity we will investigate the function $\Sigma g_p$ for any $p\in\N^*$. By Proposition~\ref{prop:gStHa}, the function $\Sigma h_p$ can then be obtained by applying the operator $\Sigma$ to both sides of the second conversion formula above.

\begin{remark}
We observe that the function $g_p$ is also the ordinary generating function for the sequence $n\mapsto\psi_n(p-1)$, where $\psi_n$ is the $n$th degree Bernoulli polynomial of the second kind\index{Bernoulli polynomials!of the second kind} (see Section~\ref{sec:BP2k27}).
\end{remark}

\parag{ID card} It is not difficult to see that both $g_p$ and $h_p$ lie in $\cC^{\infty}\cap\cD^p\cap\cK^{\infty}$ and hence also in $\cK^p$. We also have $\deg g_p=\deg h_p=p-1$.

From the integral form of $g_p$ above, we can easily derive the following explicit form of $\Sigma g_p$ (after replacing $1-s$ with $s$ in the integral)
$$
\Sigma g_p(x) ~=~ \int_0^1\zeta(s-p,2){\,}ds-\int_0^1\zeta(s-p,x+1){\,}ds,
$$
that is,
$$
\Sigma g_p(x) ~=~ \tau_p-\int_0^1\zeta(s-p,x+1){\,}ds{\,},
$$
with
$$
\tau_p ~=~ -1+\int_0^1\zeta(s-p){\,}ds{\,},
$$
where $\zeta(s,x)$ is the Hurwitz zeta function.

\begin{remark}
For any integer $n\geq 2$, the \emph{harmonic number function of order $n$}\index{harmonic number function!of order $n$} is defined on $(-1,\infty)$ by
$$
x ~\mapsto ~ H^{(n)}_x ~=~ \zeta(n)-\zeta(n,x+1),
$$
see, e.g., Srivastava and Choi \cite[p.~266]{SriCho12}. Extending this definition to noninteger orders by writing
$$
H_x^{(s)} ~=~ \zeta(s)-\zeta(s,x+1),\qquad s\in\R\setminus\{1\},
$$
we obtain the following very compact integral representation
$$
\Sigma g_p(x) ~=~ -1+\int_0^1H_x^{(s-p)}{\,}ds{\,},\qquad x>0.\qedhere
$$
\end{remark}

\parag{Analogue of Bohr-Mollerup's theorem} Thus defined, $\Sigma h_p$ is a $\log\Gamma_p$-type function that lies in $\cC^{\infty}\cap\cD^{p+1}\cap\cK^{\infty}$. This function can be characterized as follows.
\begin{quote}
\emph{All solutions $f\colon\R_+\to\R$ to the equation $\Delta f=h_p$ that lie in $\cK^p$ are of the form
$$
f(x) ~=~ c_p + \sum_{k=1}^p{\,}(-1)^{p-k}\tchoose{p-1}{k-1}{\,}\Sigma g_k(x){\,},
$$
where $c_p\in\R$.}
\end{quote}

\parag{Extended ID card} Let us compute the asymptotic constant associated with the function $g_p$. We have
\begin{eqnarray*}
\sigma[g_p] &=& \int_0^1\Sigma g_p(t+1){\,}dt ~=~ \tau_p-\int_0^1\!\int_0^1\zeta(s-p,t+2){\,}dt{\,}ds\\
&=& \tau_p +\int_0^1\frac{2^{s+p}}{s+p}{\,}ds{\,}.
\end{eqnarray*}
Using the change of variable $u=2^{s+p}$, we finally obtain
$$
\sigma[g_p] ~=~ \tau_p + \int_{2^p}^{2^{p+1}}\frac{1}{\ln t}{\,}dt ~=~ \tau_p +\mathrm{li}(2^{p+1})-\mathrm{li}(2^p).
$$
Now, we have
\begin{eqnarray*}
\int_1^xg_p(t){\,}dt &=& \int_0^1\frac{(x+1)^{s+p}-2^{s+p}}{s+p}{\,}ds\\ &=& \mathrm{li}((x+1)^{p+1})-\mathrm{li}((x+1)^p)-\mathrm{li}(2^{p+1})+\mathrm{li}(2^p)
\end{eqnarray*}
and hence the analogue of Raabe's formula is
$$
\int_x^{x+1}\Sigma g_p(t){\,}dt ~=~ \tau_p +\mathrm{li}((x+1)^{p+1})-\mathrm{li}((x+1)^p){\,},\qquad x>0.
$$

\parag{Generalized Stirling's and related formulas when $p=1$} For any $a\geq 0$, we have the following limits and
asymptotic equivalence as $x\to\infty$,
$$
\Sigma g_1(x+a)-\Sigma g_1(x)-a{\,}\frac{x}{\ln(x+1)} ~\to ~ 0,
$$
$$
\Sigma g_1(x) -\mathrm{li}((x+1)^2)+\mathrm{li}(x+1)+\frac{x}{2\ln(x+1)}~\to ~ \tau_1{\,},
$$
$$
\Sigma g_1(x+a) ~\sim ~ \mathrm{li}((x+1)^2)-\mathrm{li}(x+1).
$$
Upon differentiation,
$$
D\Sigma g_1(x)-\frac{x-\frac{1}{2}}{\ln(x+1)} ~\to ~0,\qquad D^{k+1}\Sigma g_1(x) ~\to ~0,\quad k\in\N^*,
$$
$$
D\Sigma g_1(x+a) ~\sim ~ \frac{x}{\ln(x+1)}{\,},
$$
where
$$
D\Sigma g_1(x) ~=~ \int_0^1(s-1)\,\zeta(s,x+1){\,}ds.
$$

\parag{Limit and series representations when $p=1$} The Eulerian and Weierstrassian forms are
$$
\Sigma g_1(x) ~=~ -g_1(x)+x{\,}g_1(1)-\sum_{k=1}^{\infty}\left(g_1(x+k)-g(k)-x\,\Delta_kg_1(k)\right)
$$
and
$$
\Sigma g_1(x) ~=~ -g_1(x)+x{\,}D\Sigma g_1(1)-\sum_{k=1}^{\infty}\left(g_1(x+k)-g(k)-x{\,}g'_1(k)\right),
$$
where
$$
D\Sigma g_1(1) ~=~ \int_0^1(s-1)\,\zeta(s,2){\,}ds ~=~ \frac{1}{2}-\int_0^1s\,\zeta(1-s){\,}ds.
$$

\parag{Gregory's formula-based series representation when $p=1$} Proposition~\ref{prop:6699rem} provides the following series representation: for any $x>0$ we have
\begin{eqnarray*}
\Sigma g_1(x) &=& \tau_1 +\mathrm{li}((x+1)^2)-\mathrm{li}(x+1)-\sum_{n=0}^{\infty}G_{n+1}\,\Delta^ng(x)\\
&=& \tau_1 +\mathrm{li}((x+1)^2)-\mathrm{li}(x+1)-\sum_{n=0}^{\infty}|G_{n+1}|\,\sum_{k=0}^n(-1)^k\tchoose{n}{k}\,\frac{x+k}{\ln(x+k+1)}{\,}.
\end{eqnarray*}
Setting $x=1$ in this identity, we obtain the following analogue of Fontana-Mascheroni’s series
$$
\sigma[g_1] ~=~ \tau_1 +\mathrm{li}(4)-\mathrm{li}(2) ~=~ \sum_{n=0}^{\infty}|G_{n+1}|\,\sum_{k=0}^n(-1)^k\tchoose{n}{k}\,\frac{k+1}{\ln(k+2)}{\,}.
$$

\parag{Analogue of Gauss' multiplication formula} For any $m\in\N^*$ and any $x>0$, we have
$$
\sum_{j=0}^{m-1}\Sigma g_p\left(x+\frac{j}{m}\right) ~=~ m{\,}\tau_p-\int_0^1\sum_{j=0}^{m-1}\zeta\left(s-p,x+1+\frac{j}{m}\right)ds.
$$
Using the multiplication formula for the Hurwitz zeta function, we then obtain the following analogue of Gauss' multiplication formula
$$
\sum_{j=0}^{m-1}\Sigma g_p\left(x+\frac{j}{m}\right) ~=~ m{\,}\tau_p-\int_0^1m^{s-p}\,\zeta(s-p,mx+m){\,}ds.
$$

Now, using \eqref{eq:MultThmGeneq} we obtain
\begin{eqnarray*}
\Sigma_x{\,} g_p\left(\frac{x}{m}\right) &=& \sum_{j=0}^{m-1}\Sigma g_p\left(\frac{x+j}{m}\right) - \sum_{j=1}^m\Sigma g_p\left(\frac{j}{m}\right)\\
&=& \int_0^1m^{s-p}\left(\zeta(s-p,m+1)-\zeta(s-p,x+m)\right)ds.
\end{eqnarray*}
Corollary~\ref{cor:Riem482} then tells us that the sequences
$$
m ~\mapsto ~ \int_0^1m^{s-p-1}\left(\zeta(s-p,2m)-\zeta(s-p,mx+m)\right)ds
$$
and
$$
m ~\mapsto ~ \int_0^1m^{s-p-1}\left(\zeta(s-p,m+1)-\zeta(s-p,mx+m)\right)ds
$$
converge to the integrals
$$
\int_1^x g_p(t){\,}dt \qquad\text{and}\qquad \int_0^x g_p(t){\,}dt{\,},
$$
respectively.

\index{principal indefinite sum!of the generating function for the Gregory coefficients|)}

\chapter{Further examples}
\label{chapter:12}

The scope of applications of our theory is very wide since it applies to any function lying in the domain of the map $\Sigma$. In Chapter~\ref{chapter:10}, we made a thorough study of some standard special functions. In Chapter~\ref{chapter:11}, we defined and investigated new functions as principal indefinite sums of known functions. In the present chapter, we briefly discuss further examples that the reader may want to explore in more detail.

\section{The multiple gamma functions}

The multiple gamma functions\index{multiple gamma function} introduced in Section~\ref{subsec:MLG-t} can also be studied through the sequence of functions $G_0, G_1,\ldots$, defined by (see Srivastava and Choi \cite[p.~56]{SriCho12})
$$
G_p(x) ~=~ \Gamma_p(x)^{(-1)^{p-1}},\qquad p\in\N.
$$
Equivalently, we have $G_0(x)=x$ and
$$
\ln G_p(x) ~=~ \Sigma\ln G_{p-1}(x)\qquad\text{for all $p\in\N^*$}.
$$
Clearly, the function $\ln G_{p-1}(x)$ lies in $\cC^{\infty}\cap\cD^p\cap\cK^{\infty}$ and we have $\deg(\ln\circ G_p)=p$. Moreover, this sequence of functions can naturally be extended to $p=-1$ by defining
$$
G_{-1}(x) ~=~ 1+\frac{1}{x}{\,}.
$$

Just as for the gamma function and the Barnes $G$-function,\index{Barnes's $G$-function} we can derive the following asymptotic equivalence: for any $a\geq 0$,
$$
G_p(x+a) ~\sim ~ \prod_{j=0}^pG_{p-j}(x)^{{a\choose j}} \qquad\text{as $x\to\infty$},
$$
with equality if $a\in\{0,1,\ldots,p\}$. We also have the following product representation
$$
G_p(x) ~=~ \frac{1}{G_{p-1}(x)}{\,}
\prod_{k=1}^{\infty}\frac{G_{p-1}(k)}{G_{p-1}(x+k)}{\,}G_{p-2}(k)^x G_{p-3}(k)^{{x\choose 2}}{\,}\cdots {\,} G_{-1}(k)^{{x\choose p}}
$$
and the recurrence formula
$$
\ln G_p(x) ~=~ -(x-1){\,}\sigma[D\ln\circ G_{p-1}]+\int_1^x\Sigma D\ln G_{p-1}(t){\,}dt.
$$
For example, one can show that
\begin{multline*}
\ln G_3(x) ~=~ -\frac{1}{8}{\,}x(x-1)(2x-5)+\frac{1}{4}{\,}x(x-2)\ln(2\pi)+\tchoose{x-1}{2}\ln\Gamma(x)\\
\null -\frac{1}{2}(2x-3){\,}\psi_{-2}(x)+\psi_{-3}(x)-x\,\psi_{-3}(1).
\end{multline*}
This latter formula can also be established using the characterization of $G_3$ as a $3$-convex solution to the equation $\Delta f(x)=\ln G_2(x)$.

\section{The regularized incomplete gamma function}
\index{regularized incomplete gamma function}

Consider the $2$-variable function $Q(x,s)=\Gamma(x,s)/\Gamma(x)$ on $\R_+^2$, where $\Gamma(x,s)$ is the upper incomplete gamma function. Thus defined, the function $Q(x,s)$ satisfies the difference equation
$$
Q(x+1,s)-Q(x,s) ~=~ \frac{e^{-s}s^x}{\Gamma(x+1)}{\,}.
$$
For any $s>0$, we define the function $g_s\colon\R_+\to\R$ by
$$
g_s(x) ~=~ \frac{e^{-s}s^x}{\Gamma(x+1)}{\,}.
$$
This function lies in $\cC^{\infty}\cap\cD^{-1}\cap\cK^{\infty}$ and has the property that $\Sigma g_s(x)=Q(x,s)-e^{-s}$. We also note that the Eulerian form of $Q(x,s)$ is
\begin{eqnarray*}
Q(x,s) &=& 1-\sum_{k=0}^{\infty}g_s(x+k) ~=~ 1-\frac{e^{-s}s^x}{\Gamma(x+1)}\,\sum_{k=0}^{\infty}\frac{\Gamma(x+1)}{\Gamma(x+k+1)}{\,}s^k\\
&=& 1- \frac{e^{-s}s^x}{\Gamma(x+1)}\,\sum_{k=0}^{\infty}x^{\underline{-k}}{\,}s^k{\,},
\end{eqnarray*}
where $x^{\underline{-k}}=\Gamma(x+1)/\Gamma(x+k+1)$ for any $k\in\N$.

%

\section{The error function}

Recall that the Gauss error function\index{Gauss error function} $\mathrm{erf}(x)$ is defined by the equation
$$
\mathrm{erf}(x) ~=~ \frac{2}{\sqrt{\pi}}\,\int_0^xe^{-t^2}{\,}dt\qquad\text{for $x>0$}.
$$
To study this function, we could for instance work with the function $g(x)=\Delta\mathrm{erf}(x)$. Instead, let us consider the function $g\colon\R_+\to\R$ defined by the equation
$$
g(x) ~=~ \frac{2}{\sqrt{\pi}}{\,}e^{-x^2}\qquad\text{for $x>0$}.
$$
It clearly lies in $\cC^{\infty}\cap\cD^{-1}\cap\cK^{\infty}$. Thus, the Eulerian form of $\Sigma g$ is given by the identity
$$
\Sigma g(x) ~=~ \frac{2}{\sqrt{\pi}}\,\sum_{k=0}^{\infty}(e^{-(k+1)^2}-e^{-(k+x)^2}).
$$
The generalized Stirling formula yields the following limit
$$
\mathrm{erf}(x) + \frac{2}{\sqrt{\pi}}\,\sum_{k=0}^{\infty}e^{-(k+x)^2} ~\to ~ 1\qquad\text{as $x\to\infty$}.
$$
Incidentally, the analogue of Legendre's duplication formula provides the surprising identity
$$
\sum_{k=0}^{\infty}(e^{-(k+1)^2}-e^{-(k+\frac{x}{2})^2}-e^{-(k+\frac{x+1}{2})^2}+e^{-(k+\frac{1}{2})^2}
-e^{-(\frac{k+1}{2})^2}+e^{-(\frac{k+x}{2})^2}) ~=~ 0.
$$

\section{The exponential integral}

Recall that the exponential integral\index{exponential integral} $E_1(x)$ is defined by the equation
$$
E_1(x) ~=~ \int_x^{\infty}\frac{e^{-t}}{t}{\,}dt\qquad\text{for $x>0$}.
$$
Similarly to the previous example, let us consider the function $g\colon\R_+\to\R$ defined by the equation
$$
g(x) ~=~ \frac{e^{-x}}{x}\qquad\text{for $x>0$}.
$$
It lies in $\cC^{\infty}\cap\cD^{-1}\cap\cK^{\infty}$. Thus, the Eulerian form of $\Sigma g$ is given by the identity
$$
\Sigma g(x) ~=~ \sum_{k=0}^{\infty}\left(\frac{e^{-(k+1)}}{k+1}-\frac{e^{-(k+x)}}{k+x}\right).
$$
The generalized Stirling formula easily provides the following convergence result
$$
E_1(x) -\sum_{k=0}^{\infty}\frac{e^{-(k+x)}}{k+x} ~\to ~ 0\qquad\text{as $x\to\infty$}.
$$
Moreover, the analogue of Raabe's formula is
$$
\int_x^{x+1}\Sigma g(t){\,}dt ~=~ 1-\ln(e-1)-E_1(x),\qquad x>0.
$$

\section{The hyperfactorial function}
\label{sec:hypFacF4}

The hyperfactorial function\index{hyperfactorial function|textbf} (or $K$-function) is the function $K\colon\R_+\to\R_+$ defined by the equation $\ln K=\Sigma g$, where the function $g(x)=x\ln x$ lies in $\cC^{\infty}\cap\cD^2\cap\cK^{\infty}$. Since we also have
$$
g(x) ~=~ x+\Delta \psi_{-2}(x)-\psi_{-2}(1),
$$
we immediately derive (see also Example~\ref{ex:xlnx52})
$$
\ln K(x) ~=~ \Sigma g(x) ~=~ \tchoose{x}{2}+\psi_{-2}(x)-x\,\psi_{-2}(1) ~=~ (x-1)\ln\Gamma(x)-\ln G(x).
$$
Actually, $g$ also corresponds to the special case when $(s,q)=(-1,1)$ of the function $g_{s,q}$ investigated in Section~\ref{sec:HDHu4Z8F3}. Thus, we also have
$$
\Sigma g(x) ~=~ \zeta'(-1,x)-\zeta'(-1),
$$
where $\zeta'(-1)=\frac{1}{12}-\ln A$. Finally, we note that the integer sequence $n\mapsto K(n)$ is the sequence A002109 in the OEIS \cite{Slo}.

\section{The Hurwitz-Lerch transcendent}

The Hurwitz-Lerch transcendent\index{Hurwitz-Lerch transcendent} $\Phi(z,s,a)$ is a generalization of the Hurwitz zeta function defined as an analytic continuation of the series
$$
\sum_{k=0}^{\infty}z^k(a+k)^{-s}
$$
when $|z|<1$ and $a\in\mathbb{C}\setminus(-\N)$ (see, e.g., Srivastava and Choi \cite[p.~194]{SriCho12}). It satisfies the difference equation
$$
\Phi(z,s,a+1)-z^{-1}\Phi(z,s,a) ~=~ -z^{-1}a^{-s}.
$$
It follows that the modified function
$$
\overline{\Phi}(z,s,a) ~=~ -z^a\,\Phi(z,s,a)
$$
satisfies the difference equation
$$
\overline{\Phi}(z,s,a+1)-\overline{\Phi}(z,s,a) ~=~ z^aa^{-s}.
$$
Thus, for certain real values of $z$ and $s$, the restriction to $\R_+$ of the map $a\mapsto\overline{\Phi}(z,s,a)$ fits the assumptions of our theory. Its investigation is left to the reader.

\section{The Bernoulli polynomials}

Recall that, for any $n\in\N$, the $n$th degree Bernoulli polynomial $B_n(x)$ is defined by the equation\index{Bernoulli polynomials}
$$
B_n(x) ~=~ \sum_{k=0}^n\tchoose{n}{k}{\,}B_{n-k}{\,}x^k\qquad\text{for $x\in\R$},
$$
where $B_k$ is the $k$th Bernoulli number. This polynomial satisfies the difference equation
$$
B_n(x+1)-B_n(x) ~=~ n{\,}x^{n-1}.
$$
Thus, the function $g_n\colon\R_+\to\R$ defined by the equation $g_n(x)=n{\,}x^{n-1}$ for $x>0$ lies in $\cC^{\infty}\cap\cD^n\cap\cK^{\infty}$ and has the property that
$$
\Sigma g_n(x) ~=~ B_n(x)-B_n(1),
$$
that is, in view of \eqref{eq:HurwBern43}
$$
\Sigma g_n(x) ~=~ n\zeta(1-n)-n\zeta(1-n,x){\,},\qquad n\in\N^*.
$$
Thus, the $n$th degree Bernoulli polynomial can be characterized as follows.
\begin{quote}
\emph{All solutions $f_n\colon\R_+\to\R$ to the equation $f_n(x+1)-f_n(x)=n{\,}x^{n-1}$ that lie in $\cK^n$ are of the form $f_n(x)=c_n+B_n(x)$, where $c_n\in\R$.}
\end{quote}
Using the generalized Webster functional equation (Theorem~\ref{thm:FunctEq69}), we can also easily characterize the $n$th degree Euler polynomial $E_n(x)$, which is defined by the equation\index{Euler polynomials}
$$
E_n(x) ~=~ \frac{2^{n+1}}{n+1}\left(B_{n+1}\left(\frac{x+1}{2}\right)-B_{n+1}\left(\frac{x}{2}\right)\right).
$$
We then obtain the following statement.
\begin{quote}
\emph{All solutions $f_n\colon\R_+\to\R$ to the equation $f_n(x+1)+f_n(x)=2{\,}x^n$ that lie in $\cK^n$ are of the form $f_n(x)=c_n+E_n(x)$, where $c_n\in\R$.}
\end{quote}
Finally, we also easily retrieve the multiplication formula:
$$
\sum_{j=0}^{m-1}B_n\left(\frac{x+j}{m}\right) ~=~ \frac{1}{m^{n-1}}{\,}B_n(x){\,}\qquad x>0.
$$


\section{The Bernoulli polynomials of the second kind}
\label{sec:BP2k27}

For any $n\in\N$, the $n$th degree Bernoulli polynomial of the second kind\index{Bernoulli polynomials!of the second kind|textbf} is defined by the equation
$$
\psi_n(x) ~=~ \int_x^{x+1}\tchoose{t}{n}{\,}dt\qquad\text{for $x>0$}.
$$
In particular, we have $\psi_n(0)=G_n$. Also, these polynomials satisfy the difference equation
$$
\psi_{n+1}(x+1)-\psi_{n+1}(x) ~=~ \psi_n(x).
$$
Thus, the function $g_n\colon\R_+\to\R$ defined by the equation $g_n(x)=\psi_n(x)$ for $x>0$ lies in $\cC^{\infty}\cap\cD^{n+1}\cap\cK^{\infty}$ and has the property that
$$
\Sigma g_n(x) ~=~ \psi_{n+1}(x)-\psi_{n+1}(1).
$$
Thus, the Bernoulli polynomials of the second kind can be characterized as follows.
\begin{quote}
\emph{All solutions $f_n\colon\R_+\to\R$ to the equation $f_n(x+1)-f_n(x)=\psi_n(x)$ that lie in $\cK^{n+1}$ are of the form $f_n(x)=c_n+\psi_{n+1}(x)$, where $c_n\in\R$.}
\end{quote}

\chapter{Conclusion}
\label{chapter:13}

Krull-Webster's theory offered an elegant extension of Bohr-Mollerup's theorem and has proved to be a very nice and useful contribution to the resolution of the difference equation $\Delta f=g$ on the real half-line $\R_+$. In this book, we have provided a significant generalization of Krull-Webster's theory by considerably relaxing the asymptotic condition imposed on the function $g$, and we have demonstrated through various examples how this generalization provides a unified framework to investigate the properties of many functions. This framework has indeed enabled us to derive several general formulas that now constitute a powerful toolbox and even a genuine Swiss Army knife to investigate a large variety of functions.

The key point of this generalization was the discovery of expression \eqref{eq:defFpn} for the sequence $n\mapsto f^p_n[g](x)$ for any $p\in\N$. We also observe that our uniqueness and existence results strongly rely on Lemma~\ref{lemma:VarEpsIneq} together with identities \eqref{eq:ff01flam} and \eqref{eq:fngsum}. These results actually constitute the common core and even the fundamental cornerstone of all the subsequent formulas that we derived in this book. For instance, the generalized Stirling formula \eqref{eq:dgf7dds} has been obtained almost miraculously by merely integrating both sides of the inequality given in Lemma~\ref{lemma:VarEpsIneq} (see Remark~\ref{rem:WenSti41}). Similarly, Gregory's summation formula \eqref{eq:GregoryMN} has been derived instantly by integrating both sides of identity \eqref{eq:fngsum}, and we have shown how its remainder can be controlled using Lemma~\ref{lemma:VarEpsIneq} again.

Our results clearly shed light on the way many of the classical special functions, such as the polygamma functions and the higher order derivatives of the Hurwitz zeta function, can be systematically studied, sometimes by deriving identities and formulas almost mechanically.

Beyond this systematization aspect, our theory has enabled us to introduce a number of new important and useful objects. For instance, the map $\Sigma$ itself is a new concept that appears to be as fundamental as the basic antiderivative operation (cf.\ Definition~\ref{de:PIS43}). Both concepts are actually strongly related through, e.g., Propositions~\ref{prop:Burnside0}, \ref{prop:conv6v6}, and \ref{prop:8Raab0036}. Other concepts such as the asymptotic constant and the generalized Binet function also appear to be new fundamental objects that merit further study. For instance, it is remarkable that the asymptotic constant appears not only in the generalized Stirling formula (Theorem~\ref{thm:dgf7dds}), but also in many other important formulas, such as the generalized Euler constant (Proposition~\ref{prop:linksSG46}), the Weierstrassian form (Theorem~\ref{thm:Weierst}), the analogue of Raabe's formula (Proposition~\ref{prop:8Raab0036}), the analogue of Gauss' multiplication formula (Proposition~\ref{prop:CalcSum662}), the asymptotic expansion (Proposition~\ref{prop:Richardson9c}), and the generalized Liu formula (Proposition~\ref{prop:Liu471}).

Our work has also revealed how important and natural are the higher order convexity properties. Although these properties seem to be still poorly used in mathematical analysis, they actually constitute an essential and highly useful ingredient in the development of our theory and therefore also merit further investigation (see, e.g., Proposition~\ref{thm:intCpTpKp} and Remark~\ref{rem:I5mpPConv63}).

In conclusion, the results that we have obtained as well as the new concepts that we have introduced and explored show that this area of investigation is very rich and intriguing. We have just skimmed the surface, and there are a lot of questions that emerge naturally. We now list some below.

\begin{itemize}
\item Find a simple characterization of the domain of the map $\Sigma$ (see Proposition~\ref{prop:Conj541q} and Conjecture~\ref{conj:End661}).
\item Find necessary and sufficient conditions on the function $g$ to ensure both the uniqueness and existence of solutions lying in $\cK^p$ to the equation $\Delta f=g$ (cf.\ Webster's question in Appendix~\ref{chapter:B-KW562}).
\item Find a natural extension of the map $\Sigma$ to a larger domain, e.g., a real linear space of functions that would include not only the current admissible functions but also every function that has an exponential growth.
\item Find a general method to determine a simple or compact expression for the asymptotic constant $\sigma[g]$ associated with any function $g$ lying in $\cC^0\cap\mathrm{dom}(\Sigma)$ (cf.\ our discussion in Section~\ref{sec:Raabe448}).
\item Find general methods to determine analogues of Euler's reflection formula (cf.\ our discussion on Herglotz's trick in Section~\ref{sec:ReflFor62}) and Gauss' digamma theorem for any multiple $\log\Gamma$-type function.
\item Find necessary and sufficient conditions on the function $g$ for the function $\Sigma g$ to be of class $\cC^{\infty}$ or even real analytic.
\item Find an extension of our theory to functions of a complex variable. On this issue, it is noteworthy that a very nice ``complex'' counterpart of Bohr-Mollerup's characterization of the gamma function was established by Wielandt (see, e.g., Srinivasan \cite{Sri07} and Srivastava and Choi \cite[p.~12]{SriCho12} and the references therein).
\item Find an extension of our theory by replacing the equation $\Delta f=g$ with the more general first-order linear difference equation
$$
f(x+1)-a{\,}f(x) ~=~ g(x),
$$
where $a$ is a given constant. Consider also linear difference equations of any order. Partial results along this line can be found, e.g., in John \cite[Theorem C]{Joh39}.
\end{itemize}

\appendix
\chapter{Higher order convexity properties}
\label{chapter:pconv184}
\index{higher order convexity and concavity|(}

\noindent\textsl{Summary: We establish a number of basic facts about higher order convexity and concavity properties with the aim of proving Lemma~\ref{lemma:PrelKp}.}

\medskip

Lemma~\ref{lemma:PrelKp} is a fundamental element of our theory. It can be derived from more general results established by Kuczma \cite[Chapter~15]{Kuc85}. However, this derivation is not immediate and actually requires considerable attention. In this appendix, we prove Lemma~\ref{lemma:PrelKp} almost from scratch and using elementary means only.

Let $I$ be an arbitrary nonempty open real interval. We first observe that for any functions $f,g\colon I\to\R$ and any system $x_0 < x_1 < \cdots < x_n$ of $n+1$ points in $I$, we have
$$
(f+g)[x_0,x_1,\ldots,x_n] ~=~ f[x_0,x_1,\ldots,x_n]+g[x_0,x_1,\ldots,x_n].
$$
Moreover, for any $c\in\R$, if the function $h\colon I-c\to\R$ is defined by the equation $h(x)=f(x+c)$ for $x\in I-c$, then
$$
h[x_0,x_1,\ldots,x_n] ~=~ f[x_0+c,x_1+c,\ldots,x_n+c].
$$
These properties are immediate consequences of identity \eqref{eq:DivDiffPai4Dis82}.

We now present a proposition and an immediate corollary. Let $\Delta_{[h]}$ denote the forward difference operator with step $h$.

\begin{proposition}\label{prop:4d65SIOT}
For any $n\in\N$, any system $x_0 < x_1 < \cdots < x_n$ of $n+1$ points in $I$, any function $f\colon I\to\R$, and any $h\in\R\setminus\{0\}$ such that $x_0+h,x_n+h\in I$, we have
$$
\frac{1}{h}(\Delta_{[h]}f)[x_0,x_1,\ldots,x_n] ~=~ \sum_{k=0}^nf[x_0,\ldots,x_k,x_k+h,\ldots,x_n+h].
$$
\end{proposition}

\begin{proof}
Using a telescoping sum, we obtain
\begin{eqnarray*}
\frac{1}{h}(\Delta_{[h]}f)[x_0,x_1,\ldots,x_n] &=& \frac{1}{h}{\,}(f[x_0+h,x_1+h,\ldots,x_n+h]-f[x_0,x_1,\ldots,x_n])\\
&=& \frac{1}{h}\,\sum_{k=0}^n\big(f[x_0,\ldots,x_k+h,x_{k+1}+h,\ldots,x_n+h]\\
&& \null\hspace{5ex}\null -f[x_0,\ldots,x_k,x_{k+1}+h,\ldots,x_n+h]\big).
\end{eqnarray*}
We then conclude the proof using the recurrence relation \eqref{eq:DivDiffRec7}.
\end{proof}

\begin{corollary}\label{cor:4d65SIOTc}
Let $f$ lie in $\cK^p_+(I)$ for some $p\in\N$ and let $h\in\R\setminus\{0\}$. If the function $\frac{1}{h}\,\Delta_{[h]}f$ is defined on $I$, then it lies in $\cK^{p-1}_+(I)$.
\end{corollary}

We can now readily see that Lemma~\ref{lemma:PrelKp}(b) is an immediate consequence of Corollary~\ref{cor:4d65SIOTc} (just take $h=1$).

The next result establishes Lemma~\ref{lemma:PrelKp}(c). Let us first observe that a pointwise limit of functions lying in $\cK^p_+(I)$ also lies in $\cK^p_+(I)$. This fact can be proved straightforwardly using identity \eqref{eq:DivDiffPai4Dis82}.

\begin{corollary}\label{cor:4d65SIOTc2}
If $f\colon I\to\R$ is differentiable and lies in $\cK^p_+(I)$ for some $p\in\N$, then the derivative $f'$ lies in $\cK^{p-1}_+(I)$.
\end{corollary}

\begin{proof}
It is clear that the derivative $f'$ is the pointwise limit of the sequence $n\mapsto f_n$, where, for each $n\in\N^*$, the function $f_n\colon I\to\R$ is defined by the equation
$$
f_n ~=~ n\Delta_{[1/n]}f\qquad\text{for $n\in\N^*$}.
$$
We then conclude the proof using Corollary~\ref{cor:4d65SIOTc}.
\end{proof}

We now have the following corollary, which follows from Proposition~\ref{prop:2A4DiffInt4Pol}. It immediately establishes Lemma~\ref{lemma:PrelKp}(d).

\begin{corollary}
If $f\colon I\to\R$ is differentiable and $f'$ lies in $\cK_+^{p-1}(I)$ for some $p\in\N$, then $f$ lies in $\cK_+^p(I)$.
\end{corollary}

\begin{proof}
This result is an immediate consequence of Proposition~\ref{prop:2A4DiffInt4Pol} (just use identity \eqref{eq:AppA3sq3} with $n=p+1$).
\end{proof}

It remains to establish Lemma~\ref{lemma:PrelKp}(a). To this end, we present the following technical lemma, which provides a test for differentiability of real functions on $I$.

\begin{lemma}\label{lemma:Diff4Test72}
Let $n\in\N$, let $a,x_1,\ldots,x_n$ be $n+1$ pairwise distinct points in $I$ and let $f\colon I\to\R$. If the limit
$$
\lim_{h\to 0}f[a,a+h,x_1,\ldots,x_n]
$$
exists and is finite, then $f$ is differentiable at $a$.
\end{lemma}

\begin{proof}
This result can be easily proved by induction on $n$ using the recurrence relation \eqref{eq:DivDiffRec7}. To simplify the computations, let us consider the first two cases only. For $n=0$, we have trivially
$$
\lim_{h\to 0}f[a,a+h] ~=~ \lim_{h\to 0}\frac{f(a+h)-f(a)}{h}
$$
and $f$ is clearly differentiable at $a$ if this limit exists and is finite. For $n=1$, we get
$$
f[a,a+h,x_1] ~=~ \frac{f[a,x_1]-f[a,a+h]}{x_1-a-h}
$$
and hence
$$
\lim_{h\to 0} f[a,a+h] ~=~ f[a,x_1]-\lim_{h\to 0} (x_1-a-h){\,}f[a,a+h,x_1]
$$
and this limit is finite if so is the right-hand limit. The induction process is now clear.
\end{proof}

We now have the following proposition.

\begin{proposition}\label{prop:cKp4fDiff921}
If $f$ lies in $\cK^p_+(I)$ for some integer $p\geq 2$, then $f$ is differentiable on $I$.
\end{proposition}

\begin{proof}
Let $a\in I$ and let $J$ be a compact subinterval of $I$ whose interior contains $a$. Let $\mathcal{I}_{p+1}$ denote the set of tuples of $I^{p+1}$ whose components are pairwise distinct. By Lemma~\ref{lemma:pCInc5}, the restriction of the map
$$
(z_0,\ldots,z_p) ~\mapsto ~ f[z_0,\ldots,z_p]
$$
to $\mathcal{I}_{p+1}$ is increasing in each place, hence this map is bounded on $\mathcal{I}_{p+1}\cap J^{p+1}$.

Let $x_1,\ldots,x_{p-2}$ be $p-2$ pairwise distinct points in $J$, and distinct from $a$, and let $h_1,h_2$ be sufficiently small distinct nonzero real numbers such that $a+h_1,a+h_2$ lie in $J$. Using \eqref{eq:DivDiffRec7}, we get
\begin{multline*}
f[a,a+h_1,a+h_2,x_1,\ldots,x_{p-2}]\\
=~ \frac{f[a,a+h_2,x_1,\ldots,x_{p-2}]-f[a,a+h_1,x_1,\ldots,x_{p-2}]}{h_2-h_2}
\end{multline*}
Thus, there exists $C_J>0$ such that
$$
\left|f[a,a+h_2,x_1,\ldots,x_{p-2}]-f[a,a+h_1,x_1,\ldots,x_{p-2}]\right| ~\leq ~ C_J{\,}|h_2-h_1|.
$$
It follows that for any sequence $n\mapsto h_n$ converging to zero, the sequence
$$
n ~\mapsto ~ f[a,a+h_n,x_1,\ldots,x_{p-2}]
$$
is a Cauchy sequence whose limit does not depend on the sequence $n\mapsto h_n$.
Therefore, the limit
$$
\lim_{h\to 0}f[a,a+h,x_1,\ldots,x_{p-2}]
$$
exists and is finite. By Lemma~\ref{lemma:Diff4Test72}, $f$ is differentiable at $a$.
\end{proof}

We are now in a position to prove Lemma~\ref{lemma:PrelKp}(a).

\begin{proposition}
If $f$ lies in $\cK^{p+1}_+(I)$ for some $p\in\N$, then $f$ lies in $\cC^p(I)$.
\end{proposition}

\begin{proof}
We proceed by induction on $p$. The case when $p=0$ is folklore and can be found, e.g., in Artin~\cite[Theorem~1.5]{Art15}. Suppose that the result holds for some $p\in\N$ and let us show that it still holds for $p+1$. Let $f$ lie in $\cK^{p+2}_+(I)$. By Proposition~\ref{prop:cKp4fDiff921} and Corollary~\ref{cor:4d65SIOTc2}, $f$ is differentiable on $I$ and $f'$ lies in $\cK^{p+1}_+(I)$. Using our induction hypothesis, we see that $f'$ lies in $\cC^{p}(I)$, and hence $f$ lies in $\cC^{p+1}(I)$.
\end{proof}

Let us end this study with an interesting generalization of Lemma~\ref{lemma:pCInc5}. It is an immediate corollary of the following proposition.

\begin{proposition}
Let $n,m\in\N$, let $x_0,\ldots,x_{n-1},y_0,\ldots,y_m$ be $n+m+1$ pairwise distinct points in $I$, let $f\colon I\to\R$, and let $g\colon I\setminus\{x_0,\ldots,x_{n-1}\}\to\R$ be defined by the equation
$$
g(x) ~=~ f[x_0,\ldots,x_{n-1},x]\qquad\text{for $x\in I\setminus\{x_0,\ldots,x_{n-1}\}$}.
$$
Then we have
$$
g[y_0,\ldots,y_m] ~=~ f[x_0,\ldots,x_{n-1},y_0,\ldots,y_m].
$$
\end{proposition}

\begin{proof}
This result can be easily proved by induction on $m$ for any fixed value of $n$, simply by using the recurrence relation \eqref{eq:DivDiffRec7}. To simplify the computations, let us consider the first two cases only. For $m=0$, we have trivially
$$
g[y_0] ~=~ g(y_0) ~=~ f[x_0,\ldots,x_{n-1},y_0].
$$
For $m=1$, we have
\begin{eqnarray*}
g[y_0,y_1] &=& \frac{g(y_1)-g(y_0)}{y_1-y_0} ~=~ \frac{f[x_0,\ldots,x_{n-1},y_1]-f[x_0,\ldots,x_{n-1},y_0]}{y_1-y_0}\\
&=& f[x_0,\ldots,x_{n-1},y_0,y_1].
\end{eqnarray*}
The induction process is now obvious.
\end{proof}

\begin{corollary}\label{cor:A7Ijp15r}
Let $j,p\in\N$, with $j\leq p$, and let $\mathcal{I}_{j+1}$ denote the set of tuples of $I^{j+1}$ whose components are pairwise distinct. A function $f\colon I\to\R$ lies in $\cK^p_+(I)$ if and only if the restriction of the map
$$
(z_0,\ldots,z_j) ~\mapsto ~ f[z_0,\ldots,z_j]
$$
to $\mathcal{I}_{j+1}$ is $(p-j)$-convex in each place.
\end{corollary}

\index{higher order convexity and concavity|)}

\chapter{On Krull-Webster's asymptotic condition}
\label{chapter:A-KW561}

\noindent\textsl{Summary: We show that our uniqueness and existence results fully generalize a recent attempt by Rassias and Trif \cite{RasTri07} to solve the particular case when $p=2$.}

\medskip

Recall that the original asymptotic condition imposed by Krull and Webster on the function $g$ is that, for each $x>0$,
$$
g(x+t)-g(t) ~\to ~0\qquad\text{as $t\to\infty$}{\,};
$$
see Eq.~\eqref{eq:KWac3g}. Using our notation, this means that the function $g$ lies in $\cR^1_{\R}$. Geometrically, this condition also means that the chord to the graph of $g$ on any fixed length interval has an asymptotic zero slope. Only fixed length intervals whose left endpoints are integers are to be considered if the condition reduces to requiring that $g\in\cR^1_{\N}$. The restriction of our uniqueness and existence results to the case when $p=1$ shows that this condition can actually be relaxed into $g\in\cD^1_{\N}$, which means that the chord to the graph of $g$ on any interval of the form $[n,n+1]$, $n\in\N^*$, has an asymptotic zero slope. The function $g(x)=\ln x$ is a typical example that shows, just as does every function whose derivative vanishes at infinity, that those functions need not behave asymptotically like constant functions.

It remains, however, that Krull-Webster's asymptotic condition is rather restrictive. As already mentioned in Chapter~\ref{chapter:1}, this condition is not satisfied by the functions $g(x)=x\ln x$ and $g(x)=\ln\Gamma(x)$. To overcome this restriction, Rassias and Trif \cite{RasTri07} proposed a modification of Webster's results by considering solutions lying in $\cK^2$ and replacing the asymptotic condition with a more appropriate one. Specifically, they considered any function $g\colon\R_+\to\R$ for which there exists a number $a>0$ such that
\begin{equation}\label{eq:RasTriAsC}
\lim_{t\to\infty}~g(x+t)-g(t)-x\ln t ~=~ x\ln a,\qquad\text{for all $x>0$.}
\end{equation}
It turns out that both functions $g(x)=x\ln x$ and $g(x)=\ln\Gamma(x)$ satisfy this alternative condition. However, the identity function $g(x)=x$ does not.

Let us now show that our asymptotic condition that $g\in\cD^2_{\R}$ generalizes not only Rassias and Trif's \eqref{eq:RasTriAsC} but also many other similar conditions.

\begin{proposition}\label{prop:RT4}
Let $\varphi\colon\R_+\to\R$ and suppose that $g\colon\R_+\to\R$ has the property that, for each $x>0$,
$$
g(x+t)-g(t)-x{\,}\varphi(t) ~\to ~0\qquad\text{as $t\to\infty$.}
$$
Then $g$ lies in $\cR^2_{\R}\subset\cD^2_{\R}$. In particular, $\cR^2_{\R}$ contains all the functions that satisfy Rassias and Trif's condition.
\end{proposition}

\begin{proof}
For any $t>0$ and any $g\colon\R_+\to\R$, define the function $\rho^{\varphi}_t[g]\colon [0,\infty)\to\R$ by the equation
$$
\rho^{\varphi}_t[g](x) ~=~ g(x+t)-g(t)-x{\,}\varphi(t)\qquad\text{for $x>0$}.
$$
Let also $\cR^{\varphi}_{\R}$ be the set of functions $g\colon\R_+\to\R$ with the property that, for each $x>0$, $\rho^{\varphi}_t[g](x)\to 0$ as $t\to\infty$. Then we immediately see that
$$
\rho_t^2[g](x) ~=~ \rho^{\varphi}_t[g](x)-x\rho^{\varphi}_t[g](1),
$$
which shows that $\cR^{\varphi}_{\R}\subseteq\cR^2_{\R}$. Now, if $g$ satisfies Rassias and Trif's condition, then it lies in the set $\cup_{a>0}\cR^{\varphi_a}_{\R}$, where $\varphi_a(x)=\ln(ax)$, and hence it also lies in $\cR^2_{\R}$.
\end{proof}

Proposition~\ref{prop:RT4} can be generalized to $\cR^p_{\R}$ for any value of $p\geq 2$ as follows.

\begin{proposition}
Let $p\geq 2$ be an integer, let $\varphi_1,\ldots,\varphi_{p-1}\colon\R_+\to\R$, and suppose that $g\colon\R_+\to\R$ has the property that, for each $x>0$,
$$
g(x+t)-g(t)-\sum_{j=1}^{p-1}\tchoose{x}{j}\,\varphi_j(t) ~\to ~0\qquad\text{as $t\to\infty$}.
$$
Then $g$ lies in $\cR^p_{\R}\subset\cD^p_{\R}$.
\end{proposition}

\begin{proof}
For any $t>0$ and any $g\colon\R_+\to\R$, define the function $\rho_t^{\boldsymbol{\varphi}}[g]\colon [0,\infty)\to\R$ by the equation
$$
\rho_t^{\boldsymbol{\varphi}}[g](x) ~=~ g(x+t)-g(t)-\sum_{j=1}^{p-1}\tchoose{x}{j}\,\varphi_j(t).
$$
Define also the functions $\psi_t^{\boldsymbol{\varphi},1}[g],\ldots,\psi_t^{\boldsymbol{\varphi},p}[g]\colon [0,\infty)\to\R$ recursively by the equations $\psi_t^{\boldsymbol{\varphi},1}[g]=\rho_t^{\boldsymbol{\varphi}}[g]$ and
$$
\psi_t^{\boldsymbol{\varphi},j+1}[g] ~=~ \psi_t^{\boldsymbol{\varphi},j}[g]-\tchoose{x}{j}\,\psi_t^{\boldsymbol{\varphi},j}[g](j),\qquad j=1,\ldots,p-1.
$$
Then, it is not difficult to see that
$$
\psi_t^{\boldsymbol{\varphi},j}[g](x) ~=~ \rho_t^{\boldsymbol{\varphi}}[g](x)-\sum_{i=1}^{j-1}\tchoose{x}{i}(\Delta^ig(t)-\varphi_i(t))
$$
and hence $\psi_t^{\boldsymbol{\varphi},p}[g]=\rho^p_t[g]$. Thus, if the function $g\colon\R_+\to\R$ has the property that, for each $x>0$, $\rho^{\boldsymbol{\varphi}}_t[g](x)\to 0$ as $t\to\infty$, then it lies in $\cR^p_{\R}$.
\end{proof} 

\chapter{On a question raised by Webster}
\label{chapter:B-KW562}

\noindent\textsl{Summary: We discuss conditions on the function $g$ to ensure both the uniqueness (up to an additive constant) and existence of solutions to the equation $\Delta f=g$ that lie in $\cK^p$.}

\medskip

A natural question raised by Webster~\cite[p.\ 606]{Web97b}, and that we now extend to any value of $p\in\N$, is the following.
\begin{quote}
\emph{Find necessary and sufficient conditions on the function $g\colon\R_+\to\R$ to ensure both the uniqueness (up to an additive constant) and existence of solutions lying in $\cK^p_+$ (resp.\ $\cK^p_-$) to the equation $\Delta f=g$.}
\end{quote}

Lemma~\ref{lemma:PrelKp}(b) shows that a necessary condition for this to occur is that $g\in\cK^{p-1}_+$ (resp.\ $g\in\cK^{p-1}_-$). Also, our uniqueness and existence results show that a sufficient condition is that $g\in\cD^p\cap\cK^p_-$ (resp.\ $g\in\cD^p\cap\cK^p_+$). It is tempting to believe that this latter condition is also necessary. The following two examples support this idea.
\begin{enumerate}
\item[(a)] Both functions
$$
\ln\Gamma(x)\qquad\text{and}\qquad\textstyle{\ln\Gamma(x)+\ln\left(1+\frac{1}{2}\sin(2\pi x)\right)}
$$
are solutions to the equation $\Delta f=g$ that lie in $\cK_+^0$, where $g(x)=\ln x$ does not lie in $\cD^0\cup\cK^0_-$ (see Example~\ref{ex:unicGa4}).
\item[(b)] Both functions
$$
2^x\qquad\text{and}\qquad 2^x + \sin(2\pi x)
$$
are solutions to the equation $\Delta f=g$ that lie in $\cK^p_+$ for any $p\in\N$, where $g(x)=2^x$ does not lie in $\cD^p\cup\cK^p_-$.
\end{enumerate}
Nevertheless, the following proposition shows that in general the condition above is not necessary.

\begin{proposition}\label{prop:sdf089sadf4}
There exists a function $f\in\cC^0\cap\cK^0$ such that
\begin{itemize}
\item[(a)] $\Delta f$ does not lie in $\cD^0\cup\cK^0$, and
\item[(b)] for any function $\varphi\in\cK^0$ satisfying $\Delta\varphi =\Delta f$ we have that $f-\varphi$ is constant.
\end{itemize}
\end{proposition}

\begin{proof}
Let $f\in\cK^0_+$ be the function whose graph is the polygonal line through the points $(4n,4n)$ and $(4n+2,4n+4)$ for all $n\in\N$. Thus the sequence $n\mapsto\Delta f(n)$ is the $4$-periodic sequence $2,0,0,2,2,0,0,2,\ldots$ and hence condition (a) holds. Now, let $\varphi\in\cK^0$ be such that $\Delta\varphi =\Delta f$. Clearly, we must have $\varphi\in\cK^0_+$. For the sake of a contradiction, suppose that the $1$-periodic function $\omega =f-\varphi$ is not constant. That is, there exist $0<x<y\leq 1$ such that $\omega(x)\neq\omega(y)$. There are two exclusive cases to consider.
\begin{itemize}
\item[(a)] Suppose that $\omega(x)<\omega(y)$. For large integer $n$, we then have
$$
0 ~\leq ~ \varphi(y+4n+2)-\varphi(x+4n+2) ~=~ \omega(x)-\omega(y) ~<~ 0.
$$
\item[(a)] Suppose that $\omega(x)>\omega(y)$. For large integer $n$, we then have
$$
0 ~\leq ~ \varphi(x+4n+3)-\varphi(y+4n+2) ~=~ \omega(y)-\omega(x) ~<~ 0.
$$
\end{itemize}
In both cases we reach a contradiction, and hence condition (b) holds.
\end{proof}

We note that the function $f$ arising from Propositon~\ref{prop:sdf089sadf4} is such that $g=\Delta f$ does not lie in $\cD^0\cup\cK^0$. The following proposition shows that if the equation $\Delta f=g$ has a unique solution (up to an additive constant) and if $g\in\cK^p$ for some $p\in\N$, then necessarily $g\in\cD^p\cap\cK^p$ (see also Corollary~\ref{cor:dfmm7saa5}).

\begin{proposition}
Let $g\colon\R_+\to\R$ and $p\in\N$, and suppose that the sequence $n\mapsto\Delta^p g(n)$ is eventually decreasing. Suppose also that there exists a unique (up to an additive constant) function $f\in\cK^p_+$ satisfying the equation $\Delta f=g$. Then $g$ lies in $\cD^p_{\N}$.
\end{proposition}

\begin{proof}
For the sake of a contradiction, suppose that the assumptions are satisfied and that the sequence $n\mapsto\Delta^p g(n)$ does not approach zero. Since this sequence is eventually nonnegative (because we eventually have $\Delta^p g=\Delta^{p+1}f\geq 0$), it must converge to a value $C>0$. It follows that the function $\tilde g(x)=g(x)-C{x\choose p}$ lies in $\cD^p\cap\cK^p_-$ and hence $\Sigma\tilde g$ lies in $\cK^p_+$. Now, for any $0<\tau <C/(2\pi)^p$, the functions
\begin{eqnarray*}
f(x) &=& \Sigma\tilde g(x)+C\tchoose{x}{p+1},\\
\varphi(x) &=& \Sigma\tilde g(x)+C\tchoose{x}{p+1}+\tau\sin(2\pi x),
\end{eqnarray*}
lie in $\cK^p_+$ by Lemma~\ref{lemma:PrelKp}(d); indeed, we have
$$
D^{p+1}(C\tchoose{x}{p+1}+\tau\sin(2\pi x)) ~\geq ~ C-\tau (2\pi)^p ~>~ 0.
$$
Moreover, these functions are solutions to the equation $\Delta f=g$ and satisfy $\varphi(1)=f(1)$. This contradicts the uniqueness assumption.
\end{proof}

\begin{remark}
We observe that if $f$ and $\varphi$ are solutions to $\Delta f=g$, then for any $x>0$ and any $p\in\N^*$, we have $\Delta^p f(x)\geq 0$ if and only if $\Delta^p\varphi(x)\geq 0$. Indeed, suppose on the contrary that $\Delta^p f(x)\geq 0$ and $\Delta^p\varphi(x)< 0$ for some $x>0$. Then
$$
0 ~\leq ~ \Delta^pf(x) ~=~ \Delta^{p-1}g(x) ~=~ \Delta^p\varphi(x) ~< ~ 0,
$$
a contradiction.
\end{remark}

Thus, Webster's question still remains a very interesting open problem whose solution would certainly shed light on the theory developed in this book.

Regarding uniqueness issues only, the following two results (John \cite{Joh39}) are also worth mentioning. Generalizations of these results to higher convexity properties would be welcome.

\begin{proposition}[{see \cite{Joh39}}]
Let $g\colon\R_+\to\R$ have the property that
$$
\inf_{x\in\R_+}g(x) ~=~ 0.
$$
Then there is at most one (up to an additive constant) solution $f$ to the equation $\Delta f=g$ that is increasing.
\end{proposition}

\begin{proposition}[{see \cite{Joh39}}]
Let $g\colon\R_+\to\R$ have the property that
$$
\liminf_{x\to\infty}\frac{g(x)}{x} ~=~ 0.
$$
Then there is at most one (up to an additive constant) solution $f$ to the equation $\Delta f=g$ that is convex.
\end{proposition} 

\chapter{Asymptotic behaviors and bracketing}
\label{chapter:C-StBr49}

\noindent\textsl{Summary: We show that by considering higher and higher values of $p$ in Corollary~\ref{cor:6GenStFo0BaIneqCo} we obtain closer and closer bounds for the generalized Binet function\index{Binet's function!generalized} $J^{p+1}[\Sigma g]$.}

\medskip

We have seen in Example~\ref{ex:StriCons} that the inequalities
$$
\left(1+\frac{1}{x}\right)^{-\frac{1}{2}} \leq ~ \frac{\Gamma(x)}{\sqrt{2\pi}{\,}e^{-x}{\,}x^{x-\frac{1}{2}}} ~\leq ~ \left(1+\frac{1}{x}\right)^{\frac{1}{2}}
$$
hold for any $x>0$ and that tighter inequalities can also be obtained by using different values of the integer $p\geq 1$ in Corollary~\ref{cor:6GenStFo0BaIneqCo}. In this appendix we show that and how this feature applies in general to multiple $\log\Gamma$-type functions.

Let $g$ lie in $\cC^0\cap\cD^p\cap\cK^p$, where $p=1+\deg p$. By Corollary~\ref{cor:6GenStFo0BaIneqCo}, for any $x>0$ such that $g$ is $p$-convex or $p$-concave on $[x,\infty)$ we have the inequalities
$$
-\overline{G}_p\left|\Delta^p g(x)\right| ~\leq ~ J^{p+1}[\Sigma g](x) ~\leq ~ \overline{G}_p\left|\Delta^p g(x)\right|.
$$

Let us now show how tighter inequalities can be obtained. For any $r\in\N$, define the functions $\alpha_r[\Sigma g]\colon\R_+\to\R$ and $\beta_r[\Sigma g]\colon\R_+\to\R$ respectively by the equations
\begin{eqnarray*}
\alpha_r[\Sigma g](x) &=& -\overline{G}_{p+r}\left|\Delta^{p+r} g(x)\right|-\sum_{j=p+1}^{p+r}G_j\Delta^{j-1}g(x){\,},\\
\beta_r[\Sigma g](x) &=& \overline{G}_{p+r}\left|\Delta^{p+r} g(x)\right|-\sum_{j=p+1}^{p+r}G_j\Delta^{j-1}g(x){\,},
\end{eqnarray*}
for $x>0$.

We immediately see that the equality
$$
\alpha_r[\Sigma g](x) ~=~ \beta_r[\Sigma g](x)
$$
holds if and only if $\Delta^{p+r}g(x)=0$. Moreover, by Corollary~\ref{cor:6GenStFo0BaIneqCo}, if $g\in\cK^{p+r}$ and if $x>0$ is so that $g$ is $(p+r)$-convex or $(p+r)$-concave on $[x,\infty)$, then the following inequalities hold:
$$
\alpha_r[\Sigma g](x) ~\leq ~ J^{p+1}[\Sigma g](x) ~\leq ~ \beta_r[\Sigma g](x).
$$
The following proposition shows that these inequalities get tighter and tighter as the value of $r$ increases.

\begin{proposition}\label{prop:ineqT792}
Let $g$ lie in $\cC^0\cap\cD^p\cap\cK^{p+r+1}$ for some $r\in\N$, where $p=1+\deg g$. Let $x>0$ be so that $g|_{[x,\infty)}$ lies in
$$
\cK^{p+r}([x,\infty))\cap\cK^{p+r+1}([x,\infty)).
$$
Then, we have
$$
\alpha_r[\Sigma g](x) ~\leq ~ \alpha_{r+1}[\Sigma g](x) ~\leq ~ \beta_{r+1}[\Sigma g](x) ~\leq ~ \beta_r[\Sigma g](x).
$$
These inequalities are strict if $\Delta^{p+r}g(x+1)\neq 0$.
\end{proposition}

\begin{proof}
We already know that the central inequality holds. Now, using Corollary~\ref{cor:dfmm7s}, we can assume that $g$ is $(p+r)$-convex and $(p+r+1)$-concave on $[x,\infty)$; the other case can be dealt with similarly. By Lemma~\ref{lemma:pCInc5}, it follows that $\Delta^{p+r}g\leq 0$ and $\Delta^{p+r+1}g\geq 0$ on $[x,\infty)$. Let us show that the first inequality holds; the third one can be established similarly.

We have two exclusive cases to consider.
\begin{itemize}
\item If $G_{p+r+1}<0$, then
\begin{eqnarray*}
\Delta_r{\,}\alpha_r[\Sigma g](x) &=& -\overline{G}_{p+r+1}\left(\Delta^{p+r+1}g(x)+\Delta^{p+r}g(x)\right)\\
&=& -\overline{G}_{p+r+1}\Delta^{p+r}g(x+1).
\end{eqnarray*}
\item If $G_{p+r+1}>0$, then
$$
\Delta_r{\,}\alpha_r[\Sigma g](x) ~=~ -\overline{G}_{p+r}\Delta^{p+r}g(x+1)+G_{p+r+1}\left(\Delta^{p+r+1}g(x)-\Delta^{p+r}g(x)\right).
$$
\end{itemize}
In both cases, we can see that $\Delta_r{\,}\alpha_r[\Sigma g](x)\geq 0$. Moreover, we have $\Delta_r{\,}\alpha_r[\Sigma g](x)$ $> 0$ if $\Delta^{p+r}g(x+1)\neq 0$.
\end{proof}

It is natural to wonder how the inequalities in Proposition~\ref{prop:ineqT792} behave as $r\to_{\N}\infty$. The following proposition, which is a reformulation of Proposition~\ref{prop:6699rem}, answers this question and provides a series representation for $J^{p+1}[\Sigma g]$.

\begin{proposition}\label{prop:ddozj40}
Let $g$ lie in $\cC^0\cap\cD^p\cap\cK^{\infty}$, where $p=1+\deg g$. Let $x>0$ be so that for every integer $q\geq p$ the function $g$ is $q$-convex or $q$-concave on $[x,\infty)$. Suppose also that the sequence $q\mapsto\Delta^qg(x)$ is bounded. Then the following assertions hold.
\begin{enumerate}
\item[(a)] The sequence $q\mapsto\beta_q[\Sigma g](x)-\alpha_q[\Sigma g](x)$ converges to zero.
\item[(b)] The sequence $n\mapsto G_n\Delta^{n-1}g(x)$ is summable.
\item[(c)] We have
$$
\Sigma g(x) ~=~ \sigma[g]+\int_1^xg(t)dt-\sum_{j=1}^{\infty}G_j\Delta^{j-1}g(x).
$$
Equivalently,
$$
J^{p+1}[\Sigma g](x) ~=~ -\sum_{j=p+1}^{\infty}G_j\Delta^{j-1}g(x).
$$
\end{enumerate}
\end{proposition} 

\chapter{Generalized Webster's inequality}
\label{chapter:E-GenWeIn5}
\index{Webster's inequality!generalized|(}

\noindent\textsl{Summary: Webster~\cite{Web97b} provided bounds for $\rho_x^{p+1}[\Sigma g](a)$ in the special case when $p=1$. We generalize Webster's bounds to any integer $p\in\N$ and use integration to provide new bounds for $J^{p+1}[\Sigma g](x)$ that are tighter than those given in Theorem~\ref{thm:6GenStFo0BaIneq}.}

\medskip

As we mentioned in Section~\ref{sec:6Gen4St2Fo0}, one can show that if $g$ lies in $\cD^1\cap\cK^1$ and if $x>0$ and $a>0$ are so that $g$ is concave on $[x+a,\infty)$, then the following double inequality holds\index{Webster's inequality}
\begin{eqnarray*}
\lefteqn{\sum_{k=0}^{\lfloor a\rfloor}g(x+k) + (\{a\}-1){\,}g(x+a)-a{\,}g(x) ~\leq ~ \rho^2_x[\Sigma g](a)}\\
&\leq & \sum_{k=0}^{\lfloor a\rfloor} g(x+k)-g(x+a)+\{a\}{\,}g(x+\lfloor a\rfloor+1)-a{\,}g(x).
\end{eqnarray*}
This result was proved in the multiplicative notation by Webster~\cite[Eq.~(6.4)]{Web97b} to establish the limit \eqref{eq:convRes79} in the case when $p=1$. In the following proposition, we generalize this inequality to any value of $p\in\N$. We call it the \emph{generalized Webster inequality}.

\begin{proposition}[Generalized Webster's inequality]\label{prop:E4Web7Inq6}
Let $f\colon\R_+\to\R$ and $g\colon\R_+\to\R$ be functions such that $\Delta f=g$ on $\R_+$. Let also $x>0$ and $a\geq 0$. The following assertions hold.
\begin{enumerate}
\item[(a)] If $f$ is monotone on $[x+a,\infty)$, then
$$
0 ~\leq ~ \pm\Big(\rho_x^1[f](a)+g(x+a)-\sum_{k=0}^{\lfloor a\rfloor}g(x+k)\Big) ~\leq ~ \pm {\,}g(x+\lfloor a\rfloor +1),
$$
where $\pm$ stands for $1$ or $-1$ according to whether $f$ lies in $\cK^0_+$ or $\cK^0_-$.
\item[(b)] If $f$ is $p$-convex or $p$-concave on $[x+a,\infty)$ for some $p\in\N^*$, then
\begin{eqnarray*}
0 &\leq & \pm\,\varepsilon_{p+1}(\{a\})\,\rho^{p+1}_{x+\lfloor a\rfloor +1}[f](\{a\})\\
&\leq & \pm\,\varepsilon_{p+1}(\{a\})\,\frac{\{a\}}{p}\,\rho^p_{x+\lfloor a\rfloor +1}[g](\{a\} -1),
\end{eqnarray*}
where $\varepsilon_{p+1}(\{a\})=0$, if $a\in\N$, and $\varepsilon_{p+1}(\{a\})=(-1)^p$, otherwise, and $\pm$ stands for $1$ or $-1$ according to whether $f$ lies in $\cK^p_+$ or $\cK^p_-$. Moreover, we have
\begin{eqnarray*}
\lefteqn{\rho^{p+1}_{x+\lfloor a\rfloor +1}[f](\{a\}) ~=~ \rho^{p+1}_x[f](a)+g(x+a)}\\
&& \null + \sum_{j=1}^p\left(\tchoose{a}{j}-\tchoose{\{a\}}{j}\right)\Delta^{j-1}g(x)
-\sum_{j=0}^p\tchoose{\{a\}}{j}\,\sum_{k=0}^{\lfloor a\rfloor}\Delta^jg(x+k).
\end{eqnarray*}
\end{enumerate}
\end{proposition}

\begin{proof}
Let us first prove assertion (a). Using monotonicity of $f$, we get
$$
\pm {\,} f(x+\lfloor a\rfloor +1) ~\leq ~ \pm {\,} f(x+a+1) ~\leq ~ \pm {\,} f(x+\lfloor a\rfloor +2),
$$
or equivalently, using \eqref{eq:fxnEfxSum},
\begin{eqnarray*}
\pm\Big(f(x)+\sum_{k=0}^{\lfloor a\rfloor}g(x+k)\Big) &\leq & \pm {\,}(f(x+a)+g(x+a))\\
&\leq & \pm\Big(f(x)+\sum_{k=0}^{\lfloor a\rfloor +1}g(x+k)\Big).
\end{eqnarray*}
This proves assertion (a). Let us now prove assertion (b). The first inequality immediately follows from Lemma~\ref{lemma:VarEpsIneq}. To see that the second inequality holds, we first observe that
\begin{eqnarray*}
\lefteqn{\{a\}^{\underline{p+1}}{\,}f[x+\lfloor a\rfloor +1,\ldots,x+\lfloor a\rfloor +p,x+a,x+a+1]}\\
&=& (\{a\}-p)\,\{a\}^{\underline{p}}{\,}f[x+\lfloor a\rfloor +1,\ldots,x+\lfloor a\rfloor +p,x+a+1]\\
&& \null - \{a\}{\,}(\{a\}-1)^{\underline{p}}{\,}f[x+\lfloor a\rfloor +1,\ldots,x+\lfloor a\rfloor +p,x+a]\qquad\text{(by \eqref{eq:DivDiffRec7})}\\
&=& (\{a\}-p)\,\rho^p_{x+\lfloor a\rfloor +1}[f](\{a\})-\{a\}\,\rho^p_{x+\lfloor a\rfloor +1}[f](\{a\} -1)\qquad\text{(by \eqref{eq:LDDiv})}\\
&=& \{a\}\,\rho^{p-1}_{x+\lfloor a\rfloor +1}[g](\{a\} -1)-p\,\rho^p_{x+\lfloor a\rfloor +1}[f](\{a\})\qquad\text{(by \eqref{eq:DeltaXLam})}\\
&=& \{a\}\,\rho^p_{x+\lfloor a\rfloor +1}[g](\{a\} -1)-p\tchoose{\{a\}}{p}\,\Delta^pf(x+\lfloor a\rfloor +1)\\
&& \null -p\,\rho^p_{x+\lfloor a\rfloor +1}[f](\{a\})\qquad\text{(by \eqref{eq:deflambdapt})}\\
&=& \{a\}\,\rho^p_{x+\lfloor a\rfloor +1}[g](\{a\} -1)-p\,\rho^{p+1}_{x+\lfloor a\rfloor +1}[f](\{a\})\qquad\text{(by \eqref{eq:deflambdapt})}
\end{eqnarray*}
Now, since $f$ is $p$-convex or $p$-concave on $[x+a,\infty)$, we have
$$
0 ~\leq ~ \pm\,\varepsilon_{p+1}(\{a\}){\,}\{a\}^{\underline{p+1}}{\,}f[x+\lfloor a\rfloor +1,\ldots,x+\lfloor a\rfloor +p,x+a,x+a+1],
$$
and hence
$$
0 ~\leq ~ \pm\,\varepsilon_{p+1}(\{a\}){\,}\Big(\frac{\{a\}}{p}\,\rho^p_{x+\lfloor a\rfloor +1}[g](\{a\} -1)-\rho^{p+1}_{x+\lfloor a\rfloor +1}[f](\{a\})\Big).
$$
This proves the second inequality. Finally, using \eqref{eq:deflambdapt} and then \eqref{eq:fxnEfxSum} we obtain
\begin{eqnarray*}
\lefteqn{\rho^{p+1}_{x+\lfloor a\rfloor +1}[f](\{a\}) - \rho^{p+1}_x[f](a)}\\
&=& f(x+a+1)-\sum_{j=0}^p\tchoose{\{a\}}{j}\,\Delta^jf(x+\lfloor a\rfloor +1)-f(x+a)+\sum_{j=0}^p\tchoose{a}{j}\,\Delta^jf(x)\\
&=& g(x+a) + \sum_{j=1}^p\left(\tchoose{a}{j}-\tchoose{\{a\}}{j}\right)\Delta^jf(x)
-\sum_{j=0}^p\tchoose{\{a\}}{j}\,\sum_{k=0}^{\lfloor a\rfloor}\Delta^jg(x+k).
\end{eqnarray*}
This completes the proof.
\end{proof}

The generalized Webster inequality applies to multiple $\log\Gamma$-type functions simply by taking $f=\Sigma g$ in Proposition~\ref{prop:E4Web7Inq6}, provided that $g$ lies in $\cD^p\cap\cK^p$ for some $p\in\N$. This inequality then provides bounds for the quantity $\rho_x^{p+1}[\Sigma g](a)$.

We now show how narrow bounds for $J^{p+1}[\Sigma g](x)$ can be derived by ``integrating'' the generalized Webster inequality. We also show that these new bounds are narrower than the generalized Stirling's formula-based inequalities given in Theorem~\ref{thm:6GenStFo0BaIneq} and Corollary~\ref{cor:6GenStFo0BaIneqCo}.

Let us begin with the special case when $p=0$. Thus, let $g$ lie in $\cC^0\cap\cD^0\cap\cK^0$ and let $x>0$ be so that $g$ is monotone on $[x,\infty)$. Corollary~\ref{cor:6GenStFo0BaIneqCo} provides the following bounds for $J^1[\Sigma g](x)$
$$
-|g(x)| ~\leq ~ J^1[\Sigma g](x) ~\leq ~ |g(x)|.
$$
The following proposition provides a finer approximation of $J^1[\Sigma g](x)$, where the absolute error is bounded at $x$ by $|g(x+1)|$.

\begin{proposition}\label{prop:E4Web7Inq6Int0}
Let $g$ lie in $\cC^0\cap\cD^0\cap\cK^0$ and let $x>0$ be so that $g$ is monotone on $[x,\infty)$. Then we have
\begin{eqnarray*}
0 &\leq & \pm\bigg(g(x)-\int_0^1g(x+t){\,}dt\bigg) ~\leq ~ \pm {\,}(-1){\,} J^1[\Sigma g](x)\\
&\leq & \pm\bigg(g(x)+g(x+1)-\int_0^1g(x+t){\,}dt\bigg) ~\leq ~ \pm g(x),
\end{eqnarray*}
where $\pm$ stands for $1$ or $-1$ according to whether $\Sigma g$ lies in $\cK^0_+$ or $\cK^0_-$.
\end{proposition}

\begin{proof}
Negating $g$ is necessary, we can assume that it lies in $\cK^0_-$, which means that $\Sigma g$ lies in $\cK^0_+$. This immediately establishes the first and the last inequalities. The two inner inequalities can then be obtained by integrating the expressions in assertion (a) of Proposition~\ref{prop:E4Web7Inq6} on $a\in (0,1)$.
\end{proof}

\begin{example}
Let us apply Proposition~\ref{prop:E4Web7Inq6Int0} to $g(x)=\frac{1}{x}$. For any $x>0$, we have the following inequalities
$$
\ln x -\frac{1}{x} ~\leq ~ \ln(x+1)-\frac{1}{x}-\frac{1}{x+1} ~\leq ~ \psi(x) ~\leq ~ \ln(x+1)-\frac{1}{x} ~\leq ~ \ln x.
$$
The inner approximation has an absolute error that is bounded at any $x>0$ by the quantity $\frac{1}{x+1}$.
\end{example}

Let us now assume that $p\geq 1$. Thus, let $g$ lie in $\cC^0\cap\cD^p\cap\cK^p$ for some $p\in\N^*$ and let $x>0$ be so that $g$ is $p$-convex or $p$-concave on $[x,\infty)$. Then we have seen in Theorem~\ref{thm:6GenStFo0BaIneq} that the following inequalities hold
$$
0 ~\leq ~ \pm (-1)^pJ^{p+1}[\Sigma g](x) ~\leq ~ \pm (-1)^{p+1}{\,}B^p[g](x),
$$
where $\pm$ stands for $1$ or $-1$ according to whether $g$ lies in $\cK^p_+$ or in $\cK^p_-$, and
\begin{eqnarray*}
B^p[g](x) &=& \int_0^1\tchoose{t-1}{p}{\,}(\Delta^{p-1}g(x+t)-\Delta^{p-1}g(x)){\,}dt\\
&=& \int_0^1\tchoose{t-1}{p}\,\Delta^{p-1}g(x+t){\,}dt-(-1)^p\,\overline{G}_p\,\Delta^{p-1}g(x).
\end{eqnarray*}

In the following proposition, we give finer bounds for $J^{p+1}[\Sigma g](x)$. To this end, we introduce the quantity
$$
A^p[g](x) ~=~ J^{p+1}[g](x)+\frac{1}{p}\int_0^1t\,\rho^p_{x+1}[g](t-1){\,}dt.
$$
It is not difficult to see that this quantity can be rewritten as follows
$$
A^p[g](x) ~=~ J^{p+1}[g](x)+\frac{1}{p}\int_0^1t{\,}g(x+t){\,}dt-\frac{1}{p}\sum_{j=1}^pjG_j\,\Delta^{j-1}g(x+1).
$$
Indeed, using \eqref{eq:deflambdapt} we clearly have
$$
\int_0^1t\,\rho^p_{x+1}[g](t-1){\,}dt ~=~ \int_0^1t{\,}g(x+t){\,}dt-\sum_{j=0}^{p-1}\int_0^1t\tchoose{t-1}{j}{\,}dt ~ \Delta^jg(x+1)
$$
where
$$
\int_0^1t\tchoose{t-1}{j}{\,}dt ~=~ (j+1)\int_0^1\tchoose{t}{j+1}{\,}dt ~=~ (j+1){\,}G_{j+1}.
$$
We also observe that $A^1[g]=B^1[g]$.

\begin{proposition}\label{prop:E4Web7Inq6Intp}
Let $g$ lie in $\cC^0\cap\cD^p\cap\cK^p$ for some $p\in\N^*$ and let $x>0$ be so that $g$ is $p$-convex or $p$-concave on $[x,\infty)$. Then, we have
\begin{eqnarray*}
0 ~\leq ~ \pm (-1)^{p+1}{\,}J^{p+1}[g](x) &\leq & \pm (-1)^pJ^{p+1}[\Sigma g](x)\\
&\leq & \pm (-1)^{p+1}{\,}A^p[g](x) ~\leq ~ \pm (-1)^{p+1}{\,}B^p[g](x),
\end{eqnarray*}
where $\pm$ stands for $1$ or $-1$ according to whether $g$ lies in $\cK^p_+$ or in $\cK^p_-${\,}.
\end{proposition}

\begin{proof}
Recall that if $g$ lies in $\cK^p_+$ (resp.\ $\cK^p_-$), then $\Sigma g$ lies in $\cK^p_-$ (resp.\ $\cK^p_+$). The first inequality is then clear. The second and third inequalities are obtained by integrating the expressions in assertion (b) of Proposition~\ref{prop:E4Web7Inq6} on $a\in (0,1)$. To establish the fourth inequality, we first prove the following claim.

\begin{claim}
For any $g\in\R_+\to\R$, any $p\in\N^*$, any $x>0$, and any $0<t<1$, we have
\begin{multline*}
\tchoose{t-1}{p}(\Delta^{p-1}g(x+t)-\Delta^{p-1}g(x)) + \rho_x^{p+1}[g](t)-\frac{t}{p}\,\rho^p_{x+1}[g](t-1)\\
=~ \frac{1}{p}{\,}t^{\underline{p+1}}\,\sum_{j=1}^{p-1}g[\underbrace{x+j,\ldots,x+p-1}_{\textstyle{\text{$p-j$ places}}},\underbrace{x+t,\ldots,x+t+j}_{\textstyle{\text{$j+1$ places}}}].
\end{multline*}
\end{claim}

\begin{proof}[Proof of the claim]
Using \eqref{eq:deflambdapt}, it is easy to see that the claimed identity holds when $p=1$, in which case the right-hand side is identically zero. Hence, we can assume that $p\geq 2$. Using \eqref{eq:DivDiffRec7}, we then obtain
\begin{multline*}
\frac{1}{p}{\,}t^{\underline{p+1}}\,\sum_{j=1}^{p-1}g[x+j,\ldots,x+p-1,x+t,\ldots,x+t+j]\\
=~ \frac{1}{p}{\,}\frac{t^{\underline{p+1}}}{t}\,\sum_{j=1}^{p-1}\big(g[x+j+1,\ldots,x+p-1,x+t,\ldots,x+t+j]\\
\null - g[x+j,\ldots,x+p-1,x+t,\ldots,x+t+j-1]\big),
\end{multline*}
where the latter sum telescopes to
$$
g[x+t,\ldots,x+t+p-1]-g[x+1,\ldots,x+p-1,x+t].
$$
Thus, using \eqref{eq:LDDiv} we see that the right-hand side of the claimed identity reduces to
$$
\tchoose{t-1}{p}\,\Delta^{p-1}g(x+t)-\frac{t-p}{p}\,\rho^{p-1}_{x+1}[g](t-1)
$$
Now, subtracting the left-hand side of the claimed identity from this latter expression, we get
$$
\frac{p-t}{p}\,\rho^{p-1}_{x+1}[g](t-1)+\frac{t}{p}\,\rho^p_{x+1}[g](t-1)
-\rho^{p+1}_x[g](t)+\tchoose{t-1}{p}\,\Delta^{p-1}g(x).
$$
Using identities \eqref{eq:deflambdapt}, \eqref{eq:fng2}, and the trivial identity $\frac{t}{p}\tchoose{t-1}{p-1}=\tchoose{t}{p}$, it follows that the latter expression becomes
$$
-\tchoose{t}{p}\,\Delta^{p-1}g(x+1)+\sum_{j=0}^p\tchoose{t}{j}\Delta^jg(x)
-\sum_{j=0}^{p-2}\tchoose{t-1}{j}\,\Delta^jg(x+1)+\tchoose{t-1}{p}\,\Delta^{p-1}g(x).
$$
Substituting $g(x)+\Delta g(x)$ for $g(x+1)$ in this latter expression, we obtain
\begin{multline*}
-\left(\tchoose{t}{p}{\,}\Delta^{p-1}g(x)+\tchoose{t}{p}{\,}\Delta^p g(x)\right) + \left(\tchoose{t}{p}\,\Delta^p g(x)+\tchoose{t}{p-1}\,\Delta^{p-1} g(x)+\sum_{j=0}^{p-2}\tchoose{t}{j}\Delta^jg(x)\right)\\
\null -\left(\sum_{j=0}^{p-2}\tchoose{t-1}{j}\Delta^jg(x)+\sum_{j=0}^{p-2}\tchoose{t-1}{j}\Delta^{j+1}g(x)\right)
+\tchoose{t-1}{p}\,\Delta^{p-1}g(x).
\end{multline*}
Collecting terms, this latter expression reduces to
\begin{multline*}
\tchoose{t}{p-1}{\,}\Delta^{p-1}g(x)-\tchoose{t-1}{p-1}{\,}\Delta^{p-1}g(x)+\sum_{j=1}^{p-2}\left(\tchoose{t}{j}-\tchoose{t-1}{j}\right)\Delta^jg(x)
-\sum_{j=1}^{p-1}\tchoose{t-1}{j-1}\Delta^jg(x)\\
=~ \tchoose{t-1}{p-2}{\,}\Delta^{p-1}g(x)+\sum_{j=1}^{p-2}\tchoose{t-1}{j-1}\,\Delta^jg(x)
-\sum_{j=1}^{p-1}\tchoose{t-1}{j-1}\Delta^jg(x) ~=~ 0.
\end{multline*}
This completes the proof of the claim
\end{proof}

Let us now establish the fourth inequality. Negating $g$ if necessary, we can assume that it lies in $\cK^p_-$. Using the claim, we have immediately that
$$
B^p[g](x)-A^p[g](x) ~=~ \sum_{j=1}^{p-1}\int_0^1\frac{t^{\underline{p+1}}}{p}{\,}g[x+j,\ldots,x+p-1,x+t,\ldots,x+t+j]{\,}dt,
$$
where the divided difference\index{divided difference} of $g$ has $p+1$ arguments and is therefore nonnegative since $g$ is $(p-1)$-convex by Corollary~\ref{cor:dfmm7s}. This completes the proof.
\end{proof}

\begin{example}[The gamma function]\label{ex:E5Gam66m}
Let us apply Proposition~\ref{prop:E4Web7Inq6Intp} to the function $g(x)=\ln x$ with $p=1$ (recall here that $A^1[g]=B^1[g]$). We obtain the following inequalities for $x>0$
$$
0 ~\leq ~ \frac{1}{2}{\,}(2x+1)\ln\left(1+\frac{1}{x}\right)-1 ~\leq ~ J(x) ~\leq ~ \frac{1}{2}{\,}(x+1)^2\ln\left(1+\frac{1}{x}\right)-\frac{x}{2}-\frac{3}{4}{\,}.
$$
This provides an approximation of Binet's function\index{Binet's function} $J(x)$ with an absolute error that is bounded at any $x>0$ by
$$
\frac{x^2}{2}\ln\left(1+\frac{1}{x}\right)-\frac{x}{2}+\frac{1}{4}{\,},
$$
that is, $\frac{1}{6x}-\frac{1}{8x^2}+O(x^{-3})$ as $x\to\infty$. In the multiplicative notation, we obtain
$$
1 ~\leq ~ e^{-1}\left(1+\frac{1}{x}\right)^{x+\frac{1}{2}} \leq ~ \frac{\Gamma(x)}{\sqrt{2\pi}{\,}e^{-x}{\,}x^{x-\frac{1}{2}}} ~\leq ~ e^{-\frac{x}{2}-\frac{3}{4}}\left(1+\frac{1}{x}\right)^{\frac{1}{2}(x+1)^2},
$$
thus retrieving \eqref{eq:6fAppF42}. In turn, these inequalities provide an approximation of the log-gamma function with the same absolute error.
\end{example}

\begin{example}[The Barnes $G$-function, see Section~\ref{sec:Barnes558}]\index{Barnes's $G$-function}
Let us apply Proposition~\ref{prop:E4Web7Inq6Intp} to the function $g(x)=\ln\Gamma(x)$ with $p=2$. After some calculus we obtain the following inequalities for $x>0$
\begin{eqnarray*}
0 &\leq & \ln\Gamma(x)+x-\left(x-\frac{1}{2}\right)\ln x-\frac{1}{12}\ln\left(1+\frac{1}{x}\right)-\frac{1}{2}\ln(2\pi)\\
& \leq & -\ln G(x)+\psi_{-2}(x)-\frac{1}{2}\ln\Gamma(x)+\frac{1}{12}\ln x+\frac{1}{12}-\frac{1}{4}\ln(2\pi)-2\ln A\\
& \leq & \frac{1}{2}{\,}\psi_{-2}(x)+\frac{3}{4}\ln\Gamma(x)-\frac{1}{12}{\,}(3x^2+6x-4)\ln x+\frac{3}{8}{\,}x^2+\frac{1}{2}{\,}x\\
&& \null -\frac{1}{8}(2x+3)\ln(2\pi)-\frac{1}{2}\ln A\\
& \leq & \frac{1}{12}{\,}(x+1)^2(2x+5)\ln\left(1+\frac{1}{x}\right)-\frac{1}{72}{\,}(12x^2+48x+49).
\end{eqnarray*}
Here, the absolute error is bounded by $\frac{1}{16x}-\frac{59}{1440x^2}+O(x^{-3})$ as $x\to\infty$.
\end{example}

\begin{remark}[Bounds for the generalized Euler constant]\label{rem:EzzBou0gam332}\index{Euler's constant!generalized}
If $g$ lies in $\cC^0\cap\cD^p\cap\cK^p$ for $p=1+\deg g$ and if $g$ is $p$-convex or $p$-concave on $[1,\infty)$, then \eqref{eq:RpmiIneq91} and \eqref{eq:RpmiIneq91b} provide bounds for the generalized Euler constant (see Definition~\ref{de:GEC587})
$$
\gamma[g] ~=~ -J^{p+1}[\Sigma g](1).
$$
Finer bounds can now be obtained as follows. Under the assumptions of Proposition~\ref{prop:E4Web7Inq6Intp}, we have
$$
\pm (-1)^p J^{p+1}[g](1) ~\leq ~ \pm (-1)^{p+1}\gamma[g] ~\leq ~ \pm (-1)^p A^p[g](1).
$$
For instance, when $g(x)=\ln\Gamma(x)$, we obtain
$$
1-\frac{7}{12}\ln 2-\frac{1}{2}\ln\pi ~\leq ~ \gamma[\ln\circ\Gamma] ~=~ \sigma[\ln\circ\Gamma] ~\leq ~ \frac{7}{8}-\frac{1}{2}\ln A-\frac{3}{8}\ln(2\pi).
$$
Thus, $\gamma[\ln\circ\Gamma]\approx 0.045$ lies in the interval $[0.023,0.062]$, with amplitude $<0.039$.
\end{remark}

\parag{Searching for finer approximations} We now end this appendix with an interesting observation about the approximations of $J^{p+1}[\Sigma g](x)$ (or equivalently $\Sigma g(x)$) given in Propositions~\ref{prop:E4Web7Inq6Int0} and \ref{prop:E4Web7Inq6Intp}.

For any $p\in\N$ and any $g\in\cC^0\cap\cD^p\cap\cK^p$, define the function $\varepsilon^p[g]\colon\R_+\to\R$ by the equation
$$
\varepsilon^p[g](x) ~=~
\begin{cases}
|g(x+1)|, & \text{if $p=0$},\\
|A^p[g](x)-J^{p+1}[g](x)|, & \text{if $p\geq 1$}.
\end{cases}
$$
Let us show that, if $g$ is $p$-convex or $p$-concave on $[x,\infty)$, then the function $\varepsilon^p[g]$ decreases to zero on $[x,\infty)$. This is clear if $p=0$, so we can assume that $p\geq 1$. We know from Theorem~\ref{thm:6GenStFo0BaIneq} that the function $|B^p[g]|$ vanishes at infinity, and hence so does the function $\varepsilon^p[g]$ by Proposition~\ref{prop:E4Web7Inq6Intp}. On the other hand, using \eqref{eq:LDDiv} we see that
\begin{eqnarray*}
\varepsilon^p[g](x) &=& \left|\int_0^1\frac{t}{p}\,\rho^p_{x+1}[g](t-1){\,}dt\right|\\
&=& \left|\int_0^1\frac{t^{\underline{p+1}}}{p}{\,}g[x+1,\ldots,x+p,x+t]{\,}dt\right|,
\end{eqnarray*}
and this function is monotone by Lemma~\ref{lemma:pCInc5}.

In terms of approximations of $\Sigma g(x)$ given in Propositions~\ref{prop:E4Web7Inq6Int0} and \ref{prop:E4Web7Inq6Intp}, this observation shows that, for any $m\in\N$, the approximation of $\Sigma g(x+m)$ is finer than that of $\Sigma g(x)$ and it is actually finer and finer as $m$ increases.

Thus, finer approximations of $\Sigma g(x)$ can be obtained using the following procedure.
\begin{quote}
\begin{enumerate}
\item[Step 1.] Replace $x$ with $x+m$ in Propositions~\ref{prop:E4Web7Inq6Int0} and \ref{prop:E4Web7Inq6Intp}.
\item[Step 2.] Use the substitution (cf.~\eqref{eq:56zzSec32S6})
$$
\Sigma g(x+m) ~=~ \Sigma g(x)+\sum_{k=0}^{m-1} g(x+k)
$$
in the expression of $J^{p+1}[\Sigma g](x+m)$.
\end{enumerate}
\end{quote}

Note that we already used this trick when we investigated the generalized Gautschi inequality (see Remark~\ref{rem:Gautschi44zz}).

\begin{example}[The gamma function]
Let $m\in\N$. Replacing $x$ with $x+m$ in the following approximation of the gamma function (see Example~\ref{ex:E5Gam66m})
$$
e^{-1}\left(1+\frac{1}{x}\right)^{x+\frac{1}{2}} \leq ~ \frac{\Gamma(x)}{\sqrt{2\pi}{\,}e^{-x}{\,}x^{x-\frac{1}{2}}} ~\leq ~ e^{-\frac{x}{2}-\frac{3}{4}}\left(1+\frac{1}{x}\right)^{\frac{1}{2}(x+1)^2}.
$$
and then using the substitution
$$
\Gamma(x+m) ~=~ (x+m-1)^{\underline{m}}\,\Gamma(x)
$$
we finally obtain
\begin{eqnarray*}
e^{-1}\left(1+\frac{1}{x+m}\right)^{x+m+\frac{1}{2}} &\leq & \frac{(x+m-1)^{\underline{m}}\,\Gamma(x)}{\sqrt{2\pi}{\,}e^{-x-m}{\,}(x+m)^{x+m-\frac{1}{2}}} \\
&\leq & e^{-\frac{x+m}{2}-\frac{3}{4}}\left(1+\frac{1}{x+m}\right)^{\frac{1}{2}(x+m+1)^2}.
\end{eqnarray*}
This double inequality provides an approximations of the log-gamma function with an absolute error that is bounded by $\frac{1}{6(x+m)}-\frac{1}{8(x+m)^2}+O(x^{-3})$ as $x\to\infty$.
\end{example}

\index{Webster's inequality!generalized|)}

\chapter{On the differentiability of $\Sigma g$}
\label{chapter:Diff7SigLog6}

\noindent\textsl{Summary: We establish Proposition~\ref{prop:71Surp4Not9Pres}, which states that, for every $p\in\N$, there exists a function $g$ lying in $\cC^{p+1}\cap\cD^p\cap\cK^p$ for which $\Sigma g$ does not lie in $\cC^{p+1}$.}

\medskip

To establish Proposition~\ref{prop:71Surp4Not9Pres}, we first show that it is enough to consider the special case when $p=0$. Suppose that there exists a function $g\colon\R_+\to\R$ lying in $\cC^1\cap\cD^0\cap\cK^0$ such that $\Sigma g$ does not lie in $\cC^1$. By Proposition~\ref{prop:LMpGpLMp1}, its antiderivative
$$
G(x) ~=~ \int_1^xg(t){\,}dt
$$
clearly lies in $\cC^2\cap\cD^1\cap\cK^1$. By Proposition~\ref{prop:an4tR8}, we also have
$$
D\Sigma G(x) ~=~ \Sigma g(x)-\sigma[g],\qquad x>0.
$$
Since we assumed that $\Sigma g$ does not lie in $\cC^1$, it follows that $\Sigma G$ cannot lie in $\cC^2$. Iterating this process, we obtain that the statement is true for any $p\in\N$.

We now construct a function $g$ lying in $\cC^1\cap\cD^0\cap\cK^0$ (and even in $\cC^{\infty}$) and such that the function $\Sigma g$ does not lie in $\cC^1$.

Consider first the function $\Psi\colon\R\to\R$ defined by
$$
\Psi(x) ~=~
\begin{cases}
\alpha\,\exp\left(1-\frac{1}{1-4x^2}\right), & \text{if $x\in\left(-\frac{1}{2},\frac{1}{2}\right)$},\\
0{\,}, & \text{otherwise},
\end{cases}
$$
where
$$
\frac{1}{\alpha} ~=~ \int_{-1/2}^{1/2}\exp\Big(1-\frac{1}{1-4x^2}\Big){\,}dx{\,}.
$$
Thus defined, $\Psi$ is a bump function that is of class $\cC^{\infty}$ with the compact support
$$
\mathrm{supp}(\Psi) ~=~ \textstyle{\left[-\frac{1}{2},\frac{1}{2}\right]}.
$$
For every $m\in\N^*$, define the function $\Psi_m\colon\R\to\R$ by the equation
$$
\Psi_m(x) ~=~ \Psi(2^m(x-m))\quad\text{for $x\in\R$}.
$$
We clearly have that
\begin{equation}\label{eq:FSupp0Psm4}
\mathrm{supp}(\Psi_m) ~=~ \textstyle{\left[m-\frac{1}{2^{m+1}},m+\frac{1}{2^{m+1}}\right]},
\end{equation}
$$
\int_{\R_+}\Psi_m(x){\,}dx ~=~ \frac{1}{2^m},\quad\text{and}\quad\Psi_m(m) ~=~ \alpha{\,}.
$$
Now, define the functions $\overline{\Psi}\colon\R_+\to\R$ and $g\colon\R_+\to\R$ by
$$
\overline{\Psi}(x) ~=~ \sum_{m=1}^{\infty}\Psi_m(x).
$$
and
$$
g(x) ~=~ -1+\int_0^x \overline{\Psi}(t){\,}dt.
$$
Then, we can easily see that the function $g$ lies in $\cC^{\infty}\cap\cD^0\cap\cK^0_+$, and hence the function $\Sigma g$ exists and lies in $\cC^0\cap\cD^1\cap\cK^0_-$.

We now have the following claim, which establishes Proposition~\ref{prop:71Surp4Not9Pres}.

\begin{claim}
For any $m\in\N^*$, the function $\Sigma g$ is not differentiable at $m$. More precisely, we have
$$
\lim_{h\to 0}\frac{\Sigma g(m+h)-\Sigma g(m)}{h} ~=~ {-\infty}.
$$
\end{claim}

\begin{proof}
Since $g$ lies in $\cC^{\infty}$ and satisfies the equation $\Sigma g(x+1)=\Sigma g(x)+g(x)$, it is enough to prove the claim for $m=1$. For any $h>0$, we have
\begin{eqnarray*}
\frac{1}{h}{\,}(\Sigma g(1+h)-\Sigma g(1)) &=& \frac{1}{h}{\,}\Sigma g(1+h) ~=~ -\frac{1}{h}\,\sum_{k=1}^{\infty}(g(k+h)-g(k))\\
&=& -\sum_{k=1}^{\infty}g[k,k+h].
\end{eqnarray*}
Now, for any $k\in\N^*$ the function $g$ is increasing and concave on $[k,k+\frac{1}{2})$ (because its derivative $g'|_{[k,k+\frac{1}{2})}=\Psi_k|_{[k,k+\frac{1}{2})}$ is nonnegative and decreasing). We then see that the function
$$
h ~\mapsto ~ g[k,k+h]
$$
is nonnegative and continuously decreases (by Lemma~\ref{lemma:pCInc5}) on $[0,\frac{1}{2})$ with maximum value $g[k,k]=g'(k)=\Psi_k(k)=\alpha$. It follows that, for any integers $1\leq k\leq m$, there exists $0<\delta_{k,m}<\frac{1}{2}$ such that
$$
\frac{\alpha}{2} ~\leq ~ g[k,k+h] ~\leq ~ \alpha\qquad\text{for all $h\in (0,\delta_{k,m})$}.
$$
Thus, for any $m\in\N^*$, there exists
$$
0 ~<~ h_m ~<~ \min_{k=1,\ldots,m}\delta_{k,m},
$$
such that
$$
\frac{\alpha}{2} ~\leq ~ g[k,k+h_m] ~\leq ~ \alpha\qquad k=1,\ldots,m.
$$
Thus, we have
$$
\frac{1}{h_m}{\,}\Sigma g(1+h_m) ~=~  -\sum_{k=1}^{\infty}g[k,k+h_m] ~\leq ~ -\sum_{k=1}^mg[k,k+h_m] ~\leq ~ -m\,\frac{\alpha}{2}{\,},
$$
which shows that the function $\Sigma g$ cannot be right-differentiable at $1$.

Now, since the function
$$
h ~\mapsto ~ \frac{1}{h}{\,}\Sigma g(1+h) ~=~ -\sum_{k=1}^{\infty}g[k,k+h]
$$
is increasing on $[0,\frac{1}{2})$, we can easily see that
$$
\lim_{h\to 0^+} \frac{1}{h}\,\Sigma g(1+h) ~=~ {-\infty}.
$$
Similarly, we obtain the same limit when $h\to 0^-$.
\end{proof}

Thus, we have shown that $\Sigma g$ is a continuous and decreasing function that is not differentiable at each positive integer. Let us now establish the interesting fact that $\Sigma g$ is of class $\cC^{\infty}$ on $\R_+\setminus\N$.

\begin{claim}
The function $\Sigma g$ is of class $\cC^{\infty}$ on $\R_+\setminus\N$.
\end{claim}

\begin{proof}
Since $g$ lies in $\cC^{\infty}$ and satisfies the equation $\Sigma g(x+1)=\Sigma g(x)+g(x)$, it is enough to show that $\Sigma g$ is of class $\cC^{\infty}$ on $(0,1)$, or equivalently, on every compact interval $[a,b]$, with $0<a<b<1$.

By the existence Theorem~\ref{thm:exist}, the sequence $n\mapsto f_n^0[g]$, with
$$
f_n^0[g](x) ~=~ \sum_{k=1}^{n-1}g(k)-\sum_{k=0}^{n-1}g(x+k),
$$
converges uniformly to $\Sigma g$ on $[a,b]$. Let us now show that the sequence $n\mapsto Df_n^0[g]$, with
$$
Df_n^0[g](x) ~=~ -\sum_{k=0}^{n-1}\overline{\Psi}(x+k),
$$
converges uniformly on $[a,b]$. In view of identity \eqref{eq:FSupp0Psm4}, it is then clear that there exists $k_0\in\N^*$ for which
$$
\mathrm{supp}(\Psi_k) ~\cap ~ [a+k,b+k]~\cap ~\mathrm{supp}(\Psi_{k+1}) ~=~ \varnothing\qquad\text{for every $k\geq k_0$}.
$$
Thus, for any integer $k\geq k_0$ and any $x\in [a,b]$, we have $\overline{\Psi}(x+k)=0$. Therefore, we have
$$
Df_n^0[g](x) ~=~ -\sum_{k=0}^{k_0-1}\overline{\Psi}(x+k),\qquad x\in [a,b],~n\geq k_0.
$$
It follows that the sequence $n\mapsto Df_n^0[g]|_{[a,b]}$ is eventually constant and hence uniformly convergent on $[a,b]$. Using the classical result on uniform convergence and differentiation, we obtain that $\Sigma g$ is of class $\cC^1$ on $[a,b]$. An immediate adaptation of this proof shows that $\Sigma g$ is of class $\cC^{\infty}$ on $[a,b]$.
\end{proof}

\backmatter

\clearpage
\addcontentsline{toc}{chapter}{Bibliography}
\chaptermark{Bibliography}

\clearpage
\addcontentsline{toc}{chapter}{Analogues of properties of the gamma function}
\chaptermark{Analogues of properties of the gamma function}
\chapter*{Analogues of properties of the gamma function}

\begin{tabbing}
\textsl{Analogue of Bohr-Mollerup's theorem.} Theorems~\ref{thm:int1} and \ref{thm:unic} \\
\textsl{Analogue of Burnside's formula.} Section~\ref{sec:6An4Buzzrns2} \\
\textsl{Analogue of Euler's infinite product (Eulerian form).} Section~\ref{sec:Eul4For631} \\
\textsl{Analogue of Euler's reflection formula.} Section~\ref{sec:ReflFor62} \\
\textsl{Analogue of Euler's series representation of $\gamma$.} Equation~\eqref{eq:EulerAnal5571} \\
\textsl{Analogue of Fontana-Mascheroni's series.} Section~\ref{sec:8AnzzFont2Mas} \\
\textsl{Analogue of Gauss' digamma theorem.} Section~\ref{sec:8GauDigTh7} \\
\textsl{Analogue of Gauss' limit.} Theorems~\ref{thm:int1} and \ref{thm:unic} \\
\textsl{Analogue of Gauss' multiplication formula.} Section~\ref{sec:GaussMultF51} \\
\textsl{Analogue of Gautschi's inequality.} Section~\ref{sec:8Gautschi334} \\
\textsl{Analogue of Legendre's duplication formula.} Section~\ref{sec:GaussMultF51} \\
\textsl{Analogue of Raabe's formula.} Section~\ref{sec:Raabe448} \\
\textsl{Analogue of Wallis's product formula.} Section~\ref{sec:Wallis582} \\
\textsl{Analogue of Weierstrass' infinite product (Weierstrassian form).} Section~\ref{sec:WF73} \\
\\
\textsl{Generalized Binet's function.} Section~\ref{sec:6Asym4Cons6Bine} \\
\textsl{Generalized Euler's constant.} Section~\ref{sec:GEc53} \\
\textsl{Generalized Liu's formula.} Section~\ref{sec:8AsyzzExp3Rel2} \\
\textsl{Generalized Stirling's constant.} Definition~\ref{de:GSC556} \\
\textsl{Generalized Stirling's formula.} Sections~\ref{sec:6Gen4St2Fo0} and \ref{sec:8AsyzzExp3Rel2} \\
\textsl{Generalized Webster's functional equation.} Section~\ref{sec:8GenWEb3Fuc7E} \\
\textsl{Generalized Webster's inequality.} Appendix~\ref{chapter:E-GenWeIn5} \\
\textsl{Generalized Wendel's inequality.} Section~\ref{sec:6Wen0delI6eq2} \\
\end{tabbing}

\clearpage
\addcontentsline{toc}{chapter}{Index}
\chaptermark{Index}
\printindex
\end{document}